\documentclass[11pt,reqno]{amsart}
\usepackage{geometry} 
\usepackage{amsmath,amssymb,amsthm, mathrsfs}
\usepackage{xcolor}
\usepackage{ulem}

\usepackage{graphicx} 

\usepackage{enumerate}
\usepackage{indentfirst}
\usepackage[colorlinks,
            linkcolor=black,
            anchorcolor=black,
            citecolor=black
            ]{hyperref}

\DeclareMathOperator{\diver}{\mathrm{div}}

\newtheorem{mydef}{Definition}
\newtheorem{thm}{Theorem}[section]
\newtheorem{lem}{Lemma}[section]

\newtheorem{rk}{Remark}

\newtheorem{cor}{Corollary}[section]
\numberwithin{rk}{section}
\numberwithin{prop}{section}
\numberwithin{mydef}{section}
\numberwithin{lem}{section}
\numberwithin{equation}{section}
\numberwithin{thm}{section}
\allowdisplaybreaks[4]

\title[Degenerate Compressible Navier-Stokes Equations]{Global Regular Solutions of the  Degenerate Compressible Navier-Stokes Equations with Large Initial Data of Spherical Symmetry}

\date{\today}

\author{Gui-Qiang G. Chen}
\address[Gui-Qiang G. Chen]{Mathematical Institute, University of Oxford, Oxford, OX2 6GG, UK.} \email{\tt gui-qiang.chen@maths.ox.ac.uk}

\author{Jiawen Zhang}
\address[Jiawen Zhang]{School of Mathematical Sciences, Shanghai Jiao Tong University, Shanghai 200240, P. R. China.} \email{\tt zhangjiawen317@sjtu.edu.cn}

\author{Shengguo Zhu }
\address[Shengguo Zhu]{School of Mathematical Sciences,  CMA-Shanghai,   and MOE-LSC,   Shanghai Jiao Tong University, Shanghai 200240, P. R. China.}
 \email{\tt  zhushengguo@sjtu.edu.cn}

\begin{document}

\begin{abstract}
A fundamental open problem in the theory of the compressible Navier-Stokes equations is whether regular spherically symmetric flows can develop singularities\,---\,such as cavitation or implosion\,---\,in finite time. 
A formidable challenge lies in how the well-known coordinate singularity at the origin can be overcome to control the lower or upper bound of the density. In this paper, when the viscosity coefficients are degenerately density-dependent (as in the shallow water equations), we prove that, for general large spherically symmetric initial data with bounded positive density, solutions remain globally regular and cannot undergo cavitation or implosion in two and three spatial dimensions. Moreover,  the far-field  vacuum is allowed for the data under consideration here. Our results hold for all adiabatic exponents $\gamma\in(1,\infty)$ in two dimensions, and for physical adiabatic exponents $\gamma\in (1, 3)$ in three dimensions, without any restriction on the size of the initial data.
\end{abstract}



\maketitle

\tableofcontents

\section{Introduction}\label{section-intro}
The global regularity of solutions of the multidimensional (M-D) compressible Navier--Stokes equations (\textbf{CNS}) with large initial data is a longstanding open problem, even when the initial data exhibit some form of symmetry. The main obstacle lies in obtaining uniform {\it a priori} bounds on the density, both above and below; see Lions \cite{lions}, Hoff--Smoller \cite{hoffsmoller}, Feireisl \cite{fu3}, and the references therein. In this paper, we establish the well-posedness of global regular solutions for general smooth, spherically symmetric initial data with positive and bounded density for the following \textbf{CNS} with degenerately density-dependent viscosity coefficients (as in the shallow water equations)
in $\mathbb{R}^n$ ($n=2$ or $3$):
\begin{equation}\label{eq:1.1benwen}
\begin{cases}
\rho_t+\diver(\rho \boldsymbol{u})=0,\\[3pt]
(\rho \boldsymbol{u})_t+\diver(\rho \boldsymbol{u}\otimes \boldsymbol{u})
+ \nabla P =2\alpha\diver(\rho D(\boldsymbol{u})),
\end{cases}
\end{equation}
 with the initial  data:
\begin{equation}\label{eqs:CauchyInit}
(\rho,\boldsymbol{u})(0,\boldsymbol{x})=(\rho_0,\boldsymbol{u}_0)(\boldsymbol{x})=(\rho_0(|\boldsymbol{x}|),u_0(|\boldsymbol{x}|)\frac{\boldsymbol{x}}{|\boldsymbol{x}|}) 
\qquad \text{for $\,\boldsymbol{x} \in \mathbb{R}^n$},
\end{equation}
and the far-field asymptotic condition:
\begin{equation}\label{e1.3}
\displaystyle
(\rho,\boldsymbol{u})(t,\boldsymbol{x})
\to (\bar\rho,\boldsymbol{0})
\qquad \text{as $|\boldsymbol{x}|\to \infty\,$  for $\,t\ge 0$},
\end{equation}
where  $t\geq 0$ and   $\boldsymbol{x}=(x_1,\cdots\!,x_n)^\top\in \mathbb{R}^n$   denote the time  and  Eulerian spatial coordinates, respectively,  $\rho\geq 0$  is  the density of the fluid and the initial density 
$\rho_0$
here is positive and bounded,   
$\boldsymbol{u}=(u_1,\cdots\!,u_n)^\top\in \mathbb{R}^n$ stands for the  velocity of fluid,
$P$ is  the pressure, 
$D(\boldsymbol{u})=\frac{1}{2}(\nabla \boldsymbol{u}+(\nabla \boldsymbol{u})^\top)$ is    the strain tensor, and  $\alpha>0$ and   $\bar{\rho}\geq 0$ are both constants. For the polytropic gas,  the constitutive relation is given by $P=A\rho^{\gamma}$, where $A>0$ is  the entropy  constant and  $\gamma>1$ is the adiabatic exponent.

When $\gamma=n=2$, \eqref{eq:1.1benwen} corresponds to the viscous shallow water equations, 
\begin{equation}\label{shallow}
\begin{cases}
h_t+\diver(h \boldsymbol{u})=0,\\[3pt]
(h\boldsymbol{u})_t+\diver(h \boldsymbol{u}\otimes \boldsymbol{u})+A\nabla (h^2)={V}(h, \boldsymbol{u}),
\end{cases}
\end{equation}
where $h$ is  the height of fluid surface,  $\boldsymbol{u}\in \mathbb{R}^2$ is the  horizontal
velocity, and ${V}(h, \boldsymbol{u})$ is the viscosity term. For the spherically symmetric flow,  since $D(\boldsymbol{u})=\nabla \boldsymbol{u}$, 
then  in \eqref{shallow},
\begin{equation}\label{shallownian}
{V}(h, \boldsymbol{u})=2\alpha\diver(h D(\boldsymbol{u}))
= 2\alpha\diver(h\nabla \boldsymbol{u}).
\end{equation}
We refer to \cite{BP,Gent,gpm,lions,Mar,oran} and the references therein for more details on the viscous shallow water system.

Recently, the shallow water equations have attracted considerable attention. Nevertheless, as Lions already emphasized in his 1998 monograph \cite{lions}, 
the global existence of M-D solutions of the Cauchy 
problem of system \eqref{shallow}--\eqref{shallownian} with general initial data remains completely open, especially in the context of regular solutions. Vasseur-Yu \cite{vayu}  proved the global existence of M-D weak solutions with vacuum. Their result provided the first rigorous global existence result for weak solutions with finite energy, yet left unanswered a further fundamental question of whether regular solutions with finite energy 
persist globally in $\mathbb{R}^n$. In this paper, we give a positive answer to this problem under the spherically symmetric setting.

There is a vast literature on the well-posedness of solutions to \textbf{CNS} 
when the viscosity coefficients are constants and the initial density is strictly positive. The one-dimensional (1-D) problem  has been studied extensively; 
see \cite{CHT, hoffsmoller,Kanel, KN,KS, Denis,ZA1} and the references therein. 
In the M-D case, the local well-posedness of classical solutions of the Cauchy problem 
follows from the standard symmetric hyperbolic-parabolic structure satisfying 
the well-known Kawashima's 
condition; see \cite{itaya1,  KA, nash, serrin,tani} and the references therein. 
Matsumura--Nishida \cite{MN} first established the global existence of three-dimensional (3-D) classical solutions for initial data close to a non-vacuum equilibrium in the Hilbert space $H^s(\mathbb{R}^n)$. 
We also refer the reader to  Danchin \cite{danchin1} for  
global strong solutions with small initial data in some Besov spaces of $\mathbb{R}^n$ 
for $n\geq 2$, and Hoff \cite{H3} for global weak solutions with small discontinuous data in some Sobolev spaces of $\mathbb{R}^n$ for $n=2$, $3$. 
For  spherically symmetric flow, Jiang \cite{J}  established the global existence of smooth solutions with 
large data  in the domain exterior to a ball in $\mathbb{R}^n$ for $n=2$ or $3$, and  
Hoff \cite{H}  proved the global existence of   weak solutions of the Cauchy problem of isothermal flow in $\mathbb{R}^n$ for $n\geq 2$ 
when the initial data are large and discontinuous. 
 It is worth pointing out that the analysis in \cite{H} allows the possibility that vacuum  emerges at the origin.
 Some other related progress can also be found 
in  \cite{chz,HJ} and the reference therein.

In general, a vacuum is required for the far-field 
under some physical requirements
such as finite total mass and total energy in the whole space $\mathbb{R}^n$. However,  the approaches used in the references mentioned above do not work directly 
for the case when cavitation appears, owing to the degeneracy of the time evolution 
in the momentum equations, which makes it difficult to study the dynamics of the fluid velocity near the vacuum. 
In fact, in the view of physics, it is not clear how the fluid velocity can be defined when there is no fluid. 
On the other hand, in terms of mathematical structures, the time evolution equations of the fluid velocity in  \textbf{CNS} with constant viscosity coefficients, {\it i.e.} the momentum equations, is a parabolic system in the fluid region, but degenerates to an elliptic one near the vacuum region.
By introducing some initial compatibility conditions, the local well-posedness of  3-D regular solutions with vacuum was established successfully 
in Salvi--Stra\v skraba \cite{Salvi} and Cho--Choe--Kim  \cite{CK3}. 
Later, Huang--Li--Xin \cite{HX1} proved the global well-posedness of classical solutions 
with small data for the barotropic \textbf{CNS} in $\mathbb{R}^3$ 
(see also Choe--Kim \cite{ck} for the global spherically symmetric, smooth solutions with large data in annular domains).
The main breakthrough for the well-posedness of M-D solutions with generic data 
is due to Lions \cite{lions}, where the global existence of weak solutions with finite energy for the barotropic \textbf{CNS} was established 
when the pressure $P$ satisfies $P=A\rho^\gamma$ with   $\gamma \geq \frac{9}{5}$ in $\mathbb{R}^3$ 
and $\gamma\geq \frac{3}{2}$ in $\mathbb{R}^2$. 
For the global existence of weak solutions with finite energy in M-D barotropic flow, 
we also refer the reader to Feireisl--Novotn\'{y}--Petzeltov\'{a} \cite{fu1} 
for  $\gamma>\frac{3}{2}$ in $\mathbb{R}^3$ and $\gamma>1$ in $\mathbb{R}^2$, 
Bresch--Jabin \cite{BJ} for the thermodynamically unstable pressure and anisotropic viscous stress tensor, and 
Jiang--Zhang \cite{JZ} for the spherically symmetric, weak solutions 
in $\mathbb{R}^n$ for $n = 2, 3$ when $\gamma>1$. However, the uniqueness problem of these M-D weak solutions to \textbf{CNS} obtained in \cite{BJ,fu1,JZ,lions} is widely open due to their fairly low regularity.

It is worth pointing out that, when the viscosity coefficients   
are  constants, some singular behaviors of solutions with vacuum to \textbf{CNS}  
have been observed. Recently,  studies  in  \cite{zz2}
show  that the classical solutions with vacuum of the Cauchy problem 
of the M-D  \textbf{CNS}  cannot preserve the conservation of momentum.
On the other hand,  Hoff--Serre \cite{hoffserre} showed that, when vacuum appears, 
the weak solutions of the 1-D barotropic \textbf{CNS} need not continuously depend 
on their initial data. 
In particular, these counterintuitive behaviors can be attributed 
to the unphysical assumption  
that the viscosity coefficients  are  constants  
when one utilizes such a kind of \textbf{CNS} to deal with the vacuum problems 
in viscous fluids,  which makes the vacuum exert a force on the fluid at the interface that 
separates the vacuum and the fluid, 
according to the classical impulse-momentum theorem ({\it cf}. \cite{goldstein}). 
Thus, viscous compressible fluids near the vacuum should be better modeled by the degenerate \textbf{CNS}, which can be derived from the Boltzmann equations through the Chapman--Enskog expansion; see Chapman--Cowling \cite{chap}.

In fact, the degenerate \textbf{CNS} for the barotropic flow has received extensive 
attention in recent years. Many key progresses have been made 
on the global well-posedness of smooth 
solutions when the initial density is strictly positive; 
see \cite{cons,HB,kv,vassu2} for 1-D flow with large data, and \cite{Sundbye2,weike} for two-dimensional (2-D) flow with initial data closing to a non-vacuum equilibrium.  
However,  when  $\inf_{\boldsymbol{x}} \rho_0(\boldsymbol{x}) =0$, 
the momentum equations  are degenerate both in the 
time evolution and spatial dissipation,  which makes it formidable to establish the propagation and mollification mechanisms of the regularity of solutions.
In the M-D case, a remarkable framework was initiated with a series of papers 
by Bresch--Desjardins \cite{bd6,bd8} (started in 2003 with Lin \cite{bd2} 
in the context of Navier--Stokes--Korteweg with a linear shear viscosity coefficient case), 
which provides additional information related to the gradient of a function of 
$\rho$ 
when the viscosity coefficients satisfy what is called the Bresch-Desjardin constraint. 
This information is now called the BD entropy, which 
plays an important role in the development of the global existence of 
M-D weak solutions with finite energy  of  the degenerate \textbf{CNS};
see   Bresch--Vasseur--Yu \cite{bvy},  Li--Xin \cite{lz}, Mellet--Vasseur \cite{mellet},
Vasseur--Yu \cite{vayu}, and the references therein. 
Recently, by some elaborate analysis of the intrinsic degenerate-singular  
structures of the degenerate \textbf{CNS}, 
the local well-posedness of regular solutions with far-field vacuum  has been established 
in  \cite{sz3,sz333,zz2}. Moreover, in the domain exterior to a ball,  Cao--Li--Zhu \cite{clz1} proved the global existence of \text{3-D} spherically symmetric regular solutions with large data and far-field vacuum with the help of the BD entropy estimate. Some other related progress can also be found in  \cite{Germain, zz,  tyc2, zhuthesis} and the reference therein. 

Despite these significant progresses mentioned above for \textbf{CNS}, the global regularity of large solutions in $\mathbb R^n$ $(n\ge 2)$, no matter with or without any symmetry assumption, remains an open problem,  which is extremely difficult due to the possible cavitation and implosion inside the fluids. As far as we know, there is no solid progress for compressible viscous flow along this direction in a positive way until now. On the other hand, several negative results in this direction have been obtained. First, when the viscosity coefficients are constants, it is shown in Xin \cite{zx} that when the initial density is compactly supported, any smooth solutions in $H^s$ (for suitably large $s$) of the Cauchy problem of the M-D non-isentropic   \textbf{CNS} without heat conduction will blow up in finite time, which has been extended to the case that the initial density vanishes only at far-fields with a fast decay rate by Rozanova \cite{olga}. However, it remains unclear whether the solutions considered in \cite{olga,zx} exist locally in time in the M-D case. Recently, for the 3-D spherically symmetric flow, Merle--Rapha\"el--Rodnianski--Szeftel \cite{MPI}  prove that there exists a set of finite-energy smooth initial data with far-field vacuum for which the corresponding solutions to the barotropic   \textbf{CNS} implode (with infinite density) in finite time, which has been extended to the case that the initial density is strictly positive by  Buckmaster--Cao-Labora--Gómez-Serrano  \cite{buckmaster} for $\gamma=\frac{7}{5}$. Later,  for the 3-D barotropic \textbf{CNS} without any symmetry assumption, Cao-Labora--G\'omez-Serrano--Shi--Staffilani \cite{shijia}  constructed some smooth solutions that are strictly away from the vacuum and develop an imploding finite time singularity in  $\mathbb{T}^3$ (torus) or $\mathbb{R}^3$. We also refer to Shao--Wang--Wei--Zhang \cite{shao} for some related  progress  when  $\gamma=\frac{5}{3}$.
On the other hand, for the barotropic \textbf{CNS} with degenerate viscosity coefficients, it is shown in \cite{sz333,zz,zhuthesis} that,  for certain classes of 3-D initial data with vacuum in some open set, one can construct the corresponding local classical solutions in inhomogeneous Sobolev space, which will break down in finite time, regardless of the size of the  initial data.

In this paper, we establish the global well-posedness of regular solutions of the Cauchy problem \eqref{eq:1.1benwen}--\eqref{e1.3} for general smooth initial data of spherical symmetry 
in $\mathbb{R}^n$ for $n=2,3$.  In fact, for the spherically symmetric flow with the radial coordinate 
variable $r=|\boldsymbol{x}|$, system  \eqref{eq:1.1benwen} can be reformulated into 
\begin{equation}\label{cosingu}
\begin{cases}
\displaystyle 
\rho_t+(\rho u)_r+\underbrace{\frac{m\rho u}{r}}_{\star}=0,\\[6pt]
\displaystyle
\underbrace{(\rho u)_t+(\rho u^2)_r}_{\circledast}+P_r-\underbrace{2\alpha(\rho u_r)_r}_{\Diamond}-\underbrace{2m \Big(\alpha \Big(\frac{\rho u}{r} \Big)_r-\frac{\alpha \rho_r u}{r}-\frac{\rho u^2}{2r}\Big)}_{\star} =0,
\end{cases}
\end{equation}
where $\star$ denotes the  coordinates singularity,  $\circledast$ denotes the degenerate 
time evolution, and $\Diamond$ denotes the degenerate spatial dissipation. 
Due to the compressibility of fluids, 
our analysis encounters two major obstacles:
\begin{itemize}
\item possible cavitation, {\it i.e.} $\rho(t,r)\to 0$ for some $(t,r)\in (0,T]\times [0,\infty)$;
\smallskip
\item possible implosion, {\it i.e.} $\rho(t,r)\to \infty$ for some $(t,r)\in (0,T]\times [0,\infty)$.
\end{itemize}
Overcoming these difficulties is particularly challenging because of several inherent issues:
\begin{itemize}
\item the coordinate singularity at the origin, manifested by the singular factor $\tfrac{1}{r}$ in system \eqref{cosingu};
\smallskip
\item the degeneracies in both the time evolution ($\circledast$) and the spatial dissipation ($\Diamond$), arising  from the far-field vacuum in the case $\bar\rho=0$ in \eqref{e1.3}.
\end{itemize}

In fact, almost all known results on the global spherically symmetric strong solutions with large 
data (see \cite{clz1,ck,J,WZ2} and the references therein) are established 
on the domains that exclude the origin. Moreover, for the barotropic \textbf{CNS} with constant viscosity coefficients, when the domain does include the origin, some crucial observations have been made regarding 
the occurrence of implosion and cavitation at the origin; see \cite{buckmaster, shijia, jessenshuzhi,MPI,shao}.
These observations suggest that it is unlikely one could establish the global existence of regular spherically symmetric solutions with large data for the M-D barotropic \textbf{CNS} on domains containing the origin, 
regardless of whether the initial density is strictly positive or not.

Fortunately, by exploiting the intrinsic degenerate-singular structure of \eqref{eq:1.1benwen} 
in radial coordinates, {\it e.g.} \eqref{cosingu}, 
together with a careful analysis, we establish the global well-posedness 
of regular spherically symmetric solutions of the Cauchy problem \eqref{eq:1.1benwen}--\eqref{e1.3} 
in $\mathbb{R}^n$ for $n=2,3$, provided the initial density is positive and bounded 
in the whole space. Moreover,  the far-field  vacuum is allowed for the data under consideration here. Our conclusion  holds for general smooth spherically symmetric initial data, without any restriction on their size. 
In particular, our results indicate that 
these spherically symmetric solutions of \textbf{CNS} never develop cavitation or implosion in finite time, as long as the initial density is positive and bounded. To achieve this, the central difficulty lies in establishing a uniform upper bound for $\rho$ in $\mathbb{R}^n$, which is highly intricate due to the obstacles discussed above.  Our main contribution is to develop some new radial weighted estimates for density $\rho$ by using the effective velocity $\boldsymbol{v}=\boldsymbol{u}+2\alpha\nabla\log\rho$.
This framework allows us, for the first time, to simultaneously control $(\rho, v)$ across the entire domain, including the delicate singular region near the origin.

The rest of this paper is organized as follows:
In \S \ref{maintheorem}, we present the main theorems of the paper.
In \S \ref{Section2}, we outline the main strategies underlying the proof of global well-posedness as stated in \S \ref{maintheorem}.
In \S \ref{section-upper-density}--\S \ref{se46}, we provide a detailed proof of the global well-posedness for regular solutions with far-field vacuum of the Cauchy problem \eqref{eq:1.1benwen}--\eqref{e1.3}
with general smooth spherically symmetric initial data. 
In particular, in \S \ref{section-upper-density}--\S \ref{section-global3}, 
we establish the global uniform estimates for the regular solutions in carefully designed function spaces. This is achieved in the following four steps:
\begin{enumerate}
\item[(i)] Derive the global-in-time {\it a priori} upper bound for $\rho$ (\S \ref{section-upper-density});
\vspace{3pt}
\item[(ii)] Establish the global uniform $L^\infty(\mathbb{R}^n)$-estimate for the effective velocity (\S \ref{section-effective});
\vspace{3pt}
\item[(iii)] Prove that cavitation at the origin cannot occur in finite time (\S \ref{section-nonformation});
\vspace{3pt}
\item[(iv)] Establish global uniform estimates for $2$- and $3$-order regular solutions (\S \ref{section-global2}--\S \ref{section-global3}).
\end{enumerate}
Based on these steps, in \S \ref{se46}, we obtain the global well-posedness 
of  the Cauchy problem \eqref{eq:1.1benwen}--\eqref{e1.3} for regular solutions with far-field vacuum,
by using the method of continuity.
Furthermore, in view of the double-degenerate structure of \eqref{cosingu} 
induced by the far-field vacuum, the local well-posedness of regular solutions with far-field vacuum  
(which is crucial for the arguments in \S \ref{section-upper-density}--\S \ref{se46}) is highly nontrivial. This is established in \S \ref{section-local-regular}.

Moreover, in \S \ref{nonvacuumfarfield}, we address the case when the initial density is strictly positive and establish the global well-posedness of regular solutions for general smooth, spherically symmetric data. Finally, we list some auxiliary lemmas and new Sobolev embedding theorems for spherically symmetric functions that are used frequently throughout this paper in  Appendices \ref{appA}--\ref{improve-sobolev}.

\section{Main Theorems}\label{maintheorem}
This section is devoted to stating our main theorems on the global well-posedness 
of regular solutions of the Cauchy problem \eqref{eq:1.1benwen}--\eqref{e1.3} 
with large initial data of spherical symmetry in $\mathbb{R}^n$ for $n=2,3$.
For simplicity, throughout this paper,
for any function space defined on $\mathbb{R}^n$, 
the following conventions are used for any $k\in \mathbb{N}$,
unless otherwise specified: 
\begin{equation} \label{eulerspace}
\begin{split}
&\|f\|_{L^p}=\|f\|_{L^p(\mathbb{R}^n)},\quad \|f\|_{H^k}=\|f\|_{H^k(\mathbb{R}^n)},\quad  \|f\|_{W^{k,p}}=\|f\|_{W^{k,p}(\mathbb{R}^n)},\\[4pt]
&D^{k,p}(\mathbb{R}^n)
=\big\{f\in L^1_{\mathrm{loc}}(\mathbb{R}^n):\,\|f\|_{D^{k,p}(\mathbb{R}^n)}=\|\nabla^k f\|_{L^p(\mathbb{R}^n)}<\infty\big\},\\[4pt]
&D^k(\mathbb{R}^n)=D^{k,2}(\mathbb{R}^n),   \quad  \|f\|_{D^{k,p}}=\|f\|_{D^{k,p}(\mathbb{R}^n)},\quad  \|f\|_{D^{k}}=\|f\|_{D^{k}(\mathbb{R}^n)},\\[4pt]
&H^{0}(\mathbb{R}^n)=L^2(\mathbb{R}^n),\ 
W^{0,p}(\mathbb{R}^n)= 
D^{0,p}(\mathbb{R}^n)=L^p(\mathbb{R}^n), \  H^{-k}(\mathbb{R}^n)=\big(H^k(\mathbb{R}^n)\big)^*.
\end{split}
\end{equation}

\subsection{Global spherically symmetric solutions of the degenerate \textbf{CNS} with far-field vacuum}
We first address the case that  $\bar\rho=0$ in \eqref{e1.3}. 

We consider the following physical range of the adiabatic exponent $\gamma$ in  system \eqref{eq:1.1benwen}:
\begin{equation}\label{cd1}
\gamma\in (1,\infty) \ \ \text{if} \ \  n=2, \qquad\quad
\gamma\in (1,3) \ \ \text{if} \ \  n=3.
\end{equation}

Notice that,  if  $\rho>0$, the momentum equations $\eqref{eq:1.1benwen}_2$ can be formally rewritten as
\begin{equation}\label{qiyi}
\begin{split}
\boldsymbol{u}_t+\boldsymbol{u}\cdot\nabla \boldsymbol{u} +\frac{A\gamma}{\gamma-1}\nabla\rho^{\gamma-1}+ L\boldsymbol{u}=\nabla \log\rho \cdot  Q(\boldsymbol{u}),
\end{split}
\end{equation}
where $L \boldsymbol{u}$ and $Q(\boldsymbol{u})$ are  given by
\begin{equation}\label{operatordefinition}
L\boldsymbol{u}=-\alpha\Delta \boldsymbol{u}-\alpha\nabla \diver\boldsymbol{u},\qquad Q(\boldsymbol{u})=2\alpha D(\boldsymbol{u}).
\end{equation}
If \eqref{qiyi} is used to study the time evolution of $\boldsymbol{u}$, 
then it can transfer the degeneracies both in the time evolution and spatial dissipation 
to the possible singularity of $\nabla \log \rho$. Therefore, the two quantities
\begin{equation*}
(\rho^{\gamma-1}, \nabla \log\rho)
\end{equation*}
will play significant roles in our analysis on the high-order regularity of the fluid velocity 
$\boldsymbol{u}$. 
Due to this observation, for the case that $\bar \rho=0$ in \eqref{e1.3}, 
we first introduce a proper class of solutions, called regular solutions, of 
the Cauchy problem  \eqref{eq:1.1benwen}--\eqref{e1.3}.

\begin{mydef}\label{cjk}
Assume that 
$\bar\rho=0$ in \eqref{e1.3}, $s=2$ or $3$,  and  $T>0$. 
The vector function $(\rho, \boldsymbol{u})$ is called an $s$-order regular solution of 
the Cauchy problem  \eqref{eq:1.1benwen}{\rm--}\eqref{e1.3} in $[0,T]\times \mathbb{R}^n$ $(n=2$ or $3)$, 
if
\begin{equation*}\begin{split}
\mathrm{(i)}& \ (\rho, \boldsymbol{u}) \ \text{satisfies this problem in the sense of distributions};\\
\mathrm{(ii)}& \ 0<\rho\in C([0,T];L^1(\mathbb{R}^n)), 
\quad \nabla\log\rho\in L^\infty([0,T]\times\mathbb{R}^n),\\[2pt]
&\ \nabla\rho^{\gamma-1}\in C([0,T];H^{s-1}(\mathbb{R}^n)), \quad (\rho^{\gamma-1})_t \in C([0,T]; H^{s-1}(\mathbb{R}^n)),\\[2pt]
& \ \nabla^2 \log \rho \in C([0,T]; H^{s-2}(\mathbb{R}^n)),\quad (\nabla\log\rho)_t\in C([0,T]; H^{s-2}(\mathbb{R}^n));\\[2pt]
\mathrm{(iii)}& \ \boldsymbol{u}\in C([0,T]; H^s(\mathbb{R}^n))\cap L^2([0,T]; D^{s+1}(\mathbb{R}^n)),\\[2pt]
& \ \boldsymbol{u}_t\in C([0,T]; H^{s-2}(\mathbb{R}^n))\cap L^2([0,T]; D^{s-1}(\mathbb{R}^n)).
\end{split}
\end{equation*}
\end{mydef}

\begin{rk}\label{regularsolution}
We first introduce some physical quantities to be used in this paper{\rm :}
\begin{align*}
\mathcal{M}(t)&=\int_{\mathbb{R}^n} \rho(t,\boldsymbol{x})\,\mathrm{d}\boldsymbol{x}\qquad (\textrm{total mass}),\\[2pt]
\mathcal{P}(t)&=\int_{\mathbb{R}^n} \rho(t,\boldsymbol{x})\boldsymbol{u}(t,\boldsymbol{x})\,\mathrm{d}\boldsymbol{x} \qquad (\textrm{momentum}). 
\end{align*}
It will be shown later that the $s$-order regular solutions defined here satisfy the conservation of $\mathcal{M}(t)$ and $\mathcal{P}(t)$. Next, the regularity of $\rho$ shown in {\rm Definition \ref{cjk}} 
implies that $\rho>0$ in $\mathbb{R}^n$ but decays to zero in the far-field, 
which is consistent with the facts that $\mathcal{M}(t)$ and $\mathcal{P}(t)$ are both conserved, 
and {\rm\textbf{CNS}} is a model of non-dilute fluids. Thus, the definition of the s-order regular solutions are consistent with the physical background of {\rm\textbf{CNS}}.
\end{rk}

Our first theorem is on the global well-posedness of the $2$-order regular solutions 
of the Cauchy problem  \eqref{eq:1.1benwen}--\eqref{e1.3} with large initial data of spherical symmetry. 
The regularity of this solution provides the uniqueness 
in both the 2-D and the 3-D cases. However, the $2$-order regular solution is a classical one in the 2-D case, but not in the 3-D case.

\begin{thm}\label{th1}
Let $n=2$ or $3$, $\bar\rho=0$ in \eqref{e1.3}, and  \eqref{cd1} hold.
Assume that the initial data  $(\rho_0, \boldsymbol{u}_0)(\boldsymbol{x})$ are spherically symmetric 
and satisfy 
\begin{equation}\label{id1}
0<\rho_0(\boldsymbol{x})\in L^1(\mathbb{R}^n), 
\ \  \nabla\rho_0^{\gamma-1}(\boldsymbol{x})\in H^1(\mathbb{R}^n),
\ \  \nabla\log\rho_0(\boldsymbol{x})\in  D^1(\mathbb{R}^n),
\ \ \boldsymbol{u}_0(\boldsymbol{x})\in H^2(\mathbb{R}^n),\\
\end{equation}
and, in addition,
\begin{equation}\label{shangjie3}
\nabla\log\rho_0\in L^\infty(\mathbb{R}^3) \quad \,\text{ when } \ n=3.
\end{equation}
Then, for any  $T>0$, the  Cauchy problem  \eqref{eq:1.1benwen}{\rm--}\eqref{e1.3} admits a unique global $2$-order regular solution $(\rho,\boldsymbol{u})(t,\boldsymbol{x})$
in $[0,T]\times\mathbb{R}^n$ that satisfies
\begin{equation}\label{er2}
\begin{split}
&\sqrt{t}\boldsymbol{u}\in L^\infty([0,T];D^3(\mathbb{R}^n)),\quad \sqrt{t}\boldsymbol{u}_t\in L^\infty([0,T];D^1(\mathbb{R}^n))\cap L^2([0,T];D^2(\mathbb{R}^n)),\\
&\sqrt{t}\boldsymbol{u}_{tt}\in L^2([0,T];L^2(\mathbb{R}^n)).
\end{split}
\end{equation}
Moreover,  $(\rho, \boldsymbol{u})$ is spherically symmetric with the form{\rm :}
\begin{equation}\label{duichenxingshi}
(\rho,\boldsymbol{u})(t,\boldsymbol{x})=(\rho(t, r),\, u(t,r) \frac{\boldsymbol{x}}{r} ) 
\qquad\,\, \mbox{for $r=|\boldsymbol{x}|$},
\end{equation}
and satisfies the following properties{\rm :}
\begin{itemize}
\item[$\mathrm{(i)}$] When $n=2$, 
the solution we obtain here is classical{\rm :}
\begin{equation}\label{2dclassical1}
\begin{aligned}
(\rho,\nabla \rho, \boldsymbol{u})\in C([0,T]\times \mathbb{R}^2),\quad (\rho_t,\nabla \boldsymbol{u},\nabla^2 \boldsymbol{u},\boldsymbol{u}_t)\in C((0,T]\times \mathbb{R}^2);
\end{aligned}
\end{equation}

\item[$\mathrm{(ii)}$] When $n=3$, the solution we obtain here is classical
in $\mathbb{R}^3_*:=\mathbb{R}^3\backslash\{\boldsymbol{x}\in\mathbb{R}^3:\,|\boldsymbol{x}|=0\}$
except the origin{\rm :}
\begin{equation}\label{3dclassical1}
\begin{aligned}
&(\rho, \boldsymbol{u})\in C([0,T]\times \mathbb{R}^3),\quad ( \nabla \rho,\rho_t)\in C([0,T]\times \mathbb{R}^3_*),\\
& \nabla \boldsymbol{u}\in C([0,T]\times \mathbb{R}^3_*)\cap C((0,T]\times \mathbb{R}^3),\quad ( \nabla^2 \boldsymbol{u},\boldsymbol{u}_t)\in C((0,T]\times \mathbb{R}^3_*);
\end{aligned}
\end{equation}
\item[$\mathrm{(iii)}$]  $(\rho,\boldsymbol{u})$ satisfies 
the conservations of total mass and total momentum {\rm(}that remains zero{\rm)}{\rm :}
\begin{equation}\label{massconservation}
\mathcal{M}(t)=\mathcal{M}(0), \quad\, \mathcal{P}(t)\equiv \boldsymbol{0} \qquad \ 
\text{for $t\in [0,T]$};
\end{equation}
\item[$\mathrm{(iv)}$] For any $T>0$ and $(t,r)\in (0,T]\times [0,\infty)$,
\begin{equation}\label{decay-est}
\begin{gathered}
\min\big\{C(T)^{-1}, \, (e^{-1}\underline{\rho}(r))^{C(T)(\sqrt{r}+1)}\big\}
\leq \rho(t,r) \leq \min\big\{C(T),Cr^{-n+1}\big\},
\end{gathered}
\end{equation}
where $\underline{\rho}(r):=\inf_{z\in [0,r]} \rho_0(z)$, $C\geq 1$ is a constant depending only on $(\rho_0,\boldsymbol{u}_0)$ and $(n,\alpha,\gamma,A)$, and $C(T)\geq 1$ is a constant depending only on $(C,T)$. 
\end{itemize}
\end{thm}

Our second theorem is on the global well-posedness of the $3$-order regular solutions of 
the Cauchy problem \eqref{eq:1.1benwen}--\eqref{e1.3} with large initial data of spherical 
symmetry. 
The regularity of this solution not only provides the uniqueness in both the 2-D and 3-D cases, 
but also shows that it is a classical solution of this problem considered.

\begin{thm}\label{th1-high}
Let $n=2$ or $3$, $\bar\rho=0$ in \eqref{e1.3}, and  \eqref{cd1} hold. 
Assume that the initial data  $(\rho_0, \boldsymbol{u}_0)(\boldsymbol{x})$ are spherically symmetric 
and satisfy 
\begin{equation}\label{id1-high}
\begin{split}
&0<\rho_0(\boldsymbol{x})\in L^1(\mathbb{R}^n), \quad \nabla\rho_0^{\gamma-1}(\boldsymbol{x})\in H^2(\mathbb{R}^n), \\ 
&\nabla\log\rho_0(\boldsymbol{x})\in D^1(\mathbb{R}^n)\cap D^2(\mathbb{R}^n),\quad  
\boldsymbol{u}_0(\boldsymbol{x})\in H^3(\mathbb{R}^n).
\end{split}
\end{equation}
Then, for any $T>0$, the Cauchy problem   \eqref{eq:1.1benwen}{\rm--}\eqref{e1.3} admits a unique global $3$-order regular solution $(\rho,\boldsymbol{u})(t,\boldsymbol{x})$
in $[0,T]\times\mathbb{R}^n$ that satisfies
\begin{equation}\label{er2-high}
\begin{split}
&\sqrt{t}\boldsymbol{u}\in L^\infty([0,T];D^4(\mathbb{R}^n)),\quad \sqrt{t}\boldsymbol{u}_t\in L^\infty([0,T];D^2(\mathbb{R}^n))\cap L^2([0,T];D^3(\mathbb{R}^n)),\\
&\sqrt{t}\boldsymbol{u}_{tt}\in L^\infty([0,T];L^2(\mathbb{R}^n))\cap L^2([0,T];D^1(\mathbb{R}^n)).\end{split}
\end{equation}
Moreover,  $(\rho, \boldsymbol{u})$ is spherically symmetric and satisfies  \eqref{duichenxingshi}, \eqref{massconservation}{\rm--}\eqref{decay-est}, and 
\begin{itemize}
\item[$\mathrm{(i)}$] When $n=2$, the solution we obtain here 
is a classical solution{\rm :}
\begin{equation}\label{2dclassical2}
\begin{aligned}
(\rho,\nabla \rho, \rho_t, \boldsymbol{u},\nabla \boldsymbol{u}, \nabla^2 \boldsymbol{u},\boldsymbol{u}_t)\in C([0,T]\times \mathbb{R}^2);
\end{aligned}
\end{equation}
\item[$\mathrm{(ii)}$] When $n=3$, the solution we obtain here 
is a classical solution{\rm :} 
\begin{equation}\label{3dclassical2}
\begin{aligned}
(\rho,\nabla \rho, \rho_t, \boldsymbol{u},\nabla \boldsymbol{u})\in C([0,T]\times \mathbb{R}^3),\quad ( \nabla^2 \boldsymbol{u},\boldsymbol{u}_t)\in C((0,T]\times \mathbb{R}^3).\end{aligned}
\end{equation}
\end{itemize}
\end{thm}

In addition, when  $\gamma\geq \frac{3}{2}$, the  regularity assumptions 
imposed on $\rho^{\gamma-1}_0$ in \eqref{id1} and \eqref{id1-high} can be removed.

\begin{cor}\label{coro1.1}
For $\gamma\in \big[\frac{3}{2},\infty\big)$ when $n=2$ and $\gamma\in \big[\frac{3}{2},3\big)$ 
when $n=3$, the initial conditions  \eqref{id1}{\rm--}\eqref{shangjie3} in {\rm Theorem \ref{th1}}  can be reduced to 
\begin{equation}\label{id1'}
0<\rho_0(\boldsymbol{x})\in L^1(\mathbb{R}^n), \quad  \nabla\log\rho_0(\boldsymbol{x})\in  D^1(\mathbb{R}^n), \quad \boldsymbol{u}_0(\boldsymbol{x}) \in H^2(\mathbb{R}^n),\\
\end{equation}
and, in addition,
\begin{equation}
\nabla\log\rho_0\in L^\infty(\mathbb{R}^3) \qquad \text{ when $n=3$};
\end{equation}
and the initial condition \eqref{id1-high} in {\rm Theorem \ref{th1-high}} can be reduced to 
\begin{equation}\label{id1'-high}
0<\rho_0(\boldsymbol{x})\in L^1(\mathbb{R}^n), \quad   \nabla\log\rho_0(\boldsymbol{x})\in  D^1(\mathbb{R}^n)\cap D^2(\mathbb{R}^n), \quad   \boldsymbol{u}_0(\boldsymbol{x}) \in H^3(\mathbb{R}^n).
\end{equation}
\end{cor}

\medskip
\subsection{Global spherically symmetric solutions of the degenerate \textbf{CNS} with strictly positive initial density} 
Next we consider the  case that $\bar\rho>0$ in \eqref{e1.3}. For simplicity, when the initial density is strictly positive,  we define the corresponding regular solutions of the Cauchy problem \eqref{eq:1.1benwen}--\eqref{e1.3}.

\begin{mydef}\label{cjk-po}
Assume that  $\bar\rho>0$ in \eqref{e1.3}, $s=2$ or $3$,  and $T>0$. 
The vector function $(\rho, \boldsymbol{u})$ is called an $s$-order regular solution of the Cauchy problem  \eqref{eq:1.1benwen}{\rm--}\eqref{e1.3} in $[0,T]\times \mathbb{R}^n$ $(n=2$ or $3)$, if
\begin{equation*}\begin{split}
\mathrm{(i)}& \ (\rho, \boldsymbol{u}) \ \text{satisfies this problem in the sense of distributions};\\
\mathrm{(ii)}& \ 
\inf_{(t,\boldsymbol{x})\in [0,T]\times\mathbb{R}^n} \rho(t,\boldsymbol{x})>0,\quad  
\rho-\bar\rho\in C([0,T];H^s(\mathbb{R}^n)),\\[2pt]
&\,\,  \nabla\rho  \in L^\infty([0,T]\times\mathbb{R}^n), \quad\rho_t \in C([0,T]; H^{s-1}(\mathbb{R}^n));\\[2pt]
\mathrm{(iii)}& \ \boldsymbol{u}\in C([0,T]; H^s(\mathbb{R}^n))\cap L^2([0,T]; D^{s+1}(\mathbb{R}^n)),\\[2pt]
& \ \boldsymbol{u}_t\in C([0,T]; H^{s-2}(\mathbb{R}^n))\cap L^2([0,T]; D^{s-1}(\mathbb{R}^n)).
\end{split}
\end{equation*}
\end{mydef}

When the initial density is strictly positive, our main results on the global well-posedness 
of the $s$-order  regular solutions of the  Cauchy problem  \eqref{eq:1.1benwen}--\eqref{e1.3} with  
large initial data of spherical symmetry $(s=2,3)$ are stated in the following two theorems:

\begin{thm}\label{th1-po}
Let $n=2$ or $3$, $\bar\rho>0$ in \eqref{e1.3}, and   \eqref{cd1} hold.  
Assume that the initial data  $(\rho_0, \boldsymbol{u}_0)(\boldsymbol{x})$ are spherically symmetric 
and satisfy 
\begin{equation}\label{in-start}
\inf_{\boldsymbol{x}\in \mathbb{R}^n} \rho_0(\boldsymbol{x})>0, \quad (\rho_0-\bar\rho,\boldsymbol{u}_0)\in H^2(\mathbb{R}^n), 
\end{equation}
and, in addition,
\begin{equation}\label{in-start1}
\qquad\nabla\rho_0\in L^\infty(\mathbb{R}^3) \ \text{ when $n=3$}.
\end{equation}
Then, for any  $T>0$, the  Cauchy problem  \eqref{eq:1.1benwen}{\rm--}\eqref{e1.3} 
admits a unique global $2$-order regular solution $(\rho,\boldsymbol{u})(t,\boldsymbol{x})$
in $[0,T]\times\mathbb{R}^n$. Moreover, $(\rho, \boldsymbol{u})$ is spherically symmetric with form \eqref{duichenxingshi}
and satisfies \eqref{er2}{\rm--}\eqref{massconservation} and, for any $T>0$, 
\begin{equation}\label{decay-est-po}
C(T)^{-1}\leq \rho(t,\boldsymbol{x}) \leq C(T) \qquad\,\,
\text{for $\, (t,\boldsymbol{x})\in [0,T]\times \mathbb{R}^n$},
\end{equation}
where $C(T)\geq 1$ is a constant  depending only on $(T,\rho_0,\boldsymbol{u}_0,\bar\rho,n,\alpha,\gamma,A)$.
\end{thm}

Similar to the flow with far-field vacuum, 
the $2$-order regular solution with strictly positive initial density is a classical solution when $n=2$, but not when $n=3$. In order to obtain the corresponding classical solutions for $n=3$, we need to establish the following global well-posedness of the $3$-order regular solutions. 

\begin{thm}\label{th1-high-po}
Assume that $n=2$ or $3$, $\bar\rho>0$ in \eqref{e1.3}, and    \eqref{cd1} holds.
If the initial data  $(\rho_0, \boldsymbol{u}_0)(\boldsymbol{x})$ are spherically symmetric 
and satisfy 
\begin{equation}\label{in-start-high}
\inf_{\boldsymbol{x}\in \mathbb{R}^n} \rho_0(\boldsymbol{x})>0,
\quad (\rho_0-\bar\rho,\boldsymbol{u}_0)\in H^3(\mathbb{R}^n),
\end{equation}
then, for any  $T>0$, the Cauchy problem   \eqref{eq:1.1benwen}{\rm--}\eqref{e1.3} 
admits a unique global $3$-order regular solution $(\rho,\boldsymbol{u})(t,\boldsymbol{x})$
in $[0,T]\times\mathbb{R}^n$. Moreover,  $(\rho, \boldsymbol{u})$ is spherically symmetric with form \eqref{duichenxingshi}
and satisfies  \eqref{massconservation}, \eqref{er2-high}{\rm--}\eqref{3dclassical2}, 
and \eqref{decay-est-po}.
\end{thm}

\subsection{Applications to the 2-D shallow water equations}\label{subsec-shallow}

Another purpose  of this paper is to establish the well-posedness of regular solutions 
with large data applicable to  physically relevant models in the shallow water theory. In fact, for the spherically symmetric flow, when ${V}=2\alpha\diver(h D(\boldsymbol{u}))$ or   $2\alpha\diver(h \nabla \boldsymbol{u})$ in \eqref{shallow},
due to  
$ D(\boldsymbol{u})=\nabla \boldsymbol{u}$,
\eqref{shallow} is a special case of
system \eqref{eq:1.1benwen} with  $\gamma=n=2$.
Therefore, we can simply replaces $(\rho,\boldsymbol{u},\bar\rho)$ by $(h,\boldsymbol{u},\bar h)$
in Theorems \ref{th1}--\ref{th1-high-po} to obtain the same conclusion for these 2-D viscous 
shallow water models.    
More precisely, we consider the Cauchy problem of \eqref{shallow} with the initial data:
\begin{equation}\label{shallowchu}
(h,\boldsymbol{u})(0,\boldsymbol{x})=(h_0,\boldsymbol{u}_0) (\boldsymbol{x})= ( h_0(|\boldsymbol{x}|),u_0(|\boldsymbol{x}|)\dfrac{\boldsymbol{x}}{|\boldsymbol{x}|}) \qquad
\text{for $\boldsymbol{x} \in \mathbb{R}^2$},
\end{equation}
and the far-field asymptotic condition:
\begin{equation}\label{shallowbian}
\left(h,\boldsymbol{u}\right)(t,\boldsymbol{x})\to (\bar h,\boldsymbol{0}) \qquad\,\,
\text{as $\,\left|\boldsymbol{x}\right|\to \infty\,$ for $t\ge 0$},
\end{equation}
where $\bar h\geq 0$ is a fixed constant.
We establish the global spherically symmetric (classical) solutions taking the form:
\begin{equation}\label{qiu-2}
(h,\boldsymbol{u})(t,\boldsymbol{x})=(h(t,|\boldsymbol{x}|),u(t,|\boldsymbol{x}|)\frac{\boldsymbol{x}}{|\boldsymbol{x}|}).
\end{equation}

First,  when  $\bar h=0$, the  regular solutions of the Cauchy problem \eqref{shallow} with \eqref{shallowchu}--\eqref{shallowbian} can be defined analogously to Definition \ref{cjk} 
with  $\gamma=n=2$, and $(\rho,\boldsymbol{u})$ replaced by $(h,\boldsymbol{u})$.  
Then, from  Theorems \ref{th1}--\ref{th1-high} and Corollary \ref{coro1.1},  
the following conclusions hold:

\begin{cor}\label{thshallow}
Let ${V}=2\alpha\diver(h D(\boldsymbol{u}))$ or  $2\alpha\diver(h \nabla \boldsymbol{u})$,  
and $\bar h=0$. If the initial data $(h_0, \boldsymbol{u}_0)(\boldsymbol{x})$ are 
spherically symmetric and satisfy 
\begin{equation*}
0<h_0(\boldsymbol{x})\in L^1(\mathbb{R}^2), \quad\nabla\log h_0(\boldsymbol{x})\in  D^1(\mathbb{R}^2),\quad  \boldsymbol{u}_0(\boldsymbol{x})\in H^2(\mathbb{R}^2),
\end{equation*}
then, for any $T>0$,  the Cauchy problem \eqref{shallow} with  \eqref{shallowchu}{\rm--}\eqref{shallowbian} admits a unique $2$-regular solution $(h, \boldsymbol{u})(t,\boldsymbol{x})$
in $[0,T]\times\mathbb{R}^2$ that satisfies \eqref{er2}. Moreover, $(h, \boldsymbol{u})$ is spherically symmetric with form \eqref{qiu-2}
and satisfies \eqref{2dclassical1} and \eqref{massconservation}{\rm--}\eqref{decay-est} 
with $(\rho,\boldsymbol{u})$ replaced by $(h,\boldsymbol{u})$.
\end{cor}

\begin{cor}\label{thshallow33}
Let ${V}=2\alpha\diver(h D(\boldsymbol{u}))$ or  $2\alpha\diver(h \nabla \boldsymbol{u})$, and $\bar h=0$. If the initial data $(h_0, \boldsymbol{u}_0) (\boldsymbol{x})$ are spherically symmetric, and satisfy 
\begin{equation*}
0<h_0(\boldsymbol{x})\in L^1(\mathbb{R}^2), \quad\nabla\log h_0(\boldsymbol{x})\in  D^1(\mathbb{R}^2)\cap D^2(\mathbb{R}^2),\quad  \boldsymbol{u}_0(\boldsymbol{x})\in H^3(\mathbb{R}^2),
\end{equation*}
then, for any $T>0$,  the Cauchy problem \eqref{shallow} 
with \eqref{shallowchu}{\rm--}\eqref{shallowbian} admits a unique $3$-regular 
solution $(h, \boldsymbol{u})(t,\boldsymbol{x})$
in $[0,T]\times\mathbb{R}^2$ that satisfies \eqref{er2-high}. Moreover, $(h, \boldsymbol{u})$ is spherically symmetric with form \eqref{qiu-2} and 
satisfies \eqref{massconservation}{\rm--}\eqref{decay-est} and \eqref{2dclassical2} 
with $(\rho,\boldsymbol{u})$ replaced by $(h,\boldsymbol{u})$.
\end{cor}

Second, when $\bar h>0$, the regular solutions of the Cauchy problem \eqref{shallow} with \eqref{shallowchu}--\eqref{shallowbian} can be defined analogously to Definition \ref{cjk-po}, with  $\gamma=n=2$ and $(\rho,\boldsymbol{u},\bar \rho)$ replaced by $(h,\boldsymbol{u},\bar h)$. 
Then, from  Theorems \ref{th1-po}--\ref{th1-high-po},  the following conclusions  hold:

\begin{cor}\label{thshallow-po}
Let ${V}=2\alpha\diver(h D(\boldsymbol{u}))$ or  $2\alpha\diver(h \nabla \boldsymbol{u})$, and $\bar h>0$. If the initial data $(h_0, \boldsymbol{u}_0)(\boldsymbol{x})$ are spherically symmetric and satisfy 
\begin{equation*}
\inf_{x\in \mathbb{R}^2} h_0(x)>0,\qquad (h_0-\bar h,\boldsymbol{u}_0)\in H^2(\mathbb{R}^2),
\end{equation*}
then, for any $T>0$, the Cauchy problem \eqref{shallow} 
with  \eqref{shallowchu}{\rm--}\eqref{shallowbian} 
admits a unique $2$-regular solution $(h, \boldsymbol{u})(t,\boldsymbol{x})$
in $[0,T]\times\mathbb{R}^2$ that satisfies \eqref{er2}.
Moreover, $(h, \boldsymbol{u})$ is spherically symmetric with form \eqref{qiu-2} 
and satisfies \eqref{2dclassical1}, \eqref{massconservation}, 
and \eqref{decay-est-po} with $(\rho,\boldsymbol{u},\bar\rho)$ 
replaced by $(h,\boldsymbol{u},\bar h)$.
\end{cor}

\begin{cor}\label{thshallow33-po}
Let ${V}=2\alpha \diver(h D(\boldsymbol{u}))$ or  $2\alpha\diver(h \nabla \boldsymbol{u})$, and $\bar h>0$. If the initial data $(h_0, \boldsymbol{u}_0)(\boldsymbol{x})$ are spherically symmetric 
and satisfy 
\begin{equation*}
\inf_{x\in \mathbb{R}^2} h_0(x)>0,\qquad (h_0-\bar h,\boldsymbol{u}_0)\in H^3(\mathbb{R}^2),
\end{equation*}
then, for any $T>0$,  the Cauchy problem \eqref{shallow} 
with  \eqref{shallowchu}{\rm--}\eqref{shallowbian} admits 
a unique $3$-regular solution $(h, \boldsymbol{u})(t,\boldsymbol{x})$
in $[0,T]\times\mathbb{R}^2$ that satisfies \eqref{er2-high}.
Moreover, $(h, \boldsymbol{u})$ is spherically symmetric with form \eqref{qiu-2} 
and satisfies \eqref{massconservation}, \eqref{2dclassical2}, 
and \eqref{decay-est-po}
with $(\rho,\boldsymbol{u},\bar\rho)$ replaced by $(h,\boldsymbol{u},\bar h)$.
\end{cor}

We now make some remarks on the results of this paper.

\begin{rk}
In the global-in-time well-posedness results
obtained in {\rm Theorems \ref{th1}--\ref{th1-high-po}}, 
there are no restrictions on the size of the initial data.
On the other hand,  when the initial data allow a far-field vacuum, 
the initial condition \eqref{id1}{\rm--}\eqref{shangjie3} or \eqref{id1-high} identifies 
a class of spherically symmetric initial data that provides the unique solvability 
of the  Cauchy problem  \eqref{eq:1.1benwen}{\rm--}\eqref{e1.3}. 
For example, one can  choose $(\rho_0,\boldsymbol{u}_0)$ satisfying the following constraints{\rm :} 
$\boldsymbol{u}_0(\boldsymbol{x})=\frac{\boldsymbol{x}}{r} u_0(r)$ 
with $u_0(r)\in C_{\rm c}^\infty((0,\infty))$, 
and $0<\rho_0(\boldsymbol{x})\in C^\infty(\mathbb{R}^n)$ 
with $\rho_0(\boldsymbol{x})=\rho_0(r)$ such that
\begin{enumerate}
\item[$\mathrm{(i)}$] $\rho_0(r) \rightarrow 0$  algebraically as $r\rightarrow \infty${\rm :}
\begin{equation*}
0<\lim_{r\to \infty}r^{\iota}\rho_0(r)< \infty \qquad    
\text{for $\,\iota>\max\big\{n,\frac{n-2}{2\gamma-2}\big\}$};
\end{equation*}
\item[$\mathrm{(ii)}$] 
$\rho_0(r) \rightarrow 0$ exponentially as $r\rightarrow \infty${\rm :}
\begin{equation*}
0<\lim_{r\to\infty} e^{r^\iota}\rho_0(r)<\infty \qquad \text{for $\,0<\iota< 2-\frac{n}{2}$};
\end{equation*}
\item[$\mathrm{(iii)}$] $\rho_0(r) \rightarrow 0$ as $r\rightarrow \infty$, 
and the decay rate is faster than that of the exponential rate{\rm :}
\begin{equation*}
0<\lim_{r\to\infty} r^{r^\iota}\rho_0(r)<\infty \qquad \text{for $\,0<\iota< 2-\frac{n}{2}$}.
\end{equation*}
\end{enumerate}
\end{rk}

\begin{rk}\label{rk-nabla}
We give some comments on the  constraints for $\rho_0$ given 
in {\rm Theorems \ref{th1}--\ref{th1-high-po}}.
The boundedness of the effective velocity 
$\boldsymbol{v}=\boldsymbol{u}+2\alpha\nabla\log\rho$ plays a key role in our analysis, 
which requires that $\boldsymbol{v}(0,\boldsymbol{x})\in L^\infty(\mathbb{R}^n)$. 
Since $\boldsymbol{u}_0\in L^\infty(\mathbb{R}^n)$ holds by classical Sobolev embedding theorems, 
$\nabla\log\rho_0 \in L^\infty(\mathbb{R}^n)$ is required. 

\smallskip
In $\mathbb{R}^2$, by the following critical Sobolev embedding for spherically symmetric 
vector functions {\rm(}see {\rm Lemma \ref{Hk-Ck-vector}} 
in {\rm Appendix \ref{improve-sobolev}}{\rm)}{\rm :} 
\begin{equation*}
D^1(\mathbb{R}^2) \hookrightarrow L^{\infty}(\mathbb{R}^2)
\end{equation*}
$(\nabla\log\rho_0,\, \nabla\rho_0)\in L^\infty(\mathbb{R}^2)$ 
are  implied by $\nabla\log\rho_0\in D^1(\mathbb{R}^2)$ in \eqref{id1} and \eqref{id1-high}, 
and $\nabla \rho_0\in D^1(\mathbb{R}^2)$ in \eqref{in-start} and \eqref{in-start-high}, 
respectively. 

\smallskip
In $\mathbb{R}^3$, by the following Sobolev embedding for spherically symmetric vector functions {\rm(}see {\rm Lemma \ref{lemma-L6}}{\rm)}{\rm :}
\begin{equation*}
D^1(\mathbb{R}^3) \hookrightarrow L^{6}(\mathbb{R}^3)
\end{equation*}
and the embedding $L^6(\mathbb{R}^3)\,\cap D^2(\mathbb{R}^3)\hookrightarrow \,L^{\infty}(\mathbb{R}^3)$ due to the classical  Gagliardo-Nirenberg inequality, $(\nabla\log\rho_0,\nabla\rho_0)\in L^\infty(\mathbb{R}^3)$ are implied by $(\nabla\log\rho_0,\nabla\rho_0)\in D^1(\mathbb{R}^3)\cap D^2(\mathbb{R}^3)$ in \eqref{id1-high} and \eqref{in-start-high}, respectively. 
 However, when $n=3$ in {\rm Theorems \ref{th1}} and {\rm \ref{th1-po}}, the assumptions that $\nabla\log\rho_0\in D^1(\mathbb{R}^3)$ and $\rho_0-\bar\rho\in H^2(\mathbb{R}^3)$ do not yield $(\nabla\log\rho_0,\nabla\rho_0) \in L^\infty(\mathbb{R}^3)$, 
 which is the reason why this condition is imposed additionally 
 when $n=3$ in \eqref{shangjie3} and \eqref{in-start1}.
\end{rk}

\begin{rk}
For the $2$-{\rm D} case, {\rm Theorems \ref{th1}--\ref{th1-high-po}} hold for any $\gamma \in (1, \infty)$, while  these theorems are valid for $\gamma \in (1,3)$ for the $3$-{\rm D} case. 

This distinction arises primarily from the factor $r^\frac{m}{2}$ $(m=n-1)$ in the BD entropy estimates under the spherical coordinates $($see {\rm Lemma \ref{energy-BD}} in {\rm \S\ref{section-upper-density}} and {\rm Lemma \ref{energy-bd-po}} in {\rm \S\ref{nonvacuumfarfield}}$)${\rm :} 
\begin{equation}\label{bd135}
\sup_{t\in[0,T]}\big\|r^{\frac{m}{2}}(\sqrt{\rho})_r\big\|_{L^2(0,\infty)} \leq C_0, 
\end{equation}
where $C_0>0$ is a constant depending only on $(\rho_0,u_0,\alpha,\gamma,A,n)$ and $\bar\rho$ if $\bar\rho>0$. Note that \eqref{bd135} can be treated as some weighted $L^2(0,\infty)$-estimates of $(\sqrt{\rho})_r$ with the different weight function $r^\frac{m}{2}$ in $(0,\infty)$.  
In fact, in {\rm \S\ref{section-upper-density}} and {\rm \S \ref{nonvacuumfarfield}}, for establishing the global uniform upper bound for $\rho$ on $[0,T]\times [0,1)$, the main challenging issue 
is the well-known  coordinate singularity at the origin. A key idea to overcome the difficulty is to employ the Hardy 
inequality $($see {\rm Lemma \ref{hardy}} in {\rm Appendix \ref{appA}}$)$ and the BD entropy estimate \eqref{bd135} to derive some  $r$-weighted estimates for $\rho$ on $[0,T]\times [0,1)${\rm :}
\begin{equation}\label{314}
\sup_{t\in[0,T]}\|r^K\rho\|_{L^p(0,1)} \leq C \qquad\, \text{for $p\in [1,\infty]$ and  some proper $K=K(p)$}.
\end{equation}
However, due to the dimension-dependent weight function $r^\frac{m}{2}$ in \eqref{bd135},
the exponent $K$ in \eqref{314} differs between the $2$-{\rm D} and the $3$-{\rm D} cases{\rm :} 
for the $2$-{\rm D} case, $K > -\frac{1}{p}$, while $K\geq 1-\frac{1}{p}$ for the $3$-{\rm D} case. 
Since the range of $K$ is smaller in the $3$-{\rm D} case, 
it narrows the admissible range of $\gamma$ in subsequent calculations. 
Through an elaborate analysis based on \eqref{314} in {\rm \S \ref{section-upper-density}} 
and {\rm \S \ref{nonvacuumfarfield}}, we are able to establish the global upper bound 
of $\rho$ on $[0,T]\times [0,1)$ for all $\gamma\in (1,\infty)$ in the $2$-{\rm D} case, 
and for all $\gamma\in (1,3)$ in the $3$-{\rm D} case as shown in {\rm Theorems \ref{th1}--\ref{th1-high-po}}.  
More detailed calculations can be found in {\rm \S \ref{subsub-2.2.2}}.
\end{rk}

\begin{rk}\label{generalcase}
Under proper modifications, when the initial density is positive and bounded, 
the methodology developed in this paper can be applied to proving the global well-posedness  
of regular solutions with general smooth, spherically symmetric initial data  
of the corresponding Cauchy problem of the barotropic {\rm\textbf{CNS}} 
with  nonlinear density-dependent viscosity coefficients 
in $\mathbb{R}^n$  for {\rm $n=2$, $3$}, which will be addressed 
in our forthcoming paper \cite{CZZ}.
\end{rk}

\section{Notations, Reformulations, and Main Strategies}\label{Section2}
In this section, we first present some notations and conventions in \S\ref{section-notaions}, which will be frequently used throughout this paper. Next, in \S\ref{see1}, we introduce an enlarged reformulation for the Cauchy problem \eqref{eq:1.1benwen}--\eqref{e1.3} to deal with the degeneracy caused by the far-field vacuum, which makes the corresponding problem trackable. In \S\ref{subsection-strategy}, based on such a reformulation,  we show the main strategy and new ideas in our analysis for the flow with far-field vacuum.  Finally,   in \S \ref{positive-application}, we show   that, under proper modifications, the methodology developed here  for dealing with  the large-data problems
with far-field vacuum can also be applied to solving the corresponding large-data problems with strictly 
positive initial density. 

\subsection{Notations}\label{section-notaions}
The notations and conventions in this paper are given as follows:
\subsubsection{Notations in M-D Eulerian coordinates} Throughout the rest of this paper, 
unless otherwise specified, we adopt the notations in \eqref{eulerspace} and the following ones:
\begin{itemize}
\smallskip
\item We always let $n=2$ or $3$ be the dimension of the Euclidean  space $\mathbb{R}^n$, and denote  $m:=n-1$.
\smallskip
\item For a variable $\boldsymbol{y} \in \mathbb{R}^l$ ($l\geq 2$), its $i$-th component is denoted by $y_i$ ($1\leq i\leq l$), and $\boldsymbol{y} = (y_1, \cdots\!, y_l)^\top$.  We always let $\boldsymbol{x}=(x_1,\cdots\!,x_n)^\top$ be the spatial variable of $\mathbb{R}^n$.
\smallskip
\item For any vector function $\boldsymbol{f}: E\subset \mathbb{R}^l \to \mathbb{R}^q$ ($l,q\geq 2$, $E$ is a measurable set), its $i$-th component is denoted by $f_i$ ($1\leq i\leq q$), and $\boldsymbol{f} = (f_1, \cdots\!, f_q)^\top$.
\smallskip
\item For any function $f$ defined on a measurable subset of $\mathbb{R}^l$ ($l\geq 1$), if the independent variable of $f$ is $\boldsymbol{y}=(y_1,\cdots\!,y_l)^\top$, then 
\begin{align*}
&\qquad \partial_{y}^\varsigma f=\partial_{y_1}^{\varsigma_1}\cdots\partial_{y_l}^{\varsigma_l} f=f_{\underbrace{\text{\tiny$y_1\cdots y_1$}}_{\text{$\varsigma_1$-times}}\cdots\underbrace{\text{\tiny$y_l\cdots y_l$}}_{\text{$\varsigma_l$-times}}}= \frac{\partial^{\varsigma_1+\cdots+ \varsigma_l}}{\partial y_1^{\varsigma_1}\cdots \partial y_l^{\varsigma_l}}f \quad\ \text{for }\varsigma=(\varsigma_1,\cdots\!,\varsigma_l)\in \mathbb{N}^l,\\
&\qquad \nabla_y f=(\partial_{y_1} f,\cdots\!,\partial_{y_l} f)^\top,\qquad \Delta_y f=\sum_{i=1}^l \partial_{y_i}^2 f,\\[-7pt]
&\qquad \nabla_y^k f \text{ denotes one generic } \partial_y^\varsigma f \text{ with }|\varsigma|=\sum_{i=1}^l \varsigma_i=k \text{ for integer }k\geq 2,\\[-7pt]
&\qquad |\nabla_y^k f|=\Big(\sum_{|\varsigma|=k}|\partial_{y_1}^{\varsigma_1} \cdots\partial^{\varsigma_l}_{y_l}f|^2\Big)^\frac{1}{2} \quad\text{ for } k\in \mathbb{N}^*.
\end{align*}
In particular, for the derivatives with respect to the
variable $\boldsymbol{x}=(x_1,\cdots\!,x_n)^\top\in \mathbb{R}^n$, 
we use the notations 
$(\partial_i^{\varsigma_i},\partial^\varsigma,\nabla,\Delta,\nabla^k)=(\partial_{x_i}^{\varsigma_i}, \partial_x^\varsigma,\nabla_x,\Delta_x,\nabla_x^k)$.

\smallskip
\item If $\boldsymbol{f}: E\subset \mathbb{R}^l \to \mathbb{R}^q$ ($l,q\geq 2$, $E$ is a measurable set) is a vector function with the independent variable $\boldsymbol{y}=(y_1,\cdots\!,y_l)^\top$ and $X \in \{\partial_{y_i}^{\varsigma_i},\partial_y^\varsigma,\Delta_y,\nabla_y^k\}$, then 
\begin{align*}
&X\boldsymbol{f}=\big(Xf_1,\cdots\!,Xf_q\big)^\top, \ \quad \nabla_y\boldsymbol{f}=\begin{pmatrix} 
\partial_{y_1} f_1 & \partial_{y_2} f_1 & \cdots & \partial_{y_l} f_1\\[2mm]
\partial_{y_1} f_2 & \partial_{y_2} f_2 & \cdots & \partial_{y_l} f_2 \\[2mm]
\vdots & \vdots & \ddots & \vdots \\[2mm]
\partial_{y_1} f_q & \partial_{y_2} f_q & \cdots & \partial_{y_l} f_q
\end{pmatrix}_{q\times l},\\
&|\nabla_y^k \boldsymbol{f}|=\Big(\sum_{i=1}^q\sum_{|\varsigma|=k}\big|\partial_{y_1}^{\varsigma_1} \cdots\partial^{\varsigma_l}_{y_l}f_{i}\big|^2\Big)^\frac{1}{2} \quad \text{ for } k\in \mathbb{N}^*.
\end{align*}
Moreover, if $l=j+i$ with $j\geq 0$ and the independent variable $\boldsymbol{y}$ takes the form $\boldsymbol{y}=(\boldsymbol{s},\boldsymbol{\tilde{y}})^\top$ with $\boldsymbol{s}=(s_1,\cdots\!,s_j)^\top$ and $\boldsymbol{\tilde{y}}=(\tilde{y}_1,\cdots\!,\tilde{y}_i)^\top$, then
\begin{equation*}
\diver_{\tilde{y}} \boldsymbol{f}=\sum_{k=1}^i \partial_{\tilde{y}_k}f_k.  
\end{equation*}
In particular, if $j=0,1$, $\,i=n$, and $\boldsymbol{\tilde{y}}=\boldsymbol{x}=(x_1,\cdots\!,x_n)^\top\in \mathbb{R}^n$, then $\diver=\diver_x$.
\end{itemize}

\subsubsection{Notations in M-D spherical  coordinates}
In the rest of this paper, unless otherwise specified, we always let $r=|\boldsymbol{x}|$ be the radial distance in spherical coordinates, and let $r\in I:=[0,\infty)$. Moreover,  the following conventions are adapted:
\begin{align*}
|f|_p=\|f\|_{L^p(I)},\quad   \|f\|_k=\|f\|_{H^k(I)},\quad  \|f\|_{k,p}=\|f\|_{W^{k,p}(I)}
\qquad \,\,\mbox{for $k\in \mathbb{N}^*$}.
\end{align*}
In particular, for the Sobolev spaces defined on the open interval $(a,b)\subset I$, we use the abbreviation $X(a,b)=X((a,b))$ for $X=L^p$, $W^{k,p}$, and $H^k$.

\subsubsection{Other notations and conventions}
\begin{itemize}\smallskip
\item $C^\ell(\overline\Omega)$ $(\ell\in \mathbb{N},\,C(\overline\Omega)=C^0(\overline\Omega))$ 
denotes the space of all functions $f$ for which $\nabla^j f$ $(0\leq j\leq \ell)$ is bounded and uniformly continuous in $\Omega\subset \mathbb{R}^n$, which is equipped with the norm:
\begin{equation*}
\|f\|_{C^\ell(\overline{\Omega})}:= \max_{0\leq j\leq \ell}\|\nabla^j f\|_{L^\infty(\Omega)}.
\end{equation*}
In particular, if $\Omega=\mathbb{R}^n$, $C^\ell(\overline{\mathbb{R}^n})$ denotes the space of all functions $f$ for which $\nabla^j f$ $(0\leq j\leq \ell)$ is bounded and uniformly continuous in $\mathbb{R}^n$.

\smallskip
\item $C^\infty_{\rm c}(\Omega)$ denotes the space of all functions $f$ for which $\nabla^j f$ $(j\in \mathbb{N})$ is continuous and compactly supported in $\Omega\subset \mathbb{R}^n$.

\smallskip
\item For any function spaces $(X,X_1,\cdots\!,X_k)$ and functions $(h,g,g_1,\cdots\!,g_k)$,
\begin{equation*}
\|g\|_{X_1\cap\cdots\cap X_k}:=\sum_{i=1}^k\|g\|_{X_i},\qquad \|h(g_1,\cdots\!,g_k)\|_{X}:=\sum_{i=1}^k\|hg_i\|_X.
\end{equation*}

\smallskip
\item For any $n\times n$ real matrix $\mathcal{A}$, $\mathcal{A}_{ij}$ denotes its $(i,j)$-th entry. $\mathcal{A}:\mathcal{B}:=\sum_{i,j=1}^n \mathcal{A}_{ij}\mathcal{B}_{ij}$ for any $n\times n$ matrices $(\mathcal{A},\mathcal{B})$.  Moreover, $\mathrm{SO}(n)$  denotes the set of all $n\times n$ real orthogonal matrices $\mathcal{O}$ such that $\det \mathcal{O}=1$, where $\det \mathcal{O}$ is the determinant of $\mathcal{O}$.

\smallskip
\item $\delta_{ij}$  is the Kronecker symbol satisfying $\delta_{ij} =1$ when $i = j$, 
and $\delta_{ij} =0$ otherwise.

\smallskip
\item $\langle\cdot,\cdot\rangle_{X^*\times X}$ denotes the pairing between the space $X$ and its dual space $X^*$. 
\end{itemize}

\subsection{An enlarged reformulation}\label{see1}
 In order to deal with the degeneracy caused by the far-field vacuum, we now introduce an enlarged reformulation 
 for the Cauchy problem \eqref{eq:1.1benwen}--\eqref{e1.3}.
Specifically, by introducing the following variables: 
\begin{equation}\label{bianhuan}
\phi=\frac{A\gamma}{\gamma-1}\rho^{\gamma-1},\quad \boldsymbol{\psi} =\frac{1}{\gamma-1}\nabla\log\phi=\nabla\log\rho=(\psi_{1},\cdots\!,\psi_{n})^\top,
\end{equation}
then \textbf{CNS} \eqref{eq:1.1benwen}  can be rewritten as the following enlarged system: 
\begin{equation}\label{eqn1}
\begin{cases}
\phi_t+\boldsymbol{u}\cdot \nabla\phi+(\gamma-1)\phi \diver\boldsymbol{u}=0,\\[2pt]
\boldsymbol{u}_t+\boldsymbol{u}\cdot \nabla \boldsymbol{u}+\nabla \phi+L\boldsymbol{u}=\boldsymbol{\psi}\cdot Q(\boldsymbol{u}),\\[2pt]
\displaystyle\boldsymbol{\psi}_t+\nabla(\boldsymbol{u}\cdot\boldsymbol{\psi})+\nabla\diver\boldsymbol{u}=\boldsymbol{0},
\end{cases}
\end{equation}
where the operators $L$ and $Q$ are defined in \eqref{operatordefinition}. 

We aim to establish global spherically symmetric solutions of \eqref{eqn1} with the form:
\begin{equation}\label{1.9'}
(\phi,\boldsymbol{u},\boldsymbol{\psi})(t,\boldsymbol{x})=(\phi(t,|\boldsymbol{x}|),\,u(t,|\boldsymbol{x}|) \frac{\boldsymbol{x}}{|\boldsymbol{x}|},\,\psi(t,|\boldsymbol{x}|) \frac{\boldsymbol{x}}{|\boldsymbol{x}|} ),
\end{equation}
with the initial data:
\begin{equation}\label{eqs:CauchyInit'}
\begin{aligned}
(\phi,\boldsymbol{u},\boldsymbol{\psi})(0,\boldsymbol{x})
&=(\phi_0,\boldsymbol{u}_0,\boldsymbol{\psi}_0)(\boldsymbol{x})\\
&:=(\phi_0(|\boldsymbol{x}|),\,u_0(|\boldsymbol{x}|)\frac{\boldsymbol{x}}{|\boldsymbol{x}|},
\,\frac{1}{\gamma-1}(\log \phi_0)_r(|\boldsymbol{x}|)\frac{\boldsymbol{x}}{|\boldsymbol{x}|})  
\qquad \text{for $\boldsymbol{x} \in \mathbb{R}^n$},
\end{aligned}
\end{equation}
and the far-field asymptotic condition:
\begin{equation}\label{e1.3'}
\begin{split}
\displaystyle
\left(\phi,\boldsymbol{u}\right)(t,\boldsymbol{x})\to \left(0,\boldsymbol{0}\right) 
\qquad \text{as $\left|\boldsymbol{x}\right|\to \infty\,$ for $\,t\ge 0$}.
\end{split}
\end{equation}

Then we define the corresponding regular solutions of 
problem \eqref{eqn1}--\eqref{e1.3'}.

\begin{mydef}\label{def-enlarge}
Let $s=2$ or $3$, and $T>0$. The vector function $(\phi,\boldsymbol{u},\boldsymbol{\psi})$ 
is called a $s$-order regular solution of problem \eqref{eqn1}{\rm--}\eqref{e1.3'} 
in $[0,T]\times \mathbb{R}^n$ $(n=2$ or $3)$ if 
\begin{equation*}\begin{split}
\mathrm{(i)}& \ (\phi, \boldsymbol{u},\boldsymbol{\psi}) \ \text{satisfies this problem in the sense of distributions};\\
\mathrm{(ii)}& \ 0<\phi^\frac{1}{\gamma-1} \in C([0,T];L^{1}(\mathbb{R}^n)),\ \  \boldsymbol{\psi}\in L^\infty([0,T]\times\mathbb{R}^n),\ \ \nabla\phi\in C([0,T];H^{s-1}(\mathbb{R}^n)), \\
&\,\, \nabla \boldsymbol{\psi}\in C([0,T]; H^{s-2}(\mathbb{R}^n)), \ \ \phi_t \in C([0,T]; H^{s-1}(\mathbb{R}^n)),\ \ \boldsymbol{\psi}_t \in C([0,T]; H^{s-2}(\mathbb{R}^n));\\
\mathrm{(iii)}& \ \boldsymbol{u}\in C([0,T]; H^s(\mathbb{R}^n))\cap L^2([0,T]; D^{s+1}(\mathbb{R}^n)),\\
& \ \boldsymbol{u}_t\in C([0,T]; H^{s-2}(\mathbb{R}^n))\cap L^2([0,T]; D^{s-1}(\mathbb{R}^n)).
\end{split}
\end{equation*}
\end{mydef}

Since $\partial_i\psi_j=\partial_j\psi_i$ ($i,j=1,\cdots\!,n$), 
then equations $\eqref{eqn1}_3$ can be rewritten as
\begin{equation}\label{eqn1psi}
\boldsymbol{\psi}_t+\sum\limits_{l=1}^nA_l(\boldsymbol{u})\partial_l\boldsymbol{\psi}+B(\boldsymbol{u}) \boldsymbol{\psi}+\nabla\diver\boldsymbol{u}=\boldsymbol{0},
\end{equation}
where $A_l(\boldsymbol{u})=(a_{ij}^{(l)})_{n\times n}\,(i,j,l=1,\cdots\!,n)$ are symmetric with $a_{ij}^{(l)}=u_l$ when $i=j$ and $a_{ij}^{(l)}=0$ otherwise, 
and $B(\boldsymbol{u})=(\nabla \boldsymbol{u})^{\top}$.
Then the enlarged system \eqref{eqn1}   consists of
\begin{itemize}
 \item  a {\it scalar transport} equation $\eqref{eqn1}_1$ for $\phi$;
 \smallskip
 \item a {\it  parabolic}  system  $\eqref{eqn1}_2$  with the  weak singular source term $\boldsymbol{\psi}\cdot Q(\boldsymbol{u})$ for $\boldsymbol{u}$;
 \smallskip
 \item  a {\it symmetric hyperbolic} system  $\eqref{eqn1}_3$ with the   source term $\nabla\diver\boldsymbol{u}$  for $\boldsymbol{\psi}$.
\end{itemize}

The enlarged reformulation in \eqref{eqn1}  transfers the degeneracies both in the time evolution and spatial dissipation to the singularity of $\boldsymbol{\psi}$, which enables us to establish the following local well-posedness of regular solutions of problem   \eqref{eqn1}--\eqref{e1.3'} in Theorems \ref{thh1}--\ref{thh133}.

\begin{thm}\label{thh1}
Let $n=2$ or $3$, and  $\gamma$ satisfy  
\begin{equation}\label{cd1local}
\gamma>1.
\end{equation}
Assume that the initial data $(\phi_0,\boldsymbol{u}_0,\boldsymbol{\psi}_0)(\boldsymbol{x})$ 
are spherically symmetric and satisfy 
 \begin{equation}\label{eq;th2.1-2}
\begin{gathered}
0<\phi_0^{\frac{1}{\gamma-1}}(\boldsymbol{x})\in L^1(\mathbb{R}^n), \ \ \nabla \phi_0(\boldsymbol{x})\in H^1(\mathbb{R}^n), \ \ 
\boldsymbol{\psi}_0(\boldsymbol{x})\in  D^1(\mathbb{R}^n),\ \ 
\boldsymbol{u}_0(\boldsymbol{x})\in H^2(\mathbb{R}^n),
\end{gathered}
\end{equation} 
and, in addition, 
 \begin{equation}\label{eq;th2.1-21}
\begin{gathered}
 \boldsymbol{\psi}_0\in L^\infty(\mathbb{R}^3) \qquad \text{when $n=3$}.
\end{gathered}
\end{equation}
Then there exists $T_*>0$ such that problem \eqref{eqn1}{\rm--}\eqref{e1.3'} admits a unique $2$-order regular solution $(\phi,\boldsymbol{u},\boldsymbol{\psi})(t,\boldsymbol{x})$ in $[0,T_*]\times\mathbb{R}^n$ satisfying \eqref{er2} with $T$ replaced by $T_*$, and 
\begin{equation}\label{psi=log-phi-0}
\boldsymbol{\psi}(t,\boldsymbol{x})=\frac{1}{\gamma-1}\nabla \log\phi(t,\boldsymbol{x}) \qquad \text{for {\it a.e.} $(t,\boldsymbol{x})\in (0,T_*)\times \mathbb{R}^n$}.
\end{equation} 
Moreover, $(\phi,\boldsymbol{u},\boldsymbol{\psi})$ is spherically symmetric with form  \eqref{1.9'}.
\end{thm} 

\begin{thm}\label{thh133}
Let $n=2$ or $3$, and  \eqref{cd1local} hold. Assume that
the initial data $(\phi_0, \boldsymbol{u}_0,\boldsymbol{\psi}_0)(\boldsymbol{x})$ are spherically symmetric 
and  satisfy
\begin{equation}\label{th78qq33}
\begin{aligned}
&0<\phi_0^{\frac{1}{\gamma-1}}(\boldsymbol{x})\in L^1(\mathbb{R}^n), \quad  \nabla \phi_0(\boldsymbol{x})\in H^2(\mathbb{R}^n), \\ 
&\boldsymbol{u}_0(\boldsymbol{x})\in H^3(\mathbb{R}^n),\quad
\boldsymbol{\psi}_0(\boldsymbol{x})\in  D^1(\mathbb{R}^n)\cap D^2(\mathbb{R}^n).
\end{aligned}
\end{equation}
Then there exists  $T_*>0$ such that problem \eqref{eqn1}{\rm--}\eqref{e1.3'} admits a unique $3$-order regular solution $(\phi,\boldsymbol{u},\boldsymbol{\psi})(t,\boldsymbol{x})$ in $[0,T_*]\times\mathbb{R}^n$ satisfying \eqref{er2-high} with $T$ replaced by $T_*$, \eqref{psi=log-phi-0}, and 
\begin{equation}\label{phittx-psitt}
\phi_{tt}\in C([0,T_*];L^2(\mathbb{R}^n))\cap L^2([0,T_*];D^1(\mathbb{R}^n)), \quad \boldsymbol{\psi}_{tt}\in L^2([0,T_*];L^2(\mathbb{R}^n)).
\end{equation}
Moreover, $(\phi,\boldsymbol{u},\boldsymbol{\psi})$ is spherically symmetric with  form \eqref{1.9'}.
\end{thm}

The proofs for Theorems \ref{thh1}--\ref{thh133}  will be given in \S \ref{section-local-regular}.
Moreover, at the end of \S \ref{section-local-regular}, 
we will show that Theorems \ref{thh1}--\ref{thh133} indeed imply the local well-posedness theories 
of regular solutions of problem  \eqref{eq:1.1benwen}--\eqref{e1.3} 
with general smooth, spherically symmetric data 
and far-field vacuum ({\it i.e.},  $\bar{\rho}=0$), 
which are stated in Theorems \ref{zth1}--\ref{zth2} below.

\begin{thm}\label{zth1}
Let $n=2$ or $3$, and  \eqref{cd1local} hold.
If the initial data $(\rho_0,\boldsymbol{u}_0)(\boldsymbol{x})$ are spherically symmetric and satisfy \eqref{id1}{\rm--}\eqref{shangjie3}, then there exists $T_*>0$ such that problem  \eqref{eq:1.1benwen}{\rm--}\eqref{e1.3} admits a unique $2$-order regular solution $(\rho,\boldsymbol{u})(t,\boldsymbol{x})$ in $[0,T_*]\times\mathbb{R}^n$  satisfying \eqref{er2}{\rm--}\eqref{3dclassical1} with $T$ replaced by $T_*$. Moreover, $(\rho,\boldsymbol{u})$ is spherically symmetric  with  form \eqref{duichenxingshi}.
\end{thm}

\begin{thm}\label{zth2}
Let $n=2$ or $3$, and  \eqref{cd1local} hold. If the initial data $(\rho_0,\boldsymbol{u}_0)(\boldsymbol{x})$ are spherically symmetric, and satisfy  \eqref{id1-high}, then there exists $T_*>0$ such that 
problem \eqref{eq:1.1benwen}{\rm--}\eqref{e1.3} admits a unique  $3$-order regular solution $(\rho,\boldsymbol{u})(t,\boldsymbol{x})$ in $[0,T_*]\times\mathbb{R}^n$  satisfying \eqref{er2-high}{\rm--}\eqref{3dclassical2} with $T$ replaced by $T_*$, and 
\begin{equation*}
(\rho^{\gamma-1})_{tt}\in C([0,T_*];L^2(\mathbb{R}^n))\cap L^2([0,T_*];D^1(\mathbb{R}^n)), \quad (\nabla\log\rho)_{tt}\in L^2([0,T_*];L^2(\mathbb{R}^n)).
\end{equation*}
Moreover, $(\rho,\boldsymbol{u})$ is spherically symmetric with  form \eqref{duichenxingshi}.
\end{thm}

Furthermore, we point out that this enlarged reformulation in \eqref{eqn1} plays an important role 
in the global-in-time energy estimates for our solutions (see \S \ref{section-global2}--\S \ref{section-global3}).

\subsection{Main strategies for the case that $\bar \rho=0$}\label{subsection-strategy}
Now we briefly outline our main strategies to establish the global well-posedness theory in Theorems \ref{th1}--\ref{th1-high} when $\bar\rho=0$. To overcome the difficulties caused by the degeneracy in the far-field 
and the coordinate singularity at the origin, the main ingredients of our analysis consist of
\begin{enumerate}
\item[\S\ref{subsub-2.2.2}:] establishing some $r$-weighted $L^p(0,1)$-estimates of $\rho$
for $p\in [1,\infty]$, and $L^p(I)$-estimates of the effective velocity for some $p\in (n,\infty)$ 
in deriving  the global uniform upper bound of $\rho$ 
in $[0,T]\times I$ (see \S \ref{subsection-upper density});

\item[\S\ref{subsub-2.2.3}:] establishing the $L^p(I)$-estimates (without radial weights) of $\rho^\frac{1}{p}u$ for $p\in [2,\infty)$ and the global $L^\infty(I)$-estimate 
of the effective velocity (see \S \ref{section-effective}); 

\item[\S\ref{subsub-2.2.4}:] proving that cavitation does not form in $[0,T]\times \mathbb{R}^n$ for any finite time $T$ inside the fluids and establishing some  
pointwise lower bound estimate of $\rho$ (see \S \ref{section-nonformation});

\item[\S\ref{subsub-2.2.5}:] establishing the global uniform estimates for 
the regular solutions (see \S \ref{section-global2}--\S \ref{section-global3}). 
\end{enumerate}

Throughout the rest of \S \ref{subsection-strategy},                                                                                                                                                        $ C_0\in [1,\infty)$ denotes a generic constant depending only on $(\rho_0,u_0,n,\alpha,A,\gamma)$,
and $C(\nu_1,\cdots\!,\nu_k)\in [1,\infty)$ denotes a generic constant depending on $C_0$ and parameters $(\nu_1,\cdots\!,\nu_k)$, which may be different at each occurrence.

\subsubsection{Global-in-time uniform upper bound of the density}\label{subsub-2.2.2}
In order to extend the local solutions obtained in \S\ref{section-local-regular} to global ones, one needs to establish the corresponding global uniform estimates. 
To this end,  we consider system \eqref{cosingu} in spherical coordinates. The first major task for establishing the desired global estimates for the large-data problem is to obtain the global uniform upper bound of $\rho$. We divide this proof  into two steps: 
\begin{itemize}
 \item  the upper bound of $\rho$ in $B_1=\{\boldsymbol{x}:\, |\boldsymbol{x}| <1\}$ containing  the symmetric center;
 \smallskip
 \item the upper bound of $\rho$ in the domain  exterior to $B_1$, {\it i.e.},  $\mathbb{R}^n\backslash B_1$. 
\end{itemize}

In the exterior domain $\mathbb{R}^n\backslash B_1$, due to positive lower bound of the radial distance $r$, we can obtain from  the Sobolev embedding $H^1(1,\infty)\hookrightarrow L^\infty(1,\infty)$, 
the conservation of total mass, and  the BD entropy estimates 
(see Lemma \ref{energy-BD} in \S \ref{section-upper-density}) that
\begin{equation}\label{01}
\sup_{t\in[0,T]}\|\rho\|_{L^p(1,\infty)}\leq C_0  \qquad\,\, \text{for $p\in [1,\infty]$}.
\end{equation}

In $B_1$, some new ideas are required. First,  by taking the full advantage of the effective velocity $\boldsymbol{v}=\boldsymbol{u}+2\alpha\nabla\log\rho$ (its  radial projection is $v=u+2\alpha(\log\rho)_r$), 
the conservation of total mass and the Sobolev embeddings 
$W^{1,p}(\mathbb{R}^n)\hookrightarrow L^\infty(\mathbb{R}^n)\hookrightarrow L^\infty(B_1)$ for $p>n$, 
we observe that,  for some $p>n$,
\begin{equation}\label{000}
\begin{aligned}
&\|\rho^\frac{1}{p}\|_{L^\infty(B_1)}\leq C(p)\big\|\big(\rho^\frac{1}{p},\nabla(\rho^\frac{1}{p})\big)\big\|_{L^p}\leq C(p)\|\rho\|_{L^1}^\frac{1}{p}+ C(p)\big\|\rho^\frac{1}{p}(\boldsymbol{u},\boldsymbol{v})\big\|_{L^p},\\[2mm]
&\implies\,\,\sup_{t\in[0,T]} \|\rho\|_{L^\infty(0,1)} \leq  C(p) + C(p)\sup_{t\in[0,T]}\big|(r^m\rho)^\frac{1}{p}(u,v)\big|_{p}^p. 
\end{aligned}  
\end{equation}

Next, to establish the $L^p(I)$-estimates of $(r^m\rho)^\frac{1}{p}(u,v)$ for some $p>n$, 
we employ an idea involving the Hardy inequality 
(see Lemma \ref{hardy} in Appendix \ref{appA}). 
By treating the radial coordinate $r$ as a weight function defined on $[0,1)$, 
we are able to derive some $r$-weighted $L^p(0,1)$-estimates of $\rho$ 
based on the BD entropy estimates and the Hardy inequality:
\begin{equation}\label{02}
\sup_{t\in[0,T]}\|r^K\rho\|_{L^p(0,1)}\leq C(p,K)  \qquad \text{for $p\in [1,\infty]$},
\end{equation}
where $K>-\frac{1}{p}$ ($K>0$ for $p=\infty$) if $n=2$ and $K\geq 1-\frac{1}{p}$ if $n=3$. 
These weighted estimates make full use of the spherical symmetry, and are different from those obtained solely through the BD entropy estimate 
$\nabla\sqrt{\rho}\in L^\infty([0,T];L^2(\mathbb{R}^n))$ 
and the classical Sobolev embeddings of the type $H^1(\mathbb{R}^n)\hookrightarrow L^q(\mathbb{R}^n)$ in the M-D coordinates.

\smallskip
Now we briefly outline how to obtain the $L^p(I)$-estimates of $(r^m\rho)^\frac{1}{p} (u, v)$ for some $p>n$ 
via the weighted estimates \eqref{02}. We take the 3-D case as an example, since the 2-D case  can be treated similarly. 

First, multiply $\eqref{cosingu}_2$ by $r^2|u|^{p-2}u$ and integrate the resulting equality over $I$. 
Then, using \eqref{01} to handle the estimates in the exterior domain, we arrive at 
\begin{equation}\label{04} 
\begin{split}
&\frac{\mathrm{d}}{\mathrm{d}t}\big|(r^2\rho)^{\frac{1}{p}} u\big|_{p}^{p} +\big|(r^2\rho)^{\frac12} |u|^{\frac{p-2}{2}} u_r\big|_{2}^2 +\big|\rho^{\frac{1}{p}}u\big|_{p}^{p}\\
&\leq C(p)\big(1+\big|(r^2\rho)^{\frac{1}{p}} u\big|_{p}^{p}\big)+ C(p)\underline{\big\|r^\frac{p}{p\gamma-p+1}\rho\big\|_{L^{p\gamma-p+1}(0,1)}^{p\gamma-p+1}}_{:=\mathcal{J}_1}.
\end{split}
\end{equation}
Note that, if $\gamma\in (1,2]$, it follows directly from \eqref{02} 
that $\mathcal{J}_1\leq C(p)$. 
If $\gamma\in (2,3)$, we can first obtain from the Hardy inequality 
and $2\alpha\rho_r=\rho(v-u)$ that, for $\epsilon\in (0,1)$ and $j\in \mathbb{N}^*$,
\begin{equation}\label{055}
\begin{aligned}
\mathcal{J}_1&\leq C(p,j)+C(p,j,\epsilon)\underline{\int_0^1 r^{a_{j+1}} \rho^{b_{j+1}}\,\mathrm{d}r}_{:=\mathcal{J}_2}\\
&\quad +C(p,j)\epsilon\underline{\|r^\frac{2}{\gamma-1}\rho\|_{L^\infty(0,1)}^{\gamma-1}}_{:=\mathcal{J}_3} \big|\rho^{\frac{1}{p}}u\big|_{p}^{p} +C(p,j)\epsilon\underline{\big|(r^2\rho^\gamma)^{\frac{1}{p}}v\big|_{p}^{p}}_{:=\mathcal{J}_4},
\end{aligned}    
\end{equation}
where $(a_j,b_j)$ satisfy $(a_1,b_1)=(p,p\gamma-p+1)$ and, for $j\in \mathbb{N}^*$,
\begin{equation}
a_{j+1}=2(p-1)\Big(\frac{p}{p-1}\Big)^j-p+2,\quad b_{j+1}=(p-1)(\gamma-1)\Big(\frac{p}{p-1}\Big)^j+\gamma.
\end{equation}

For $\mathcal{J}_4$, we multiply  \eqref{eq:effective2} (the equation of $v$)
by $r^2|v|^{p-2}v$ and integrate the resulting equality over $I$. 
Employing \eqref{01} to handle the estimates in the exterior domain, we have
\begin{equation}\label{066}
\frac{\mathrm{d}}{\mathrm{d}t}\big|(r^2\rho)^{\frac{1}{p}}v\big|_{p}^{p}+\mathcal{J}_4
\leq  C(p)\Big(\mathcal{J}_3\,\big|\rho^{\frac{1}{p}}u\big|_{p}^{p}+\big|(r^2\rho)^{\frac{1}{p}}u\big|_{p}^{p}\Big).
\end{equation}
For $\mathcal{J}_2$--$\mathcal{J}_3$, since $\gamma\in (2,3)$, we can check that \eqref{02} is applicable to $\mathcal{J}_3$ and $\mathcal{J}_2$ for all $p\in [2,\infty)$ by fixing $j=j_0$ sufficiently large, so that
$\mathcal{J}_2+\mathcal{J}_3\leq C(p)$.
Therefore, collecting \eqref{04}--\eqref{066}, then choosing  $\epsilon$ sufficiently small, 
and using the Gr\"onwall inequality, we can obtain the desired $L^p(I)$-estimates 
of $(r^2\rho)^\frac{1}{p}(u,v)$ for some $p>3$ (see \S \ref{subsection-upper density}).

\subsubsection{Global-in-time $L^\infty(I)$-estimate of the effective velocity}\label{subsub-2.2.3}
Another major task for the global estimates is to obtain the global $L^\infty(I)$-estimate of the effective velocity $v=u+2\alpha(\log\rho)_r$. 
Once this is obtained, we can gain control over the derivative of density $\rho_r$ through $\rho(v-u)$, 
which then allows us to establish the required global uniform estimates.

First, applying the method of characteristics to $v$'s time evolution equation \eqref{eq:effective2} gives
\begin{align*}
v(t,\eta(t,r))&=v_0(r)\exp\Big(-\int_0^t \frac{A\gamma}{2\alpha} \rho^{\gamma-1}(\tau,\eta(\tau,r))\,\mathrm{d}\tau\Big)\\
&\quad+\frac{A\gamma}{2\alpha}\int_0^t (\rho^{\gamma-1}u)(s,\eta(s,r))\!\cdot\! 
\exp\Big(\!-\!\int_s^t\frac{A\gamma}{2\alpha}\rho^{\gamma-1}(\tau,\eta(\tau,r))\mathrm{d}\tau\Big)\mathrm{d}s,
\end{align*}
where $v_0=u_0+2\alpha(\log\rho_0)_r$, and $\eta:[0,T]\times I\to I$ denotes the flow mapping satisfying $\eta_t(t,r)=u(t,\eta(t,r))$ and $\eta(0,r)=r$. Taking the $L^\infty(I)$-norm above then yields
\begin{equation}\label{05}
\begin{aligned}
\sup_{t\in[0,T]}|v|_{\infty} 
\leq |v_0|_{\infty} +\frac{A\gamma}{2\alpha}\int_0^T |\rho^{\gamma-1}u|_{\infty} \,\mathrm{d}s.
\end{aligned}
\end{equation}
Thus, it suffices to control the $L^1([0,T];L^\infty(I))$-norm of $\rho^{\gamma-1}u$.
 
To achieve this, we establish some unweighed $L^p(I)$-estimates 
of $\rho^\frac{1}{p}u$. 
Take $n=3$ as an example. Unlike the standard $L^p(I)$-energy estimates of $u$ for $p\geq 2$, we  multiply $\eqref{cosingu}_2$ by $|u|^{p-2}u$ without radial weight $r^2$. However, this operation inevitably leads to the appearance of some undesired integral terms, 
such as $\mathcal{G}_3=-\frac{2}{p}\int_0^\infty \frac{\rho v |u|^p}{r}\,\mathrm{d}r$ 
in \eqref{R2-dim3}. Fortunately, $\mathcal{G}_3$ can be treated by the $L^\infty(I)$-norm of $v$ 
and the Young inequality as
\begin{align*}
\mathcal{G}_3&\leq C(\epsilon,p)|v|_{\infty}^2\big|\rho^\frac{1}{p} u\big|_{p}^p+\epsilon\big|(r^{-2}\rho)^\frac{1}{p}u\big|_{p}^p \qquad \text{for $\epsilon\in (0,1)$}.
\end{align*}
Based on this treatment, we can obtain the following type of estimates for $u$:
\begin{equation}\label{type0}
\sup_{t\in[0,T]}\big|\rho^\frac{1}{p}u\big|_{p}^p+\int_0^T\Big(\big|\rho^\frac{1}{2}|u|^\frac{p-2}{p}u_r\big|_{2}^2+\big|(r^{-2}\rho)^\frac{1}{p}u\big|_{p}^p\Big)\mathrm{d}t\leq  C(p,T)\Big(\big(\sup_{t\in[0,T]}|v|_{\infty}\big)^2+1\Big).
\end{equation}
Subsequently, by the Sobolev embedding $W^{1,1}(I)\hookrightarrow L^\infty(I)$, 
we can obtain from the above inequality with suitable fixed $p$ and  
Lemma \ref{lemma-uv-p} that, for any $\epsilon \in (0,1)$,
\begin{equation}\label{rhogamma1}
\begin{aligned}
\int_0^T|\rho^{\gamma-1}u|_{\infty}\mathrm{d}t&\leq C(T)\int_0^T\Big(1+|v|_{\infty} + \big|\rho^\frac{1}{2}|u|^\frac{p-2}{2}u_r\big|_{2}^\frac{2}{2p-1}\Big)\mathrm{d}t+C(T)\\
&\leq C(\epsilon,T)\Big(1+\int_0^T |v|_{\infty}\mathrm{d}t\Big)
+\epsilon\sup_{t\in[0,T]}|v|_{\infty}.
\end{aligned}
\end{equation}
Finally, substituting \eqref{rhogamma1} into \eqref{05} and choosing $\epsilon$ sufficiently small, 
we obtain the global $L^\infty(I)$-estimate of $v$ via the Gr\"onwall inequality. 
More details are provided in \S \ref{section-effective}.

\subsubsection{Non-formation of vacuum states inside the fluids  in finite time and the pointwise estimate of the density}\label{subsub-2.2.4}
Based on the global upper bound of $\rho$ and the $L^\infty(I)$-estimate of  $v$, 
we can show that the cavitation does not form inside the fluids in finite time, 
provided that no vacuum states are presented initially inside the fluids. 
Moreover, based on the specific decay rate of $\rho_0(r)$ in  $I$, 
we can derive the corresponding pointwise estimates for $\rho(t,r)$ in  $[0,T]\times I$.

Specifically, let $T>0$ be any finite time. 
We first multiply $\eqref{cosingu}_2$ by $\rho^{-1}u$ and integrate the resulting equality over $I$. 
Then it follows from the $L^\infty(I)$-estimates of $(\rho,v)$ that
\begin{equation}\label{07}
\sup_{t\in[0,T]}|u|_{2}^2+\int_0^T \Big|\big(u_r,\frac{u}{r}\big)\Big|_{2}^2\,\mathrm{d}t\leq C(T),
\end{equation}
which, together with  the  $L^\infty(I)$-estimate of $v=u+2\alpha(\log\rho)_r$, yields
\begin{equation}\label{08}
\sup_{t\in[0,T]}\|(\log\rho)_r\|_{L^2(0,1)}\leq C_0\sup_{t\in[0,T]}\big(|u|_{2}+|v|_{\infty}\big)\leq C(T).
\end{equation}
Next, multiplying $\eqref{cosingu}_1$ by $\rho^{-1}$ and integrating the resulting equality over $r\in[0,1]$, 
we then obtain from \eqref{07} and the $L^\infty(I)$-estimate of $v$ that 
\begin{equation}\label{09}
\sup_{t\in[0,T]}\|\log\rho\|_{L^2(0,1)}\leq C(T)\big(\|\log\rho_0\|_{L^2(0,1)}+1\big)\leq C(T), 
\end{equation}
which, along with \eqref{07}--\eqref{08} and the fact that $H^1(0,1) \hookrightarrow L^\infty(0,1)$,  
implies that 
\begin{equation}\label{10}
\sup_{t\in[0,T]}\|\log\rho\|_{L^\infty(0,1)}\leq C_0\sup_{t\in[0,T]}\big\|\big(\log\rho,(\log\rho)_r\big)\big\|_{L^2(0,1)}\leq C(T).
\end{equation}
This indicates  that $\rho$ admits a strictly positive lower bound in  $[0,T] \times [0,1]$.
The pointwise estimate of $\rho(t,r)$ on $[0,T]\times [0,R]$ $(R>0)$ can be derived similarly, 
which requires precisely analyzing the dependence of the constants in \eqref{08}--\eqref{10} on $R$ (see \S\ref{section-nonformation}).

\subsubsection{Global-in-time uniform energy estimates}\label{subsub-2.2.5}

Based on the $L^\infty(I)$-estimates of $(\rho, v)$, we consider the enlarged system \eqref{eqn1} in spherical coordinates, {\it i.e.}, $\eqref{e2.2}_1$--$\eqref{e2.2}_3$ in \S\ref{section-global2}, for establishing the desired global estimates
in \S\ref{section-global2}--\S\ref{section-global3}.

We take the $2$-order regular solutions in the 3-D case as an example. First, since any spherically symmetric vector field is curl-free, by employing the classical div-curl estimates and Lemma \ref{lemma-initial} in Appendix \ref{appb}, we establish the equivalences of the $W^{k,p}(\mathbb{R}^3)$-norms ($0\leq k\leq 3$ and $p\in (1,\infty)$) between the gradient $\nabla \boldsymbol{f}$ 
and the divergence $\diver \boldsymbol{f} = f_r + \frac{2}{r} f$ of the vector function $\boldsymbol{f}=\frac{\boldsymbol{x}}{r}f$, namely, Lemma \ref{im-1} in \S \ref{section-global2}. For example, according to Lemma \ref{im-1}, deriving the uniform $L^\infty([0,T];L^2(\mathbb{R}^3))$-estimate for $\nabla^3 \boldsymbol{u}$ is equivalent to proving
\begin{equation*}
\sup_{t\in[0,T]}\Big|\Big(r\big(u_r+\frac{2}{r}u\big)_{rr}, \big(u_r+\frac{2}{r}u\big)_{r}\Big)\Big|_{2}\leq C(T).
\end{equation*}
Thus, Lemma \ref{im-1}, combined with the structure of system $\eqref{e2.2}_1$--$\eqref{e2.2}_3$, 
subsequently assists us in establishing higher-order spatial derivative estimates for $(\phi, u, \psi)$.

The second point is to obtain the $L^2([0,T];L^\infty(I))$-estimate for $(u_r,\frac{u}{r})$, which is crucial in establishing the higher-order estimates of $u$. We first build a bridge between the $L^\infty(I)$-norm of $ru_r$ 
and the $L^2(I)$-norm of $ru_t$ via the Sobolev 
embedding $H^1(1,\infty)\hookrightarrow L^\infty(1,\infty)$, the  Hardy inequality,  Lemmas \ref{im-1} and \ref{lemma-initial},  $\eqref{e2.2}_2$, and some lower order estimates of $u$:
\begin{align*}
|ru_r|_{\infty}& \leq \|ru_r\|_{L^\infty(0,1)}+\|ru_r\|_{L^\infty(1,\infty)} \leq C_0\big(\big\|r^\frac{3}{2}(u_{r},u_{rr})\big\|_{L^2(0,1)}+\big\|r(u_{r},u_{rr})\big\|_{L^2(1,\infty)}\big)\\
& \leq C_0|(ru_{r},ru_{rr})|_{2} \leq C(T)\big(|ru_{t}|_{2}+1\big).
\end{align*}
Subsequently, the above inequality and some detailed analysis enable us to close the energy 
estimates for $u_t$, that is,
\begin{equation*}
\sup_{t\in[0,T]}|r u_t|_{2}^2+\int_0^T \Big|r\big(u_{tr}+\frac{2}{r}u_t\big)\Big|_{2}^2\,\mathrm{d}t\leq C(T).
\end{equation*}
This, along with Lemmas \ref{im-1} and \ref{lemma-initial}, $\eqref{e2.2}_2$, the Sobolev embedding $H^2(\mathbb{R}^3)\hookrightarrow L^{\infty}(\mathbb{R}^3)$, and the lower-order estimates of $u$, 
yields the $L^2([0,T];L^\infty(I))$-estimate for $(u_r,\frac{u}{r})$:
\begin{align*}
\int_0^T \Big|\big(u_r,\frac{u}{r}\big)\Big|_{\infty}^2 \mathrm{d}t&\leq C\int_0^T \|\nabla\boldsymbol{u}\|_{L^\infty(\mathbb{R}^3)}^2 \mathrm{d}t 
\leq C\int_0^T\Big|r\big(u_{tr}+\frac{2}{r}u_t\big)\Big|_{2}^2 \mathrm{d}t+C(T)\leq C(T).
\end{align*}

\subsection{Some comments for  the case that $\bar\rho>0$}\label{positive-application} 

We consider the case $n=3$ as an example and provide some brief comments on establishing the global boundedness of $(\rho, v)$.

\subsubsection{Global uniform upper bound of the density} 
We  divide the proof of the upper bound of $\rho$ into two parts: 
the estimates in $B_1=\{\boldsymbol{x}:\,|\boldsymbol{x}| <1\}$ 
and those in $\mathbb{R}^n\backslash B_1$. 
In fact, one can directly derive the uniform upper bound of $\rho$ 
in $\mathbb{R}^n\backslash B_1$ based on the fundamental energy estimate 
and the BD entropy estimate  (see Lemmas \ref{energy-bd-po}--\ref{Lemma-12.4} in \S\ref{subsub-1132}).

To establish the uniform upper bound of $\rho$ in $B_1$, 
by \eqref{000} in \S \ref{subsub-2.2.2}, it suffices to derive the $L^p(I)$-estimates of 
$(r^m\rho)^\frac{1}{p}(u,v)$ for some $p>3$. 
First, repeating the derivation of \eqref{04} and using the uniform upper bound of $\rho$ 
in  $\mathbb{R}^n\backslash B_1$ to handle the estimates in the exterior domain 
imply that
\begin{equation*} 
\frac{\mathrm{d}}{\mathrm{d}t}\big|(r^2\rho)^{\frac{1}{p}} u\big|_{p}^{p}\!+\! \big|(r^2\rho)^{\frac12} |u|^{\frac{p-2}{2}} u_r\big|_{2}^2 + \!\big|\rho^{\frac{1}{p}}u\big|_{p}^{p}\leq C(p)\big|(r^2\rho)^{\frac{1}{p-2}}u\big|_{p-2}^{p-2}+ C(p)\mathcal{J}_1,
\end{equation*}
 where $\mathcal{J}_1$ is defined in \eqref{04}.
Then, using the same arguments as in \eqref{055}--\eqref{066} and the Gr\"onwall inequality, we obtain 
\begin{equation}
\sup_{t\in [0,T]}\big|(r^2\rho)^{\frac{1}{p}}(u,v)\big|_{p}^{p}
\leq  C(p,T)\sup_{t\in [0,T]}\big|(r^2\rho)^{\frac{1}{p-2}}u\big|_{p-2}^{p-2}+C(p,T).
\end{equation}
Finally, taking $p=2N+2$ with $N\in \mathbb{N}^*$ in the above inequality and using Lemma \ref{energy-bd-po},
we recursively obtain the $L^{2N}(I)$-estimates for $(r^2\rho)^{\frac{1}{2N}}(u,v)$ and hence, via the interpolation, 
the $L^{p}(I)$-estimates for $(r^2\rho)^{\frac{1}{p}}(u,v)$ for $p>3$.

\subsubsection{Global $L^\infty(I)$-estimate of the effective velocity} 
The global $L^\infty(I)$-estimate of $v$ follows from the same argument as in \S \ref{subsub-2.2.3}, 
with $p\in [2,\infty)$ replaced by $p\in[4,\infty)$. 
This adjustment is due to the term: 
$\mathcal{I}_3=A(p-1)\int_0^\infty \rho^\gamma |u|^{p-2} u_r\,\mathrm{d}r$ in \eqref{503-po}, 
which arises in establishing estimate \eqref{type0}. 
In fact, $\mathcal{I}_3$ can be handled via the $L^\infty(I)$-estimate of $\rho$, 
and the H\"older and Young inequalities: for all $\epsilon\in(0,1)$,
\begin{equation}\label{230}
\begin{aligned}
\mathcal{I}_3&\leq  C(p)|\rho|_{\infty}^{\gamma-1} \big|(r^{-2}\rho)^\frac{1}{p}u\big|_{p}^\frac{p}{2} \big|(r^2\rho)^\frac{1}{2}|u|^\frac{p-4}{2}u_r\big|_2\\
&\leq C(\epsilon,p,T)\underline{\big|(r^2\rho)^\frac{1}{2}|u|^\frac{p-4}{2}u_r\big|_2^2}_{:=\mathcal{J}_5}+\epsilon\big|(r^{-2}\rho)^\frac{1}{p}u\big|_{p}^p.
\end{aligned}    
\end{equation}
Note that $\mathcal{J}_5$ can be treated via the $L^p(I)$-energy estimates 
of $u$ (Lemma \ref{lemma-uv-p-po} in \S\ref{subsub-1133}) only if $p\in [4,\infty)$. 
By \eqref{230}, we obtain  \eqref{type0} from the same argument as in \S \ref{subsub-2.2.3}. 

\subsubsection{Global uniform lower bound of the density}
Based on the $L^\infty(I)$-estimates of $(\rho,v)$, we can show
that $\rho$ admits a global uniform lower bound in $[0,T]\times I$. 
To achieve this, it suffices to establish the global $L^\infty(I)$-estimate for $\log(\rho/\bar\rho)$. 

First, we can obtain \eqref{07} in a similar way as in \S\ref{subsub-2.2.4}. 
Then we obtain from \eqref{07}, the $L^\infty(I)$-estimate of $v$, Lemma \ref{calculus}, 
and $2\alpha(\log\rho)_r=v-u$ that
\begin{equation}\label{17}
\big|\log(\rho/\bar\rho)\big|_\infty\leq C(T)\big(\big|\log(\rho/\bar\rho)\big|_2+1\big).
\end{equation}
Next, multiplying $\eqref{cosingu}_1$ by $\rho^{-1}\log(\rho/\bar\rho)$ 
and integrating the resulting equality over $I$, we obtain
from \eqref{07} and the $L^\infty(I)$-estimate of $v$ that
\begin{equation}\label{18} 
\big|\log(\rho/\bar\rho)\big|_2 \leq C(T)\big(\big|\log(\rho_0/\bar\rho)\big|_2+1\big)\leq C(T).
\end{equation}
Note that the $L^2(I)$-boundedness of $\log(\rho_0/\bar\rho)$ can be obtained 
by following the calculations \eqref{comment1}--\eqref{comment2} in the proof of 
Lemma \ref{lemma-lowerbound}. 
Finally, combining \eqref{17}--\eqref{18}, we derive 
the desired global $L^\infty(I)$-estimate of $\log(\rho/\bar\rho)$.

\section{Global Uniform Upper Bound of the Density}\label{section-upper-density}

The purpose of this section is to establish the global-in-time upper bound of the density
for the case that $\bar\rho=0$. 
In \S \ref{section-upper-density}--\S\ref{se46}, we denote  $C_0\in [1,\infty)$ a generic constant depending only 
on $(\rho_0,u_0,n,\alpha,A,\gamma)$, and $C(\nu_1,\cdots\!,\nu_k)\in [1,\infty)$ a generic  constant depending on $C_0$
and parameters $(\nu_1,\cdots\!,\nu_k)$, which may be different at each occurrence. Moreover, we emphasize that, throughout this paper, 
for any function space $X$ and functions $(h,g_1,\cdots\!,g_k)$,
\begin{equation*}
\|h(g_1,\cdots\!,g_k)\|_{X}:=\sum_{i=1}^k\|hg_i\|_X.
\end{equation*}

\subsection{Reformulation in the spherical coordinates}
Let $T>0$ be any fixed time, and let $(\rho,\boldsymbol{u})(t,\boldsymbol{x})$ be 
the unique $s$-order $(s=2,3)$ regular solution 
of the Cauchy problem \eqref{eq:1.1benwen}--\eqref{e1.3} in $[0,T]\times\mathbb{R}^n$ ($n=2,3$), 
which is spherically symmetric with the form:
\begin{equation}\label{e1.4}
\rho(t,\boldsymbol{x})=\rho(t,r), \quad \boldsymbol{u}(t,\boldsymbol{x})=u(t,r)\frac{\boldsymbol{x}}{r} \qquad\,\, \mbox{for $r=|\boldsymbol{x}|$}.
\end{equation}
Such local well-posedness results have been established in Theorems \ref{zth1}--\ref{zth2}.
Then the Cauchy problem \eqref{eq:1.1benwen}--\eqref{e1.3} can be transformed to the initial-boundary 
value problem of $(\rho,u)(t,r)$ in $(t,r)\in [0,T]\times I$:
\begin{equation}\label{e1.5}
\begin{cases}
\displaystyle 
\rho_t+(\rho u)_r+\frac{m\rho u}{r}=0,\\[3pt]
\displaystyle
\rho u_t+\rho uu_r+A(\rho^\gamma)_r=2\alpha\big(\rho u_r+\frac{m\rho u}{r} \big)_r-\frac{2\alpha m\rho_r u}{r},\\[4pt]
\displaystyle
(\rho,u)|_{t=0}=(\rho_0,u_0) \qquad\,\, \text{for $r\in I$},\\[4pt]
\displaystyle
u|_{r=0}=0 \qquad\qquad\qquad\,\,\,\, \text{for $t\in [0,T]$},\\[4pt]
\displaystyle
(\rho,u)\to \left(0,0\right)  \qquad\qquad\,\,\,\text{as $r\to \infty$ \ for $t\in [0,T]$}.
\end{cases}
\end{equation}
Here, based on relation  \eqref{e1.4}, we say that $(\rho, u)(t,r)$ is the $s$-order $(s=2,3)$ regular 
solution of problem \eqref{e1.5} in $[0,T]\times I$ if the vector function 
$(\rho,\boldsymbol{u})(t,\boldsymbol{x})$ is the  $s$-order $(s=2,3)$ regular solution (as defined in Definition \ref{cjk}) of the Cauchy problem \eqref{eq:1.1benwen}--\eqref{e1.3} 
with $\bar\rho=0$ in $[0,T]\times\mathbb{R}^n$ {\rm ($n=2,3$)}.

Next, by Lemma  \ref{lemma-initial} in Appendix B, one can transform the statements of Theorems \ref{zth1} and \ref{zth2} 
into the following conclusions in spherical coordinates.

\begin{thm}\label{rth1} Let \eqref{cd1local} hold. Assume the initial data $(\rho_0, u_0)(r)$ satisfy 
\begin{align*}
&r^m\rho_0\in L^1(I),\qquad r^{\frac{m}{2}}\big(\frac{(\log\rho_0)_r}{r},(\log\rho_0)_{rr}\big)\in L^2(I),\\
&r^{\frac{m}{2}}\big((\rho_0^{\gamma-1})_r, \frac{(\rho_0^{\gamma-1})_r}{r},(\rho_0^{\gamma-1})_{rr},u_0,\frac{u_0}{r},(u_0)_r,(\frac{u_0}{r})_r,(u_0)_{rr}\big)\in L^2(I),
\end{align*}
and, in addition, $(\log\rho_0)_r\in L^\infty(I)$ when $n=3$. Then there exists $T_*>0$ such that problem \eqref{e1.5} admits a unique $2$-order regular solution $(\rho, u)(t,r)$ in $[0,T_*]\times I$ that satisfies
\begin{equation}\label{spd}
\begin{aligned}
&r^m\rho \in C([0,T_*];L^1(I)), \quad\, (\log\rho)_r\in L^\infty([0,T_*]\times I),\\
&r^{\frac{m}{2}}\big((\rho^{\gamma-1})_r, \frac{(\rho^{\gamma-1})_r}{r},(\rho^{\gamma-1})_{rr},(\rho^{\gamma-1})_t,(\rho^{\gamma-1})_{tr}\big)\in C([0,T_*];L^2(I)),\\
& r^{\frac{m}{2}}\big(\frac{(\log \rho)_r}{r}, (\log \rho)_{rr},(\log \rho)_{tr},u,\frac{u}{r},u_r,(\frac{u}{r})_r,u_{rr},u_t\big)\in C([0,T_*];L^2(I)),\\
&r^{\frac{m}{2}}\big(\frac{1}{r}(\frac{u}{r})_r, (\frac{u}{r})_{rr},\frac{u_{rr}}{r},u_{rrr},\frac{u_t}{r},u_{tr},t^\frac{1}{2}u_{tt},t^\frac{1}{2}(\frac{u_t}{r})_r,t^\frac{1}{2}u_{trr}\big)\in L^2([0,T_*];L^2(I)),\\
&t^{\frac{1}{2}}r^{\frac{m}{2}}\big(\frac{1}{r}(\frac{u}{r})_r,(\frac{u}{r})_{rr},\frac{u_{rr}}{r},u_{rrr},\frac{u_t}{r},u_{tr}\big)\in L^\infty([0,T_*];L^2(I)),
\end{aligned}    
\end{equation}
and 
\begin{equation}\label{spd2}
\begin{aligned}
\big(\rho,\rho_r,u,\frac{u}{r},u_r\big)\in C((0,T_*];C(\bar I)).
\end{aligned}
\end{equation}
\end{thm}

\begin{thm}\label{rth133} Let \eqref{cd1local} hold. Assume the initial data $(\rho_0, u_0)(r)$ satisfy 
\begin{align*} 
&r^m\rho_0\in L^1(I),\qquad r^{\frac{m}{2}}\big(\frac{(\log\rho_0)_r}{r},(\log\rho_0)_{rr},(\frac{(\log\rho_0)_r}{r})_r,(\log\rho_0)_{rrr}\big)\in L^2(I),\notag\\
&r^{\frac{m}{2}}\big((\rho_0^{\gamma-1})_r, \frac{(\rho_0^{\gamma-1})_r}{r},(\rho_0^{\gamma-1})_{rr},(\frac{(\rho_0^{\gamma-1})_r}{r})_r,(\rho_0^{\gamma-1})_{rrr}\big)\in L^2(I),\\
&r^{\frac{m}{2}}\big(u_0,\frac{u_0}{r},(u_0)_r,(\frac{u_0}{r})_r,(u_0)_{rr},\frac{1}{r}(\frac{u_0}{r})_r,
(\frac{u_0}{r})_{rr},\frac{(u_0)_{rr}}{r},(u_0)_{rrr}\big)\in L^2(I).
\end{align*}
Then there exists $T_*>0$ such that problem \eqref{e1.5} admits a unique $3$-order regular solution $(\rho, u)(t,r)$ in $[0,T_*]\times I$ that satisfies 
\begin{align}
&r^m\rho \in C([0,T_*];L^1(I)),\qquad (\log\rho)_r\in L^\infty([0,T_*]\times I),\notag\\
&r^{\frac{m}{2}}\big((\rho^{\gamma-1})_r, \frac{(\rho^{\gamma-1})_r}{r},(\rho^{\gamma-1})_{rr},(\frac{(\rho^{\gamma-1})_r}{r})_r,(\rho^{\gamma-1})_{rrr}\big)\in C([0,T_*];L^2(I)),\notag\\
&r^{\frac{m}{2}}\big((\rho^{\gamma-1})_t,(\rho^{\gamma-1})_{tr},\frac{(\rho^{\gamma-1})_{tr}}{r},(\rho^{\gamma-1})_{trr}\big)\in C([0,T_*];L^2(I)),\notag\\
& r^{\frac{m}{2}}\big(\frac{(\log \rho)_r}{r}, (\log \rho)_{rr},(\frac{(\log\rho)_r}{r})_r,(\log\rho)_{rrr}\big)\in C([0,T_*];L^2(I)),\notag\\
& r^\frac{m}{2}\big((\log \rho)_{tr},\frac{(\log \rho)_{tr}}{r}, (\log \rho)_{trr}\big)\in C([0,T_*];L^2(I)),\notag\\
&r^\frac{m}{2}(\rho^{\gamma-1})_{tt}\in C([0,T_*];L^2(I)),\qquad r^\frac{m}{2}\big((\rho^{\gamma-1})_{ttr},(\log \rho)_{ttr} \big)\in L^2([0,T_*];L^2(I)),\label{spd33}\\
&r^{\frac{m}{2}}\big(u,\frac{u}{r},u_r,(\frac{u}{r})_r,u_{rr},\frac{1}{r}(\frac{u}{r})_r,(\frac{u}{r})_{rr},\frac{u_{rr}}{r},u_{rrr},u_t,\frac{u_t}{r},u_{tr}\big)\in C([0,T_*];L^2(I)),\notag\\
&r^{\frac{m}{2}}\big((\frac{1}{r}(\frac{u}{r})_r)_r,(\frac{u}{r})_{rrr},(\frac{u_{rr}}{r})_r,u_{rrrr},(\frac{u_t}{r})_r,u_{trr}\big)\in L^2([0,T_*];L^2(I)),\notag\\
&t^\frac{1}{2}r^{\frac{m}{2}}\big((\frac{1}{r}(\frac{u}{r})_r)_r,(\frac{u}{r})_{rrr},(\frac{u_{rr}}{r})_r,u_{rrrr},(\frac{u_t}{r})_r,u_{trr},u_{tt}\big)\in L^\infty([0,T_*];L^2(I)),\notag\\
&t^\frac{1}{2}r^{\frac{m}{2}}\big(\frac{u_{tt}}{r},u_{ttr},u_{trrr},\frac{u_{trr}}{r},(\frac{u_t}{r})_{rr},\frac{1}{r}\big(\frac{u_{tr}}{r}\big)_r\big)\in L^2([0,T_*];L^2(I)),\notag
\end{align}
and
\begin{equation}\label{spd233}
\big(\rho,\rho_r,\rho_t,u,\frac{u}{r},u_r\big)\in C([0,T_*];C(\bar I)) \ \ \text{and} \ \ \big(u_t,\big(\frac{u}{r}\big)_r,u_{rr}\big)\in C((0,T_*];C(\bar I)).
\end{equation}
\end{thm}

\smallskip
\subsection{Upper bound of the density}\label{subsection-upper density}

Let $T>0$ be any fixed time, and let $(\rho, u)(t,r)$ be the $s$-order $(s=2,3)$ regular solution of 
problem \eqref{e1.5} in $[0,T]\times I$ obtained in Theorems \ref{rth1}--\ref{rth133}. 
We define the characteristic functions $(\chi^\flat_\sigma,\chi^\sharp_\sigma)$ ($\sigma>0$) as
\begin{equation}\label{chi-sigma}
\chi^\flat_\sigma(r):=\begin{cases}
1&r\in[0,\sigma),\\
0&r\in[\sigma,\infty),
\end{cases}\qquad \chi^\sharp_\sigma:=1-\chi^\flat_\sigma.
\end{equation}

First, we give the definition of the so-called effective velocity.

\begin{mydef}\label{def-effective}
Let $(\rho,u,\alpha)$ be defined as in {\rm \S \ref{section-intro}}. 
Then 
\begin{equation}\label{V-expression}
v:=u+2\alpha (\log\rho)_r
\end{equation}
is said to be the effective velocity in spherical coordinates.
In addition, define $v_0:=v|_{t=0}=u_0+2\alpha (\log\rho_0)_r$.
\end{mydef} 

Then we have the following standard energy estimates and the BD entropy estimates:

\begin{lem}\label{energy-BD}
There exists a constant $C_0>0$ such that, for any $t\in [0,T]$,
\begin{align*} 
\int_0^\infty r^m\left(\rho u^2+\rho^\gamma\right)(t,\cdot)\mathrm{d}r+\int_0^t\int_0^\infty r^m\Big(\rho |u_r|^2+\rho\frac{u^2}{r^2}\Big)\,\mathrm{d}r\mathrm{d}s&\leq  C_0,\\
\int_0^\infty r^m\left(\rho v^2+|(\sqrt{\rho})_r|^2+\rho^\gamma\right)(t,\cdot)\mathrm{d}r+ \int_0^t\int_0^\infty r^m\rho^{\gamma-2}|\rho_r|^2\,\mathrm{d}r\mathrm{d}s &\leq  C_0.
\end{align*}
\end{lem}
\begin{proof} We divide the proof into two steps. 

\textbf{1.} First, multiplying $\eqref{e1.5}_2$ by $r^m u$, along with $\eqref{e1.5}_1$, yields
\begin{equation}\label{eap1}
\begin{split}
&\Big(\frac{r^m}{2} \rho u^2+\frac{A}{\gamma-1}r^m\rho^\gamma\Big)_t+ 2\alpha r^m\rho u_r^2+ 2\alpha mr^{m-2}\rho u^2\\
&=\Big(-\frac{A\gamma}{\gamma-1}r^m\rho^\gamma u+2\alpha r^m\rho u u_r-\frac{r^m}{2}\rho u^3\Big)_r.
\end{split}
\end{equation}
Integrating the above over $I$, we have 
\begin{equation}\label{eap2}
\begin{aligned}
&\frac{\mathrm{d}}{\mathrm{d}t}\int_0^\infty \Big(\frac{1}{2}r^m\rho u^2+\frac{A}{\gamma-1}r^m\rho^\gamma\Big) \,\mathrm{d}r+ 2\alpha \int_0^\infty r^m\Big(\rho |u_r|^2+ m\rho \frac{u^2}{r^2}\Big)\,\mathrm{d}r\\
&=\int_0^\infty \Big(\underline{-\frac{A\gamma}{\gamma-1}r^m\rho^{\gamma}u+2\alpha r^m\rho u u_r-\frac{r^m}{2}\rho u^3}_{:=\mathcal{B}_1}\Big)_r\,\mathrm{d}r=-\mathcal{B}_1|_{r=0}=0.
\end{aligned} 
\end{equation}
To verify this, we need to show  that $\mathcal{B}_1\in W^{1,1}(I)$ and $\mathcal{B}_1|_{r=0}=0$ 
for {\it a.e.} $t\in (0,T)$, which allows us to apply Lemma \ref{calculus} to obtain
\begin{equation}\label{in-B1}
\int_0^\infty (\mathcal{B}_1)_r\,\mathrm{d}r=-\mathcal{B}_1|_{r=0}=0.
\end{equation}
On one hand,  it follows from \eqref{spd}--\eqref{spd2} (or \eqref{spd33}--\eqref{spd233}) that 
\begin{equation*}
r^m\rho\in L^1(I),\quad  \big(\rho,\rho_r,u,\frac{u}{r},u_r\big)\in C(\bar I), \quad 
r^\frac{m}{2}\big((\rho^{\gamma-1})_r,u,\frac{u}{r},u_r,u_{rr}\big)\in L^2(I)
\end{equation*}
for {\it a.e.} $t\in (0,T)$, so that we can conclude that $\mathcal{B}_1|_{r=0}=0$.
On the other hand, it follows from the H\"older inequality 
that $\mathcal{B}_1\in W^{1,1}(I)$ for {\it a.e.} $t\in (0,T)$:
\begin{align*}
|\mathcal{B}_1|_1 &\leq  C_0\big|r^m(\rho^\gamma u,\rho u u_r,\rho u^3)\big|_1\\
&\leq C_0|r^m\rho|_1|\rho|_\infty^{\gamma-1}|u|_\infty+C_0|\rho|_\infty\big(|r^\frac{m}{2}u|_2|r^\frac{m}{2}u_r|_2+|u|_\infty|r^\frac{m}{2}u|_2^2\big)<\infty,\\
|(\mathcal{B}_1)_r|_1&\leq C_0\big|r^{m-1}(\rho^\gamma u,\rho u u_r,\rho u^3)\big|_1\\
&\quad\,+C_0\big|r^m\big(\rho^{\gamma-1}\rho_r u,\rho^\gamma u_r,\rho_r u u_r,\rho (u_r)^2,\rho u u_{rr},\rho_r u^3,\rho u^2u_r\big)\big|_1\\
&\leq C_0|r^m\rho|_1|\rho|_\infty^{\gamma-1}\Big|\frac{u}{r}\Big|_\infty+C_0|\rho|_\infty\Big(|r^\frac{m-2}{2}u|_2|r^\frac{m}{2}u_r|_2+\Big|\frac{u}{r}\Big|_\infty|r^\frac{m}{2}u|_2^2\Big)\\
&\quad\,+ C_0|\rho|_\infty \big|r^\frac{m}{2}(\rho^{\gamma-1})_r\big|_2|r^\frac{m}{2}u|_2+C_0|r^m\rho|_1|\rho|_\infty^{\gamma-1}|u_r|_\infty\\
&\quad\,+ C_0|\rho_r|_\infty|r^\frac{m}{2}u|_2|r^\frac{m}{2}u_r|_2+ C_0|\rho|_\infty\big(|r^\frac{m}{2}u_r|_2^2+|r^\frac{m}{2}u_r|_2|r^\frac{m}{2}u_{rr}|_2\big)\\
&\quad\,+ C_0|\rho_r|_\infty|u|_\infty|r^\frac{m}{2}u|_2^2+ C_0|\rho|_\infty|u_r|_\infty|r^\frac{m}{2}u|_2^2<\infty.
\end{align*}

Integrating \eqref{eap2} over $[0,t]$, we obtain the energy estimates, provided that 
$(r\rho_0)^{\frac{m}{2}}u_0\in L^2(I)$ and $r^\frac{m}{\gamma}\rho_0 \in L^\gamma(I)$. 
Indeed, it follows from Lemmas \ref{ale1}, \ref{initial3}, and \ref{lemma-initial} that
\begin{equation*}
\begin{aligned}
\big|(r^m\rho_0)^\frac{1}{2}u_0\big|_2&\leq  |r^m\rho_0|_1^\frac{1}{2}|u_0|_\infty\leq C_0\|\rho_0\|_{L^1}^\frac{1}{2}\|\boldsymbol{u}_0\|_{L^\infty}\leq C_0,\\
|r^\frac{m}{\gamma}\rho_0|_\gamma&\leq |r^m\rho_0|_1 |\rho_0|_\infty^{\gamma-1}\leq C_0\|\rho_0\|_{L^1} \|\rho_0\|_{L^\infty}^{\gamma-1}\leq C_0.
\end{aligned}
\end{equation*}

\smallskip
\textbf{2.} To obtain the BD entropy estimates, we can first apply $\partial_r$ to $\eqref{e1.5}_1$ to obtain
\begin{equation}\label{e5.2}
\rho((\log\rho)_{r})_t+\rho u(\log\rho)_{rr}+\Big(\rho \big(u_r+ \frac{m}{r} u\big)\Big)_r-\frac{m\rho_r u}{r}=0.
\end{equation}
Then combining with $\eqref{e1.5}_2$ and $\eqref{e5.2}$ leads to
\begin{equation}\label{ess}
\rho(v_t+u v_r)+A(\rho^\gamma)_r=0.
\end{equation}

Thus, multiplying \eqref{ess} by $r^m v$ and 
integrating the resulting equality over $[0,t]\times I$, we obtain from $\eqref{e1.5}_1$ that
\begin{equation}\label{es-v}
\begin{aligned}
&\int_0^\infty r^m\Big(\frac12\rho v^2+\frac{A}{\gamma-1}\rho^\gamma\Big)(t,\cdot)\,\mathrm{d}r+ 2A\alpha \gamma\int_0^t\int_0^\infty r^m\rho^{\gamma-2}|\rho_r|^2\,\mathrm{d}r\mathrm{d}s\\
&\leq \frac{1}{2}\big|(r\rho_0)^\frac{1}{2}v_0\big|_2^2+\frac{A}{\gamma-1}|r^\frac{m}{\gamma}\rho_0|_\gamma^\gamma\leq \frac{1}{2}\big|(r\rho_0)^\frac{1}{2}v_0\big|_2^2+C_0.
\end{aligned}
\end{equation}
For the $L^2(I)$-boundedness of $(r^m\rho_0)^{\frac{1}{2}}v_0$, it follows from Lemmas \ref{ale1}, \ref{initial3}, 
and \ref{lemma-initial} that
\begin{equation*}
\big|(r^m\rho_0)^\frac{1}{2}v_0\big|_2\leq |r^m\rho_0|_1^\frac{1}{2}|(u_0,(\log\rho_0)_r)|_\infty\leq C_0\|\rho_0\|_{L^1}^\frac{1}{2}\|(\boldsymbol{u}_0,\nabla\log\rho_0)\|_{L^\infty}\leq C_0.
\end{equation*}

Finally, \eqref{V-expression}, together with the energy estimates and \eqref{es-v}, yields
\begin{equation*}
|r^\frac{m}{2}(\sqrt{\rho})_r|_2=\frac{1}{2}\big|(r^m\rho)^\frac{1}{2}(\log\rho)_r\big|_2\leq C_0\big|(r^m\rho)^\frac{1}{2}(v,u)\big|_2\leq C_0.
\end{equation*}
The proof of Lemma \ref{energy-BD} is completed. 
\end{proof}

Clearly, by \eqref{V-expression} and \eqref{ess}, we have 
\begin{cor}\label{cor-v}
The effective velocity $v$ satisfies the following equation{\rm :}
\begin{equation}\label{eq:effective2}
v_t+uv_r+ \frac{A\gamma}{2\alpha} \rho^{\gamma-1} (v-u)=0.
\end{equation}
\end{cor}
 
Next, we show the $L^1(I)$ and $L^p(\omega,\infty)$-estimates ($p\in (1, \infty]$)  of $r^m\rho$ for $t\in [0,T]$ and $\omega>0$.

\begin{lem}\label{far-p-infty}
There exists a constant  $C_0>0$  such that
\begin{equation*}
|r^m \rho(t)|_1\leq C_0 \qquad\text{for any $t\in [0,T]$}.
\end{equation*}
Moreover, for any $\omega>0$, there exists a constant   $C(\omega)>0$ such that
\begin{equation*}
|\chi^\sharp_\omega r^m \rho(t)|_p\leq C(\omega) \qquad\text{for any $p\in (1, \infty]$ and $t\in [0,T]$}.
\end{equation*}

\end{lem}

\begin{proof}
First, integrating $\eqref{e1.5}_1$ over $I$, we obtain 
from $\eqref{e1.5}_4$, $r^m \rho u \in W^{1,1}(I)$, 
$r^m\rho u|_{r=0}=0$ due to $(\rho,u)\in C(\bar I)$ for $t\in (0,T]$, 
and Lemma \ref{calculus} that 
\begin{equation*} 
\frac{\mathrm{d}}{\mathrm{d}t}\int_0^\infty r^m\rho \,\mathrm{d}r=-\int_0^\infty (r^m \rho u)_r \,\mathrm{d}r=r^m \rho u|_{r=0}=0.
\end{equation*}
Integrating above over $[0,t]$, along with Lemma \ref{lemma-initial}, yields 
\begin{equation}\label{L1-rho}
|r^m\rho(t)|_1=|r^m\rho_0|_1\leq C_0 \qquad \text{for any $t\in [0,T]$}.
\end{equation}

Next, let $\omega>0$. It follows from \eqref{L1-rho}, Lemmas \ref{energy-BD} and \ref{calculus}, and the H\"older and Young inequalities that, for any $t\in [0,T]$,
\begin{equation}\label{yinli4dian2}
\begin{split}
|\chi^\sharp_\omega r^m\rho|_\infty&\leq |\chi^\sharp_\omega (r^m\rho)_r|_1\leq C_0\big(|\chi_\omega^\sharp r^m\sqrt{\rho} (\sqrt{\rho})_r|_1 + |\chi_\omega^\sharp r^{m-1}\rho|_1\big)\\
& \leq C_0\big(|r^m\rho|_1^\frac{1}{2}|r^\frac{m}{2}(\sqrt{\rho})_r|_2 +  |\chi_\omega^\sharp r^{-1}|_\infty |r^{m}\rho|_1\big) \leq C(\omega).
\end{split}
\end{equation}

Consequently, it follows from \eqref{L1-rho}--\eqref{yinli4dian2} that, 
for any $p\in(1,\infty)$ and $t\in [0,T]$,
\begin{equation*}
|\chi^\sharp_\omega r^m\rho|_p\leq |\chi^\sharp_\omega r^m\rho|_\infty^{1-\frac{1}{p}}|r^m\rho|_1^\frac{1}{p} \leq C(\omega),
\end{equation*}
where $C(\omega)\in [1,\infty)$ is  a generic  constant depending on $C_0$
and $\omega$, which may be different at each occurrence. This completes the proof.
\end{proof}

The next lemma concerns the weighted $L^p(0,1)$-estimates of $\rho$ for $p\in [1,\infty]$.

\begin{lem}\label{l4.3}
For any $t\in [0,T]$,
\begin{enumerate}
\item[$\mathrm{(i)}$] When $n=2$, there exist two positive constants $C(p,\nu)$ and $C(\nu) $ such that  
\begin{equation}\label{wprho2}
\begin{aligned}
\big|\chi_1^\flat r^{\nu}\rho(t)\big|_p\leq C(p,\nu) \qquad &\text{for any }\nu>-\frac{1}{p}\text{ and }p\in[1,\infty),\\
\big|\chi_1^\flat r^{\nu}\rho(t)\big|_\infty\leq C(\nu) \qquad &\text{for any }\nu>0;
\end{aligned}
\end{equation}
\item[$\mathrm{(ii)}$] When $n=3$, there exist two positive constants $C_0$ and $C(p) $ such that 
\begin{equation}\label{wprho3}
 \big|\chi_1^\flat r \rho(t)\big|_\infty\leq C_0, \qquad\,\, \big|\chi_1^\flat r^{1-\frac{1}{p}}\rho(t)\big|_p\leq C(p)  \quad\text{for any $p\in[1,\infty)$}.
\end{equation}
\end{enumerate}
\end{lem}
\begin{proof} We divide the proof into two steps.

\smallskip
\textbf{1.} {\bf Proof for the $2$-D case.}
First, we can obtain from Lemmas \ref{energy-BD}--\ref{far-p-infty}, 
and \ref{hardy} that, for any $\nu>0$ and $t\in [0,T]$,
\begin{equation}\label{gn1}
\begin{aligned}
|\chi_1^\flat r^\nu \rho|_\infty&=|\chi_1^\flat r^\frac{\nu}{2} \sqrt{\rho}|_\infty^2\leq C(\nu)\big(\big|\chi_1^\flat r^{\frac{\nu+1}{2}} \sqrt{\rho}\big|_2^2+\big|\chi_1^\flat r^{\frac{\nu+1}{2}} (\sqrt{\rho})_r\big|_2^2\big)\\
&\leq C(\nu)|\chi_1^\flat r^\nu|_\infty\big(|r \rho|_1+ \big| r^{\frac{1}{2}} (\sqrt{\rho})_r\big|_2^2\big) \leq C(\nu).
\end{aligned}
\end{equation}

Then, for all $p\in [1,\infty)$ and $\nu>-\frac{1}{p}$, 
letting $\varepsilon\in \big(0,\,\min\{p\nu+1,1\}\big)$ be any fixed constant, 
we obtain from \eqref{gn1} that,  for all $t\in [0,T]$,
\begin{equation}
\begin{aligned}
|\chi_1^\flat r^\nu \rho|_p^p&=\int_0^1 r^{p\nu-\varepsilon} r^{\varepsilon}\rho^p\,\mathrm{d}r \leq \Big(\int_0^1 r^{p\nu-\varepsilon}\,\mathrm{d}r\Big)\big|\chi_1^\flat r^\frac{\varepsilon}{p}\rho\big|_\infty^p\leq  C(p,\nu).
\end{aligned}
\end{equation}
The proof of (i) is completed.

\smallskip
\textbf{2.} {\bf Proof for the $3$-D case.}
Let $p\in[1,\infty)$. It follows from Lemmas \ref{energy-BD}--\ref{far-p-infty} and \ref{hardy} that, 
for any $t\in [0,T]$,
\begin{align*}
\big|\chi_1^\flat r^{1-\frac{1}{p}}\rho\big|_p &=\big|\chi_1^\flat r^{\frac{1}{2}-\frac{1}{2p}}\sqrt{\rho}\big|_{2p}^2\leq C(p)\big(|\chi_1^\flat r \sqrt{\rho}|_2^2+|\chi_1^\flat r\sqrt{\rho}_r|_2^2\big)\leq C(p),\\
\big|\chi_1^\flat r \rho\big|_\infty& =\big|\chi_1^\flat r^{\frac{1}{2}}\sqrt{\rho}\big|_{\infty}^2\leq C_0\big(|\chi_1^\flat r \sqrt{\rho}|_2^2+|\chi_1^\flat r\sqrt{\rho}_r|_2^2\big)\leq C_0.
\end{align*}
The proof of (ii) is completed.
\end{proof}

Based on Lemmas \ref{far-p-infty}--\ref{l4.3}, we can show 
the $L^{p}(I)$-estimates  of $(r^m\rho)^{\frac{1}{p}}u$ for $p\geq 2$.

\begin{lem} \label{lma}
Let $\gamma\in (1,\infty)$ if $n=2$ and $\gamma\in (1,3)$ if $n=3$. For any  $p\in[2,\infty)$ and $\epsilon\in (0,1)$, there exist two positive  constants $C(p)$ and $C(p,\epsilon) $ such that, 
for any $t\in [0,T]$,
\begin{equation}\label{dt-u-p}
\begin{aligned}
&\frac{\mathrm{d}}{\mathrm{d}t}\big|(r^m\rho)^{\frac{1}{p}} u\big|_{p}^{p} + p\alpha \Big(\big|(r^m\rho)^{\frac{1}{2}}|u|^{\frac{p-2}{2}} u_r\big|_2^2+\big|(r^{m-2}\rho)^{\frac{1}{p}}u\big|_{p}^{p}\Big)\\
&\leq C(p)\Big(1+\big|(r^m\rho)^\frac{1}{p}u\big|_{p}^{p}\Big)+ C(p,\epsilon)+\epsilon\big|(r^m\rho^\gamma)^{\frac{1}{p}} v\big|_{p}^{p}.
\end{aligned}
\end{equation}
\end{lem}
\begin{proof}We divide the proof into three steps.

\smallskip
\textbf{1.} Let $p\in [2,\infty)$. Multiplying $\eqref{e1.5}_2$ by $r^m|u|^{p-2}u$, along with $\eqref{e1.5}_1$, gives
\begin{equation}\label{eq:522}
\begin{aligned}
&\frac{1}{p}(r^m\rho |u|^{p})_t+2\alpha(p-1)r^m\rho|u|^{p-2}|u_r|^2+2\alpha mr^{m-2}\rho |u|^{p}\\
&=(p-1)Ar^m\rho^\gamma|u|^{p-2}u_r+m Ar^{m-1}\rho^\gamma|u|^{p-2}u\\
&\quad +\Big(\underline{2\alpha r^m\rho|u|^{p-2}u u_r-A r^m\rho^\gamma|u|^{p-2}u-\frac{1}{p}r^m\rho u |u|^{p}}_{:=\mathcal{B}_2}\Big)_r.
\end{aligned}
\end{equation}
Here we need to show  that $\mathcal{B}_2\in W^{1,1}(I)$ and $\mathcal{B}_2|_{r=0}=0$ 
for {\it a.e.} $t\in (0,T)$, which allows us to apply Lemma \ref{calculus} to obtain
\begin{equation}\label{eq:523}
\int_0^\infty (\mathcal{B}_2)_r\,\mathrm{d}r=-\mathcal{B}_2|_{r=0}=0.
\end{equation}
On one hand,  $\mathcal{B}_2|_{r=0}=0$ follows from the fact 
that $(\rho,u,u_r)\in C(\bar I)$ for each $t\in (0,T]$ 
due to \eqref{spd2} (or \eqref{spd233}). 
On the other hand,  it follows from \eqref{spd}--\eqref{spd2} 
(or \eqref{spd33}--\eqref{spd233}) that
\begin{equation*}
r^m\rho\in L^1(I),\quad \big(\rho,\rho_r,u,\frac{u}{r},u_r\big)\in L^\infty(I),
\quad  r^\frac{m}{2}\big((\rho^{\gamma-1})_r,u,u_r\big)\in L^2(I)
\end{equation*}
 for {\it a.e.} $t\in (0,T)$. 
 Thus, we obtain from the H\"older inequality that 
\begin{align*}
|\mathcal{B}_2|_1& \leq C_0\big|r^m\big(\rho |u|^{p-1}|u_r|,\rho^\gamma|u|^{p-1},\rho |u|^{p+1}\big)\big|_1\\
&\leq C_0|r^m \rho|_1 |u|_\infty^{p-1}|u_r|_\infty+C_0|r^m\rho|_1|\rho|_\infty^{\gamma-1}|u|_\infty^{p-1}+C_0|r^m\rho|_1 |u|_\infty^{p+1}<\infty,\\[1mm]
|(\mathcal{B}_2)_r|_1&\leq C_0\big|r^{m-1}\big(\rho |u|^{p-1}|u_r|,\rho^\gamma|u|^{p-1},\rho |u|^{p+1}\big)\big|_1\\
&\quad+ C(p)\big|r^m\big(\rho_r |u|^{p-1}u_r, \rho |u|^{p-2}u_r^2, \rho |u|^{p-1} u_{rr}\big)\big|_1\\
&\quad+C(p)\big|r^m\big(\rho^{\gamma-1}\rho_r |u|^{p-1}, \rho^\gamma |u|^{p-2}u_r, \rho_r |u|^{p+1}, \rho |u|^{p}u_r\big)\big|_1\\
&\leq C_0|r^m \rho|_1 |u|_\infty^{p-2}\Big|\frac{u}{r}\Big|_\infty\big(|u_r|_\infty+ |\rho|_\infty^{\gamma-1} +  |u|_\infty^2\big)\\
&\quad+ C(p)|u|_\infty^{p-2}\big(|\rho_r|_\infty|r^\frac{m}{2}u|_2 |r^\frac{m}{2}u_r|_2+ |\rho|_\infty |r^\frac{m}{2}u_r|_2^2+|\rho|_\infty|r^\frac{m}{2}u|_2|r^\frac{m}{2}u_{rr}|_2\big) \\
&\quad+C(p)\big(|\rho|_\infty|u|_\infty^{p-2} \big|r^\frac{m}{2}(\rho^{\gamma-1})_r\big|_2|r^\frac{m}{2}u|_2+|r^m\rho|_1|\rho|_\infty^{\gamma-1} |u|_\infty^{p-2}|u_r|_\infty\big) \\
&\quad +C(p)\big(|\rho_r|_\infty|u|^{p-1}_\infty|r^\frac{m}{2}u|_2^2+|\rho|_\infty|u|_\infty^{p-1}|r^\frac{m}{2} u|_2|r^\frac{m}{2} u_r|_2\big)<\infty.
\end{align*}

Then integrating \eqref{eq:522} over $I$, together with \eqref{eq:523}, leads to
\begin{equation}\label{add} 
\begin{split}
&\frac{1}{p} \frac{\mathrm{d}}{\mathrm{d}t}\big|(r^m\rho)^{\frac{1}{p}} u\big|_{p}^{p}  +2\alpha(p-1) \big|(r^m\rho)^{\frac12} |u|^{\frac{p-2}{2}} u_r\big|_2^2 +2\alpha m\big|(r^{m-2}\rho)^{\frac{1}{p}}u\big|_{p}^{p}\\
&= (p-1)A\int_0^\infty r^m \rho^{\gamma} |u|^{p-2} u_r\,\mathrm{d}r+mA\int_0^\infty r^{m-1}\rho^{\gamma} |u|^{p-2} u\,\mathrm{d}r:=\sum_{i=1}^2 \mathcal{G}_i.
\end{split}
\end{equation}

\smallskip
\textbf{2. Estimate of $\mathcal{G}_1$.}
For $\mathcal{G}_1$, it follows from Lemmas \ref{far-p-infty}--\ref{l4.3}, 
and the H\"older and Young inequalities that
\begin{equation}\label{J1}
\begin{aligned}
\mathcal{G}_1&\leq C(p)\big|(r^m\rho)^{\frac12}|u|^{\frac{p-2}{2}} u_r\big|_2\big|r^{\frac{m}{2}}\rho^{\gamma-\frac12}|u|^{\frac{p-2}{2}}\big|_2\\
&\leq \frac{\alpha}{8}\big|(r^m\rho)^{\frac12} u^{\frac{p-2}{2}} u_r\big|_2^2+C(p)\big|\chi_1^\flat r^{\frac{m}{2}}\rho^{\gamma-\frac12}|u|^{\frac{p-2}{2}}\big|_2^2+C(p)\big|\chi_1^\sharp r^{\frac{m}{2}}\rho^{\gamma-\frac12}|u|^{\frac{p-2}{2}}\big|_2^2\\
&\leq  \frac{\alpha}{8}\big|(r^m\rho)^{\frac12} u^{\frac{p-2}{2}} u_r\big|_2^2+C(p)\Big|\chi_1^\flat r^\frac{p+m-2}{p\gamma-p+1} \rho\Big|_{p\gamma-p+1}^\frac{2p\gamma-2p+2}{p} \big|(r^{m-2}\rho)^\frac{1}{p}u\big|_{p}^{p-2}\\
&\quad +C(p)|\chi_1^\sharp r^{2m(1-\gamma)}|_\infty |\chi_1^\sharp r^m\rho|_{p\gamma-p+1}^\frac{2p\gamma-2p+2}{p} \big|(r^m\rho)^\frac{1}{p}u\big|_{p}^{p-2}\\
&\leq \frac{\alpha}{8}\Big(\big|(r^m\rho)^{\frac12} |u|^{\frac{p-2}{2}} u_r\big|_2^2+\big|(r^{m-2}\rho)^{\frac1p}u\big|_p^p\Big)+C(p)\Big(1+\big|(r^m\rho)^\frac{1}{p}u\big|_{p}^{p}\Big)\\
&\quad +C(p)\underline{\big|\chi_1^\flat r^\frac{p+m-2}{p\gamma-p+1} \rho\big|_{p\gamma-p+1}^{p\gamma-p+1}}_{:=\mathcal{G}_{1,1}} .
\end{aligned}    
\end{equation}

\smallskip
\textbf{2.1. The estimate for  $\mathcal{G}_{1,1}$ when $n=2$ or $n=3$ and $\gamma\in (1,2]$.} Note that 
\begin{align*}
&\frac{p+m-2}{p\gamma-p+1}=\frac{p-1}{p\gamma-p+1}>0 && \text{if $n=2,\ m=1$, and $\gamma\in (1,\infty)$};\\
&\frac{p+m-2}{p\gamma-p+1}=\frac{p}{p\gamma-p+1}\geq 1-\frac{1}{p\gamma-p+1} 
&& \text{if $n=3,\, m=2$, and $\gamma\in (1,2]$}.
\end{align*}
Thus, we can simply use Lemma \ref{l4.3} to obtain
\begin{equation}\label{g1,1-1}
\mathcal{G}_{1,1}\leq C(p) \qquad \text{for all $t\in[0,T]$}.    
\end{equation}

\textbf{2.2. The estimate for $\mathcal{G}_{1,1}$ when $n=3$ and $\gamma\in (2,3)$.} Set 
\begin{equation*}
a_1:=p,\quad\,\, b_1:=p\gamma-p+1.
\end{equation*}
It follows from Lemma \ref{far-p-infty} (with $\omega=\frac{1}{2}$), integration by parts, and the H\"older and Young inequalities that, for any $\epsilon_1\in (0,1)$,
\begin{equation}\label{a1-b1}
\begin{aligned}
\mathcal{G}_{1,1}&=\int_0^1 r^{a_1}\rho^{b_1}\,\mathrm{d}r=\frac{1}{a_1+1}\rho^{b_1}(1)-\frac{b_1p}{\gamma(a_1+1)}\int_0^1 r^{a_1+1} \rho^{b_1-\frac{\gamma}{p}} (\rho^\frac{\gamma}{p})_r\,\mathrm{d}r\\
&\leq \frac{1}{a_1+1}\big|\chi_\frac{1}{2}^\sharp \rho\big|_\infty^{b_1}+\frac{b_1p}{\gamma(a_1+1)}\Big(\int_0^1 r^\frac{pa_1+p-2}{p-1} \rho^\frac{pb_1-\gamma}{p-1} \,\mathrm{d}r\Big)^\frac{p-1}{p}\big|\chi_1^\flat r^\frac{2}{p}(\rho^\frac{\gamma}{p})_r\big|_p\\
&\leq C(a_1,b_1)+\frac{C(p,a_1,b_1)}{\epsilon_1}\int_0^1 r^\frac{pa_1+p-2}{p-1} \rho^\frac{pb_1-\gamma}{p-1} \,\mathrm{d}r+\epsilon_1 \big|\chi_1^\flat r^\frac{2}{p}(\rho^\frac{\gamma}{p})_r\big|_p^p.
\end{aligned}
\end{equation}
For $j\in \mathbb{N}^*$, define the following two sequences $\{a_j\}_{j=1}^\infty$ and $\{b_j\}_{j=1}^\infty$ 
as
\begin{equation}\label{tongxiang-0}
a_{j+1}=\frac{p}{p-1}a_j+\frac{p-2}{p-1},\quad\,\, b_{j+1}=\frac{p}{p-1}b_j-\frac{\gamma}{p-1}.
\end{equation}
Then we can obtain from a calculation similar to \eqref{a1-b1} that, 
for all $\epsilon_j\in (0,1)$,
\begin{equation}\label{aj-bj}
\int_0^1 r^{a_j}\rho^{b_j}\,\mathrm{d}r \leq C(a_j,b_j)+\frac{C(p,a_j,b_j)}{\epsilon_j}\int_0^1 r^{a_{j+1}} \rho^{b_{j+1}} \,\mathrm{d}r+\epsilon_j \big|\chi_1^\flat r^\frac{2}{p}(\rho^\frac{\gamma}{p})_r\big|_p^p. 
\end{equation}
Define $(\epsilon_0,C(p,a_0,b_0))=(1,1)$. Then collecting \eqref{a1-b1}--\eqref{aj-bj} yields that, for $j\in \mathbb{N}^*$,
\begin{equation}\label{g11-po}
\begin{aligned}
\mathcal{G}_{1,1}
&\leq \sum_{k=1}^j\frac{\prod_{l=0}^{k-1}C(p,a_l,b_l) C(a_{k},b_k)}{\prod_{0=1}^{k-1}\epsilon_l}
+\prod_{l=1}^j \frac{C(p,a_l,b_l)}{\epsilon_l}\,\underline{\int_0^1 r^{a_{j+1}} \rho^{b_{j+1}} \,\mathrm{d}r}_{:=\tilde{\mathcal{G}}_{1,1}}\\
&\quad + \Big(\sum_{k=1}^j\frac{\prod_{l=0}^{k-1} C(p,a_l,b_l)}{\prod_{l=0}^{k-1}\epsilon_l} \epsilon_{k} \Big) \big|\chi_1^\flat r^\frac{2}{p}(\rho^\frac{\gamma}{p})_r\big|_p^p.    
\end{aligned}    
\end{equation}

For the estimate of $\tilde{\mathcal{G}}_{1,1}$, we can first obtain from \eqref{tongxiang-0} that, for $j\in \mathbb{N}^*$,
\begin{equation}\label{tongxiang}
a_{j+1}=2(p-1)\Big(\frac{p}{p-1}\Big)^{j}-p+2,\quad\,\, 
b_{j+1}=(p-1)(\gamma-1)\Big(\frac{p}{p-1}\Big)^{j}+\gamma.
\end{equation}
Note that, for each $p\in [2,\infty)$ and $\gamma\in (2,3)$, 
both $\{a_j\}_{j=1}^\infty$ and $\{b_j\}_{j=1}^\infty$ are strictly increasing  as $j\to\infty$. 
Moreover, for these $(p,\gamma)$, since
\begin{align*}
\frac{a_{j+1}}{b_{j+1}}\geq 1-\frac{1}{b_{j+1}}&\iff \frac{a_{j+1}}{b_{j+1}-1}\geq 1\\
&\iff  \frac{2\Big((p-1)\big(\frac{p}{p-1}\big)^{j}+1\Big)-p}{(\gamma-1)\Big((p-1)\big(\frac{p}{p-1}\big)^{j}+1\Big)}\geq 1\\
&\iff 3-\frac{p}{(p-1)\big(\frac{p}{p-1}\big)^{j}+1}\geq \gamma>2,
\end{align*}
we can choose $j=j_0$ to be sufficiently large such that 
$\frac{a_{j_0+1}}{b_{j_0+1}}\geq 1-\frac{1}{b_{j_0+1}}$. 
Thus, fixing this constant $j_0$ (depending only on $(p,\gamma)$), 
we obtain from Lemma \ref{l4.3-po} that
\begin{align*}
\tilde{\mathcal{G}}_{1,1}=\int_0^1 r^{a_{j_0+1}} \rho^{b_{j_0+1}} \,\mathrm{d}r\leq C(p).
\end{align*}
Substituting $\tilde{\mathcal{G}}_{1,1}$ into \eqref{g11-po} then yields
\begin{equation}\label{g11-po'}
\begin{aligned}
\mathcal{G}_{1,1}&\leq \sum_{k=1}^{j_0}\frac{\prod_{l=0}^{k-1}C(p,a_l,b_l) C(a_{k},b_k)}{\prod_{l=0}^{k-1}\epsilon_l}+\prod_{l=1}^{j_0} \frac{C(p,a_l,b_l)}{\epsilon_l}\,C(p)\\
&\quad + \Big(\sum_{k=1}^{j_0}\frac{\prod_{l=0}^{k-1} C(p,a_l,b_l)}{\prod_{l=0}^{k-1}\epsilon_l} \epsilon_{k} \Big) \big|\chi_1^\flat r^\frac{2}{p}(\rho^\frac{\gamma}{p})_r\big|_p^p.    
\end{aligned}    
\end{equation}
To further reduce the above inequality, for $\tilde\epsilon\in (0,1)$ sufficiently small, 
we set 
\begin{align*}
\epsilon_{k}&=\frac{\tilde\epsilon}{j_0}\,
\frac{\prod_{l=0}^{k-1}\epsilon_l}{\prod_{l=0}^{k-1} C(p,a_l,b_l)}=\frac{\tilde\epsilon^{2^{k-1}}}{j_0^{2^{k-1}}\prod_{l=0}^{k-1} C(p,a_l,b_l)^{2^{k-l-1}}} \qquad 
\text{for $1\leq k\leq j_0$}.
\end{align*}
Then it follows from  \eqref{V-expression}, \eqref{g11-po'}, 
and Lemma \ref{l4.3} that
\begin{equation}\label{g1,1-2}
\begin{aligned}
\mathcal{G}_{1,1}&\leq C(p,\tilde\epsilon) + \sum_{k=1}^{j_0} \frac{\tilde\epsilon}{j_0} \big|\chi_1^\flat r^\frac{2}{p}(\rho^\frac{\gamma}{p})_r\big|_p^p= C(p,\tilde\epsilon) + \tilde\epsilon \big|\chi_1^\flat r^\frac{2}{p}(\rho^\frac{\gamma}{p})_r\big|_p^p\\   
&\leq C(p,\tilde\epsilon)+C_0\tilde\epsilon\big(\big|(r^2\rho^{\gamma})^\frac{1}{p}v\big|_p^p+\big|\chi_1^\flat r^\frac{2}{\gamma-1}\rho\big|_\infty^{\gamma-1}|\rho^\frac{1}{p} u|_{p}^{p}\big)\\
&\leq C(p,\tilde\epsilon)+C_0\tilde\epsilon\big(\big|(r^m\rho^{\gamma})^\frac{1}{p}v\big|_p^p+\big|(r^{m-2}\rho)^\frac{1}{p} u\big|_{p}^{p}\big).
\end{aligned}    
\end{equation}
Thus, collecting \eqref{J1}--\eqref{g1,1-1} and \eqref{g1,1-2} leads to 
\begin{equation}\label{g1''}
\begin{aligned}
\mathcal{G}_1&\leq \frac{\alpha}{8}\Big(\big|(r^m\rho)^{\frac12} |u|^{\frac{p-2}{2}} u_r\big|_2^2+\big|(r^{m-2}\rho)^{\frac1p}u\big|_p^p\Big)+C(p)\big(1+\big|(r^m\rho)^\frac{1}{p}u\big|_{p}^{p}\big)\\
&\quad +C(p,\tilde\epsilon)+C(p)\tilde\epsilon \big|(r^m\rho^{\gamma})^\frac{1}{p}v\big|_p^p \qquad \text{for all $\tilde\epsilon\in (0,1)$}.
\end{aligned}    
\end{equation}

\smallskip
\textbf{3. Estimate of $\mathcal{G}_2$.} This can be derived similarly from Lemmas \ref{far-p-infty}--\ref{l4.3}, the bound of $\mathcal{G}_{1,1}$, and the H\"older and Young inequalities as follows: for any $\tilde\epsilon\in (0,1)$,
\begin{equation}\label{J2}
\begin{aligned}
\mathcal{G}_2&\leq C_0\big|\chi_1^\flat r^{\frac{m-1}{p-1}}\rho^\frac{\gamma}{p-1} u\big|_{p-1}^{p-1}+C_0\big|\chi_1^\sharp r^{\frac{m-1}{p-1}}\rho^\frac{\gamma}{p-1} u\big|_{p-1}^{p-1}\\
&\leq C_0\Big|\chi_1^\flat r^\frac{p+m-2}{p\gamma-p+1}\rho\Big|_{p\gamma-p+1}^\frac{p\gamma-p+1}{p}\big|(r^{m-2}\rho)^{\frac{1}{p}}u\big|_{p}^{p-1}\\
&\quad + C_0|\chi_1^\sharp r^{m(1-\gamma)-1}|_\infty |\chi_1^\sharp r^m\rho|_{p\gamma-p+1}^\frac{p\gamma-p+1}{p}\big|(r^m\rho)^{\frac{1}{p}}u\big|_{p}^{p-1}\\
&\leq \frac{\alpha}{8}\big|(r^{{m-2}}\rho)^{\frac{1}{p}}u\big|_{p}^{p}+C(p)\big(1+\big|(r^m\rho)^\frac{1}{p}u\big|_{p}^{p}\big)+C(p)\mathcal{G}_{1,1}\\
&\leq \frac{\alpha}{4}\big|(r^{{m-2}}\rho)^{\frac{1}{p}}u\big|_{p}^{p}+C(p)\big(1+\big|(r^m\rho)^\frac{1}{p}u\big|_{p}^{p}\big)+C(p,\tilde\epsilon)+C(p)\tilde\epsilon \big|(r^m\rho^{\gamma})^\frac{1}{p}v\big|_p^p.
\end{aligned}    
\end{equation}

Setting $\tilde\epsilon=C(p)^{-1}\epsilon$ in \eqref{g1''}--\eqref{J2} for sufficiently small $\epsilon\in (0,1)$ and then substituting \eqref{g1''}--\eqref{J2} into \eqref{add} lead to the desired conclusion.
 
The proof of Lemma \ref{lma} is completed.
\end{proof}

In addition, we can show the following $L^p(I)$-estimates of $(r^m\rho)^{\frac{1}{p}}v$.
\begin{lem}\label{lem-v-lp}
Let $\gamma\in(1,\infty)$ if $n=2$ and $\gamma\in (1,3)$ if $n=3$. For any  $p\in[2,\infty)$, there exists a constant $C(p)>0$  such that,  
for any $t\in [0,T]$,
\begin{equation}\label{dt-v-p}
\frac{\mathrm{d}}{\mathrm{d}t}\big|(r^m\rho)^\frac{1}{p} v\big|_{p}^{p}+\frac{A\gamma p}{4\alpha}\big|(r^m\rho^\gamma)^\frac{1}{p} v\big|_{p}^{p}\leq C(p)\Big(\big|(r^m\rho)^\frac{1}{p} u\big|_{p}^{p}+\big|(r^{m-2}\rho)^\frac{1}{p} u\big|_{p}^{p}\Big).
\end{equation}
\end{lem}
\begin{proof}
First, multiplying \eqref{eq:effective2} by $r^m\rho|v|^{p-2}v$ with $p\in[2,\infty)$, 
along with $\eqref{e1.5}_1$, gives
\begin{equation}\label{eq:532}
\frac{1}{p}\big(r^m\rho |v|^p\big)_t+ \frac{1}{p}\big(\underline{r^m\rho u|v|^p}_{:=\mathcal{B}_3}\big)_r+\frac{A\gamma}{2\alpha} r^m\rho^{\gamma} |v|^p=\frac{A\gamma}{2\alpha}r^m\rho^\gamma uv|v|^{p-2}.
\end{equation}

Here we need to show that $\mathcal{B}_3\in W^{1,1}(I)$ and $\mathcal{B}_3|_{r=0}=0$ for {\it a.e.} $t\in (0,T)$, which allows us to apply Lemma \ref{calculus} to obtain
\begin{equation}\label{eq:B3}
\int_0^\infty (\mathcal{B}_3)_r\,\mathrm{d}r=-\mathcal{B}_3|_{r=0}=0.
\end{equation}
On the one hand,  we can obtain $\mathcal{B}_3|_{r=0}=0$ 
from \eqref{V-expression} and the fact that $(\rho,\rho_r,u)\in L^\infty(I)$ 
for {\it a.e.} $t\in (0,T)$ due to \eqref{spd}--\eqref{spd2} 
or \eqref{spd33}--\eqref{spd233}. 
On the other hand,  
based on \eqref{spd}--\eqref{spd2} (or \eqref{spd33}--\eqref{spd233}), we have 
\begin{equation*}
r^m\rho\in L^1(I),\quad r^\frac{m}{2}\big(\rho_r,u,(\log\rho)_{rr}\big)\in L^2(I),
\quad \Big(u,\frac{u}{r},u_r,(\log\rho)_r\Big)\in L^\infty(I)
\end{equation*} 
for {\it a.e.} $t\in (0,T)$. Thus, we obtain from the H\"older inequality that 
\begin{align*}
|\mathcal{B}_3|_1&\leq C(p)\big(\big|r^m\rho |u|^{p+1}\big|_1+\big|r^m\rho u|(\log\rho)_r|^p\big|_1\big)\\
&\leq C(p)\big(|r^m\rho|_1 |u|_\infty^{p+1}+ |r^m\rho|_1|u|_\infty|(\log\rho)_r|_\infty^p\big)<\infty,\\[1mm]
|(\mathcal{B}_3)_r|_1&\leq C(p)\big(\big|r^{m-1}\rho |u|^{p+1}\big|_1+\big|r^{m-1}\rho u|(\log\rho)_r|^p\big|_1+\big|r^{m}\rho_r |u|^{p+1}\big|_1\big)\\
&\quad+C(p)\big(\big|r^{m}\rho_r u|(\log\rho)_r|^p\big|_1+\big|r^{m}\rho |u|^pu_r\big|_1+\big|r^{m}\rho u_r|(\log\rho)_r|^p\big|_1\big)\\
&\quad +C(p)\big|r^{m}\rho u|(\log\rho)_r|^{p-1}|(\log\rho)_{rr}|\big|_1\\
&\leq C(p)|r^m\rho|_1 \Big|\frac{u}{r}\Big|_\infty\big( |u|_\infty^{p}+ |(\log\rho)_r|_\infty^p\big)+C(p)|r^\frac{m}{2}\rho_r|_2|r^\frac{m}{2}u|_2|u|_\infty^p\\
&\quad+C(p)|r^\frac{m}{2}\rho_r|_2|r^\frac{m}{2}u|_2|(\log\rho)_r|_\infty^p+C(p)|r^m\rho|_1|u_r|_\infty\big(|u|_\infty^p+|(\log\rho)_r|_\infty^p\big) \\
&\quad +C(p)|\rho|_\infty|r^\frac{m}{2}u|_2|(\log\rho)_r|_\infty^{p-1}|r^\frac{m}{2}(\log\rho)_{rr}|_2<\infty.
\end{align*}

Integrating \eqref{eq:532} over $I$, we obtain from \eqref{eq:B3}, 
$\frac{2}{\gamma-1}>1$ whenever $\gamma\in (1,3)$,
Lemmas \ref{far-p-infty}--\ref{lma}, and the Young inequality that
\begin{align*}
&\frac{1}{p}\frac{\mathrm{d}}{\mathrm{d}t}\big|(r^m\rho)^\frac{1}{p} v\big|_{p}^{p}+\frac{A\gamma}{2\alpha}\big|(r^m\rho^\gamma)^\frac{1}{p} v\big|_{p}^{p}\\
&\leq \frac{A\gamma}{2\alpha}\int_0^\infty r^m\rho^\gamma uv|v|^{p-2}\,\mathrm{d}r\\
&\leq \frac{A\gamma}{4\alpha}\big|(r^m\rho^\gamma)^\frac{1}{p} v\big|_{p}^p+C(p)\big|(r^m\rho^\gamma)^\frac{1}{p} u\big|_{p}^p\\
&\leq \frac{A\gamma}{4\alpha}\big|(r^m\rho^\gamma)^\frac{1}{p} v\big|_{p}^{p}+C(p)\Big(\big|\chi_1^\flat(r^m\rho^\gamma)^\frac{1}{p} u\big|_{p}^p+\big|\chi_1^\sharp(r^m\rho^\gamma)^\frac{1}{p} u\big|_{p}^p\Big)\\
&\leq \frac{A\gamma}{4\alpha}\big|(r^m\rho^\gamma)^\frac{1}{p} v\big|_{p}^{p}+C(p)\big|\chi_1^\flat r^\frac{2}{\gamma-1}\rho\big|_\infty^{\gamma-1} \big|(r^{m-2}\rho)^\frac{1}{p} u\big|_{p}^{p}\\
&\quad\, +C(p)|\chi_1^\sharp r^{-m}|_\infty^{\gamma-1} |\chi_1^\sharp r^m\rho|_\infty^{\gamma-1} \big|(r^m\rho)^\frac{1}{p} u\big|_{p}^{p}\\
&\leq \frac{A\gamma}{4\alpha}\big|(r^m\rho^\gamma)^\frac{1}{p} v\big|_{p}^{p} +C(p)\Big(\big|(r^m\rho)^\frac{1}{p} u\big|_{p}^{p}+\big|(r^{m-2}\rho)^\frac{1}{p} u\big|_{p}^{p}\Big).
\end{align*}

The proof of Lemma \ref{lem-v-lp} is completed.
\end{proof}

Consequently, based on Lemmas \ref{lma}--\ref{lem-v-lp}, 
we can derive the following important estimates on $(u,v)$.

\begin{lem}\label{lemma-uv-p}
Let $\gamma\in(1,\infty)$ if $n=2$ and $\gamma\in (1,3)$ if $n=3$. 
Then, for any  $p\in[2,\infty)$, there exists a constant $C(p, T)>0$ such that, 
for any $t\in [0,T]$,
\begin{align*}
&\big|(r^m\rho)^{\frac{1}{p}}u(t)\big|_{p}^{p}+\big|(r^m\rho)^{\frac{1}{p}}v(t)\big|_{p}^{p}\\
&+ \int_0^t\Big(\big|(r^m\rho)^{\frac{1}{2}}|u|^{\frac{p-2}{2}} u_r\big|_2^2+\big|(r^{m-2}\rho)^{\frac{1}{p}}u\big|_{p}^{p}+\big|(r^m\rho^\gamma)^\frac{1}{p} v\big|_{p}^{p}\Big)\,\mathrm{d}s\leq C(p,T).
\end{align*}
\end{lem}

\begin{proof}
First, multiplying \eqref{dt-v-p} by $\frac{8\alpha\epsilon}{A\gamma p}$ 
with $\epsilon\in (0,1)$ leads to
\begin{equation*} 
\frac{8\alpha\epsilon}{A\gamma p}\frac{\mathrm{d}}{\mathrm{d}t}\big|(r^m\rho)^\frac{1}{p} v\big|_{p}^{p}+2\epsilon\big|(r^m\rho^\gamma)^\frac{1}{p} v\big|_{p}^{p}\leq C(p)\epsilon\Big(\big|(r^m\rho)^\frac{1}{p} u\big|_{p}^{p}+\big|(r^{m-2}\rho)^\frac{1}{p} u\big|_{p}^{p}\Big).
\end{equation*}
Then summing the above inequality with \eqref{dt-u-p} yields 
\begin{equation*}
\begin{aligned}
&\frac{\mathrm{d}}{\mathrm{d}t}\Big(\big|(r^m\rho)^{\frac{1}{p}} u\big|_{p}^{p}+\frac{8\alpha\epsilon}{A\gamma p} \big|(r^m\rho)^\frac{1}{p} v\big|_{p}^{p}\Big)\\
&+ p\alpha \Big(\big|(r^m\rho)^{\frac{1}{2}}|u|^{\frac{p-2}{2}} u_r\big|_2^2+\big|(r^{m-2}\rho)^{\frac{1}{p}}u\big|_{p}^{p}\Big)+\epsilon\big|(r^m\rho^\gamma)^\frac{1}{p} v\big|_{p}^{p}\\
&\leq C(p)\Big(1+\big|(r^m\rho)^\frac{1}{p}u\big|_{p}^{p}\Big)+ C(p,\epsilon)+C(p)\epsilon\big|(r^{m-2}\rho)^\frac{1}{p} u\big|_{p}^{p}.
\end{aligned}
\end{equation*}

As a consequence, we can set 
\begin{equation*}
\epsilon=\min\big\{\frac{\alpha}{C(p)},\frac{1}{2}\big\},
\end{equation*}
and then apply the Gr\"onwall inequality to the resulting inequality to obtain
the desired result. 
To achieve this, it still requires checking the $L^p(I)$-boundedness 
of $(r^m\rho_0)^{\frac{1}{p}}(u_0,v_0)$. 
Indeed, it follows from Lemmas \ref{ale1}, \ref{initial3}, and \ref{lemma-initial} that
\begin{align*}
\big|(r^m\rho_0)^\frac{1}{p}(u_0,v_0)\big|_p&\leq |r^m\rho_0|_1^\frac{1}{p}|(u_0,(\log\rho_0)_r)|_{\infty}\leq C_0\|\rho_0\|_{L^1}^\frac{1}{p}\|(\boldsymbol{u}_0,\nabla\log\rho_0)\|_{L^\infty} \leq C(p).
\end{align*}
This completes the proof.
\end{proof}

We now show the uniform $L^\infty(I)$-estimate of the density.
\begin{lem}\label{important2}
Let $\gamma\in (1,\infty)$ if $n=2$ and $\gamma\in (1,3)$ if $n=3$. Then there exists a constant $C(T)>0$  such that 
\begin{equation*}
|\rho(t)|_\infty\leq C(T)  \qquad \text{for any $t\in [0,T]$}.
\end{equation*}
\end{lem}
\begin{proof}
Based on Lemma \ref{far-p-infty}, it suffices to show the $L^\infty(0,1)$-estimate of $\rho$:
\begin{equation}\label{near}
|\chi_1^\flat \rho(t)|_\infty \leq C(T) \qquad\text{for all $t\in [0,T]$}.
\end{equation}

It follows from Lemma \ref{lemma-uv-p} that, for $n=2,3$, there exists a constant $p_0>n$ such that 
\begin{equation}\label{p0-uv}
\big|(r^m\rho)^\frac{1}{p_0}(u,v)(t)\big|_{p_0}\leq C(T) \qquad
\text{for all $t\in [0,T]$}.
\end{equation}
Thus, we obtain from Lemma \ref{l4.3}, \eqref{V-expression}, \eqref{p0-uv}, 
$\frac{m}{p_0-1}\in (0,1)$, and the H\"older and Young inequalities that, 
for all $t\in [0,T]$,
\begin{align*}
|\chi_1^\flat \rho|_\infty&\leq  C_0\big(|\chi_1^\flat \rho|_1+|\chi_1^\flat \rho_r|_1\big) \leq  C_0\big(1+|\chi_1^\flat\rho (u,v)|_1\big)\\
&\leq  C_0+C_0\big|\chi_1^\flat r^{-\frac{m}{p_0-1}}\big|_1^\frac{p_0-1}{p_0}|\chi_1^\flat\rho|_\infty^\frac{p_0-1}{p_0} \big|(r^m\rho)^\frac{1}{p_0}(u,v)\big|_{p_0}\leq C(T)+\frac{1}{2}|\chi_1^\flat \rho|_\infty,
\end{align*}
which implies \eqref{near}.
\end{proof}

\section{Global Uniform Bound of the Effective Velocity}\label{section-effective}

This section is devoted to establishing the global uniform $L^\infty(I)$-estimate of the effective velocity $v$ in spherical coordinates for the case that $\bar \rho=0$. 
Let $T>0$ be any fixed time, and let $(\rho, u)(t,r)$ be the $s$-order $(s=2,3)$ 
regular solution of problem \eqref{e1.5} in $[0,T]\times I$ 
obtained in Theorems \ref{rth1}--\ref{rth133}. 
Moreover, throughout \S \ref{section-effective}--\S \ref{se46}, we always assume that 
\begin{equation}\label{gammafanwei}
\gamma\in (1,\infty)\,\,\,\,\text{when $n=2$},\qquad\,\, 
\gamma\in (1,3)\,\,\,\,\text{when $n=3$}.
\end{equation}
To obtain the  boundedness of $v$, we first  derives the following $L^p(I)$-estimates of $\rho^{\frac{1}{p}}u$:

\begin{lem}\label{rho u-L2}
For  any  $p\in[2,\infty)$, 
there exists a constant $C(p,T)>0$  such that, 
for any $t\in [0,T]$,
\begin{equation*}
\big|\rho^\frac{1}{p}u(t)\big|_p^p+ \int_0^t\Big(\big|\rho^\frac{1}{2}|u|^\frac{p-2}{2} u_r\big|_2^2+\big|(r^{-2}\rho)^\frac{1}{p}u\big|_p^p\Big)\,\mathrm{d}s \leq C(p,T)\Big(\big(\sup_{s\in[0,t]}|v|_\infty\big)^2+1\Big).
\end{equation*}
\end{lem}

\begin{proof}
Multiplying both sides of $\eqref{e1.5}_2$ by $|u|^{p-2}u$ with $p\in [2,\infty)$, together with $\eqref{e1.5}_1$ and \eqref{V-expression}, leads to
\begin{equation}\label{eq:6.3}
\begin{aligned}
&\frac{1}{p}(\rho |u|^p)_t+2\alpha(p-1)\rho |u|^{p-2} |u_r|^2+\frac{2\alpha m(p-1)}{p}\frac{\rho |u|^p}{r^2}-A(p-1)\rho^\gamma |u|^{p-2}u_r\\
&=\Big(\underline{2\alpha\rho u_r|u|^{p-2}u+\frac{2\alpha m}{p}\frac{\rho |u|^p}{r}-\frac{1}{p}\rho u |u|^p-A\rho^\gamma |u|^{p-2}u}_{:=\mathcal{B}_4}\Big)_r-\frac{m}{p}\frac{\rho v |u|^p}{r}.
\end{aligned}
\end{equation}
Here we need to show  that $\mathcal{B}_4\in W^{1,1}(I)$ 
and $\mathcal{B}_4|_{r=0}=0$ for {\it a.e.} $t\in (0,T)$, 
which allows us to apply Lemma \ref{calculus} to obtain
\begin{equation}\label{eq:B4}
\int_0^\infty (\mathcal{B}_4)_r\,\mathrm{d}r=-\mathcal{B}_4|_{r=0}=0.
\end{equation}
On one hand,  $\mathcal{B}_4|_{r=0}=0$ follows directly from 
\eqref{V-expression}, $p\geq 2$, $u|_{r=0}=0$, 
and the fact that $(\rho,u,\frac{u}{r},u_r)\in C(\bar I)$ for each $t\in (0,T]$ due to \eqref{spd2} (or \eqref{spd233}). 
On the other hand,  if $n=2$ $(m=1)$, \eqref{spd}--\eqref{spd2} 
or \eqref{spd33}--\eqref{spd233} imply that 
\begin{equation*}
\sqrt{r}\big(u,\frac{u}{r},u_r,u_{rr}\big)\in L^2(I),\quad \big(\rho,u,\frac{u}{r},u_r,(\log\rho)_r\big)\in L^\infty(I),\quad r\rho\in L^1(I)
\end{equation*}
for {\it a.e.} $t\in (0,T)$, so that 
\begin{equation*}
|\rho|_1\leq |\chi_1^\flat\rho|_1+|\chi_1^\sharp\rho|_1\leq |\rho|_\infty+|\chi_1^\sharp r^{-1}|_\infty|r\rho|_1<\infty, 
\end{equation*}
which, along with the H\"older inequality, yields that
\begin{align*}
|\mathcal{B}_4|_1 &\leq  C(p)\Big(\big|\rho |u|^{p-1}|u_r|\big|_1
+\Big|\frac{1}{r}\rho |u|^p\Big|_1+\big|\rho |u|^{p+1}\big|_1+\big|\rho^\gamma|u|^{p-1}\big|_1\Big)\\
&\leq C(p)|\rho|_\infty |u|_\infty^{p-2}\Big(\Big|\frac{u}{\sqrt{r}}\Big|_2|\sqrt{r}u_r|_2+\Big|\frac{u}{\sqrt{r}}\Big|_2^2\Big)\\
&\quad + C(p)|\rho|_\infty|u|_\infty^{p-1}|\sqrt{r}u|_2\Big|\frac{u}{\sqrt{r}}\Big|_2+C(p)|r\rho|_1|\rho|_\infty^{\gamma-1}\Big|\frac{u}{r}\Big|_\infty|u|_\infty^{p-2}<\infty,\\[1mm]
|(\mathcal{B}_4)_r|_1&\leq C(p)\big|\big(|\rho_r||u_r||u|^{p-1},\rho |u_{rr}||u|^{p-1},\rho u_r^2|u|^{p-2}\big)\big|_1\\
&\quad +C(p)\Big|\Big(\frac{1}{r^2}\rho |u|^p,\frac{1}{r}\rho_r |u|^p,\frac{1}{r}\rho |u|^{p-1}|u_r|\Big)\Big|_1+C(p)\big|\big(|\rho_r||u|^{p+1},\rho |u|^p |u_r|\big)\big|_1\\
&\quad + C(p)\big|\big(\rho^{\gamma-1}|\rho_r||u|^{p-1},\rho^{\gamma} |u|^{p-2} |u_r|\big)\big|_1\\
&\leq C(p)|\rho|_\infty|u|_\infty^{p-2} \Big|\frac{u}{\sqrt{r}}\Big|_2\!\big(|(\log\rho)_r|_\infty|\sqrt{r}u_r|_2\!+\!|\sqrt{r}u_{rr}|_2\big)\!+\!C(p)|\rho|_1|u|_\infty^{p-2}|u_r|_\infty^2\\
&\quad+C(p)|\rho|_1|u|_\infty^{p-2}\Big|\frac{u}{r}\Big|_\infty\Big(\Big|\frac{u}{r}\Big|_\infty +|(\log\rho)_r|_\infty |u|_\infty+|u_r|_\infty\Big)\\
&\quad+C(p)\big(|\rho|_1|u|_\infty^{p}+|\rho|_1|\rho|_\infty^{\gamma-1}|u|_\infty^{p-2}\big)\big(|(\log\rho)_r|_\infty |u|_\infty +|u_r|_\infty\big)<\infty.
\end{align*}
Furthermore, if $n=3$ $(m=2)$, \eqref{spd}--\eqref{spd2} or \eqref{spd33}--\eqref{spd233} imply that
\begin{equation*}
(u,ru_r,ru_{rr})\in L^2(I),\quad \big(\rho,u,\frac{u}{r},u_r,(\log\rho)_r\big)\in L^\infty(I),
\quad r^2\rho\in L^1(I) 
\end{equation*}
for {\it a.e.} $t\in (0,T)$, so that
\begin{align*}
|\rho|_1\leq |\chi_1^\flat\rho|_1+|\chi_1^\sharp\rho|_1\leq |\rho|_\infty+ |\chi_1^\sharp r^{-2}|_\infty |r^2\rho|_1<\infty,\qquad |\rho|_2\leq |\rho|_1^\frac{1}{2}|\rho|_\infty^\frac{1}{2}<\infty,
\end{align*}
which, along with the H\"older inequality, yields that
\begin{align*}
|\mathcal{B}_4|_1 &\leq C(p)\Big(\big|\rho |u|^{p-1}|u_r|\big|_1
+\Big|\frac{1}{r}\rho |u|^p\Big|_1+\big|\rho |u|^{p+1}\big|_1+\big|\rho^\gamma|u|^{p-1}\big|_1\Big)\\
&\leq C(p)|\rho|_1 |u|_\infty^{p-1}\Big(\Big|\big(u_r,\frac{u}{r}\big)\Big|_\infty+|u|_\infty^2+|\rho|_\infty^{\gamma-1}\Big) <\infty,\\[1mm]
|(\mathcal{B}_4)_r|_1&\leq C(p)\big|\big(|\rho_r||u_r||u|^{p-1},\rho |u_{rr}||u|^{p-1},\rho u_r^2|u|^{p-2}\big)\big|_1\\
&\quad +C(p)\Big|\Big(\frac{1}{r^2}\rho |u|^p,\frac{1}{r}\rho_r |u|^p,\frac{1}{r}\rho |u|^{p-1}|u_r|\Big)\Big|_1\!\!
+\!C(p)\big|\big(|\rho_r||u|^{p+1},\rho |u|^p |u_r|\big)\big|_1\\
&\quad + C(p)\big|\big(\rho^{\gamma-1}|\rho_r||u|^{p-1},\rho^{\gamma} |u|^{p-2} |u_r|\big)\big|_1\\
&\leq C(p)|\rho|_2|u|_\infty^{p-2} \Big|\frac{u}{r}\Big|_\infty \big(|(\log\rho)_r|_\infty|ru_r|_2+|ru_{rr}|_2\big)+C(p)|\rho|_1|u|_\infty^{p-2}|u_r|_\infty^2\\
&\quad+C(p)|\rho|_1|u|_\infty^{p-2}\Big|\frac{u}{r}\Big|_\infty\Big(\Big|\frac{u}{r}\Big|_\infty +|(\log\rho)_r|_\infty |u|_\infty+|u_r|_\infty\Big)\\
&\quad+C(p)\big(|\rho|_1|u|_\infty^{p}+|\rho|_1|\rho|_\infty^{\gamma-1}|u|_\infty^{p-2}\big)\big(|(\log\rho)_r|_\infty |u|_\infty +|u_r|_\infty\big)<\infty.
\end{align*}

As a consequence, integrating \eqref{eq:6.3} over $I$, we obtain from \eqref{eq:B4}, 
Lemmas \ref{far-p-infty} and \ref{important2}, and the H\"older and Young inequalities that
\begin{equation}\label{503}
\begin{aligned}
&\frac{1}{p}\frac{\mathrm{d}}{\mathrm{d}t}\big|\rho^\frac{1}{p}u\big|_p^p+ 2\alpha (p-1)\big|\rho^\frac{1}{2} |u|^\frac{p-2}{2} u_r\big|_2^2+\frac{2\alpha m(p-1)}{p}\big|(r^{-2}\rho)^\frac{1}{p}u\big|_p^p\\
&=A(p-1)\int_0^\infty \rho^\gamma |u|^{p-2} u_r\,\mathrm{d}r \underline{-\frac{m}{p}\int_0^\infty \frac{\rho v|u|^p}{r}\,\mathrm{d}r}_{:=\mathcal{G}_3}\\
&\leq  C(p)\big(\big|\chi_1^\flat \rho^{\gamma-\frac{1}{2}}|u|^\frac{p-2}{2}\big|_2+\big|\chi_1^\sharp \rho^{\gamma-\frac{1}{2}}|u|^\frac{p-2}{2}\big|_2\big)\,\big|\rho^\frac{1}{2}|u|^\frac{p-2}{2}u_r\big|_2+\mathcal{G}_3\\
&\leq  C(p)\Big(\big|\chi_1^\flat \rho|_{\infty}^{\frac{p\gamma-p+1}{p}}+|\chi_1^\sharp r^{-m}|_\infty^{\frac{p\gamma-p+1}{p}}\big|\chi_1^\sharp r^m\rho|_{p\gamma-p+1}^{\frac{p\gamma-p+1}{p}}\Big) \,\big|\rho^\frac{1}{p}u\big|_p^\frac{p-2}{2} \big|\rho^\frac{1}{2}|u|^\frac{p-2}{2}u_r\big|_2+\mathcal{G}_3\\
&\leq  C(p,T)\big(1+\big|\rho^\frac{1}{p}u\big|_p^p\big)+\frac{\alpha}{8}\big|\rho^\frac{1}{2} |u|^\frac{p-2}{2} u_r\big|_2^2+\mathcal{G}_3.
\end{aligned}    
\end{equation}
For $\mathcal{G}_3$, if $n=2$ ($m=1$), we directly have
\begin{equation}\label{R2-dim2}
\mathcal{G}_3=-\frac{1}{p}\int_0^\infty \frac{\rho v |u|^p}{r}\,\mathrm{d}r\leq C(p)|v|_\infty  \big|(r^{-1}\rho)^\frac{1}{p}u\big|_p^p\leq C(p)(|v|_\infty^2+1)  \big|(r^{m-2}\rho)^\frac{1}{p}u\big|_p^p;
\end{equation}
while, if  $n=3$ ($m=2$), it follows from the H\"older and Young inequalities that
\begin{equation}\label{R2-dim3}
\begin{aligned}
\mathcal{G}_3&=-\frac{2}{p}\int_0^\infty \frac{\rho v |u|^p}{r}\,\mathrm{d}r\leq C(p)|v|_\infty \big|\rho^\frac{1}{p} u\big|_{p}^\frac{p}{2}\big|(r^{-2}\rho)^\frac{1}{p}u\big|_p^{\frac{p}{2}}\\
&\leq C(p)|v|_\infty^2\big|(r^{m-2}\rho)^\frac{1}{p} u\big|_{p}^p+\frac{\alpha}{8}\big|(r^{-2}\rho)^\frac{1}{p}u\big|_p^p.
\end{aligned}
\end{equation}

Thus, collecting \eqref{503}--\eqref{R2-dim3} gives
\begin{equation}\label{711}
\begin{aligned}
&\frac{\mathrm{d}}{\mathrm{d}t}\big|\rho^\frac{1}{p}u\big|_p^p+ \big|\rho^\frac{1}{2}|u|^\frac{p-2}{2} u_r\big|_2^2+\big|(r^{-2}\rho)^\frac{1}{p}u\big|_p^p\\
&\leq C(p,T)\big(1+\big|\rho^\frac{1}{p}u\big|_{p}^{p}\big)+C(p)(|v|_\infty^2+1)  \big|(r^{m-2}\rho)^\frac{1}{p}u\big|_p^p,
\end{aligned}
\end{equation}
which, along with the Gr\"onwall inequality and Lemma \ref{lemma-uv-p}, yields that, for all $t\in [0,T]$,
\begin{equation*}
\begin{aligned}
&\big|\rho^\frac{1}{p}u(t)\big|_p^p+ \int_0^t\Big(\big|\rho^\frac{1}{2}|u|^\frac{p-2}{2} u_r\big|_2^2+\big|(r^{-2}\rho)^\frac{1}{p}u\big|_p^p\Big)\,\mathrm{d}s\\
&\leq C(p,T)\Big(\big|\rho_0^\frac{1}{p}u_0\big|_p^p+\big(\sup_{s\in[0,t]}|v|_\infty\big)^2+1\Big)\leq C(p,T)\Big(\big(\sup_{s\in[0,t]}|v|_\infty\big)^2+1\Big),
\end{aligned}
\end{equation*}
provided that the $L^p(I)$-boundedness of $\rho_0^{\frac{1}{p}}u_0$ can be obtained. 
Indeed, it follows from Lemmas \ref{ale1}, \ref{initial3}, and \ref{lemma-initial} that
\begin{align*}
\big|\rho_0^\frac{1}{p}u_0\big|_p&\leq \big|\chi_1^\flat \rho_0^\frac{1}{p}u_0\big|_p+\big|\chi_1^\sharp \rho_0^\frac{1}{p}u_0\big|_p\leq C_0|\rho_0|_\infty^\frac{1}{p}|u_0|_\infty+ |\rho_0|_\infty^\frac{1}{p}|\chi_1^\sharp r^{-\frac{m}{p}}|_\infty \big|r^\frac{m}{p}u_0\big|_p\\
&\leq C(p)\big(\|\boldsymbol{u}_0\|_{L^\infty}+\|\boldsymbol{u}_0\|_{L^p}\big)\leq C(p)\|\boldsymbol{u}_0\|_{H^2}\leq C(p).
\end{align*}

The proof of Lemma \ref{rho u-L2} is completed.
\end{proof}

Next, we obtain the $L^1([0,T];L^\infty(I))$-estimate of $\rho^{\gamma-1} u$.
\begin{lem}\label{cru3}
For any $\epsilon\in (0,1)$, there exists a constant $C(\epsilon,T)>0$ such that, for any $t\in [0,T]$,
\begin{equation*}
\int_0^t |\rho^{\gamma-1}u|_{\infty}\, \mathrm{d}s\leq C(\epsilon,T)\Big(1+\int_0^t |v|_\infty \,\mathrm{d}s\Big)+\epsilon\sup_{s\in[0,t]}|v|_\infty.
\end{equation*}
\end{lem}

\begin{proof}
Let $q\geq 2$ be determined later. 
First, it follows from \eqref{V-expression}, Lemma \ref{ale1}, and  the H\"older inequality that
\begin{equation}\label{6.10}
\begin{aligned}
|\rho^{\gamma-1}u|_\infty^{q}&=\big||\rho^{\gamma-1}u|^q\big|_\infty\leq C_0\int_0^\infty \rho^{q\gamma-q} |u|^{q} \,\mathrm{d}r+C_0\int_0^\infty |(\rho^{q\gamma-q} |u|^{q})_r| \,\mathrm{d}r\\
&\leq C_0\int_0^\infty \rho^{q\gamma-q} |u|^{q} \,\mathrm{d}r+C(q)\int_0^\infty \rho^{q\gamma-q}(|v|+|u|)|u|^{q}\,\mathrm{d}r\\
&\quad + C(q)\int_0^\infty \rho^{q\gamma-q} |u|^{q-1} |u_r|\, \mathrm{d}r \\
&\leq C(q)(1+|v|_\infty)\big|\rho^{q\gamma-q}u^{q}\big|_1+C(q)\big|\rho^{q\gamma-q}u^{q+1}\big|_1\\
&\quad +C(q)\big|\rho^{2q\gamma-2q-1}u^{q}\big|_1^\frac{1}{2}\big|\rho^\frac{1}{2}|u|^\frac{q-2}{2}u_r\big|_2 \\
&\leq C(q)(1+|v|_\infty)|\rho|_\infty^{q\gamma-q-1}\big|\rho^{\frac{1}{q}}u\big|_{q}^{q}+C|\rho^{\gamma-1}u|_\infty|\rho|_\infty^{(q-1)\gamma-q}\big|\rho^{\frac{1}{q}}u\big|_{q}^{q}\\
&\quad +C(q)|\rho|_\infty^{q\gamma-q-1}\big|\rho^{\frac{1}{q}}u\big|_{q}^\frac{q}{2}\big|\rho^\frac{1}{2}|u|^\frac{q-2}{2}u_r\big|_2.
\end{aligned}    
\end{equation}

Next, setting $q=\tilde q\geq 2$ in \eqref{6.10} large enough such that 
\begin{equation*}
\gamma>\frac{\tilde q}{\tilde q-1}>1,
\end{equation*}
we obtain from the resulting inequality, Lemma \ref{important2}, and the  Young inequality  that
\begin{align*}
|\rho^{\gamma-1}u|_\infty^{\tilde q}&\leq C(T)(1+|v|_\infty)\big|\rho^\frac{1}{\tilde q}u\big|_{\tilde q}^{\tilde q}+C(T)\big|\rho^\frac{1}{\tilde q}u\big|_{\tilde q}^\frac{\tilde q^2}{\tilde q-1}\\
&\quad+C(T)\big|\rho^\frac{1}{\tilde q}u\big|_{\tilde q}^\frac{\tilde q}{2}\big|\rho^\frac{1}{2}|u|^\frac{\tilde q-2}{2}u_r\big|_2 +\frac{1}{2}|\rho^{\gamma-1}u|_\infty^{\tilde q},
\end{align*}
which, along with the fact that $\tilde q\geq 2$, Lemma \ref{lemma-uv-p}, and the Young inequality, leads to
\begin{align*}
|\rho^{\gamma-1}u|_\infty&\leq C(T)\Big(\big(1+|v|_\infty^\frac{1}{\tilde q}\big)\big|\rho^\frac{1}{\tilde q}u\big|_{\tilde q}+\big|\rho^\frac{1}{\tilde q}u\big|_{\tilde q}^\frac{\tilde q}{\tilde q-1}+\big|\rho^\frac{1}{\tilde q}u\big|_{\tilde q}^\frac{1}{2}\big|\rho^\frac{1}{2}|u|^\frac{\tilde q-2}{2}u_r\big|_2^\frac{1}{\tilde q}\Big)\\
&\leq C(T)\Big(1+|v|_\infty+\big|\rho^\frac{1}{\tilde q}u\big|_{\tilde q}^{\tilde q}+\big|\rho^\frac{1}{2}|u|^\frac{\tilde q-2}{2}u_r\big|_2^\frac{2}{2\tilde q-1}\Big)\\
&\leq C(T)\Big(1+|v|_\infty+\big|\chi_1^\flat\rho^\frac{1}{\tilde q}u\big|_{\tilde q}^{\tilde q}+\big|\chi_1^\sharp\rho^\frac{1}{\tilde q}u\big|_{\tilde q}^{\tilde q}+\big|\rho^\frac{1}{2}|u|^\frac{\tilde q-2}{2}u_r\big|_2^\frac{2}{2\tilde q-1}\Big)\\
&\leq C(T)\Big(1+|v|_\infty+|\chi_1^\flat r^{2-m}|_\infty \big|(r^{m-2}\rho)^\frac{1}{\tilde q}u\big|_{\tilde q}^{\tilde q}\Big)\\
&\quad +C(T)\Big(|\chi_1^\sharp r^{-m}| \big|(r^m\rho)^\frac{1}{\tilde q}u\big|_{\tilde q}^{\tilde q}+\big|\rho^\frac{1}{2}|u|^\frac{\tilde q-2}{2}u_r\big|_2^\frac{2}{2\tilde q-1}\Big)\\
&\leq C(T)\Big(1+|v|_\infty+\big|(r^{m-2}\rho)^\frac{1}{\tilde q}u\big|_{\tilde q}^{\tilde q}+\big|\rho^\frac{1}{2}|u|^\frac{\tilde q-2}{2}u_r\big|_2^\frac{2}{2\tilde q-1}\Big).
\end{align*}

Finally, integrating above over $[0,t]$, we obtain from Lemmas \ref{lemma-uv-p} and \ref{rho u-L2}, 
and the H\"older and Young inequalities that, for all $\epsilon\in (0,1)$,
\begin{align*}
\int_0^t |\rho^{\gamma-1}u|_\infty\,\mathrm{d}s&\leq C(T)\int_0^t\Big(1+|v|_\infty+\big|(r^{m-2}\rho)^\frac{1}{\tilde q}u\big|_{\tilde q}^{\tilde q}+\big|\rho^\frac{1}{2}|u|^\frac{\tilde q-2}{2}u_r\big|_2^\frac{2}{2\tilde q-1}\Big)\,\mathrm{d}s\\
&\leq C(T)\Big(1+ \int_0^t |v|_\infty\,\mathrm{d}s\Big)+C(T)\Big(\int_0^t\big|\rho^\frac{1}{2}|u|^\frac{\tilde q-2}{2}u_r\big|_2^2 \,\mathrm{d}s\Big)^\frac{1}{2\tilde q-1}\\
&\leq C(T)\Big(1+ \int_0^t |v|_\infty\,\mathrm{d}s\Big)+C(T)\Big(\sup_{s\in[0,t]}|v|_\infty\Big)^\frac{2}{2\tilde q-1}\\
&\leq C(\epsilon,T)\Big(1+ \int_0^t |v|_\infty\,\mathrm{d}s\Big)+\epsilon\sup_{s\in[0,t]}|v|_\infty.
\end{align*}

The proof of Lemma \ref{cru3} is completed.
\end{proof}

Now, the uniform $L^\infty(I)$-estimate of the effective velocity can be derived as follows:
\begin{lem}\label{l4.4}
There exists a constant $C(T)>0$  such that 
\begin{equation*}
|v(t)|_\infty\leq C(T)   \qquad \text{for any $t\in [0,T]$}.
\end{equation*}
\end{lem}

\begin{proof} 
First,  define the flow map $y=\eta(t,r): [0,T]\times I \to I$ as 
\begin{equation}\label{flow}
\eta_t(t,r)=u(t,\eta(t,r)) \qquad\text{with $\eta(0,r)=r$}.
\end{equation}
Then \eqref{eq:effective2}, together with \eqref{flow}, implies the following ODE:
\begin{equation*}
\frac{\mathrm{d}}{\mathrm{d}t}v(t,\eta(t,r))+\frac{A\gamma}{2\alpha}(\rho^{\gamma-1}v)(t,\eta(t,r))=\frac{A\gamma}{2\alpha}(\rho^{\gamma-1}u)(t,\eta(t,r)),
\end{equation*}
which, along with the characteristic method, yields that 
\begin{equation}\label{try}
\begin{aligned}
v(t,\eta(t,r))&=v_0(r)\exp\Big(-\int_0^t \frac{A\gamma}{2\alpha} \rho^{\gamma-1}(\tau,\eta(\tau,r))\,\mathrm{d}\tau\Big)\\
&\quad+\frac{A\gamma}{2\alpha}\int_0^t (\rho^{\gamma-1}u)(s,\eta(s,r))\, \exp\Big(-\int_s^t\frac{A\gamma}{2\alpha}\rho^{\gamma-1}(\tau,\eta(\tau,r))\,\mathrm{d}\tau\Big)\,\mathrm{d}s.
\end{aligned}
\end{equation}

Since $\rho\geq 0$, it follows from \eqref{try} and Lemma \ref{cru3} that, for all $\epsilon\in (0,1)$,
\begin{equation}\label{e-1.17} 
\begin{aligned}
\sup_{s\in[0,t]}|v|_\infty&\leq C_0 \Big(|v_0|_\infty+\int_0^t|\rho^{\gamma-1}u|_\infty\mathrm{d}s\Big)\\
&\leq C(\epsilon,T)\Big(1+\int_0^t |v|_\infty\,\mathrm{d}s\Big)+C_0\epsilon\sup_{s\in[0,t]}|v|_\infty,
\end{aligned}
\end{equation}
where the $L^\infty(I)$-norm of $v_0$ can be derived from Lemmas \ref{ale1}, \ref{initial3}, 
and \ref{lemma-initial} as
\begin{equation*}
|v_0|_\infty\leq C_0|(u_0,(\log\rho_0)_r)|_\infty \leq C_0\|(\boldsymbol{u}_0,\nabla\log\rho_0)\|_{L^\infty} \leq C_0.
\end{equation*}

Finally, letting $\epsilon=\frac{1}{2C_0}$ in \eqref{e-1.17} 
and applying the Gr\"onwall inequality to \eqref{e-1.17} lead to 
the desired conclusion.
\end{proof}

\section{Non-Formation of Cavitation
inside the Fluids  in Finite Time}\label{section-nonformation}

This section is devoted to showing that the cavitation does not form inside the fluids in finite 
time for the case that $\bar\rho=0$, and establishing the lower bound estimates of the density. 
Let $T>0$ be any fixed time, and let $(\rho, u)(t,r)$ be the $s$-order $(s=2,3)$ regular solution of 
problem \eqref{e1.5} in $[0,T]\times I$ obtained in Theorems \ref{rth1}--\ref{rth133}. 
Moreover, throughout this section, we always assume that \eqref{gammafanwei} holds.

We first show the following $L^2(I)$-estimate of $u$.
\begin{lem}\label{ele}
There  exists  a constant $C(T)>0$ such that 
\begin{equation*}
|(\rho^{\gamma-1} v,u)(t)|_2^2 +\int_0^t\Big(\Big|\big(u_r,\frac{u}{r}\big)\Big|_2^2+|u|^2_\infty\Big)\,\mathrm{d}s\leq C(T) \qquad\mbox{for any $t\in [0,T]$}.
\end{equation*}
\end{lem}

\begin{proof} We divide the proof into two steps.

\smallskip
\textbf{1.}
First, multiplying $\eqref{e1.5}_2$ by $u$, along with \eqref{V-expression}, gives
\begin{equation}\label{eq:6.14}
\begin{aligned}
&\frac{1}{2}(u^2)_t+2\alpha |u_r|^2+\alpha m\frac{u^2}{r^2}\\
&=\Big(\underline{2\alpha u_ru+\alpha m\frac{u^2}{r}-\frac{2}{3}u^3}_{:=\mathcal{B}_5}\Big)_r-\frac{A\gamma}{2\alpha}\rho^{\gamma-1}(v-u)u+vu_ru.
\end{aligned}
\end{equation}
We need to show that $\mathcal{B}_5\in W^{1,1}(I)$ and $\mathcal{B}_5|_{r=0}=0$ for {\it a.e.} 
$t\in (0,T)$, which allows us to apply Lemma \ref{calculus} to obtain
\begin{equation}\label{eq:B5}
\int_0^\infty (\mathcal{B}_5)_r\,\mathrm{d}r=-\mathcal{B}_5|_{r=0}=0.
\end{equation}
On one hand,  $\mathcal{B}_5|_{r=0}=0$ follows from $u|_{r=0}=0$ 
and the fact that $(u,u_r)\in C(\bar I)$ for each $t\in (0,T]$ due to \eqref{spd2} or \eqref{spd233}. 
On the other hand,  it follows from \eqref{spd}--\eqref{spd2} or \eqref{spd33}--\eqref{spd233} that 
\begin{equation*}
(u,u_r.\frac{u}{r})\in L^\infty(I),\quad 
r^\frac{m}{2}(u,u_r,\frac{u}{r},\frac{u_{rr}}{r},\frac{1}{r}(\frac{u}{r})_r)\in L^2(I)\qquad \,\mbox{for {\it a.e.} $t\in (0,T)$}.
\end{equation*}
Then we obtain from Lemma \ref{hardy} that
\begin{equation}\label{ul2-3}
\begin{aligned}
\Big|\big(u,u_r,\frac{u}{r}\big)\Big|_2&\leq \Big|\chi^\flat_1\big(u,u_r,\frac{u}{r}\big)\Big|_2+\Big|\chi^\sharp_1\big(u,u_r,\frac{u}{r}\big)\Big|_2\\
&\leq \Big|\big(u,u_r,\frac{u}{r}\big)\Big|_\infty+|\chi^\sharp_1 r^{-\frac{m}{2}}|_\infty\Big|r^\frac{m}{2}\big(u,u_r,\frac{u}{r}\big)\Big|_2<\infty,
\end{aligned}    
\end{equation}
which, along with the H\"older's inequality, leads to
\begin{align*}
|\mathcal{B}_5|_1&\leq C_0|u|_2\Big(\Big|\big(u_r,\frac{u}{r}\big)\Big|_2+|u|_2|u|_\infty\Big) <\infty,\\
|(\mathcal{B}_5)_r|_1&\leq C_0\Big|\big((u_r)^2,u_{rr}u,u_r\frac{u}{r},u(\frac{u}{r})_r,u^2u_r\big)\Big|_1\\
&\leq C_0\Big(|u_r|_2\Big|\big(u_r,\frac{u}{r}\big)\Big|_2+|r^\frac{2-m}{2}u|_2\Big|r^\frac{m-2}{2}\big(u_{rr},(\frac{u}{r})_r\big)\Big|_2+|u|_\infty|u|_2|u_r|_2\Big)<\infty.
\end{align*}

Then integrating \eqref{eq:6.14} over $I$, together with \eqref{eq:B5}, yields
\begin{equation}\label{ell}
\frac{1}{2}\frac{\mathrm{d}}{\mathrm{d}t}|u|_2^2+ 2\alpha|u_r|_2^2+\alpha m\Big|\frac{u}{r}\Big|_2^2
=-\frac{A\gamma}{2\alpha}\int_0^\infty \rho^{\gamma-1}(v-u) u\,\mathrm{d}r+ \int_0^\infty v u_ru \,\mathrm{d}r :=\sum_{i=4}^5 \mathcal{G}_i.
\end{equation}
It follows from Lemmas \ref{important2} and \ref{l4.4}, and the H\"older and Young inequalities that 
\begin{equation}\label{j3-j5}
\begin{aligned}
\mathcal{G}_4&\leq C_0\big(|\rho^{\gamma-1} v|_2|u|_2 +|\rho|_\infty^{\gamma-1} |u|_2^2\big) \leq  C(T)|(\rho^{\gamma-1} v,u)|_2^2,\\
\mathcal{G}_5&\leq C_0|v|_\infty|u|_2 |u_r|_2 \leq \frac{\alpha}{8}|u_r|_2^2+C(T)|u|_2^2.
\end{aligned}
\end{equation}

Combining with \eqref{ell}--\eqref{j3-j5} yields
\begin{equation}\label{pol}
\frac{\mathrm{d}}{\mathrm{d}t}|u|_2^2+ \alpha\Big|\big(u_r,\frac{u}{r}\big)\Big|_2^2\leq  C(T)|(\rho^{\gamma-1} v,u)|_2^2. 
\end{equation}

\smallskip
\textbf{2.} For the $L^2(I)$-estimate of $\rho^{\gamma-1} v$, 
we first multiply \eqref{eq:effective2} by $\rho^{2\gamma-2}v$ and then obtain from $\eqref{e1.5}_1$ that
\begin{equation}\label{eq:B6-pre}
\begin{aligned}
&\frac{1}{2}(\rho^{2\gamma-2}v^2)_t+ \frac{1}{2} (\underline{u\rho^{2\gamma-2}v^2}_{:=\mathcal{B}_6})_r+\frac{A\gamma}{2\alpha}\rho^{3\gamma-3} v^2\\
&=\big(\frac{3}{2}-\gamma\big)\rho^{2\gamma-2} v^2 u_r-(\gamma-1)m\rho^{2\gamma-2} v^2\frac{u}{r}+\frac{A\gamma}{2\alpha}\rho^{3\gamma-3} vu.
\end{aligned}
\end{equation}

We need to show that $\mathcal{B}_6\in W^{1,1}(I)$ and $\mathcal{B}_6|_{r=0}=0$ 
for {\it a.e.} $t\in (0,T)$, which allows us to apply Lemma \ref{calculus} to obtain
\begin{equation}\label{eq:B6}
\int_0^\infty (\mathcal{B}_6)_r\,\mathrm{d}r=-\mathcal{B}_6|_{r=0}=0.   
\end{equation}
On one hand,  to obtain $\mathcal{B}_6|_{r=0}=0$, we first note that $v=u+2\alpha(\log\rho)_r$ and
\begin{equation*}
(u,(\log\rho)_r)\in L^\infty(I)
\qquad\mbox{for {\it a.e.} $t\in(0,T)$}
\end{equation*}
due to \eqref{spd}--\eqref{spd2} or \eqref{spd33}--\eqref{spd233}, which implies that 
\begin{equation}\label{psi,wuqiong}
v\in L^\infty(I)  \qquad\text{for {\it a.e.} $t\in (0,T)$}.
\end{equation}
Then it follows from $u|_{r=0}$ and $(\rho,u) \in C(\bar I)$ for each $t\in (0,T]$ (due to \eqref{spd2} or \eqref{spd233}) that $\mathcal{B}_6|_{r=0}=0$. 
On the other hand, it follows from \eqref{spd}--\eqref{spd2} (or \eqref{spd33}--\eqref{spd233}) 
and \eqref{ul2-3} that 
\begin{align*}
&(u,u_r,\frac{u}{r})\in L^2(I),\quad  (\rho,u,u_r,(\log\rho)_r)\in L^\infty(I),\\
&r^\frac{m}{2}\Big(u,u_r,(\rho^{\gamma-1})_r,\frac{1}{r}(\rho^{\gamma-1})_r,(\log\rho)_{rr}\Big)\in L^2(I)
\end{align*}
for {\it a.e.} $t\in (0,T)$. Thus, we obtain from Lemma \ref{hardy} that
\begin{align*}
|(\rho^{\gamma-1})_r|_2&\leq  \big|\chi_1^\flat(\rho^{\gamma-1})_r\big|_2+\big|\chi_1^\sharp(\rho^{\gamma-1})_r\big|_2\\
&\leq |\chi_1^\flat r^\frac{2-m}{2}|_\infty\big|r^\frac{m-2}{2}(\rho^{\gamma-1})_r\big|_2 + |\chi_1^\sharp r^{-\frac{m}{2}}|_\infty \big|r^\frac{m}{2}(\rho^{\gamma-1})_r\big|_2<\infty,
\end{align*}
which, along with \eqref{V-expression} and the H\"older and Young inequalities, yields that
\begin{align*}
|\mathcal{B}_6|_1&\leq C_0|u|_\infty\big(|\rho|_\infty^{2\gamma-2}|u|_2^2+|(\rho^{\gamma-1})_r|_2^2\big)<\infty,\\[1mm]
|(\mathcal{B}_6)_r|_1&\leq C_0\big|\big(u_r\rho^{2\gamma-2} v^2, u\rho^{\gamma-1}(\rho^{\gamma-1})_r v^2,u\rho^{2\gamma-2} vv_r\big)\big|_1\\
&\leq C_0|u_r|_\infty\big(|\rho|_\infty^{2\gamma-2}|u|_2^2+|(\rho^{\gamma-1})_r|_2^2\big)+C_0|\rho|_\infty^{\gamma-1}|(\rho^{\gamma-1})_r|_2|u|_2|(u,(\log\rho)_r)|_\infty^2\\
&\quad+ C_0|\rho|_\infty^{2\gamma-2}\Big|\frac{u}{r}\Big|_2|(u,(\log\rho)_r)|_\infty|\chi_1^\flat r(u_r,(\log\rho)_{rr})|_2\\
&\quad +C_0|\rho|_\infty^{2\gamma-2}|u|_2|(u,(\log\rho)_r)|_\infty|\chi_1^\sharp (u_r,(\log\rho)_{rr})|_2)\\
&\leq C_0|u_r|_\infty\big(|\rho|_\infty^{2\gamma-2}|u|_2^2+|(\rho^{\gamma-1})_r|_2^2\big)+C_0|\rho|_\infty^{\gamma-1}|(\rho^{\gamma-1})_r|_2|u|_2|(u,(\log\rho)_r)|_\infty^2\\
&\quad+ C_0|\rho|_\infty^{2\gamma-2}\Big|\frac{u}{r}\Big|_2|(u,(\log\rho)_r)|_\infty|\chi_1^\flat r^\frac{2-m}{2}|_\infty \big|r^\frac{m}{2} (u_r,(\log\rho)_{rr})\big|_2\\
&\quad +C_0|\rho|_\infty^{2\gamma-2}|u|_2|(u,(\log\rho)_r)|_\infty|\chi_1^\sharp r^{-\frac{m}{2}}|_\infty \big|r^\frac{m}{2}(u_r,(\log\rho)_{rr})\big|_2<\infty. 
\end{align*}

Thus, integrating \eqref{eq:B6-pre} over $I$, we obtain from \eqref{eq:B6}, Lemmas \ref{important2} and \ref{l4.4}, and the H\"older and Young inequalities that
\begin{equation}\label{new-v}
\begin{aligned}
&\frac{1}{2}\frac{\mathrm{d}}{\mathrm{d}t}|\rho^{\gamma-1} v|_2^2+ \frac{A\gamma}{2\alpha} \big|\rho^\frac{3\gamma-3}{2} v\big|_2^2\\
&=\big(\frac{3}{2}-\gamma\big)\int \rho^{2\gamma-2} v^2 u_r\,\mathrm{d}r-(\gamma-1)m\int \rho^{2\gamma-2}  v^2\frac{u}{r}\,\mathrm{d}r+\frac{A\gamma}{2\alpha}\int \rho^{3\gamma-3} vu\,\mathrm{d}r\\
&\leq  C_0|\rho|_\infty^{\gamma-1}|v|_\infty |\rho^{\gamma-1} v|_2 \Big|\big(u_r,\frac{u}{r}\big)\Big|_2+C_0|\rho|_\infty^{2\gamma-2} |\rho^{\gamma-1} v|_2 |u|_2\\
&\leq  C(T)|(\rho^{\gamma-1} v,u)|_2^2+\frac{\alpha}{8}\Big|\big(u_r,\frac{u}{r}\big)\Big|_2^2. 
\end{aligned}        
\end{equation}

Combining \eqref{pol} with \eqref{new-v} gives
\begin{equation*} 
\frac{\mathrm{d}}{\mathrm{d}t}|(\rho^{\gamma-1} v,u)|_2^2+ \frac{\alpha}{2}\Big|\big(u_r,\frac{u}{r}\big)\Big|_2^2\leq  C(T)|(\rho^{\gamma-1} v,u)|_2^2. 
\end{equation*}
which, along with the Gr\"onwall inequality, yields that, for all $t\in [0,T]$,
\begin{equation}\label{47}
|(\rho^{\gamma-1} v,u)(t)|_2^2 + \int_0^t\Big|\big(u_r,\frac{u}{r}\big)\Big|_2^2\,\mathrm{d}s\leq  C(T).
\end{equation}
We still needs to check the $L^2(I)$-boundedness of $(\rho_0^{\gamma-1} v_0,u_0)$. 
Indeed, it follows from Lemmas \ref{ale1}, \ref{initial3}, and \ref{lemma-initial} that
\begin{align*}
|u_0|_2&\leq |\chi_1^\flat u_0|_2+|\chi_1^\sharp u_0|_2\leq |u_0|_\infty+|\chi_1^\sharp r^{-\frac{m}{2}}|_\infty |r^\frac{m}{2}u_0|_2 \\
&\leq C_0\big(\|\boldsymbol{u}_0\|_{L^\infty}+\|\boldsymbol{u}_0\|_{L^2}\big)\leq C_0\|\boldsymbol{u}_0\|_{H^2}\leq C_0,\\[1mm]
|\rho_0^{\gamma-1} v_0|_2&\leq C_0|\rho_0^{\gamma-1}(u_0,(\log\rho_0)_r)|_2\leq C_0\big(|\rho_0|_\infty^{\gamma-1}|u_0|_2+ |(\rho_0^{\gamma-1})_r|_2\big) \\
&\leq C_0|\rho_0|_\infty^{\gamma-1}+C_0\big(|\chi_1^\flat(\rho_0^{\gamma-1})_r|_2+|\chi_1^\sharp(\rho_0^{\gamma-1})_r|_2\big)\\
&\leq C_0|\rho_0|_\infty^{\gamma-1}+C_0|\chi_1^\flat r^\frac{2-m}{2} |_\infty |r^\frac{m-2}{2}(\rho_0^{\gamma-1})_r|_2+C_0|\chi_1^\sharp r^{-\frac{m}{2}}|_\infty |r^\frac{m}{2} (\rho_0^{\gamma-1})_r|_2\\
&\leq C_0\Big(|\rho_0|_\infty^{\gamma-1}+\Big|r^\frac{m}{2}\Big(\frac{(\rho_0^{\gamma-1})_r}{r},(\rho_0^{\gamma-1})_{r}\Big)\Big|_2\Big)\\
&\leq C_0\big(\|\rho_0\|_{L^\infty}^{\gamma-1} +\|\nabla (\rho_0^{\gamma-1})\|_{H^1}\big) \leq C_0.
\end{align*}
 
Finally, it follows from \eqref{47} and Lemma \ref{ale1} that 
\begin{equation*}
\int_0^t |u|^2_\infty\,\mathrm{d}s\leq C_0\int_0^t |(u,u_r)|_2^2 \,\mathrm{d}s\leq C_0\Big(t\sup_{s\in [0,t]}|u|_2^2+\int_0^t |u_r|_2^2 \,\mathrm{d}s\Big)\leq C(T).
\end{equation*}

The proof of Lemma \ref{ele} is completed.
\end{proof}

Now, with the help of Lemmas \ref{l4.4} and \ref{ele}, we can show the pointwise estimates of $\rho$  
in the domain containing the origin.

\begin{lem}\label{lemma-inf-rho}
 Suppose that 
\begin{equation}\label{inf-rho0}
\inf_{z\in[0,r]}\rho_0(z)=\underline{\rho}(r)>0 \qquad\text{for $r>0$},
\end{equation}
with $\underline{\rho}(r)$, defined on $I$, satisfying $\underline{\rho}(r)\to 0$ as $r\to\infty$. 
Then,  for any  $R>0$, there exists a constant $C(T)>0$ such that 
\begin{equation*}
\begin{gathered}
\inf_{(t,r)\in [0,T]\times [0,R]} \rho(t,r)\geq \min\big\{C(T)^{-1},\big(e^{-1}\underline{\rho}(R)\big)^{C(T)(\sqrt{R}+1)}\big\}.
\end{gathered}   
\end{equation*}
In particular, the cavitation does not form in $[0,T]\times \{\boldsymbol{x}\in\mathbb{R}^n:\,|\boldsymbol{x}|\leq R\}$ 
for any  $R>0$.
\end{lem}

\begin{proof}
First, it follows from  \eqref{V-expression} and Lemmas \ref{l4.4} and \ref{ele} that, for all $t\in [0,T]$,
\begin{equation}\label{323}
|\chi_R^\flat (\log\rho)_r(t)|_2\leq C_0|\chi_R^\flat (v,u)(t)|_2 \leq C_0\big(\sqrt{R}|v(t)|_\infty+ |u(t)|_2\big)\leq C(T)(\sqrt{R}+1).
\end{equation}

Next, multiplying $\eqref{e1.5}_1$ by $\chi_R^\flat \rho^{-1}\log\rho$ and integrating over $I$,  we obtain from \eqref{323} and the H\"older and Young inequalities that
\begin{align*}
\frac{1}{2}\frac{\mathrm{d}}{\mathrm{d}t}|\chi_R^\flat \log\rho|_2^2&=-\int_0^\infty \chi_R^\flat u(\log\rho)_r\log\rho\,\mathrm{d}r- \int_0^\infty \chi_R^\flat\big(u_r+\frac{m}{r}u\big)\log\rho\,\mathrm{d}r\\
&\leq |u|_\infty|\chi_R^\flat(\log\rho)_r|_2|\chi_R^\flat\log\rho|_2+C_0\Big|\big(u_r,\frac{u}{r}\big)\Big|_2|\chi_R^\flat\log\rho|_2\\
&\leq C(T)\Big((R+1)|u|_\infty^2+ \Big|\big(u_r,\frac{u}{r}\big)\Big|_2^2+|\chi_R^\flat\log\rho|_2^2\Big),
\end{align*}
which, along with Lemma \ref{ele} and the Gr\"onwall inequality, leads to
\begin{equation}\label{324}
\begin{aligned}
|\chi_R^\flat \log\rho(t)|_2^2 &\leq C(T)\big(|\chi_R^\flat\log\rho_0|_2^2+R+1\big) \leq C(T)\big(R|\chi_R^\flat\log\rho_0|_\infty^2+R+1\big).
\end{aligned}
\end{equation}

Combining \eqref{323} with \eqref{324} and setting $R_0\geq 1$ large enough 
such that 
\begin{equation*}
\underline{\rho}(R)\leq \min\{1,|\rho_0|_\infty^{-1}\} \qquad \text{for $R\geq R_0$},
\end{equation*}
together with Lemma \ref{ale1}, yield that, for all $t\in [0,T]$ and $R\geq R_0$,
\begin{equation}\label{5220}
\begin{aligned}
|\chi_R^\flat\log\rho(t)|_\infty
& \leq C_0\Big(\sqrt{\frac{1+R}{R}} |\chi_R^\flat\log\rho(t)|_2+|\chi_R^\flat(\log\rho)_r(t)|_2\Big) \\
&\leq C(T)(\sqrt{R}+1)\big(|\chi_R^\flat\log\rho_0|_\infty+1\big)\\
&\leq C(T)(\sqrt{R}+1)\Big(\max\big\{\log\big(|\chi_R^\flat\rho_0|_{\infty}\big),\log\big(|\chi_R^\flat\rho_0^{-1}|_{\infty}\big)\big\}+1\Big)\\
&\leq C(T)(\sqrt{R}+1)\Big(\max\big\{\log(|\rho_0|_{\infty}), -\log\underline{\rho}(R)\big\}+1\Big)\\
&= -C(T)(\sqrt{R}+1)\big(\log(\underline{\rho}(R))-1\big)
= -C(T)(\sqrt{R}+1)\,\log\big(e^{-1}\underline{\rho}(R)\big),
\end{aligned}    
\end{equation}
which, along with $e^{-1}\underline{\rho}(R)<1$, implies that, 
for all $(t,r)\in[0,T]\times [0,R]$ and $R\geq R_0$, 
\begin{equation*}
 \rho(t,r)\geq \big(e^{-1}\underline{\rho}(R)\big)^{C(T)(\sqrt{R}+1)}.  
\end{equation*}

Finally, for $R\leq R_0$, it follows from Lemma \ref{ale1} and \eqref{323}--\eqref{324} that 
\begin{align*}
|\chi_{R}^\flat\log\rho(t)|_\infty&\leq |\chi_{R_0}^\flat\log\rho(t)|_\infty \leq C_0\Big(\sqrt{\frac{1+R_0}{R_0}} |\chi_{R_0}^\flat\log\rho(t)|_2+|\chi_{R_0}^\flat(\log\rho)_r(t)|_2\Big)\notag \\
&\leq  C(T)\big(|\chi_{R_0}^\flat\log\rho_0|_\infty+1\big)\leq C(T),
\end{align*}
which implies that $\rho(t,r)\geq C(T)^{-1}$ for all $(t,r)\in[0,T]\times [0,R]$ and $R\leq R_0$.
\end{proof}

\section{Global  Estimates for the \texorpdfstring{$2$}{}-Order 
Regular Solutions with Far-Field Vacuum}\label{section-global2}
The goal of this section is to establish the global-in-time uniform estimates for the $2$-order regular solutions 
when  $\bar\rho=0$.
Let $T>0$ be any fixed time, and let $(\rho, u)(t,r)$ be the $2$-order regular solution of problem \eqref{e1.5} in $[0,T]\times I$ 
obtained in Theorems \ref{rth1}.
Moreover, throughout this section, we always assume that 
\eqref{gammafanwei} holds.

Next, we consider the enlarged system \eqref{eqn1} in spherical coordinates. 
Specifically, we introduce the following two important quantities: 
\begin{equation}\label{tr}
\phi=\frac{A\gamma}{\gamma-1}\rho^{\gamma-1}, \qquad \psi=\frac{1}{\gamma-1}(\log \phi)_r=(\log \rho)_r.
 \end{equation}
Then \eqref{e1.5} can be rewritten as the problem of $(\phi,u,\psi)$ in $[0,T]\times I$:
\begin{equation}\label{e2.2}
\begin{cases}
\displaystyle \phi_t+u\phi_r+(\gamma-1)\phi\big(u_r+ \frac{m}{r}u\big)=0,\\[6pt]
\displaystyle u_t+u u_r+\phi_r=2\alpha\big(u_r+ \frac{m}{r}u\big)_r+2\alpha \psi u_r,\\[6pt]
\displaystyle \psi_t+(\psi u)_r+\big(u_r+ \frac{m}{r}u\big)_r=0,\\[6pt]
\displaystyle  (\phi,u,\psi)|_{t=0}=(\phi_0,u_0,\psi_0)
=(\frac{A\gamma}{\gamma-1}\rho_0^{\gamma-1},u_0,(\log \rho_0)_r) \qquad\text{for $r\in I$},\\[6pt]
\displaystyle u|_{r=0}=0 \qquad\qquad\qquad\qquad\quad \text{for $t\in [0,T]$},\\[6pt]
\displaystyle  (\phi,u)\to (0,0) \ \ \text{as $r\to\infty$}  \qquad \text{for $t\in [0,T]$}.
\end{cases}
\end{equation}
Clearly, the effective velocity $v=u+2\alpha(\log\rho)_r$ and its equation \eqref{eq:effective2} take the following forms, respectively:
\begin{align}
&v=u+2\alpha\psi=u+\frac{2\alpha}{\gamma-1}(\log\phi)_r,\label{v-express-2}\\
&v_t+uv_r+\frac{\gamma-1}{2\alpha}\phi(v-u)=0.\label{eq:effective1}
\end{align}

First, the following lemma will be frequently used in the later analysis, which can be seen as a consequence of the div-curl estimates for spherically symmetric functions in spherical coordinates.

\begin{lem}\label{im-1}
Let $p\in(1,\infty)$ and $\boldsymbol{f}(\boldsymbol{x})=f(r)\frac{\boldsymbol{x}}{r} \in C_{\rm c}^\infty(\mathbb{R}^n)$. 
Then 
\begin{align*}
&\|\nabla\boldsymbol{f}\|_{L^p} \sim \Big|r^\frac{m}{p}\big(f_r+\frac{m}{r}f\big)\Big|_{p},\quad \|\nabla^2\boldsymbol{f}\|_{L^p} \sim \Big|r^\frac{m}{p}\big(f_r+\frac{m}{r}f\big)_r\Big|_p,\\
&\|\nabla^3\boldsymbol{f}\|_{L^p} \sim \Big|r^\frac{m}{p}\Big(\big(f_r+\frac{m}{r}f\big)_{rr},\frac{1}{r}\big(f_r+\frac{m}{r}f\big)_r\Big)\Big|_p,\\
&\|\nabla^4\boldsymbol{f}\|_{L^p} \sim \Big|r^\frac{m}{p}\Big(\big(f_r+\frac{m}{r}f\big)_{rrr},\Big(\frac{1}{r}\big(f_r+\frac{m}{r}f\big)_{r}\Big)_r\Big)\Big|_p,
\end{align*}
where $F_1\sim F_2$ denotes that there exists a constant $C\geq 1$ depending only on $(n,p)$ 
such that $C^{-1}F_1\leq F_2\leq CF_1$. As a consequence, along with {\rm Lemma \ref{lemma-initial}}, 
the following estimates hold{\rm :} 
\begin{align*}
&\Big|r^\frac{m}{p}\big(f_r,\frac{f}{r}\big)\Big|_{p}\sim \Big|r^\frac{m}{p}\big(f_r+\frac{m}{r}f\big)\Big|_{p},\quad \Big|r^\frac{m}{p}\Big(f_{rr},\big(\frac{f}{r}\big)_r\Big)\Big|_{p}\sim \Big|r^\frac{m}{p}\big(f_r+\frac{m}{r}f\big)_r\Big|_p,\\
&\Big|r^\frac{m}{p}\Big(f_{rrr},\frac{f_{rr}}{r},\big(\frac{f}{r}\big)_{rr},\frac{1}{r}\big(\frac{f}{r}\big)_r\Big)\Big|_{p}\sim \Big|r^\frac{m}{p}\Big(\big(f_r+\frac{m}{r}f\big)_{rr},\frac{1}{r}\big(f_r+\frac{m}{r}f\big)_r\Big)\Big|_p,\\
&\Big|r^\frac{m}{p}\Big(f_{rrrr},\big(\frac{f_{rr}}{r}\big)_r,\big(\frac{f}{r}\big)_{rrr},\Big(\frac{1}{r}(\frac{f}{r}\big)_r\Big)_r\Big)\Big|_{p}\!\sim\! \Big|r^\frac{m}{p}\Big(\big(f_r+\frac{m}{r}f\big)_{rrr},\Big(\frac{1}{r}\big(f_r+\frac{m}{r}f\big)_{r}\Big)_r\Big)\Big|_p.
\end{align*}
\end{lem}
\begin{proof}
The proof can be derived via the $\boldsymbol{x}$-coordinates in $\mathbb{R}^n$.  Let $\boldsymbol{f}(\boldsymbol{x})=f(r)\frac{\boldsymbol{x}}{r}$. First, notice that $\boldsymbol{f}$ is curl-free, {\it i.e.}
\begin{equation}\label{curl}
\begin{aligned}
&\mathrm{curl}\,\boldsymbol{f}= \partial_1f_2 - \partial_2 f_1 =0 &&\quad\text{if $n=2$},\\
&\mathrm{curl}\,\boldsymbol{f}=(\partial_2f_3- \partial_3 f_2, \partial_3 f_1-\partial_1f_3, \partial_1f_2- \partial_2 f_1)^\top=\boldsymbol{0} &&\quad\text{if $n=3$},
\end{aligned}
\end{equation}
and the divergence of $\boldsymbol{f}$ takes the form:
\begin{equation}\label{div}
\diver\boldsymbol{f}=\sum_{i=1}^n \partial_if_i=f_r+\frac{m}{r}f.   
\end{equation}

Next, combining \eqref{curl}--\eqref{div} leads to the following two identities:
\begin{equation*}
\begin{aligned}
&-\Delta \boldsymbol{f} = -\nabla\diver\boldsymbol{f}-\nabla^\perp\mathrm{curl}\,\boldsymbol{f}=-\nabla\diver\boldsymbol{f}&&\quad\text{if $n=2$},\\
&-\Delta \boldsymbol{f} = -\nabla\diver\boldsymbol{f}+\nabla\times\mathrm{curl}\,\boldsymbol{f}=-\nabla\diver\boldsymbol{f}&&\quad\text{if $n=3$},
\end{aligned}    
\end{equation*}
where $\nabla^\perp=(-\partial_2,\partial_1)$. This implies that 
\begin{equation*}
\nabla \boldsymbol{f}= -\nabla(-\Delta)^{-1}\nabla\diver \boldsymbol{f},
\end{equation*}
where $(-\Delta)^{-1}$ is defined via the Fourier transform: 
\begin{equation*}
\big((-\Delta)^{-1}g\big)(\boldsymbol{x})
:=\mathcal{F}^{-1}\big[\frac{1}{4\pi|\boldsymbol{\omega}|^2}\mathcal{F}[g](\boldsymbol{\omega})\big] (\boldsymbol{x}),
\end{equation*}
with $\mathcal{F}[g](\boldsymbol{\omega})$ as
the Fourier transform of $g\in C^\infty_{\rm c}(\mathbb{R}^n)$, {\it i.e.} 
\begin{equation*}
\mathcal{F}[g](\boldsymbol{\omega})=\int_{\mathbb{R}^n} g(\boldsymbol{x})e^{-2\pi\mathrm{i}\boldsymbol{x}\cdot\boldsymbol{\omega}}\,\mathrm{d}\boldsymbol{x},
\end{equation*}
and $\mathcal{F}^{-1}[g](\boldsymbol{x})$ as its inverse Fourier transform.

It follows from the Mihlin-H\"ormander multiplier theorem (see \cite[Theorem 6.2.7]{lg}) that 
\begin{equation}\label{63}
\|\diver\boldsymbol{f}\|_{L^p}\leq C_0\|\nabla \boldsymbol{f}\|_{L^p}\leq C(p)\|\diver \boldsymbol{f}\|_{L^p},
\end{equation}
which, together with \eqref{div}, leads to
\begin{equation}
\|\nabla \boldsymbol{f}\|_{L^p}\sim \Big|r^\frac{m}{p}\big(f_r+\frac{m}{r}f\big)_r\Big|_{p}.
\end{equation}

Finally, for the higher derivatives, one can get from \eqref{63} that
\begin{equation} 
\|\nabla^{j+1} \boldsymbol{f}\|_{L^p}\sim \|\nabla^j\diver \boldsymbol{f}\|_{L^p} \qquad \text{for $j=1,2,3$},
\end{equation}
which, along with Lemma \ref{lemma-initial} in Appendix \ref{appb}, gives the desired results. 
\end{proof}

\subsection{Zeroth- and first-order estimates of the velocity}
The first lemma concerns the zeroth-order energy estimate for $u$.
\begin{lem}\label{l4.5}
There exists a constant $C(T)>0$ such that, for any $t\in [0,T]$,
\begin{equation*}
|r^{\frac{m}{2}}(u,\phi v)(t)|_2^2 +\int_0^t \Big|r^{\frac{m}{2}}\big(u_r,\frac{u}{r}\big)\Big|_2^2 \, \mathrm{d}s\leq C(T).
\end{equation*}
\end{lem}

\begin{proof} We divide the proof into two steps.

\smallskip
\textbf{1.} Multiplying $\eqref{e2.2}_2$ by $r^m u$ gives
\begin{equation}\label{eq:B7-pre}
\begin{aligned}
&\frac{1}{2}(r^m u^2)_t+2\alpha r^m\big(u_r+\frac{m}{r}u\big)^2\\
&=2\alpha\Big(\underline{r^m u\big(u_r+\frac{m}{r}u\big)}_{:=\mathcal{B}_7}\Big)_r -r^m \big(u u_r+\phi_r -2\alpha \psi u_r\big)u.
\end{aligned}
\end{equation}
 Then we need to show that $\mathcal{B}_7\in W^{1,1}(I)$ and $\mathcal{B}_7|_{r=0}=0$ for {\it a.e.} $t\in (0,T)$, 
 which allows us to apply Lemma \ref{calculus} to obtain
\begin{equation}\label{eq:B7}
\int_0^\infty (\mathcal{B}_7)_r\,\mathrm{d}r=-\mathcal{B}_7|_{r=0}=0.
\end{equation}

On one hand, $\mathcal{B}_7|_{r=0}=0$ follows from $u|_{r=0}=0$ and $(u,u_r)\in C(\bar I)$ for each $t\in (0,T]$ due to \eqref{spd2}. 
On the other hand,  based on \eqref{spd}, we see that $r^\frac{m}{2}(u,\frac{u}{r},u_r,u_{rr})\in L^2(I)$ for each $t\in (0,T]$.
Then it follow from the H\"older inequality that 
\begin{align*}
|\mathcal{B}_7|_1&\leq C_0|r^\frac{m}{2}u|_2\Big|r^\frac{m}{2}\big(u_r,\frac{u}{r}\big)\Big|_2<\infty, \\
|(\mathcal{B}_7)_r|_1&\leq C_0\big|\big(r^{m-1}uu_r,r^m u_r^2,r^m uu_{rr},r^{m-2}u^2\big)\big|_1\\
&\leq C_0 \Big|r^\frac{m}{2}\big(\frac{u}{r},u_r\big)\Big|_2|r^\frac{m}{2}u_r|_2+C_0|r^\frac{m}{2}u|_2|r^\frac{m}{2}u_{rr}|_2+C_0|r^\frac{m-2}{2}u|_2^2<\infty.
\end{align*}    

Thus, integrating \eqref{eq:B7-pre} over $I$, together with \eqref{eq:B7}, yields
\begin{equation}\label{e4.15}
\frac12\frac{\mathrm{d}}{\mathrm{d}t} |r^{\frac{m}{2}}u|_2^2 + 2\alpha\Big|r^{\frac{m}{2}}\big(u_r+\frac{m}{r}u\big)\Big|_2^2=-\int_0^\infty r^m \big(uu_r+\phi_r- 2\alpha\psi u_r\big) u\mathrm{d}r:=\sum_{i=6}^{8} \mathcal{G}_i.
\end{equation}
Then, for $\mathcal{G}_6$--$\mathcal{G}_8$, it follows from \eqref{v-express-2}, Lemmas \ref{important2}, \ref{l4.4}, and \ref{im-1}, 
and the H\"older and Young inequalities that
\begin{equation}\label{G6-G8}
\begin{aligned}
\mathcal{G}_{6}&=-\int_0^\infty r^m u^2 u_r\, \mathrm{d}r \leq  C_0|u|_\infty|r^\frac{m}{2}u|_2 |r^\frac{m}{2}u_r|_2 \\
&\leq C_0|u|_\infty^2 |r^{\frac{m}{2}}u|_2^2+\frac{\alpha}{8}\Big|r^{\frac{m}{2}}\big(u_r+\frac{m}{r}u\big)\Big|_2^2,\\
\mathcal{G}_{7}&=- \int_0^\infty r^m \phi_r u\,\mathrm{d}r=-\frac{\gamma-1}{2\alpha}\int_0^\infty r^m \phi(v-u) u\, \mathrm{d}r\\
&\leq  C_0|r^\frac{m}{2}\phi v|_{2} |r^{\frac{m}{2}}u|_{2} +C_0|\phi|_{\infty} |r^{\frac{m}{2}}u|_{2}^2 \leq  C(T)|r^\frac{m}{2}(\phi v,u)|_2^2,\\
\mathcal{G}_{8}&= 2\alpha\int_0^\infty r^m \psi u u_r\,\mathrm{d}r=\int_0^\infty r^m v u_r u\,\mathrm{d}r+\mathcal{G}_6\\
&\leq C_0|v|_{\infty}|r^{\frac{m}{2}}u|_{2} |r^{\frac{m}{2}}u_r|_{2}+\mathcal{G}_6\\
&\leq  C(T)(1+ |u|_\infty^2)|r^{\frac{m}{2}}u|_{2}^2+\frac{\alpha}{8}\Big|r^{\frac{m}{2}}\big(u_r+\frac{m}{r}u\big)\Big|_2^2.
\end{aligned}    
\end{equation}
Combining with \eqref{e4.15}--\eqref{G6-G8}, along with Lemma \ref{im-1}, leads to
\begin{equation}\label{eq4.16}
\frac{\mathrm{d}}{\mathrm{d}t} |r^{\frac{m}{2}}u|_2^2 + \alpha\Big|r^{\frac{m}{2}}\big(u_r,\frac{u}{r}\big)\Big|_2^2 \leq C(T)(1+|u|_\infty^2) |r^\frac{m}{2}(\phi v,u)|_2^2.
\end{equation}

\smallskip
\textbf{2.} For the $L^2(I)$-estimate of $r^\frac{m}{2}\phi v$,  
multiplying \eqref{eq:effective1} by $r^m\phi^2v$ and then using $\eqref{e2.2}_1$ yield
\begin{equation}\label{b8pre}
\begin{aligned}
&\frac{1}{2} (r^m \phi^2v^2)_t+\frac{\gamma-1}{2\alpha}r^m \phi^3 v^2\\
&=-\frac{1}{2}(\underline{r^m u\phi^2v^2}_{:=\mathcal{B}_8})_r+\big(\frac{3}{2}-\gamma\big)r^m\phi^2v^2\big(u_r+\frac{m}{r}u\big)+\frac{\gamma-1}{2\alpha}r^m \phi^3 vu.
\end{aligned}
\end{equation}
Next, we need to  show that $\mathcal{B}_8\in W^{1,1}(I)$ and $\mathcal{B}_8|_{r=0}=0$ for {\it a.e.} $t\in (0,T)$, which allows us to apply Lemma \ref{calculus} to obtain
\begin{equation}\label{eq:B8}
\int_0^\infty (\mathcal{B}_8)_r\,\mathrm{d}r=-\mathcal{B}_8|_{r=0}=0.
\end{equation}

On one hand, $\mathcal{B}_8|_{r=0}=0$ follows from \eqref{psi,wuqiong}, and 
$u|_{r=0}=0$ and $(\phi,u)\in C(\bar I)$ for {\it a.e.} $t\in (0,T]$ due to \eqref{spd}. On the other hand,  \eqref{spd}--\eqref{spd2} imply that
\begin{equation*}
\Big(\phi,u,\frac{u}{r},u_r,\psi\Big)\in L^\infty(I),\quad  
r^\frac{m}{2}(\phi_r,u,u_r,\psi_{r})\in L^2(I)
\qquad \mbox{for {\it a.e.} $t\in (0,T)$},
\end{equation*}
Then it follows from \eqref{v-express-2} 
and the H\"older inequality that
\begin{align*}
|\mathcal{B}_8|_1&\leq C_0|u|_\infty\big(|\phi|_\infty^2|r^\frac{m}{2}u|_2^2+ |r^\frac{m}{2}\phi_r|_2^2\big) <\infty, \\
|(\mathcal{B}_8)_r|_1&\leq C_0\big|\big(r^{m-1}u\phi^2 v^2,r^m u_r\phi^2 v^2,r^m u\phi \phi_r v^2,r^{m}u\phi^2 vv_r\big)\big|_1\\
&\leq C_0\Big|\big(\frac{u}{r},u_r\big)\Big|_\infty\big(|\phi|_\infty^2|r^\frac{m}{2}u|_2^2+ |r^\frac{m}{2}\phi_r|_2^2\big)+C_0|\phi|_\infty|(u,\psi)|_\infty^2 |r^\frac{m}{2}u|_2|r^\frac{m}{2}\phi_r|_2\\
&\quad + C_0|\phi|_\infty^2|(u,\psi)|_\infty|r^\frac{m}{2}u|_2|r^\frac{m}{2}(u_r,\psi_r)|_2<\infty. 
\end{align*}    

Thus, integrating \eqref{b8pre} over $I$, we obtain from  \eqref{eq:B8}, Lemmas \ref{important2}, and \ref{l4.4}, 
and the H\"older and Young inequalities that
\begin{equation}\label{555}
\begin{aligned}
&\frac{1}{2}\frac{\mathrm{d}}{\mathrm{d}t}|r^\frac{m}{2}\phi v|_2^2+ \frac{\gamma-1}{2\alpha} \big|r^\frac{m}{2}\phi^\frac{3}{2} v\big|_2^2\\
&=\Big(\frac{3}{2}-\gamma\Big)\int_0^\infty  r^m\phi^2 v^2 \big(u_r+\frac{m}{r}u\big)\,\mathrm{d}r +\frac{\gamma-1}{2\alpha}\int_0^\infty r^m\phi^3 vu\,\mathrm{d}r\\
&\leq  C_0|\phi|_\infty|v|_\infty |r^\frac{m}{2}\phi v|_2 \Big|r^\frac{m}{2}\big(u_r,\frac{u}{r}\big)\Big|_2+C_0|\phi|_\infty^2 |r^\frac{m}{2}\phi v|_2 |r^\frac{m}{2}u|_2\\
&\leq  C(T)|r^\frac{m}{2}(\phi v,u)|_2^2+\frac{\alpha}{8} \Big|r^\frac{m}{2}\big(u_r,\frac{u}{r}\big)\Big|_2^2,
\end{aligned}    
\end{equation}
which, together with \eqref{eq4.16}, gives
\begin{equation}\label{eq4.17}
\frac{\mathrm{d}}{\mathrm{d}t}|r^\frac{m}{2}(\phi v,u)|_2^2 + \frac{\alpha}{2}\Big|r^{\frac{m}{2}}\big(u_r,\frac{u}{r}\big)\Big|_2^2\leq C(T)(1+|u|_\infty^2)|r^\frac{m}{2}(\phi v,u)|_2^2.
\end{equation}

Applying  the Gr\"onwall  inequality to the above inequality, together with Lemma \ref{ele}, yields that, for all $t\in [0,T]$,
\begin{equation}\label{e4.16*}
|r^\frac{m}{2}(\phi v,u)|_2^2+\int_0^t \Big|r^{\frac{m}{2}}\big(u_r,\frac{u}{r}\big)\Big|_2^2 \,\mathrm{d}s\leq C(T)|r^\frac{m}{2}(\phi_0 v_0,u_0)|_2^2\leq C(T).
\end{equation}
Here, it remains to check the $L^2(I)$-boundedness of $r^\frac{m}{2}(\phi_0 v_0,u_0)$. Indeed, it follows from Lemmas \ref{ale1}, \ref{initial3}, and  \ref{lemma-initial} that
\begin{equation}\label{chuzhi0}
\begin{aligned}
|r^\frac{m}{2}(\phi_0v_0,u_0)|_2 &\leq C_0|r^\frac{m}{2}(\phi_0u_0,(\phi_0)_r,u_0)|_2 \leq C_0\big(\|\phi_0\boldsymbol{u}_0\|_{L^2}+\|\nabla \phi_0\|_{L^2}+\|\boldsymbol{u}_0\|_{L^2}\big)\\
&\leq C_0\big(\|\phi_0\|_{L^\infty}\|\boldsymbol{u}_0\|_{L^2}+\|\nabla \phi_0\|_{L^2}+\|\boldsymbol{u}_0\|_{L^2}\big)\leq C_0. 
\end{aligned}
\end{equation}
The proof of Lemma \ref{l4.5} is completed.
\end{proof}

The following lemma concerns the first-order energy estimate of $u$.
\begin{lem}\label{l4.6}
There exists a constant $C(T)>0$ such that, for any $t\in [0,T]$,
\begin{equation*} 
\Big|r^{\frac{m}{2}}\big(u_r,\frac{u}{r}\big)(t)\Big|_2^2+\int_0^t |r^{\frac{m}{2}}u_t|_2^2\,\mathrm{d}s\leq C(T).
\end{equation*}
\end{lem}
\begin{proof}
First, multiplying $\eqref{e2.2}_2$ by $r^m u_t$ gives
\begin{equation}\label{eq:B9pre}
\begin{aligned}
&r^m u_t^2 +\alpha \Big(r^m\big(u_r+ \frac{m}{r}u\big)^2\Big)_t\\
&=2\alpha \Big(\underline{r^m u_t\big(u_r+ \frac{m}{r}u\big)}_{:=\mathcal{B}_9}\Big)_r- r^m \big(u u_r +\phi_r -2\alpha \psi u_r\big)u_t.    
\end{aligned}    
\end{equation}
Next, we need to  show that $\mathcal{B}_9\in W^{1,1}(I)$ and $\mathcal{B}_9|_{r=0}=0$ for {\it a.e.} $t\in (0,T)$, which allows us to apply Lemma \ref{calculus} to obtain
\begin{equation}\label{eq:B9}
\int_0^\infty (\mathcal{B}_9)_r\,\mathrm{d}r=-\mathcal{B}_9|_{r=0}=0.
\end{equation}

On one hand,   by \eqref{spd}, $r^\frac{m}{2}(u_t,\frac{u_t}{r},u_{tr})\in L^2([0,T];L^2(I))$ so that $r^\frac{m}{2}u_t\in L^2([0,T];H^1(I))$, which implies from Lemma \ref{ale1} 
and $u|_{r=0}=0$ that 
\begin{equation}\label{ut,wuqiong}
r^\frac{m}{2}u_t\in L^2([0,T];C(\bar I))\implies r^\frac{m}{2}u_t|_{r=0}=0 \qquad \text{ for {\it a.e.} $t\in (0,T)$}.   
\end{equation}
This,  together with $(u_r,\frac{u}{r})\in C(\bar I)$ for $t\in (0,T]$ due to \eqref{spd2}, yields that $\mathcal{B}_9|_{r=0}=0$ for {\it a.e.} $t\in (0,T)$. On the other hand,  to obtain $\mathcal{B}_9\in W^{1,1}(I)$, by \eqref{spd}, we have 
\begin{equation*}
r^\frac{m}{2}\Big(\frac{u}{r},u_r,u_t,\big(\frac{u}{r}\big)_r,u_{rr},\frac{u_t}{r},u_{tr}\Big)\in L^2(I)    \qquad \text{for {\it a.e.} $t\in (0,T)$}.
\end{equation*}
This implies from the H\"older inequality that 
\begin{align*}
|\mathcal{B}_9|_1&\leq C_0|r^\frac{m}{2}u_t|_2\Big|r^\frac{m}{2}\big(u_r,\frac{u}{r}\big)\Big|_2<\infty,\\
|(\mathcal{B}_9)_r|_1&\leq C_0\Big|r^{m-1}u_t\big(u_r,\frac{u}{r}\big)\Big|_1+C_0\Big|r^{m}u_{tr}\big(u_r,\frac{u}{r}\big)\Big|_1+C_0\Big|r^{m}u_{t}\Big(u_{rr},\big(\frac{u}{r}\big)_r\Big)\Big|_1\\
&\leq C_0 \Big|r^\frac{m}{2}\big(u_{tr},\frac{u_t}{r}\big)\Big|_2\Big|r^\frac{m}{2}\big(u_r,\frac{u}{r}\big)\Big|_2+C_0|r^\frac{m}{2}u_t|_2\Big|r^\frac{m}{2}\Big(u_{rr},\big(\frac{u}{r}\big)_r\Big)\Big|_2<\infty.
\end{align*}

Therefore, integrating \eqref{eq:B9pre} over $I$, together with \eqref{eq:B9}, yields
\begin{equation}\label{e-1.20}
\alpha\frac{\mathrm{d}}{\mathrm{d}t}\Big|r^\frac{m}{2}\big(u_r+\frac{m}{r}u\big)\Big|_2^2+| r^{\frac{m}{2}}u_t|_2^2=-\int_0^\infty r^m\big(u u_r +\phi_r- 2\alpha \psi u_r\big) u_t\,\mathrm{d}r:=\sum_{i=9}^{11} \mathcal{G}_i.
\end{equation}

Then, for $\mathcal{G}_9$--$\mathcal{G}_{11}$, it follows from
\eqref{v-express-2}, Lemmas \ref{important2}, \ref{l4.4}, 
and \ref{l4.5}, and the H\"older and Young inequalities that
\begin{equation}\label{G9-G11}
\begin{aligned}
\mathcal{G}_{9}&=-\int_0^\infty r^mu u_r u_t\,\mathrm{d}r\leq C_0|u|_{\infty}|r^{\frac{m}{2}}u_r|_{2}|r^{\frac{m}{2}}u_t|_{2}\leq  C_0|u|_\infty^2|r^{\frac{m}{2}}u_r|_{2}^2+\frac{1}{8} |r^{\frac{m}{2}}u_t|_{2}^2,\\
\mathcal{G}_{10}&=- \int_0^\infty r^m \phi_r u_t\,\mathrm{d}r
= -\frac{\gamma-1}{2\alpha}\int r^m \phi (v-u) u_t\,\mathrm{d}r\\
&\leq  C_0\big(|r^{\frac{m}{2}} \phi v|_{2} + |\phi|_\infty |r^{\frac{m}{2}}u|_2\big) |r^{\frac{m}{2}}u_t|_2\leq C(T)+\frac{1}{8} |r^{\frac{m}{2}}u_t|^2_2,\\
\mathcal{G}_{11}&= 2\alpha\int_0^\infty r^m\psi u_r u_t\,\mathrm{d}r= \int_0^\infty r^mvu_r u_t\,\mathrm{d}r+ \mathcal{G}_9\\
&\leq C_0|v|_{\infty} |r^{\frac{m}{2}}u_r|_{2}|r^{\frac{m}{2}}u_t|_{2}+\mathcal{G}_9\leq C(T)(1+|u|_{\infty}^2) |r^{\frac{m}{2}}u_r|_2^2+\frac18|r^{\frac{m}{2}}u_t|_{2}^2.    
\end{aligned}    
\end{equation}

Substituting \eqref{G9-G11} into \eqref{e-1.20}, we obtain 
from Lemma \ref{im-1} that
\begin{equation}\label{e-1.23}
\begin{aligned}
\alpha\frac{\mathrm{d}}{\mathrm{d}t}\Big|r^\frac{m}{2}\big(u_r+\frac{m}{r}u\big)\Big|_2^2+|r^{\frac{m}{2}}u_t|_2^2&\leq C(T)(1+|u|_\infty^2)|r^\frac{m}{2}u_r|_2^2+C(T)\\
&\leq C(T)(1+|u|_\infty^2)\Big|r^\frac{m}{2}\big(u_r+\frac{m}{r}u\big)_r\Big|_2^2+C(T),
\end{aligned}
\end{equation}
which, along with Lemmas \ref{ele} and \ref{im-1}, and the Gr\"onwall inequality, yields that, for all $t\in [0,T]$, 
\begin{equation}\label{eg}
\Big|r^{\frac{m}{2}}\big(u_r,\frac{u}{r}\big)(t)\big|_2^2+\int_0^t  |r^{\frac{m}{2}}u_t|_2^2\,\mathrm{d}s\leq C(T)\Big|r^{\frac{m}{2}}\big((u_0)_r,\frac{u_0}{r}\big)\Big|_2^2+C(T)\leq C(T).  
\end{equation}
This completes the proof. 
\end{proof}

\subsection{The  second-order estimates of the velocity}
\begin{lem}\label{l4.8}
There exists a constant $C(T)>0$ such that
\begin{equation*}
|r^{\frac{m}{2}}u_t(t)|^2_2+\int_0^t \Big|r^{\frac{m}{2}}\big(u_{tr},\frac{u_t}{r}\big)\Big|_2^2 \,\mathrm{d}s \leq C(T)
\qquad \mbox{for any $t\in[0,T]$}.
\end{equation*}
\end{lem}
 
\begin{proof} We divide the proof into two steps.

\smallskip
\textbf{1.} It follows from Lemmas \ref{im-1}--\ref{l4.5}, \ref{GN-ineq}, and \ref{lemma-initial} that
\begin{equation}\label{u-infty}
\begin{aligned}
|u|_{\infty}&=\|\boldsymbol{u}\|_{L^\infty}\leq C_0\|\boldsymbol{u}\|_{L^2}^\frac{4-n}{4}\|\nabla^2\boldsymbol{u}\|_{L^2}^\frac{n}{4}\\
&\leq C_0|r^\frac{m}{2}u|_2^\frac{4-n}{4}\Big|r^\frac{m}{2}\big(u_r+\frac{m}{r}u\big)_r\Big|_2^\frac{n}{4}\leq C(T)\Big|r^\frac{m}{2}\big(u_r+\frac{m}{r}u\big)_r\Big|_2^\frac{n}{4}.
\end{aligned}
\end{equation}
Then, multiplying $\eqref{e2.2}_2$ by $r^\frac{m}{2}$ and taking the $L^2(I)$-norm of the resulting equality, we obtain from 
from \eqref{u-infty}, Lemmas \ref{important2}, \ref{l4.4}, 
and \ref{l4.5}--\ref{l4.6}, and the Young inequality that   
\begin{align*}
\Big|r^{\frac{m}{2}}\big(u_r+\frac{m}{r}u\big)_r\Big|_2&\leq C_0\big(|r^{\frac{m}{2}}u_t|_2+|r^{\frac{m}{2}} uu_r|_2+|r^{\frac{m}{2}} \phi_r|_2+ |r^{\frac{m}{2}} \psi u_r|_2\big)\\
&\leq  C_0\big(|r^{\frac{m}{2}}u_t|_2+|r^{\frac{m}{2}} uu_r|_2+|r^{\frac{m}{2}} \phi(v-u)|_2+|r^{\frac{m}{2}} (v-u) u_r|_2\big) \\
&\leq  C_0\big(|r^{\frac{m}{2}}u_t|_2+|r^{\frac{m}{2}}u_r|_2|(v,u)|_\infty+|r^\frac{m}{2}\phi v|_2 +|\phi|_\infty|r^{\frac{m}{2}}u|_2\big)\\
&\leq C(T)\bigg(|r^{\frac{m}{2}}u_t|_2+\Big|r^\frac{m}{2}\big(u_r+\frac{m}{r}u\big)_r\Big|_2^\frac{n}{4}+1\bigg)\\
&\leq C(T)\big(|r^{\frac{m}{2}}u_t|_2+1\big)+\frac{1}{2}\Big|r^\frac{m}{2}\big(u_r+\frac{m}{r}u\big)_r\Big|_2,
\end{align*}
which, along with \eqref{u-infty} and Lemma \ref{im-1}, leads to
\begin{equation}\label{618}
|u|_\infty+\Big|r^{\frac{m}{2}}\Big(u_{rr},\big(\frac{u}{r}\big)_{r}\Big)\Big|_2\leq C(T)\big(|r^{\frac{m}{2}}u_t|_2+1\big).
\end{equation}

Next, it follows from \eqref{618}, Lemmas \ref{ele}, \ref{l4.6}, \ref{ale1}, and \ref{hardy}, and the H\"older and Young inequalities that, for all $t\in[0,T]$,
\begin{equation}\label{6.4-d}
\begin{aligned}
|r^\frac{m}{2}u_r|_\infty&\leq |\chi_1^\flat r^\frac{m}{2}u_r|_\infty+|\chi_1^\sharp r^\frac{m}{2}u_r|_\infty\\
&\leq C_0\big(|\chi_1^\flat r^\frac{m+1}{2} u_r|_2+|\chi_1^\flat r^\frac{m+1}{2} u_{rr}|_2+|\chi_1^\sharp r^{\frac{m}{2}}u_r|_2+|\chi_1^\sharp(r^{\frac{m}{2}}u_r)_r|_2\big)\\
&\leq C_0\big(|\chi^\flat r^\frac{1}{2}|_\infty+|\chi^\sharp r^{-1}|_\infty+1\big)|r^{\frac{m}{2}}u_r|_2+C_0(|\chi^\flat r^\frac{1}{2}|_\infty+1)|r^{\frac{m}{2}}u_{rr}|_2\\
&\leq C(T)\big(|r^\frac{m}{2}u_{rr}|_2+1\big)\leq C(T)\big(|r^\frac{m}{2}u_{t}|_2+1\big).
\end{aligned}
\end{equation}

Finally,  according to \eqref{v-express-2},  Lemmas \ref{important2}, \ref{l4.4}, and \ref{l4.5}, we have 
\begin{equation}\label{6.6-1}
|r^{\frac{m}{2}}\phi_r|_2\leq C_0|r^{\frac{m}{2}}\phi(v-u)|_2 \leq C_0\big(|r^\frac{m}{2}\phi v|_2+|r^{\frac{m}{2}}u|_2|\phi|_\infty\big)\leq C(T).
\end{equation}
Multiplying $\eqref{e2.2}_1$ by $r^\frac{m}{2}$ and taking the $L^2(I)$-norm of the resulting equation, we obtain from \eqref{618}, 
\eqref{6.6-1}, and Lemmas \ref{important2} and \ref{l4.6} that 
\begin{equation}\label{6.6-2}
\begin{aligned}
|r^{\frac{m}{2}}\phi_t|_2&\leq  |r^{\frac{m}{2}}u\phi_r|_2+(\gamma-1)\Big|r^\frac{m}{2}\phi\big(u_r+\frac{m}{r}u\big)\Big|_2\\
&\leq |u|_\infty|r^{\frac{m}{2}}\phi_r|_2+C_0|\phi|_\infty \Big|r^{\frac{m}{2}}\big(u_{r},\frac{u}{r}\big)\Big|_2 \le C(T)\big(|r^{\frac{m}{2}}u_t|_2+1\big).
\end{aligned}    
\end{equation}

\smallskip
\textbf{2.}
Applying $r^mu_t\partial_t$ to both sides of $\eqref{e2.2}_2$, together with $\eqref{e2.2}_3$, gives
\begin{equation}\label{eq:B10pre}
\begin{aligned}
&\frac{1}{2}(r^m u_t^2)_t+2\alpha r^m\big(u_{tr}+ \frac{m}{r}u_t\big)^2\\
&=\Big(2\alpha r^m u_t \big(u_{tr}+ \frac{m}{r}u_t\big)-r^mu_t \phi_{t}\Big)_{r}- r^mu_t(uu_r)_t+r^m\phi_{t}\big(u_{tr}+\frac{m}{r}u_t\big)\\
&\quad+2\alpha r^m \psi u_t u_{tr}+2\alpha r^m  \psi_t u_tu_r \\
&=\Big(2\alpha r^m u_t \big(u_{tr}+ \frac{m}{r}u_t\big)-r^mu_t \phi_{t}\Big)_{r}- r^mu_t(uu_r)_t+r^m\phi_{t}\big(u_{tr}+\frac{m}{r}u_t\big)\\
&\quad +2\alpha r^m \psi u_t u_{tr}-2\alpha r^m(\psi u)_ru_tu_r -2\alpha r^m \big(u_t+\frac{m}{r}u\big)_ru_tu_r\\
&=\Big(\underline{2\alpha r^m u_t \big(u_{tr}+ \frac{m}{r}u_t\big)-r^mu_t \phi_{t}-2\alpha r^m \psi u u_tu_r}_{:=\mathcal{B}_{10}}\Big)_{r}- r^mu_t(uu_r)_t\\
&\quad +r^m\phi_{t}\big(u_{tr}+\frac{m}{r}u_t\big)+2\alpha r^m \psi u_t u_{tr}\\
& \quad +2\alpha r^m \psi u\big(u_{tr}u_r+u_tu_{rr}+\frac{m}{r}u_tu_r\big) -2\alpha r^m \big(u_r+\frac{m}{r}u\big)_ru_tu_r.
\end{aligned}    
\end{equation}
Next, we need to  show that $\mathcal{B}_{10}\in W^{1,1}(I)$ and $\mathcal{B}_{10}|_{r=0}=0$ for {\it a.e.} $t\in (0,T)$, which allows us to apply Lemma \ref{calculus} to obtain
\begin{equation}\label{eq:B10}
\int_0^\infty (\mathcal{B}_{10})_r\,\mathrm{d}r=-\mathcal{B}_{10}|_{r=0}=0.
\end{equation}

We first show that $\mathcal{B}_{10}|_{r=0}=0$. Indeed, according to \eqref{psi,wuqiong}, \eqref{ut,wuqiong}, and $(u,u_r)\in C(\bar I)$ for $t\in (0,T]$ due to \eqref{spd2}, it remains to show that 
\begin{equation}\label{727}
r^\frac{m}{2}\big(u_{tr}+\frac{m}{r}u_t\big)\Big|_{r=0}<\infty,\quad\ r^\frac{m}{2}\phi_t|_{r=0}<\infty \qquad \text{for {\it a.e.} $t\in (0,T)$}.
\end{equation}
Thanks to \eqref{spd}, we see that $r^\frac{m}{2}\big(u_{tr},\frac{u_t}{r},u_{trr},(\frac{u_t}{r})_r\big)\in L^2(I)$ for {\it a.e.} $t\in (0,T)$, which, along with Lemma \ref{im-1}, implies that 
\begin{equation}\label{837}
r^\frac{m}{2}\big(u_{tr}+\frac{m}{r}u_t,(u_{tr}+\frac{m}{r}u_t)_r\big) \in L^2(I) \qquad\,\,\ \text{for \ {\it a.e.} $t\in (0,T)$}.
\end{equation}

If $n=2$ ($m=1$), \eqref{837} implies that 
$r\big(u_{tr}+\frac{1}{r}u_t,(u_{tr}+\frac{1}{r}u_t)_r\big) \in L^2(0,1)$.
Thus, it follows from Lemma \ref{hardy} that 
\begin{equation*}
r^\frac{1}{2}\big(u_{tr}+\frac{1}{r}u_t\big)=r^\frac{m}{2}\big(u_{tr}+\frac{m}{r}u_t\big) \in C([0,1]) \,\qquad \text{for \ {\it a.e.} $t\in (0,T)$}.
\end{equation*}

If $n=3$ ($m=2$), we deduce from Lemmas \ref{ale1} and \ref{hardy} that, for {\it a.e.} $t\in (0,T)$,
\begin{equation*}
r\big(u_{tr}+\frac{2}{r}u_t\big)\in H^1(0,1) \implies r\big(u_{tr}+\frac{2}{r}u_t\big)\in C([0,1]),
\end{equation*}
that is, $r (u_{tr}+\frac{2}{r}u_t)|_{r=0}<\infty$ for {\it a.e.} $t\in (0,T)$. 
Similarly, the method mentioned above can be applied to show that 
$r^\frac{m}{2}\phi_t|_{r=0}<\infty$ for {\it a.e.} $t\in (0,T)$, since $r^\frac{m}{2}(\phi_t,\phi_{tr})\in L^2(I)$ for {\it a.e.} $t\in (0,T)$ due to \eqref{spd}. 

Next, we show that $\mathcal{B}_{10}\in W^{1,1}(I)$ for {\it a.e.} $t\in (0,T)$. By \eqref{spd}--\eqref{spd2}, we have
\begin{equation*}
(\psi,u,u_r)\in L^\infty(I),\,\,\,\, 
r^\frac{m}{2}\Big(\phi_t,u_r,u_t,u_{rr},\frac{u_t}{r},u_{tr},\big(\frac{u_t}{r}\big)_r,u_{trr},\psi_r\Big)\in L^2(I)
\qquad\mbox{for {\it a.e.} $t\in (0,T)$}.
\end{equation*}
Then it follows from the H\"older inequality that
\begin{equation}\label{cal-b10}
\begin{aligned}
|\mathcal{B}_{10}|_1&\leq C_0|r^\frac{m}{2}u_t|_2\Big(\Big|r^\frac{m}{2}\big(u_{tr},\frac{u_t}{r}\big)\Big|_2+|r^\frac{m}{2}\phi_t|_2+|\psi|_\infty|u|_\infty |r^\frac{m}{2}u_r|_2\Big)<\infty,\\
|(\mathcal{B}_{10})_r|_1&\leq C_0\Big|r^{m-1}u_t\big(u_{tr},\frac{u_t}{r},\phi_t,\psi uu_r\big)\Big|_1+C_0\Big|r^{m}u_{tr}\big(u_{tr},\frac{u_t}{r},\phi_t,\psi uu_r\big)\Big|_1\\
&\quad +C_0\Big|r^{m}u_{t}\Big(u_{trr},\big(\frac{u_t}{r}\big)_r,\phi_{tr},\psi_r uu_r,\psi (u_r)^2,\psi u u_{rr}\Big)\Big|_1\\
&\leq C_0 \Big|r^\frac{m}{2}\big(u_{tr},\frac{u_t}{r}\big)\Big|_2\Big(\Big|r^\frac{m}{2}\big(u_{tr},\frac{u_t}{r},\phi_t\big)\Big|_2+|\psi|_\infty |u|_\infty |r^\frac{m}{2}u_r|_2\Big)\\
&\quad +C_0|r^\frac{m}{2}u_t|_2\Big(\Big|r^\frac{m}{2}\Big(u_{trr},\big(\frac{u_t}{r}\big)_r,\phi_{tr}\Big)\Big|_2+|r^\frac{m}{2}\psi_r|_2|u|_\infty|u_r|_\infty\Big)\\
&\quad +C_0|r^\frac{m}{2}u_t|_2\big(|\psi|_\infty|u_r|_\infty|r^\frac{m}{2}u_r|_2+|\psi|_\infty|u|_\infty |r^\frac{m}{2}u_{rr}|_2\big)<\infty,
\end{aligned}    
\end{equation}
which implies the assertion.

Now, integrating \eqref{eq:B10pre} over $I$, together with \eqref{eq:B10}, yields
\begin{equation}\label{util2}
\begin{aligned}
&\frac{1}{2} \frac{\mathrm{d}}{\mathrm{d}t}|r^{\frac{m}{2}}u_t|^2_2+2\alpha \Big|r^{\frac{m}{2}}\big(u_{tr}+\frac{m}{r}u_t\big)\Big|_2^2\\
&=-\int_0^\infty r^m(uu_{r})_t u_t\,\mathrm{d}r+\int_0^\infty r^m \phi_t\big(u_{tr}+\frac{m}{r}u_t\big)\,\mathrm{d}r\\
&\quad + \!2\alpha\!\int_0^\infty \!r^m  \Big(\psi\big(u_t u_{tr}+ u(u_{tr}u_r+u_tu_{rr}+\frac{m}{r}u_tu_r)\big)\!-\!\big(u_r+\frac{m}{r}u\big)_ru_tu_r\!\Big)\mathrm{d}r\!:=\!\sum_{i=13}^{15} \mathcal{G}_i.
\end{aligned}    
\end{equation}

For $\mathcal{G}_{13}$--$\mathcal{G}_{14}$, it follows from \eqref{618}--\eqref{6.4-d}, \eqref{6.6-2}, Lemma \ref{im-1}, and the H\"older and Young inequalities that
\begin{equation}\label{G13-G14}
\begin{aligned}
\mathcal{G}_{13}& =-\int_0^\infty r^m(uu_{tr}+u_r u_t)u_t\,\mathrm{d}r\\
&\leq |u|_{\infty}|r^{\frac{m}{2}}u_{tr}|_2|r^{\frac{m}{2}}u_t|_2+ \big(|\chi_1^\flat r^\frac{m}{2} u_ru_t|_{2} +  |\chi_1^\sharp r^\frac{m}{2} u_ru_t|_{2}\big) |r^{\frac{m}{2}}u_t|_2\\
&\leq |u|_{\infty}|r^{\frac{m}{2}}u_{tr}|_2|r^{\frac{m}{2}}u_t|_2+  |\chi_1^\flat r^\frac{2-m}{2}|_\infty|r^\frac{m}{2} u_r|_{\infty} |r^{\frac{m-2}{2}}u_t|_2 |r^{\frac{m}{2}}u_t|_2\\
&\quad +  |\chi_1^\sharp r^{-\frac{m}{2}}|_\infty |r^\frac{m}{2}u_r|_{\infty}|r^{\frac{m}{2}}u_t|_2^2\\
&\leq C_0\big(|(u,r^\frac{m}{2}u_r)|_\infty^2+1\big)|r^{\frac{m}{2}}u_t|_2^2 +\frac{\alpha}{8}\Big|r^{\frac{m}{2}}\big(u_{tr}+\frac{m}{r}u_t\big)\Big|_2^2\\
&\leq C(T)\big(|r^\frac{m}{2}u_t|_2^4+1\big) +\frac{\alpha}{8}\Big|r^{\frac{m}{2}}\big(u_{tr}+\frac{m}{r}u_t\big)\Big|_2^2,\\
\mathcal{G}_{14}&=\int_0^\infty r^m \phi_t\big(u_{tr}+\frac{m}{r}u_t\big)\,\mathrm{d}r
\leq  |r^{\frac{m}{2}}\phi_t|_2 \Big|r^{\frac{m}{2}}\big(u_{tr}+\frac{m}{r}u_t\big)\Big|_2\\
&\leq C(T)\big(|r^{\frac{m}{2}}u_t|_2^2+1\big) +\frac{\alpha}{8}\Big|r^{\frac{m}{2}}\big(u_{tr}+\frac{m}{r}u_t\big)\Big|_2^2\\
&\leq C(T)\big(|r^{\frac{m}{2}}u_t|_2^4+1\big) +\frac{\alpha}{8}\Big|r^{\frac{m}{2}}\big(u_{tr}+\frac{m}{r}u_t\big)\Big|_2^2.
\end{aligned}    
\end{equation}

For $\mathcal{G}_{15}$, it follows from \eqref{v-express-2}, \eqref{618}, Lemmas \ref{im-1} and \ref{l4.6},  and the H\"older and Young inequalities that 
\begin{align}
\mathcal{G}_{15} &=\int_0^\infty  r^m (v-u) u_t u_{tr}\,\mathrm{d}r+\int_0^\infty r^m(v-u)  u\big(u_{tr}u_r+u_tu_{rr}+\frac{m}{r}u_tu_r\big) \,\mathrm{d}r\notag\\
&\quad - 2\alpha\int_0^\infty r^m\big(u_r+\frac{m}{r}u\big)_ru_tu_r \,\mathrm{d}r \notag\\
&\leq |(v,u)|_\infty|r^{\frac{m}{2}}u_t|_{2}|r^{\frac{m}{2}}u_{tr}|_2+ |(v,u)|_\infty|u|_\infty|r^\frac{m}{2}u_{rr}|_2|r^\frac{m}{2}u_t|_2 \notag\\
&\quad + |(v,u)|_\infty |u|_\infty |r^{\frac{m}{2}}u_{r}|_2 \Big|r^\frac{m}{2}\big(u_{tr}+\frac{m}{r}u_t\big)\Big|_2 \notag\\
&\quad + C_0\Big|r^\frac{m}{2}\Big(u_{rr},\big(\frac{u}{r}\big)_r\Big)\Big|_2\big(|\chi_1^\flat r^\frac{m}{2} u_r u_t|_2+|\chi_1^\sharp r^\frac{m}{2} u_r u_t|_2\big)\notag\\
&\leq |(v,u)|_\infty|r^{\frac{m}{2}}u_t|_{2}|r^{\frac{m}{2}}u_{tr}|_2+ |(v,u)|_\infty|u|_\infty|r^\frac{m}{2}u_{rr}|_2|r^\frac{m}{2}u_t|_2\label{G15} \\
&\quad + |(v,u)|_\infty |u|_\infty |r^{\frac{m}{2}}u_{r}|_2 \Big|r^\frac{m}{2}\big(u_{tr}+\frac{m}{r}u_t\big)\Big|_2\notag \\
&\quad + C_0\Big|r^\frac{m}{2}\Big(u_{rr},\big(\frac{u}{r}\big)_r\Big)\Big|_2|r^\frac{m}{2} u_r|_\infty\big(|\chi_1^\flat r^\frac{2-m}{2}|_\infty  |r^\frac{m-2}{2} u_t|_2+|\chi_1^\sharp r^{-\frac{m}{2}}|_\infty  |r^\frac{m}{2} u_t|_2\big)\notag\\
&\leq C(T)\big(|r^{\frac{m}{2}}u_t|_{2}^2+1\big)\Big(\Big|r^\frac{m}{2}\big(u_{tr}+\frac{m}{r}u_t\big)\Big|_2+ |r^\frac{m}{2} u_t|_2+1\Big)\notag\\
&\leq C(T)\big(|r^\frac{m}{2}u_t|_2^4+1\big)+\frac{\alpha}{8}\Big|r^\frac{m}{2}\big(u_{tr}+\frac{m}{r}u_t\big)\Big|_2^2.\notag
\end{align}

Substituting \eqref{G13-G14}--\eqref{G15} into \eqref{util2}, along with Lemma \ref{im-1}, gives
\begin{equation*} 
\frac{\mathrm{d}}{\mathrm{d}t}|r^{\frac{m}{2}}u_t|^2_2+\alpha\Big|r^{\frac{m}{2}}\big(u_{tr},\frac{u_t}{r}\big)\Big|_2^2\le C(T)\big(|r^{\frac{m}{2}}u_{t}|_2^4+1\big).
\end{equation*}
Integrating above over $(\tau, t)(\tau\in (0,t))$ yields
\begin{equation}\label{lg1}
|r^{\frac{m}{2}}u_t(t)|_2^2+\alpha\int_\tau^t\Big|r^{\frac{m}{2}}\big(u_{tr},\frac{u_t}{r}\big)\Big|_2^2\,\mathrm{d}s\leq  |r^{\frac{m}{2}}u_t(\tau)|_2^2+C(T)\int_\tau^t\big(|r^{\frac{m}{2}}u_{t}|_2^4+1\big)\,\mathrm{d}s.
\end{equation}
For the $L^2(I)$-boundedness of $r^{\frac{m}{2}}u_t(\tau)$ 
on the right-hand side of the above, we multiply $\eqref{e2.2}_2$ by $r^\frac{m}{2}$ and then take the $L^2(I)$-norm of the resulting equality to obtain 
\begin{equation*}
|r^{\frac{m}{2}}u_t(\tau)|_2\leq C_0\Big(|u|_\infty|r^{\frac{m}{2}}u_r|_2+|r^{\frac{m}{2}}\phi_r|_2+\Big|r^{\frac{m}{2}}\big(u_{r}+\frac{m}{r}u\big)_r\Big|_2+|\psi|_\infty|r^{\frac{m}{2}}u_r|_2\Big)(\tau),
\end{equation*}
which, along with the time-continuity of $(\phi,u,\psi)$ and 
Lemmas \ref{im-1}, \ref{ale1}, and \ref{lemma-initial}, yields
\begin{align*}
\limsup_{\tau\to 0} |r^{\frac{m}{2}}u_t(\tau)|_2&\leq C_0\Big(|u_0|_\infty|r^{\frac{m}{2}}(u_0)_r|_2+ |r^{\frac{m}{2}}(\phi_0)_r|_2\Big)\\
&\quad +C_0\Big(\Big|r^{\frac{m}{2}}\big((u_0)_{r}+\frac{m}{r}u_0\big)_r\Big|_2+ |\psi_0|_\infty|r^{\frac{m}{2}}(u_0)_r|_2\Big)\\
&\leq C_0 \big(\|\boldsymbol{u}_0\|_{L^\infty}\|\nabla\boldsymbol{u}_0\|_{L^2}+ \|\nabla\phi_0\|_{L^2}\big)\\
&\quad +C_0 \big(\|\nabla^2 \boldsymbol{u}_0\|_{L^2} + \|\boldsymbol{\psi}_0\|_{L^\infty} \|\nabla\boldsymbol{u}_0\|_{L^2}\big)\\
&\leq C_0\big(\|\boldsymbol{u}_0\|_{H^2}+\|\boldsymbol{\psi}_0\|_{L^\infty}+1\big)\|\boldsymbol{u}_0\|_{H^2} +C_0 \|\nabla\phi_0\|_{L^2} \leq C_0.
\end{align*}
Consequently, based on the above discussions, letting $\tau\to 0$ in \eqref{lg1}, then it follows from Lemma \ref{l4.6} and the Gr\"onwall inequality that, for all $t\in [0,T]$,
\begin{equation}\label{rut}
|r^{\frac{m}{2}}u_t(t)|^2_2+\int_0^t \Big|r^{\frac{m}{2}}\big(u_{tr},\frac{u_t}{r}\big)\Big|_2^2\, \mathrm{d}s \le C(T).
\end{equation}

This completes the proof of Lemma \ref{l4.8}.
\end{proof}

With the help of \eqref{618} and Lemma \ref{l4.8}, we can also obtain the following estimates:

\begin{lem}\label{lemma66}
There exists a constant $C(T)>0$ such that, for any $t\in [0,T]$,
\begin{equation*}
|(u,r^\frac{m}{2}u_r)(t)|_\infty + \Big|r^{\frac{m}{2}}\Big(\phi_r,\phi_t,u_{rr}, \big(\frac{u}{r}\big)_{r}\Big)(t)\Big|_2+ \int_0^t \Big|\big(u_r,\frac{u}{r}\big)\Big|_\infty^2\,\mathrm{d}s \leq C(T).
\end{equation*}
\end{lem}
\begin{proof}
First, it follows from \eqref{618}--\eqref{6.6-2} and Lemma \ref{l4.8} that, for all $t\in[0,T]$,
\begin{equation}\label{6026}
|(u,r^\frac{m}{2}u_r)(t)|_\infty + \Big|r^{\frac{m}{2}}\Big(\phi_r,\phi_t,u_{rr}, \big(\frac{u}{r}\big)_{r}\Big)(t)\Big|_2\leq C(T)\big(|r^{\frac{m}{2}}u_t(t)|_2+1\big)\leq C(T).
\end{equation}

Next, multiplying $\eqref{e2.2}_2$ by $r^\frac{m}{4}$ and then taking the $L^4(I)$-norm of the resulting equality, we obtain from \eqref{v-express-2}, \eqref{6026}, and Lemmas \ref{im-1}, \ref{GN-ineq}, and \ref{lemma-initial} that
\begin{align*}
\|\nabla^2 \boldsymbol{u}\|_{L^4}&\leq C_0\Big|r^\frac{m}{4}\big(u_r+\frac{m}{r}u\big)_r\Big|_4\leq C_0\big|r^\frac{m}{4} \big(u_t,uu_r,\phi_r,\psi u_r\big)\big|_4\\
&\leq C_0\big(|r^\frac{m}{4} u_t|_4+|(u,v)|_\infty|r^\frac{m}{4} u_r|_4+|r^\frac{m}{4}\phi(v-u)|_4\big)\\
&\leq C(T)|r^\frac{m}{4}(u_t,u_r)|_4+C_0|\phi|_\infty^\frac{1}{2}|(u,v)|_\infty^\frac{1}{2}\big(|r^\frac{m}{2}\phi v|_2^\frac{1}{2}+|\phi|_\infty^\frac{1}{2} |r^\frac{m}{2}u|_2^\frac{1}{2}\big)\\
&\leq C(T)\big(\|(\boldsymbol{u}_t,\nabla\boldsymbol{u})\|_{L^4}+1\big)\\
&\leq C(T)\Big(\|\boldsymbol{u}_t\|_{L^2}^\frac{4-n}{4}\|\nabla\boldsymbol{u}_t\|_{L^2}^\frac{n}{4}+\|\nabla \boldsymbol{u}\|_{L^2}^\frac{4-n}{4}\|\nabla^2 \boldsymbol{u}\|_{L^2}^\frac{n}{4}+1\Big)\\
&\leq C(T)\Big(|r^\frac{m}{2}u_t|_2^\frac{4-n}{4}\Big|r^\frac{m}{2}\big(u_{tr},\frac{u_t}{r}\big)\Big|_2^\frac{n}{4}+\Big|r^\frac{m}{2}\big(u_r,\frac{u}{r}\big)\Big|_2^\frac{4-n}{4}\Big|r^\frac{m}{2}\Big(u_{rr},\big(\frac{u}{r}\big)_r\Big)\Big|_2^\frac{n}{4}+1\Big)\\
&\leq C(T)\Big(\Big|r^\frac{m}{2}\big(u_{tr},\frac{u_t}{r}\big)\Big|_2^\frac{n}{4}+1\Big),
\end{align*}
which, together with Lemmas \ref{GN-ineq} and \ref{lemma-initial}, and the Young inequality, yields that
\begin{equation}\label{ur-infty}
\begin{aligned}
\Big|\big(u_r,\frac{u}{r}\big)\Big|_\infty&\leq C_0\|\nabla \boldsymbol{u}\|_{L^\infty}\leq C_0\|\nabla \boldsymbol{u}\|_{L^2}^\frac{4-n}{n+4}\|\nabla^2\boldsymbol{u}\|_{L^4}^\frac{2n}{n+4}\\
&\leq C(T)\Big|r^\frac{m}{2}\big(u_r,\frac{u}{r}\big)\Big|_{2}^\frac{4-n}{n+4}\Big(\Big|r^\frac{m}{2}\big(u_{tr},\frac{u_t}{r}\big)\Big|_2^\frac{n^2}{2(n+4)}+1\Big)\\
&\leq C(T)\Big(\Big|r^\frac{m}{2}\big(u_{tr},\frac{u_t}{r}\big)\Big|_2^\frac{n^2}{2(n+4)}+1\Big)\leq C(T)\Big(\Big|r^\frac{m}{2}\big(u_{tr},\frac{u_t}{r}\big)\Big|_2+1\Big).
\end{aligned}   
\end{equation}
Thus, taking the square of the above and integrating the resulting inequality, along with Lemma \ref{l4.8}, lead to the desired conclusion.
\end{proof}

We now show the estimate of $\psi$.

\begin{lem}\label{l4.7}
There exists a constant $C(T)>0$ such that
\begin{equation*}\label{e4.34}
|\psi(t)|_\infty+\Big|r^{\frac{m}{2}}\big(\psi_r,\frac{\psi}{r},\psi_t\big)(t)\Big|_{2}\leq C(T)\qquad\,\mbox{for any $t\in [0,T]$}.
\end{equation*}
\end{lem}
\begin{proof}
First, it follows from \eqref{v-express-2} 
and Lemmas \ref{l4.4} and \ref{lemma66} that, for all $t\in [0,T]$,
\begin{equation*}
|\psi(t)|_\infty\leq C_0|(v,u)(t)|_\infty\leq C(T).    
\end{equation*}

Next, the equation of $\boldsymbol{v}$ in the 
M-D coordinates (see \eqref{eq-vv} in \S \ref{section-local-regular}) takes the form: 
\begin{equation}\label{845}
\boldsymbol{v}_t+\boldsymbol{u}\cdot\nabla\boldsymbol{v}+\frac{\gamma-1}{2\alpha}\phi(\boldsymbol{v}-\boldsymbol{u})=\boldsymbol{0}.
\end{equation}
Applying the divergence to both sides of the above leads to
the following equation holds in the sense of distributions:
\begin{equation}\label{nabla-v}
\begin{aligned}
(\diver\boldsymbol{v})_t+\diver\big(\boldsymbol{u} (\diver\boldsymbol{v})\big)&=(\diver\boldsymbol{u})(\diver\boldsymbol{v})-\nabla\boldsymbol{u}^\top:\nabla\boldsymbol{v}\\
&\quad -\frac{\gamma-1}{2\alpha}\nabla\phi\cdot(\boldsymbol{v}-\boldsymbol{u})-\frac{\gamma-1}{2\alpha}\phi(\diver\boldsymbol{v}-\diver\boldsymbol{u}):=\tilde{\mathcal{G}}.
\end{aligned}
\end{equation}
Using $\boldsymbol{v}=\boldsymbol{u}+2\alpha\nabla\log\rho$
and Lemmas \ref{zth1} and \ref{ale1}, 
it can be checked that
\begin{equation}
\diver\boldsymbol{v}\in L^\infty([0,T];L^2(\mathbb{R}^n)),\qquad \boldsymbol{u}\in L^1([0,T];W^{1,\infty}(\mathbb{R}^n)),    
\end{equation}
and $\tilde{\mathcal{G}}\in L^1([0,T];L^2(\mathbb{R}^n))$. 
Thus, it follows from Lemma \ref{lemma-lions} that, for {\it a.e.} $t\in (0,T)$,
\begin{align*}
&\frac{\mathrm{d}}{\mathrm{d}t}\|\diver\boldsymbol{v}\|_{L^2}^{2}+\frac{\gamma-1}{\alpha}\big\|\phi^\frac{1}{2}\diver\boldsymbol{v}\big\|_{L^2}^{2}\notag\\
&=\int_{\mathbb{R}^n}  (\diver\boldsymbol{u}) (\diver\boldsymbol{v})^2\,\mathrm{d}\boldsymbol{x}-2\int_{\mathbb{R}^n} (\nabla\boldsymbol{u}^\top:\nabla\boldsymbol{v})(\diver\boldsymbol{v})\,\mathrm{d}\boldsymbol{x} \\
&\quad + \frac{\gamma-1}{\alpha}\int_{\mathbb{R}^n} \big(-\nabla\phi\cdot(\boldsymbol{v}-\boldsymbol{u})+\phi(\diver\boldsymbol{u})\big)(\diver\boldsymbol{v})\,\mathrm{d}\boldsymbol{x},\notag
\end{align*}
which, along with the spherical coordinate transformation, yields that, 
for {\it a.e.} $t\in (0,T)$,
\begin{equation}\label{derive-vr}
\begin{aligned}
&\frac{\mathrm{d}}{\mathrm{d}t}\Big|r^{\frac{m}{2}}\big(v_r+\frac{m}{r}v\big)\Big|_{2}^{2}+\frac{\gamma-1}{\alpha}\Big|(r^m\phi)^\frac{1}{2}\big(v_r+\frac{m}{r}v\big)\Big|_{2}^{2}\\
&=\int_0^\infty r^m \big(u_r+\frac{m}{r}u\big)\big(v_r+\frac{m}{r}v\big)^2\,\mathrm{d}r-2\int_0^\infty r^m \Big(u_rv_r+\frac{m}{r^2}uv\Big)\big(v_r+\frac{m}{r}v\big)\,\mathrm{d}r\\
&\quad + \frac{\gamma-1}{\alpha}\int_0^\infty r^m\Big(-\phi_r(v-u)+\phi\big(u_r+\frac{m}{r}u\big)\Big)\big(v_r+\frac{m}{r}v\big)\,\mathrm{d}r.
\end{aligned}    
\end{equation}

We now continue to estimate the right-hand side of \eqref{derive-vr}. 
It follows from Lemmas \ref{important2}, \ref{im-1}, \ref{l4.6}, and \ref{lemma66}, and the H\"older and Young inequalities that
\begin{equation}\label{cal-vr}
\begin{aligned}
&\frac{\mathrm{d}}{\mathrm{d}t}\Big|r^{\frac{m}{2}}\big(v_r+\frac{m}{r}v\big)\Big|_{2}^{2}+\frac{\gamma-1}{\alpha}\Big|(r^m\phi)^\frac{1}{2}\big(v_r+\frac{m}{r}v\big)\Big|_{2}^{2}\\
&\leq C_0\Big|\big(u_r,\frac{u}{r}\big)\Big|_\infty\Big|r^\frac{m}{2}\big(v_r+\frac{m}{r}v\big)\Big|_{2}^{2}+C_0\Big|\big(u_r,\frac{u}{r}\big)\Big|_\infty\Big|r^\frac{m}{2}\big(v_r,\frac{v}{r}\big)\Big|_{2}\Big|r^\frac{m}{2}\big(v_r+\frac{m}{r}v\big)\Big|_{2}\\
&\quad +C_0\Big(|r^\frac{m}{2}\phi_r|_2|(v,u)|_\infty+|\phi|_{\infty}\Big|r^\frac{m}{2}\big(u_r+\frac{m}{r}u\big)\Big|_{2}\Big)\Big|r^\frac{m}{2}\big(v_r+\frac{m}{r}v\big)\Big|_{2}\\
&\leq C(T)\Big(\Big|\big(u_r,\frac{u}{r}\big)\Big|_\infty^2+1\Big)\Big|r^\frac{m}{2}\big(v_r+\frac{m}{r}v\big)\Big|_{2}^{2}+C(T),
\end{aligned}    
\end{equation}
which, along with Lemmas \ref{im-1}, \ref{lemma66}, and \ref{lemma-initial}, and the Gr\"onwall inequality, gives
\begin{equation}\label{v_r}
\begin{aligned}
\Big|r^\frac{m}{2}\big(v_r,\frac{v}{r}\big)(t)\Big|_{2}&\leq C(T)\Big|r^\frac{m}{2}\big((v_0)_r,\frac{v_0}{r}\big)\Big|_{2}+C(T)\\
&\leq C(T) \Big|r^\frac{m}{2}\Big((u_0)_r,\frac{u_0}{r},(\log\rho_0)_{rr},\frac{(\log\rho_0)_r}{r}\Big)\Big|_{2}+C(T)\\
&\leq C(T) \|(\nabla \boldsymbol{u}_0,\nabla^2\log\rho_0)\|_{L^{2}}+C(T)\leq C(T).
\end{aligned}  
\end{equation}
This, together with \eqref{v-express-2} and Lemma \ref{l4.6}, yields that, for all $t\in [0,T]$,
\begin{equation}\label{psi_r}
\Big|r^\frac{m}{2}\big(\psi_r,\frac{\psi}{r}\big)(t)\Big|_{2}\leq C_0\Big|r^\frac{m}{2}\big(v_r,\frac{v}{r}\big)(t)\Big|_{2}+C_0\Big|r^\frac{m}{2}\big(u_r,\frac{u}{r}\big)(t)\Big|_{2}\leq C(T).
\end{equation}

Finally, multiplying $\eqref{e2.2}_3$ by $r^\frac{m}{2}$ and then taking the $L^2(I)$-norm of the resulting equality, we obtain from the above and Lemmas \ref{l4.6} and \ref{lemma66} that
\begin{equation}\label{psi_t}
\begin{split}
|r^{\frac{m}{2}}\psi_t|_2&\leq  |r^{\frac{m}{2}}(\psi u)_r|_2+\Big|r^{\frac{m}{2}}\big(u_{r}+\frac{m}{r}u\big)_r\Big|_2 \\
&\leq  |r^{\frac{m}{2}}u_r|_2|\psi|_\infty+|r^{\frac{m}{2}}\psi_r|_2|u|_\infty+\Big|r^{\frac{m}{2}}\Big(u_{rr},\big(\frac{u}{r}\big)_r\Big)\Big|_2 \le C(T).
\end{split}
\end{equation}
The proof of Lemma \ref{l4.7} is completed. 
\end{proof}

The following lemma provides the high-order estimates of $(\phi,u)$.
\begin{lem}\label{l4.9}  
There exists a constant $C(T)>0$ such that, for any $t\in [0,T]$,
\begin{equation*}
\begin{split}
&\Big|r^{\frac{m}{2}}\Big(\phi_{rr},\frac{\phi_r}{r},\phi_{tr}\Big)(t)\Big|^2_2+\int_{0}^{t}  \Big|r^{\frac{m}{2}}\Big(u_{rrr},\frac{u_{rr}}{r},\big(\frac{u}{r}\big)_{rr},\frac{1}{r}\big(\frac{u}{r}\big)_r\Big)\Big|_2^2 \,\mathrm{d}s\leq C(T).
\end{split}
\end{equation*}
\end{lem}
\begin{proof} We divide the proof into two steps.

\smallskip
\textbf{1. Estimates on $\phi$.} For the $L^2(I)$-estimates of $r^{\frac{m}{2}}(\phi_{rr},\frac{\phi_r}{r})$ and $r^\frac{m}{2}\phi_{tr}$, it follows from Lemmas \ref{important2} and \ref{lemma66}--\ref{l4.7} that
\begin{equation}\label{phixxl1}
\begin{aligned}
\Big|r^\frac{m}{2}\big(\phi_{rr},\frac{\phi_r}{r}\big)\Big|_2&=(\gamma-1)\Big|r^\frac{m}{2}\Big((\phi\psi)_r,\frac{\phi\psi}{r}\Big)\Big|_2\leq (\gamma-1) \Big|r^\frac{m}{2}\Big(\phi_r\psi,\phi\psi_r,\frac{\phi\psi}{r}\Big)\Big|_2\\
&\leq C_0|\psi|_\infty|r^\frac{m}{2}\phi_r|_2+C_0|\phi|_\infty\Big|r^\frac{m}{2}\big(\psi_r,\frac{\psi}{r}\big)\Big|_2\leq C(T),\\
|r^{\frac{m}{2}}\phi_{tr}|_2&=(\gamma-1)|r^\frac{m}{2}(\phi\psi)_t|_2\leq (\gamma-1)|r^\frac{m}{2}(\phi_t\psi,\phi\psi_t)|_2 \\
&\leq C_0|\psi|_\infty |r^\frac{m}{2}\phi_t|_2+C_0|\phi|_\infty |r^\frac{m}{2}\psi_t|_2\leq C(T).
\end{aligned}
\end{equation}

\smallskip
\textbf{2. Estimates on $u$.} First, multiplying $\eqref{e2.2}_2$ by $r^{\frac{m-2}{2}}$ and taking 
the $L^2(I)$-norm of the resulting equality, we obtain from Lemmas \ref{important2}, \ref{l4.6}, and \ref{l4.7} that
\begin{equation}\label{6028}
\begin{aligned}
\Big|r^\frac{m-2}{2}\big(u_r+\frac{m}{r}u\big)_r\Big|_2 
&\leq |r^\frac{m-2}{2} u_t|_2+C_0|u_r|_\infty|r^\frac{m-2}{2}(u,\psi)|_2+C_0|\phi|_\infty |r^\frac{m-2}{2}\psi|_2\\
&\leq C(T)\big(|r^\frac{m-2}{2} u_t|_2+|u_r|_\infty+1\big). 
\end{aligned}    
\end{equation}

Similarly, applying $r^\frac{m}{2}\partial_r$ on both sides of $\eqref{e2.2}_2$ and taking $L^2(I)$-norm of the resulting equality,
we obtain from Lemmas \ref{important2}, \ref{l4.6}, 
and \ref{lemma66}--\ref{l4.7} that
\begin{equation}\label{uxxxl1}
\begin{aligned}
\Big|r^{\frac{m}{2}}\big(u_r+ \frac{m}{r}u\big)_{rr}\Big|_2&\leq  \big|r^{\frac{m}{2}}\big(u_{tr},(u  u_r)_r,\phi_{rr},2\alpha(\psi u_r)_r\big)\big|_2\\
&\leq |r^{\frac{m}{2}}u_{tr}|_{2}+|r^{\frac{m}{2}} u_r^2|_2+|r^{\frac{m}{2}}u u_{rr}|_2+C_0|r^{\frac{m}{2}}\phi_r\psi|_2\\
&\quad +C_0\big(|r^{\frac{m}{2}}\phi\psi_r|_2+|r^{\frac{m}{2}}\psi_r u_r|_2+|r^{\frac{m}{2}}\psi u_{rr}|_2)\\
&\leq |r^{\frac{m}{2}}u_{tr}|_{2}+|u_r|_\infty|r^{\frac{m}{2}} u_r|_2+C_0|(u,\psi)|_\infty|r^{\frac{m}{2}}u_{rr}|_2\\
&\quad +C_0\big(|\psi|_\infty|r^{\frac{m}{2}}\phi_r|_2
+|(\phi,u_r)|_\infty|r^{\frac{m}{2}}\psi_r|_2\big)\\
&\leq C(T)\big(|r^{\frac{m}{2}}u_{tr}|_{2}+|u_r|_\infty+1\big).
\end{aligned}    
\end{equation}

Thus, combining with \eqref{6028}--\eqref{uxxxl1}, together with Lemmas \ref{im-1} and \ref{l4.8}--\ref{lemma66}, gives
\begin{align*}
&\int_0^t \Big|r^{\frac{m}{2}}\Big(u_{rrr},\frac{u_{rr}}{r},\big(\frac{u}{r}\big)_{rr},\frac{1}{r}\big(\frac{u}{r}\big)_r\Big)\Big|_2^2\,\mathrm{d}s\\
&\leq C_0\int_0^t\Big|r^\frac{m}{2}\Big(\big(u_r+ \frac{m}{r}u\big)_{rr},\frac{1}{r}\big(u_r+ \frac{m}{r}u\big)_{r}\Big)\Big|_2^2\,\mathrm{d}s\\
&\leq C(T)\int_0^t\Big(\Big|r^{\frac{m}{2}}\big(u_{tr},\frac{u_t}{r}\big)\Big|_2^2+|u_r|_\infty^2\Big)\,\mathrm{d}s+C(T)\leq C(T).
\end{align*}

This completes the proof of Lemma \ref{l4.9}.
\end{proof}

\subsection{ Time-weighted energy estimates of the velocity}
\begin{lem}\label{l4.10}  
There exists a constant $C(T)>0$ such that
\begin{equation*}
\begin{split}
&t\Big|r^{\frac{m}{2}}\big(u_{tr},\frac{u_t}{r}\big)(t)\Big|^2_2+\int_{0}^{t} s |r^{\frac{m}{2}}u_{tt}|^2_2 \,\mathrm{d}s\leq C(T)
\qquad\mbox{for any $t\in [0,T]$}.
\end{split}
\end{equation*}
\end{lem}

\begin{proof}
First, applying $r^m u_{tt}\partial_t$ to both sides of $(\ref{e2.2})_2$ and integrating the resulting equality over $I$ yield that, for any $t\in [\tau,T]$ and $\tau\in (0,T)$,
\begin{equation}\label{eq:b12pre}
\begin{aligned}
&\underline{-2\alpha \int_0^\infty r^m\big(u_r+ \frac{m}{r}u\big)_{tr}u_{tt}\,\mathrm{d}r}_{:=\mathcal{B}_{11}}+|r^\frac{m}{2}u_{tt}|_2^2\\
&=- \int_0^\infty r^m\big((u u_r)_{t}+\phi_{tr}-2\alpha (\psi u_r)_t\big)u_{tt}\,\mathrm{d}r.
\end{aligned}
\end{equation}

Then we claim that, for any $t\in [\tau,T]$,
\begin{equation}\label{eq:b12}
\mathcal{B}_{11}=\alpha \frac{\mathrm{d}}{\mathrm{d}t}\Big|r^{\frac{m}{2}}\big(u_{tr}+\frac{m}{r}u_t\big)\Big|_2^2.
\end{equation}
From the spherical coordinate transformation 
and Lemma \ref{lemma-initial}, it suffices to prove that
\begin{equation}\label{7-45}
\frac{\mathrm{d}}{\mathrm{d}t}\|\diver \boldsymbol{f}\|_{L^2}^2=-2\int_{\mathbb{R}^n} \nabla\diver \boldsymbol{f}\cdot\boldsymbol{f}_t\,\mathrm{d}\boldsymbol{x}
\end{equation}
for any vector function $\boldsymbol{f}\in L^2([\tau,T];H^2(\mathbb{R}^n))$ with 
$\boldsymbol{f}_t\in L^2([\tau,T];L^2(\mathbb{R}^n))$. 
To achieve this, consider the standard mollifier 
$\omega_\varepsilon$ defined on $\mathbb{R}$, and set
\begin{equation*}
g^\varepsilon(t,\boldsymbol{x})=\int_{-\infty}^{\infty} g(t-t',\boldsymbol{x})\omega_\varepsilon(t')\,\mathrm{d}t'.
\end{equation*}
Then, after extension and regularization, $\boldsymbol{f}^\varepsilon\in C^\infty([\tau,T];H^2(\mathbb{R}^n))$, we see from integration by parts that, for any $\varepsilon,\iota>0$, 
\begin{equation*}
\begin{aligned}
\frac{\mathrm{d}}{\mathrm{d}t}\|\diver \boldsymbol{f}^\varepsilon-\diver \boldsymbol{f}^\iota\|_{L^2}^2&=2\int_{\mathbb{R}^n} \big(\diver \boldsymbol{f}^\varepsilon-\diver \boldsymbol{f}^\iota\big)\big(\diver (\boldsymbol{f}_t)^\varepsilon-\diver (\boldsymbol{f}_t)^\iota\big)\,\mathrm{d}\boldsymbol{x}\\
&=-2\int_{\mathbb{R}^n} \big(\nabla\diver \boldsymbol{f}^\varepsilon-\nabla\diver \boldsymbol{f}^\iota\big)\cdot \big((\boldsymbol{f}_t)^\varepsilon-(\boldsymbol{f}_t)^\iota\big)\,\mathrm{d}\boldsymbol{x}.
\end{aligned}
\end{equation*}
Integrating above over $[\tau,T]$, along with the Young inequality, yields
\begin{align*}
&\sup_{t\in [\tau,T]}\|\diver \boldsymbol{f}^\varepsilon-\diver \boldsymbol{f}^\iota\|_{L^2}^2\\
&\leq \|\diver \boldsymbol{f}^\varepsilon(\tau)-\diver \boldsymbol{f}^\iota(\tau)\|_{L^2}^2+\int_\tau^T \big(\|\nabla\diver(\boldsymbol{f}^\varepsilon- \boldsymbol{f}^\iota)\|_{L^2}^2+\|(\boldsymbol{f}_t)^\varepsilon-(\boldsymbol{f}_t)^\iota\|_{L^2}^2\big)\,\mathrm{d}t.
\end{align*}
Letting $(\varepsilon,\iota)\to (0,0)$ in the above inequality 
implies that $\{\diver \boldsymbol{f}^\varepsilon\}_{\varepsilon>0}$ is a Cauchy sequence in $C([\tau,T];L^2(\mathbb{R}^n))$, and hence $\diver\boldsymbol{f}^\varepsilon$ converges to $\diver \boldsymbol{f}\in C([\tau,T];L^2(\mathbb{R}^n))$ in $C([\tau,T];L^2(\mathbb{R}^n))$ as $\varepsilon\to 0$. Similarly, we obtain from the above calculation that, 
for all $s,t\in [\tau, T]$,
\begin{equation*}
\|\diver \boldsymbol{f}^\varepsilon(t)\|_{L^2}^2=\|\diver \boldsymbol{f}^\varepsilon(s)\|_{L^2}^2-2\int_s^t\int_{\mathbb{R}^n} \nabla\diver\boldsymbol{f}^\varepsilon \cdot  (\boldsymbol{f}_t)^\varepsilon \,\mathrm{d}\boldsymbol{x}\,\mathrm{d}s'. 
\end{equation*}
Passing to the limit $\varepsilon \to 0$ leads to 
\begin{equation*}
\|\diver \boldsymbol{f}(t)\|_{L^2}^2=\|\diver \boldsymbol{f}(s)\|_{L^2}^2-2\int_s^t\int_{\mathbb{R}^n} \nabla\diver\boldsymbol{f} \cdot  \boldsymbol{f}_t\,\mathrm{d}\boldsymbol{x}\,\mathrm{d}s',
\end{equation*}
which implies that the mapping: $t\to \|\diver \boldsymbol{f}(t)\|_{L^2}^2$ is absolutely continuous. Thus, differentiating above with respect to $t$
leads to \eqref{7-45} so that claim \eqref{eq:b12} holds.

Now, substituting \eqref{eq:b12} into \eqref{eq:b12pre}, we obtain 
from \eqref{ur-infty}, Lemmas \ref{im-1} and \ref{l4.8}--\ref{l4.9}, 
and the H\"older and Young inequalities that 
\begin{equation}\label{etrq}
\begin{aligned}
&\alpha\frac{\mathrm{d}}{\mathrm{d}t}\Big|r^{\frac{m}{2}}\big(u_{tr}+\frac{m}{r}u_t\big)\Big|_2^2+ |r^{\frac{m}{2}}u_{tt}|_2^2=-\int r^m\big((u u_r)_t+\phi_{tr}-2\alpha(\psi u_r)_t\big)u_{tt}\,\mathrm{d}r\\
&\leq C_0\big(|r^{\frac{m}{2}}(u_t,\psi_t)|_2|u_r|_\infty + |r^{\frac{m}{2}}u_{tr}|_2|(u,\psi)|_\infty+|r^{\frac{m}{2}}\phi_{tr}|_2\big)|r^{\frac{m}{2}}u_{tt}|_2\\
&\leq C(T)\Big(\Big|r^{\frac{m}{2}}\big(u_{tr},\frac{u_t}{r}\big)\Big|_2+1\Big)|r^{\frac{m}{2}}u_{tt}|_2\\
&\leq  C(T)\Big(\Big|r^{\frac{m}{2}}\big(u_{tr}+\frac{m}{r}u_t\big)\Big|_2^2+1\Big)+\frac{1}{8}|r^{\frac{m}{2}}u_{tt}|_2^2.
\end{aligned}    
\end{equation}
Multiplying above by $t$ and integrating the resulting inequality over $[\tau, t]$, along with Lemmas \ref{im-1} and \ref{l4.8}, imply that
\begin{equation}\label{etrq2}
t\Big|r^{\frac{m}{2}}\big(u_{tr},\frac{u_t}{r}\big)(t)\Big|_2^2+ \int_\tau^t s|r^{\frac{m}{2}}u_{tt}|_2^2\,\mathrm{d}s\leq C(T)\Big(\tau\Big|r^{\frac{m}{2}}\big(u_{tr},\frac{u_t}{r}\big)(\tau)\Big|_2^2+1\Big).
\end{equation}

Next, thanks to Lemma \ref{l4.8}, $r^{\frac{m}{2}}(u_{tr},\frac{u_t}{r})\in L^2([0,T];L^2(I))$. It follows from  
Lemma \ref{bjr} that there exists a sequence $\{\tau_k\}_{k=1}^\infty$ such that 
\begin{equation}
\tau_k\to 0, \quad \tau_k\Big|r^\frac{m}{2}\big(u_{tr},\frac{u_t}{r}\big)(\tau_k)\Big|_2^2\to 0 \qquad\,\, \text{as $k\to \infty$}.
\end{equation}
Choosing $\tau=\tau_k\to 0$  in  \eqref{etrq2} yields that, for all $t\in [0,T]$, 
\begin{equation}\label{t-use}
t\Big|r^\frac{m}{2}\big(u_{tr},\frac{u_t}{r}\big)(t)\Big|_2^2+ \int_0^t s|r^{\frac{m}{2}}u_{tt}|_2^2\,\mathrm{d}s\leq C(T).
\end{equation}
The proof of Lemma \ref{l4.10} is completed.
\end{proof}

\begin{lem}\label{l4.10-ell}  
There exists a constant $C(T)>0$ such that, for any $t\in [0,T]$,
\begin{equation*}
t\Big|r^\frac{m}{2}\Big(u_{rrr},\frac{u_{rr}}{r},\big(\frac{u}{r}\big)_{rr},\frac{1}{r}\big(\frac{u}{r}\big)_r\Big)(t)\Big|_{2}^2+\int_{0}^{t} s \Big|r^\frac{m}{2}\Big(u_{trr},\big(\frac{u_t}{r}\big)_{r}\Big)\Big|_{2}^2 \,\mathrm{d}s\leq C(T).
\end{equation*}
\end{lem}
\begin{proof}
First, it follows from \eqref{ur-infty} by multiplying $\sqrt{t}$ and Lemma \ref{l4.10} that, for all $t\in [0,T]$,
\begin{equation}\label{ur-infty-t}
\sqrt{t}\Big|\big(u_r,\frac{u}{r}\big)\Big|_\infty \leq C(T)\Big(\sqrt{t}\big|r^\frac{m}{2}\big(u_{tr},\frac{u_t}{r}\big)\Big|_2+1\Big)\leq C(T),
\end{equation}
which, along with \eqref{6028}--\eqref{uxxxl1}, \eqref{ur-infty-t}, and Lemmas \ref{im-1} and \ref{l4.10}, leads to
\begin{equation}\label{urrr-t}
\begin{aligned}
&\sqrt{t}\Big|r^\frac{m}{2}\Big(u_{rrr},\frac{u_{rr}}{r},\big(\frac{u}{r}\big)_{rr},\frac{1}{r}\big(\frac{u}{r}\big)_r\Big)\Big|_{2}\\
&\leq  C_0\sqrt{t}\Big|r^{\frac{m}{2}}\Big(\big(u_r+ \frac{m}{r}u\big)_{rr},\frac{1}{r}\big(u_r+ \frac{m}{r}u\big)_{r}\Big)\Big|_2\\
&\leq  C(T)\Big(\sqrt{t}\Big|r^{\frac{m}{2}}\big(u_{tr},\frac{u_t}{r}\big)\Big|_{2}+\sqrt{t}|u_r|_\infty+1\Big)\leq C(T).
\end{aligned}    
\end{equation}

Next, applying $\sqrt{t}r^\frac{m}{2}\partial_t$ to $(\ref{e2.2})_2$ and taking $L^2(I)$-norm of the resulting equality, along with \eqref{ur-infty-t}, and Lemmas \ref{im-1} and \ref{l4.8}--\ref{l4.10}, we have 
\begin{equation}\label{utt}
\begin{split}
& \sqrt{t}\Big|r^\frac{m}{2}\Big(u_{trr},\big(\frac{u_t}{r}\big)_{r}\Big)\Big|_{2}\leq C_0\sqrt{t}\Big|r^\frac{m}{2}\big(u_r+ \frac{m}{r}u\big)_{tr}\Big|_2\\
&\leq C_0\sqrt{t}\big|r^{\frac{m}{2}}\big(u_{tt},(u u_r)_t,\phi_{tr},(\psi u_r)_t\big)\big|_2\\
&\leq C_0\sqrt{t}\big(|r^{\frac{m}{2}}u_{tt}|_2 +|r^{\frac{m}{2}}(u_t,\psi_t)|_2|u_r|_\infty+|r^{\frac{m}{2}}u_{tr}|_2|(u,\psi)|_\infty+|r^{\frac{m}{2}}\phi_{tr}|_2\big) \\
&\leq C(T)\big(\sqrt{t}|r^{\frac{m}{2}}u_{tt}|_2 +1\big).
 \end{split}
\end{equation}
This, along with Lemma \ref{l4.10}, implies that, for all $t\in [0,T]$,
\begin{equation}
\int_{0}^{t} s \Big|r^\frac{m}{2}\Big(u_{trr},\big(\frac{u_t}{r}\big)_{r}\Big)\Big|_{2}^2\,\mathrm{d}s\leq C(T)\int_{0}^{t} s  |r^{\frac{m}{2}}u_{tt}|_2^2\,\mathrm{d}s+C(T)\leq C(T).
\end{equation}
This completes the proof of Lemma \ref{l4.10-ell}.
\end{proof}

\section{Global  Estimates for the \texorpdfstring{$3$}{}-Order Regular Solutions with Far-Field Vacuum}\label{section-global3}
This section is devoted to establishing the global estimates for the $3$-order regular solutions when  $\bar\rho=0$. 
Let $T>0$ be any fixed time,  
and let $(\rho, u)(t,r)$ be the $3$-order regular solution of 
problem \eqref{e1.5} in $[0,T]\times I$ obtained in Theorems \ref{rth133}.
Moreover, throughout this section, we always assume that \eqref{gammafanwei} holds.
Since the $H^2(\mathbb{R}^n)$-estimates are the same as those presented in Lemmas \ref{l4.5}--\ref{l4.9}, it suffices to focus on the $D^3(\mathbb{R}^n)$-estimates and the time-weighted $D^4(\mathbb{R}^n)$-estimates. 

\subsection{The third-order estimates of the velocity}
First, following the proofs of Lemmas \ref{l4.10}--\ref{l4.10-ell}, we can derive the following estimates: 
\begin{lem}\label{H3-1}
There exists a constant $C(T)>0$ such that, for any $t\in [0,T]$,
\begin{align*}
&\Big|r^{\frac{m}{2}}\big(u_{tr},\frac{u_t}{r}\big)(t)\Big|_2+\Big|r^\frac{m}{2}\Big(u_{rrr},\frac{u_{rr}}{r},\big(\frac{u}{r}\big)_{rr},\frac{1}{r}\big(\frac{u}{r}\big)_r\Big)(t)\Big|_{2}\\
&+\Big|\big(u_r,\frac{u}{r}\big)(t)\Big|_\infty+\int_{0}^{t} |r^{\frac{m}{2}}u_{tt}|^2_2\,\mathrm{d}s+\int_0^t\Big|r^\frac{m}{2}\Big(u_{trr},\big(\frac{u_t}{r}\big)_{r}\Big)\Big|_{2}^2\,\mathrm{d}s\leq C(T).
\end{align*}
\end{lem}
\begin{proof}We divide the proof into two steps.

\smallskip
\textbf{1.} First, it follows from \eqref{etrq} that, for all $t\in [\tau,T]$ and $\tau\in (0,T)$,
\begin{equation*}
\frac{\mathrm{d}}{\mathrm{d}t}\Big|r^{\frac{m}{2}}\big(u_{tr}+\frac{m}{r}u_t\big)\Big|_2^2+ |r^{\frac{m}{2}}u_{tt}|_2^2 \leq C(T)\Big(\Big|r^{\frac{m}{2}}\big(u_{tr}+\frac{m}{r}u_t\big)\Big|_2^2+1\Big). 
\end{equation*}

Integrating the above over $[\tau, t]$, together with Lemma \ref{im-1}, yields
\begin{equation}\label{lg1'} 
\begin{aligned}
&\Big|r^{\frac{m}{2}}\big(u_{tr},\frac{u_t}{r}\big)(t)\Big|_2^2+\int_\tau^t|r^{\frac{m}{2}}u_{tt}|_2^2\,\mathrm{d}s\\
&\leq C_0\Big|r^{\frac{m}{2}}\big(u_{tr},\frac{u_t}{r}\big)(\tau)\Big|_2^2+C(T)\Big(1+\int_\tau^t\Big|r^{\frac{m}{2}}\big(u_{tr},\frac{u_t}{r}\big)(\tau)\Big|_2^2\,\mathrm{d}s\Big).
\end{aligned}
\end{equation}
For the $L^2(I)$-boundedness of $r^{\frac{m}{2}}(u_{tr},\frac{u_t}{r})(\tau)$ on the right-hand side 
of the above, we apply $r^\frac{m}{2}\partial_r$ to $\eqref{e2.2}_2$ and multiply $\eqref{e2.2}_2$ 
by $r^\frac{m-2}{2}$, respectively, and then take the $L^2(I)$-norm of these two resulting equality to obtain
\begin{equation}\label{cont-limsup}
\begin{aligned}
\Big|r^\frac{m}{2}\big(u_{tr},\frac{u_t}{r}\big)(\tau)\Big|_2&\leq C_0|(u,\psi)(\tau)|_\infty|r^\frac{m}{2}u_{rr}(\tau)|_2\\
&\quad +C_0\Big|\big(u_r,\frac{u}{r}\big)(\tau)\Big|_\infty\Big|r^\frac{m}{2}\Big(u_r,\psi_r,\frac{\psi}{r}\Big)(\tau)\Big|_2\\
&\quad +\Big|r^\frac{m}{2}\big(\phi_{rr},\frac{\phi_r}{r}\big)(\tau)\Big|_2 +\Big|r^\frac{m}{2}\Big(\big(u_r+\frac{m}{r}u\big)_{rr},\frac{1}{r}\big(u_r+\frac{m}{r}u\big)_{r}\Big)(\tau)\Big|_2,
\end{aligned}    
\end{equation}
which, along with the time continuity of $(\phi,u,\psi)$, Lemmas \ref{im-1}, \ref{ale1}, 
and \ref{lemma-initial}, yields
\begin{equation}\label{limsup}
\begin{aligned}
&\limsup_{\tau\to 0} \Big|r^\frac{m}{2}\big(u_{tr},\frac{u_t}{r}\big)(\tau)\Big|_2\\
&\leq C_0|(u_0,\psi_0)|_\infty|r^\frac{m}{2}(u_0)_{rr}|_2  +C_0\Big|\big((u_0)_r,\frac{u_0}{r}\big)\Big|_\infty\Big|r^\frac{m}{2}\Big((u_0)_r,(\psi_0)_r,\frac{\psi_0}{r}\Big)\Big|_2\\
& \quad +\Big|r^\frac{m}{2}\big((\phi_0)_{rr},\frac{(\phi_0)_r}{r}\big)\Big|_2 +\Big|r^\frac{m}{2}\Big(\big((u_0)_r+\frac{m}{r}u_0\big)_{rr},\frac{1}{r}\big((u_0)_r+\frac{m}{r}u_0\big)_{r}\Big)\Big|_2\\
&\leq  C_0\|(\boldsymbol{u}_0,\boldsymbol{\psi}_0,\nabla \boldsymbol{u}_0)\|_{L^\infty}  \|(\nabla^2\boldsymbol{u}_0,\nabla\boldsymbol{u}_0,\nabla\boldsymbol{\psi}_0)\|_{L^2} +C_0 \|(\nabla^2\phi_0,\nabla^3\boldsymbol{u}_0)\|_{L^2} \leq C_0.
\end{aligned}    
\end{equation}
Consequently, based on the above discussions, letting $\tau\to 0$ in \eqref{lg1'}, we obtain 
from the Gr\"onwall inequality that, for all $t\in [0,T]$,
\begin{equation}\label{645}
\Big|r^\frac{m}{2}\big(u_{tr},\frac{u_t}{r}\big)(t)\Big|_2+\int_0^t |r^{\frac{m}{2}}u_{tt}|_2^2 \, \mathrm{d}s \leq C(T).
\end{equation}

\smallskip
\textbf{2.} Next, based on the estimates in \eqref{ur-infty} and \eqref{6028}--\eqref{uxxxl1}, following the calculations of \eqref{ur-infty-t}--\eqref{utt} in  Lemma \ref{l4.10-ell} without the time-weight $\sqrt{t}$, 
we have 
\begin{equation*}
\Big|r^\frac{m}{2}\Big(u_{rrr},\frac{u_{rr}}{r},\big(\frac{u}{r}\big)_{rr},\frac{1}{r}\big(\frac{u}{r}\big)_r\Big)(t)\Big|_{2}\!+\!\Big|\big(u_r,\frac{u}{r}\big)(t)\Big|_\infty\!+\!\int_0^t\Big|r^\frac{m}{2}\Big(u_{trr},\big(\frac{u_t}{r}\big)_{r}\Big)\Big|_{2}^2\mathrm{d}s\!\leq\! C(T).
\end{equation*}
This completes the proof of Lemma \ref{H3-1}.
\end{proof}

Next, we derive the higher-order estimates for $\psi$:
\begin{lem}\label{lemma-psi-high}
There exists a constant $C(T)>0$ such that, for any $t\in [0,T]$,
\begin{equation*}
\Big|r^\frac{m}{2}\Big(\psi_{rr},\big(\frac{\psi}{r}\big)_r,\psi_{tr},\frac{\psi_t}{r}\Big)(t)\Big|_2\leq C(T).
\end{equation*}
\end{lem}
\begin{proof} We divide the proof into two steps.

\smallskip
\textbf{1.}
Applying $\partial_l$ ($l=1,\cdots\!,n)$ to both sides of \eqref{nabla-v} leads to the following equation in the sense of distributions:
\begin{equation}\label{partdivev}
\!\!\!\!\begin{aligned}
(\partial_l\diver\boldsymbol{v})_t+\mathrm{div}\big(\boldsymbol{u} (\partial_l\diver\boldsymbol{v})\big)\!&=(\diver\boldsymbol{u})(\partial_l\diver\boldsymbol{v})-\partial_l\boldsymbol{u}\cdot \nabla(\diver\boldsymbol{v})-\partial_l(\nabla\boldsymbol{u}^\top:\nabla\boldsymbol{v})\\
&\quad -\frac{\gamma-1}{2\alpha}\nabla(\partial_l\phi)\cdot(\boldsymbol{v}-\boldsymbol{u})-\frac{\gamma-1}{2\alpha}\nabla\phi\cdot(\partial_l\boldsymbol{v}-\partial_l\boldsymbol{u})\\
&\quad  -\frac{\gamma-1}{2\alpha}\big(\partial_l\phi(\diver\boldsymbol{v}-\diver\boldsymbol{u})\!+\!\phi(\partial_l\diver\boldsymbol{v}-\partial_l\diver\boldsymbol{u})\big)\!:=\tilde{\mathcal{G}}'\!.
\end{aligned}
\end{equation}
It can be checked due to $\boldsymbol{v}=\boldsymbol{u}+2\alpha\nabla\log\rho$, and Lemmas \ref{zth2} and \ref{ale1} that 
\begin{equation}
\nabla\diver\boldsymbol{v}\in L^\infty([0,T];L^2(\mathbb{R}^n)),\quad \boldsymbol{u}\in L^1([0,T];W^{1,\infty}(\mathbb{R}^n)),
\end{equation}
and $\tilde{\mathcal{G}}'\in L^1([0,T];L^2(\mathbb{R}^n))$, since
\begin{equation}
\begin{aligned}
\|\tilde{\mathcal{G}}'\|_{L^1([0,T];L^2(\mathbb{R}^n))}&\leq C_0\|\nabla\boldsymbol{u}\|_{L^1([0,T];L^\infty(\mathbb{R}^n))}\|\nabla^2\boldsymbol{v}\|_{L^\infty([0,T];L^2(\mathbb{R}^n))}\\
&\quad + C_0\|\nabla^2\boldsymbol{u}\|_{L^1([0,T];L^4(\mathbb{R}^n))}\|\nabla\boldsymbol{v}\|_{L^\infty([0,T];L^4(\mathbb{R}^n))}\\
&\quad +C_0\|\nabla^2\phi\|_{L^\infty([0,T];L^2(\mathbb{R}^n))}\|(\boldsymbol{u},\boldsymbol{v})\|_{L^1([0,T];L^\infty(\mathbb{R}^n))}\\
&\quad +C_0\|\nabla\phi\|_{L^\infty([0,T]\times \mathbb{R}^n)}\|(\nabla\boldsymbol{u},\nabla\boldsymbol{v})\|_{L^1([0,T];L^2(\mathbb{R}^n))}\\
&\quad +C_0\|\phi\|_{L^\infty([0,T]\times\mathbb{R}^n)}\|(\nabla^2\boldsymbol{u},\nabla^2\boldsymbol{v})\|_{L^1([0,T];L^2(\mathbb{R}^n))}<\infty.
\end{aligned}    
\end{equation}
Then it follows from Lemma \ref{lemma-lions} that, for {\it a.e.} $t\in (0,T)$, 
\begin{equation}
\begin{aligned}
&\frac{\mathrm{d}}{\mathrm{d}t}\|\partial_l\diver\boldsymbol{v}\|_{L^2}^{2}+\frac{\gamma-1}{\alpha}\big\|\phi^\frac{1}{2}\partial_l\diver\boldsymbol{v}\big\|_{L^2}^{2}\\
&=\int_{\mathbb{R}^n} \Big( (\diver\boldsymbol{u}) (\partial_l\diver\boldsymbol{v})^2 -2 \big(\partial_l\boldsymbol{u}\cdot\nabla\diver\boldsymbol{v}\big)(\partial_l\diver\boldsymbol{v})\Big) \,\mathrm{d}\boldsymbol{x}\\
&\quad-2\int_{\mathbb{R}^n} \partial_l(\nabla\boldsymbol{u}^\top:\nabla\boldsymbol{v})(\partial_l\diver\boldsymbol{v})\,\mathrm{d}\boldsymbol{x}- \frac{\gamma-1}{\alpha}\int_{\mathbb{R}^n} \big(\nabla(\partial_l\phi)\cdot(\boldsymbol{v}-\boldsymbol{u}) \big)(\partial_l\diver\boldsymbol{v})\,\mathrm{d}\boldsymbol{x}\\
&\quad - \frac{\gamma-1}{\alpha}\int_{\mathbb{R}^n} \Big(\nabla\phi\cdot(\partial_l\boldsymbol{v}-\partial_l\boldsymbol{u})+\partial_l\phi(\diver\boldsymbol{v}-\diver\boldsymbol{u})\Big)(\partial_l\diver\boldsymbol{v})\,\mathrm{d}\boldsymbol{x}\\
&\quad + \frac{\gamma-1}{\alpha}\int_{\mathbb{R}^n}  \phi(\partial_l\diver\boldsymbol{u}) (\partial_l\diver\boldsymbol{v})\,\mathrm{d}\boldsymbol{x}.
\end{aligned}    
\end{equation}
Summing  above over $l=1,\cdots\!,n$, together with the spherical coordinate transformation, yields that, for {\it a.e.} $t\in (0,T)$,
\begin{equation}\label{nabla2-v}
\begin{aligned}
&\frac{\mathrm{d}}{\mathrm{d}t}\Big|r^\frac{m}{2}\big(v_r+\frac{m}{r}v\big)_{r}\Big|_2^2+\frac{\gamma-1}{\alpha}\Big|(r^m\phi)^\frac{1}{2}\big(v_r+\frac{m}{r}v\big)_{r}\Big|_2^2\\
&=\int_0^\infty  r^m\big(\frac{m}{r}u- u_r\big)\big(v_r+\frac{m}{r}v\big)_{r}^2\,\mathrm{d}r-2\int_0^\infty r^m\Big(u_rv_r+\frac{m}{r^2}uv\Big)_r\big(v_r+\frac{m}{r}v\big)_{r}\,\mathrm{d}r\\
&\quad  -\frac{\gamma-1}{ \alpha}\int_0^\infty r^m\phi_{rr}(v-u)\big(v_r+\frac{m}{r}v\big)_{r}\,\mathrm{d}r\\
&\quad -\frac{\gamma-1}{ \alpha}\int_0^\infty  r^m\phi_{r}\Big(\big(2v_r+\frac{m}{r}v\big)-\big(2u_r+\frac{m}{r}u\big)\Big)\big(v_r+\frac{m}{r}v\big)_{r}\,\mathrm{d}r\\
&\quad +\frac{\gamma-1}{ \alpha}\int_0^\infty  r^m\phi \big(u_r+\frac{m}{r}u\big)_{r}\big(v_r+\frac{m}{r}v\big)_{r}\,\mathrm{d}r:=\sum_{i=16}^{20}\mathcal{G}_i.
\end{aligned}    
\end{equation}
For $\mathcal{G}_{16}$--$\mathcal{G}_{20}$, it follows from \eqref{v_r}, Lemmas \ref{important2}, \ref{l4.4}, \ref{im-1}, \ref{lemma66}, \ref{H3-1}, \ref{ale1}, and \ref{hardy}, and the H\"older and Young inequalities that
\begin{equation}\label{G16-20}
\begin{aligned}
\mathcal{G}_{16}&\leq C_0\Big|\big(u_r,\frac{u}{r}\big)\Big|_\infty \Big|\big(v_r+\frac{m}{r}v\big)_r\Big|_2^2\leq C(T)\Big|\big(v_r+\frac{m}{r}v\big)_r\Big|_2^2, \\
\mathcal{G}_{17}&=-\int_0^\infty r^m\Big(u_{rr}v_r+u_rv_{rr}+m\big(\frac{u}{r}\big)_r\frac{v}{r}+m \frac{u}{r}\big(\frac{v}{r}\big)_r\Big)\big(v_r+\frac{m}{r}v\big)_{r}\,\mathrm{d}r\\
&\leq C_0\big(|r^\frac{m-2}{2}u_{rr}|_2|\chi_1^\flat rv_r|_\infty+|r^\frac{m}{2}u_{rr}|_2|\chi_1^\sharp v_r|_\infty\big)\Big|r^\frac{m}{2}\big(v_r+\frac{m}{r}v\big)_r\Big|_2\\
&\quad +C_0|u_r|_\infty |r^\frac{m}{2}v_{rr}|_2 \Big|r^\frac{m}{2}\big(v_r+\frac{m}{r}v\big)_r\Big|_2\\
&\quad +C_0\Big(|v|_\infty\Big|r^{\frac{m-2}{2}}\big(\frac{u}{r}\big)_r\Big|_2+\Big|\frac{u}{r}\Big|_\infty \Big|r^\frac{m}{2}\big(\frac{v}{r}\big)_r\Big|_2\Big)\Big|r^\frac{m}{2}\big(v_r+\frac{m}{r}v\big)_r\Big|_2\\ 
&\leq C(T)\Big(|\chi_1^\flat r^\frac{3}{2}(v_r,v_{rr})|_2 + |\chi_1^\sharp (v_r,v_{rr})|_2\Big)\Big|r^\frac{m}{2}\big(v_r+\frac{m}{r}v\big)_r\Big|_2\\
&\quad +C(T)\Big|r^\frac{m}{2}\big(v_r+\frac{m}{r}v\big)_r\Big|_2^2\\ 
&\leq C(T)\Big(\big(|\chi_1^\flat r^\frac{3-m}{2}|^2_\infty + |\chi_1^\sharp r^{-\frac{m}{2}}|^2_\infty\big)|r^\frac{m}{2}(v_r,v_{rr})|^2_2 +\Big|r^\frac{m}{2}\big(v_r+\frac{m}{r}v\big)_r\Big|^2_2 +1\Big)\\
&\leq C(T)\Big(\Big|r^\frac{m}{2}\big(v_r+\frac{m}{r}v\big)_r\Big|_2^2+1\Big),\\
\mathcal{G}_{18}&\leq C_0|(u,v)|_\infty |r^\frac{m}{2}\phi_{rr}|_2\Big|r^\frac{m}{2}\big(v_r+\frac{m}{r}v\big)_r\Big|_2\leq C(T)\Big(\Big|r^\frac{m}{2}\big(v_r+\frac{m}{r}v\big)_r\Big|_2^2+1\Big), \\
\mathcal{G}_{19}&\leq C_0\big(|r^{\frac{m-2}{2}}\phi_r|_2 \big|(\chi_1^\flat rv_r,v)\big|_\infty+|r^{\frac{m}{2}}\phi_r|_2 |\chi_1^\sharp v_r|_\infty\big)\Big|r^\frac{m}{2}\big(v_r+\frac{m}{r}v\big)_r\Big|_2\\
&\quad+C_0|r^{\frac{m}{2}}\phi_r|_2\Big|\big(u_r,\frac{u}{r}\big)\Big|_\infty \Big|r^\frac{m}{2}\big(v_r+\frac{m}{r}v\big)_r\Big|_2\\
&\leq C(T) \big(|\chi_1^\flat r^\frac{3}{2}(v_r,v_{rr})|_2+|\chi_1^\sharp (v_r,v_{rr})|_2+1\big)\Big|r^\frac{m}{2}\big(v_r+\frac{m}{r}v\big)_r\Big|_2\\
&\leq C(T) \Big(\big(|\chi_1^\flat r^\frac{3-m}{2}|_\infty+|\chi_1^\sharp r^{-\frac{m}{2}}|_\infty\big)|r^\frac{m}{2}(v_r,v_{rr})|_2+1\Big)\Big|r^\frac{m}{2}\big(v_r+\frac{m}{r}v\big)_r\Big|_2\\
&\leq C(T)\Big(\Big|r^\frac{m}{2}\big(v_r+\frac{m}{r}v\big)_r\Big|_2^2+1\Big),\\
\mathcal{G}_{20}&\leq C|\phi|_\infty\Big|r^\frac{m}{2}\big(u_r+\frac{m}{r}u\big)_r\Big|_2 \Big|r^\frac{m}{2}\big(v_r+\frac{m}{r}v\big)_r\Big|_2\leq C(T)\Big(\Big|r^\frac{m}{2}\big(v_r+\frac{m}{r}v\big)_r\Big|_2^2+1\Big).
\end{aligned}    
\end{equation}

Thus, combining with \eqref{nabla2-v}--\eqref{G16-20} leads to
\begin{equation}
\frac{\mathrm{d}}{\mathrm{d}t}\Big|r^\frac{m}{2}\big(v_r+\frac{m}{r}v\big)_r\Big|_2^2\leq C(T)\Big(\Big|r^\frac{m}{2}\big(v_r+\frac{m}{r}v\big)_r\Big|_2^2+1\Big),
\end{equation}
which, along with Lemmas \ref{im-1} and \ref{lemma-initial}, and the Gr\"onwall inequality, yields that
\begin{equation}
\begin{aligned}
\Big|r^\frac{m}{2}\Big(v_{rr},\big(\frac{v}{r}\big)_r\Big)(t)\Big|_2^2&\leq C(T)\big(\|\nabla^2\boldsymbol{v}_0\|_{L^2}^2+1\big)\\
&\leq C(T)\big(\|(\nabla^2\boldsymbol{u}_0,\nabla^2\boldsymbol{\psi}_0)\|_{L^2}^2+1\big)\leq C(T).
\end{aligned}
\end{equation}
 
Consequently, it follows from the above estimate, \eqref{v-express-2}, and Lemma \ref{lemma66} that
\begin{equation}\label{high-psi}
\Big|r^\frac{m}{2}\Big(\psi_{rr},\big(\frac{\psi}{r}\big)_r\Big)(t)\Big|_2\leq C_0\Big|r^\frac{m}{2}\Big(v_{rr},\big(\frac{v}{r}\big)_r\Big)(t)\Big|_2+C_0\Big|r^\frac{m}{2}\Big(u_{rr},\big(\frac{u}{r}\big)_r\Big)(t)\Big|_2\leq C(T).
\end{equation}

\smallskip
\textbf{2.} Finally,  multiplying $\eqref{e2.2}_3$ by $r^\frac{m-2}{2}$ and applying $r^\frac{m}{2}\partial_r$ to $\eqref{e2.2}_3$, respectively, then taking the $L^2(I)$-norm of these two resulting equalities, we obtain from \eqref{high-psi} and Lemmas \ref{lemma66}--\ref{l4.7} and \ref{H3-1} that

\begin{equation}\label{psi-tr}
\begin{aligned}
\Big|r^\frac{m}{2}\big(\psi_{tr},\frac{\psi_t}{r}\big)\Big|_2&\leq \Big|r^\frac{m}{2}\Big((\psi u)_{rr},\frac{1}{r}(\psi u)_{r}\Big)\Big|_2 +\Big|r^\frac{m}{2}\Big(\big(u_r+ \frac{m}{r}u\big)_{rr},\frac{1}{r}\big(u_r+ \frac{m}{r}u\big)_r\Big)\Big|_2\\
&\leq C_0\Big|r^\frac{m}{2}\Big(\psi_{rr} u,\psi_r u_{r},\psi u_{rr},\psi_r\frac{u}{r},\frac{\psi}{r} u_r\Big)\Big|_2\\
&\quad +C_0\Big|r^\frac{m}{2}\Big(u_{rrr},\frac{u_{rr}}{r},\big(\frac{u}{r}\big)_{rr},\frac{1}{r}\big(\frac{u}{r}\big)_r\Big)\Big|_{2}\\
&\leq C_0|(u,\psi)|_\infty |r^\frac{m}{2}(u_{rr},\psi_{rr})|_2+C_0\Big|\big(u_r,\frac{u}{r}\big)\Big|_\infty\Big|r^\frac{m}{2}\big(\psi_r,\frac{\psi}{r}\big)\Big|_2\\
&\quad +C_0\Big|r^\frac{m}{2}\Big(u_{rrr},\frac{u_{rr}}{r},\big(\frac{u}{r}\big)_{rr},\frac{1}{r}\big(\frac{u}{r}\big)_r\Big)\Big|_{2}\leq C(T).
\end{aligned}    
\end{equation}
The proof of Lemma \ref{lemma-psi-high} is completed.
\end{proof}

We now derive the higher-order estimates for $(\phi,u)$.
\begin{lem}\label{Lemma6.12}
There exists a constant $C(T)>0$ such that
\begin{align*}
&\Big|r^\frac{m}{2}\Big(\phi_{rrr},\big(\frac{\phi_r}{r}\big)_r,\phi_{trr},\frac{\phi_{tr}}{r}\Big)(t)\Big|_2\\
&+\int_0^t\Big|r^\frac{m}{2}\Big(u_{rrrr},\big(\frac{u_{rr}}{r}\big)_{r},\big(\frac{u}{r}\big)_{rrr},\Big(\frac{1}{r}\big(\frac{u}{r}\big)_r\Big)_r\Big)\Big|_{2}^2 \,\mathrm{d}s\leq C(T)
\qquad\mbox{for all $t\in [0,T]$}.
\end{align*}
\end{lem}

\begin{proof} We divide the proof into two steps.

\smallskip
\textbf{1.} First, for the $L^2(I)$-estimate of $r^\frac{m}{2}\big(\phi_{rrr},(\frac{\phi_{r}}{r})_{r}\big)$, it follows from Lemmas \ref{important2}, \ref{l4.7}--\ref{l4.9}, and \ref{lemma-psi-high} that
\begin{equation}\label{lem612-2}
\begin{aligned}
\Big|r^\frac{m}{2}\big(\frac{\phi_{r}}{r}\big)_{r}\Big|_2&\leq C_0\Big|r^\frac{m}{2}\big(\phi\frac{\psi}{r}\big)_{r}\Big|_2\leq  C_0|\phi|_\infty\Big|r^\frac{m}{2}\big(\frac{\psi}{r}\big)_{r}\Big|_2+C_0|\psi|_\infty|r^\frac{m-2}{2}\phi_r|_2 \leq  C(T),\\
|r^\frac{m}{2}\phi_{rrr}|_2 &\leq C_0|r^\frac{m}{2}\phi_{rr}\psi|_2 +C_0|r^\frac{m}{2}\phi_{r}\psi_r|_2+C_0|r^\frac{m}{2}\phi\psi_{rr}|_2\\
&\leq C_0|r^\frac{m}{2}\phi_{rr}\psi|_2 +C_0|r^\frac{m}{2}\phi\psi\psi_r|_2+C_0|r^\frac{m}{2}\phi\psi_{rr}|_2\\
&\leq C_0|(\phi,\psi)|_\infty |r^\frac{m}{2}(\phi_{rr},\psi_{rr})|_2+C_0|\phi|_\infty|\psi|_\infty|r^\frac{m}{2}\psi_r|_2 \leq C(T).
\end{aligned}
\end{equation}

Notice from Lemmas \ref{important2}, \ref{l4.7}--\ref{l4.9}, and \ref{lemma-psi-high} that
\begin{equation}\label{lem612-8}
\begin{aligned}
\Big|r^\frac{m}{2}\big(\phi_{trr},\frac{\phi_{tr}}{r}\big)\Big|_2&\leq C_0\Big|r^\frac{m}{2}\Big((\phi\psi)_{tr},\frac{1}{r}(\phi\psi)_t\Big)\Big|_2\\
&\leq C_0|(\phi,\psi)|_\infty \Big|r^\frac{m}{2}\Big(\phi_{tr},\psi_{tr},\frac{\psi_t}{r}\Big)\Big|_2+C_0|r^\frac{m-2}{2}\phi_t\psi|_2 \\
&\quad +C_0|r^\frac{m}{2}\phi_r\psi_t|_2+C_0|r^\frac{m}{2}\phi_t\psi_r|_2\\
&\leq C(T) +C_0|r^\frac{m}{2}\phi_r\psi_t|_2+C_0\Big|r^\frac{m}{2}\phi_t\big(\psi_r,\frac{\psi}{r}\big)\Big|_2.
\end{aligned}
\end{equation}
For the estimates of $|r^\frac{m}{2}\phi_r\psi_t|_2$ and $|r^\frac{m}{2}\phi_t(\psi_r,\frac{\psi}{r})|_2$
on the right-hand side of the above, it follows from Lemmas \ref{lemma66}--\ref{l4.9}, \ref{lemma-psi-high}, \ref{ale1}, and \ref{hardy} that
\begin{equation}
\begin{aligned}
|r^\frac{m}{2}\phi_r\psi_t|_2&\leq |\chi_1^\flat r^\frac{m}{2}\phi_r\psi_t|_2+|\chi_1^\sharp r^\frac{m}{2}\phi_r\psi_t|_2\\
&\leq  |\chi_1^\flat r\phi_r|_\infty|r^\frac{m-2}{2}\psi_t|_2+|\chi_1^\sharp \phi_r|_\infty|r^\frac{m}{2}\psi_t|_2\\
&\leq  C(T)\big(|\chi_1^\flat r^\frac{3}{2}(\phi_r,\phi_{rr})|_2+|\chi_1^\sharp (\phi_r,\phi_{rr})|_2\big) \\
&\leq  C(T)\big(|\chi_1^\flat r^\frac{3-m}{2}|_\infty+ |\chi_1^\sharp r^{-\frac{m}{2}}|_\infty\big) |r^\frac{m}{2}(\phi_r,\phi_{rr})|_2\leq C(T),    
\end{aligned}
\end{equation}
\begin{equation}\label{lem612-9}
\begin{aligned}
\Big|r^\frac{m}{2}\phi_t\big(\psi_r,\frac{\psi}{r}\big)\Big|_2&\leq \Big|\chi_1^\flat r^\frac{m}{2}\phi_t\big(\psi_r,\frac{\psi}{r}\big)\Big|_2+\Big|\chi_1^\sharp r^\frac{m}{2}\phi_t\big(\psi_r,\frac{\psi}{r}\big)\Big|_2\\
&\leq  C_0\Big|\chi_1^\flat r^\frac{m+2}{2}\phi_t\big(\psi_r,\frac{\psi}{r}\big)\Big|_2+C_0\Big|\chi_1^\flat r^\frac{m+2}{2}\phi_{tr} \big(\psi_r,\frac{\psi}{r}\big)\Big|_2\\
&\quad +C_0\Big|\chi_1^\flat r^\frac{m+2}{2}\phi_t\Big(\psi_{rr},\big(\frac{\psi}{r}\big)_r\Big)\Big|_2+|\chi_1^\sharp \phi_t|_\infty\Big|r^\frac{m}{2}\big(\psi_r,\frac{\psi}{r}\big)\Big|_2\\
&\leq  C_0 |(\chi_1^\flat r\psi_r,\psi)|_\infty |r^\frac{m}{2}(\phi_t,\phi_{tr})|_2\\
&\quad +C_0|\chi_1^\flat r\phi_t|_\infty \Big|r^\frac{m}{2}\Big(\psi_{rr},\big(\frac{\psi}{r}\big)_r\Big)\Big|_2 +C(T)|\chi_1^\sharp \phi_{t}|_\infty \\
&\leq  C(T) \big(1+|\chi_1^\flat r^\frac{3}{2}(\psi_r,\psi_{rr},\phi_t,\phi_{tr})|_2 +|\chi_1^\sharp(\phi_{t},\phi_{tr})|_2\big) \\
&\leq  C(T) \big(1+|r^\frac{m}{2}(\psi_r,\psi_{rr},\phi_{t},\phi_{tr})|_2\big)\leq C(T).
\end{aligned}
\end{equation}
Combining with \eqref{lem612-8}--\eqref{lem612-9} yields that, for all $t\in [0,T]$,
\begin{equation}\label{lem612-10}
\Big|r^\frac{m}{2}\big(\phi_{trr},\frac{\phi_{tr}}{r}\big)(t)\Big|_2\leq C(T).
\end{equation}

\smallskip
\textbf{2.} First, multiplying $\eqref{e2.2}_2$ by $\frac{1}{r}$, applying $r^\frac{m}{2}\partial_r$, and taking the $L^2(I)$-norm of both sides of the resulting equality, we obtain  from \eqref{lem612-2}, and Lemmas \ref{lemma66}--\ref{l4.7} and \ref{H3-1} that 
\begin{equation}\label{lem612-1}
\begin{aligned}
\Big|r^\frac{m}{2}\Big(\frac{1}{r}\big(u_r+\frac{m}{r}u\big)_{r}\Big)_r\Big|_2&\leq  \Big|r^\frac{m}{2}\Big(\big(\frac{u_{t}}{r}\big)_r,\big(\frac{u}{r}u_r\big)_{r},\big(\frac{\phi_{r}}{r}\big)_{r},2\alpha\big(\frac{\psi}{r} u_r\big)_{r}\Big)\Big|_2\\
&\leq  \Big|r^\frac{m}{2}\Big(\big(\frac{u_{t}}{r}\big)_r,\big(\frac{\phi_r}{r}\big)_r\Big)\Big|_2+C_0|(u,\psi)|_\infty|r^\frac{m-2}{2}u_{rr}|_2\\
&\quad +C_0\Big|r^\frac{m}{2}\Big(\big(\frac{u}{r}\big)_r,\big(\frac{\psi}{r}\big)_r\Big)\Big|_2|u_r|_\infty\leq \Big|r^\frac{m}{2}\big(\frac{u_{t}}{r}\big)_r\Big|_2 +C(T).
\end{aligned}
\end{equation}

Taking the square of above and then integrating over $[0,t]$, together with Lemma \ref{H3-1}, yield that, for all $t\in [0,T]$,
\begin{equation}\label{lem612-3}
\int_0^t\Big|r^\frac{m}{2}\Big(\frac{1}{r}\big(u_r+\frac{m}{r}u\big)_{r}\Big)_r\Big|_2^2\,\mathrm{d}s\leq C(T).
\end{equation}

Next, applying $r^\frac{m}{2}\partial_{r}^2$ to $\eqref{e2.2}_2$ and taking $L^2(I)$-norm of the resulting equality, we first obtain 
from \eqref{lem612-2}, and Lemmas \ref{lemma66}--\ref{l4.7} and \ref{H3-1}--\ref{lemma-psi-high} that
\begin{equation}\label{lem612-4}
\begin{aligned}
\Big|r^\frac{m}{2}\big(u_r+\frac{m}{r}u\big)_{rrr}\Big|_2&\leq \big|r^\frac{m}{2}\big(u_{trr},(uu_r)_{rr},\phi_{rrr},2\alpha(\psi u_r)_{rr}\big)\big|_2\\
&\leq |r^\frac{m}{2}(u_{trr},\phi_{rrr})|_2+|(u,\psi)|_\infty |r^\frac{m}{2}u_{rrr}|_2+C_0|r^\frac{m}{2}u_{rr} \psi_r |_2\\
&\quad + C_0|r^\frac{m}{2}(u_{rr},\psi_{rr})|_2|u_r|_\infty \\
&\leq |r^\frac{m}{2}u_{trr}|_2+C_0|r^\frac{m}{2}u_{rr} \psi_r|_2+C(T).
\end{aligned}
\end{equation}
Then, for the estimate of $|r^\frac{m}{2}u_{rr} \psi_r|_2$ on the right-hand side of the above, 
it follows from Lemmas \ref{lemma66}--\ref{l4.7}, \ref{H3-1}--\ref{lemma-psi-high}, \ref{ale1}, 
and \ref{hardy} that 
\begin{equation}\label{lem612-5}
\begin{aligned}
|r^\frac{m}{2}u_{rr}\psi_r|_2&\leq |\chi_1^\flat r^\frac{m}{2}u_{rr}\psi_r|_2+|\chi_1^\sharp r^\frac{m}{2}u_{rr}\psi_r|_2\\
&\leq C_0\big(|\chi_1^\flat r^\frac{m+2}{2}(u_{rr}\psi_r,u_{rrr}\psi_r,u_{rr}\psi_{rr})|_2+|\chi_1^\sharp r^\frac{m}{2}u_{rr}\psi_r|_2\big)\\
&\leq C_0\big(|r^\frac{m}{2}(u_{rr},u_{rrr})|_2|\chi_1^\flat r \psi_r|_\infty+ |r^\frac{m}{2}\psi_{rr}|_2|\chi_1^\flat ru_{rr}|_\infty+ |\chi_1^\sharp \psi_r|_\infty |r^\frac{m}{2}u_{rr}|_2\big)\\
&\leq C(T)\big(|\chi_1^\flat r^\frac{3}{2}(\psi_r,\psi_{rr},u_{rr},u_{rrr})|_2+|\chi_1^\sharp(\psi_r,\psi_{rr})|_2\big)\\
&\leq C(T)\big(|\chi_1^\flat r^\frac{3-m}{2}|_\infty+|\chi_1^\sharp r^{-\frac{m}{2}}|_\infty\big)|r^\frac{m}{2}(u_{rr},u_{rrr},\psi_r,\psi_{rr})|_2 \leq C(T).
\end{aligned}
\end{equation}
Thus, combining \eqref{lem612-4}--\eqref{lem612-5} gives
\begin{equation}\label{lem612-4''}
\Big|r^\frac{m}{2}\big(u_r+\frac{m}{r}u\big)_{rrr}\Big|_2 \leq |r^\frac{m}{2}u_{trr}|_2 +C(T).
\end{equation}
Taking the square of the above and then integrating over $[0,t]$, together with Lemma \ref{H3-1}, give that, for all $t\in [0,T]$,
\begin{equation}\label{lem612-6}
\int_0^t\Big|r^\frac{m}{2}\big(u_r+\frac{m}{r}u\big)_{rrr}\Big|_2^2\,\mathrm{d}s\leq C(T),
\end{equation}
which, along with \eqref{lem612-3}, and Lemma \ref{im-1}, leads to
\begin{equation}\label{lem612-7}
\int_0^t\Big|r^\frac{m}{2}\Big(u_{rrrr},\big(\frac{u_{rr}}{r}\big)_r,\big(\frac{u}{r}\big)_{rrr},\Big(\frac{1}{r}\big(\frac{u}{r}\big)_r\Big)_r\Big)\Big|_2^2\,\mathrm{d}s\leq C(T).
\end{equation}
This completes the proof.
\end{proof}

\subsection{Time-weighted estimates of the velocity}
We now establish the time-weighted fourth-order tangential estimates for $u$.
\begin{lem}\label{Lemma6.13}  
There exists a constant $C(T)>0$ such that
\begin{equation*}
\begin{split}
&t|r^{\frac{m}{2}}u_{tt}(t)|^2_2+\int_{0}^{t} s \Big|r^{\frac{m}{2}}\big(u_{ttr},\frac{u_{tt}}{r}\big)\Big|^2_2 \,\mathrm{d}s\leq C(T) \qquad \mbox{for any $t\in [0,T]$}.
\end{split}
\end{equation*}
\end{lem}

\begin{proof} We divide the proof into three steps.

\smallskip
\textbf{1.} We give some estimates to be used later. First, it follows from Lemmas \ref{l4.8}, \ref{H3-1}, \ref{ale1}, and \ref{hardy} that
\begin{equation}\label{lem613-1}
\begin{aligned}
|u_t|_\infty&\leq C_0|(u_t,u_{tr})|_2\leq C_0|\chi_1^\flat (u_t,u_{tr})|_2+C_0|\chi_1^\sharp (u_t,u_{tr})|_2\\
&\leq C_0|\chi_1^\flat r(u_t,u_{tr},u_{trr})|_2+C_0|\chi_1^\sharp (u_t,u_{tr})|_2\\
&\leq C_0  |r^\frac{m}{2}(u_t,u_{tr},u_{trr})|_2 \leq C(T) \big(|r^\frac{m}{2}u_{trr}|_2+1\big),\\
|\chi_1^\flat ru_{tr}|_\infty &\leq C_0|\chi_1^\flat r^\frac{3}{2}(u_{tr},u_{trr})|_2 \leq C_0 |r^\frac{m}{2}(u_{tr},u_{trr})|_2 \leq C(T) \big(|r^\frac{m}{2}u_{trr}|_2+1\big),\\
|\chi_1^\sharp u_{tr}|_\infty&\leq C_0|\chi_1^\sharp (u_{tr},u_{trr})|_2\leq C_0 |r^\frac{m}{2}(u_{tr},u_{trr})|_2 \leq C(T) \big(|r^\frac{m}{2}u_{trr}|_2+1\big).
\end{aligned}    
\end{equation}

Next, applying $r^\frac{m}{2}\partial_t$ to $\eqref{e2.2}_3$ and taking the $L^2(I)$-norm of both sides of the resulting equality, we obtain from \eqref{lem613-1}, and
Lemmas \ref{lemma66}--\ref{l4.7} and \ref{H3-1}--\ref{lemma-psi-high} that
\begin{equation}\label{lem613-1'}
\begin{aligned}
|r^\frac{m}{2}\psi_{tt}|_2&\leq \Big|r^\frac{m}{2}\Big((\psi u)_{tr},\big(u_r+\frac{m}{r}u\big)_{tr}\Big)\Big|_2\\
&\leq |\psi|_\infty |r^\frac{m}{2}u_{tr}|_2+|u_r|_\infty|r^\frac{m}{2}\psi_t|_2+|u_t|_\infty|r^\frac{m}{2}\psi_r|_2+|u|_\infty |r^\frac{m}{2}\psi_{tr}|_2\\
&\quad +C_0\Big|r^\frac{m}{2}\Big(u_{trr},\big(\frac{u_t}{r}\big)_{r}\Big)\Big|_2\leq C(T) \Big(\Big|r^\frac{m}{2}\Big(u_{trr},\big(\frac{u_t}{r}\big)_{r}\Big)\Big|_2+1\Big).
\end{aligned}
\end{equation}

Finally, based on $\eqref{e2.2}_1$, it follows from Lemmas \ref{important2}, \ref{l4.8}--\ref{l4.9}, 
and \ref{H3-1} that
\begin{equation}\label{cal-phitt}
\begin{aligned}
|\phi_t|_\infty&\leq C_0 \Big|\Big(u\phi_r,\phi\big(u_r+\frac{m}{r}u\big)\Big)\Big|_\infty \leq C_0 |\phi|_\infty \Big(|u|_\infty |\psi|_\infty+\Big|\big(u_r,\frac{u}{r}\big)\Big|_\infty\Big)\leq C(T),\\
|r^\frac{m}{2}\phi_{tt}|_2&\leq |r^\frac{m}{2}(u\phi_r)_t|_2+C_0\Big|r^\frac{m}{2}\Big(\phi\big(u_r+\frac{m}{r}u\big)\Big)_t\Big|_2\\
&\leq |u|_\infty |r^\frac{m}{2}\phi_{tr}|_2+|\phi|_\infty|\psi|_\infty|r^\frac{m}{2}u_t|_2\\
&\quad +C_0|r^\frac{m}{2}\phi_t|_2\Big|\big(u_r,\frac{u}{r}\big)\Big|_\infty+C_0|\phi|_\infty\Big|r^\frac{m}{2}\big(u_{tr},\frac{u_t}{r}\big)\Big|_2\leq C(T),
\end{aligned}    
\end{equation}
which, along with the chain rule, \eqref{lem613-1'}, and Lemma \ref{lemma66}, leads to
\begin{equation}\label{phittr}
\begin{aligned}
|r^\frac{m}{2}\phi_{ttr}|_2&\leq C_0\big|r^\frac{m}{2}(\phi_{tt}\psi,\phi_t\psi_t,\phi\psi_{tt})\big|_2 \leq C_0\big(|(\phi,\psi)|_\infty \big|r^\frac{m}{2}(\phi_{tt},\psi_{tt})\big|_2+ |\phi_t|_\infty |r^\frac{m}{2}\psi_t|_2\big) \\
&\leq C(T) \Big(\Big|r^\frac{m}{2}\Big(u_{trr},\big(\frac{u_t}{r}\big)_{r}\Big)\Big|_2+1\Big).
\end{aligned}    
\end{equation}

\smallskip
\textbf{2.} We now proceed to prove Lemma \ref{Lemma6.13}. 
Formally, applying $r^mu_{tt}\partial_{t}^2$ to both sides of $\eqref{e2.2}_2$ and integrating the resulting equality over $I$ yield that
\begin{equation}\label{lem613-2}
\begin{aligned}
&\frac{1}{2}\frac{\mathrm{d}}{\mathrm{d}t}|r^\frac{m}{2}u_{tt}|_2^2+2\alpha\Big|r^\frac{m}{2}\big(u_{ttr}+\frac{m}{r}u_{tt}\big)\Big|_2^2\\
&=-\int_0^\infty r^m(uu_r)_{tt}u_{tt}\,\mathrm{d}r-\int_0^\infty r^m \phi_{ttr}u_{tt}\,\mathrm{d}r +2\alpha\int_0^\infty r^m(\psi u_r)_{tt}u_{tt}\,\mathrm{d}r:=\sum_{i=21}^{23}\mathcal{G}_i
\end{aligned}
\end{equation}
for {\it a.e.} $t\in (\tau,T)$ and $\tau\in (0,T)$. Here, we temporarily assume that the above energy equality holds, and the specific proof will be provided in Step 3 below.

Then, for $\mathcal{G}_{21}$--$\mathcal{G}_{22}$, it follows from \eqref{lem613-1}, \eqref{phittr}, Lemmas \ref{im-1}, \ref{lemma66}, and \ref{H3-1}, and the  H\"older and Young inequalities that
\begin{equation}\label{G21-G22}
\begin{aligned}
\mathcal{G}_{21}&= -\int_0^\infty r^m(u_{tt}u_r+2u_tu_{tr}+uu_{ttr})u_{tt}\,\mathrm{d}r\\
&\leq |u_r|_\infty |r^\frac{m}{2}u_{tt}|_2^2+2|u_t|_\infty |r^\frac{m}{2}u_{tr}|_2|r^\frac{m}{2}u_{tt}|_2+|u|_\infty |r^\frac{m}{2}u_{ttr}|_2|r^\frac{m}{2}u_{tt}|_2\\
&\leq C(T)\big(|r^\frac{m}{2}u_{tt}|_2^2+1\big)+\frac{\alpha}{8}\Big|r^\frac{m}{2}\big(u_{ttr}+\frac{m}{r}u_{tt}\big)\Big|_2^2,\\
\mathcal{G}_{22}&\leq C(T) \Big(|r^\frac{m}{2}u_{tt}|_2^2+\Big|r^\frac{m}{2}\Big(u_{trr},\big(\frac{u_t}{r}\big)_{r}\Big)\Big|_2^2+1\Big).
\end{aligned}
\end{equation}
For $\mathcal{G}_{23}$, it follows from $\eqref{e2.2}_3$, \eqref{lem613-1}--\eqref{lem613-1'}, 
Lemmas \ref{im-1}, \ref{lemma66}--\ref{l4.7}, and \ref{H3-1}--\ref{lemma-psi-high}, and the H\"older and Young inequalities that
\begin{equation}\label{G23}
\begin{aligned}
\mathcal{G}_{23}&=-2\alpha\int_0^\infty r^m(\psi_{tt}u_r+2\psi_tu_{tr}+\psi u_{ttr})u_{tt}\,\mathrm{d}r\\
&\leq C_0|u_r|_\infty|r^\frac{m}{2}\psi_{tt}|_2|r^\frac{m}{2}u_{tt}|_2+C_0|\chi_1^\flat ru_{tr}|_\infty |r^\frac{m-2}{2}\psi_t|_2|r^\frac{m}{2}u_{tt}|_2\\
&\quad +C_0|\chi_1^\sharp u_{tr}|_\infty |r^\frac{m}{2}\psi_t|_2|r^\frac{m}{2}u_{tt}|_2+C_0|\psi|_\infty |r^\frac{m}{2}u_{ttr}|_2|r^\frac{m}{2}u_{tt}|_2\\
&\leq C(T)\Big(\Big|r^\frac{m}{2}\Big(u_{trr},\big(\frac{u_t}{r}\big)_{r}\Big)\Big|_2+|r^\frac{m}{2}u_{ttr}|_2+1\Big)|r^\frac{m}{2}u_{tt}|_2\\
&\leq C(T)|r^\frac{m}{2}u_{tt}|_2^2+C(T)\Big(\Big|r^\frac{m}{2}\Big(u_{trr},\big(\frac{u_t}{r}\big)_{r}\Big)\Big|_2+1\Big)+\frac{\alpha}{8} \Big|r^\frac{m}{2}\big(u_{ttr}+\frac{m}{r}u_{tt}\big)\Big|_2^2.
\end{aligned}
\end{equation}

Thus, collecting \eqref{lem613-2}--\eqref{G23}, along with Lemma \ref{im-1}, gives
\begin{equation*} 
\begin{aligned}
\frac{\mathrm{d}}{\mathrm{d}t}|r^\frac{m}{2}u_{tt}|_2^2+\alpha\Big|r^\frac{m}{2}\big(u_{ttr},\frac{u_{tt}}{r}\big)\Big|_2^2&\leq C(T)\Big(|r^\frac{m}{2}u_{tt}|_2^2+\Big|r^\frac{m}{2}\Big(u_{trr},\big(\frac{u_{t}}{r}\big)_r\Big)\Big|_2^2 +1\Big).
\end{aligned}
\end{equation*}
Multiplying the above by $t$ and integrating the resulting inequality over $[\tau, t]$ for $\tau\in(0,t)$, along with Lemmas \ref{H3-1} and \ref{Lemma6.12}, imply that
\begin{equation}\label{etrq2'}
t|r^{\frac{m}{2}}u_{tt}(t)|_2^2+ \int_\tau^t s\Big|r^\frac{m}{2}\big(u_{ttr},\frac{u_{tt}}{r}\big)\Big|_2^2\,\mathrm{d}s\leq C(T)\big(\tau|r^{\frac{m}{2}}u_{tt}(\tau)|_2^2+1\big).
\end{equation}

Next, thanks to Lemma \ref{H3-1}, $r^{\frac{m}{2}}u_{tt}\in L^2([0,T];L^2(I))$, it follows from  Lemma \ref{bjr} that there exists a sequence $\{\tau_k\}_{k=1}^\infty$ such that 
\begin{equation}\label{qujin}
\tau_k\to 0, \quad\, \tau_k|r^\frac{m}{2}u_{tt}(\tau_k)|_2^2\to 0 \qquad\,\, \text{as $k\to \infty$}.
\end{equation}
Choosing $\tau=\tau_k\to 0$  in  \eqref{etrq2'} yields that, for all $t\in [0,T]$, 
\begin{equation}\label{qunjin2}
t|r^{\frac{m}{2}}u_{tt}(t)|_2^2+ \int_0^t s\Big|r^\frac{m}{2}\big(u_{ttr},\frac{u_{tt}}{r}\big)\Big|_2^2\,\mathrm{d}s\leq C(T).
\end{equation}

\smallskip 
\textbf{3.} In the final step, we give a rigorous proof of the energy equality \eqref{lem613-2}. Suppose that $\boldsymbol{\varphi}(\boldsymbol{x})=\varphi(r)\frac{\boldsymbol{x}}{r}$ is any given spherically symmetric vector function satisfying $r^\frac{m}{2}(\varphi,\varphi_r,\frac{\varphi}{r})\in L^2(I)$. Of course, by Lemma \ref{lemma-initial}, this is equivalent to $\boldsymbol{\varphi}\in H^1(\mathbb{R}^n)$. 
Then applying $r^m\varphi\partial_t$ to both sides of $\eqref{e2.2}_2$ and integrating the resulting equality over $I$ yield from integration by parts that 
\begin{equation*}
\begin{aligned}
\int_0^\infty r^mu_{tt}\varphi\,\mathrm{d}r&=-2\alpha\int_0^\infty r^m\big(u_{tr}+\frac{m}{r}u_t\big)\big(\varphi_r+\frac{m}{r}\varphi\big)\,\mathrm{d}r \\
&\quad -\int_0^\infty r^m\big((uu_r)_t+\phi_{tr}-2\alpha(\psi u_r)_t\big)\varphi\,\mathrm{d}r.
\end{aligned}    
\end{equation*}
Here, the process of integration by parts can be justified by following the similar discussions in \eqref{eq:B10pre}--\eqref{util2}. 

Next, differentiating the above with respect to $t$ leads to
\begin{equation}\label{above-for}
\begin{aligned}
\frac{\mathrm{d}}{\mathrm{d}t}\int_0^\infty r^mu_{tt}\varphi\,\mathrm{d}r&=-2\alpha\int_0^\infty r^m\big(u_{ttr}+\frac{m}{r}u_{tt}\big)\big(\varphi_r+\frac{m}{r}\varphi\big)\,\mathrm{d}r\\
&\quad -\int_0^\infty r^m\big((uu_r)_{tt}+\phi_{ttr}-2\alpha(\psi u_r)_{tt}\big)\varphi\,\mathrm{d}r.
\end{aligned}    
\end{equation}
Based on the {\it a priori} assumption $r^\frac{m}{2}(u_{ttr},\frac{u_{tt}}{r},\phi_{ttr},\psi_{tt})\in L^2([\tau,T];L^2(I))$ for $\tau\in (0,T)$, 
and the calculations of $\mathcal{G}_{21}$--$\mathcal{G}_{23}$ in Step 2, we obtain 
\begin{equation}
\frac{1}{\omega_n}\frac{\mathrm{d}}{\mathrm{d}t}\int_{\mathbb{R}^n} \boldsymbol{u}_{tt}\cdot\boldsymbol{\varphi}\,\mathrm{d}\boldsymbol{x}=\frac{\mathrm{d}}{\mathrm{d}t}\int_0^\infty r^mu_{tt}\varphi\,\mathrm{d}r\leq F(t)\Big|r^\frac{m}{2}\Big(\varphi,\varphi_r,\frac{\varphi}{r}\Big)\Big|_2
\end{equation}
for some positive function $F(t)\in L^2(\tau,T)$, where $\omega_n$ denotes the surface area of the $n$-sphere. Thus, it follows from Lemma \ref{lemma-initial} and Lemma 1.1 on \cite[page 250]{temam} that $\boldsymbol{u}_{ttt}\in L^2([\tau,T];H^{-1}(\mathbb{R}^n))$. Consequently, the energy equality follows easily from the spherical coordinates transformation and the following identity, due to Lemma \ref{triple}:
\begin{equation*}
\frac{\mathrm{d}}{\mathrm{d}t}\int_0^\infty r^m|u_{tt}|^2\,\mathrm{d}r=\frac{1}{\omega_n}\frac{\mathrm{d}}{\mathrm{d}t}\int_{\mathbb{R}^n} |\boldsymbol{u}_{tt}|^2\,\mathrm{d}\boldsymbol{x}=\frac{2}{\omega_n}\langle \boldsymbol{u}_{ttt},\boldsymbol{u}_{tt}\rangle_{H^{-1}(\mathbb{R}^n)\times H^1(\mathbb{R}^n)}.
\end{equation*}
The proof of Lemma \ref{Lemma6.13} is completed.
\end{proof}

\begin{lem}\label{Lemma6.14}  
There exists a constant $C(T)>0$ such that, for any $t\in [0,T]$,
\begin{equation*}
\begin{aligned}
&\sqrt{t}\Big|r^{\frac{m}{2}}\Big(u_{trr},\big(\frac{u_{t}}{r}\big)_r,u_{rrrr},\big(\frac{u_{rr}}{r}\big)_r,\big(\frac{u}{r}\big)_{rrr},\Big(\frac{1}{r}\big(\frac{u}{r}\big)_r\Big)_r\Big)(t)\Big|_2\\
&+\int_0^t s\Big|r^{\frac{m}{2}}\Big(u_{trrr},\frac{u_{trr}}{r},\big(\frac{u_t}{r}\big)_{rr},\frac{1}{r}\big(\frac{u_{tr}}{r}\big)_r\Big)\Big|_2^2\,\mathrm{d}s\leq C(T).
\end{aligned}
\end{equation*}
\end{lem}
\begin{proof}We divide the proof into two steps.

\smallskip
\textbf{1.} It follows from \eqref{utt} and Lemma \ref{Lemma6.13} that 
\begin{equation}\label{lem614-1}
\sqrt{t}\Big|r^{\frac{m}{2}}\Big(u_{trr},\big(\frac{u_{t}}{r}\big)_r\Big)(t)\Big|_2\leq C(T)\big(\sqrt{t}|r^\frac{m}{2}u_{tt}(t)|_2+1\big)\leq C(T).
\end{equation}
Then, according to \eqref{lem612-1}--\eqref{lem612-2}, \eqref{lem612-4}--\eqref{lem612-5}, \eqref{lem614-1}, 
and Lemma \ref{im-1}, we have
\begin{equation}\label{lem614-1'}
\begin{aligned}
&\sqrt{t}\Big|r^{\frac{m}{2}}\Big(u_{rrrr},\big(\frac{u_{rr}}{r}\big)_r,\big(\frac{u}{r}\big)_{rrr},\Big(\frac{1}{r}\big(\frac{u}{r}\big)_r\Big)_r\Big)(t)\Big|_2\\
&\leq C_0\sqrt{t}\Big|r^{\frac{m}{2}}\Big(u_{trr},\big(\frac{u_{t}}{r}\big)_r\Big)(t)\Big|_2+C(T)\leq C(T).
\end{aligned}    
\end{equation}

\smallskip
\textbf{2.} Apply $r^\frac{m-2}{2}\partial_t$ and $r^\frac{m}{2}\partial_{t}\partial_{r}$ to $\eqref{e2.2}_2$, respectively. Then, taking the $L^2(I)$-norm of these two resulting equality, 
we obtain from \eqref{lem613-1}, and Lemmas \ref{lemma66}--\ref{l4.7} and \ref{H3-1}--\ref{Lemma6.12} that
\begin{align*}
&\begin{aligned}
&\Big|r^\frac{m}{2}\Big(\frac{1}{r}\big(u_{tr}+\frac{m}{r}u_t\big)_r,\big(u_{tr}+\frac{m}{r}u_t\big)_{rr}\Big)\Big|_2\\
&\leq \Big|r^\frac{m}{2}\big(\frac{u_{tt}}{r},u_{ttr}\big)\Big|_2+\Big|r^\frac{m}{2}\Big(\frac{1}{r}(uu_r)_t,(uu_r)_{tr}\Big)\Big|_2+\Big|r^\frac{m}{2}\big(\frac{\phi_{tr}}{r},\phi_{trr}\big)\Big|_2\\
& \quad +2\alpha\Big|r^\frac{m}{2}\Big(\frac{1}{r}(\psi u_r)_t,(\psi u_r)_{tr}\Big)\Big|_2
\end{aligned}\\
&\begin{aligned}
&\leq \Big|r^\frac{m}{2}\big(\frac{u_{tt}}{r},u_{ttr}\big)\Big|_2+C_0|(u,\psi)|_\infty|r^\frac{m}{2}u_{trr}|_2\\
&\quad +C_0\Big|r^\frac{m}{2}\Big(\frac{u_t}{r},u_{tr},\frac{\psi_t}{r},\psi_{tr}\Big)\Big|_2|u_r|_\infty+C_0|r^\frac{m}{2}u_{tr}|_2\Big|\frac{u}{r}\Big|_\infty\\
&\quad +C_0|\chi_1^\flat r^\frac{m}{2}u_{tr}|_\infty \Big|\chi_1^\flat \big(\frac{\psi}{r},\psi_r\big)\Big|_2+C_0|\chi_1^\sharp u_{tr}|_\infty \Big|r^\frac{m}{2}\big(\frac{\psi}{r},\psi_r\big)\Big|_2\\
&\quad +C_0|u_t|_\infty|r^\frac{m}{2}u_{rr}|_2+C_0|r^\frac{m-2}{2}\psi_t|_2|\chi_1^\flat ru_{rr}|_\infty+C_0|r^\frac{m}{2}\psi_t|_2|\chi_1^\sharp u_{rr}|_\infty+C(T)
\end{aligned}\\
&\begin{aligned}
&\leq \Big|r^\frac{m}{2}\big(\frac{u_{tt}}{r},u_{ttr}\big)\Big|_2+C(T)|r^\frac{m}{2}u_{trr}|_2+C_0|\chi_1^\flat r^\frac{m}{2}u_{tr}|_\infty \Big|\chi_1^\flat \big(\frac{\psi}{r},\psi_r\big)\Big|_2+C_0|\chi_1^\sharp u_{tr}|_\infty\\
&\quad +C(T)\big(|\chi_1^\flat ru_{rr}|_\infty+|\chi_1^\sharp u_{rr}|_\infty+1\big)
\end{aligned}\\
&\begin{aligned}
&\leq \Big|r^\frac{m}{2}\big(\frac{u_{tt}}{r},u_{ttr}\big)\Big|_2+C(T)|r^\frac{m}{2}u_{trr}|_2\\
&\quad +C_0|\chi_1^\flat r^\frac{m+1}{2}(u_{tr},u_{trr})|_2 \Big|\chi_1^\flat r\Big(\frac{\psi}{r},\big(\frac{\psi}{r}\big)_r,\psi_r,\psi_{rr}\Big)\Big|_2+C_0|\chi_1^\sharp (u_{tr},u_{trr})|_2\\
&\quad +C(T)\big(|\chi_1^\flat r^\frac{3}{2}(u_{rr},u_{rrr})|_2+|\chi_1^\sharp (u_{rr},u_{rrr})|_2+1\big)
\end{aligned}\\
&\begin{aligned}
&\leq \Big|r^\frac{m}{2}\big(\frac{u_{tt}}{r},u_{ttr}\big)\Big|_2+C(T)|r^\frac{m}{2}u_{trr}|_2\\
&\quad +C_0|\chi_1^\flat r^\frac{1}{2}|_\infty|\chi_1^\flat r^\frac{2-m}{2}|_\infty |r^\frac{m}{2}(u_{tr},u_{trr})|_2 \Big|r^\frac{m}{2}\Big(\frac{\psi}{r},\big(\frac{\psi}{r}\big)_r,\psi_r,\psi_{rr}\Big)\Big|_2\\
&\quad +C_0|\chi_1^\sharp r^{-\frac{m}{2}}|_\infty |r^\frac{m}{2}(u_{tr},u_{trr})|_2\\
&\quad +C(T)\big(|\chi_1^\flat r^\frac{3-m}{2}|_\infty+|\chi_1^\sharp r^{-\frac{m}{2}}|_\infty\big) |r^\frac{m}{2}(u_{rr},u_{rrr})|_2+C(T)
\end{aligned}\\
&\begin{aligned}
&\leq \big|r^\frac{m}{2}(\frac{u_{tt}}{r},u_{ttr})\big|_2+C(T)|r^\frac{m}{2}u_{trr}|_2+C(T).
\end{aligned}
\end{align*}

It follows  from the above estimate, \eqref{lem614-1}, and Lemmas \ref{im-1} and \ref{Lemma6.13} that
\begin{align*}
&\int_0^t s\Big|r^{\frac{m}{2}}\Big(u_{trrr},\frac{u_{trr}}{r},\big(\frac{u_t}{r}\big)_{rr},\frac{1}{r}\big(\frac{u_{tr}}{r}\big)_r\Big)\Big|_2^2\,\mathrm{d}s\\
&\leq  C_0\int_0^t s\Big|r^\frac{m}{2}\Big(\frac{1}{r}\big(u_{tr}+\frac{m}{r}u_t\big)_r,\big(u_{tr}+\frac{m}{r}u_t\big)_{rr}\Big)\Big|_2^2\,\mathrm{d}s\\
&\leq C_0 \int_0^t s\Big|r^\frac{m}{2}\big(\frac{u_{tt}}{r},u_{ttr}\big)\Big|_2^2\,\mathrm{d}s+C(T)t \sup_{s\in[0,t]}|r^\frac{m}{2}u_{trr}|_2^2+C(T)t\leq C(T).
\end{align*}
This completes the proof.
\end{proof}

\section{Global Well-Posedness of  Regular Solutions with Far-Field Vacuum}\label{se46}
This section is devoted to proving Theorems \ref{th1}--\ref{th1-high} for the flow with far-field vacuum. 

\subsection{\textbf{Proof of Theorem \ref{th1}}}\label{solo} 
Based on  the  local well-posedness and the corresponding global uniform  estimates obtained above, now we can give the proof of   Theorem \ref{th1}. We divide the proof into four steps.

\smallskip
\textbf{1. Global well-posedness of $2$-order regular solutions.} First, according to Theorem \ref{zth1}, there exists a $2$-order 
regular solution $(\rho, \boldsymbol{u})(t,\boldsymbol{x})$ of the Cauchy problem \eqref{eq:1.1benwen}--\eqref{e1.3} in $[0,T_*]\times \mathbb{R}^n$ for some $T_*>0$, which takes form \eqref{duichenxingshi}.

Second, let $\overline{T}_*>0$ be the life span of $(\rho, \boldsymbol{u})(t,\boldsymbol{x})$, 
and let $T$ be any fixed time satisfying $T\in (0,\overline{T}_*)$.
Collecting the uniform {\it a priori} bounds obtained in Lemmas \ref{far-p-infty}, \ref{important2}, 
and \ref{l4.5}--\ref{l4.10-ell}, and then using \eqref{tr} and Lemma \ref{lemma-initial} yield that, for any  $t\in [0,T]$,
\begin{equation}\label{global-unifrom}
\begin{aligned}
\|\rho(t)\|_{L^1\cap L^\infty}+\|(\nabla\rho^{\gamma-1},(\rho^{\gamma-1})_t)(t)\|_{H^1}&\leq C(T),\\[4pt]
\|\nabla\log\rho(t)\|_{L^\infty}+\big\|\big(\nabla^2\log\rho,(\nabla\log\rho)_t\big)(t)\big\|_{L^2}&\leq C(T),\\[4pt]
\|\boldsymbol{u}(t)\|_{H^2}+\|\boldsymbol{u}_t(t)\|_{L^2}+\int_0^t \|(\nabla^3\boldsymbol{u},\nabla\boldsymbol{u}_t)\|_{L^2}^2 \,\mathrm{d}s&\leq C(T),\\
\sqrt{t}\|(\nabla^3\boldsymbol{u},\nabla\boldsymbol{u}_t)(t)\|_{L^2}+\int_0^t s\|(\nabla^2\boldsymbol{u}_t,\boldsymbol{u}_{tt})\|_{L^2}^2\,\mathrm{d}s&\leq C(T).
\end{aligned}    
\end{equation}

Clearly, $\overline{T}_*\ge T_*$.  Next, we show  that $\overline{T}_*=\infty$. Otherwise, if $\overline{T}_*<\infty$, 
according to  the uniform {\it a priori} estimates \eqref{global-unifrom} and the standard weak convergence arguments, for any sequence $\{t_k\}_{k=1}^\infty$ satisfying $0<t_k<\overline{T}_*$ and $t_k\to \overline{T}_*$ as $k\to \infty$, there exist a subsequence (still denoted by) $\{t_{k}\}_{k=1}^\infty$ and limits $(\rho,\boldsymbol{u}, \bar{\boldsymbol{f}}, \bar{\boldsymbol{\psi}},\bar{\mathcal{K}})(\overline{T}_*,\boldsymbol{x})$ satisfying 
\begin{align*}
&\rho(\overline{T}_*,\boldsymbol{x})\in L^p(\mathbb{R}^n) \ \ \text{for any $p\in (1,\infty]$},\quad\, \boldsymbol{u}(\overline{T}_*,\boldsymbol{x})\in H^2(\mathbb{R}^n),\\
&\bar{\boldsymbol{f}}(\overline{T}_*,\boldsymbol{x})\in H^1(\mathbb{R}^n),\quad \bar{\boldsymbol{\psi}}(\overline{T}_*,\boldsymbol{x})\in L^\infty(\mathbb{R}^n),\quad \bar{\mathcal{K}}(\overline{T}_*,\boldsymbol{x})\in L^2(\mathbb{R}^n), 
\end{align*}
and, as $k\to\infty$,
\begin{equation}\label{f2}
\begin{aligned}
\rho(t_k,\boldsymbol{x})\to \rho(\overline{T}_*,\boldsymbol{x}) \qquad &\text{weakly\  \, in } L^p(\mathbb{R}^n),\\
\boldsymbol{u}(t_{k},\boldsymbol{x})\to  \boldsymbol{u}(\overline{T}_*,\boldsymbol{x}) \qquad &\text{weakly\  \,  in } H^2(\mathbb{R}^n),\\
(\rho,\nabla\log\rho)(t_{k},\boldsymbol{x})\to  (\rho,\bar{\boldsymbol{\psi}})(\overline{T}_*,\boldsymbol{x}) \qquad &\text{weakly* in } L^\infty(\mathbb{R}^n),\\
\nabla\rho^{\gamma-1}(t_k,\boldsymbol{x})\to \bar{\boldsymbol{f}}(\overline{T}_*,\boldsymbol{x}) \qquad &\text{weakly  \, in } H^1(\mathbb{R}^n),\\ 
\nabla^2\log\rho(t_{k},\boldsymbol{x})\to  \bar{\mathcal{K}}(\overline{T}_*,\boldsymbol{x}) \qquad &\text{weakly  \, in } L^2(\mathbb{R}^n).
\end{aligned}
\end{equation}
Then we claim that
\begin{equation}\label{claim103}
(\bar{\boldsymbol{f}},\bar{\boldsymbol{\psi}},\bar{\mathcal{K}})(\overline{T}_*,\boldsymbol{x})=(\nabla\rho^{\gamma-1},\nabla\log\rho,\nabla^2\log\rho)(\overline{T}_*,\boldsymbol{x})
\qquad\mbox{for {\it a.e.} $\boldsymbol{x}\in \mathbb{R}^n$}.
\end{equation}
 For simplicity, we prove that $\bar{\boldsymbol{f}}(\overline{T}_*,\boldsymbol{x})=\nabla\rho^{\gamma-1}(\overline{T}_*,\boldsymbol{x})$ for {\it a.e.} $\boldsymbol{x}\in \mathbb{R}^n$, since the rest of \eqref{claim103} can be derived analogously. 

First, due to $\eqref{global-unifrom}_1$, there exists a subsequence such that 
\begin{equation}\label{phibar}
\rho^{\gamma-1}(t_k,\boldsymbol{x})\to  \bar\phi(\overline{T}_*,\boldsymbol{x}) \qquad \text{weakly* in $L^\infty(\mathbb{R}^n)$}
\end{equation}
for some limit  $\phi(\overline{T}_*,\boldsymbol{x}) \in L^\infty(\mathbb{R}^n)$.
On the other hand, it follows from $\eqref{global-unifrom}_1$ and Lemma \ref{lemma-inf-rho} that,  
for any fixed $R>0$, 
\begin{equation}\label{lll-222}
\sup_{t\in [0,T]}\|\rho^{\gamma-1}(t)\|_{H^2(B_R)}\leq C(R,T),\quad \quad \inf_{(t,\boldsymbol{x})\in [0,T]\times B_R} \rho(t,\boldsymbol{x})\geq C^{-1}(R,T)
\end{equation}
for some constant $C(R,T)>0$ depending only on $(C_0,R,T)$,
where $B_R=\{\boldsymbol{x}\in \mathbb{R}^n:\,|\boldsymbol{x}|< R\}$.

Then, since $H^2(B_R)$ is compactly embedded in $C(\overline{B_R})$, by extracting a subsequence, 
there exists a limit $0<\bar{\bar\phi}(\overline{T}_*,\boldsymbol{x})\in C(\mathbb{R}^n)$ such that, 
for each $R>0$,
\begin{equation}\label{BR}
\rho^{\gamma-1}(t_k,\boldsymbol{x})\to \bar{\bar\phi}(\overline{T}_*,\boldsymbol{x}) \qquad
\text{uniformly on $B_R$} \,\,\, \text{as $k\to\infty$}.
\end{equation}
Clearly,  $0<\bar\phi(\overline{T}_*,\boldsymbol{x})=\bar{\bar\phi}(\overline{T}_*,\boldsymbol{x})\in L^\infty(\mathbb{R}^n)\cap C(\mathbb{R}^n)$ due to the uniqueness of limits in \eqref{phibar} and \eqref{BR}, which, together with $\eqref{global-unifrom}_1$ and  \eqref{lll-222}--\eqref{BR},   yields  that 
\begin{equation} \label{denstylimit}
\rho(t_k,\boldsymbol{x})\to \bar\phi^\frac{1}{\gamma-1}(\overline{T}_*,\boldsymbol{x}) \qquad\text{for any $\boldsymbol{x}\in\mathbb{R}^n$} \,\,\, \text{as $k\to\infty$}. 
\end{equation}
Then it follows from  the uniqueness of the limits in $\eqref{f2}_1$ and \eqref{denstylimit} that  $\bar\phi^\frac{1}{\gamma-1}(\overline{T}_*,\boldsymbol{x})=\rho (\overline{T}_*,\boldsymbol{x})$, {\it i.e.}, $\bar\phi(\overline{T}_*,\boldsymbol{x})=\rho^{\gamma-1} (\overline{T}_*,\boldsymbol{x})$. 

Next, it follows from $\eqref{f2}_4$, \eqref{phibar},  $\bar\phi(\overline{T}_*,\boldsymbol{x})=\rho^{\gamma-1} (\overline{T}_*,\boldsymbol{x})$, and the Lebesgue dominated convergence theorem that, for any $\zeta(\boldsymbol{x})\in C^\infty_{\rm c}(\mathbb{R}^n)$ and $i=1,\cdots\!,n$, 
\begin{align*}
&\int_{\mathbb{R}^n}\rho^{\gamma-1}(\overline{T}_*,\boldsymbol{x})\zeta_{x_i}(\boldsymbol{x})\,\mathrm{d}\boldsymbol{x}=\lim_{k\to\infty}\int_{\mathbb{R}^n} \rho^{\gamma-1}(t_k,\boldsymbol{x})\zeta_{x_i}(\boldsymbol{x})\,\mathrm{d}\boldsymbol{x}\\
&=-\lim_{k\to\infty}\int_{\mathbb{R}^n}(\rho^{\gamma-1})_{x_i}(t_k,\boldsymbol{x})\zeta(\boldsymbol{x})\,\mathrm{d}\boldsymbol{x}=-\int_{\mathbb{R}^n}\bar{f}_i(\overline{T}_*,\boldsymbol{x})\zeta(\boldsymbol{x})\,\mathrm{d}\boldsymbol{x}.
\end{align*}
This implies that $\rho^{\gamma-1}(\overline{T}_*,\boldsymbol{x})$ admits the weak derivatives $(\rho^{\gamma-1})_{x_i}(\overline{T}_*,\boldsymbol{x})=\bar{f}_i(\overline{T}_*,\boldsymbol{x})\in L^2(\mathbb{R}^n)$ for $i=1,\cdots\!, n$, so that $\nabla\rho^{\gamma-1}(\overline{T}_*,\boldsymbol{x})=\bar{\boldsymbol{f}}(\overline{T}_*,\boldsymbol{x})$ for {\it a.e.} $\boldsymbol{x}\in \mathbb{R}^n$. 

We now continue to prove $\overline{T}_*=\infty$. We aim to show that functions 
$(\rho, \boldsymbol{u})(\overline{T}_*,\boldsymbol{x})$ satisfy all the initial assumptions 
given in Theorem \ref{zth1}, which consist of showing that 
$(\rho, \boldsymbol{u})(\overline{T}_*,\boldsymbol{x})$ are spherically symmetric 
and satisfy \eqref{id1}--\eqref{shangjie3}. 

\smallskip
\textbf{1.1. $(\rho,\boldsymbol{u})(\overline{T}_*,\boldsymbol{x})$ are spherically symmetric.} 
It suffices to show that $\boldsymbol{u}(\overline{T}_*,\boldsymbol{x})$ is spherically symmetric, 
since the proof for $\rho(\overline{T}_*,\boldsymbol{x})$ can be derived analogously. 
To achieve this, it suffices to show that 
$\boldsymbol{u}(\overline{T}_*,\boldsymbol{x})=(\mathcal{O}^\top\boldsymbol{u})(\overline{T}_*,\mathcal{O}\boldsymbol{x})$ for any $\mathcal{O}\in \mathrm{SO}(n)$. 
Indeed, since $\boldsymbol{u}(t_k,\boldsymbol{x})$ is spherically symmetric for each $t_k$ and, by \eqref{f2}, $\boldsymbol{u}(t_k,\boldsymbol{x})$ converges to $\boldsymbol{u}(\overline{T}_*,\boldsymbol{x})$ weakly in $L^2(\mathbb{R}^n)$ as $k\to \infty$, it follows from the coordinate transformation that, 
for any vector function $\boldsymbol{\zeta}(\boldsymbol{x})\in L^2(\mathbb{R}^n)$,
\begin{align*}
&\int_{\mathbb{R}^n} \boldsymbol{u}(\overline{T}_*,\boldsymbol{x})\cdot \boldsymbol{\zeta}(\boldsymbol{x})\,\mathrm{d}\boldsymbol{x}=\lim_{k\to\infty}\int_{\mathbb{R}^n}\boldsymbol{u}(t_k,\boldsymbol{x})\cdot \boldsymbol{\zeta}(\boldsymbol{x})\,\mathrm{d}\boldsymbol{x}\\
&=\lim_{k\to\infty}\int_{\mathbb{R}^n}(\mathcal{O}^\top\boldsymbol{u})(t_k,\mathcal{O}\boldsymbol{x})\cdot \boldsymbol{\zeta}(\boldsymbol{x})\,\mathrm{d}\boldsymbol{x}=\lim_{k\to\infty}\int_{\mathbb{R}^n}\boldsymbol{u}(t_k,\boldsymbol{x})\cdot (\mathcal{O}\boldsymbol{\zeta})(\mathcal{O}^\top\boldsymbol{x})\,\mathrm{d}\boldsymbol{x}\\
&=\int_{\mathbb{R}^n}\boldsymbol{u}(\overline{T}_*,\boldsymbol{x})\cdot (\mathcal{O}\boldsymbol{\zeta})(\mathcal{O}^\top\boldsymbol{x})\,\mathrm{d}\boldsymbol{x}=\int_{\mathbb{R}^n}(\mathcal{O}^\top\boldsymbol{u})(\overline{T}_*,\mathcal{O}\boldsymbol{x})\cdot \boldsymbol{\zeta}(\boldsymbol{x})\,\mathrm{d}\boldsymbol{x},
\end{align*}
which implies that 
$\boldsymbol{u}(\overline{T}_*,\boldsymbol{x})=(\mathcal{O}^\top\boldsymbol{u})(\overline{T}_*,\mathcal{O}\boldsymbol{x})$ for any $\mathcal{O}\in \mathrm{SO}(n)$.

\smallskip
\textbf{1.2. $(\rho,\boldsymbol{u})(\overline{T}_*,\boldsymbol{x})$ satisfies \eqref{id1}--\eqref{shangjie3}.} Clearly, by \eqref{f2}, it remains to show that $\rho(\overline{T}_*,\boldsymbol{x})\in L^1(\mathbb{R}^n)$. Since $\bar\phi^\frac{1}{\gamma-1}(\overline{T}_*,\boldsymbol{x})= \rho(\overline{T}_*,\boldsymbol{x})$ in \eqref{denstylimit} and $\rho(t_k,\boldsymbol{x})\geq 0$, it follows from \eqref{global-unifrom}, \eqref{denstylimit}, and Lemma \ref{Fatou} that
\begin{equation*}
\int_{\mathbb{R}^n} \rho(\overline{T}_*,\boldsymbol{x})\,\mathrm{d}\boldsymbol{x}=\int_{\mathbb{R}^n} \liminf_{k\to\infty}\rho(t_{k},\boldsymbol{x})\,\mathrm{d}\boldsymbol{x}\le \liminf_{k\to\infty}\int_{\mathbb{R}^n} \rho(t_{k},\boldsymbol{x})\,\mathrm{d}\boldsymbol{x}\leq C(T),
\end{equation*}
which implies that $\rho(\overline{T}_*,\boldsymbol{x})\in L^1(\mathbb{R}^n)$. 

\smallskip
\textbf{1.3. $\overline{T}_*=\infty$.}
To sum up, we have shown that $(\rho,\boldsymbol{u})(\overline{T}_*,\boldsymbol{x})$ satisfies all the initial assumptions on the initial data of Theorem \ref{zth1}. 
As a consequence, according to Theorem \ref{zth1},  there exists a constant $T_{**}>0$ such that the Cauchy problem \eqref{eq:1.1benwen}--\eqref{e1.3} admits a unique $2$-order regular solution 
$(\tilde\rho, \tilde{\boldsymbol{u}})(t,\boldsymbol{x})$ in 
$[\overline{T}_*,\overline{T}_*+T_{**}]\times \mathbb{R}^n$. 
Thus, by setting 
$(\rho,\boldsymbol{u}) (t,\boldsymbol{x})=(\tilde\rho, \tilde{\boldsymbol{u}})(t,\boldsymbol{x})$ in $[\overline{T}_*,\overline{T}_*+T_{**}]\times \mathbb{R}^n$, it can be shown that  $(\rho, \boldsymbol{u})(t,\boldsymbol{x})$ is actually the $2$-order regular solution in  $[0,\overline{T}_*+T_{**}]\times \mathbb{R}^n$ due to their time-continuities, which  contradicts to the maximality of $\overline{T}_*<\infty$. Therefore $\overline{T}_*=\infty$.

\medskip
\textbf{2. Proof of Theorem \ref{th1} (i).} 

\smallskip
\textbf{2.1. Time-spatial continuity of $(\rho,\nabla\rho)$.}
First, it follows  from \eqref{global-unifrom} and the relation:  
$\nabla \rho=\rho \nabla \log \rho$ that, for any finite  $T>0$,
\begin{equation}\label{706}
\begin{aligned}
\|\rho\|_{L^2(\mathbb{R}^2)}&\leq \|\rho\|_{L^1(\mathbb{R}^2)}^\frac{1}{2}\|\rho\|_{L^\infty(\mathbb{R}^2)}^\frac{1}{2}\leq C(T), \\
\|\rho\|_{D^1(\mathbb{R}^2)}&\leq \|\rho\|_{L^2(\mathbb{R}^2)}\|\nabla\log \rho\|_{L^\infty(\mathbb{R}^2)}\leq C(T),\\
\|\rho\|_{D^2(\mathbb{R}^2)}&\leq \|\rho\|_{L^2(\mathbb{R}^2)}\|\nabla\log \rho\|_{L^\infty(\mathbb{R}^2)}^2 +\|\rho\|_{L^\infty(\mathbb{R}^2)}\|\nabla\log \rho\|_{D^1(\mathbb{R}^2)}\leq C(T),
\end{aligned}
\end{equation}
which implies that 
$\rho\in L^\infty([0,T];H^2(\mathbb{R}^2))$. 
Since $\rho$ is the $2$-order regular solution, this implies that $\rho\in C([0,T];L^1(\mathbb{R}^2))$ and $\nabla\log\rho\in C([0,T];D^1(\mathbb{R}^2))$. Then, for any $t, t_0\in [0,T]$, 
repeating the calculations in \eqref{706} with $\rho(t)$ replaced by $(\rho(t)-\rho(t_0))$, 
along with \eqref{global-unifrom}, gives
\begin{equation}\label{706-con}
\begin{aligned}
&\lim_{t\to t_0}\|\rho(t)-\rho(t_0)\|_{H^2(\mathbb{R}^2)}\\
&\leq C(T)\lim_{t\to t_0}\|\rho(t)-\rho(t_0)\|_{L^1(\mathbb{R}^2)}^\frac{1}{2}
+C(T)\lim_{t\to t_0}\big\|\nabla\log \rho(t)-\nabla\log\rho(t_0)\big\|_{D^1(\mathbb{R}^2)}=0,\qquad
\end{aligned}
\end{equation}
so that $\rho\in C([0,T];H^2(\mathbb{R}^2))$. Then it follows from Lemma \ref{Hk-Ck-scalar} in Appendix \ref{improve-sobolev} that $\rho\in C([0,T];C^1(\overline{\mathbb{R}^2}))$.

\smallskip
\textbf{2.2. Time-spatial continuity of $(\boldsymbol{u},\nabla \boldsymbol{u},\nabla^2\boldsymbol{u},\boldsymbol{u}_t)$.} 
First, since $\boldsymbol{u}$ is the $2$-order regular solution, $\boldsymbol{u}\in C([0,T];H^2(\mathbb{R}^2))$. Hence, Lemma \ref{ale1} yields $\boldsymbol{u}\in C([0,T];C(\overline{\mathbb{R}^2}))$. 

Next, due to \eqref{global-unifrom}, then $t\boldsymbol{u}_t\in L^2([0,T];H^2(\mathbb{R}^3))$ and $(t\boldsymbol{u}_t)_{t}\in L^2([0,T];L^2(\mathbb{R}^3))$. 
It follows from Lemmas \ref{triple} and \ref{Hk-Ck-vector} that
\begin{equation}\label{emb}
t\boldsymbol{u}_t\in C([0,T];H^1(\mathbb{R}^2))\implies \boldsymbol{u}_t\in C((0,T];C(\overline{\mathbb{R}^2})). 
\end{equation}
 
Finally, to obtain $\boldsymbol{u}\in C((0,T];C^2(\overline{\mathbb{R}^2}))$, 
we show that $t\boldsymbol{u}\in C([0,T];H^3(\mathbb{R}^2))$. Rewrite \eqref{qiyi} as the following elliptic system:
\begin{equation}\label{guoc-1}
L\boldsymbol{u}=\underline{-\boldsymbol{u}_t-\boldsymbol{u}\cdot\nabla\boldsymbol{u}-\frac{A\gamma}{\gamma-1}\nabla\rho^{\gamma-1}}_{:=\boldsymbol{F}_1}+\underline{\nabla\log\rho\cdot Q(\boldsymbol{u})}_{:=\boldsymbol{F}_2}.
\end{equation}
Then, by the classical regularity theory for elliptic equations in Lemma \ref{df3}, we have
\begin{equation}\label{f1-f2-t}
\|t\boldsymbol{u}(t)-t_0\boldsymbol{u}(t_0)\|_{D^3(\mathbb{R}^2)}\leq C_0\sum_{i=1}^2\|t\boldsymbol{F}_i(t)-t_0\boldsymbol{F}_i(t_0)\|_{D^1(\mathbb{R}^2)}.
\end{equation}
For $\boldsymbol{F}_1$,  by \eqref{emb} and $\nabla\rho^{\gamma-1}\in C([0,T];H^1(\mathbb{R}^2))$, then $(t\boldsymbol{u}_t,t\nabla\rho^{\gamma-1})\in C([0,T];D^1(\mathbb{R}^2))$. 
Moreover, it follows from \eqref{global-unifrom} and Lemma \ref{ale1} that 
\begin{equation*}
t\boldsymbol{u}\cdot\nabla\boldsymbol{u}\in L^2([0,T]; H^2(\mathbb{R}^2)),\qquad (t\boldsymbol{u}\cdot\nabla\boldsymbol{u})_t\in L^2([0,T]; L^2(\mathbb{R}^2)).
\end{equation*}
This, along with Lemma \ref{triple}, implies that $t\boldsymbol{u}\cdot\nabla\boldsymbol{u}\in C([0,T]; H^1(\mathbb{R}^2))$, so that
\begin{equation}\label{f1-t}
t\boldsymbol{F}_1\in C([0,T];D^1(\mathbb{R}^2)).
\end{equation}
For $\boldsymbol{F}_2$, it follows from \eqref{global-unifrom}, Lemmas \ref{ale1}--\ref{GN-ineq} 
and \ref{Hk-Ck-vector}, and the Young inequality that, for all $0\leq t,t_0\leq T$ and $\omega\in (0,1)$,
\begin{equation}\label{f2-t}
\begin{aligned}
&\|t\boldsymbol{F}_2(t)-t_0\boldsymbol{F}_2(t_0)\|_{D^1(\mathbb{R}^2)}\\
&\leq C_0\big\|\nabla\log\rho(t)-\nabla\log\rho(t_0)\big\|_{D^1(\mathbb{R}^2)}\|t\boldsymbol{u}(t)\|_{D^{1,\infty}(\mathbb{R}^2)}\\[-4pt]
&\quad +C_0\big\|\nabla\log\rho(t_0)\big\|_{D^1(\mathbb{R}^2)}\|t\boldsymbol{u}(t)-t_0\boldsymbol{u}(t_0)\|_{D^{1}(\mathbb{R}^2)}^\frac{1}{2}\|t\boldsymbol{u}(t)-t_0\boldsymbol{u}(t_0)\|_{D^{3}(\mathbb{R}^2)}^\frac{1}{2}\\
&\quad +C_0\big\|\nabla\log\rho(t)-\nabla\log\rho(t_0)\big\|_{L^\infty(\mathbb{R}^2)}\|t\boldsymbol{u}(t)\|_{D^2(\mathbb{R}^2)}\\
&\quad +C_0\big\|\nabla\log\rho(t_0)\big\|_{L^\infty(\mathbb{R}^2)}\|t\boldsymbol{u}(t)-t_0\boldsymbol{u}(t_0)\|_{D^2(\mathbb{R}^2)}\\
&\leq C(T)\big\|\nabla\log\rho(t)-\nabla\log\rho(t_0)\big\|_{D^1(\mathbb{R}^2)}+C(\omega,T)\|\boldsymbol{u}(t)-\boldsymbol{u}(t_0)\|_{H^2(\mathbb{R}^2)}\\
&\quad +C(\omega,T)|t-t_0|+\omega\|t\boldsymbol{u}(t)-t_0\boldsymbol{u}(t_0)\|_{D^3(\mathbb{R}^2)}.
\end{aligned}
\end{equation}
Consequently, collecting \eqref{f1-f2-t}--\eqref{f2-t} and setting $\omega$ small enough, we see 
from the time-continuity of $(\boldsymbol{u},\nabla\log\rho)$ that, as $t\to t_0$,
\begin{equation}\label{guo3}
\begin{aligned}
&\|t\boldsymbol{u}(t)-t_0\boldsymbol{u}(t_0)\|_{D^3(\mathbb{R}^2)}\\
&\leq C_0\|t\boldsymbol{F}_1(t)-t_0\boldsymbol{F}_1(t_0)\|_{D^1(\mathbb{R}^2)}
+C(T)\big\|\nabla\log\rho(t)-\nabla\log\rho(t_0)\big\|_{D^1(\mathbb{R}^2)}\\
&\quad+C(T)\|\boldsymbol{u}(t)-\boldsymbol{u}(t_0)\|_{H^2(\mathbb{R}^2)}+C(T)|t-t_0|\to 0.
\end{aligned}
\end{equation}
This, together with $t\boldsymbol{u}\in L^\infty([0,T];D^3(\mathbb{R}^2))$ and $\boldsymbol{u}\in C([0,T];H^2(\mathbb{R}^2))$, implies that $t\boldsymbol{u}\in C([0,T];H^3(\mathbb{R}^2))$. Therefore, it follows from Lemma \ref{Hk-Ck-vector-3} that $\boldsymbol{u}\in C((0,T];C^2(\overline{\mathbb{R}^2}))$.

\smallskip
\textbf{2.3. Time-spatial continuity of $\rho_t$.} It follows from the relation:
\begin{equation}\label{776}
\rho_t=-\boldsymbol{u}\cdot\nabla \rho -\rho \diver\boldsymbol{u},
\end{equation}
and the conclusions obtained in Steps 2.1--2.2 that $\rho_t\in C((0,T];C(\overline{\mathbb{R}^2}))$, which implies Theorem \ref{th1} (i).

\medskip
\textbf{3. Proof of Theorem \ref{th1} (ii).} Let $\mathbb{R}^3_*$ be defined in Theorem \ref{th1} (ii). First, since $(\rho,\boldsymbol{u})$ is the $2$-order regular solution, then $(\rho,\boldsymbol{u})\in C([0,T];H^2(\mathbb{R}^3))$, which, along with Lemma \ref{ale1}, implies that $(\rho,\boldsymbol{u})\in C([0,T];C(\overline{\mathbb{R}^3}))$.

Next, we show that $\nabla \boldsymbol{u}\in C((0,T];C(\overline{\mathbb{R}^3}))$. 
To obtain this, we claim that
\begin{equation}\label{707}
t\boldsymbol{u}_t\in C([0,T];H^1(\mathbb{R}^3)),\qquad t\boldsymbol{u}\in C([0,T];H^3(\mathbb{R}^3)).   
\end{equation}
$\eqref{707}_1$ follows from the same argument as in Step 2.2 above, with $\mathbb{R}^2$ replaced by $\mathbb{R}^3$. 
For $\eqref{707}_2$, based on the discussions in \eqref{guoc-1}--\eqref{guo3} and $\eqref{707}_1$,
we see that 
\begin{equation}\label{f1-t-3}
t\boldsymbol{F}_1\in C([0,T];D^1(\mathbb{R}^3)),
\end{equation}
so it suffices to derive the estimates on $\boldsymbol{F}_2$. 
To this end, it follows from \eqref{global-unifrom}, Lemmas \ref{ale1}--\ref{GN-ineq}, 
and \ref{lemma-L6} that, for all $0\leq t,t_0\leq T$ and $\omega\in (0,1)$,
\begin{equation}\label{f2-t-3}
\begin{aligned}
&\|t\boldsymbol{F}_2(t)-t_0\boldsymbol{F}_2(t_0)\|_{D^1(\mathbb{R}^3)}\\
&\leq C_0\big\|\nabla\log\rho(t)-\nabla\log\rho(t_0)\big\|_{D^1(\mathbb{R}^3)}\|t\boldsymbol{u}(t)\|_{D^{1,\infty}(\mathbb{R}^3)}\\
&\quad +C_0\big\|\nabla\log\rho(t_0)\big\|_{D^1(\mathbb{R}^3)}\|t\boldsymbol{u}(t)-t_0\boldsymbol{u}(t_0)\|_{D^{1}(\mathbb{R}^3)}^\frac{2}{3}\|t\boldsymbol{u}(t)-t_0\boldsymbol{u}(t_0)\|_{D^{3}(\mathbb{R}^3)}^\frac{1}{3}\\
&\quad +C_0\big\|\nabla\log\rho(t)-\nabla\log\rho(t_0)\big\|_{L^6(\mathbb{R}^3)}\|t\boldsymbol{u}(t)\|_{D^{2,3}(\mathbb{R}^3)}\\
&\quad +C_0\big\|\nabla\log\rho(t_0)\big\|_{L^6(\mathbb{R}^3)}\|t\boldsymbol{u}(t)-t_0\boldsymbol{u}(t_0)\|_{D^{1}(\mathbb{R}^3)}^\frac{2}{3}\|t\boldsymbol{u}(t)-t_0\boldsymbol{u}(t_0)\|_{D^{3}(\mathbb{R}^3)}^\frac{1}{3}\\
&\leq C(T)\big\|\nabla\log\rho(t)-\nabla\log\rho(t_0)\big\|_{D^1(\mathbb{R}^3)}+C(\omega,T)\|\boldsymbol{u}(t)-\boldsymbol{u}(t_0)\|_{D^1(\mathbb{R}^3)}\\
&\quad +C(\omega,T)|t-t_0|
+\omega\|t\boldsymbol{u}(t)-t_0\boldsymbol{u}(t_0)\|_{D^3(\mathbb{R}^3)}.
\end{aligned}
\end{equation}
Consequently, collecting \eqref{f1-f2-t} (with $\mathbb{R}^2$ replaced by $\mathbb{R}^3$) and \eqref{f1-t-3}--\eqref{f2-t-3} and setting $\omega$ small enough, it follows from the time-continuity of $(\boldsymbol{u},\nabla\log\rho)$ that, as $t\to t_0$,
\begin{align*}
\|t\boldsymbol{u}(t)-t_0\boldsymbol{u}(t_0)\|_{D^3(\mathbb{R}^3)}
&\leq C(T)\|t\boldsymbol{F}_2(t)-t_0\boldsymbol{F}_2(t_0)\|_{D^1(\mathbb{R}^3)}\notag\\
&\quad+C(T)\big\|\nabla\log\rho(t)-\nabla\log\rho(t_0)\big\|_{D^1(\mathbb{R}^3)}\\
&\quad+C(T)\|\boldsymbol{u}(t)-\boldsymbol{u}(t_0)\|_{D^1(\mathbb{R}^3)}+C(T)|t-t_0|\to 0.\notag
\end{align*}
This, together with $t\boldsymbol{u}\in L^\infty([0,T];D^3(\mathbb{R}^3))$ and $\boldsymbol{u}\in C([0,T];H^2(\mathbb{R}^3))$, implies that $t\boldsymbol{u}\in C([0,T];H^3(\mathbb{R}^3))$.  Therefore, it follows from Lemma \ref{ale1} that $\nabla\boldsymbol{u}\in C((0,T];C(\overline{\mathbb{R}^3}))$.  

Now, it remains to show that
\begin{equation}
(\nabla\rho,\rho_t,\nabla\boldsymbol{u})\in C([0,T]\times \mathbb{R}^3_*),\qquad (\nabla^2\boldsymbol{u},\boldsymbol{u}_t)\in C((0,T]\times \mathbb{R}^3_*).
\end{equation}
To this end, it follows from \eqref{707}, $(\rho,\boldsymbol{u})\in C([0,T];H^2(\mathbb{R}^3))$, 
and Lemma \ref{lemma-initial} that
\begin{align*}
&r\Big(\rho_r,\rho_{rr},\frac{\rho_{r}}{r},u_r,\frac{u}{r},u_{rr},\big(\frac{u}{r}\big)_r\Big)\in C([0,T];L^2(I)),\\
&r\Big(u_{rrr},\frac{u_{rr}}{r},\big(\frac{u}{r}\big)_{rr},\frac{1}{r}\big(\frac{u}{r}\big)_r,u_{t},u_{tr},\frac{u_t}{r}\Big) \in C((0,T];L^2(I)),
\end{align*}
which, along with Lemma \ref{ale1}, leads to
\begin{gather*}
\big(\rho_r,u_r,\frac{u}{r}\big)\in C([0,T]\times[\sigma,\infty)), \qquad \big(u_{rr},(\frac{u}{r})_r,u_{t}\big)\in C((0,T]\times[\sigma,\infty)), 
\end{gather*}
for any $\sigma>0$. Hence, $(\rho_r,u_r,\frac{u}{r})\in C([0,T]\times(0,\infty))$ and $(u_{rr},(\frac{u}{r})_r,u_{t})\in C((0,T]\times(0,\infty))$, that is, $(\nabla\rho,\nabla\boldsymbol{u}) \in C([0,T]\times \mathbb{R}^3_*)$ and $(\nabla^2\boldsymbol{u},\boldsymbol{u}_t) \in C((0,T]\times \mathbb{R}^3_*)$ due to Lemma \ref{lemma-initial}. Finally, $\rho_t \in C([0,T]\times \mathbb{R}^3_*)$ can be derived directly via \eqref{776} and the fact that  $(\rho,\nabla\rho,\boldsymbol{u},\nabla\boldsymbol{u})\in C([0,T]\times\mathbb{R}^3_*)$. The proof of Theorem \ref{th1} (ii) is completed.   

\medskip
\textbf{4. Proof of Theorem \ref{th1}(iii)--(iv).}  We first show that $\mathcal{P}(t)\equiv \boldsymbol{0}$ for $t\in [0,T]$. For simplicity, we only give the proof for the 3-D case, since the 2-D case follows analogously.  Since $(\sqrt{\rho},\sqrt{\rho}\boldsymbol{u})\in C([0,T];L^2(\mathbb{R}^n))$, then 
$\rho\boldsymbol{u}\in C([0,T];L^1(\mathbb{R}^n))$, which implies that $\mathcal{P}(t)$ is bounded 
and continuous on $[0,T]$ and $r^2\rho u\in C([0,T];L^1(I))$. 
Next, according to the spherical coordinate transformations:
\begin{equation*}
\begin{aligned}
&\boldsymbol{x}=(x_1,x_2,x_3)^\top=(r\cos\theta_1\sin\theta_2,r\sin\theta_1\sin\theta_2,r\cos\theta_2)^\top,\\
&r\in I,\qquad \theta_1\in [0,2\pi],\qquad \theta_2\in [0,\pi],
\end{aligned}
\end{equation*}
we obtain from the spherical symmetry of $(\rho,\boldsymbol{u})$ that
\begin{align*}
\mathcal{P}(t)&=\int_{\mathbb{R}^3} (\rho\boldsymbol{u})(t,\boldsymbol{x}) \,\mathrm{d}\boldsymbol{x}=\int_0^\pi\int_0^{2\pi}\int_0^\infty \frac{(x_1,x_2,x_3)^\top}{r}(\rho u)(t,r) r^2\sin\theta_2\,\mathrm{d}r\mathrm{d}\theta_1\mathrm{d}\theta_2\\
&=\mathfrak{p}(t)\int_0^\pi\int_0^{2\pi} (\cos\theta_1\sin^2\theta_2,\sin\theta_1\sin^2\theta_2,\cos\theta_2\sin\theta_2)^\top \, \mathrm{d}\theta_1\mathrm{d}\theta_2\\
&=\mathfrak{p}(t)\int_0^\pi (0,0,\cos\theta_2\sin\theta_2)^\top \,  \mathrm{d}\theta_2=\boldsymbol{0},\qquad \text{with }\mathfrak{p}(t):=\int_0^\infty r^2(\rho u)(t,r)\,\mathrm{d}r,
\end{align*}
which yields that $\mathcal{P}(t)\equiv \boldsymbol{0}$ for $t\in [0,T]$.

The conservation of total mass can be simply derived by integrating the mass equation $\eqref{eq:1.1benwen}_1$ over $[0,T]\times \mathbb{R}^n$ and using the fact that $\rho\boldsymbol{u} \in C([0,T];W^{1,1}(\mathbb{R}^n))$.  Finally, Theorem \ref{th1} (iv) is a directly consequence of Lemmas \ref{far-p-infty}, \ref{important2}, and \ref{lemma-inf-rho}. 

The proof of Theorem \ref{th1} is completed.

\subsection{Proof of Theorem \ref{th1-high}}\label{soloo2}
Now we are ready to provide the proof of Theorem \ref{th1-high}. We divide the proof into three steps.

\smallskip
\textbf{1. Global well-posedness of $3$-order regular solutions.} 
First, according to Theorem \ref{zth2}, there exists a $3$-order 
regular solution $(\rho, \boldsymbol{u})(t,\boldsymbol{x})$ of the Cauchy problem \eqref{eq:1.1benwen}--\eqref{e1.3} in $[0,T_*]\times \mathbb{R}^n$ for some $T_*>0$, which takes form \eqref{duichenxingshi}.

Second, let $\overline{T}_*>0$ be the life span of $(\rho, \boldsymbol{u})(t,\boldsymbol{x})$, 
and let $T$ be any fixed time satisfying $0<T<\overline{T}_*$.
Then collecting the uniform {\it a priori} bounds obtained in Lemmas \ref{far-p-infty}, \ref{important2}, \ref{l4.5}--\ref{l4.10-ell}, and \ref{H3-1}--\ref{Lemma6.14}, together with \eqref{tr} and Lemma \ref{lemma-initial}, yields that, for any  $t\in [0,T]$,
\begin{equation}\label{global-unifrom33}
\begin{aligned}
\|\rho(t)\|_{L^1\cap L^\infty}+\big\|(\nabla\rho^{\gamma-1},(\rho^{\gamma-1})_t)(t)\big\|_{H^2} \leq C(T),\\[4pt]
\big\|\nabla\log\rho(t)\big\|_{L^\infty}+\big\|\big(\nabla^2\log\rho,(\nabla\log\rho)_t\big)(t)\big\|_{H^1}\leq C(T),\\
\|\boldsymbol{u}(t)\|_{H^3}+\|\boldsymbol{u}_t(t)\|_{H^1}+\int_0^t \big\|(\nabla^4\boldsymbol{u},\nabla^2\boldsymbol{u}_t,\boldsymbol{u}_{tt})\big\|_{L^2}^2 \,\mathrm{d}s\leq C(T),\\[-2pt]
\sqrt{t}\big\|(\nabla^4\boldsymbol{u},\nabla^2\boldsymbol{u}_t,\boldsymbol{u}_{tt})(t)\big\|_{L^2}+\int_0^t s\big\|(\nabla^3\boldsymbol{u}_t,\nabla\boldsymbol{u}_{tt})\big\|_{L^2}^2\,\mathrm{d}s\leq C(T).
\end{aligned}
\end{equation}
The remaining  proof here  is basically the same as that in Step 1 of \S\ref{solo}. We omit it here. 

\smallskip
\textbf{2. Proof of Theorem \ref{th1-high} (i).}
According to  Step 2 in the proof for Theorem \ref{th1} (see \S\ref{solo}), it suffices to 
show that 
\begin{equation}\label{theab}
(\rho_t,\nabla\boldsymbol{u},\nabla^2\boldsymbol{u},\boldsymbol{u}_t)\in C([0,T];C(\overline{\mathbb{R}^2})).
\end{equation}
Since $\boldsymbol{u}$ is the $3$-order regular solution, we have
\begin{equation*}
\boldsymbol{u}\in C([0,T];H^3(\mathbb{R}^2)),\qquad \boldsymbol{u}_t\in C([0,T];H^1(\mathbb{R}^2)),
\end{equation*}
which, along with  Lemma \ref{Hk-Ck-vector} and \eqref{776}, yields that 
$(\nabla\boldsymbol{u},\nabla^2\boldsymbol{u},\boldsymbol{u}_t,\rho_t)\in C([0,T];C(\overline{\mathbb{R}^2}))$.  

\smallskip
\textbf{3. Proof of Theorem \ref{th1-high} (ii).} According to  Step 3 in the proof for Theorem \ref{th1} (see \S\ref{solo}), it only remains to show that 
\begin{equation*}
(\nabla\rho,\rho_t,\nabla\boldsymbol{u})\in C([0,T];C(\overline{\mathbb{R}^3})),\qquad (\nabla^2\boldsymbol{u},\boldsymbol{u}_t)\in C((0,T];C(\overline{\mathbb{R}^3})).
\end{equation*}

\smallskip
\textbf{3.1. Time-spatial continuity of $(\nabla\rho,\rho_t,\nabla\boldsymbol{u})$.}
First, it follows from the fact that $\boldsymbol{u}\in C([0,T];H^3(\mathbb{R}^3))$ and  Lemma \ref{ale1} that $\boldsymbol{u}\in C([0,T];C^1(\overline{\mathbb{R}^3}))$. Next, following a calculation similar to \eqref{706}, with $\mathbb{R}^2$ replaced by $\mathbb{R}^3$, we obtain $\rho\in C([0,T];H^2(\mathbb{R}^3))$ and 
\begin{equation}
\begin{aligned}
\|\rho(t)\|_{D^3(\mathbb{R}^3)} &\leq \|\rho(t)\|_{L^1(\mathbb{R}^3)}^\frac{1}{2}\|\rho(t)\|_{L^\infty(\mathbb{R}^3)}^\frac{1}{2}\|\nabla\log \rho(t)\|_{L^\infty(\mathbb{R}^3)}^3\notag\\
&\quad +\|\rho(t)\|_{L^\infty(\mathbb{R}^3)}\|\nabla\log \rho(t)\|_{L^\infty(\mathbb{R}^3)}\|\nabla\log \rho(t)\|_{D^1(\mathbb{R}^3)}\\
&\quad +\|\rho(t)\|_{L^\infty(\mathbb{R}^3)}\|\nabla\log \rho(t)\|_{D^2(\mathbb{R}^3)}\leq C(T),
\end{aligned}
\end{equation}
which implies that $\rho\in L^\infty([0,T];H^3(\mathbb{R}^3))$. 
Since $\rho$ is the $3$-order regular solution, then $\rho\in C([0,T];L^1(\mathbb{R}^3))$ and $\nabla^2\log\rho\in C([0,T];H^1(\mathbb{R}^3))$. 
For any $t_0\in [0,T]$, repeating the above calculation with $\rho(t)$ replaced by $\rho(t)-\rho(t_0)$, 
together with \eqref{global-unifrom33}, gives
\begin{align*}
\lim_{t\to t_0}\|\rho(t)-\rho(t_0)\|_{D^3(\mathbb{R}^3)} &\leq C(T)\lim_{t\to t_0}\|\rho(t)-\rho(t_0)\|_{L^1(\mathbb{R}^3)}^\frac{1}{2}\\
&\quad +C(T)\lim_{t\to t_0}\|\nabla^2\log \rho(t)-\nabla^2\log \rho(t_0)\|_{H^1(\mathbb{R}^3)}=0,
\end{align*}
which implies $\rho\in C([0,T];H^3(\mathbb{R}^3))$. 
Hence, by Lemma \ref{ale1}, $\rho\in C([0,T];C^1(\overline{\mathbb{R}^3}))$. 
Finally, $\rho_t\in C([0,T];C(\overline{\mathbb{R}^3}))$ due to \eqref{776} and the facts that $(\rho,\boldsymbol{u})\in C([0,T];C^1(\overline{\mathbb{R}^3}))$.

\smallskip
\textbf{3.2. Time-spatial continuity of $(\nabla^2\boldsymbol{u}, \boldsymbol{u}_t)$.}
Thanks to \eqref{global-unifrom33}, we have 
\begin{align*}
(t\boldsymbol{u}_t,t\nabla^2\boldsymbol{u})\in L^\infty([0,T];H^2(\mathbb{R}^3)), \qquad
((t\boldsymbol{u}_t)_t,(t\nabla^2\boldsymbol{u})_t)\in L^2([0,T];L^2(\mathbb{R}^3)),
\end{align*}
which, along with Lemma \ref{triple}, yields
\begin{equation}\label{theaobve}
(t\boldsymbol{u}_t,t\nabla^2\boldsymbol{u})\in C([0,T];W^{1,q}(\mathbb{R}^3)) \qquad \text{for all $q\in [2,6)$}.
\end{equation} 
Therefore, it follows from \eqref{theaobve}  and Lemma \ref{ale1} that $(\boldsymbol{u}_t,\nabla^2\boldsymbol{u})\in C((0,T]\times \mathbb{R}^3)$. 

The proof of Theorem \ref{th1-high} is completed.

\subsection{Proof of Corollary \ref{coro1.1}} Corollary \ref{coro1.1} is a direct consequence of Lemma \ref{initial3} in Appendix \ref{appA} by taking $(f,\nu)=(\rho_0,\gamma-1)$.

\section{Global Well-Posedness of   Regular  Solutions  with Strictly Positive Initial Density}\label{nonvacuumfarfield}

This section is devoted to establishing the global well-posedness of spherically symmetric classical solutions of the Cauchy problem of system  \eqref{eq:1.1benwen} with general smooth initial data and
strictly positive initial density, {\it i.e.}, $\inf_{\boldsymbol{x}\in \mathbb{R}^n} \rho_0(\boldsymbol{x})>0$. Certainly, in this case, under the spherical coordinates, the Cauchy problem \eqref{eq:1.1benwen}--\eqref{e1.3} in $[0,T]\times \mathbb{R}^n$ for some $T>0$ can be written as the following initial-boundary value problem in $[0,T]\times I$:
\begin{equation}\label{e1.5hh}
\begin{cases}
\displaystyle 
\rho_t+(\rho u)_r+\frac{m\rho u}{r}=0,\\[3pt]
\displaystyle
\rho u_t+\rho uu_r+A(\rho^\gamma)_r=2\alpha\big(\rho u_r+\frac{m\rho u}{r} \big)_r-\frac{2\alpha m\rho_r u}{r},\\[5pt]
\displaystyle
(\rho,u)|_{t=0}=(\rho_0,u_0) \quad\,\,\, \text{for $r\in I$},\\[4pt]
\displaystyle
u|_{r=0}=0 \qquad\qquad\qquad\text{for $t\in [0,T]$},\\[4pt]
\displaystyle
(\rho,u)\to \left(\bar\rho>0,0\right)  \qquad\text{as $r\to \infty\,$ \, for $t\in [0,T]$}.
\end{cases}
\end{equation}

\subsection{Some notations and conventions}
Denote by $\chi_E$ the characteristic function with respect of a set $E\subset \mathbb{R}^n$
({\it i.e.}, $\chi_E=1$ on $E$ and $\chi_E=0$ on $\mathbb{R}^n\backslash E$). 

For $\sigma\geq 0$ and $t\geq 0$, define
\begin{equation*}
\begin{aligned}
E_{\sigma}:=E_{\sigma}(t)=\{\boldsymbol{x}\in\mathbb{R}^n:\,|\rho(t,\boldsymbol{x})-\bar\rho|\leq\sigma\},\\
E^{\sigma}:=E^{\sigma}(t)=\{\boldsymbol{x}\in\mathbb{R}^n:\,|\rho(t,\boldsymbol{x})-\bar\rho|>\sigma\}.
\end{aligned}
\end{equation*}
Clearly, $\mathbb{R}^n=E_\sigma\cup E^\sigma$ for each $\sigma\geq 0$ and $t\geq 0$.  Moreover, since the solution $(\rho, \boldsymbol{u})$ we considered is spherically symmetric,  we will still  denote by $(E_\sigma,E^\sigma)$ the subsets of $I$ under the spherical coordinates when no confusion arises, namely,
\begin{equation*}
\begin{aligned}
E_{\sigma}=E_{\sigma}(t)=\{r\in I:\,|\rho(t,r)-\bar\rho|\leq\sigma\},\\
E^{\sigma}=E^{\sigma}(t)=\{r\in I:\,|\rho(t,r)-\bar\rho|>\sigma\}.
\end{aligned}
\end{equation*}
In this case, $I=E_\sigma\cup E^\sigma$ for each $\sigma\geq 0$ and $t\geq 0$. 

Besides, we adapt the notations $(\chi_\sigma^\flat,\chi_\sigma^\sharp)$ defined in \eqref{chi-sigma} and the notation $j_\gamma(z)$ in \cite[Chapter 5]{lions} as
\begin{equation*}
j_\gamma(z)=(z^\gamma-\bar\rho^\gamma)-\gamma\bar\rho^{\gamma-1}(z-\bar\rho)\qquad 
\mbox{for $z\geq 0$}.
\end{equation*}

In the rest of this section, $C_0\in [1,\infty)$ denotes a generic constant depending only on $(\rho_0,\boldsymbol{u}_0)$ and fixed constants $(n,\alpha,\gamma,A,\bar\rho)$; $C(\nu_1,\cdots\!,\nu_k)\in [1,\infty)$ denotes a generic constant depending only on $C_0$ and parameters $(\nu_1,\cdots\!,\nu_k)$, 
which may be different at each occurrence.

\subsection{Some equivalent norms on the initial data} 
In order to obtain the desired global uniform estimates on the regular solutions 
by employing analogous arguments 
shown in \S\ref{section-upper-density}--\S\ref{section-global3}, 
we first give some equivalent norms on the initial density $\rho_0$, 
which will be frequently used in the subsequent analysis.

\begin{lem}\label{equiv-initial}
Let the initial condition \eqref{in-start} or \eqref{in-start-high} hold. Then
\begin{enumerate}
\item[$\mathrm{(i)}$] $0\leq j_\gamma(\rho_0)\in L^1(\mathbb{R}^n)${\rm ;}
\item[$\mathrm{(ii)}$] $\nabla\sqrt{\rho_0}\in L^2(\mathbb{R}^n)${\rm ;}
\item[$\mathrm{(iii)}$] if \eqref{in-start} holds, then $\nabla\rho_0^{\gamma-1}\in H^1(\mathbb{R}^n)$ and $\nabla\log\rho_0\in L^\infty(\mathbb{R}^n)\cap H^1(\mathbb{R}^n)${\rm ;}
\item[$\mathrm{(iv)}$] if \eqref{in-start-high} holds, then $\nabla\rho_0^{\gamma-1}\in H^2(\mathbb{R}^n)$ and $\nabla\log\rho_0\in H^2(\mathbb{R}^n)$.
\end{enumerate}
\end{lem}
\begin{proof}We divide the proof into two steps.

\smallskip
\textbf{1. Proof of (i).} 
A direct calculation gives
\begin{equation*}
\begin{aligned}
&j_\gamma'(z)= \gamma(z^{\gamma-1}-\bar\rho^{\gamma-1}), \quad j_\gamma''(z)=\gamma(\gamma-1) z^{\gamma-2}> 0 \qquad\,\,\text{for $z>0$}. 
\end{aligned}
\end{equation*}
with 
$j_\gamma(\bar\rho)=j_\gamma'(\bar\rho)=0$.
Thus, $j_\gamma(z)$ is a convex function and $j_\gamma(z)\geq 0$ for $z\geq 0$. 
Moreover, by the Taylor expansion, we have
\begin{equation}\label{erjie-taylor}
j_\gamma(z)= \frac{1}{2!}j_\gamma''(z_{\vartheta})(z-\bar\rho)^2=\frac{\gamma(\gamma-1)}{2} z_{\vartheta}^{\gamma-2}(z-\bar\rho)^2,
\end{equation}
where $z_{\vartheta}=\vartheta z+(1-\vartheta)\bar\rho$ for some $\vartheta\in [0,1]$. Setting $\rho_{0*}=\inf_{\boldsymbol{x}\in\mathbb{R}^n}\rho_0(\boldsymbol{x})$ and taking $z=\rho_0$ imply that
\begin{align*}
&0\leq j_\gamma(\rho_0)\leq C_0\big(\min\{\rho_{0*},\bar\rho\}\big)^{\gamma-2}(\rho_0-\bar\rho)^2 \qquad\quad\,\text{if $\gamma\in(1,2)$},\\
&0\leq j_\gamma(\rho_0)\leq C_0\big(\max\{|\rho_0|_\infty,\bar\rho\}\big)^{\gamma-2}(\rho_0-\bar\rho)^2 \qquad\text{if $\gamma\in [2,\infty)$},
\end{align*}
which, along with $\rho_0-\bar\rho\in L^2(\mathbb{R}^n)$, yields (i).

\smallskip
\textbf{2. Proof of (ii)--(iv).}
(ii)--(iv) follows easily from the facts that $\rho_0\geq \rho_{0*}>0$, 
$\rho_0\in L^\infty(\mathbb{R}^n)$, 
and $\nabla\rho_0\in H^{l-1}(\mathbb{R}^n)$ ($l=2$ or $3$). For example, if \eqref{in-start-high} holds, it follows from Lemma \ref{ale1} and  
$\nabla \log\rho=\frac{\nabla \rho}{\rho}$  that 
\begin{align*}
\|\nabla\log\rho\|_{D^2}&\leq C_0\Big(\frac{1}{\rho_{0*}} +\frac{1}{\rho_{0*}^2}\|\nabla\rho_0\|_{L^\infty} +\frac{1}{\rho_{0*}^3}\|\nabla\rho_0\|_{L^\infty}^2\Big)\|\nabla\rho_0\|_{H^2}\leq C_0.
\end{align*}
The rest of (ii)--(iv) can be proved in the same manner, we omit the details here.
\end{proof}

\subsection{Local-in-time well-posedness with strictly positive initial density}

\subsubsection{Local-in-time well-posedness in M-D coordinates} 
When $\inf_{\boldsymbol{x}\in \mathbb{R}^n} \rho_0(\boldsymbol{x})>0$, 
we consider the following reformulated system in the M-D cases: 
\begin{equation}\label{re}
\begin{cases}
\rho_t+\diver(\rho\boldsymbol{u})=0,\\
\displaystyle\boldsymbol{u}_t+\boldsymbol{u}\cdot\nabla\boldsymbol{u}+\frac{A\gamma}{\gamma-1} \nabla\rho^{\gamma-1}+L\boldsymbol{u}=\nabla \log \rho \cdot Q(\boldsymbol{u}),
\end{cases}
\end{equation}
where the operators $(L,Q)$ are defined in \eqref{operatordefinition}.  Since the positive lower bound of $\rho$ can be obtained via its transport equation $\eqref{re}_1$ in short time for strong solutions, by adaptation of the methods used in Matsumura--Nishida \cite{MN} and Sundbye \cite{Sundbye2}, one can establish the following local well-posedness results of the $s$-order ($s=2,3$) regular solutions (as defined in Definition \ref{cjk-po})  when $\inf_{\boldsymbol{x}\in \mathbb{R}^n} \rho_0(\boldsymbol{x})>0$. For brevity, we omit the proof here.

\begin{thm}\label{zth1-po}
Let $n=2$ or $3$, $\bar\rho>0$ in \eqref{e1.3}, and  \eqref{cd1local} hold. 
If the initial data $(\rho_0,\boldsymbol{u}_0)(\boldsymbol{x})$ 
are spherically symmetric and satisfy \eqref{in-start}--\eqref{in-start1}, 
then there exists $T_*>0$ such that the Cauchy 
problem \eqref{eq:1.1benwen}{\rm--}\eqref{e1.3} admits a unique $2$-order regular solution $(\rho,\boldsymbol{u})(t,\boldsymbol{x})$ 
in $[0,T_*]\times\mathbb{R}^n$  satisfying \eqref{decay-est-po} with $T$ 
replaced by $T_*$. 
Moreover, $(\rho,\boldsymbol{u})$ is spherically symmetric with form \eqref{duichenxingshi}, and \eqref{2dclassical1}{\rm--}\eqref{3dclassical1} 
hold with $T$ replaced by $T_*$.
\end{thm}

\begin{thm}\label{zth2-po}
Let $n=2$ or $3$, $\bar\rho>0$ in \eqref{e1.3}, and  \eqref{cd1local} hold. If the initial data $(\rho_0,\boldsymbol{u}_0)(\boldsymbol{x})$ are spherically symmetric and satisfy  \eqref{in-start-high}, 
then there exists $T_*>0$ such that the Cauchy 
problem \eqref{eq:1.1benwen}{\rm--}\eqref{e1.3} admits a unique  $3$-order regular solution $(\rho,\boldsymbol{u})(t,\boldsymbol{x})$ in $[0,T_*]\times\mathbb{R}^n$  satisfying \eqref{decay-est-po} with $T$ replaced by $T_*$ and satisfying
\begin{equation}
\rho_{tt}\in C([0,T_*];L^2(\mathbb{R}^n))\cap L^2([0,T_*];D^1(\mathbb{R}^n)).   
\end{equation}
Moreover, $(\rho,\boldsymbol{u})$ is spherically symmetric  with form \eqref{duichenxingshi}, 
and \eqref{2dclassical2}{\rm--}\eqref{3dclassical2} hold with $T$ replaced by $T_*$.
\end{thm}

\subsubsection{Local-in-time well-posedness in spherical coordinates}

Based on relation  \eqref{duichenxingshi}, 
we say that $(\rho, u)(t,r)$ is the  $s$-order $(s=2,3)$ regular solution of problem \eqref{e1.5hh} in $[0,T_*]\times I$ if the vector function  $(\rho,\boldsymbol{u})(t,\boldsymbol{x})$ is the  $s$-order $(s=2,3)$ regular solution (as defined in Definition \ref{cjk-po}) 
of the Cauchy problem \eqref{eq:1.1benwen}--\eqref{e1.3} in $[0,T_*]\times\mathbb{R}^n$ ($n=2,3$) 
with $\inf_{\boldsymbol{x}\in \mathbb{R}^n} \rho_0(\boldsymbol{x})>0$.

Now, based on Lemma \ref{lemma-initial} in Appendix \ref{appb}, 
Theorems \ref{zth1-po}--\ref{zth2-po} in spherical coordinates 
can be stated as follows:

\begin{thm}\label{rth10po} Let \eqref{cd1local} hold. Assume the initial data $(\rho_0, u_0)(r)$ satisfy 
\begin{equation*}
\inf_{r\in I} \rho_0(r)>0, \quad  \ r^{\frac{m}{2}}\Big(\rho_0-\bar\rho, (\rho_0)_r, \frac{(\rho_0)_r}{r},(\rho_0)_{rr},u_0,\frac{u_0}{r},(u_0)_r,\big(\frac{u_0}{r}\big)_r,(u_0)_{rr}\Big)\in L^2(I),
\end{equation*}
and, in addition, $(\rho_0)_r\in L^\infty(I)$  when $n=3$.
Then there exists $T_*>0$ such that problem \eqref{e1.5hh} 
admits a unique $2$-order regular solution $(\rho, u)(t,r)$ in $[0,T_*]\times I$ satisfying 
\begin{equation}\label{spd-po}
\begin{aligned}
&\inf_{(t,r)\in [0,T_*]\times I}\rho(t,r)>0,\qquad (\rho, \rho_r)\in L^\infty([0,T_*]\times I),\\
&r^{\frac{m}{2}}\Big(\rho-\bar\rho,\rho_r, \frac{\rho_r}{r},\rho_{rr},\rho_t,\rho_{tr}\Big)\in C([0,T_*];L^2(I)),\\
&r^{\frac{m}{2}}\Big(u,\frac{u}{r},u_r,\big(\frac{u}{r}\big)_r,u_{rr},u_t\Big)\in C([0,T_*];L^2(I)),\\
&r^{\frac{m}{2}}\Big(\frac{1}{r}\big(\frac{u}{r}\big)_r,\big(\frac{u}{r}\big)_{rr},\frac{u_{rr}}{r},u_{rrr},\frac{u_t}{r},u_{tr}\Big)\in L^2([0,T_*];L^2(I)),\\
&t^{\frac{1}{2}}r^{\frac{m}{2}}\Big(\frac{1}{r}\big(\frac{u}{r}\big)_r,\big(\frac{u}{r}\big)_{rr},\frac{u_{rr}}{r},u_{rrr},\frac{u_t}{r},u_{tr}\Big)\in L^\infty([0,T_*];L^2(I)),\\
&t^{\frac{1}{2}}r^{\frac{m}{2}}\Big(u_{tt},\big(\frac{u_t}{r}\big)_r,u_{trr}\Big)\in L^2([0,T_*];L^2(I)),
\end{aligned}
\end{equation}
and 
\begin{equation}\label{spd2-po}
\begin{aligned}
\Big(\rho,\rho_r,u,\frac{u}{r},u_r\Big)\in C((0,T_*];C(\bar I)).
\end{aligned}
\end{equation}
\end{thm}

\begin{thm}\label{rth133-po} 
Let \eqref{cd1local} hold. Assume the initial data $(\rho_0, u_0)(r)$ satisfy 
\begin{align*}
&\inf_{r\in I} \rho_0(r)>0, \quad
r^{\frac{m}{2}}\Big(\rho_0-\bar\rho, (\rho_0)_r, \frac{(\rho_0)_r}{r},(\rho_0)_{rr},\big(\frac{(\rho_0)_r}{r}\big)_r,(\rho_0)_{rrr}\Big)\in L^2(I),\\
&r^{\frac{m}{2}}\Big(u_0,\frac{u_0}{r},(u_0)_r,\big(\frac{u_0}{r}\big)_r,(u_0)_{rr},\frac{1}{r}\big(\frac{u_0}{r}\big)_r,\big(\frac{u_0}{r}\big)_{rr},\frac{(u_0)_{rr}}{r},(u_0)_{rrr}\Big)\in L^2(I).
\end{align*}
Then there exists $T_*>0$ such that problem \eqref{e1.5hh} 
admits a unique $3$-order regular solution $(\rho, u)(t,r)$ in $[0,T_*]\times I$ satisfying 
\begin{equation}\label{spd33-po}
\begin{aligned}
&\inf_{(t,r)\in [0,T_*]\times I}\rho(t,r)>0,\qquad (\rho, \rho_r)\in L^\infty([0,T_*]\times I),\\
&r^{\frac{m}{2}}\Big(\rho-\bar\rho,\rho_r, \frac{\rho_r}{r},\rho_{rr},\big(\frac{\rho_r}{r}\big)_r,\rho_{rrr},\rho_t,\rho_{tr},\frac{\rho_{tr}}{r},\rho_{trr}\Big)\in C([0,T_*];L^2(I)),\\
&r^{\frac{m}{2}}\rho_{tt}\in C([0,T_*];L^2(I)),\quad r^\frac{m}{2} \rho_{ttr} \in L^2([0,T_*];L^2(I)),\\
&r^{\frac{m}{2}}\Big(u,\frac{u}{r},u_r,\big(\frac{u}{r}\big)_r,u_{rr},\frac{1}{r}\big(\frac{u}{r}\big)_r,\big(\frac{u}{r}\big)_{rr},\frac{u_{rr}}{r},u_{rrr},u_t,\frac{u_t}{r},u_{tr}\Big)\in C([0,T_*];L^2(I)),\\
&r^{\frac{m}{2}}\Big(\big(\frac{1}{r}(\frac{u}{r})_r\big)_r,\big(\frac{u}{r}\big)_{rrr},\big(\frac{u_{rr}}{r}\big)_r,u_{rrrr},\big(\frac{u_t}{r}\big)_r,u_{trr}\Big)\in L^2([0,T_*];L^2(I)),\\
&t^\frac{1}{2}r^{\frac{m}{2}}\Big(\big(\frac{1}{r}(\frac{u}{r})_r\big)_r,\big(\frac{u}{r}\big)_{rrr},\big(\frac{u_{rr}}{r}\big)_r,u_{rrrr},\big(\frac{u_t}{r}\big)_r,u_{trr},u_{tt}\Big)\in L^\infty([0,T_*];L^2(I)),\\
&t^\frac{1}{2}r^{\frac{m}{2}}\Big(\frac{u_{tt}}{r},u_{ttr},u_{trrr},\frac{u_{trr}}{r},\big(\frac{u_t}{r}\big)_{rr},\frac{1}{r}\big(\frac{u_{tr}}{r}\big)_r\Big)\in L^2([0,T_*];L^2(I)),
\end{aligned}
\end{equation}
and 
\begin{equation}\label{spd233-po}
\Big(\rho,\rho_r,\rho_t,u,\frac{u}{r},u_r\Big)\in C([0,T_*];C(\bar I)),\quad\, \Big(u_t,\big(\frac{u}{r}\big)_r,u_{rr}\Big)\in C((0,T_*];C(\bar I)).
\end{equation}
\end{thm}

\smallskip
\subsection{Outline of the proof}

Our following proof is organized as follows: 
First, in \S \ref{subsub-1132}, we establish the global 
uniform upper bound of $\rho$. Second, in \S \ref{subsub-1133}, we establish the global uniform 
$L^\infty(\mathbb{R}^n)$-estimate for the effective velocity. 
Next, in \S \ref{subsub-1134}, we establish the global uniform lower bound of $\rho$. 
Finally, in \S\ref{globalestimates-2po}--\S\ref{subsub-final}, we make the global uniform estimates 
for the $2$- and $3$-order regular solutions and establish the global well-posedness of regular solutions for general smooth data with strictly positive initial density.

\subsection{Global uniform upper bound of the density}\label{subsub-1132}
In \S\ref{subsub-1132}--\S\ref{subsub-1134}, let $T>0$ be any fixed time, and 
let $(\rho, u)(t,r)$ be the $s$-order ($s=2,3$)  regular solution of problem \eqref{e1.5hh} 
in $[0,T]\times I$ obtained in Theorems \ref{rth10po}--\ref{rth133-po}.
This subsection is devoted to proving the uniform upper bound of $\rho$. 

First, we define the effective velocity, which is identical to 
Definition \ref{def-effective} in \S \ref{section-upper-density}; 
we present it here for subsequent development. 
\begin{mydef}\label{def-effective3}
Let $(\rho,u,\alpha)$ be defined as in {\rm \S \ref{section-intro}}. 
Define the effective velocity $v$ as
\begin{equation}\label{V-expression3}
v:=u+2\alpha (\log\rho)_r.
\end{equation}
Besides, define $v_0:=v|_{t=0}=u_0+2\alpha (\log\rho_0)_r$.
\end{mydef} 

Next, we establish the fundamental energy estimate and the BD entropy estimate.
\begin{lem}\label{energy-bd-po}
Let $j_\gamma(\rho)=j_\gamma(\rho)(t,r):=j_\gamma(\rho(t,r))$. There exists a constant $C_0>0$ such that, for any $t\in [0,T]$,
\begin{align*}
\int_0^\infty \big(r^m\rho u^2+r^mj_\gamma(\rho))(t,\cdot)\,\mathrm{d}r+\int_0^t\int_0^\infty r^m\Big(\rho |u_r|^2+\rho\frac{u^2}{r^2}\Big)\,\mathrm{d}r\mathrm{d}s\leq C_0,\\
\int_0^\infty \big(r^m\rho v^2+r^m |(\sqrt{\rho})_r|^2+r^mj_\gamma(\rho))(t,\cdot)\,\mathrm{d}r+\int_0^t\int_0^\infty r^m\rho^{\gamma-2}|\rho_r|^2\,\mathrm{d}r\mathrm{d}s\leq C_0.
\end{align*}
\end{lem}

\smallskip
\begin{proof}We divide the proof into two steps.

\smallskip
\textbf{1.} First, a direct calculation yields that 
\begin{equation}\label{12.10}
\frac{A}{\gamma-1}(r^m j_\gamma(\rho))_t+\Big(\frac{A\gamma}{\gamma-1}r^m (\rho^{\gamma-1}-\bar\rho^{\gamma-1}) \rho u\Big)_r=A(\rho^\gamma)_rr^mu.
\end{equation} 
Then multiplying $\eqref{e1.5hh}_2$ by $r^m u$, together with \eqref{12.10}, gives
\begin{equation}\label{eap1-0}
\begin{split}
&\Big(\frac{r^m}{2} \rho u^2+\frac{A}{\gamma-1} r^m j_\gamma(\rho)\Big)_t+ 2\alpha r^m\Big(\rho |u_r|^2+  m \rho \frac{u^2}{r^2}\Big)\\
&=\Big(\underline{-\frac{A\gamma}{\gamma-1}r^m (\rho^{\gamma-1}-\bar\rho^{\gamma-1}) \rho u+2\alpha r^m\rho u u_r-\frac{r^m}{2}\rho u^3}_{:=\tilde{\mathcal{B}}_1}\Big)_r.
\end{split}
\end{equation}

Next, we need to show  that $\tilde{\mathcal{B}}_1\in W^{1,1}(I)$ and $\tilde{\mathcal{B}}_1|_{r=0}=0$ for {\it a.e.} $t\in (0,T)$, which allows us to apply Lemma \ref{calculus} to obtain
\begin{equation}\label{in-B1-po}
\int_0^\infty (\tilde{\mathcal{B}}_1)_r\,\mathrm{d}r=-\tilde{\mathcal{B}}_1|_{r=0}=0.
\end{equation}
On one hand,  it follows from \eqref{spd-po}--\eqref{spd2-po} (or \eqref{spd33-po}--\eqref{spd233-po}) that 
\begin{equation*}
\rho_r\in L^\infty(I),\qquad \Big(\rho,u,\frac{u}{r},u_r\Big)\in C(\bar I), \qquad r^\frac{m}{2}\Big(\rho-\bar\rho,\rho_r,u,\frac{u}{r},u_r,u_{rr}\Big)\in L^2(I)
\end{equation*}
for {\it a.e.} $t\in (0,T)$, which yields 
that $\tilde{\mathcal{B}}_1|_{r=0}=0$. 
On the other hand,  it follows from \eqref{spd-po}--\eqref{spd2-po} (or \eqref{spd33-po}--\eqref{spd233-po})
and  the Taylor expansion that 
\begin{equation*}
z^{\gamma-1}-\bar\rho^{\gamma-1}=(\gamma-1)z_\vartheta^{\gamma-2}(z-\bar\rho) \qquad \text{with $z_\vartheta=\vartheta z+(1-\vartheta)\bar\rho\,\,$ for some $\vartheta\in [0,1]$}.
\end{equation*}
Then this, together with the H\"older inequality, yields that $\tilde{\mathcal{B}}_1\in W^{1,1}(I)$ for {\it a.e.} $t\in (0,T)$:
\begin{align*}
|\tilde{\mathcal{B}}_1|_1&\leq C_0\big|r^m\big((\rho^{\gamma-1}-\bar\rho^{\gamma-1})\rho u,\rho u u_r,\rho u^3\big)\big|_1\\
&\leq C_0|\rho_\vartheta^{\gamma-2}\rho|_\infty \big|r^\frac{m}{2}(\rho-\bar\rho)\big|_2|r^\frac{m}{2}u|_2\\
&\quad +C_0|\rho|_\infty\big(|r^\frac{m}{2}u|_2|r^\frac{m}{2}u_r|_2+|u|_\infty|r^\frac{m}{2}u|_2^2\big) <\infty,\\
|(\tilde{\mathcal{B}}_1)_r|_1&\leq C_0\big|r^{m-1}\big((\rho^{\gamma-1}-\bar\rho^{\gamma-1})\rho u,\rho u u_r,\rho u^3\big)\big|_1\\
&\quad+C_0\big|r^m\big(\rho^{\gamma-1}\rho_r u,\bar\rho^{\gamma-1}\rho_r u,(\rho^{\gamma-1}-\bar\rho^{\gamma-1})\rho u_r\big)\big|_1\\
&\quad +C_0\big|r^m\big(\rho_r u u_r,\rho (u_r)^2,\rho u u_{rr},\rho_r u^3,\rho u^2u_r\big)\big|_1\\
&\leq C_0|\rho_\vartheta^{\gamma-2}\rho|_\infty\big|r^\frac{m}{2}(\rho-\bar\rho)\big|_2\big|r^\frac{m-2}{2}u\big|_2\\
&\quad +C_0|\rho|_\infty\Big(|r^\frac{m-2}{2}u|_2|r^\frac{m}{2}u_r|_2+\Big|\frac{u}{r}\Big|_\infty|r^\frac{m}{2}u|_2^2\Big)\\
&\quad+ C_0(|\rho|_\infty^{\gamma-1}+1) \big|r^\frac{m}{2}\rho_r\big|_2|r^\frac{m}{2}u|_2+C_0|\rho_\vartheta^{\gamma-2} \rho|_\infty\big|r^\frac{m}{2}(\rho-\bar\rho)\big|_2|r^\frac{m}{2}u_r|_2\\
&\quad+ C_0|\rho_r|_\infty|r^\frac{m}{2}u|_2|r^\frac{m}{2}u_r|_2+ C_0|\rho|_\infty\big(|r^\frac{m}{2}u_r|_2^2+|r^\frac{m}{2}u|_2|r^\frac{m}{2}u_{rr}|_2\big)\\
&\quad+ C_0|\rho_r|_\infty|u|_\infty|r^\frac{m}{2}u|_2^2+ C_0|\rho|_\infty|u_r|_\infty|r^\frac{m}{2}u|_2^2<\infty,
\end{align*}
where $\rho_\vartheta=\vartheta\rho+(1-\vartheta)\bar\rho$. This, along with Lemma \ref{calculus}, yields \eqref{in-B1-po}.

Integrating above over $[0,t]\times I$, we obtain from \eqref{in-B1-po} and Lemma \ref{equiv-initial} that
\begin{equation}\label{eap2-0}
\begin{aligned}
&\int_0^\infty \Big(\frac{1}{2}r^m\rho u^2+ \frac{A}{\gamma-1}r^m j_\gamma(\rho)\Big)(t,r) \,\mathrm{d}r+ 2\alpha \int_0^t \int_0^\infty r^m\Big(\rho |u_r|^2+ m\rho \frac{u^2}{r^2}\Big)\,\mathrm{d}r\mathrm{d}t\\
&\leq  \int_0^\infty \Big(\frac{1}{2}r^m\rho_0 |u_0|^2+ \frac{A}{\gamma-1}r^m j_\gamma(\rho_0)\Big) \,\mathrm{d}r\leq C_0.
\end{aligned} 
\end{equation}

\smallskip
\textbf{2.} By Lemmas \ref{equiv-initial} and \ref{lemma-initial} and Definition \ref{def-effective3},  $(r^m\rho_0)^\frac{1}{2}v_0\in L^2(I)$. Thus, the BD estimate can be obtained via an analogous argument used 
in Step 2 of the proof of Lemma \ref{energy-BD}. 
First, we obtain the equation of effective velocity:
\begin{equation}\label{ess-po}
\rho(v_t+u v_r)+P_r=0.
\end{equation}

Next, multiplying \eqref{ess-po} by $r^m v$ and 
integrating the resulting equality over $[0,t]\times I$, we 
obtain from $\eqref{e1.5hh}_1$ and Lemma \ref{equiv-initial} that
\begin{equation}\label{es-v-po}
\begin{aligned}
&\int_0^\infty\Big(\frac12 r^m\rho v^2+\frac{A}{\gamma-1}r^mj_\gamma(\rho)\Big)(t, r)\,\mathrm{d}r+ 2A\alpha \gamma\int_0^t\int_0^\infty r^m\rho^{\gamma-2}|\rho_r|^2\,\mathrm{d}r\mathrm{d}s\\
&\leq  \int_0^\infty \Big(\frac{1}{2}r^m\rho_0 |v_0|^2+ \frac{A}{\gamma-1}r^m j_\gamma(\rho_0)\Big) \,\mathrm{d}r\leq C_0.
\end{aligned}
\end{equation}

Finally, \eqref{V-expression3}, together with the energy estimates and \eqref{es-v-po}, yields
\begin{equation*}
|r^\frac{m}{2}(\sqrt{\rho})_r|_2=\frac{1}{2}\big|(r^m\rho)^\frac{1}{2}(\log\rho)_r\big|_2\leq C_0\big|(r^m\rho)^\frac{1}{2}(v,u)\big|_2\leq C_0.
\end{equation*}
\end{proof}

Clearly, in this case, Corollary \ref{cor-v} in \S \ref{section-upper-density} still holds. We present it here for later use.
\begin{cor}\label{cor-v''}
The effective velocity $v$ satisfies the following equation{\rm :}
\begin{equation}\label{eq:effective3}
v_t+uv_r+ \frac{A\gamma}{2\alpha} \rho^{\gamma-1} (v-u)=0.
\end{equation}
\end{cor}

Besides, we can also obtain the following estimates of $\rho$:
\begin{lem}\label{Oliciz-norm}
 $j_\gamma(\rho)(t,r)\geq 0$ for all $(t,r)\in[0,T]\times I$. Moreover, for any  $\sigma\in (0,\bar\rho)$, 
there exists a constant $C(\sigma)>0$ such that,
for all $t\in [0,T]$,
\begin{equation*}
\big|\chi_{E_{\sigma}}r^\frac{m}{2}(\rho(t)-\bar\rho)\big|_2+\big|\chi_{E^{\sigma}}r^\frac{m}{\gamma}(\rho(t)-\bar\rho)\big|_\gamma\leq C(\sigma).
\end{equation*}
\end{lem}

\begin{proof}
Since $j_\gamma(\rho)\ge 0$, it follows from Lemmas \ref{energy-bd-po} and \ref{lemma-initial} that $\|j_\gamma(\rho)(t)\|_{L^1}\leq C_0$ for all $t\in [0,T]$. Then, following the proof of Lemma 5.3 on 
\cite[Page 43]{lions}, we obtain
\begin{equation}\label{lionss}
\big\|\chi_{E_{\sigma}}(\rho(t)-\bar\rho)\big\|_{L^2}+\big\|\chi_{E^{\sigma}}(\rho(t)-\bar\rho)\big\|_{L^\gamma}\leq C(\sigma),
\end{equation}
for $\sigma\in (0,\bar\rho)$, which, along with Lemma \ref{lemma-initial}, yields the desired estimates. 

For the reader's convenience, we still provide a brief proof of \eqref{lionss}. 
Let $\sigma\in (0,\bar\rho)$. First, on the set $E_\sigma$, $\bar\rho-\sigma\leq \rho\leq 2\bar\rho$. Then it follows from \eqref{erjie-taylor} that
\begin{equation*} 
j_\gamma(\rho)=\frac{\gamma(\gamma-1)}{2} \rho_{\vartheta}^{\gamma-2}(\rho-\bar\rho)^2\geq C(\sigma)^{-1}(\rho-\bar\rho)^2,
\end{equation*} 
where $\rho_\vartheta=\vartheta \rho+(1-\vartheta)\bar\rho$. This implies that $\big\|\chi_{E_{\sigma}}(\rho(t)-\bar\rho)\big\|_{L^2}\leq C(\sigma)$. 

Next, on the set $E^\sigma$, we can show that
\begin{equation}\label{jg}
j_\gamma(\rho)\geq C_0^{-1}|\rho-\bar\rho|^\gamma.
\end{equation}
Indeed, define the function
\begin{equation*}
J_\gamma(z):=\frac{j_\gamma(z)}{|z-\bar\rho|^{\gamma}}=\frac{(z^\gamma-\bar\rho^\gamma)-\gamma\bar\rho^{\gamma-1}(z-\bar\rho)}{|z-\bar\rho|^\gamma} \qquad\,\, \text{for $|z-\bar\rho|>\sigma$ and $z\geq 0$}.
\end{equation*}
A direct calculation gives that 
\begin{equation*}
\begin{aligned}
J_\gamma'(z)&=\frac{z-\bar\rho}{|z-\bar\rho|}\frac{H_\gamma(z)}{|z-\bar\rho|^{\gamma+1}} \qquad \text{with $H_\gamma(z):=j_\gamma'(z)(z-\bar\rho)-\gamma j_\gamma(z)$},\\
H'_\gamma(z)&=j_\gamma''(z)(z-\bar\rho)-(\gamma-1)j_\gamma'(z)=\gamma(\gamma-1)\bar\rho(\bar\rho^{\gamma-2}-z^{\gamma-2}).
\end{aligned}
\end{equation*}
This yields that $J'_\gamma(z)<0$ on $(\bar\rho+\sigma,\infty)$, 
$J'_\gamma(z)>0$ on $(0,\bar\rho-\sigma)$, and $J_\gamma(z)$ is strictly decreasing on $(\bar\rho+\sigma,\infty)$ and strictly increasing on $(0,\bar\rho-\sigma)$. Hence, we see that
\begin{align*}
\frac{j_\gamma(z)}{|z-\bar\rho|^{\gamma}}&=J_\gamma(z)\geq \lim_{z\to\infty}J_\gamma(z)=1 &&\text{for all $z\in (\bar\rho+\sigma,\infty)$},\\
\frac{j_\gamma(z)}{|z-\bar\rho|^{\gamma}}&=J_\gamma(z)\geq \lim_{z\to 0}J_\gamma(z)=\gamma-1 &&\text{for all $z\in (0,\bar\rho-\sigma)$}.
\end{align*}
Setting $z=\rho$ above yields claim \eqref{jg}. 
Therefore, it follows from \eqref{jg} and Lemma \ref{equiv-initial} that $\big\|\chi_{E^{\sigma}}(\rho(t)-\bar\rho)\big\|_{L^\gamma}\leq C_0$. 
\end{proof}

The next lemma concerns the weighted $L^p(0,\omega)$-estimates of $\rho$
for $\omega>0$ and $p\in [1,\infty]$.
\begin{lem}\label{l4.3-po}
For any $t\in [0,T]$ and $\omega>0$,
\begin{enumerate}
\item[$\mathrm{(i)}$] When $n=2$, there exist two positive constants $C(p,\nu,\omega)$ and $C(\nu,\omega)$ 
such that
\begin{equation}\label{wprho2-po}
\begin{aligned}
\big|\chi_{\omega}^\flat r^{\nu}\rho(t)\big|_p&\leq C(p,\nu,\omega) &&\quad\text{for any }\nu>-\frac{1}{p}\text{ and }p\in[1,\infty),\\
\big|\chi_{\omega}^\flat r^{\nu}\rho(t)\big|_\infty&\leq C(\nu,\omega) &&\quad\text{for any }\nu>0;
\end{aligned}
\end{equation}
\item[$\mathrm{(ii)}$] When $n=3$, there exist two positive constants $C(p,\omega)$ and $C(\omega)$ such that
\begin{equation}\label{wprho3-po}
\big|\chi_{\omega}^\flat r^{1-\frac{1}{p}}\rho(t)\big|_p\leq C(p,\omega)  \ \ \text{for any $p\in[1,\infty)$},\qquad \big|\chi_{\omega}^\flat r \rho(t)\big|_\infty\leq C(\omega).
\end{equation}
\end{enumerate}
\end{lem}
\begin{proof}
Since we focus only on the estimates of $\rho$ on $[0,\omega]$ for $\omega>0$,
we perform the argument as in the proof of Lemma \ref{l4.3} and use Lemmas \ref{energy-bd-po}--\ref{Oliciz-norm} and \ref{hardy} to obtain the desired conclusions. 
Here, for brevity, we only sketch the proof for the 3-D case.

First, it follows from Lemma \ref{Oliciz-norm} and the H\"older inequality that, for fixed $\sigma\in (0,\bar\rho)$, 
\begin{align*}
|\chi_{\omega}^\flat r\sqrt{\rho}|_2^2&=|\chi_{\omega}^\flat r^2 \rho|_1\leq \big|\chi_{\omega}^\flat \chi_{E_\sigma} r^2 \rho\big|_1+\big|\chi_{\omega}^\flat \chi_{E^\sigma} r^2 \rho\big|_1\\
&\leq \big|\chi_{\omega}^\flat \chi_{E_\sigma} r^2 (\rho-\bar\rho)\big|_1+\big|\chi_{\omega}^\flat \chi_{E^\sigma} r^2 (\rho-\bar\rho)\big|_1+C(\omega)\\
&\leq |\chi_{\omega}^\flat r|_2 \big|\chi_{E_\sigma} r(\rho-\bar\rho)\big|_2+|\chi_{\omega}^\flat r |_2^\frac{2\gamma-2}{\gamma}\big|\chi_{E^\sigma} r^\frac{2}{\gamma}(\rho-\bar\rho)\big|_\gamma+C(\omega)\leq C(\omega). 
\end{align*}
Next, let $p\in[1,\infty)$. 
According to the above estimate and Lemmas \ref{energy-bd-po} and \ref{hardy}, we obtain 
\begin{align*}
\big|\chi_{\omega}^\flat r^{1-\frac{1}{p}}\rho\big|_p&= \big|\chi_{\omega}^\flat r^{\frac{1}{2}-\frac{1}{2p}}\sqrt{\rho}\big|_{2p}^2\leq C(p,\omega)\big(\big|\chi_{\omega}^\flat r\sqrt{\rho}\big|_2^2+\big|\chi_{\omega}^\flat r(\sqrt{\rho})_r\big|_2^2\big)\leq C(p,\omega),\\
\big|\chi_{\omega}^\flat r \rho\big|_\infty& =\big|\chi_{\omega}^\flat r^{\frac{1}{2}}\sqrt{\rho}\big|_{\infty}^2\leq C(\omega)\big(\big|\chi_{\omega}^\flat r\sqrt{\rho}\big|_2^2+\big|\chi_{\omega}^\flat r(\sqrt{\rho})_r\big|_2^2\big)\leq C(\omega),
\end{align*}
which thus leads to \eqref{wprho3-po}.
\end{proof}

Now, based on Lemmas \ref{Oliciz-norm}--\ref{l4.3-po}, we can obtain the global uniform upper bound of $\rho$ in the exterior domain.
\begin{lem}\label{Lemma-12.4}
There exist a constant $C_0>0$ such that,  for any $t\in [0,T]$, $p\in [2,\infty]$, 
and $\omega\in (0,\infty)$,
\begin{equation}\label{ccc}
\big|\chi^\sharp_\omega r^\frac{m}{2}(\sqrt{\rho}(t)-\sqrt{\bar\rho})\big|_p\leq C_0.
\end{equation}
In particular, for any  $\omega\in (0,\infty)$, there exist a constant $C(\omega)>0$ such that,  for any $t\in [0,T]$,
\begin{equation}\label{cccc}
|\chi_\omega^\sharp \rho(t)|_\infty\leq C(\omega).
\end{equation}
\end{lem}
\begin{proof}We divide the proof into two steps.

\smallskip
\textbf{1.} We first show that Lemma \ref{Lemma-12.4} holds for all $\omega\in [\omega_0,\infty)$ with some constant $\omega_0>0$ depending only on $(n,\alpha,\gamma,A,\bar\rho)$. Let $\omega>0$ and $\sigma\in (0,\bar\rho)$. It follows from Lemma \ref{Oliciz-norm} and the H\"older and Chebyshev inequality (Lemma \ref{cheby}) that
\begin{equation}\label{1-lwuqiong-l2'}
\begin{aligned}
&\big|\chi_\omega^\sharp\chi_{E^{\sigma}}r^\frac{m}{2} (\sqrt{\rho}-\sqrt{\bar\rho})\big|_2^2\\
&\leq \big|\chi_\omega^\sharp r^\frac{m}{2} (\sqrt{\rho}-\sqrt{\bar\rho}) \big|_\infty \big|\chi_\omega^\sharp r^\frac{m}{2}(\sqrt{\rho}-\sqrt{\bar\rho})\big|_2\Big(\int_\omega^\infty \chi_{E^\sigma}\,\mathrm{d}r\Big)^\frac{1}{2}\\
&\leq \omega^{-\frac{m}{2}}\big|\chi_\omega^\sharp r^\frac{m}{2} (\sqrt{\rho}-\sqrt{\bar\rho}) \big|_\infty \big|\chi_\omega^\sharp r^\frac{m}{2}(\sqrt{\rho}-\sqrt{\bar\rho})\big|_2\Big(\int_0^\infty \chi_{E^\sigma}r^m\,\mathrm{d}r\Big)^\frac{1}{2}\\
&\leq \omega^{-\frac{m}{2}}\big|\chi_\omega^\sharp r^\frac{m}{2} (\sqrt{\rho}-\sqrt{\bar\rho}) \big|_\infty \big|\chi_\omega^\sharp r^\frac{m}{2}(\sqrt{\rho}-\sqrt{\bar\rho})\big|_2 \frac{\big|r^\frac{m}{\gamma}(\rho(t)-\bar\rho)\chi_{E^{\sigma}}\big|_\gamma^\frac{\gamma}{2}}{\sigma^\frac{\gamma}{2}} \\
&\leq C(\sigma)\omega^{-\frac{m}{2}}\big|\chi_\omega^\sharp r^\frac{m}{2} (\sqrt{\rho}-\sqrt{\bar\rho}) \big|_\infty \big|\chi_\omega^\sharp r^\frac{m}{2}(\sqrt{\rho}-\sqrt{\bar\rho})\big|_2.
\end{aligned}
\end{equation}
On the other hand, it follows from Lemmas \ref{energy-bd-po} and \ref{calculus}, 
and the H\"older inequality that
\begin{equation}\label{2-lwuqiong-l2}
\begin{aligned}
\big|\chi_\omega^\sharp r^\frac{m}{2} (\sqrt{\rho}-\sqrt{\bar\rho}) \big|_\infty^2&\leq \int_\omega^\infty \big|\big(r^m (\sqrt{\rho}-\sqrt{\bar\rho})^2\big)_r\big|\,\mathrm{d}r\\
&\leq \big|\chi_\omega^\sharp r^\frac{m}{2} (\sqrt{\rho}-\sqrt{\bar\rho}) \big|_2\big|r^\frac{m}{2} (\sqrt{\rho})_r\big|_2+m\big|\chi_\omega^\sharp  r^\frac{m-1}{2} (\sqrt{\rho}-\sqrt{\bar\rho}) \big|_2^2\\
&\leq C_0\big|\chi_\omega^\sharp r^\frac{m}{2} (\sqrt{\rho}-\sqrt{\bar\rho}) \big|_2+m\omega^{-1}\big|\chi_\omega^\sharp r^\frac{m}{2} (\sqrt{\rho}-\sqrt{\bar\rho}) \big|_2^2. 
\end{aligned}
\end{equation}
Substituting the above in to \eqref{1-lwuqiong-l2'} leads to
\begin{equation}\label{1-lwuqiong-l2}
\begin{aligned}
\big|\chi_\omega^\sharp\chi_{E^{\sigma}}r^\frac{m}{2} (\sqrt{\rho}-\sqrt{\bar\rho})\big|_2&\leq C(\sigma) \omega^{-\frac{m}{4}} \big|\chi_\omega^\sharp r^\frac{m}{2}(\sqrt{\rho}-\sqrt{\bar\rho})\big|_2^\frac{3}{4}\\
&\quad +C(\sigma)\omega^{-\frac{m+1}{4}} \big|\chi_\omega^\sharp r^\frac{m}{2}(\sqrt{\rho}-\sqrt{\bar\rho})\big|_2. 
\end{aligned}
\end{equation}

Next, it follows from the Taylor expansion that 
\begin{equation*}
\sqrt{\rho}-\sqrt{\bar\rho}=\frac{1}{2\sqrt{\rho_\vartheta}}(\rho-\bar\rho),
\end{equation*}
where $\rho_\vartheta=\vartheta\rho+(1-\vartheta)\bar\rho$ for some $\vartheta\in [0,1]$. 
If $r\in E_\sigma$ and $\sigma\in (0,\frac{\bar\rho}{2})$, then $\rho_\vartheta\in [\frac{\bar\rho}{2},\frac{3\bar\rho}{2}]$ so that, for all $(t,r)\in [0,T]\times I$,
\begin{equation*}
C_0^{-1}\big(\chi_{E_\sigma}|\rho-\bar\rho|\big)\leq \chi_{E_\sigma}|\sqrt{\rho}-\sqrt{\bar\rho}|\leq C_0\big(\chi_{E_\sigma}|\rho-\bar\rho|\big),
\end{equation*}
which, along with Lemma \ref{Oliciz-norm}, implies that
\begin{equation}\label{3-wuqiong}
\big|\chi_\omega^\sharp\chi_{E_{\sigma}} r^\frac{m}{2} (\sqrt{\rho}-\sqrt{\bar\rho})\big|_2\leq C_0\big|\chi_{E_{\sigma}}r^\frac{m}{2}(\rho(t)-\bar\rho)\big|_2\leq C(\sigma).
\end{equation}

Therefore, choosing suitable fixed $\sigma\in (0,\frac{\bar\rho}{2})$ 
and collecting \eqref{1-lwuqiong-l2}--\eqref{3-wuqiong}, 
we obtain from the Young inequality that 
\begin{equation}\label{1123}
\begin{aligned}
&\big|\chi^\sharp_\omega r^\frac{m}{2}(\sqrt{\rho}(t)-\sqrt{\bar\rho})\big|_2\\
&\leq \big|\chi^\sharp_\omega\chi_{E^\sigma} r^\frac{m}{2}(\sqrt{\rho}(t)-\sqrt{\bar\rho})\big|_2+\big|\chi^\sharp_\omega\chi_{E_\sigma} r^\frac{m}{2}(\sqrt{\rho}(t)-\sqrt{\bar\rho})\big|_2\\
&\leq C_0\omega^{-\frac{m}{4}}\big|\chi_\omega^\sharp r^\frac{m}{2} (\sqrt{\rho}-\sqrt{\bar\rho}) \big|_2^\frac{3}{4}+C_0\omega^{-\frac{m+1}{4}}\big|\chi_\omega^\sharp r^\frac{m}{2} (\sqrt{\rho}-\sqrt{\bar\rho}) \big|_2+C_0\\
&\leq C_0\big(\omega^{-\frac{m}{3}}+\omega^{-\frac{m+1}{4}}\big)\big|\chi_\omega^\sharp r^\frac{m}{2} (\sqrt{\rho}-\sqrt{\bar\rho}) \big|_2+C_0.
\end{aligned}
\end{equation}
Choosing $\omega_0>0$ sufficiently large such that 
\begin{equation}\label{step1-def}
\omega_0^{-\frac{m}{3}}+\omega_0^{-\frac{m+1}{4}}= (2C_0)^{-1},
\end{equation}
we obtain from \eqref{1123} that, for all $t\in [0,T]$ and $\omega\in [\omega_0,\infty)$,
\begin{equation}\label{1214}
\big|\chi^\sharp_\omega r^\frac{m}{2}(\sqrt{\rho}(t)-\sqrt{\bar\rho})\big|_2\leq C_0.
\end{equation}
Of course, we also obtain from the above and \eqref{2-lwuqiong-l2} that, 
for all $t\in [0,T]$ and $\omega\in [\omega_0,\infty)$,
\begin{equation}\label{12.15}
\big|\chi^\sharp_\omega r^\frac{m}{2}(\sqrt{\rho}(t)-\sqrt{\bar\rho})\big|_\infty\leq C_0.
\end{equation}

The $L^p(I)$-estimates ($p\in (2,\infty)$) of $\chi^\sharp_\omega r^\frac{m}{2}(\sqrt{\rho}(t)-\sqrt{\bar\rho})$ follow from \eqref{1214}--\eqref{12.15}. Finally, it follows from \eqref{12.15} that, for all $t\in[0,T]$ and $\omega\in[\omega_0,\infty)$,
\begin{equation}\label{cccp}
\begin{aligned}
|\chi_\omega^\sharp \rho|_\infty =|\chi_\omega^\sharp \sqrt{\rho}|_\infty^2 &\leq C_0|\chi_\omega^\sharp(\sqrt{\rho}-\sqrt{\bar\rho})|_\infty^2+C_0\bar\rho\\
&\leq C_0\omega_0^{-m}\big|\chi_\omega^\sharp r^\frac{m}{2}(\sqrt{\rho}-\sqrt{\bar\rho})\big|_\infty^2+C_0\bar\rho\leq C_0.
\end{aligned}   
\end{equation}

To sum up, we have shown that Lemma \ref{Lemma-12.4} holds for all $\omega\in [\omega_0,\infty)$ with some constant $\omega_0>0$, depending only on $(n,\alpha,\gamma,A,\bar\rho)$

\smallskip
\textbf{2.} Now, based on Lemma \ref{l4.3-po}, we can show that Lemma \ref{Lemma-12.4} holds for arbitrary constant $\omega>0$. For brevity, we take $n=3$ $(m=2)$ as an example, since the 2-D case can be derived analogously. Let $\omega_0$ be defined in \eqref{step1-def}. 
Due to Lemma \ref{l4.3-po} and the fact that \eqref{ccc} holds for $\omega\in [\omega_0,\infty)$, 
then, for arbitrary $\omega\in (0,\omega_0)$,
\begin{align*}
\big|\chi^\sharp_{\omega} r(\sqrt{\rho}(t)-\sqrt{\bar\rho})\big|_p
&\leq \big|\chi_{[\omega,\omega_0]}r(\sqrt{\rho}-\sqrt{\bar\rho})\big|_{p}+\big|\chi^\sharp_{\omega_0} r(\sqrt{\rho}-\sqrt{\bar\rho})\big|_p\\
&\leq |\chi_{\omega_0}^\flat\sqrt{r}|_{p}\big|\chi_{\omega_0}^\flat \sqrt{r}(\sqrt{\rho}-\sqrt{\bar\rho})\big|_\infty+C_0\leq C_0\big|\chi_{\omega_0}^\flat r\rho\big|_\infty^\frac{1}{2}+C_0\leq C_0,
\end{align*}
where $p\in [2,\infty]$. Similarly, for $\omega\in (0,\omega_0)$, it follows from Lemma \ref{l4.3-po} and \eqref{cccp} that
\begin{equation*}
|\chi_\omega^\sharp \rho|_\infty \leq \big|\chi_{[\omega,\omega_0]}\rho\big|_{\infty}+|\chi_{\omega_0}^\sharp \rho|_\infty \leq \omega^{-1}|\chi_{\omega_0}^\flat r\rho|_\infty+C_0\leq C(\omega),
\end{equation*}
which thus leads to \eqref{cccc}. 
\end{proof}

Due to Lemmas \ref{l4.3-po}--\ref{Lemma-12.4}, we  can obtain the $L^2(I)$-estimate of $r^\frac{m}{2}(\rho-\bar\rho)$.
\begin{lem}\label{lll222}
There exists a constant $C_0>0$ such that, for any $t\in[0,T]$,
\begin{equation*}
\big|r^\frac{m}{2}(\rho(t)-\bar\rho)\big|_2\leq C_0.
\end{equation*}
\end{lem}
\begin{proof}
This can be seen directly from Lemmas \ref{l4.3-po}--\ref{Lemma-12.4} that 
\begin{align*}
\big|r^\frac{m}{2}(\rho-\bar\rho)\big|_2&\leq  \big|\chi_{1}^\sharp r^\frac{m}{2}(\sqrt{\rho}-\sqrt{\bar\rho})\big|_2|\chi_{1}^\sharp(\sqrt{\rho}+\sqrt{\bar\rho})|_\infty+\big|\chi_{1}^\flat r^\frac{m}{2}(\rho-\bar\rho)\big|_2 \\
&\leq  C_0|\chi_{1}^\sharp \rho|_\infty^\frac{1}{2}+|\chi_{1}^\flat r^\frac{m}{2}\rho|_2+C_0\leq C_0. 
\end{align*}
\end{proof}

Next, we can show the $L^{p}(I)$-estimates $(p\in [4,\infty))$ of $(r^m\rho)^{\frac{1}{p}}u$.
\begin{lem} \label{lma-po}
Let $\gamma\in (1,\infty)$ if $n=2$ and $\gamma\in (1,3)$ if $n=3$. For any $p\in [4,\infty)$ and $\epsilon\in (0,1)$,
there exist two positive  constants $C(p)$ and $C(p,\epsilon)$ such that, for any $t\in [0,T]$, 
\begin{equation}\label{dt-u-p-po}
\begin{aligned}
&\frac{\mathrm{d}}{\mathrm{d}t}\big|(r^m\rho)^{\frac{1}{p}} u\big|_{p}^{p} + p\alpha \Big(\big|(r^m\rho)^{\frac{1}{2}}|u|^{\frac{p-2}{2}} u_r\big|_2^2+\big|(r^{m-2}\rho)^{\frac{1}{p}}u\big|_{p}^{p}\Big)\\
&\leq C(p)\big|(r^m\rho)^\frac{1}{p-2}u\big|_{p-2}^{p-2} + C(p,\epsilon)+\epsilon\big|(r^m\rho^\gamma)^{\frac{1}{p}} v\big|_{p}^{p}.
\end{aligned}
\end{equation}
\end{lem}
\begin{proof}We divide the proof into three steps.

\smallskip
\textbf{1.} Let $p\in [4,\infty)$. Multiplying $\eqref{e1.5hh}_2$ by $r^m|u|^{p-2}u$, along with $\eqref{e1.5hh}_1$, gives
\begin{equation}\label{eq:522-po}
\begin{aligned}
&\frac{1}{p}(r^m\rho |u|^{p})_t+2\alpha(p-1)r^m\rho|u|^{p-2}|u_r|^2+2\alpha mr^{m-2}\rho |u|^{p}\\
&=(p-1)Ar^m\rho^\gamma|u|^{p-2}u_r+mA r^{m-1}\rho^\gamma|u|^{p-2}u\\
&\quad +\Big(\underline{2\alpha r^m\rho|u|^{p-2}u u_r-Ar^m\rho^\gamma|u|^{p-2}u-\frac{1}{p}r^m\rho u |u|^{p}}_{:=\tilde{\mathcal{B}}_2}\Big)_r.
\end{aligned}
\end{equation}

Next, we need to show  that $\tilde{\mathcal{B}}_2\in W^{1,1}(I)$ and $\tilde{\mathcal{B}}_2|_{r=0}=0$ for {\it a.e.} $t\in (0,T)$, which allows us to apply Lemma \ref{calculus} to obtain
\begin{equation}\label{eq:523-po}
\int_0^\infty (\tilde{\mathcal{B}}_2)_r\,\mathrm{d}r=-\tilde{\mathcal{B}}_2|_{r=0}=0.
\end{equation}
$\tilde{\mathcal{B}}_2|_{r=0}=0$ follows easily from the fact that $(\rho,u,u_r)\in C(\bar I)$ for each $t\in (0,T]$ due to \eqref{spd2-po} (or \eqref{spd233-po}). On the other hand,  based on \eqref{spd-po}--\eqref{spd2-po} (or \eqref{spd33-po}--\eqref{spd233-po}), one has 
\begin{equation*}
\Big(\rho,\rho_r, u, \frac{u}{r}, u_r\Big)\in L^\infty(I),\quad\, r^\frac{m}{2}\Big(\rho_r,u,\frac{u}{r},u_r\Big)\in L^2(I)  \qquad\,
\mbox{for {\it a.e.} $t\in (0,T)$}.
\end{equation*}
Then we obtain from $p\in [4,\infty)$ and the H\"older inequality that 
\begin{align*}
|\tilde{\mathcal{B}}_2|_1&\leq C_0\big|r^m\big(\rho |u|^{p-1}|u_r|,\rho^\gamma|u|^{p-1},\rho |u|^{p+1}\big)\big|_1\\
&\leq C_0\big(|\rho|_\infty |u|_\infty^{p-2}|r^\frac{m}{2}u|_2|r^\frac{m}{2}u_r|_2+ |\rho|_\infty^{\gamma}|u|_\infty^{p-3}|r^\frac{m}{2}u|_2^2  + |\rho|_\infty |u|_\infty^{p-1}|r^\frac{m}{2}u|_2^2\big)<\infty,\\[1mm]
|(\tilde{\mathcal{B}}_2)_r|_1&\leq C_0\big|r^{m-1}\big(\rho |u|^{p-1}|u_r|,\rho^\gamma|u|^{p-1},\rho |u|^{p+1}\big)\big|_1\\
&\quad+ C(p)\big|r^m\big(\rho_r |u|^{p-1}u_r, \rho |u|^{p-2}u_r^2, \rho |u|^{p-1} u_{rr}\big)\big|_1\\
&\quad+C(p)\big|r^m\big(\rho^{\gamma-1}\rho_r |u|^{p-1}, \rho^\gamma |u|^{p-2}u_r, \rho_r |u|^{p+1}, \rho |u|^{p}u_r\big)\big|_1\\
&\leq C_0|\rho|_\infty |u|_\infty^{p-2}\big|r^\frac{m-2}{2}u\big|_2|r^\frac{m}{2}u_r|_2+ C_0|\rho|_\infty^{\gamma} |u|_\infty^{p-3}|r^\frac{m}{2}u|_2\big|r^\frac{m-2}{2}u\big|_2 \\
&\quad +  C_0|\rho|_\infty |u|_\infty^{p-1}|r^\frac{m}{2}u|_2\big|r^\frac{m-2}{2}u\big|_2\\
&\quad+ C(p)|u|_\infty^{p-2}\big(|\rho_r|_\infty|r^\frac{m}{2}u|_2|r^\frac{m}{2}u_r|_2\!+\! |\rho|_\infty |r^\frac{m}{2}u_r|_2^2\!+\!|\rho|_\infty|r^\frac{m}{2}u|_2|r^\frac{m}{2}u_{rr}|_2\big) \\
&\quad+C(p)\big(|\rho|_\infty^{\gamma-1}|u|_\infty^{p-2} \big|r^\frac{m}{2}\rho_r\big|_2|r^\frac{m}{2}u|_2+|\rho|_\infty^{\gamma} |u|_\infty^{p-3}|r^\frac{m}{2}u|_2|r^\frac{m}{2}u_r|_2\big) \\
&\quad +C(p)\big(|r^\frac{m}{2}\rho_r|_2|u|^{p}_\infty|r^\frac{m}{2}u|_2+|\rho|_\infty|u|_\infty^{p-1}|r^\frac{m}{2} u|_2|r^\frac{m}{2} u_r|_2\big)<\infty.
\end{align*}

Thus, integrating \eqref{eq:522-po} over $I$, together with \eqref{eq:523-po}, leads to
\begin{equation}\label{add-po} 
\begin{aligned}
&\frac{1}{p} \frac{\mathrm{d}}{\mathrm{d}t}\big|(r^m\rho)^{\frac{1}{p}} u\big|_{p}^{p}  +2\alpha(p-1) \big|(r^m\rho)^{\frac12} |u|^{\frac{p-2}{2}} u_r\big|_2^2 +2\alpha m\big|(r^{m-2}\rho)^{\frac{1}{p}}u\big|_{p}^{p}\\
&=  (p-1)A\int_0^\infty r^m \rho^{\gamma} |u|^{p-2} u_r\,\mathrm{d}r+mA\int_0^\infty r^{m-1}\rho^{\gamma} |u|^{p-2} u\,\mathrm{d}r:=\sum_{i=1}^2 \mathcal{I}_i.
\end{aligned}
\end{equation}

\smallskip
\textbf{2. Estimate of $\mathcal{I}_1$.} For $\mathcal{I}_1$, it follows from the H\"older and Young inequalities that
\begin{equation}\label{I1-zong}
\begin{aligned}
\mathcal{I}_1&\leq C(p)\big|(r^m\rho)^{\frac12}|u|^{\frac{p-2}{2}} u_r\big|_2\big|r^{\frac{m}{2}}\rho^{\gamma-\frac12}|u|^{\frac{p-2}{2}}\big|_2\\
&\leq \frac{\alpha}{8}\big|(r^m\rho)^{\frac12} u^{\frac{p-2}{2}} u_r\big|_2^2+C(p)\underline{\big|\chi_1^\flat r^{\frac{m}{2}}\rho^{\gamma-\frac12}|u|^{\frac{p-2}{2}}\big|_2^2}_{:=\mathcal{I}_1^\flat}\\
&\quad +C(p)\underline{\big|\chi_1^\sharp r^{\frac{m}{2}}\rho^{\gamma-\frac12}|u|^{\frac{p-2}{2}}\big|_2^2}_{:=\mathcal{I}_1^\sharp}.
\end{aligned}
\end{equation}

The calculation of $\mathcal{I}_1^\flat$ is the same as that of Step 2 in the proof of Lemma \ref{lma}. Indeed, for any $p\in [4,\infty)$, it follows from the H\"older and Young inequalities that
\begin{equation*}
\begin{aligned}
\mathcal{I}_1^\flat&\leq \big|\chi_1^\flat r^\frac{p+m-2}{p\gamma-p+1} \rho\big|_{p\gamma-p+1}^\frac{2p\gamma-2p+2}{p} \big|(r^{m-2}\rho)^\frac{1}{p}u\big|_{p}^{p-2}\\
&\leq \frac{\alpha}{8}\big|(r^{m-2}\rho)^{\frac1p}u\big|_p^p+C(p)\underline{\big|\chi_1^\flat r^\frac{p+m-2}{p\gamma-p+1} \rho\big|_{p\gamma-p+1}^{p\gamma-p+1}}_{:=\mathcal{I}_{1,1}^\flat}.  
\end{aligned}
\end{equation*}
Notice that the term $\mathcal{I}_{1,1}^\flat$ is nothing but $\mathcal{G}_{1,1}$ 
in Step 2 of the proof of Lemma \ref{lma}. 
Thus, it follows from Lemma \ref{l4.3-po} and Step 2 in the proof of Lemma \ref{lma} that, for all $p\in[4,\infty)$ and $\epsilon\in (0,1)$,
\begin{equation}\label{g1''-po}
\mathcal{I}_1^\flat\leq C(p,\epsilon)+\epsilon \big|(r^m\rho^{\gamma})^\frac{1}{p}v\big|_p^p+\epsilon\big|(r^{m-2}\rho)^{\frac1p}u\big|_p^p.    
\end{equation}

For the estimate of $\mathcal{I}_1^\sharp$, we employ Lemma \ref{Lemma-12.4} to obtain
\begin{equation}\label{I1,1-sharp}
\mathcal{I}_1^\sharp\leq  |\chi_1^\sharp \rho|_{\infty}^{2\gamma-2} \big|(r^m\rho)^\frac{1}{p-2}u\big|_{p-2}^{p-2}\leq C_0\big|(r^m\rho)^\frac{1}{p-2}u\big|_{p-2}^{p-2}.
\end{equation}

Thus, collecting \eqref{I1-zong}--\eqref{I1,1-sharp}, we obtain that, 
for all $p\in [4,\infty)$ and $\epsilon\in (0,1)$,
\begin{equation}\label{I1-zongzong}
\begin{aligned}
\mathcal{I}_1&\leq \frac{\alpha}{8}\big|(r^m\rho)^{\frac12} u^{\frac{p-2}{2}} u_r\big|_2^2+C(p)\epsilon\big|(r^{m-2}\rho)^{\frac1p}u\big|_p^p +C(p)\big|(r^m\rho)^\frac{1}{p-2}u\big|_{p-2}^{p-2}\\
&\quad +C(p,\epsilon)+ C(p)\epsilon \big|(r^m\rho^{\gamma})^\frac{1}{p}v\big|_p^p.
\end{aligned}
\end{equation}

\smallskip
\textbf{3. Estimate of $\mathcal{I}_2$.}
Now we deal with  $\mathcal{I}_{2}$. It follows from \eqref{g1''-po}--\eqref{I1,1-sharp} and the H\"older inequality that, for all $\epsilon\in (0,1)$,
\begin{equation} \label{J2-po}
\begin{aligned}
\mathcal{I}_2&\leq C(p)\big|(r^{m-2}\rho)^{\frac1p}u\big|_p^\frac{p}{2}\big|r^{\frac{m}{2}}\rho^{\gamma-\frac12}|u|^{\frac{p-2}{2}}\big|_2\leq \frac{\alpha}{8}\big|(r^{m-2}\rho)^{\frac1p}u\big|_p^p+C(p)(\mathcal{I}_1^\flat+\mathcal{I}_1^\sharp)\\
&\leq \big(\frac{\alpha}{8} +C(p)\epsilon\big)\big|(r^{m-2}\rho)^{\frac1p}u\big|_p^p+C(p)\big|(r^m\rho)^\frac{1}{p-2}u\big|_{p-2}^{p-2}\\
&\quad +C(p,\epsilon)+C(p)\epsilon \big|(r^m\rho^{\gamma})^\frac{1}{p}v\big|_p^p.
\end{aligned}
\end{equation}

Substituting \eqref{I1-zongzong}--\eqref{J2-po} into \eqref{add-po}, then choosing $\epsilon$ small enough (for example, set $\epsilon=C(p)^{-1}\tilde\epsilon$ with $\tilde\epsilon\in (0,1)$), we can arrive at the desired conclusion of this lemma.
\end{proof}

In addition, we  can show the following $L^p(I)$-estimates of $(r^m\rho)^{\frac{1}{p}}v$
for $p\in [4,\infty)$: 
\begin{lem}\label{lem-v-lp-po}
Let $\gamma\in(1,\infty)$ if $n=2$ and $\gamma\in (1,3)$ if $n=3$. 
Then, for any  $p\in[4,\infty)$, there exists a constant $C(p)>0$ such that, for any $t\in [0,T]$,
\begin{equation}\label{dt-v-p-po}
\frac{\mathrm{d}}{\mathrm{d}t}\big|(r^m\rho)^\frac{1}{p} v\big|_{p}^{p}+\frac{A\gamma p}{4\alpha}\big|(r^m\rho^\gamma)^\frac{1}{p} v\big|_{p}^{p}\leq C(p)\Big(\big|(r^m\rho)^\frac{1}{p} u\big|_{p}^{p}+\big|(r^{m-2}\rho)^\frac{1}{p} u\big|_{p}^{p}\Big).
\end{equation}
\end{lem}

\begin{proof}
First, multiplying \eqref{eq:effective3} by $r^m\rho|v|^{p-2}v$ with $p\in[4,\infty)$, along with $\eqref{e1.5hh}_1$, gives
\begin{equation}\label{eq:532-po}
\frac{1}{p}\big(r^m\rho |v|^p\big)_t+ \frac{1}{p}\big(\underline{r^m\rho u|v|^p}_{:=\tilde{\mathcal{B}}_3}\big)_r+\frac{A\gamma}{2\alpha} r^m\rho^{\gamma} |v|^p=\frac{A\gamma}{2\alpha}r^m\rho^\gamma uv|v|^{p-2}.
\end{equation}

Next, we need to show  that $\tilde{\mathcal{B}}_3\in W^{1,1}(I)$ and $\tilde{\mathcal{B}}_3|_{r=0}=0$ for {\it a.e.} $t\in (0,T)$, which allows us to apply Lemma \ref{calculus} to obtain
\begin{equation}\label{eq:B3-po}
\int_0^\infty (\tilde{\mathcal{B}}_3)_r\,\mathrm{d}r=-\tilde{\mathcal{B}}_3|_{r=0}=0.
\end{equation}
On one hand,  we obtain $\mathcal{B}_3|_{r=0}=0$ from \eqref{V-expression3} 
and the fact that $(\rho,\rho_r,u)\in C(\bar I)$ for {\it a.e.} $t\in (0,T)$ 
due to \eqref{spd-po}--\eqref{spd2-po} (or \eqref{spd33-po}--\eqref{spd233-po}). 
On the other hand,  based on \eqref{spd-po}--\eqref{spd2-po} (or \eqref{spd33-po}--\eqref{spd233-po}), 
we have
\begin{equation*}
r^\frac{m}{2}\Big(\rho_r,\frac{\rho_r}{r},\rho_{rr},u\Big)\in L^2(I), \quad\, \Big(\rho,\rho^{-1},\rho_r,u,\frac{u}{r},u_r\Big)\in L^\infty(I)
\qquad\,\,\mbox{for {\it a.e.} $t\in (0,T)$}.
\end{equation*}
Then we obtain from the H\"older inequality that 
\begin{align*}
|\tilde{\mathcal{B}}_3|_1&\leq C(p)\big(\big|r^m\rho |u|^{p+1}\big|_1+\big|r^m\rho u|(\log\rho)_r|^p\big|_1\big)\\[1mm]
&\leq C(p)\big(|\rho|_\infty |u|_\infty^{p-1}|r^\frac{m}{2}u|_2^2+ |\rho^{1-p}|_\infty |u|_\infty|\rho_r|_\infty^{p-2}|r^\frac{m}{2}\rho_r|_2^2\big)<\infty,\\
|(\tilde{\mathcal{B}}_3)_r|_1&\leq C(p)\big(\big|r^{m-1}\rho |u|^{p+1}\big|_1+\big|r^{m-1}\rho u|(\log\rho)_r|^p\big|_1+\big|r^{m}\rho_r |u|^{p+1}\big|_1\big)\\
&\quad+C(p)\big(\big|r^{m}\rho_r u|(\log\rho)_r|^p\big|_1+\big|r^{m}\rho |u|^pu_r\big|_1+\big|r^{m}\rho u_r|(\log\rho)_r|^p\big|_1\big)\\
&\quad +C(p)\big|r^{m}\rho u|(\log\rho)_r|^{p-1}|(\log\rho)_{rr}|\big|_1\\
&\leq C(p)\big(|\rho|_\infty|u|_\infty^{p-1}|r^\frac{m}{2}u|_2 \big|r^\frac{m-2}{2}u\big|_2\!+\!|\rho^{1-p}|_\infty |u|_\infty |\rho_r|_\infty^{p-2}|r^\frac{m}{2}\rho_r|_2\big|r^\frac{m-2}{2}\rho_r\big|_2\big) \\
&\quad +C(p)|r^\frac{m}{2}\rho_r|_2|r^\frac{m}{2}u|_2|u|_\infty^p+C(p)|r^\frac{m}{2}\rho_r|_2|r^\frac{m}{2}u|_2|\rho^{-p}|_\infty |\rho_r|_\infty^p\\
&\quad +C(p)|\rho|_\infty|r^\frac{m}{2}u_r|_2\big(|u|_\infty^{p-1}|r^\frac{m}{2}u|_2+|\rho^{-p}|_\infty |\rho_r|_\infty^{p-1}|r^\frac{m}{2}\rho_r|_2\big) \\
&\quad +C(p)|r^\frac{m}{2}u|_2|\rho^{2-p}|_\infty|\rho_r|_\infty^{p-1}\big(|\rho^{-2}|_\infty|\rho_r|_\infty |r^\frac{m}{2}\rho_r|_2\!+\!|\rho^{-1}|_\infty|r^\frac{m}{2}\rho_{rr}|_2\big)<\infty.
\end{align*}

Integrating \eqref{eq:532-po} over $I$, then we obtain from \eqref{eq:B3-po}, the fact that $\frac{2}{\gamma-1}>1$ whenever $\gamma\in (1,3)$, Lemmas \ref{l4.3-po}--\ref{Lemma-12.4}, and the Young inequality that
\begin{equation}\label{4-13-po}
\begin{aligned}
&\frac{1}{p}\frac{\mathrm{d}}{\mathrm{d}t}\big|(r^m\rho)^\frac{1}{p} v\big|_{p}^{p}+\frac{A\gamma}{2\alpha}\big|(r^m\rho^\gamma)^\frac{1}{p} v\big|_{p}^{p}\\
&\leq \frac{A\gamma}{2\alpha}\int_0^\infty r^m\rho^\gamma uv|v|^{p-2}\,\mathrm{d}r \leq \frac{A\gamma}{4\alpha}\big|(r^m\rho^\gamma)^\frac{1}{p} v\big|_{p}^p+C(p)\big|(r^m\rho^\gamma)^\frac{1}{p} u\big|_{p}^p\\
&\leq \frac{A\gamma}{4\alpha}\big|(r^m\rho^\gamma)^\frac{1}{p} v\big|_{p}^{p}+C(p)\Big(\big|\chi_1^\flat(r^m\rho^\gamma)^\frac{1}{p} u\big|_{p}^p+\big|\chi_1^\sharp(r^m\rho^\gamma)^\frac{1}{p} u\big|_{p}^p\Big)\\
&\leq \frac{A\gamma}{4\alpha}\big|(r^m\rho^\gamma)^\frac{1}{p} v\big|_{p}^{p}+C(p)\big|\chi_1^\flat r^\frac{2}{\gamma-1}\rho\big|_\infty^{\gamma-1} \big|(r^{m-2}\rho)^\frac{1}{p} u\big|_{p}^{p}+C(p) |\chi_1^\sharp\rho|_\infty^{\gamma-1} \big|(r^m\rho)^\frac{1}{p} u\big|_{p}^{p}\\
&\leq \frac{A\gamma}{4\alpha}\big|(r^m\rho^\gamma)^\frac{1}{p} v\big|_{p}^{p} +C(p)\Big(\big|(r^m\rho)^\frac{1}{p} u\big|_{p}^{p}+\big|(r^{m-2}\rho)^\frac{1}{p} u\big|_{p}^{p}\Big),
\end{aligned}
\end{equation}
which implies the desired estimates of this lemma.
\end{proof}

Based on Lemmas \ref{lma-po}--\ref{lem-v-lp-po}, we  can derive the following important estimates on $(u,v)$:
\begin{lem}\label{lemma-uv-p-po}
Let $\gamma\in(1,\infty)$ if $n=2$ and $\gamma\in (1,3)$ if $n=3$. Then, for  any $p\in[2,\infty)$,
there exists a constant $C(p,T)>0$ such that, for any $t\in [0,T]$,
\begin{equation*}
\begin{aligned} 
&\big|(r^m\rho)^{\frac{1}{p}}u(t)\big|_{p}^{p}+\big|(r^m\rho)^{\frac{1}{p}}v(t)\big|_{p}^{p}\\
&+ \int_0^t\Big(\big|(r^m\rho)^{\frac{1}{2}}|u|^{\frac{p-2}{2}} u_r\big|_2^2+\big|(r^{m-2}\rho)^{\frac{1}{p}}u\big|_{p}^{p}+\big|(r^m\rho^\gamma)^\frac{1}{p} v\big|_{p}^{p}\Big)\,\mathrm{d}s\leq C(p,T).
\end{aligned}
\end{equation*}
\end{lem}
\begin{proof}
First, multiplying \eqref{dt-v-p-po} by $\frac{8\alpha\epsilon}{A\gamma p}$ with $\epsilon\in (0,1)$ leads to
\begin{equation*} 
\frac{8\alpha\epsilon}{A\gamma p}\frac{\mathrm{d}}{\mathrm{d}t}\big|(r^m\rho)^\frac{1}{p} v\big|_{p}^{p}+2\epsilon\big|(r^m\rho^\gamma)^\frac{1}{p} v\big|_{p}^{p}\leq C(p)\epsilon\Big(\big|(r^m\rho)^\frac{1}{p} u\big|_{p}^{p}+\big|(r^{m-2}\rho)^\frac{1}{p} u\big|_{p}^{p}\Big).
\end{equation*}
Then summing the above inequality with \eqref{dt-u-p-po} yields that 
\begin{equation*}
\begin{aligned}
&\frac{\mathrm{d}}{\mathrm{d}t}\Big(\big|(r^m\rho)^{\frac{1}{p}} u\big|_{p}^{p}+\frac{8\alpha\epsilon}{A\gamma p} \big|(r^m\rho)^\frac{1}{p} v\big|_{p}^{p}\Big)\\
&+ p\alpha \Big(\big|(r^m\rho)^{\frac{1}{2}}|u|^{\frac{p-2}{2}} u_r\big|_2^2+\big|(r^{m-2}\rho)^{\frac{1}{p}}u\big|_{p}^{p}\Big)+\epsilon\big|(r^m\rho^\gamma)^\frac{1}{p} v\big|_{p}^{p}\\
& \leq C(p)\Big(\big|(r^m\rho)^\frac{1}{p}u\big|_{p}^{p}+\big|(r^m\rho)^\frac{1}{p-2}u\big|_{p-2}^{p-2}\Big)+ C(p,\epsilon)+C(p)\epsilon\big|(r^{m-2}\rho)^\frac{1}{p} u\big|_{p}^{p}.
\end{aligned}
\end{equation*}
As a consequence, setting
\begin{equation*}
\epsilon=\min\big\{\frac{\alpha}{C(p)},\frac{1}{2}\big\},
\end{equation*}
and then applying the Gr\"onwall inequality imply that, for all $t\in [0,T]$ and $p\in [4,\infty)$,
\begin{equation}\label{p-dt}
\begin{aligned}
&\sup_{t\in [0,T]}\Big(\big|(r^m\rho)^{\frac{1}{p}} u\big|_{p}^{p}+\big|(r^m\rho)^\frac{1}{p} v\big|_{p}^{p}\Big)\\
&+ \int_0^T \Big(\big|(r^m\rho)^{\frac{1}{2}}|u|^{\frac{p-2}{2}} u_r\big|_2^2+\big|(r^{m-2}\rho)^{\frac{1}{p}}u\big|_{p}^{p}+\big|(r^m\rho^\gamma)^\frac{1}{p} v\big|_{p}^{p}\Big)\,\mathrm{d}t\\
&\leq C(p,T)\sup_{t\in [0,T]}\big|(r^m\rho)^\frac{1}{p-2}u\big|_{p-2}^{p-2}+C(p,T).
\end{aligned}
\end{equation}
Here, we still need to check the $L^p(I)$-boundedness of $(r^m\rho_0)^{\frac{1}{p}}(u_0,v_0)$. 
Indeed, it follows from Lemmas \ref{equiv-initial}, \ref{ale1}, and \ref{lemma-initial} that
\begin{align*}
\big|(r^m\rho_0)^\frac{1}{p}(u_0,v_0)\big|_p&\leq |\rho_0|_\infty^\frac{1}{p}\big|r^\frac{m}{p}(u_0,(\log\rho_0)_r)\big|_{p}\leq C_0\|\rho_0\|_{L^\infty}^\frac{1}{p}\|(\boldsymbol{u}_0,\nabla\log\rho_0)\|_{L^p}\\
&\leq C_0\big(\|(\rho_0-\bar\rho)\|_{H^2}+1\big)^\frac{1}{p}\|(\boldsymbol{u}_0,\nabla\log\rho_0)\|_{H^1}\leq C(p).
\end{align*}

Next, setting $p=4$ in \eqref{p-dt}, together with Lemma \ref{energy-bd-po}, yields
\begin{equation*}
\begin{aligned}
&\sup_{t\in [0,T]}\Big(\big|(r^m\rho)^{\frac{1}{4}} u\big|_{4}^{4}+\big|(r^m\rho)^\frac{1}{4} v\big|_{4}^{4}\Big)\\
&+ \int_0^T \Big(\big|(r^m\rho)^{\frac{1}{2}}|u| u_r\big|_2^2+\big|(r^{m-2}\rho)^{\frac{1}{4}}u\big|_{4}^{4}+\big|(r^m\rho^\gamma)^\frac{1}{4} v\big|_{4}^{4}\Big)\,\mathrm{d}t\leq C(T).
\end{aligned}
\end{equation*}
Hence, taking $p=2N$ with $N\in \mathbb{N}$ and $N\geq 3$ in \eqref{p-dt}, 
we can iteratively obtain from the above 
and Lemma \ref{energy-bd-po} that, for all $N\in \mathbb{N}^*$,
\begin{equation*}
\begin{aligned}
&\sup_{t\in [0,T]}\Big(\big|(r^m\rho)^{\frac{1}{2N}} u\big|_{2N}^{2N}+\big|(r^m\rho)^\frac{1}{2N} v\big|_{2N}^{2N}\Big)\\
&+\int_0^T \Big(\big|(r^m\rho)^{\frac{1}{2}}|u|^{N-1} u_r\big|_2^2+\big|(r^{m-2}\rho)^{\frac{1}{2N}}u\big|_{2N}^{2N}+\big|(r^m\rho^\gamma)^\frac{1}{2N} v\big|_{2N}^{2N}\Big)\,\mathrm{d}t\leq C(N,T),
\end{aligned}
\end{equation*}
which, along with the interpolation, leads to the desired result. 
\end{proof}

Now, based on Lemmas \ref{Lemma-12.4} and \ref{lemma-uv-p-po}, using the same argument as in the proof of Lemma \ref{important2}, we can obtain the global uniform upper bound of $\rho$ in $[0,T]\times I$. 
\begin{lem}\label{important2-po}
Let $\gamma\in (1,\infty)$ if $n=2$ and $\gamma\in (1,3)$ if $n=3$. Then there exists a constant $C(T)>0$ such that
\begin{equation*}
|\rho(t)|_\infty\leq C(T)  \qquad \text{for any $t\in [0,T]$}.
\end{equation*}
\end{lem}

\subsection{Global uniform $L^\infty(I)$-estimate of $v$}\label{subsub-1133} 
In  \S\ref{subsub-1133}, we always assume that $\gamma\in (1,\infty)$ if $n=2$ and $\gamma\in (1,3)$ if $n=3$.
This subsection is devoted to providing the uniform boundedness of $v$. 
To get this, we  first derive the $L^p(I)$-estimates  of $\rho^{\frac{1}{p}}u$.
\begin{lem}\label{rho u-L2-po}
For any  $p\in[4,\infty)$,
there exists a constant $C(p,T)>0$ such that, for any $t\in [0,T]$,
\begin{equation*}
\big|\rho^\frac{1}{p}u(t)\big|_p^p+ \int_0^t\Big(\big|\rho^\frac{1}{2}|u|^\frac{p-2}{2} u_r\big|_2^2+\big|(r^{-2}\rho)^\frac{1}{p}u\big|_p^p\Big)\,\mathrm{d}s \leq C(p,T)\Big(\big(\sup_{s\in[0,t]}|v|_\infty\big)^2+1\Big).
\end{equation*}
\end{lem}
\begin{proof}
First, multiplying both sides of $\eqref{e1.5hh}_2$ by $|u|^{p-2}u$ with $p\in [4,\infty)$, together with $\eqref{e1.5hh}_1$ and \eqref{V-expression3}, leads to
\begin{equation}\label{eq:6.3-po}
\begin{aligned}
&\frac{1}{p}\big(\rho |u|^p\big)_t+2\alpha(p-1)\rho |u|^{p-2} |u_r|^2+\frac{2\alpha m(p-1)}{p}\frac{\rho |u|^p}{r^2}-A(p-1)\rho^\gamma|u|^{p-2}u_r\\
&=\Big(2\alpha\rho u_r|u|^{p-2}u+\frac{2\alpha m}{p}\frac{\rho |u|^p}{r}-\frac{1}{p}\rho u |u|^p-A\rho^\gamma|u|^{p-2}u\Big)_r-\frac{m}{p}\frac{\rho v |u|^p}{r}.
\end{aligned}
\end{equation}

Next, we need to show  that $\tilde{\mathcal{B}}_4\in W^{1,1}(I)$ and $\tilde{\mathcal{B}}_4|_{r=0}=0$ for {\it a.e.} $t\in (0,T)$, which allows us to apply Lemma \ref{calculus} to obtain
\begin{equation}\label{eq:B4-po}
\int_0^\infty (\tilde{\mathcal{B}}_4)_r\,\mathrm{d}r=-\tilde{\mathcal{B}}_4|_{r=0}=0.
\end{equation}
On one hand,  $\tilde{\mathcal{B}}_4|_{r=0}=0$ follows directly from $p\geq 4$, $u|_{r=0}=0$ and the fact that $(\rho,u,\frac{u}{r},u_r)\in C(\bar I)$ for each $t\in (0,T]$ due to \eqref{spd2-po} (or \eqref{spd233-po}). 

On the other hand,  when $n=2$ $(m=1)$, \eqref{spd-po}--\eqref{spd2-po} (or \eqref{spd33-po}--\eqref{spd233-po}) imply that 
\begin{equation*}
 \sqrt{r}\Big(u,\frac{u}{r},u_r,u_{rr}\Big)\in L^2(I), \quad\,\,
\Big(\rho,\rho^{-1},\rho_r,u,\frac{u}{r},u_r\Big)\in L^\infty(I)
\qquad\,\, \mbox{for {\it a.e.} $t\in (0,T)$},
\end{equation*}
which, along with the H\"older inequality, yields 
\begin{align*}
|\tilde{\mathcal{B}}_4|_1&\leq C(p)\Big(\big|\rho |u|^{p-1}|u_r|\big|_1+\Big|\frac{1}{r}\rho |u|^p\Big|_1+\big|\rho |u|^{p+1}\big|_1+\big|\rho^\gamma|u|^{p-1}\big|_1\Big)\\
&\leq C(p)|\rho|_\infty |u|_\infty^{p-2}\Big(\Big|\frac{u}{\sqrt{r}}\Big|_2|\sqrt{r}u_r|_2+\Big|\frac{u}{\sqrt{r}}\Big|_2^2\Big)\\
&\quad+ C(p)|\rho|_\infty|u|_\infty^{p-1}|\sqrt{r}u|_2\Big|\frac{u}{\sqrt{r}}\Big|_2+C(p)|\rho|_\infty^{\gamma}|u|_\infty^{p-2}|\sqrt{r}u|_2\Big|\frac{u}{\sqrt{r}}\Big|_2<\infty,\\
|(\tilde{\mathcal{B}}_4)_r|_1&\leq C(p)\big|\big(|\rho_r||u_r||u|^{p-1},\rho |u_{rr}||u|^{p-1},\rho u_r^2|u|^{p-2}\big)\big|_1\\
&\quad +C(p)\Big|\Big(\frac{1}{r^2}\rho |u|^p,\frac{1}{r}\rho_r |u|^p,\frac{1}{r}\rho |u|^{p-1}|u_r|\Big)\Big|_1\!\!+\!C(p)\big|\big(|\rho_r||u|^{p+1},\rho |u|^p |u_r|\big)\big|_1\\
&\quad + C(p)\big|\big(\rho^{\gamma-1}|\rho_r||u|^{p-1},\rho^{\gamma} |u|^{p-2} |u_r|\big)\big|_1\\
&\leq C(p)|\rho|_\infty|u|_\infty^{p-2} \Big|\frac{u}{\sqrt{r}}\Big|_2\big(|\rho^{-1}|_\infty|\rho_r|_\infty|\sqrt{r}u_r|_2+|\sqrt{r}u_{rr}|_2\big)\\
&\quad +C(p)|\rho|_\infty|u|_\infty^{p-4}|\sqrt{r}u|_2\Big|\frac{u}{\sqrt{r}}\Big|_2\Big(|u_r|_\infty^2+\Big|\frac{u}{r}\Big|_\infty^2\Big)\\
&\quad+ C(p)|u|_\infty^{p-4}\Big|\frac{u}{r}\Big|_\infty|\sqrt{r}u|_2\Big|\frac{u}{\sqrt{r}}\Big|_2\big(|u|_\infty|\rho_r|_\infty+|\rho|_\infty |u_r|_\infty\big)\\
&\quad+C(p) |u|_\infty^{p-2}|\sqrt{r}u|_2\Big|\frac{u}{\sqrt{r}}\Big|_2\big(|u|_\infty |\rho_r|_\infty+|\rho|_\infty |u_r|_\infty\big)\\
&\quad+C(p)|\rho|_\infty^{\gamma-1} |u|_\infty^{p-4}|\sqrt{r}u|_2\Big|\frac{u}{\sqrt{r}}\Big|_2\big(|u|_\infty |\rho_r|_\infty+|\rho|_\infty |u_r|_\infty\big)<\infty.
\end{align*}

When $n=3$ $(m=2)$, by \eqref{spd-po}--\eqref{spd2-po} (or \eqref{spd33-po}--\eqref{spd233-po}), we have 
\begin{equation*}
(u,ru_r,ru_{rr})\in L^2(I), \quad\,\, \Big(\rho,\rho^{-1},\rho_r,u,\frac{u}{r},u_r\Big)\in L^\infty(I)\qquad\,\,\text{for {\it a.e.} $t\in (0,T)$},
\end{equation*}
which yields from the H\"older inequality that
\begin{align*}
|\tilde{\mathcal{B}}_4|_1&\leq C(p)\big(\big|\rho |u|^{p-1}|u_r|\big|_1+\Big|\frac{1}{r}\rho |u|^p\Big|_1+\big|\rho |u|^{p+1}\big|_1+\big|\rho^\gamma|u|^{p-1}\big|_1\big)\\
&\leq C(p)|\rho|_\infty |u|_\infty^{p-3}|u|_2^2\Big(\Big|\big(u_r,\frac{u}{r}\big)\Big|_\infty+|u|_\infty^2+|\rho|_\infty^{\gamma-1}\Big)<\infty,\\[1mm]
|(\tilde{\mathcal{B}}_4)_r|_1&\leq C(p)\big|\big(|\rho_r||u_r||u|^{p-1},\rho |u_{rr}||u|^{p-1},\rho u_r^2|u|^{p-2}\big)\big|_1\\
&\quad +C(p)\Big|\Big(\frac{1}{r^2}\rho |u|^p,\frac{1}{r}\rho_r |u|^p,\frac{1}{r}\rho |u|^{p-1}|u_r|\Big)\Big|_1\!\!+\!C(p)\big|\big(|\rho_r||u|^{p+1},\rho |u|^p |u_r|\big)\big|_1\\
&\quad + C(p)\big|\big(\rho^{\gamma-1}|\rho_r||u|^{p-1},\rho^{\gamma} |u|^{p-2} |u_r|\big)\big|_1\\
&\leq C(p)\Big|\frac{u}{r}\Big|_\infty\!|u|_\infty^{p-4}|u|_2\big(|\rho_r|_\infty|u|_\infty|ru_r|_2 \! + \!|\rho|_\infty|u|_\infty |ru_{rr}|_2\!+\!|\rho|_\infty|u_r|_\infty|ru_r|_2\big)\\
&\quad+C(p)|u|_\infty^{p-4}\Big|\frac{u}{r}\Big|_\infty|u|_2^2\Big(\Big|\frac{u}{r}\Big|_\infty|\rho|_\infty+|\rho_r|_\infty|u|_\infty+|\rho|_\infty|u_r|_\infty\Big)\\
&\quad +C(p)|u|_\infty^{p-2}|u|_2\Big(|\rho_r|_\infty|u|_\infty|u|_2+|\rho|_\infty\Big|\frac{u}{r}\Big|_\infty|ru_r|_2\Big)\\
&\quad +C(p)|\rho|_\infty^{\gamma-1}|u|_\infty^{p-4}|u|_2^2 \big(|\rho_r|_\infty|u|_\infty +|\rho|_\infty|u_r|_\infty\big) <\infty.
\end{align*}

As a consequence, integrating \eqref{eq:6.3-po} over $I$, we obtain from \eqref{eq:B4-po} that, for $p\in[4,\infty)$,
\begin{equation}\label{503-po}
\begin{aligned}
&\frac{1}{p}\frac{\mathrm{d}}{\mathrm{d}t}\big|\rho^\frac{1}{p}u\big|_p^p+ 2\alpha (p-1)\big|\rho^\frac{1}{2} |u|^\frac{p-2}{2} u_r\big|_2^2+\frac{2\alpha m(p-1)}{p}\big|(r^{-2}\rho)^\frac{1}{p}u\big|_p^p\\
&=A(p-1)\int_0^\infty \rho^\gamma |u|^{p-2} u_r\,\mathrm{d}r-\frac{m}{p}\int_0^\infty \frac{\rho v|u|^p}{r}\,\mathrm{d}r:=\sum_{i=3}^4\mathcal{I}_i.
\end{aligned}    
\end{equation}
For $\mathcal{I}_3$, it follows from Lemma \ref{important2-po} and the H\"older and Young inequality that
\begin{equation}
\begin{aligned}
\mathcal{I}_3&\leq C(p)\big|\chi_1^\flat\rho^\gamma |u|^{p-2} u_r\big|_1+C(p)\big|\chi_1^\sharp\rho^\gamma |u|^{p-2} u_r\big|_1\\
&\leq  C(p)|\rho|_{\infty}^{\gamma-1}\big|(r^m\rho)^\frac{1}{2}|u|^\frac{p-4}{2}u_r\big|_2\Big( |\chi_1^\flat r^\frac{2-m}{2}|_\infty\big|(r^{-2}\rho)^\frac{1}{p}u\big|_{p}^\frac{p}{2}+ |\chi_1^\sharp r^{-\frac{m}{2}}|_\infty\big|\rho^\frac{1}{p}u\big|_{p}^\frac{p}{2}\Big)\\
&\leq C(p,T)\Big(\big|(r^m\rho)^\frac{1}{2}|u|^\frac{p-4}{2}u_r\big|_2^2+\big|\rho^\frac{1}{p}u\big|_{p}^p\Big)+\frac{\alpha}{8}\big|(r^{-2}\rho)^\frac{1}{p}u\big|_{p}^p.
\end{aligned}
\end{equation}
For $\mathcal{I}_4$, we obtain from the same calculation of $\mathcal{G}_3$ in the proof 
of Lemma \ref{rho u-L2} that
\begin{equation}\label{R2-dim2-po}
\mathcal{I}_4\leq C(p)(|v|_\infty^2+1) \big|(r^{m-2}\rho)^\frac{1}{p}u\big|_p^p+\frac{\alpha}{8}\big|(r^{-2}\rho)^\frac{1}{p}u\big|_p^p.
\end{equation}

Thus, collecting \eqref{503-po}--\eqref{R2-dim2-po} gives
\begin{equation}\label{711-po}
\begin{aligned}
&\frac{\mathrm{d}}{\mathrm{d}t}\big|\rho^\frac{1}{p}u\big|_p^p+ \big|\rho^\frac{1}{2}|u|^\frac{p-2}{2} u_r\big|_2^2+\big|(r^{-2}\rho)^\frac{1}{p}u\big|_p^p\\
&\leq C(p,T)\Big(\big|(r^m\rho)^\frac{1}{2}|u|^\frac{p-4}{2}u_r\big|_2^2+\big|\rho^\frac{1}{p}u\big|_{p}^p\Big)+C(p)(|v|_\infty^2+1)  \big|(r^{m-2}\rho)^\frac{1}{p}u\big|_p^p,
\end{aligned}
\end{equation}
which, along with the Gr\"onwall inequality and Lemma \ref{lemma-uv-p-po}, yields that, for all $t\in[0,T]$ and $p\in[4,\infty)$,
\begin{equation*}
\big|\rho^\frac{1}{p}u(t)\big|_p^p+ \int_0^t\Big(\big|\rho^\frac{1}{2}|u|^\frac{p-2}{2} u_r\big|_2^2+\big|(r^{-2}\rho)^\frac{1}{p}u\big|_p^p\Big)\mathrm{d}s \leq C(p,T)\Big(\big(\sup_{s\in[0,t]}|v|_\infty\big)^2+1\Big).
\end{equation*}
Here, we still need to check the $L^p(I)$-boundedness of $\rho_0^{\frac{1}{p}}u_0$. Indeed, it follows from Lemmas \ref{ale1} and \ref{lemma-initial} that
\begin{align*}
\big|\rho_0^\frac{1}{p}u_0\big|_p&\leq \big|\chi_1^\flat \rho_0^\frac{1}{p}u_0\big|_p+\big|\chi_1^\sharp \rho_0^\frac{1}{p}u_0\big|_p\leq |\rho_0|_\infty^\frac{1}{p}|u_0|_\infty+ |\rho_0|_\infty^\frac{1}{p}|\chi_1^\sharp r^{-\frac{m}{p}}|_\infty \big|r^\frac{m}{p}u_0\big|_p\\
&\leq C(p)\big(\|\boldsymbol{u}_0\|_{L^\infty}+\|\boldsymbol{u}_0\|_{L^p}\big)\leq C(p)\|\boldsymbol{u}_0\|_{H^2}\leq C(p).
\end{align*}
This completes the proof.
\end{proof}

Next, based on Lemma \ref{rho u-L2-po}, we can obtain the $L^1([0,T];L^\infty(I))$-estimate of $\rho^{\gamma-1}u$.
\begin{lem}\label{cru3-po}
For any $\epsilon\in (0,1)$,
there exists a constant $C(\epsilon,T)>0$ such that
\begin{equation*}
\int_0^t |\rho^{\gamma-1}u|_{\infty}\, \mathrm{d}s\leq C(\epsilon,T)\Big(1+\int_0^t |v|_\infty \,\mathrm{d}s\Big)+\epsilon\sup_{s\in[0,t]}|v|_\infty \qquad \text{for any $\,t\in [0,T]$}.
\end{equation*}
\end{lem}

\begin{proof}
Let $q\geq 4$ be a constant to  be determined later. First, it follows from \eqref{V-expression3}, Lemma \ref{ale1}, 
and  the H\"older inequality that
\begin{equation}\label{6.10-po}
\begin{aligned} 
|\rho^{\gamma-1}u|_\infty^{q}&=\big||\rho^{\gamma-1}u|^q\big|_\infty\leq C_0\int_0^\infty \rho^{q\gamma-q} |u|^{q} \,\mathrm{d}r+C_0\int_0^\infty |(\rho^{q\gamma-q} |u|^{q})_r| \,\mathrm{d}r\\
&\leq C_0\int_0^\infty \rho^{q\gamma-q} |u|^{q}\mathrm{d}r\!+\!C(q)\!\int_0^\infty \rho^{q\gamma-q}\Big((|v|+|u|)|u|^{q}+|u|^{q-1} |u_r|\Big)\mathrm{d}r\\
&\leq C(q)(1+|v|_\infty)\big|\rho^{q\gamma-q}u^{q}\big|_1+C(q)\big|\rho^{q\gamma-q}u^{q+1}\big|_1\\
&\quad +C(q)\big|\rho^{2q\gamma-2q-1}u^{q}\big|_1^\frac{1}{2}\big|\rho^\frac{1}{2}|u|^\frac{q-2}{2}u_r\big|_2 \\
&\leq C(q)(1+|v|_\infty)|\rho|_\infty^{q\gamma-q-1}\big|\rho^{\frac{1}{q}}u\big|_{q}^{q}+C(q)|\rho^{\gamma-1}u|_\infty|\rho|_\infty^{(q-1)\gamma-q}\big|\rho^{\frac{1}{q}}u\big|_{q}^{q}\\
&\quad +C(q)|\rho|_\infty^{q\gamma-q-1}\big|\rho^{\frac{1}{q}}u\big|_{q}^\frac{q}{2}\big|\rho^\frac{1}{2}|u|^\frac{q-2}{2}u_r\big|_2. 
\end{aligned}
\end{equation}

Next, setting $q=\tilde q\geq 4$ in \eqref{6.10-po} large enough such that 
\begin{equation*}
\gamma>\frac{\tilde q}{\tilde q-1}>1,
\end{equation*}
then we obtain from the resulting inequality, Lemma \ref{important2-po}, and the Young inequality that
\begin{align*}
|\rho^{\gamma-1}u|_\infty^{\tilde q}&\leq C(T)(1+|v|_\infty)\big|\rho^\frac{1}{\tilde q}u\big|_{\tilde q}^{\tilde q}+C(T)\big|\rho^\frac{1}{\tilde q}u\big|_{\tilde q}^\frac{\tilde q^2}{\tilde q-1}\\
&\quad+C(T)\big|\rho^\frac{1}{\tilde q}u\big|_{\tilde q}^\frac{\tilde q}{2}\big|\rho^\frac{1}{2}|u|^\frac{\tilde q-2}{2}u_r\big|_2 +\frac{1}{2}|\rho^{\gamma-1}u|_\infty^{\tilde q},
\end{align*}
which, along with the fact that $\tilde q\geq 4$, Lemma \ref{lemma-uv-p-po}, 
and the Young inequality, leads to
\begin{align*}
|\rho^{\gamma-1}u|_\infty&\leq C(T)\Big(\big(1+|v|_\infty^\frac{1}{\tilde q}\big)\big|\rho^\frac{1}{\tilde q}u\big|_{\tilde q}+\big|\rho^\frac{1}{\tilde q}u\big|_{\tilde q}^\frac{\tilde q}{\tilde q-1}+\big|\rho^\frac{1}{\tilde q}u\big|_{\tilde q}^\frac{1}{2}\big|\rho^\frac{1}{2}|u|^\frac{\tilde q-2}{2}u_r\big|_2^\frac{1}{\tilde q}\Big)\\
&\leq C(T)\Big(1+|v|_\infty+\big|\rho^\frac{1}{\tilde q}u\big|_{\tilde q}^{\tilde q}+\big|\rho^\frac{1}{2}|u|^\frac{\tilde q-2}{2}u_r\big|_2^\frac{2}{2\tilde q-1}\Big)\\
&\leq C(T)\Big(1+|v|_\infty+\big|\chi_1^\flat\rho^\frac{1}{\tilde q}u\big|_{\tilde q}^{\tilde q}+\big|\chi_1^\sharp\rho^\frac{1}{\tilde q}u\big|_{\tilde q}^{\tilde q}+\big|\rho^\frac{1}{2}|u|^\frac{\tilde q-2}{2}u_r\big|_2^\frac{2}{2\tilde q-1}\Big)\\
&\leq C(T)\Big(1+|v|_\infty+|\chi_1^\flat r^{2-m}|_\infty \big|(r^{m-2}\rho)^\frac{1}{\tilde q}u\big|_{\tilde q}^{\tilde q}\Big)\\
&\quad +C(T)\Big(|\chi_1^\sharp r^{-m}| \big|(r^m\rho)^\frac{1}{\tilde q}u\big|_{\tilde q}^{\tilde q}+\big|\rho^\frac{1}{2}|u|^\frac{\tilde q-2}{2}u_r\big|_2^\frac{2}{2\tilde q-1}\Big)\\
&\leq C(T)\Big(1+|v|_\infty+\big|(r^{m-2}\rho)^\frac{1}{\tilde q}u\big|_{\tilde q}^{\tilde q}+\big|\rho^\frac{1}{2}|u|^\frac{\tilde q-2}{2}u_r\big|_2^\frac{2}{2\tilde q-1}\Big).
\end{align*}

Finally, integrating the above over $[0,t]$, we see from Lemmas \ref{lemma-uv-p-po} and \ref{rho u-L2-po}, and the H\"older and Young inequalities that, for all $\epsilon\in (0,1)$,
\begin{align*}
\int_0^t |\rho^{\gamma-1}u|_\infty\,\mathrm{d}s&\leq C(T)\int_0^t\Big(1+|v|_\infty+\big|(r^{m-2}\rho)^\frac{1}{\tilde q}u\big|_{\tilde q}^{\tilde q}+\big|\rho^\frac{1}{2}|u|^\frac{\tilde q-2}{2}u_r\big|_2^\frac{2}{2\tilde q-1}\Big)\,\mathrm{d}s\\
&\leq C(T)\Big(1+ \int_0^t |v|_\infty\,\mathrm{d}s\Big)+C(T)\Big(\int_0^t\big|\rho^\frac{1}{2}|u|^\frac{\tilde q-2}{2}u_r\big|_2^2 \,\mathrm{d}s\Big)^\frac{1}{2\tilde q-1}\\
&\leq C(T)\Big(1+ \int_0^t |v|_\infty\,\mathrm{d}s\Big)+C(T)\Big(\sup_{s\in[0,t]}|v|_\infty\Big)^\frac{2}{2\tilde q-1}\\
&\leq C(\epsilon,T)\Big(1+ \int_0^t |v|_\infty\,\mathrm{d}s\Big)+\epsilon\sup_{s\in[0,t]}|v|_\infty.
\end{align*}

The proof of Lemma \ref{cru3-po} is completed.
\end{proof}

Now, with the help of Lemma \ref{cru3-po}, we obtain the $L^\infty(I)$-estimate of $v$. The proof is identical to that of Lemma \ref{l4.4}, we omit it here for brevity.
\begin{lem}\label{l4.4-po}
There exists a constant $C(T)>0$ such that,
\begin{equation*}
|v(t)|_\infty\leq C(T) \qquad \text{for any $t\in [0,T]$}.
\end{equation*}
\end{lem}

\smallskip
\subsection{Global uniform lower bound of the density}\label{subsub-1134}

In  \S\ref{subsub-1134}, we always assume that $\gamma\in (1,\infty)$ if $n=2$ and $\gamma\in (1,3)$ if $n=3$.
By the $L^\infty(I)$-estimates of $(\rho,v)$, we are able to obtain the uniform lower bound of $\rho$ in $[0,T]\times I$. Similar to the arguments presented in \S \ref{section-nonformation}, we first show the following $L^2(I)$-estimate of $u$.

\begin{lem}\label{ele-po}
There exists a constant $C(T)>0$ such that 
\begin{equation*}
|(\rho^{\gamma-1} v,u)(t)|_2^2 +\int_0^t\Big(\Big|\big(u_r,\frac{u}{r}\big)\Big|_2^2+|u|^2_\infty\Big)\,\mathrm{d}s\leq C(T) \qquad \text{for any $t\in [0,T]$}.   
\end{equation*}
\end{lem}
\begin{proof}
The proof of this lemma is similar to that of Lemma \ref{ele}, with merely minor modifications to the justifications of integration by parts. We divide the proof into two steps.

\smallskip
\textbf{1.} First, multiplying $\eqref{e1.5hh}_2$ by $u$, along with \eqref{V-expression3}, gives
\begin{equation}\label{eq:6.14-po}
\begin{aligned}
&\frac{1}{2}(u^2)_t+2\alpha |u_r|^2+\alpha m\frac{u^2}{r^2}\\
&=\Big(\underline{2\alpha u_ru+\alpha m\frac{u^2}{r}-\frac{2}{3}u^3}_{:=\tilde{\mathcal{B}}_5}\Big)_r-\frac{A\gamma}{2\alpha}\rho^{\gamma-1}(v-u)u+vu_ru.
\end{aligned}
\end{equation}
Here we need to show  that $\tilde{\mathcal{B}}_5\in W^{1,1}(I)$ and $\tilde{\mathcal{B}}_5|_{r=0}=0$ for {\it a.e.} $t\in (0,T)$, which allows us to apply Lemma \ref{calculus} to obtain
\begin{equation}\label{eq:B5-po}
\int_0^\infty (\tilde{\mathcal{B}}_5)_r\,\mathrm{d}r=-\tilde{\mathcal{B}}_5|_{r=0}=0.
\end{equation}
Since $\tilde{\mathcal{B}}_5$ consists of $u$ only, the proof of \eqref{eq:B5-po} is the same as that of \eqref{eq:B5} in the proof for Lemma \ref{ele}.

Thus, integrating \eqref{eq:6.14-po} over $I$, together with \eqref{eq:B5-po}, yields
\begin{equation}\label{ell-po}
\frac{1}{2}\frac{\mathrm{d}}{\mathrm{d}t}|u|_2^2+ 2\alpha|u_r|_2^2+\alpha m\Big|\frac{u}{r}\Big|_2^2=\frac{A\gamma}{2\alpha}\int_0^\infty \rho^{\gamma-1}(v-u) u\,\mathrm{d}r+ \int_0^\infty v u_ru \,\mathrm{d}r :=\sum_{i=5}^6 \mathcal{I}_i.
\end{equation}
Then it follows from Lemmas \ref{important2-po} and \ref{l4.4-po}, and the  H\"older and Young inequalities that 
\begin{equation}\label{j3-j5-po}
\begin{aligned}
\mathcal{I}_5&\leq C_0\big(|\rho^{\gamma-1} v|_2|u|_2 +|\rho|_\infty^{\gamma-1} |u|_2^2\big) \leq  C(T)|(\rho^{\gamma-1} v,u)|_2^2,\\
\mathcal{I}_6&\leq C_0|v|_\infty|u|_2 |u_r|_2 \leq \frac{\alpha}{8}|u_r|_2^2+C(T)|u|_2^2.
\end{aligned}
\end{equation}

Combining with \eqref{ell-po}--\eqref{j3-j5-po} yields
\begin{equation}\label{pol-po}
\frac{\mathrm{d}}{\mathrm{d}t}|u|_2^2+ \alpha\Big|\big(u_r,\frac{u}{r}\big)\Big|_2^2\leq  C(T)|(\rho^{\gamma-1} v,u)|_2^2. 
\end{equation}

\smallskip
\textbf{2.} For $|\rho^{\gamma-1} v|_2$, multiplying \eqref{eq:effective3} by $\rho^{2\gamma-2}v$, together with $\eqref{e1.5hh}_1$, yields
\begin{equation}\label{eq:B6-pre-po}
\begin{aligned}
&\frac{1}{2}(\rho^{2\gamma-2}v^2)_t+ \frac{1}{2} (\underline{u\rho^{2\gamma-2}v^2}_{:=\tilde{\mathcal{B}}_6})_r+\frac{A\gamma}{2\alpha}\rho^{3\gamma-3} v^2\\
&=\Big(\frac{3}{2}-\gamma\Big)\rho^{2\gamma-2} v^2 u_r-(\gamma-1)m\rho^{2\gamma-2} v^2\frac{u}{r}+\frac{A\gamma}{2\alpha}\rho^{3\gamma-3} vu.
\end{aligned}
\end{equation}

Next, we need to show that $\tilde{\mathcal{B}}_6\in W^{1,1}(I)$ and $\tilde{\mathcal{B}}_6|_{r=0}=0$ 
for {\it a.e.} $t\in (0,T)$, which allows us to apply Lemma \ref{calculus} to obtain
\begin{equation}\label{eq:B6-po}
\int_0^\infty (\tilde{\mathcal{B}}_6)_r\,\mathrm{d}r=-\tilde{\mathcal{B}}_6|_{r=0}=0.   
\end{equation}
To show $\tilde{\mathcal{B}}_6|_{r=0}=0$, we first note that $v=u+2\alpha(\log\rho)_r$ and
\begin{equation*}
(\rho, \rho^{-1}, u,\rho_r)\in L^\infty(I)
\qquad\,\,\mbox{ for {\it a.e.} $t\in (0,T)$}
\end{equation*}
due to \eqref{spd-po}--\eqref{spd2-po} (or \eqref{spd33-po}--\eqref{spd233-po}), which implies that $v\in L^\infty(I)$ for {\it a.e.} $t\in (0,T)$. 
Then it follows from $u|_{r=0}$ and $(\rho,u) \in C(\bar I)$ 
for each $t\in (0,T]$ (due to \eqref{spd2-po} or \eqref{spd233-po}) that $\tilde{\mathcal{B}}_6|_{r=0}=0$. 
On the other hand,  based on \eqref{spd-po}--\eqref{spd2-po} (or \eqref{spd33-po}--\eqref{spd233-po}), 
we have
\begin{equation*}
(\rho,\rho^{-1},\rho_r,u,u_r)\in L^\infty(I),\quad\,\, r^\frac{m}{2}(u,u_r,\rho_r,\rho_{rr})\in L^2(I)
\qquad\,\,\mbox{for {\it a.e.} $t\in (0,T)$}.
\end{equation*}
Thus, one can first obtains from Lemma \ref{hardy} that
\begin{align*}
|(u,\rho_r)|_2&\leq \big|\chi_1^\flat(u,\rho_r)\big|_2+\big|\chi_1^\sharp(u,\rho_r)\big|_2\leq C_0\big|\chi_1^\flat r\big(u,u_r,\rho_r,\rho_{rr}\big)\big|_2+\big|\chi_1^\sharp(u,\rho_r)\big|_2\\
&\leq C_0\big|\chi_1^\flat r^\frac{2-m}{2}\big|_\infty\big|r^\frac{m}{2}\big(u,u_r,\rho_r,\rho_{rr}\big)\big|_2+\big|\chi_1^\sharp r^{-\frac{m}{2}}\big|_\infty|r^\frac{m}{2}(u,\rho_r)|_2<\infty.
\end{align*}
Then it follows from the above, \eqref{V-expression3}, and the H\"older and Young inequalities that
\begin{align*}
|\tilde{\mathcal{B}}_6|_1&\leq C_0|u|_\infty\big(|\rho|_\infty^{2\gamma-2}|u|_2^2+|\rho^{2\gamma-4}|_\infty|\rho_r|_2^2\big)<\infty,\\[1mm]
|(\tilde{\mathcal{B}}_6)_r|_1&\leq C_0\big|\big(u_r\rho^{2\gamma-2} v^2, u\rho^{2\gamma-3}\rho_r v^2,u\rho^{2\gamma-2} vv_r\big)\big|_1\\
&\leq C_0(|u_r|_\infty|\rho|_\infty^{2\gamma-2}+ |\rho^{2\gamma-3}|_\infty |\rho_r|_\infty|u|_\infty) (|u|_2^2+|\rho^{-2}|_\infty|\rho_r|_2^2)\\
&\quad+ C_0\,|\rho|_\infty^{2\gamma-2}\big(|u|_\infty+|\rho^{-1}|_\infty|\rho_r|_\infty\big)\Big|\frac{u}{r}\Big|_2|\chi_1^\flat r^\frac{2-m}{2}|_\infty \\
&\quad\qquad\times \big(|r^\frac{m}{2}u_r|_2+|\rho^{-2}|_\infty|\rho_r|_\infty|r^\frac{m}{2}\rho_r|_2 +|\rho^{-1}|_\infty |r^\frac{m}{2}\rho_{rr}|_2\big)\\
&\quad+ C_0\,|\rho|_\infty^{2\gamma-2}\big(|u|_\infty+|\rho^{-1}|_\infty|\rho_r|_\infty\big)|u|_2|\chi_1^\sharp r^{-\frac{m}{2}}|_\infty\\
&\quad\qquad\times \big(|r^\frac{m}{2} u_r|_2+|\rho^{-2}|_\infty|\rho_r|_\infty|r^\frac{m}{2}\rho_r|_2 +|\rho^{-1}|_\infty |r^\frac{m}{2}\rho_{rr}|_2\big)<\infty.
\end{align*}

Thus, integrating \eqref{eq:B6-pre-po} over $I$, then we obtain from \eqref{eq:B6-po}, 
Lemmas \ref{important2-po} and \ref{l4.4-po}, and the H\"older and Young inequalities that
\begin{equation}\label{new-v-po}
\begin{aligned}
&\frac{1}{2}\frac{\mathrm{d}}{\mathrm{d}t}|\rho^{\gamma-1} v|_2^2+ \frac{A\gamma}{2\alpha} \big|\rho^\frac{3\gamma-3}{2} v\big|_2^2\\
&=\big(\frac{3}{2}-\gamma\big)\int \rho^{2\gamma-2} v^2 u_r\,\mathrm{d}r-(\gamma-1)m\int \rho^{2\gamma-2}  v^2\frac{u}{r}\,\mathrm{d}r+\frac{A\gamma}{2\alpha}\int \rho^{3\gamma-3} vu\,\mathrm{d}r\\
&\leq  C_0|\rho|_\infty^{\gamma-1}|v|_\infty |\rho^{\gamma-1} v|_2 \Big|\big(u_r,\frac{u}{r}\big)\Big|_2+C_0|\rho|_\infty^{2\gamma-2} |\rho^{\gamma-1} v|_2 |u|_2\\
&\leq  C(T)|(\rho^{\gamma-1} v,u)|_2^2+\frac{\alpha}{8}\Big|\big(u_r,\frac{u}{r}\big)\Big|_2^2.
\end{aligned}
\end{equation}

Combining with \eqref{pol-po} and \eqref{new-v-po} gives
\begin{equation*} 
\frac{\mathrm{d}}{\mathrm{d}t}|(\rho^{\gamma-1} v,u)|_2^2+ \frac{\alpha}{2}\Big|\big(u_r,\frac{u}{r}\big)\Big|_2^2\leq  C(T)|(\rho^{\gamma-1} v,u)|_2^2, 
\end{equation*}
which, along with the Gr\"onwall inequality, yields that, for all $t\in [0,T]$,
\begin{equation}\label{47-po}
|(\rho^{\gamma-1} v,u)(t)|_2^2 + \int_0^t\Big|\big(u_r,\frac{u}{r}\big)\Big|_2^2\,\mathrm{d}s\leq  C(T).
\end{equation}
Here, it is still required to check the $L^2(I)$-boundedness of $(\rho_0^{\gamma-1} v_0,u_0)$. 
Indeed, according to Lemmas \ref{equiv-initial}, \ref{ale1}, and \ref{lemma-initial}, we have
\begin{align*}
|u_0|_2&\leq |\chi_1^\flat u_0|_2+|\chi_1^\sharp u_0|_2\leq |u_0|_\infty+|\chi_1^\sharp r^{-\frac{m}{2}}|_\infty |r^\frac{m}{2}u_0|_2 \\
&\leq C_0\big(\|\boldsymbol{u}_0\|_{L^\infty}+\|\boldsymbol{u}_0\|_{L^2}\big)\leq C_0\|\boldsymbol{u}_0\|_{H^2}\leq C_0,\\[1mm]
|\rho_0^{\gamma-1} v_0|_2&\leq C_0|\rho_0^{\gamma-1}(u_0,(\log\rho_0)_r)|_2\leq  C_0|\rho_0|_\infty^{\gamma-1}|u_0,(\log\rho_0)_r|_2\\
&\leq C_0\big|\chi_1^\flat\big(u_0,(\log\rho_0)_r\big)\big|_2+C_0\big|\chi_1^\sharp\big(u_0,(\log\rho_0)_r\big)\big|_2 \\
&\leq C_0|\chi_1^\flat r^\frac{2-m}{2} |_\infty \big|r^\frac{m-2}{2}\big(u_0,(\log\rho_0)_r\big)\big|_2+C_0|\chi_1^\sharp r^{-\frac{m}{2}}|_\infty\big|r^\frac{m}{2}\big(u_0,(\log\rho_0)_r\big)\big|_2\\
&\leq C_0\big\|\big(\boldsymbol{u}_0,\nabla\log\rho\big)\big\|_{H^1} \leq C_0.
\end{align*}

Finally, it follows from \eqref{47-po} and Lemma \ref{ale1} that 
\begin{equation*}
\int_0^t |u|^2_\infty\,\mathrm{d}s\leq C_0\int_0^t |(u,u_r)|_2^2 \,\mathrm{d}s\leq C_0\Big(t\sup_{s\in [0,t]}|u|_2^2+\int_0^t |u_r|_2^2 \,\mathrm{d}s\Big)\leq C(T).
\end{equation*}

The proof of Lemma \ref{ele-po} is completed.
\end{proof}

Now, based on Lemmas \ref{l4.4-po}--\ref{ele-po}, we  obtain the global uniform lower bound of $\rho$.
\begin{lem}\label{lemma-lowerbound}
For any $(t,r)\in[0,T]\times I$,
\begin{equation*}
\rho(t,r)\geq C(T)^{-1},
\end{equation*}
where $C(T)\geq 1$ is a constant depending only on $(T,\bar\rho,\rho_0,u_0,n,\alpha,\gamma,A)$.
\end{lem}

\begin{proof}
First, it follows from \eqref{V-expression3}, and Lemmas \ref{l4.4-po}--\ref{ele-po} and \ref{ale1} that
\begin{align*}
\big|\log(\rho/\bar\rho)\big|_\infty^3&\leq C_0\int_0^\infty\big|\log(\rho/\bar\rho)\big|^3\,\mathrm{d}r 
+ C_0\int_0^\infty \big|\log(\rho/\bar\rho)\big|^2\big|(\log\rho)_r\big|\,\mathrm{d}r\\
&\leq C_0\int_0^\infty\big|\log(\rho/\bar\rho)\big|^3\,\mathrm{d}r +C_0\int_0^\infty \big|\log(\rho/\bar\rho)\big|^2\big(|u|+|v|\big)\,\mathrm{d}r\\
&\leq C_0\big|\big(\log(\rho/\bar\rho),v\big)\big|_\infty\big|
\log(\rho/\bar\rho)\big|_2^2 +C_0 \big|\log(\rho/\bar\rho)\big|_\infty\big|\log(\rho/\bar\rho)\big|_2|u|_2\\
&\leq C(T)\Big(\big(\big|\log(\rho/\bar\rho)\big|_\infty+1\big)
\big|\log(\rho/\bar\rho)\big|_2^2+ \big|\log(\rho/\bar\rho)\big|_\infty\big|\log(\rho/\bar\rho)\big|_2\Big),
\end{align*}
which, along with the Young inequality, leads to
\begin{equation}\label{wuqiong-log-42}
\big|\log(\rho/\bar\rho)\big|_\infty\leq C(T)\big(\big|\log(\rho/\bar\rho)\big|_2+1\big).   
\end{equation}

Next, rewrite $\eqref{e1.5hh}_1$ as
\begin{equation}\label{new-e1.5}
\big(\log(\rho/\bar\rho)\big)_t+u(\log\rho)_r+\big(u_r+\frac{m}{r}u\big)=0.
\end{equation}
Multiplying the above by $\log(\rho/\bar\rho)$ and integrating over $I$, we obtain from \eqref{V-expression3}, Lemmas \ref{l4.4-po}--\ref{ele-po}, and the H\"older and Young inequalities that
\begin{equation*}
\begin{aligned}
\frac{1}{2}\frac{\mathrm{d}}{\mathrm{d}t}\big|\log(\rho/\bar\rho)\big|_2^2&=-\int_0^\infty  u(\log\rho)_r\log(\rho/\bar\rho)\,\mathrm{d}r
- \int_0^\infty  \big(u_r+\frac{m}{r}u\big)\log(\rho/\bar\rho)\,\mathrm{d}r\\
&\leq |u|_2|(\log\rho)_r|_\infty\big|\log(\rho/\bar\rho)\big|_2
+C_0\Big|\big(u_r,\frac{u}{r}\big)\Big|_2\big|\log(\rho/\bar\rho)\big|_2\\
&\leq C_0\Big(|u|_2^2|(u,v)|_\infty^2+\Big|\big(u_r,\frac{u}{r}\big)\Big|_2^2\Big)+\big|\log(\rho/\bar\rho)\big|_2^2\\
&\leq C(T)\Big(1+|u|_\infty^2+ \Big|\big(u_r,\frac{u}{r}\big)\Big|_2^2\Big)+\big|\log(\rho/\bar\rho)\big|_2^2,
\end{aligned}
\end{equation*}
which, along with Lemma \ref{ele-po} and the Gr\"onwall inequality, leads to
\begin{equation}\label{324-po}
\big|\log(\rho/\bar\rho)\big|_2\leq C(T)\big(\big|\log(\rho_0/\bar\rho)\big|_2+1\big).
\end{equation}

For the $L^2(I)$-boundedness of $\log(\rho_0/\bar\rho)$, observing that $\rho_0(r)\to \bar\rho$ as $r\to \infty$, one can find a sufficiently large $R_0>0$, depending only on $\bar\rho$, such that, for all $r\in [R_0,\infty)$, 
\begin{equation}\label{comment1}
\chi_{R_0}^\sharp\big|\log(\rho_0/\bar\rho)\big|
=\chi_{R_0}^\sharp\big|\log\big(1+(\rho_0/\bar\rho-1)\big)\big|\leq 2\chi_{R_0}^\sharp|\rho_0/\bar\rho-1|
\leq C_0\chi_{R_0}^\sharp|\rho_0-\bar\rho|.
\end{equation}
For such $R_0>0$, Lemmas \ref{equiv-initial} and \ref{lemma-initial}, together with $0<\inf_{r\in I}\rho_0\leq \rho_0\leq |\rho_0|_\infty$, imply 
\begin{equation}\label{comment2}
\begin{aligned}
\big|\log(\rho_0/\bar\rho)\big|_2&\leq \big|\chi_{R_0}^\flat\log(\rho_0/\bar\rho)\big|_2
+\big|\chi_{R_0}^\sharp\log(\rho_0/\bar\rho)\big|_2\\
&\leq C_0\sqrt{R_0}\big(|\log\rho_0|_\infty+|\log\bar\rho|\big)+C_0\big|\chi_{R_0}^\sharp (\rho_0-\bar\rho)\big|_2\\
&\leq C_0+C_0|r^\frac{m}{2} (\rho_0-\bar\rho)|_2\leq C_0+C_0\|\rho_0-\bar\rho\|_{L^2}\leq C_0.
\end{aligned}
\end{equation}

Thus, substituting above into \eqref{324-po}, together with \eqref{wuqiong-log-42}, gives
\begin{equation}
\big|\log(\rho/\bar\rho)\big|_\infty\leq C(T)\big(\big|\log(\rho/\bar\rho)\big|_2 +1\big)\leq C(T)\big(\big|\log(\rho_0/\bar\rho)\big|_2+1\big)\leq C(T).
\end{equation}
This implies that, for all $(t,r)\in [0,T]\times I$, $\rho(t,r)\geq C(T)^{-1}$ for some constant $C(T)\geq 1$ depending only on $(T,\bar\rho,\rho_0,u_0,n,\alpha,\gamma,A)$.
\end{proof}

\subsection{Global uniform \textbf{\textit{a priori}} estimates for the 2-order regular solutions}\label{globalestimates-2po}
In  \S\ref{globalestimates-2po}, 
we always assume that $\gamma\in (1,\infty)$ if $n=2$ and $\gamma\in (1,3)$ if $n=3$.
Let  $T>0$ be any fixed time, and  let $(\rho, u)(t,r)$ be the $2$-order regular solution 
of problem \eqref{e1.5hh} in $[0,T]\times I$ obtained in Theorem \ref{rth10po}. 
To establish the global-in-time uniform estimates for the $2$-order regular solutions 
when  $\bar\rho>0$, we consider system \eqref{re} in spherical coordinates:
\begin{equation}\label{re-r}
\begin{cases}
\displaystyle 
\rho_t+ u\rho_r+ \rho\big(u_r+\frac{m }{r}u\big)=0,\\[4pt]
\displaystyle
u_t+uu_r+\frac{A\gamma}{\gamma-1}(\rho^{\gamma-1})_r=2\alpha\big(u_r+\frac{m}{r}u\big)_r-2\alpha \frac{\rho_r}{\rho}u_r,\\[6pt]
\displaystyle
(\rho,u)|_{t=0}=(\rho_0,u_0) \qquad\qquad\qquad \ \ \text{for $r\in I$},\\[5pt]
\displaystyle
u|_{r=0}=0 \qquad\qquad\qquad\qquad\qquad\quad \text{for $t\in [0,T]$},\\[5pt]
\displaystyle
(\rho,u)\to \left(\bar\rho>0,0\right)  \ \ \text{as $r\to \infty$} \qquad \text{for $t\in [0,T]$}.
\end{cases}
\end{equation}
Besides, let $(\phi,\psi,\boldsymbol{\psi})$ be defined as in \eqref{bianhuan} and \eqref{tr}: 
\begin{equation}\label{new-def2}
\begin{aligned}
\phi=\frac{A\gamma}{\gamma-1}\rho^{\gamma-1},\qquad \psi=(\log\rho)_r,\qquad \boldsymbol{\psi}=\nabla\log\rho. 
\end{aligned}
\end{equation}

The first lemma is on the zeroth-order energy estimate of $u$. 
\begin{lem}\label{l4.5-po}
There exists a constant $C(T)>0$ such that, for any $t\in [0,T]$,
\begin{equation*}
\big|r^{\frac{m}{2}}(u,\rho^{\gamma-1} v)(t)\big|_2^2 +\int_0^t \Big|r^{\frac{m}{2}}\big(u_r,\frac{u}{r}\big)\Big|_2^2 \, \mathrm{d}s\leq C(T).
\end{equation*}
\end{lem}

\begin{proof} We divide the proof into two steps.

\smallskip
\textbf{1.} Using the same arguments as in Step 1 of the proof for Lemma \ref{l4.5}, with $(\phi,\psi)$ replaced by $(\frac{A\gamma}{\gamma-1}\rho^{\gamma-1},(\log\rho)_r)$, we obtain \eqref{eq4.16}:
\begin{equation}\label{eq4.16-po}
\frac{\mathrm{d}}{\mathrm{d}t} |r^{\frac{m}{2}}u|_2^2 + \alpha\Big|r^{\frac{m}{2}}\big(u_r,\frac{u}{r}\big)\Big|_2^2 \leq C(T)(1+|u|_\infty^2) |r^\frac{m}{2}(\rho^{\gamma-1}v,u)|_2^2.
\end{equation}

\smallskip
\textbf{2.} For the $L^2(I)$-estimate of $r^\frac{m}{2}\rho^{\gamma-1} v$, one can first multiply \eqref{eq:effective3} by $r^m\rho^{2\gamma-2}v$ and then obtain from $\eqref{re-r}_1$ that 
\begin{equation}\label{b8pre-po}
\begin{aligned}
&\frac{1}{2} (r^m \rho^{2\gamma-2}v^2)_t+\frac{A\gamma}{2\alpha}r^m \rho^{3\gamma-3} v^2\\
&=-\frac{1}{2}(\underline{r^m u\rho^{2\gamma-2}v^2}_{:=\tilde{\mathcal{B}}_7})_r+\big(\frac{3}{2}-\gamma\big)r^m\rho^{2\gamma-2}v^2\big(u_r+\frac{m}{r}u\big)+\frac{A\gamma}{2\alpha}r^m \rho^{3\gamma-3} vu.
\end{aligned}
\end{equation}
Next, we need to  show that $\tilde{\mathcal{B}}_7\in W^{1,1}(I)$ and $\tilde{\mathcal{B}}_7|_{r=0}=0$ for {\it a.e.} $t\in (0,T)$, which allows us to apply Lemma \ref{calculus} to obtain
\begin{equation}\label{eq:B8-po}
\int_0^\infty (\tilde{\mathcal{B}}_7)_r\,\mathrm{d}r=-\tilde{\mathcal{B}}_7|_{r=0}=0.
\end{equation}

On one hand, $\tilde{\mathcal{B}}_7|_{r=0}=0$ follows from $v\in L^\infty(I)$ for {\it a.e.} $t\in (0,T)$, and $u|_{r=0}=0$ and $(\rho,u)\in C(\bar I)$ for {\it a.e.} $t\in (0,T]$ due to \eqref{spd-po}. On the other hand, thanks to \eqref{spd-po}--\eqref{spd2-po}, 
\begin{equation*}
\Big(\rho,\rho^{-1},\rho_r,u,\frac{u}{r},u_r\Big)\in L^\infty(I),\quad\,\, 
r^\frac{m}{2}(\rho_r,\rho_{rr},u,u_r)\in L^2(I) \qquad\,\, \mbox{for {\it a.e.} $t\in (0,T)$},
\end{equation*}
we obtain from \eqref{V-expression3} and the H\"older inequality that
\begin{align*}
|\tilde{\mathcal{B}}_7|_1&\leq C_0|u|_\infty\big(|\rho|_\infty^{2\gamma-2}|r^\frac{m}{2}u|_2^2+ |\rho^{2\gamma-4}|_\infty|r^\frac{m}{2}\rho_r|_2^2\big) <\infty, \\[1mm]
|(\tilde{\mathcal{B}}_7)_r|_1&\leq C_0\big|\big(r^{m-1}u\rho^{2\gamma-2} v^2,r^m u_r\rho^{2\gamma-2} v^2,r^m u\rho^{2\gamma-3}\rho_r v^2,r^{m}u\rho^{2\gamma-2} vv_r\big)\big|_1\\
&\leq C_0\Big|\big(\frac{u}{r},u_r\big)\Big|_\infty\big(|\rho|_\infty^{2\gamma-2}|r^\frac{m}{2}u|_2^2+ |\rho^{2\gamma-4}|_\infty|r^\frac{m}{2}\rho_r|_2^2\big)\\
&\quad +C_0|\rho^{2\gamma-3}|_\infty|v|_\infty^2 |r^\frac{m}{2}u|_2|r^\frac{m}{2}\rho_r|_2+C_0|\rho|_\infty^{2\gamma-2}|v|_\infty|r^\frac{m}{2}u|_2 |r^\frac{m}{2}u_r|_2\\
&\quad + C_0|\rho|_\infty^{2\gamma-2}|v|_\infty|r^\frac{m}{2}u|_2\big(|\rho^{-2}|_\infty|\rho_r|_\infty|r^\frac{m}{2}\rho_r|_2+|\rho^{-1}|_\infty|r^\frac{m}{2}\rho_{rr}|_2\big)<\infty. 
\end{align*}    

Integrating \eqref{b8pre-po} over $I$, then repeating the same calculations \eqref{555}--\eqref{chuzhi0} in the proof of Lemma \ref{l4.5}, with $(\phi,\psi)$ replaced by $(\frac{A\gamma}{\gamma-1}\rho^{\gamma-1},(\log\rho)_r)$, and using Lemma \ref{equiv-initial}, we can obtain the desired estimate of this lemma. 
\end{proof}

The second lemma concerns the first-order energy estimate for $u$.
\begin{lem}\label{l4.6-po}
There exists a constant $C(T)>0$ such that, for any $t\in [0,T]$,
\begin{equation*} 
\Big|r^{\frac{m}{2}}\big(u_r,\frac{u}{r}\big)(t)\Big|_2^2+\int_0^t |r^{\frac{m}{2}}u_t|_2^2\,\mathrm{d}s\leq C(T).
\end{equation*}
\end{lem}
\begin{proof}
Lemma \ref{l4.6-po} can be derived by the same argument as in the proof of Lemma \ref{l4.6}, with $(\phi,\psi)$ replaced by $(\frac{A\gamma}{\gamma-1}\rho^{\gamma-1},(\log\rho)_r)$. We omit the details here.
\end{proof}

The third lemma concerns the second-order energy estimate for $u$.
\begin{lem}\label{l4.8-po}
There exists a constant $C(T)>0$ such that, for any $t\in[0,T]$,
\begin{equation*}
|r^{\frac{m}{2}}u_t(t)|^2_2+\int_0^t \Big|r^{\frac{m}{2}}\big(u_{tr},\frac{u_t}{r}\big)\Big|_2^2 \,\mathrm{d}s \leq C(T).
\end{equation*}
\end{lem}
\begin{proof} 
Note that, due to Lemmas \ref{important2-po} and \ref{lemma-lowerbound}, we can show that $(|\phi_r|,|\psi|)$ and $|\phi_t|$ are equivalent to  $|\rho_r|$ and $|\rho_t|$, respectively: 
\begin{equation}\label{equiv-eu}
\begin{gathered}
|\phi_r| \sim  |\psi| \sim  |\rho_r|,\quad\, |\phi_t|\sim  |\rho_t|
\qquad\,\,\,\mbox{for all $(t,r)\in [0,T]\times I$},
\end{gathered}
\end{equation}
where $E\sim F$ denotes $C(T)^{-1}E\leq F\leq C(T)E$. Hence, based on \eqref{equiv-eu} and Lemma \ref{equiv-initial},  following the same argument as in the proof of Lemma \ref{l4.8}, with $(\phi,\psi)$ replaced by $(\frac{A\gamma}{\gamma-1}\rho^{\gamma-1},(\log\rho)_r)$, we can complete the proof.

\smallskip
Thus, recalling \eqref{eq:B10pre}, we find that the major modification is to justify the process of integration by parts for 
\begin{equation*}
\mathcal{B}_{10}:=2\alpha r^m u_t \big(u_{tr}+ \frac{m}{r}u_t\big)-r^mu_t \phi_{t}-2\alpha r^m \psi u u_tu_r.
\end{equation*}
 
We first prove $\mathcal{B}_{10}|_{r=0}=0$. Using the same arguments as in Step 2 of the proof of Lemma \ref{l4.8}, we need to show that 
\begin{equation}\label{727-po}
r^\frac{m}{2}\phi_t|_{r=0}=A\gamma r^\frac{m}{2}\rho^{\gamma-2}\rho_t|_{t=0}<\infty \qquad \text{for {\it a.e.} $t\in (0,T)$}.
\end{equation}
Thanks to \eqref{spd-po}--\eqref{spd2-po}, one has 
\begin{equation}\label{shi}
 (\rho,\rho^{-1})\in C(\bar I), \quad\,\, r^\frac{m}{2}(\rho_t,\rho_{tr})\in L^2(I) 
 \qquad\,\, \text{for {\it a.e.} $t\in (0,T)$}.
\end{equation}
Then, if $n=2$ ($m=1$), it follows from \eqref{shi} and Lemma \ref{hardy} that, for {\it a.e.} $t\in (0,T)$, 
\begin{equation*}
r(\rho_t,\rho_{tr})\in L^2(0,1)\implies r^\frac{1}{2}\rho_t \in C([0,1]),
\end{equation*}
and hence $r^\frac{1}{2}\rho^{\gamma-2}\rho_t|_{r=0}<\infty$ for {\it a.e.} $t\in (0,T)$; while if $n=3$ ($m=2$), one can deduce from \eqref{shi} 
and Lemmas \ref{ale1} and \ref{hardy} that, for {\it a.e.} $t\in (0,T)$,
\begin{equation*}
r\rho_t\in H^1(0,1) \implies r\rho_t\in C([0,1]),
\end{equation*}
that is, 
\begin{equation*}
r \rho^{\gamma-2}\rho_t|_{r=0}<\infty \qquad \mbox{for {\it a.e.} $t\in (0,T)$}.
\end{equation*}
Thus, we conclude \eqref{727-po}.

Next, we show that $\mathcal{B}_{10}\in W^{1,1}(I)$ for {\it a.e.} $t\in (0,T)$. 
By \eqref{spd-po}--\eqref{spd2-po}, we have 
\begin{equation*}
(\rho_r,u,u_r)\in L^\infty(I),\quad\,\,\, 
r^\frac{m}{2}\Big(\rho_t,\rho_{tr},\rho_{rr},u_r,u_t,u_{rr},\frac{u_t}{r},u_{tr},\big(\frac{u_t}{r}\big)_r,u_{trr}\Big)\in L^2(I)
\end{equation*}
for {\it a.e.} $t\in (0,T)$. 
Thus, following a calculation similar to \eqref{cal-b10} in the proof of Lemma \ref{l4.8}, 
we obtain from the above, \eqref{equiv-eu}, and \eqref{shi} that 
\begin{align*}
|\mathcal{B}_{10}|_1&\leq C_0|r^\frac{m}{2}u_t|_2\Big(\Big|r^\frac{m}{2}\big(u_{tr},\frac{u_t}{r}\big)\Big|_2+|r^\frac{m}{2}\phi_t|_2+|\psi|_\infty|u|_\infty |r^\frac{m}{2}u_r|_2\Big)\\
&\leq C(T)|r^\frac{m}{2}u_t|_2\Big(\Big|r^\frac{m}{2}\big(u_{tr},\frac{u_t}{r}\big)\Big|_2+|r^\frac{m}{2}\rho_t|_2+|\rho_r|_\infty|u|_\infty |r^\frac{m}{2}u_r|_2\Big)<\infty,\\[1mm]
|(\mathcal{B}_{10})_r|_1
&\leq C_0 \Big|r^\frac{m}{2}\big(u_{tr},\frac{u_t}{r}\big)\Big|_2\Big(\Big|r^\frac{m}{2}\big(u_{tr},\frac{u_t}{r},\phi_t\big)\Big|_2+|\psi|_\infty |u|_\infty |r^\frac{m}{2}u_r|_2\Big)\notag\\
&\quad +C_0|r^\frac{m}{2}u_t|_2\Big(\Big|r^\frac{m}{2} u_{trr},\big(\frac{u_t}{r}\big)_r,\phi_{tr}\Big)\Big|_2+|r^\frac{m}{2}\psi_r|_2|u|_\infty|u_r|_\infty\Big)\notag\\
&\quad +C_0|r^\frac{m}{2}u_t|_2\big(|\psi|_\infty|u_r|_\infty|r^\frac{m}{2}u_r|_2+|\psi|_\infty|u|_\infty |r^\frac{m}{2}u_{rr}|_2\big)\\
&\leq C(T)\Big|r^\frac{m}{2}\big(u_{tr},\frac{u_t}{r}\big)\Big|_2\Big(\Big|r^\frac{m}{2}\big(u_{tr},\frac{u_t}{r},\rho_t\big)\Big|_2+|\rho_r|_\infty |u|_\infty |r^\frac{m}{2}u_r|_2\Big)\notag\\
&\quad +C_0|r^\frac{m}{2}u_t|_2\Big(\Big|r^\frac{m}{2} u_{trr},\big(\frac{u_t}{r}\big)_r\Big|_2+|\rho^{\gamma-3}|_\infty |r^\frac{m}{2}\rho_t|_2|\rho_{r}|_\infty+|\rho^{\gamma-2}|_\infty |r^\frac{m}{2}\rho_{tr}|_2\Big)\\
&\quad +C_0|r^\frac{m}{2}u_t|_2\Big(|\rho^{-2}|_\infty|\rho_r|_\infty\big|r^\frac{m}{2}\rho_r\big|_2+|\rho^{-1}|_\infty|r^\frac{m}{2}\rho_{rr}|_2\Big)|u|_\infty|u_r|_\infty\notag\\
&\quad +C(T)|r^\frac{m}{2}u_t|_2\big(|\rho_r|_\infty|u_r|_\infty|r^\frac{m}{2}u_r|_2+|\rho_r|_\infty|u|_\infty |r^\frac{m}{2}u_{rr}|_2\big)<\infty,
\end{align*}
which implies the assertion.
\end{proof}

With the help of \ref{l4.8-po}, we can also obtain the following estimates:
\begin{lem}\label{lemma66-po}
There exists a constant $C(T)>0$ such that, for any $t\in [0,T]$,
\begin{equation*}
|(u,r^\frac{m}{2}u_r)(t)|_\infty + \Big|r^{\frac{m}{2}}\Big(\rho_r,\rho_t,u_{rr}, \big(\frac{u}{r}\big)_{r}\Big)(t)\Big|_2+ \int_0^t \Big|\big(u_r,\frac{u}{r}\big)\Big|_\infty^2\,\mathrm{d}s \leq C(T).
\end{equation*}
\end{lem}
\begin{proof}
Based on \eqref{equiv-eu}, Lemma \ref{lemma66-po} follows directly from the same argument as in the proof of Lemma \ref{lemma66}, with $(\phi,\psi)$ replaced by $(\frac{A\gamma}{\gamma-1}\rho^{\gamma-1},(\log\rho)_r)$. 
We omit the details. Of course, we can also obtain from \eqref{ur-infty} that
\begin{equation}\label{ur-infty-po}
\Big|\big(u_r,\frac{u}{r}\big)\Big|_\infty \leq C(T)\Big(\Big|r^\frac{m}{2}\big(u_{tr},\frac{u_t}{r}\big)\Big|_2+1\Big). 
\end{equation}
This, together with Lemma \ref{l4.8-po}, leads to the desired estimates of this lemma. 
\end{proof}

The following lemma provides the high-order estimates of $(\rho,u)$.
\begin{lem}\label{l4.9-po}  
There exists a constant $C(T)>0$ such that, for any $t\in [0,T]$,
\begin{equation*}
|\rho_r(t)|_\infty+ \Big|r^{\frac{m}{2}}\Big(\rho_{rr},\frac{\rho_r}{r},\rho_{tr}\Big)(t)\Big|^2_2+\int_{0}^{t}\Big|r^{\frac{m}{2}}\Big(u_{rrr},\frac{u_{rr}}{r},\big(\frac{u}{r}\big)_{rr},\frac{1}{r}\big(\frac{u}{r}\big)_r\Big)\Big|_2^2 \,\mathrm{d}s\leq C(T).
\end{equation*}
\end{lem}
\begin{proof}We divide the proof into two steps.

\smallskip
\textbf{1.} {\bf Estimates on $\rho$.} First, it follows from \eqref{V-expression3} and Lemmas \ref{important2-po}, \ref{l4.4-po}, and \ref{lemma66-po} that, for all $t\in [0,T]$,
\begin{equation}\label{rhor-infty}
|\rho_r(t)|_\infty\leq |\rho(t)|_\infty |(\log\rho)_r(t)|_\infty\leq C(T)|(v,u)(t)|_\infty\leq C(T).    
\end{equation}
Then, by \eqref{new-def2}, Lemmas \ref{important2-po} and \ref{lemma-lowerbound}, and the following identities: for $\partial=\partial_r$ or $\partial_t$,
\begin{equation}\label{cal-equiv-high}
\begin{aligned}
&\partial\phi_{r}=A\gamma(\gamma-2) \rho^{\gamma-3}\rho_r\partial\rho+A\gamma \rho^{\gamma-2}\partial\rho_{r},\quad \partial\psi= -\rho^{-2} \rho_r\partial\rho+ \rho^{-1}\partial\rho_{r},\\
&\partial\rho_{r}=(2-\gamma)\rho^{-1} \rho_r\partial\rho+(A\gamma)^{-1}\rho^{2-\gamma}\partial\phi_{r}=\psi\partial\rho+\rho\partial\psi,
\end{aligned}
\end{equation}
we obtain from \eqref{rhor-infty} that the following mutual constraints between $(|\partial\phi_{r}|,|\partial\psi|)$ and $(|\partial\rho|,|\partial\rho_{r}|)$ hold:
For all $(t,r)\in [0,T]\times I$,
\begin{equation}\label{equiv-eu-high}
\begin{aligned}
&|\partial\phi_{r}|+|\partial\psi|\leq C(T)(|\partial\rho|+ |\partial\rho_{r}|),\\[1mm]
&|\partial\rho_{r}|\leq C(T)(|\partial\rho|+ \mathcal{Z}_1) \qquad \text{with $\mathcal{Z}_1=|\partial\phi_{r}|$ or $|\partial\psi|$}.
\end{aligned}    
\end{equation}

Next, repeating the same calculations \eqref{845}--\eqref{psi_t} in the proof of Lemma \ref{l4.7}, with $\phi$ replaced by $\frac{A\gamma}{\gamma-1}\rho^{\gamma-1}$, we see from \eqref{equiv-eu} and Lemma \ref{equiv-initial} that, for $t\in [0,T]$,
\begin{equation*}
\Big|r^\frac{m}{2}\big(\psi_r,\frac{\psi}{r}\big)(t)\Big|_{2}+\big|r^\frac{m}{2}\psi_t(t)\big|_2 \leq C(T). 
\end{equation*}
Thus, it follows from  \eqref{equiv-eu} and \eqref{equiv-eu-high}, and Lemma \ref{lemma66-po} that 
\begin{equation}\label{rhorr}
\begin{aligned}
\Big|r^\frac{m}{2}\big(\rho_{rr},\frac{\rho_r}{r}\big)(t)\Big|_{2}&\leq C(T)\Big|r^\frac{m}{2}\big(\rho_r,\psi_r,\frac{\psi}{r}\big)(t)\Big|_{2} \leq C(T),\\
|r^{\frac{m}{2}}\rho_{tr}(t)|_2&\leq C(T)|r^{\frac{m}{2}}(\rho_t,\psi_t)|_2 \leq C(T).
\end{aligned}    
\end{equation}

\smallskip
\textbf{2.} {\bf Estimates on $u$.} 
To obtain the estimates for $u$, repeating the same calculations \eqref{6028}--\eqref{uxxxl1} 
as in the proof of Lemma \ref{l4.9}, we obtain from Lemmas \ref{important2-po} and \ref{lemma66-po}, 
\eqref{equiv-eu}, and \eqref{equiv-eu-high}--\eqref{rhorr} that
\begin{align*}
\Big|r^\frac{m-2}{2}\big(u_r+\frac{m}{r}u\big)_r\Big|_2
&\leq |r^\frac{m-2}{2} u_t|_2+C_0|u_r|_\infty|r^\frac{m-2}{2}(u,\psi)|_2+C_0|\phi|_\infty |r^\frac{m-2}{2}\psi|_2\notag\\
&\leq C(T)\big(|r^\frac{m-2}{2} u_t|_2+ |u_r|_\infty|r^\frac{m-2}{2}(u,\rho_r)|_2 +|\rho|_\infty^{\gamma-1} |r^\frac{m-2}{2}\rho_r|_2\big)\notag\\
&\leq C(T)\big(|r^\frac{m-2}{2} u_t|_2+|u_r|_\infty+1\big),\notag\\[1mm]
\Big|r^{\frac{m}{2}}\big(u_r+ \frac{m}{r}u\big)_{rr}\Big|_2&\leq |r^{\frac{m}{2}}u_{tr}|_{2}+|u_r|_\infty|r^{\frac{m}{2}} u_r|_2+C_0|(u,\psi)|_\infty|r^{\frac{m}{2}}u_{rr}|_2 \\
&\quad +C_0|\psi|_\infty|r^{\frac{m}{2}}\phi_r|_2+C_0|(\phi,u_r)|_\infty|r^{\frac{m}{2}}\psi_r|_2\notag\\
&\leq C(T)\big(r^{\frac{m}{2}}u_{tr}|_{2}+|u_r|_\infty|r^{\frac{m}{2}} u_r|_2+ |(u,\rho_r)|_\infty|r^{\frac{m}{2}}u_{rr}|_2\big)\notag\\
&\quad +C(T)\big(|\rho_r|_\infty|r^{\frac{m}{2}}\rho_r|_2+ |(\rho^{\gamma-1},u_r)|_\infty|r^{\frac{m}{2}}(\rho_r,\rho_{rr})|_2\big)\notag\\
&\leq C(T)\big(|r^{\frac{m}{2}}u_{tr}|_{2}+|u_r|_\infty+1\big).\notag
\end{align*}

Thus, the above estimates, together with Lemmas \ref{im-1} and \ref{l4.8-po}--\ref{lemma66-po}, 
give the desired estimates of this lemma. 
\end{proof}

Finally, the following two lemmas concern time-weighted estimates of the velocity.
\begin{lem}\label{l4.10-po}  
There exists a constant $C(T)>0$ such that, for any $t\in [0,T]$,
\begin{equation*}
\begin{split}
&t\Big|r^{\frac{m}{2}}\big(u_{tr},\frac{u_t}{r}\big)(t)\Big|^2_2+\int_{0}^{t} s |r^{\frac{m}{2}}u_{tt}|^2_2 \,\mathrm{d}s\leq C(T).
\end{split}
\end{equation*}
\end{lem}
\begin{proof}
First, following the same calculations \eqref{eq:b12pre}--\eqref{etrq} in the proof of Lemma \ref{l4.10}, we see from \eqref{equiv-eu}, \eqref{ur-infty-po} and \eqref{equiv-eu-high}, Lemmas \ref{im-1} and \ref{l4.8-po}--\ref{l4.9-po}, and the H\"older and Young inequalities that 
\begin{align*}
&\alpha\frac{\mathrm{d}}{\mathrm{d}t}\Big|r^{\frac{m}{2}}\big(u_{tr}+\frac{m}{r}u_t\big)\Big|_2^2+ |r^{\frac{m}{2}}u_{tt}|_2^2\\
&=-\int r^m\big((u u_r)_t+\phi_{tr}-2\alpha(\psi u_r)_t\big)u_{tt}\,\mathrm{d}r\notag\\
&\leq C_0\big(|r^{\frac{m}{2}}(u_t,\psi_t)|_2|u_r|_\infty + |r^{\frac{m}{2}}u_{tr}|_2|(u,\psi)|_\infty+|r^{\frac{m}{2}}\phi_{tr}|_2\big)|r^{\frac{m}{2}}u_{tt}|_2 \\
&\leq C(T)\big(|r^{\frac{m}{2}}(u_t,\rho_t,\rho_{tr})|_2|u_r|_\infty + |r^{\frac{m}{2}}u_{tr}|_2|(u,\rho_r)|_\infty+|r^{\frac{m}{2}}(\rho_t,\rho_{tr})|_2\big)|r^{\frac{m}{2}}u_{tt}|_2\notag\\
&\leq C(T)\Big(\Big|r^{\frac{m}{2}}\big(u_{tr},\frac{u_t}{r}\big)\Big|_2+1\Big)|r^{\frac{m}{2}}u_{tt}|_2\\
&\leq   C(T)\Big(\Big|r^{\frac{m}{2}}\big(u_{tr}+\frac{m}{r}u_t\big)\Big|_2^2+1\Big)+\frac{1}{8}|r^{\frac{m}{2}}u_{tt}|_2^2.\notag
\end{align*}
 
Finally, based on the above inequality, repeating the same calculations \eqref{etrq2}--\eqref{t-use} 
as in the proof of Lemma \ref{l4.10}, we obtain the desired estimate.
\end{proof}

\begin{lem}\label{l4.10-ell-po}  
There exists a constant $C(T)>0$ such that, for any $t\in [0,T]$,
\begin{equation*}
t\Big|r^\frac{m}{2}\Big(u_{rrr},\frac{u_{rr}}{r},\big(\frac{u}{r}\big)_{rr},\frac{1}{r}\big(\frac{u}{r}\big)_r\Big)(t)\Big|_{2}^2+\int_{0}^{t} s \Big|r^\frac{m}{2}\Big(u_{trr},\big(\frac{u_t}{r}\big)_{r}\Big)\Big|_{2}^2 \,\mathrm{d}s\leq C(T).
\end{equation*}
\end{lem}

\begin{proof}
First, following the same calculations \eqref{ur-infty-t}--\eqref{urrr-t} as in the proof 
of Lemma \ref{l4.10-ell}, we see that 
\begin{equation}\label{ur-infty-t-po}
\sqrt{t}\Big|\big(u_r,\frac{u}{r}\big)\Big|_\infty+\sqrt{t}\Big|r^\frac{m}{2}\Big(u_{rrr},\frac{u_{rr}}{r},\big(\frac{u}{r}\big)_{rr},\frac{1}{r}\big(\frac{u}{r}\big)_r\Big)\Big|_{2} \leq C(T).
\end{equation}

Next, repeating the same calculation \eqref{utt} in the proof of Lemma \ref{l4.10-ell}, together with \eqref{equiv-eu}, \eqref{equiv-eu-high}, 
\eqref{ur-infty-t-po}, and Lemmas \ref{l4.8-po}--\ref{l4.10-po}, yields
\begin{align*}
&\sqrt{t}\Big|r^\frac{m}{2}\Big(u_{trr},\big(\frac{u_t}{r}\big)_{r}\Big)\Big|_{2}\\
&\leq C_0\sqrt{t}\big(|r^{\frac{m}{2}}u_{tt}|_2 +|r^{\frac{m}{2}}(u_t,\psi_t)|_2|u_r|_\infty+|r^{\frac{m}{2}}u_{tr}|_2|(u,\psi)|_\infty+|r^{\frac{m}{2}}\phi_{tr}|_2\big) \\
&\leq C(T)\sqrt{t}\big(|r^{\frac{m}{2}}u_{tt}|_2 +|r^{\frac{m}{2}}(u_t,\rho_t,\rho_{tr})|_2|u_r|_\infty+|r^{\frac{m}{2}}u_{tr}|_2|(u,\rho_r)|_\infty+|r^{\frac{m}{2}}(\rho_t,\rho_{tr})|_2\big) \\
&\leq C(T)\big(\sqrt{t}|r^{\frac{m}{2}}u_{tt}|_2 +1\big).
\end{align*}
Taking the square of the above and integrating the resulting inequality over $[0,t]$, 
together with Lemma \ref{l4.10-po}, leads to the desired estimates.
\end{proof}

\subsection{Global uniform \textbf{\textit{a priori}} estimates for 3-order regular solutions}\label{globalestimates-3po}
In  \S\ref{globalestimates-3po}, we always assume that $\gamma\in (1,\infty)$ if $n=2$ and $\gamma\in (1,3)$ if $n=3$. Let $T>0$ be any fixed time, and  let $(\rho, u) (t,r)$ be the $3$-order regular solution of
problem \eqref{e1.5hh} in $[0,T]\times I$ obtained in Theorems \ref{rth133-po}. 
The $H^2(\mathbb{R}^n)$-estimates are the same as those presented 
in Lemmas \ref{l4.5-po}--\ref{l4.9-po}, and thus we focus only on 
the $D^3(\mathbb{R}^n)$-estimates and the time-weighted estimates. 

We first give the $L^\infty(I)$-estimate of $(u_r,\frac{u}{r})$ and the third-order estimates of $u$.
\begin{lem}\label{H3-1-po}
There exists a constant $C(T)>0$ such that, for any $t\in [0,T]$,
\begin{align*}
&\Big|r^{\frac{m}{2}}\big(u_{tr},\frac{u_t}{r}\big)(t)\Big|_2+\Big|r^\frac{m}{2}\Big(u_{rrr},\frac{u_{rr}}{r},\big(\frac{u}{r}\big)_{rr},\frac{1}{r}\big(\frac{u}{r}\big)_r\Big)(t)\Big|_{2}\\
&+\Big|\big(u_r,\frac{u}{r}\big)(t)\Big|_\infty+\int_{0}^{t} |r^{\frac{m}{2}}u_{tt}|^2_2\,\mathrm{d}s+\int_0^t\Big|r^\frac{m}{2}\Big(u_{trr},\big(\frac{u_t}{r}\big)_{r}\Big)\Big|_{2}^2\,\mathrm{d}s\leq C(T).
\end{align*}
\end{lem}
\begin{proof}
This can be proved by the same argument as in the proof of Lemma \ref{H3-1}, 
with $(\phi,\psi)$ replaced by $(\frac{A\gamma}{\gamma-1}\rho^{\gamma-1},(\log\rho)_r)$. 
To achieve this, it suffices to show that \eqref{limsup} holds.

Indeed, taking the limit as $\tau\to0$ in \eqref{cont-limsup}, 
we obtain from  Lemmas \ref{im-1}, \ref{equiv-initial}, \ref{ale1}, and \ref{lemma-initial},
and the time-continuity of $(\rho,u)$ presented in Lemma \ref{zth2-po}
that
\begin{align*}
\limsup_{\tau\to 0} \Big|r^\frac{m}{2}\big(u_{tr},\frac{u_t}{r}\big)(\tau)\Big|_2&\leq  C_0\|(\boldsymbol{u}_0,\nabla\log\rho_0,\nabla \boldsymbol{u}_0)\|_{L^\infty}  \|(\nabla^2\boldsymbol{u}_0,\nabla\boldsymbol{u}_0,\nabla^2\log\rho_0)\|_{L^2} \\
&\quad +C_0 \|(\nabla^2\rho_0^{\gamma-1},\nabla^3\boldsymbol{u}_0)\|_{L^2} \leq C_0.
\end{align*}
The proof of Lemma \ref{H3-1-po} is completed.
\end{proof}

Next, we derive the higher-order estimates for $(\rho,u)$.
\begin{lem}\label{Lemma6.12-po}
There exists a constant $C(T)>0$ such that, for any $t\in [0,T]$,
\begin{equation*}
\Big|r^\frac{m}{2}\!\Big(\rho_{rrr},\!\big(\frac{\rho_r}{r}\big)_r,\rho_{trr},\!\frac{\rho_{tr}}{r}\Big)(t)\Big|_2\!+\!\int_0^t\!\Big|r^\frac{m}{2}\!\Big(u_{rrrr},\!\big(\frac{u_{rr}}{r}\big)_{r},\!\big(\frac{u}{r}\big)_{rrr},\!\Big(\frac{1}{r}\big(\frac{u}{r}\big)_r\Big)_r\Big)\Big|_{2}^2\mathrm{d}s\leq C(T).
\end{equation*}
\end{lem}
\begin{proof}We divide the proof into two steps.

\smallskip
\textbf{1.} {\bf Estimates on $\rho$.}
First, it follows from $\eqref{re-r}_1$, \eqref{rhor-infty}, and Lemmas \ref{lemma66-po} and \ref{H3-1-po} that, for all $t\in[0,T]$,
\begin{equation}\label{rhot-infty}
|\rho_t|_\infty\leq C_0\Big(|u|_\infty|\rho_r|_\infty+|\rho|_\infty\Big|\big(u_r,\frac{u}{r}\big)\Big|_\infty\Big)\leq C(T).
\end{equation}
Then, based on the above, \eqref{equiv-eu} and \eqref{rhor-infty}, following an argument 
similar to \eqref{cal-equiv-high}--\eqref{equiv-eu-high} in Step 1 of the proof of Lemma \ref{l4.9-po}, 
we can obtain some mutual constraints for the high-order derivatives of $(\rho,\phi,\psi)$: 
For all $(t,r)\in [0,T]\times I$,
\begin{align}
&\begin{cases}
|\phi_{rrr}|+|\psi_{rr}|\leq C(T)(|\rho_r|+|\rho_{rr}|+|\rho_{rrr}|),\\[6pt]
|\rho_{rrr}|\leq C(T)(|\rho_r|+|\rho_{rr}|+\mathcal{Z}_2) \qquad \text{with $\mathcal{Z}_2=|\phi_{rrr}|$ or $|\psi_{rr}|$},    
\end{cases}\notag\\
&\begin{cases}
|\phi_{trr}|+|\psi_{tr}|\leq C(T)(|\rho_t|+|\rho_{tr}|+|\rho_{rr}|+|\rho_{trr}|),\\[6pt]
|\rho_{tr}|\leq C(T)(|\rho_r|+\mathcal{Z}_3)  \qquad \text{with $\mathcal{Z}_3=|\phi_{tr}|$ or $|\psi_{t}|$},\\[6pt]
|\rho_{trr}|\leq C(T)(|\rho_t|+|\rho_{tr}|+|\rho_{rr}|+\mathcal{Z}_4) \qquad \text{with $\mathcal{Z}_4=|\phi_{trr}|$ or $|\psi_{tr}|$},    
\end{cases}\label{equiv-high-so}\\
&\begin{cases}
\displaystyle\Big|\big(\frac{\phi_r}{r}\big)_r\Big|+\Big|\big(\frac{\psi}{r}\big)_r\Big|\leq C(T)\Big(\Big|\frac{\rho_r}{r}\Big|+\Big|\big(\frac{\rho_r}{r}\big)_r\Big|\Big),\\[6pt]
\displaystyle\Big|\big(\frac{\rho_{r}}{r}\big)_r\Big|\leq C(T)\Big(\Big|\frac{\rho_r}{r}\Big|+\mathcal{Z}_5\Big) \qquad \text{with $\mathcal{Z}_5=\Big|\big(\frac{\phi_r}{r}\big)_r\Big|$ or $\Big|\big(\frac{\psi}{r}\big)_r\Big|$}.    
\end{cases}\notag
\end{align}

As a consequence, following the calculations similar to \eqref{partdivev}--\eqref{high-psi} in the proof of Lemma \ref{Lemma6.12}, together with \eqref{equiv-high-so}, and Lemmas \ref{equiv-initial} and \ref{lemma66-po}--\ref{l4.9-po}, one arrives at
\begin{equation}\label{high-psi-po}
\Big|r^\frac{m}{2}\Big(\rho_{rrr},\big(\frac{\rho_r}{r}\big)_r\Big)\Big|_2\leq C(T)\Big|r^\frac{m}{2}\Big(\rho_r,\rho_{rr},\frac{\rho_r}{r},\psi_{rr},\big(\frac{\psi}{r}\big)_r\Big)\Big|_2\leq C(T).
\end{equation}

Next, repeating the same calculation \eqref{psi-tr} in the proof of Lemma \ref{Lemma6.12}, we see from \eqref{equiv-eu}, \eqref{equiv-eu-high}, \eqref{equiv-high-so}--\eqref{high-psi-po}, and Lemmas \ref{lemma66-po}--\ref{l4.9-po} and \ref{H3-1-po} that 
\begin{align*}
\Big|r^\frac{m}{2}\big(\psi_{tr},\frac{\psi_t}{r}\big)\Big|_2&\leq C_0|(u,\psi)|_\infty |r^\frac{m}{2}(u_{rr},\psi_{rr})|_2+C_0\Big|\big(u_r,\frac{u}{r}\big)\Big|_\infty\Big|r^\frac{m}{2}\big(\psi_r,\frac{\psi}{r}\big)\Big|_2\\
&\quad +C_0\Big|r^\frac{m}{2}\Big(u_{rrr},\frac{u_{rr}}{r},\big(\frac{u}{r}\big)_{rr},\frac{1}{r}\big(\frac{u}{r}\big)_r\Big)\Big|_{2}\\
&\leq C(T)|(u,\rho_r)|_\infty |r^\frac{m}{2}(u_{rr},\rho_r,\rho_{rr},\rho_{rrr})|_2\\
&\quad +C(T)\Big|r^\frac{m}{2}\big(\rho_r,\rho_{rr},\frac{\rho_r}{r}\big)\Big|_2 +C(T)\leq C(T),
\end{align*}
which, along with \eqref{equiv-high-so} and Lemmas \ref{lemma66-po}--\ref{l4.9-po}, leads to
\begin{equation}
\Big|r^\frac{m}{2}\big(\rho_{trr},\frac{\rho_{tr}}{r}\big)\Big|_2\leq \Big|r^\frac{m}{2}\big(\rho_t,\rho_{tr},\rho_{rr},\frac{\rho_r}{r},\psi_{tr},\frac{\psi_t}{r}\big)\Big|_2\leq C(T).
\end{equation}

\smallskip
\textbf{2.} {\bf Estimates on $u$.} Repeating the calculations \eqref{lem612-1} 
and \eqref{lem612-4} in the proof of Lemma \ref{Lemma6.12}, 
we obtain from \eqref{equiv-eu}, \eqref{equiv-eu-high}, \eqref{equiv-high-so}--\eqref{high-psi-po}, 
and Lemmas \ref{lemma66-po}--\ref{l4.9-po} and \ref{H3-1-po} that 
\begin{equation} \label{lem612-1-po}
\begin{aligned}
\Big|r^\frac{m}{2}\Big(\frac{1}{r}\big(u_r+\frac{m}{r}u\big)_{r}\Big)_r\Big|_2 
&\leq  \Big|r^\frac{m}{2}\Big(\big(\frac{u_{t}}{r}\big)_r,\big(\frac{\phi_r}{r}\big)_r\Big)\Big|_2+C_0|(u,\psi)|_\infty|r^\frac{m-2}{2}u_{rr}|_2\\
&\quad+C_0\Big|r^\frac{m}{2}\Big(\big(\frac{u}{r}\big)_r,\big(\frac{\psi}{r}\big)_r\Big)_r\Big|_2|u_r|_\infty\\
&\leq  \Big|r^\frac{m}{2} \big(\frac{u_{t}}{r}\big)_r\Big|_2+C(T)\Big(|\rho_r|_\infty+\Big|r^\frac{m}{2}\Big(\frac{\rho_r}{r},\big(\frac{\rho_r}{r}\big)_r\Big)\Big|_2+1\Big)\\
&\leq \Big|r^\frac{m}{2}\big(\frac{u_{t}}{r}\big)_r\Big|_2 +C(T),\\
\Big|r^\frac{m}{2}\big(u_r+\frac{m}{r}u\big)_{rrr}\Big|_2
&\leq \big|r^\frac{m}{2}(u_{trr},\phi_{rrr})\big|_2+|(u,\psi)|_\infty |r^\frac{m}{2}u_{rrr}|_2+ C_0|r^\frac{m}{2}u_{rr}\psi_r|_2\\
&\quad + C_0|r^\frac{m}{2}(u_{rr},\psi_{rr})|_2|u_r|_\infty\\
&\leq |r^\frac{m}{2}u_{trr}|_2+C(T)|\rho_r|_\infty+C(T)|r^\frac{m}{2}u_{rr}|_2|\rho_r|_\infty \\
&\quad +C(T)|r^\frac{m}{2}u_{rr} \rho_{rr}|_2+C(T)|r^\frac{m}{2}(\rho_r,\rho_{rr},\rho_{rrr})|_2+C(T)\\
&\leq |r^\frac{m}{2}u_{trr}|_2+C(T)|r^\frac{m}{2}u_{rr} \rho_{rr}|_2 +C(T).
\end{aligned}
\end{equation}
Then, for the estimates of $|r^\frac{m}{2}u_{rr}\rho_{rr}|_2$ on the right-hand side of $\eqref{lem612-1-po}_2$, it follows from \eqref{high-psi-po}, Lemmas \ref{important2-po}, \ref{lemma66-po}--\ref{l4.9-po}, \ref{H3-1-po}, \ref{ale1}, and \ref{hardy} that 
\begin{align}
|r^\frac{m}{2}u_{rr}\rho_{rr}|_2&\leq |\chi_1^\flat r^\frac{m}{2}u_{rr}\rho_{rr}|_2+|\chi_1^\sharp r^\frac{m}{2}u_{rr}\rho_{rr}|_2\notag\\
&\leq C_0|\chi_1^\flat r^\frac{m+2}{2}u_{rr}\rho_{rr}|_2+C_0|\chi_1^\flat r^\frac{m+2}{2}u_{rrr}\rho_{rr}|_2\notag\\
&\quad +C_0|\chi_1^\flat r^\frac{m+2}{2}u_{rr}\rho_{rrr}|_2+|\chi_1^\sharp r^\frac{m}{2}u_{rr}\rho_{rr}|_2 \notag\\
&\leq C_0|r^\frac{m}{2}(u_{rr},u_{rrr})|_2|\chi_1^\flat r\rho_{rr}|_\infty\label{lem612-5-po}\\
&\quad + C_0 |r^\frac{m}{2}\rho_{rrr}|_2 |\chi_1^\flat ru_{rr}|_\infty +C_0|r^\frac{m}{2}u_{rr}|_2|\chi_1^\sharp\rho_{rr}|_\infty \notag\\
&\leq C(T)|\chi_1^\flat r^\frac{3}{2}(\rho_{rr},\rho_{rrr},u_{rr},u_{rrr})|_2 +C(T) |\chi_1^\sharp(\rho_{rr},\rho_{rrr})|_2 \notag\\
&\leq C(T)\big(|\chi_1^\flat r^\frac{3-m}{2}|_\infty+|\chi_1^\sharp r^{-\frac{m}{2}}|_\infty\big)|r^\frac{m}{2}(\rho_{rr},\rho_{rrr},u_{rr},u_{rrr})|_2 \leq C(T).\notag
\end{align}

Combining with  \eqref{lem612-1-po}--\eqref{lem612-5-po}, together with Lemmas \ref{im-1} and \ref{H3-1-po}, gives the desired estimates of this lemma. 
\end{proof}

\begin{rk}\label{rk11.1}
As shown by \eqref{equiv-eu}, \eqref{equiv-eu-high}, \eqref{equiv-high-so}, 
and {\rm Lemmas \ref{lemma66-po}--\ref{l4.9-po}} and {\rm \ref{Lemma6.12-po}}, 
each of the following $L^2(I)$-norms that we use:
\begin{equation*}
|r^\frac{m}{2}X|_2,\ |r^\frac{m}{2}X_r|_2,\ \big|r^\frac{m-2}{2}X\big|_2,\ |r^\frac{m}{2}X_{rr}|_2,\ \Big|r^\frac{m}{2}\big(\frac{X}{r}\big)_r\Big|_2,\ |r^\frac{m}{2}X_t|_2,\ |r^\frac{m}{2}X_{tr}|_2,\ \big|r^\frac{m-2}{2}X_t\big|_2 
\end{equation*}
with $X=\phi_r$ or $\psi$, can be controlled by the corresponding $L^2(I)$-norms 
with $X$ replaced by $\rho_r$, combined with some $L^2(I)$-norms of the
lower-order derivatives of $\rho$, for instance, 
\begin{equation*}
|r^\frac{m}{2}X_{tr}|_2\leq C(T)|r^\frac{m}{2}\rho_{trr}|_2+C(T)\underbrace{\big|r^\frac{m}{2}(\rho_t,\rho_{tr},\rho_{rr})\big|_2}_{\text{lower-order norms}}\leq C(T).
\end{equation*}
\end{rk}

Now, we establish the time-weighted estimates for $u$. The following two lemmas are similar to Lemmas \ref{Lemma6.13}--\ref{Lemma6.14}.
\begin{lem}\label{Lemma6.13-po}  
There exists a constant $C(T)>0$ such that, for any $t\in [0,T]$,
\begin{equation*}
\begin{split}
&t|r^{\frac{m}{2}}u_{tt}(t)|^2_2+\int_{0}^{t} s \Big|r^{\frac{m}{2}}\big(u_{ttr},\frac{u_{tt}}{r}\big)\Big|^2_2 \,\mathrm{d}s\leq C(T).
\end{split}
\end{equation*}
\end{lem}
\begin{proof}
First, it follows from the same calculation \eqref{lem613-1} in the proof of Lemma \ref{Lemma6.13} that
\begin{equation}\label{lem613-1-po}
|u_t|_\infty+|\chi_1^\flat ru_{tr}|_\infty +|\chi_1^\sharp u_{tr}|_\infty  \leq C(T) \big(|r^\frac{m}{2}u_{trr}|_2+1\big).
\end{equation}
Next, based on Remark \ref{rk11.1}, repeating the same calculations \eqref{lem613-1'}--\eqref{phittr} in the proof of Lemma \ref{Lemma6.13}, one obtains that
\begin{equation}\label{phittr-psittr}
\big|r^\frac{m}{2}(\phi_{ttr},\psi_{tt})\big|_2 \leq C(T) \Big(\Big|r^\frac{m}{2}\Big(u_{trr},\big(\frac{u_t}{r}\big)_{r}\Big)\Big|_2+1\Big).
\end{equation}
As a consequence, with the help of \eqref{lem613-1-po}--\eqref{phittr-psittr}, the expected estimate follows directly from the same arguments as in Steps 2--3 of the proof of Lemma \ref{Lemma6.13}.
\end{proof}

\begin{lem}\label{Lemma6.14-po}  
There exists a constant $C(T)>0$ such that, for any $t\in [0,T]$,
\begin{equation*}
\begin{aligned}
&\sqrt{t}\Big|r^{\frac{m}{2}}\Big(u_{trr},\big(\frac{u_{t}}{r}\big)_r,u_{rrrr},\big(\frac{u_{rr}}{r}\big)_r,\big(\frac{u}{r}\big)_{rrr},\Big(\frac{1}{r}\big(\frac{u}{r}\big)_r\Big)_r\Big)(t)\Big|_2\\
&+\int_0^t s\Big|r^{\frac{m}{2}}\Big(u_{trrr},\frac{u_{trr}}{r},\big(\frac{u_t}{r}\big)_{rr},\frac{1}{r}\big(\frac{u_{tr}}{r}\big)_r\Big)\Big|_2^2\,\mathrm{d}s\leq C(T).
\end{aligned}
\end{equation*}
\end{lem}
\begin{proof}
Using Remark \ref{rk11.1} and the same argument as in the proof of Lemma \ref{Lemma6.14}, we can obtain the desired estimates of this lemma. We omit the details.
\end{proof}

\subsection{Global well-posedness of the M-D regular solutions with strictly positive initial density}\label{subsub-final} This part can be proved via the same argument as in \S \ref{se46}.

\smallskip
\section{Local Well-Posedness of Regular Solutions with Far-Field Vacuum}\label{section-local-regular}

This section is devoted to establishing the local existence of the unique regular solution of the Cauchy problem  \eqref{eq:1.1benwen}--\eqref{e1.3} in the M-D space variables when $\bar\rho=0$. Moreover, we show that, 
if the initial data are spherically symmetric, so is the corresponding M-D regular solution of this problem in late time. 
In the rest of \S \ref{section-local-regular}, $C\in [1,\infty)$ denotes a generic constant depending only on  $(n,\alpha,A,\gamma)$ and may be different at each occurrence. 
 
\subsection{Local well-posedness of the 2-order regular solutions 
with far-field vacuum}\label{section-local-2-regular}
We prove  Theorem \ref{thh1} in \S \ref{linear2}--\S\ref{bani}. 
Moreover, at the end of \S \ref{section-local-2-regular}, we show that 
this theorem indeed implies Theorem \ref{zth1}.

\subsubsection{Linearization}\label{linear2}
We start with the proof of Theorem \ref{thh1} by considering  the following linearized problem of $(\phi, \boldsymbol{u},\boldsymbol{\psi})$ in $[0,T]\times \mathbb{R}^n$:
\begin{equation}\label{li4}
\begin{cases}
\phi_t+\boldsymbol{w}\cdot \nabla\phi+(\gamma-1)\phi \diver\boldsymbol{w}=0,\\[5pt]
\boldsymbol{u}_t+\boldsymbol{w}\cdot \nabla \boldsymbol{w}+\nabla \phi+L\boldsymbol{u}=\boldsymbol{\psi}\cdot Q(\boldsymbol{w}),\\ 
\displaystyle\boldsymbol{\psi}_t+\sum\limits_{l=1}^nA_l(\boldsymbol{w})\partial_{l}\boldsymbol{\psi}+B(\boldsymbol{w}) \boldsymbol{\psi}+\nabla\diver\boldsymbol{w}=\boldsymbol{0},\\ 
\displaystyle (\phi,\boldsymbol{u},\boldsymbol{\psi})|_{t=0}=\Big(\phi_0,\boldsymbol{u}_0,\boldsymbol{\psi}_0=\frac{1}{\gamma-1}\nabla\log\phi_0\Big) \ \ \text{for} \ \ \boldsymbol{x}\in\mathbb{R}^n,\\[6pt]
(\phi,\boldsymbol{u})\to (0,\boldsymbol{0}) \ \ \text{as} \ \ |\boldsymbol{x}|\to \infty \ \ \text{for} \ \ t\in [0,T],
\end{cases}
\end{equation}
where the operators $(L,Q)$ are defined in \eqref{operatordefinition}, $A_l(\boldsymbol{w})$ ($l=1,\cdots\!,n$) and $B(\boldsymbol{w})$ are defined in \S \ref{see1}.
Moreover, we assume that the initial data $(\phi_0,\boldsymbol{u}_0,\boldsymbol{\psi}_0)$ are spherically symmetric, taking form \eqref{eqs:CauchyInit'}, and satisfy 
\begin{equation}\label{linear-initial}
\begin{gathered}
\phi_0>0,\quad \phi_0^\frac{1}{\gamma-1}\in L^1(\mathbb{R}^n),\quad \nabla \phi_0\in H^1(\mathbb{R}^n),\quad 
\boldsymbol{\psi}_0\in  D^1(\mathbb{R}^n),\quad \boldsymbol{u}_0\in H^2(\mathbb{R}^n),
\end{gathered}
\end{equation}
and $\boldsymbol{w}=(w_1,\cdots\!,w_n)^{\top}\in \mathbb{R}^n$ is a given spherically symmetric vector function satisfying $\boldsymbol{w}(0,\boldsymbol{x})=\boldsymbol{u}_0(\boldsymbol{x})$ and
\begin{equation}\label{3.3}
\boldsymbol{w}(t,\boldsymbol{x})=w(t,|\boldsymbol{x}|)\frac{\boldsymbol{x}}{|\boldsymbol{x}|},
\end{equation}
and, for any $T>0$,
\begin{equation}\label{vg}
\begin{aligned}
&\partial_t^l \boldsymbol{w}\in C([0,T]; H^{2-2l}(\mathbb{R}^n))\cap L^2([0,T]; D^{3-2l}(\mathbb{R}^n)),\quad l=0,1,\\
&\sqrt{t}\partial_t^l\boldsymbol{w}\in L^{\infty}([0,T]; D^{3-2l}(\mathbb{R}^n)),\quad \sqrt{t} \partial_t^{l+1}\boldsymbol{w} \in  L^2([0,T]; D^{2-2l}(\mathbb{R}^n)),\quad l=0,1.
\end{aligned}
\end{equation}

Then we have the following global well-posedness for the linearized problem \eqref{li4}, which can be obtained by classical  arguments shown in  \cite{CK3, evans}. 
\begin{lem}\label{lem1}
Let $n=2$ or $3$, and  \eqref{cd1local} hold. Let the initial data $(\phi_0,\boldsymbol{u}_0,\boldsymbol{\psi}_0) (\boldsymbol{x})$ be spherically symmetric and satisfy \eqref{linear-initial}. 
Then, for any $T>0$, there exists a unique solution $(\phi,\boldsymbol{u},\boldsymbol{\psi})(t,\boldsymbol{x})$ of the linearized problem \eqref{li4} in $[0,T]\times\mathbb{R}^n$ such that \eqref{psi=log-phi-0} holds, and
\begin{align*}
\mathrm{(i)}& \ (\phi>0, \boldsymbol{u}) \ \text{satisfies this problem in the sense of distributions};\\
\mathrm{(ii)}& \ \phi^\frac{1}{\gamma-1}\in C([0,T];L^1(\mathbb{R}^n)),\ \ \nabla\phi\in C([0,T];H^{1}(\mathbb{R}^n)), \ \ \phi_t \in C([0,T]; H^{1}(\mathbb{R}^n)),\\
& \ \boldsymbol{\psi}\in C([0,T];D^1(\mathbb{R}^n)), \quad \boldsymbol{\psi}_t \in C([0,T]; L^2(\mathbb{R}^n));\\
\mathrm{(iii)}& \ \boldsymbol{u}(t,\boldsymbol{x})|_{|\boldsymbol{x}|=0}=\boldsymbol{0} \ \ \text{for} \ \ t\in [0,T],\\
& \ \partial_t^l \boldsymbol{u}\in C([0,T]; H^{2-2l}(\mathbb{R}^n))\cap L^2([0,T]; D^{3-2l}(\mathbb{R}^n)),\quad l=0,1;\\
\mathrm{(iv)}& \ \sqrt{t}\partial_t^l\boldsymbol{u}\in L^{\infty}([0,T]; D^{3-2l}(\mathbb{R}^n)), \ \ \sqrt{t} \partial_t^{l+1}\boldsymbol{u} \in  L^2([0,T]; D^{2-2l}(\mathbb{R}^n)),\ \ l=0,1.
\end{align*}
Moreover, $(\phi,\boldsymbol{u},\boldsymbol{\psi})$ is spherically symmetric with  form \eqref{1.9'}. 
\end{lem}

\begin{proof} We divide the proof into three steps.

\smallskip
\textbf{1.} We can first solve for $\phi$ from $\eqref{li4}_1$ via the standard theory of transport equations. Next, the well-posedness of $\boldsymbol{\psi}$ follows from the standard theory for the symmetric hyperbolic systems. Finally, rewrite $\eqref{li4}_2$ as
\begin{equation*}
{u}_t+L\boldsymbol{u}=-\boldsymbol{w}\cdot \nabla \boldsymbol{w}-\nabla \phi+\boldsymbol{\psi}\cdot Q(\boldsymbol{w}).
\end{equation*}
Then we can solve for $u$ from the above via the regularity properties of $(\phi,\boldsymbol{\psi})$ and the standard theory of linear parabolic equations.

\smallskip
\textbf{2.} Next, we give a brief proof for \eqref{psi=log-phi-0}. Let $\tilde{\boldsymbol{\psi}}:=\frac{1}{\gamma-1}\nabla\log\phi$. Then multiplying $\eqref{li4}_1$ by $\frac{1}{(\gamma-1)\phi}$ and applying $\nabla$ to the resulting equality yield 
\begin{equation}
\tilde{\boldsymbol{\psi}}_t+\sum_{l=1}^n A_l(\boldsymbol{w})\partial_l\tilde{\boldsymbol{\psi}}+B(\boldsymbol{w}) \tilde{\boldsymbol{\psi}}+\nabla \diver\boldsymbol{w}=0.
\end{equation}
This implies that $\boldsymbol{\psi}$ and $\tilde{\boldsymbol{\psi}}$ satisfy the same equation and take the same initial data. 
Therefore, the uniqueness of the solution yields \eqref{psi=log-phi-0}.

\smallskip
\textbf{3.} Finally, we show that the solution $(\phi,\boldsymbol{u},\boldsymbol{\psi})(t,\boldsymbol{x})$ is spherically symmetric and takes form \eqref{1.9'}. 
Indeed, since $(\phi_0,\boldsymbol{u}_0,\boldsymbol{\psi}_0)$ is spherically symmetric, for any  $\mathcal{O}\in \mathrm{SO}(n)$, 
\begin{equation}\label{vl}
\begin{gathered}
\phi_0(\boldsymbol{x})=\phi_0(\mathcal{O}\boldsymbol{x}), \ \  \boldsymbol{u}_0(\boldsymbol{x})= \mathcal{O}^\top\boldsymbol{u}_0(\mathcal{O}\boldsymbol{x}), \ \ \boldsymbol{\psi}_0(\boldsymbol{x})= \mathcal{O}^\top\boldsymbol{\psi}_0(\mathcal{O}\boldsymbol{x}).
\end{gathered}
\end{equation}
Thus, if denoting 
\begin{equation*}
\begin{gathered}
\hat\phi(t,\boldsymbol{x})=\phi(t,\mathcal{O}\boldsymbol{x}),\quad \hat{\boldsymbol{u}}(t,\boldsymbol{x})=\mathcal{O}^{\top}\boldsymbol{u}(t,\mathcal{O}\boldsymbol{x}),\quad \hat{\boldsymbol{\psi}}(t,\boldsymbol{x})=\mathcal{O}^{\top}\boldsymbol{\psi}(t,\mathcal{O}\boldsymbol{x}),
\end{gathered}
\end{equation*}
one obtains from  $|\mathcal{O}\boldsymbol{x}|=|\boldsymbol{x}|$, \eqref{e1.3'}, and \eqref{vl} that
\begin{align*}
&(\hat\phi,\hat{\boldsymbol{u}},\hat{\boldsymbol{\psi}})|_{t=0}=(\hat\phi_0,\hat{\boldsymbol{u}}_0,\hat{\boldsymbol{\psi}}_0)=(\phi_0,\boldsymbol{u}_0,\boldsymbol{\psi}_0),\\
&\hat{\boldsymbol{u}}(t,\boldsymbol{x})|_{|\boldsymbol{x}|=0} =\mathcal{O}^{\top}\boldsymbol{u}(t, \mathcal{O}\boldsymbol{x})|_{|\boldsymbol{x}|=0} =\boldsymbol{0} \qquad \text{for $t\in [0,T]$},\\
&(\hat\phi,\hat{\boldsymbol{u}})\to (0,\boldsymbol{0}) \qquad \text{for $t\in [0,T]$}.
\end{align*}
 
Next, we show that $(\hat\phi,\hat{\boldsymbol{u}},\hat{\boldsymbol{\psi}})$ is also a solution of \eqref{li4}.  For any $\mathcal{O}=(\mathcal{O}_{lk})_{n\times n}\in \mathrm{SO}(n)$, let $\boldsymbol{x}=(x_1,\cdots\!, x_n)^\top$ and $\boldsymbol{y}=(y_1,\cdots\!, y_n)^\top$ satisfy
\begin{equation}\label{sbo}
\boldsymbol{y}=\mathcal{O}\boldsymbol{x}\qquad  \text{{\it i.e.}, $\,\,y_l=\sum_{k=1}^n\mathcal{O}_{lk}x_k$}.
\end{equation}
Clearly, $\sum_{k=1}^n\mathcal{O}_{ki}\mathcal{O}_{kj}=\delta_{ij}$ and  $\boldsymbol{w}(t,\boldsymbol{x})=\mathcal{O}^\top \boldsymbol{w}(t,\boldsymbol{y})$.  For convenience, in Steps 3.1--3.3 below, 
We adopt the Einstein summation convention: any index that appears exactly twice in a term is summed over.

\smallskip
\textbf{3.1.} We first show that $\hat\phi$ satisfies $\eqref{li4}_1$ 
by the following direct calculations: 
\begin{align*}
\hat\phi_t(t,\boldsymbol{x})&=\phi_t(t,\boldsymbol{y}),\\
(\boldsymbol{w}\cdot\nabla\hat\phi)(t,\boldsymbol{x})&=(\mathcal{O}_{ki}\mathcal{O}_{li}w_l\partial_{y_k}\phi) (t,\boldsymbol{y})\\
&=(\delta_{kl}w_l\partial_{y_k}\phi)(t,\boldsymbol{y})=(\boldsymbol{w}\cdot\nabla_y\phi)(t,\boldsymbol{y}),\\
(\hat\phi\diver\boldsymbol{w})(t,\boldsymbol{x})&=(\phi\mathcal{O}_{kj}\mathcal{O}_{lj} \partial_{y_k}w_l)(t,\boldsymbol{y})\\
&=(\phi \delta_{kl} \partial_{y_k}w_l) (t,\boldsymbol{y})=(\phi\diver_y\boldsymbol{w})(t,\boldsymbol{y}).
\end{align*}
These identities imply that $\hat\phi(t,\boldsymbol{x})$ satisfies  $\eqref{li4}_1$ indeed.

\smallskip
\textbf{3.2.} Next, we show that $(\hat\phi,\hat{\boldsymbol{u}},\hat{\boldsymbol{\psi}})$ satisfies $\eqref{li4}_2$. 
Notice that
\begin{equation*}
\begin{aligned}
\hat{\boldsymbol{u}}_t(t,\boldsymbol{x})&=(\mathcal{O}_{qi} \partial_tu_q)(t,\boldsymbol{y})=(\mathcal{O}^{\top} \boldsymbol{u}_t)(t,\boldsymbol{y}),\\
(\boldsymbol{w}\cdot\nabla\boldsymbol{w})(t,\boldsymbol{x})
&=(\mathcal{O}_{lj}\mathcal{O}_{kj}\mathcal{O}_{qi} w_l \partial_{y_k}w_q) (t,\boldsymbol{y})=(\delta_{lk}\mathcal{O}_{qi} w_l \partial_{y_k}w_q) (t,\boldsymbol{y})\\
&=(\mathcal{O}_{qi} w_k \partial_{y_k}w_q) (t,\boldsymbol{y})=(\mathcal{O}^\top (\boldsymbol{w}\cdot\nabla_y \boldsymbol{w}))(t,\boldsymbol{y}),\\
\nabla\hat\phi(t,\boldsymbol{x})&=(\mathcal{O}_{ki}  \partial_{y_k}\phi) (t,\boldsymbol{y})=(\mathcal{O}^\top\nabla_y\phi)(t,\boldsymbol{y}),\\
\Delta \hat{\boldsymbol{u}}(t,\boldsymbol{x})&=(\mathcal{O}_{lj} \mathcal{O}_{kj} \mathcal{O}_{qi}\partial_{y_ly_k}u_q)(t,\boldsymbol{y})\\
&=(\delta_{lk}  \mathcal{O}_{qi}\partial_{y_ly_k}u_q)(t,\boldsymbol{y})=(\mathcal{O}^\top \Delta_y \boldsymbol{u})(t,\boldsymbol{y}),\\
\nabla\diver\hat{\boldsymbol{u}}(t,\boldsymbol{x})&=(\mathcal{O}_{li}\mathcal{O}_{kj}\mathcal{O}_{qj}\partial_{y_ly_k}u_q)(t,\boldsymbol{y})\\
&=(\mathcal{O}_{li}\delta_{kq}\partial_{y_ly_k}u_q)(t,\boldsymbol{y})=(\mathcal{O}^\top \nabla_y\diver_y\boldsymbol{u}) (t,\boldsymbol{y}),\\
2(\hat{\boldsymbol{\psi}}\cdot D(\boldsymbol{w}))(t,\boldsymbol{x})
&=(\mathcal{O}_{kj} \psi_k (\mathcal{O}_{lj}\mathcal{O}_{pi} \partial_{y_l}w_p+\mathcal{O}_{li}\mathcal{O}_{pj} \partial_{y_l}w_p))(t,\boldsymbol{y})\\
&=(\psi_k (\delta_{kl}\mathcal{O}_{pi} \partial_{y_l}w_p+\delta_{kp}\mathcal{O}_{li} \partial_{y_l}w_p))(t,\boldsymbol{y})\\
&=(\mathcal{O}_{li}\psi_k (\partial_{y_k}w_l+ \partial_{y_l}w_k))(t,\boldsymbol{y})=2(\mathcal{O}^\top(\boldsymbol{\psi}\cdot D_y(\boldsymbol{w})))(t,\boldsymbol{y}).
\end{aligned}
\end{equation*}
The above identities imply that $(\hat\phi,\hat{\boldsymbol{u}},\hat{\boldsymbol{\psi}})$ satisfies $\eqref{li4}_2$.

\smallskip
\textbf{3.3.} Finally, we show that $\hat{\boldsymbol{\psi}}$ satisfies $\eqref{li4}_3$. 
Since 
\begin{align*}
\hat{\boldsymbol{\psi}}_t(t,\boldsymbol{x})&=(\mathcal{O}_{qi}\partial_t\psi_q)(t,\boldsymbol{y})=(\mathcal{O}^{\top} \boldsymbol{\psi}_t)(t,\boldsymbol{y}),\\
(A_j(\boldsymbol{w}) \partial_j\hat{\boldsymbol{\psi}})(t,\boldsymbol{x})
&=(\mathcal{O}_{lj}\mathcal{O}_{kj}\mathcal{O}_{qi} w_l\partial_{y_k}\psi_{q}) (t,\boldsymbol{y})=(\delta_{lk}\mathcal{O}_{qi} w_l\partial_{y_k}\psi_{q}) (t,\boldsymbol{y})\\
&=(\mathcal{O}_{qi} w_k \partial_{y_k}\psi_{q})(t,\boldsymbol{y})=(\mathcal{O}^\top(A_k(\boldsymbol{w}) \partial_{y_k}\boldsymbol{\psi}))(t,\boldsymbol{y}),\\
(B(\boldsymbol{w}) \hat{\boldsymbol{\psi}})(t,\boldsymbol{x})
&=(\mathcal{O}_{ki}\mathcal{O}_{lj}\mathcal{O}_{qj} \partial_{y_k}w_l \psi_q)(t,\boldsymbol{y}) =(\mathcal{O}_{ki}\delta_{lq} \partial_{y_k}w_l \psi_q)(t,\boldsymbol{y})\\
&=(\mathcal{O}_{ki} \partial_{y_k}w_q \psi_q)(t,\boldsymbol{y})=(\mathcal{O}^\top(B_y(\boldsymbol{w})  \boldsymbol{\psi}))(t,\boldsymbol{y}),\\
\nabla\diver\boldsymbol{w}(t,\boldsymbol{x})&=(\mathcal{O}_{li}\mathcal{O}_{kj}\mathcal{O}_{qj} \partial_{y_l y_k}w_q) (t,\boldsymbol{y}) \\
&=(\mathcal{O}_{li}\delta_{kq} \partial_{y_l y_k}w_q) (t,\boldsymbol{y}) =(\mathcal{O}^\top \nabla_y\diver_y\boldsymbol{w}) (t,\boldsymbol{y}),
\end{align*}
with $B_y(\boldsymbol{w})=(\nabla_y \boldsymbol{w})^\top$, we see that $\hat{\boldsymbol{\psi}}$ satisfies  $\eqref{li4}_3$. 

\smallskip
\textbf{3.4.} In Steps 3.1--3.3, we have shown that $(\hat \phi,\hat{\boldsymbol{u}},\hat{\boldsymbol{\psi}})$  is also a solution of 
the linearized problem \eqref{li4}, which takes the initial value $(\phi_0,\boldsymbol{u}_0,\boldsymbol{\psi}_0)$. Consequently, the uniqueness of the solution implies that 
\begin{equation}\label{cls}
\phi(t,\mathcal{O}\boldsymbol{x})=\phi(t,\boldsymbol{x}), \qquad (\mathcal{O}^{\top}\boldsymbol{u})(t,\mathcal{O}\boldsymbol{x})=\boldsymbol{u}(t,\boldsymbol{x}), \qquad (\mathcal{O}^{\top}\boldsymbol{\psi})(t,\mathcal{O}\boldsymbol{x})=\boldsymbol{\psi}(t,\boldsymbol{x}).
\end{equation}
Clearly, it follows from $\eqref{cls}_1$ that  $\phi(t,\boldsymbol{x})=\phi(t,|\boldsymbol{x}|)$ with some function $\phi(t,r)$ defined on $[0,T]\times I$. 
It remains to show that
\begin{equation*}
(\boldsymbol{u},\boldsymbol{\psi})(t,\boldsymbol{x})
=(u(t,|\boldsymbol{x}|)\frac{\boldsymbol{x}}{|\boldsymbol{x}|},\psi(t,|\boldsymbol{x}|)\frac{\boldsymbol{x}}{|\boldsymbol{x}|})
\end{equation*} 
with some function $(u,\psi)(t,r)$ defined on $[0,T]\times I$. 
Here, we take the 3-D case and $\boldsymbol{u}(t,\boldsymbol{x})$ as an example, since
the proof for the 2-D case and $\boldsymbol{\psi}(t,\boldsymbol{x})$ can be derived similarly. 
Let $t_0\in [0,T]$ be any fixed time and $\boldsymbol{x}_0\in\mathbb{R}^3$ be any fixed displacement vector. Let $\boldsymbol{e}_1=\frac{\boldsymbol{x}_0}{|\boldsymbol{x}_0|}$ and $\mathcal{O}_1\in \mathrm{SO}(n)$ be a rotation by 180 degrees about an axis parallel to $\boldsymbol{x}_0$. Then $\eqref{cls}_2$ yields
\begin{equation}\label{cls'}
\mathcal{O}_1\boldsymbol{x}_0=\boldsymbol{x}_0 ,\qquad (\mathcal{O}_1\boldsymbol{u})(t_0,\boldsymbol{x}_0)=\boldsymbol{u}(t_0,\boldsymbol{x}_0).
\end{equation}

Next, let $\{\boldsymbol{e}_2,\boldsymbol{e}_3\}$ be two unit vectors such that $\{\boldsymbol{e}_1,\boldsymbol{e}_2,\boldsymbol{e}_3\}$ becomes an orthonormal basis in $\mathbb{R}^3$. 
Then there exist some constants $a_i=a_i(t_0,\boldsymbol{x}_0)\in \mathbb{R}$ $(i=1,2,3)$ depending only on $(t_0,\boldsymbol{x}_0)$ such that 
\begin{equation*}
\boldsymbol{u}(t_0,\boldsymbol{x}_0)=a_1\boldsymbol{e}_1+a_2\boldsymbol{e}_2+a_3\boldsymbol{e}_3.
\end{equation*}
This, together with \eqref{cls'}, gives 
\begin{equation*}
\begin{aligned}
& \ \, a_1\boldsymbol{e}_1+a_2\boldsymbol{e}_2+a_3\boldsymbol{e}_3=a_1\boldsymbol{e}_1-a_2\boldsymbol{e}_2-a_3\boldsymbol{e}_3\\[2pt]
&\implies a_2\boldsymbol{e}_2+a_3\boldsymbol{e}_3=\boldsymbol{0}\\
&\implies  a_2=a_3=0 \ \ \text{due to the linear independence of }\{\boldsymbol{e}_2,\boldsymbol{e}_3\},
\end{aligned}    
\end{equation*}
and hence, for any fixed $(t_0,\boldsymbol{x}_0)\in [0,T]\times \mathbb{R}^3$,
\begin{equation*}
\boldsymbol{u}(t_0,\boldsymbol{x}_0)=a_1\boldsymbol{e}_1=a_1(t_0,\boldsymbol{x}_0)\frac{\boldsymbol{x}_0}{|\boldsymbol{x}_0|}.
\end{equation*}
Based on $\eqref{cls}_2$, we see that $a_1(t_0,\boldsymbol{x}_0)=a_1(t_0,|\boldsymbol{x}_0|)$, which yields $\boldsymbol{u}(t,\boldsymbol{x})=u(t,|\boldsymbol{x}|)\frac{\boldsymbol{x}}{|\boldsymbol{x}|}$ with some function $u(t,r)$ defined on $[0,T]\times I$. Finally, $\boldsymbol{u}(t,\boldsymbol{x})|_{|\boldsymbol{x}|=0}=\boldsymbol{0}$ follows from Lemma \ref{rmk31} in Appendix \ref{appA}. 

The proof of Lemma \ref{lem1} is completed.
\end{proof}

\subsubsection{The uniform {\it a priori} estimates}\label{uape}
Let $(\phi,\boldsymbol{u},\boldsymbol{\psi})$ be a solution in $[0,T]\times\mathbb{R}^n$ obtained in Lemma \ref{lem1}. We now establish the corresponding  {\it a priori} estimates. 

We first choose a positive constant $c_0$ such that 
\begin{equation}\label{houmian}
2+\big\|\phi_0^{\frac{1}{\gamma-1}}\big\|_{L^1}+\|\nabla\phi_0\|_{H^1}+\|\boldsymbol{u}_0\|_{H^2}+\|\boldsymbol{\psi}_0\|_{D^1}\leq  c_0.
\end{equation}
Assume there exist $T^*\in (0,T)$ and constants $(c_1,c_2,c_3)$ 
such that $1< c_0\leq c_1 \leq c_2 \leq c_3$ and
\begin{equation}\label{jizhu1}
\begin{aligned}
\sup_{0\leq t \leq T^*}\|\boldsymbol{w}\|^2_{H^1}+\int_{0}^{T^*}\big(\|\nabla \boldsymbol{w}\|^2_{H^1}+\|\boldsymbol{w}_t\|_{L^2}^2\big)\,\mathrm{d}t \leq c^2_1,\\
\sup_{0\leq t \leq T^*}\|(\nabla^2\boldsymbol{w},\boldsymbol{w}_t)\|^2_{L^2}+\int_{0}^{T^*}\|(\nabla^3\boldsymbol{w},\nabla\boldsymbol{w}_t)\|^2_{L^2}\,\mathrm{d}t \leq c^2_2,\\
\sup_{0\leq t \leq T^*} t\|(\nabla^3\boldsymbol{w},\nabla\boldsymbol{w}_t)\|^2_{L^2}+\int_{0}^{T^*} t\|(\nabla^2\boldsymbol{w}_t,\boldsymbol{w}_{tt})\|^2_{L^2}\,\mathrm{d}t \leq c^2_3.
\end{aligned}
\end{equation}
Here,  $T^*$ and  $c_i$ ($i=1,2,3$) will be determined later, which depend only on $c_0$ 
and the fixed constants $(\alpha,\gamma,A,T)$.

We first give the estimates for $(\phi,\boldsymbol{\psi})$. 
\begin{lem}\label{bos-2-regular} 
For any $t\in [0,T_1]$ with $T_1:=\min \{T^{*}, (1+Cc_3)^{-2}\}$,
\begin{align*}
&\big\|\phi^\frac{1}{\gamma-1}(t)\big\|_{L^1}\leq c_0, \quad \|\phi(t)\|_{L^\infty}+ \|\nabla\phi(t)\|_{H^1}\leq Cc_0^{7\gamma-3},\\[1mm]
&\|\nabla^{k-1}\phi_t(t)\|_{L^2}\leq Cc_0^{7\gamma-3}c_k \quad\text{for $k=1,2$},\\[1mm]
&\|\boldsymbol{\psi}(t)\|_{D^1}\leq Cc_0, \quad\, \|\boldsymbol{\psi}_t(t)\|_{L^2}\leq Cc_0c_2.
\end{align*}
\end{lem}

\smallskip
\begin{proof} We divide the proof into five steps.

\smallskip
\textbf{1. $L^1(\mathbb{R}^n)$-estimate on $\phi^\frac{1}{\gamma-1}$.} 
First, we multiply $\eqref{li4}_1$ by $\phi^{\frac{2-\gamma}{\gamma-1}}$ to obtain
\begin{equation*}
\frac{\mathrm{d}}{\mathrm{d}t}\phi^\frac{1}{\gamma-1}+\diver\big(\phi^\frac{1}{\gamma-1}\boldsymbol{w}\big)=0.
\end{equation*}
Then integrating the above over $\mathbb{R}^n$, along with \eqref{houmian}, yields 
\begin{equation}\label{gb-2-regular}
\begin{aligned}
\big\|\phi^{\frac{1}{\gamma-1}}(t)\big\|_{L^1}=\big\|\phi_0^{\frac{1}{\gamma-1}}\big\|_{L^1}\leq c_0 \qquad \text{for $t\in [0,T^*]$}.
\end{aligned}
\end{equation}

\smallskip
\textbf{2. $D^1(\mathbb{R}^n)$-estimates on $\boldsymbol{\psi}$.} Applying $\partial^\varsigma\boldsymbol{\psi}\partial^\varsigma$ with the multi-index $|\varsigma|=1$ to $\eqref{li4}_3$, then integrating the resulting equality over $\mathbb{R}^n$, we obtain 
from integration by parts, \eqref{jizhu1}, Lemmas \ref{ale1}, \ref{lemma-L6}, and \ref{Hk-Ck-vector}, and the H\"older and Young inequalities that
\begin{equation*}
\begin{aligned}
\frac{\mathrm{d}}{\mathrm{d}t} \|\partial^\varsigma\boldsymbol{\psi}\|_{L^2}^2 
&\leq C\big(\|\nabla \boldsymbol{w}\|_{L^\infty}\| \boldsymbol{\psi}\|_{D^1} +\delta_{2n}\|\nabla^2 \boldsymbol{w}\|_{L^2}\|\boldsymbol{\psi}\|_{D^1}\big) \\
&\quad +C\big(\delta_{3n}\|\nabla^2 \boldsymbol{w}\|_{H^1}\|\boldsymbol{\psi}\|_{D^1} +\|\nabla^3\boldsymbol{w}\|_{L^2}\big)\|\boldsymbol{\psi}\|_{D^1} 
\leq Cc_2^2\big(\|\boldsymbol{\psi}\|_{D^1}^2 +1\big),
\end{aligned}
\end{equation*}
which,  along with the Gr\"onwall inequality, yields that, for $t\in [0,T_1]$,
\begin{equation}\label{nabla-psi-2-regular}
\|\boldsymbol{\psi}(t)\|_{D^1}^2\leq e^{Cc_2^2t}\big(\|\boldsymbol{\psi}_0\|_{D^1}^2+Cc_2^2t\big)\leq Cc_0^2.
\end{equation} 

\medskip
\textbf{3. $L^\infty(\mathbb{R}^n)$-estimate on $\phi$.}
First, it follows from \eqref{psi=log-phi-0}, \eqref{gb-2-regular}--\eqref{nabla-psi-2-regular}, Lemma \ref{GN-ineq}, and the H\"older and Young inequalities that
\begin{equation}\label{eq:I1}
\begin{aligned}
\|\phi\|_{L^\infty}^\frac{1}{\gamma-1}&\leq \big\|\phi^\frac{1}{\gamma-1}\big\|_{L^\infty}\leq C\big\|\phi^\frac{1}{\gamma-1}\big\|_{L^1}^\frac{4-n}{n+4}\big\|\nabla^2(\phi^\frac{1}{\gamma-1})\big\|_{L^2}^\frac{2n}{n+4}\\
&\leq C\big\|\phi^\frac{1}{\gamma-1}\big\|_{L^1}^\frac{4-n}{n+4}\big(\underline{\big\|\phi^\frac{1}{\gamma-1}|\boldsymbol{\psi}|^2\big\|_{L^2}}_{:=I_1}+\big\|\phi^\frac{1}{\gamma-1}\nabla\boldsymbol{\psi}\big\|_{L^2}\big)^\frac{2n}{n+4}\\
&\leq C\big\|\phi^\frac{1}{\gamma-1}\big\|_{L^1}^\frac{4-n}{n+4}\big(I_1+\|\phi\|_{L^\infty}^\frac{1}{\gamma-1}\|\nabla\boldsymbol{\psi}\|_{L^2}\big)^\frac{2n}{n+4}\\
&\leq C\big\|\phi^\frac{1}{\gamma-1}\big\|_{L^1}^\frac{4-n}{n+4} I_1^\frac{2n}{n+4}+C\big\|\phi^\frac{1}{\gamma-1}\big\|_{L^1}^\frac{4-n}{n+4}\|\phi\|_{L^\infty}^\frac{2n}{(\gamma-1)(n+4)}\|\nabla\boldsymbol{\psi}\|_{L^2}^\frac{2n}{n+4}\\
&\leq Cc_0^\frac{4-n}{n+4} I_1^\frac{2n}{n+4}+Cc_0 \|\phi\|_{L^\infty}^\frac{2n}{(\gamma-1)(n+4)}.
\end{aligned}
\end{equation}
For the bound of $I_1$, if $n=2$, it follows from \eqref{gb-2-regular}--\eqref{nabla-psi-2-regular} and Lemma \ref{Hk-Ck-vector} that 
\begin{equation}\label{eq:I1-2}
\begin{aligned}
I_1&\leq \big\|\phi^\frac{1}{\gamma-1}\big\|_{L^2} \|\boldsymbol{\psi}\|_{L^\infty(\mathbb{R}^2)}^2\leq \big\|\phi^\frac{1}{\gamma-1}\big\|_{L^1}^\frac{1}{2}\|\phi\|_{L^\infty}^\frac{1}{2(\gamma-1)}  \|\boldsymbol{\psi}\|_{D^1}^2\leq C c_0^\frac{5}{2}\|\phi\|_{L^\infty}^\frac{1}{2(\gamma-1)};
\end{aligned}
\end{equation}
while, if $n=3$, it follows from \eqref{nabla-psi-2-regular}, Lemma \ref{lemma-L6}, 
and the H\"older inequality that
\begin{equation}\label{eq:I1-3}
\begin{aligned}
I_1&\leq \big\|\phi^\frac{1}{\gamma-1}\big\|_{L^6} \|\boldsymbol{\psi}\|_{L^6(\mathbb{R}^3)}^2\leq \big\|\phi^\frac{1}{\gamma-1}\big\|_{L^1}^\frac{1}{6}\|\phi\|_{L^\infty}^\frac{5}{6(\gamma-1)}  \|\boldsymbol{\psi}\|_{D^1}^2\leq C c_0^\frac{13}{6}\|\phi\|_{L^\infty}^\frac{5}{6(\gamma-1)}.
\end{aligned}
\end{equation}

Thus, collecting \eqref{eq:I1}--\eqref{eq:I1-3}, together with the Young inequality, gives
\begin{align*}
\|\phi\|_{L^\infty}^\frac{1}{\gamma-1}&\leq \delta_{2n}Cc_0^2\|\phi\|_{L^\infty}^\frac{1}{3(\gamma-1)}+\delta_{3n}Cc_0^2\|\phi\|_{L^\infty}^\frac{5}{7(\gamma-1)}+Cc_0 \|\phi\|_{L^\infty}^\frac{2n}{(\gamma-1)(n+4)}\notag\\
&\leq  Cc_0^3+Cc_0^7+Cc_0^\frac{n+4}{4-n}+\frac{1}{8}\|\phi\|_{L^\infty}^\frac{1}{\gamma-1} \leq Cc_0^7+\frac{1}{8}\|\phi\|_{L^\infty}^\frac{1}{\gamma-1},\notag
\end{align*}
which implies that, for all $t\in [0,T_1]$,
\begin{equation}\label{3.12-2regular}
\|\phi(t)\|_{L^\infty}\leq Cc_0^{7\gamma-7}.
\end{equation}

\smallskip

\textbf{4. $D^1(\mathbb{R}^n)$- and $D^2(\mathbb{R}^n)$-estimates on $\phi$.} 
Applying $\partial^\varsigma\phi\partial^\varsigma$ with the multi-index $|\varsigma|=1$ to $\eqref{li4}_1$ and then integrating the resulting equality over $\mathbb{R}^n$, we see 
from integration by parts, Lemma \ref{ale1}, and the H\"older inequality that
\begin{align*}
\frac{\mathrm{d}}{\mathrm{d}t}\|\partial^\varsigma \phi\|_{L^2}^2 &\leq C\|\nabla\boldsymbol{w}\|_{L^\infty}\|\nabla \phi\|_{L^2}^2 +C \|\nabla^{2}\boldsymbol{w}\|_{L^2}\|\phi\|_{L^\infty}\|\nabla \phi\|_{L^2}\\
&\leq C\|\nabla \boldsymbol{w}\|_{H^2}\big(\|\nabla\phi\|_{L^2}^2+c_0^{7\gamma-7}\|\nabla\phi\|_{L^2}\big),  
\end{align*}
which, along with the Gr\"onwall inequality, the  H\"older inequality, 
and \eqref{jizhu1}, yields that, for all $t\in [0,T_1]$,
\begin{equation}\label{guji323}
\begin{split}
\|\nabla\phi(t)\|_{L^2}&\leq \exp\Big(C\int_0^t \|\nabla\boldsymbol{w}\|_{H^2}\,\mathrm{d}s\Big)\Big(\|\nabla\phi_0\|_{L^2}+Cc_0^{7\gamma-7}\int_0^t \|\nabla\boldsymbol{w}\|_{H^2}\,\mathrm{d}s\Big)\\
&\leq e^{C c_2\sqrt{t}}\big(c_0^2+Cc_0^{7\gamma-7}c_2\sqrt{t}\big)\leq Cc_0^{7\gamma-5}.
\end{split}
\end{equation}

For the $L^2(\mathbb{R}^n)$-estimate of $\nabla^2\phi$, it follows from \eqref{psi=log-phi-0}, \eqref{nabla-psi-2-regular}, \eqref{3.12-2regular}, the H\"older inequality, Lemmas \ref{GN-ineq}, \ref{lemma-L6}, and \ref{Hk-Ck-vector} that
\begin{equation}\label{nabla-phi-guocheng-H1}
\begin{aligned}
\|\nabla^2\phi\|_{L^2}&\leq C\big(\delta_{2n}\|\nabla \phi\|_{L^2}\|\boldsymbol{\psi}\|_{L^\infty(\mathbb{R}^2)}+\delta_{3n} \|\nabla \phi\|_{L^3(\mathbb{R}^3)}\|\boldsymbol{\psi}\|_{L^6(\mathbb{R}^3)}+\|\phi\|_{L^\infty}\|\boldsymbol{\psi}\|_{D^1}\big) \\
&\leq C\big(\delta_{2n}\|\nabla \phi\|_{L^2}+\delta_{3n} \|\nabla \phi\|^{\frac{1}{2}}_{L^2}  \|\nabla^2\phi\|^{\frac{1}{2}}_{L^2}+\|\phi\|_{L^\infty}\big)\|\boldsymbol{\psi}\|_{D^1}\\
&\leq  C\big(c_0^{7\gamma-4}+\delta_{3n}c_0^{\frac{7\gamma-3}{2}} \|\nabla^2\phi\|^{\frac{1}{2}}_{L^2}\big),
\end{aligned}
\end{equation}
which, along with the Young inequality, yields that 
\begin{equation}\label{nabla-phi-H1}
\|\nabla^2\phi(t)\|_{L^2}\leq  Cc_0^{7\gamma-3} \qquad \text{for $t\in [0,T_1]$}.
\end{equation}

\smallskip
\textbf{5. Estimates on $(\phi_t,\boldsymbol{\psi}_t)$.}
First, it follows from $\eqref{li4}_3$, \eqref{jizhu1}, \eqref{nabla-psi-2-regular}, and Lemmas \ref{ale1}, \ref{lemma-L6}, and \ref{Hk-Ck-vector}
that, for all $t\in[0,T_1]$,
\begin{align*}
\|\boldsymbol{\psi}_t\|_{L^2} &\leq C\big(\delta_{2n}\|\nabla\boldsymbol{w}\|_{L^2(\mathbb{R}^2)} \|\boldsymbol{\psi}\|_{L^\infty(\mathbb{R}^2)}+\delta_{3n}\|\nabla\boldsymbol{w}\|_{L^3(\mathbb{R}^3)} \|\boldsymbol{\psi}\|_{L^6(\mathbb{R}^3)}\big)\\
&\quad +C\big(\|\boldsymbol{w}\|_{L^\infty}\|\boldsymbol{\psi}\|_{D^1}+\|\nabla^2\boldsymbol{w}\|_{L^2}\big)\leq C(\|\psi\|_{D^1}+1)\|\boldsymbol{w}\|_{H^2} \le Cc_0c_2.
\end{align*}

Next, it follows from $\eqref{li4}_1$, \eqref{jizhu1}, \eqref{3.12-2regular}--\eqref{nabla-phi-H1}, Lemma \ref{ale1}, and the H\"older inequality that, for all $t\in [0,T_1]$,
\begin{align*}
\|\phi_t(t)\|_{L^2}&\leq C\big(\|\boldsymbol{w}\|_{L^4}\|\nabla\phi\|_{L^4}+\|\phi\|_{\infty}\|\nabla\boldsymbol{w}\|_{L^2}\big)\\
&\leq C\|\boldsymbol{w}\|_{H^1}\big(\|\nabla\phi\|_{H^1}+\|\phi\|_{\infty}\big)\leq Cc_0^{7\gamma-3}c_1,\\
\|\nabla\phi_t(t)\|_{L^2}&\leq C\big(\|\nabla\boldsymbol{w}\|_{L^4}\|\nabla\phi\|_{L^4}+\|\boldsymbol{w}\|_{L^\infty}\|\nabla^2\phi\|_{L^2}+\|\phi\|_{\infty}\|\nabla^2\boldsymbol{w}\|_{L^2}\big)\\
&\leq C\|\boldsymbol{w}\|_{H^2}\big(\|\nabla\phi\|_{H^1}+\|\phi\|_{\infty}\big)\leq Cc_0^{7\gamma-3}c_2.
\end{align*} 
This completes the proof of Lemma \ref{bos}. 
\end{proof}

Next, to derive all the energy estimates for $\boldsymbol{u}$ systematically, we rewrite $\eqref{li4}_2$ as
\begin{equation}\label{eqZZ}
\boldsymbol{u}_t+L\boldsymbol{u}=\boldsymbol{\psi}\cdot Q(\boldsymbol{w})-\boldsymbol{w}\cdot \nabla \boldsymbol{w}-\nabla \phi:=Z(\boldsymbol{w}).
\end{equation}
Then we make the following useful estimates for $Z(\boldsymbol{w})$.
\begin{lem}\label{lemma-usefulZ}
For any $t\in [0,T_1]$,
\begin{align*}
&\|Z(\boldsymbol{w})\|_{L^2}\leq Cc_1^{7\gamma-3}c_2^\frac{n}{4},
\quad\, \|\nabla Z(\boldsymbol{w})\|_{L^2}\leq Cc_2^{7\gamma-3}\big(\|\nabla^3 \boldsymbol{w}\|_{L^2}^\frac{n}{4}+1\big),\\
&\|(Z(\boldsymbol{w}))_t\|_{L^2}\leq  C\Big(c_0c_2\big(\|\nabla \boldsymbol{w}\|_{H^2}+ \|\nabla\boldsymbol{w}_t\|_{L^2}\big)+c_0\|\nabla\boldsymbol{w}_t\|_{L^2}^\frac{1}{2}\|\nabla^2\boldsymbol{w}_t\|_{L^2}^\frac{1}{2}+ c_0^{7\gamma-3}c_2\Big).
\end{align*}
\end{lem}
\begin{proof}
It follows from \eqref{jizhu1}, \eqref{eqZZ}, Lemmas \ref{bos-2-regular}, \ref{ale1}--\ref{GN-ineq}, \ref{lemma-L6}, and \ref{Hk-Ck-vector}, and the H\"older inequality that, for all $t\in [0,T_1]$,
\begin{equation*}
\begin{aligned}
\|Z(\boldsymbol{w})\|_{L^2}&\leq C\big(\delta_{2n} \|\boldsymbol{\psi}\|_{L^\infty(\mathbb{R}^2)}\|\nabla\boldsymbol{w}\|_{L^2(\mathbb{R}^2)}+\delta_{3n} \|\boldsymbol{\psi}\|_{L^6(\mathbb{R}^3)}\|\nabla\boldsymbol{w}\|_{L^3(\mathbb{R}^3)}\big)\notag\\
&\quad +C\big(\|\boldsymbol{w}\|_{L^\infty} \|\nabla\boldsymbol{w}\|_{L^2} + \|\nabla\phi\|_{L^2}\big)\notag\\
&\leq  C\big(\delta_{2n} \|\nabla\boldsymbol{w}\|_{L^2}+\delta_{3n} \|\nabla\boldsymbol{w}\|_{L^2}^\frac{1}{2}\|\nabla^2\boldsymbol{w}\|_{L^2}^\frac{1}{2}\big)\|\boldsymbol{\psi}\|_{D^1} \\
&\quad +C\big(\|\boldsymbol{w}\|_{L^2}^\frac{4-n}{4}\|\nabla^2\boldsymbol{w}\|_{L^2}^\frac{n}{4}\|\nabla\boldsymbol{w}\|_{L^2} +  \|\nabla\phi\|_{L^2}\big)\notag\\
&\leq C\big(\delta_{2n}c_0c_1+\delta_{3n}c_0c_1^\frac{1}{2}c_2^\frac{1}{2}+c_1^\frac{8-n}{4}c_2^\frac{n}{4}+c_0^{7\gamma-3}\big)\leq Cc_1^{7\gamma-3}c_2^\frac{n}{4},\notag
\end{aligned}
\end{equation*}
\begin{equation*}
\begin{aligned}
\|\nabla Z(\boldsymbol{w})\|_{L^2}&\leq C\big(\delta_{2n}\|\boldsymbol{\psi}\|_{L^\infty(\mathbb{R}^2)} \|\nabla^2\boldsymbol{w}\|_{L^2(\mathbb{R}^2)}+\delta_{3n}\|\boldsymbol{\psi}\|_{L^6(\mathbb{R}^3)} \|\nabla^2\boldsymbol{w}\|_{L^3(\mathbb{R}^3)}\big)\notag\\
&\quad +C\big(\|(\nabla\boldsymbol{\psi},\nabla\boldsymbol{w})\|_{L^2}\|\nabla \boldsymbol{w}\|_{L^\infty}+\|\boldsymbol{w}\|_{L^\infty} \|\nabla^2\boldsymbol{w}\|_{L^2}+ \|\nabla^2\phi\|_{L^2}\big)\notag\\
&\leq C\big(\delta_{2n} \|\nabla^2\boldsymbol{w}\|_{L^2} + \delta_{3n}\|\nabla^2\boldsymbol{w}\|_{L^2}^\frac{1}{2}\|\nabla^3\boldsymbol{w}\|_{L^2}^\frac{1}{2}\big)\|\boldsymbol{\psi}\|_{D^1}\\
&\quad +C\big(\|(\nabla\boldsymbol{\psi},\nabla\boldsymbol{w})\|_{L^2}\|\nabla \boldsymbol{w}\|_{L^2}^\frac{4-n}{4}\|\nabla^3 \boldsymbol{w}\|_{L^2}^\frac{n}{4}+\|\boldsymbol{w}\|_{H^2}^2 + \|\nabla^2\phi\|_{L^2}\big)\notag\\
&\leq C\big(\delta_{2n}c_0c_2+\delta_{3n}c_0c_2^\frac{1}{2}\|\nabla^3 \boldsymbol{w}\|_{L^2}^\frac{1}{2}+c_1^\frac{8-n}{4}\|\nabla^3 \boldsymbol{w}\|_{L^2}^\frac{n}{4}+c_2^2+c_0^{7\gamma-3}\big)\notag\\
&\leq Cc_2^{7\gamma-3}\big(\|\nabla^3 \boldsymbol{w}\|_{L^2}^\frac{n}{4}+1\big),\notag\\
\|(Z(\boldsymbol{w}))_t\|_{L^2}&\leq C\big(\|(\boldsymbol{\psi}_t,\boldsymbol{w}_t)\|_{L^2} \|\nabla\boldsymbol{w}\|_{L^\infty}+\delta_{2n}\|\boldsymbol{\psi}\|_{L^\infty(\mathbb{R}^2)} \|\nabla\boldsymbol{w}_t\|_{L^2(\mathbb{R}^2)}\big)\notag\\
&\quad +C\big(\delta_{3n}\|\boldsymbol{\psi}\|_{L^6(\mathbb{R}^3)} \|\nabla\boldsymbol{w}_t\|_{L^3(\mathbb{R}^3)}+\|\boldsymbol{w}\|_{L^\infty} \|\nabla\boldsymbol{w}_t\|_{L^2} + \|\nabla\phi_t\|_{L^2}\big)\notag\\
&\leq C\big(\|(\boldsymbol{\psi}_t,\boldsymbol{w}_t)\|_{L^2} \|\nabla\boldsymbol{w}\|_{H^2}+\delta_{2n}\|\boldsymbol{\psi}\|_{D^1} \|\nabla\boldsymbol{w}_t\|_{L^2}\big)\notag\\
&\quad +C\big(\delta_{3n}\|\boldsymbol{\psi}\|_{D^1} \|\nabla\boldsymbol{w}_t\|_{L^2}^\frac{1}{2}\|\nabla^2\boldsymbol{w}_t\|_{L^2}^\frac{1}{2}+ \|\boldsymbol{w}\|_{H^2} \|\nabla\boldsymbol{w}_t\|_{L^2} + \|\nabla\phi_t\|_{L^2}\big)\\
&\leq C(c_0c_2+c_2)\|\nabla \boldsymbol{w}\|_{H^2}+C\delta_{2n}c_0\|\nabla\boldsymbol{w}_t\|_{L^2}\notag\\
&\quad +C\delta_{3n}c_0\|\nabla\boldsymbol{w}_t\|_{L^2}^\frac{1}{2}\|\nabla^2\boldsymbol{w}_t\|_{L^2}^\frac{1}{2}+Cc_2\|\nabla\boldsymbol{w}_t\|_{L^2}+Cc_0^{7\gamma-3}c_2\notag\\
&\leq C\Big(c_0c_2\big(\|\nabla \boldsymbol{w}\|_{H^2}+ \|\nabla\boldsymbol{w}_t\|_{L^2}\big)+c_0\|\nabla\boldsymbol{w}_t\|_{L^2}^\frac{1}{2}\|\nabla^2\boldsymbol{w}_t\|_{L^2}^\frac{1}{2}+ c_0^{7\gamma-3}c_2\Big).\notag
\end{aligned}
\end{equation*}
\end{proof}

Now, we derive the uniform estimates for $\boldsymbol{u}$.
\begin{lem}\label{llm3}
For any $t\in[0, T_2]$ with $T_2:=\min\big\{T_1, (1+Cc_3)^{-56\gamma}\big\}$,
\begin{align*}
\|\boldsymbol{u}(t)\|_{H^1}^2+\int_0^t\big(\| \nabla \boldsymbol{u}\|_{H^1}^2+\|\boldsymbol{u}_t\|_{L^2}^2\big)\,\mathrm{d} s&\leq Cc_0^2,\\
\|(\nabla^2\boldsymbol{u},\boldsymbol{u}_t)(t)\|_{L^2}^2+\int_0^t \|(\nabla^3\boldsymbol{u},\nabla\boldsymbol{u}_t)\|_{L^2}^2\,\mathrm{d} s&\leq Cc_1^{14\gamma-6} c_2^\frac{n}{2},\\
t\|(\nabla^3\boldsymbol{u},\nabla\boldsymbol{u}_t)(t)\|_{L^2}^2+\int_0^t s\|(\nabla^2\boldsymbol{u}_t,\boldsymbol{u}_{tt})\|_{L^2}^2\,\mathrm{d}s&\leq Cc_0^2.
\end{align*}
\end{lem}
\begin{proof}
We divide the proof into four steps.

\smallskip
\textbf{1. $L^2(\mathbb{R}^n)$-estimate on $\boldsymbol{u}$.}  Multiplying  $\eqref{li4}_2$ by $\boldsymbol{u}$ and integrating the resulting equality over $\mathbb{R}^n$, 
we obtain from Lemma \ref{lemma-usefulZ} and the H\"older inequality that
\begin{equation}\label{y5}
\begin{split}
&\frac12\frac{\mathrm{d}}{\mathrm{d}t} \|\boldsymbol{u}\|_{L^2}^2+\alpha \|\nabla \boldsymbol{u}\|_{L^2}^2+\alpha\|\diver\boldsymbol{u}\|_{L^2}^2\\
&=\int_{\mathbb{R}^n} Z(\boldsymbol{w})\cdot \boldsymbol{u}\,\mathrm{d}x \leq  \|Z(\boldsymbol{w})\|_{L^2}\|\boldsymbol{u}\|_{L^2}\leq  C\|\boldsymbol{u}\|_{L^2}^2+Cc_1^{14\gamma-6}c_2^\frac{n}{2},
\end{split}
\end{equation}
which, along with the Gr\"onwall inequality, yields that 
\begin{equation}\label{dilingjie}
\|\boldsymbol{u}(t)\|_{L^2}^2+\int_0^t \|\nabla\boldsymbol{u}\|_{L^2}^2\,\mathrm{d} s\le Cc_0^2,
\end{equation}
for all $t\in [0, T_2]$ with $T_2:=\min\big\{T_1, (1+Cc_3)^{-56\gamma}\big\}$.

\smallskip
\textbf{2. $D^1(\mathbb{R}^n)$-estimate on $\boldsymbol{u}$.} 
Multiplying  $\eqref{li4}_2$ by $\boldsymbol{u}_t$ and integrating the resulting equality over $\mathbb{R}^n$, we obtain from Lemma \ref{lemma-usefulZ}, and
the H\"older and Young inequalities that
\begin{equation*}
\frac{1}{2}\frac{\mathrm{d}}{\mathrm{d}t} \big(\alpha \|\nabla \boldsymbol{u}\|_{L^2}^2+\alpha\|\diver\boldsymbol{u}\|_{L^2}^2\big)+\|\boldsymbol{u}_t\|_{L^2}^2=\int_{\mathbb{R}^n} Z(\boldsymbol{w})\cdot \boldsymbol{u}_t\,\mathrm{d}x\leq  Cc_1^{14\gamma-6}c_2^\frac{n}{2}+\frac{1}{8}\|\boldsymbol{u}_t\|_{L^2}^2,
\end{equation*}
which, along with the Gr\"onwall inequality, implies that, for all $t\in[0,T_2]$,
\begin{equation}\label{diyijie}
\|\nabla \boldsymbol{u}(t)\|_{L^2}^2+\int_0^t \|\boldsymbol{u}_t\|_{L^2}^2\, \mathrm{d} s\le Cc_0^2.
\end{equation}

Next, rewrite \eqref{eqZZ} as
\begin{equation}\label{eqzz-tuoyuan}
L\boldsymbol{u}=-\boldsymbol{u}_t+Z(\boldsymbol{w}).
\end{equation}
Then it follows from \eqref{eqzz-tuoyuan}, 
Lemma \ref{lemma-usefulZ}, and the classical regularity theory for elliptic equations in Lemma \ref{df3} that 
\begin{equation}\label{y7}
\|\nabla^2\boldsymbol{u}\|_{L^2}\leq C\|(Z(\boldsymbol{w}),\boldsymbol{u}_t)\|_{L^2}\leq C\big(\|\boldsymbol{u}_t\|_{L^2}+c_1^{7\gamma-3}c_2^\frac{n}{4}\big).
\end{equation}
Consequently, it follows from \eqref{diyijie} and \eqref{y7} that, for all $t\in [0,T_2]$,
\begin{equation*}
\int_0^t \|\nabla^2\boldsymbol{u}\|_{L^2}^2\, \mathrm{d}s\leq C\int_0^t \|\boldsymbol{u}_t\|_{L^2}^2\, \mathrm{d}s+Cc_1^{14\gamma-6}c_2^\frac{n}{2}t\leq  Cc_0^2.    
\end{equation*}

\smallskip
\textbf{3. $D^2(\mathbb{R}^n)$-estimate on $\boldsymbol{u}$.}
Differentiating $\eqref{li4}_2$ with respect to $t$ gives
\begin{equation}\label{eqzz-t}
\boldsymbol{u}_{tt}+L\boldsymbol{u}_t=(Z(\boldsymbol{w}))_t.
\end{equation}
Then multiplying above by $\boldsymbol{u}_t$ and integrating over $\mathbb{R}^n$,
we see from the H\"older inequality that
\begin{align*}
&\frac12\frac{\mathrm{d}}{\mathrm{d}t} \|\boldsymbol{u}_t\|_{L^2}^2 +\alpha \|\nabla \boldsymbol{u}_t\|_{L^2}^2+\alpha\|\diver \boldsymbol{u}_t\|_{L^2}^2
= \int_{\mathbb{R}^n} (Z(\boldsymbol{w}))_t\cdot \boldsymbol{u}_t\,\mathrm{d}x\leq \|Z(\boldsymbol{w})_t\|_{L^2}\|\boldsymbol{u}_t\|_{L^2},\notag
\end{align*}
which, along with Lemma \ref{lemma-usefulZ},  implies that
\begin{align*} 
\frac{\mathrm{d}}{\mathrm{d}t} \|\boldsymbol{u}_t\|_{L^2}\leq Cc_0c_2\big(\|\nabla \boldsymbol{w}\|_{H^2}+\|\nabla\boldsymbol{w}_t\|_{L^2}\big)+Cc_0\|\nabla\boldsymbol{w}_t\|_{L^2}^\frac{1}{2}\|\nabla^2\boldsymbol{w}_t\|_{L^2}^\frac{1}{2}+Cc_0^{7\gamma-3}c_2.
\end{align*}
Integrating the above over $(\tau, t)$ with $\tau\in(0,t)$, together with \eqref{jizhu1}, and the H\"older and Young inequalities, yields that,  for all $t\in [\tau, T_2]$,
\begin{equation}\label{polo}
\begin{aligned}
\|\boldsymbol{u}_t(t)\|_{L^2}&\leq \|\boldsymbol{u}_t(\tau)\|_{L^2} +Cc_0c_2
\int_\tau^t \big(\|\nabla \boldsymbol{w}\|_{H^2}+\|\nabla\boldsymbol{w}_t\|_{L^2}\big)\,\mathrm{d}s\\
&\quad +Cc_0\int_0^t \|\nabla\boldsymbol{w}_t\|_{L^2}^\frac{1}{2}\|\nabla^2\boldsymbol{w}_t\|_{L^2}^\frac{1}{2}\,\mathrm{d}s+Cc_0^{7\gamma-3}c_2t\\
&\leq \|\boldsymbol{u}_t(\tau)\|_{L^2} +Cc_0c_2t^\frac{1}{2}
\Big(\int_0^t \big(\|\nabla \boldsymbol{w}\|_{H^2}^2+\|\nabla\boldsymbol{w}_t\|_{L^2}^2\big) \,\mathrm{d}s\Big)^\frac{1}{2}\\
&\quad +Cc_0\Big(\int_0^t \frac{1}{\sqrt{s}} \,\mathrm{d}s\Big)^\frac{1}{2}\Big(\int_0^t \|\nabla\boldsymbol{w}_t\|_{L^2}^2\,\mathrm{d}s\Big)^\frac{1}{4}\Big(\int_0^t s\|\nabla^2\boldsymbol{w}_t\|_{L^2}^2\,\mathrm{d}s\Big)^\frac{1}{4}+Cc_0\\
&\leq \|\boldsymbol{u}_t(\tau)\|_{L^2}+Cc_0c_2^2t^\frac{1}{2}+Cc_0c_2^\frac{1}{2}c_3^\frac{1}{2}t^\frac{1}{4}+Cc_0\leq \|\boldsymbol{u}_t(\tau)\|_{L^2} +Cc_0.
\end{aligned}
\end{equation}
It follows from $\eqref{li4}_2$, the time continuity of $(\phi,\boldsymbol{u},\boldsymbol{\psi})$, \eqref{vg}, \eqref{houmian}, Lemmas \ref{ale1} and \ref{lemma-L6}, and the H\"older inequality that 
\begin{align*}
\lim\sup_{\tau\to 0}\|\boldsymbol{u}_t(\tau)\|_{L^2}&\leq C \big(\|\boldsymbol{u}_0\|_{L^\infty} \|\nabla \boldsymbol{u}_0\|_{L^2} +\|\nabla\phi_0\|_{L^2}+\|\nabla^2 \boldsymbol{u}_0\|_{L^2}\big)\notag\\
&\quad+ C\big(\delta_{2n} \|\boldsymbol{\psi}_0\|_{L^\infty(\mathbb{R}^2)} \|\nabla \boldsymbol{u}_0\|_{L^2}+\delta_{3n} \|\boldsymbol{\psi}_0\|_{L^6(\mathbb{R}^3)} \|\nabla \boldsymbol{u}_0\|_{L^3(\mathbb{R}^3)}\big)\notag\\
&\leq C \big(\|\boldsymbol{u}_0\|_{L^\infty} \|\nabla \boldsymbol{u}_0\|_{L^2} +\|\nabla\phi_0\|_{L^2}+\|\nabla^2 \boldsymbol{u}_0\|_{L^2}\big)\\
&\quad+ C\big(\delta_{2n} \|\boldsymbol{\psi}_0\|_{L^\infty(\mathbb{R}^2)}+\delta_{3n} \|\boldsymbol{\psi}_0\|_{D^1(\mathbb{R}^3)} \big)\|\nabla \boldsymbol{u}_0\|_{H^1} \leq Cc_0^2.\notag
\end{align*}
Based on this, we let $\tau\to 0$ in \eqref{polo} and apply the Gr\"onwall inequality to 
obtain
\begin{equation}\label{y9}
\|\boldsymbol{u}_t(t)\|_{L^2}^2 +\int_0^t \|\nabla\boldsymbol{u}_t\|_{L^2}^2 \,\mathrm{d}s \leq Cc_0^4 \qquad\,\,  \text{for all $t\in [0,T_2]$}.
\end{equation}
It thus  follows from \eqref{y7}  that
\begin{equation}\label{y7'}
\|\nabla^2\boldsymbol{u}\|_{L^2}\leq C\big(\|\boldsymbol{u}_t\|_{L^2}+c_1^{7\gamma-3}c_2^\frac{n}{4}\big)\leq Cc_1^{7\gamma-3}c_2^\frac{n}{4}.
\end{equation}

Finally, it follows from  \eqref{eqzz-tuoyuan}, Lemma \ref{lemma-usefulZ}, 
and the classical  regularity theory for elliptic equations in Lemma \ref{df3} that
\begin{equation}\label{nabla3u}
\begin{split}
\|\nabla^3 \boldsymbol{u}\|_{L^2}&\leq C\|(\nabla\boldsymbol{u}_t,\nabla Z(\boldsymbol{w}))\|_{L^2}\leq  C\big(\|\nabla\boldsymbol{u}_t\|_{L^2} +c_2^{7\gamma-3}\|\nabla^3 \boldsymbol{w}\|_{L^2}^\frac{n}{4}+c_2^{7\gamma-3}\big),
\end{split}
\end{equation}
which, along with \eqref{jizhu1}, \eqref{y9}, and the H\"older inequality, 
implies that, for all $t\in [0,T_2]$,
\begin{equation}
\begin{aligned}
\int_0^t \|\nabla^3\boldsymbol{u}\|_{L^2}^2\,\mathrm{d}s &\leq C\int_0^t \|\nabla \boldsymbol{u}_t\|_{L^2}^2\,\mathrm{d}s+Cc_2^{14\gamma-6}\int_0^t \|\nabla^3\boldsymbol{w}\|_{L^2}^\frac{n}{2}\,\mathrm{d}s+Cc_2^{14\gamma-6}t\\
&\leq Cc_0^4+Cc_2^{14\gamma-6}t^\frac{4-n}{4}\Big(\int_0^t \|\nabla^3\boldsymbol{w}\|_{L^2}^2\,\mathrm{d}s\Big)^\frac{n}{4}\leq Cc_0^4. 
\end{aligned}
\end{equation}

\smallskip
\textbf{4. Time-weighted estimates on $\boldsymbol{u}$.}
First, multiplying \eqref{eqzz-t} by $\boldsymbol{u}_{tt}$ and integrating over $\mathbb{R}^n$, we see from  Lemma \ref{lemma-usefulZ} and the H\"older inequality that
\begin{equation}\label{lrq-pre}
\begin{aligned}
&\frac{\alpha}{2}\frac{\mathrm{d}}{\mathrm{d}t}\big(\|\nabla \boldsymbol{u}_t\|_{L^2}^2+\|\diver \boldsymbol{u}_t\|_{L^2}^2\big)+\|\boldsymbol{u}_{tt}\|_{L^2}^2=\int_{\mathbb{R}^n} (Z(\boldsymbol{w}))_t\cdot \boldsymbol{u}_{tt}\mathrm{d}x\\
& \leq  C\Big(c_0c_2\|\nabla \boldsymbol{w}\|_{H^2}+c_0c_2\|\nabla\boldsymbol{w}_t\|_{L^2} + c_0\|\nabla\boldsymbol{w}_t\|_{L^2}^\frac{1}{2}\|\nabla^2\boldsymbol{w}_t\|_{L^2}^\frac{1}{2}+c_0^{7\gamma-3}c_2\Big)\|\boldsymbol{u}_{tt}\|_{L^2}.
\end{aligned}
\end{equation}
Then multiplying above by $t$, together with \eqref{jizhu1} and the Young inequality, gives 
\begin{equation}\label{lrq}
\begin{aligned}
&\alpha\frac{\mathrm{d}}{\mathrm{d}t}\big(t\|\nabla \boldsymbol{u}_t\|_{L^2}^2+t\|\diver \boldsymbol{u}_t\|_{L^2}^2\big)+t\|\boldsymbol{u}_{tt}\|_{L^2}^2\\
& \leq Cc_0^2c_2^2c_3^2+Cc_0^2c_3\sqrt{t}\|\nabla^2\boldsymbol{w}_t\|_{L^2}+Cc_0^{14\gamma-6}c_2^2t +\alpha\|\nabla \boldsymbol{u}_t\|_{L^2}^2\\
&\leq Cc_3^6+ Cc_3^3\sqrt{t}\|\nabla^2\boldsymbol{w}_t\|_{L^2}+Cc_3^{14\gamma-4}t.
\end{aligned}
\end{equation}

Next, integrating the above over $[\tau,t]$ with $\tau\in (0,t)$, along with \eqref{jizhu1}, \eqref{y9}, and the H\"older inequality, gives that, for all $t\in [0,T_2]$,
\begin{equation}\label{yilu}
\begin{aligned}
&t\|\nabla \boldsymbol{u}_t(t)\|_{L^2}^2+\int_\tau^t s\|\boldsymbol{u}_{tt}\|_{L^2}^2\,\mathrm{d}s\\
&\leq C \tau\|\nabla \boldsymbol{u}_t(\tau)\|_{L^2}^2 +Cc_3^3\int_\tau^t \sqrt{s}\|\nabla^2\boldsymbol{w}_t\|_{L^2}\,\mathrm{d}s+Cc_3^{14\gamma-4}t^2+Cc_3^6t\\
&\leq C\tau\|\nabla \boldsymbol{u}_t(\tau)\|_{L^2}^2+Cc_3^3\sqrt{t}\Big(\int_0^t s \|\nabla^2\boldsymbol{w}_t\|_{L^2}^2\,\mathrm{d}s\Big)^\frac{1}{2}+Cc_0^2\leq C\tau\|\nabla \boldsymbol{u}_t(\tau)\|_{L^2}^2+Cc_0^2.
\end{aligned}
\end{equation}
Due to  \eqref{y9} and Lemma \ref{bjr}, there exists a sequence $\{\tau_k\}_{k=1}^\infty$ such that 
\begin{equation*}
\tau_k\to 0, \quad\, \tau_k\|\nabla \boldsymbol{u}_t(\tau_k)\|_{L^2}^2\to 0  \qquad\,\, \text{as $k\to \infty$}.
\end{equation*}
Thus, setting $\tau=\tau_k\to 0$  in  \eqref{yilu} yields that, for all $t\in[0,T_2]$,
\begin{equation}\label{cv}
t\|\nabla \boldsymbol{u}_t(t)\|_{L^2}^2+\int_\tau^t s\|\boldsymbol{u}_{tt}\|_{L^2}^2\,\mathrm{d}s\leq Cc_0^2. 
\end{equation}

Now, rewrite \eqref{eqzz-t} as
\begin{equation}\label{eqzz-t-tuoyuan}
L\boldsymbol{u}_t=-\boldsymbol{u}_{tt}+(Z(\boldsymbol{w}))_t.
\end{equation}
Then it follows from \eqref{eqzz-t-tuoyuan},  
Lemma \ref{lemma-usefulZ}, 
and the classical regularity theory for elliptic equations in Lemma \ref{df3} that
\begin{equation}\label{nabla2x1tu}
\begin{aligned}
&\|\nabla^2\boldsymbol{u}_t\|_{L^2}\leq C \big\|(\boldsymbol{u}_{tt},(Z(\boldsymbol{w}))_t)\big\|_{L^2}\\
& \leq C \|\boldsymbol{u}_{tt}\|_{L^2}+Cc_0c_2\big(\|\nabla \boldsymbol{w}\|_{H^2}+\|\nabla\boldsymbol{w}_t\|_{L^2}\big)+Cc_0\|\nabla\boldsymbol{w}_t\|_{L^2}^\frac{1}{2}\|\nabla^2\boldsymbol{w}_t\|_{L^2}^\frac{1}{2}+Cc_0^{7\gamma-3}c_2,
\end{aligned}
\end{equation}
which, along with \eqref{jizhu1}, \eqref{cv}, and the H\"older inequality, 
yields that, for $t\in [0,T_2]$,
\begin{equation}
\begin{aligned}
\int_0^t s\|\nabla^2\boldsymbol{u}_t\|_{L^2}^2\,\mathrm{d}s&\leq C \int_0^t s\|\boldsymbol{u}_{tt}\|_{L^2}^2\,\mathrm{d}s+Cc_2^4\int_0^t s\big(\|\nabla \boldsymbol{w}\|_{H^2}^2+\|\nabla\boldsymbol{w}_t\|_{L^2}^2\big)\,\mathrm{d}s\\
&\quad +Cc_0^2\int_0^ts\|\nabla\boldsymbol{w}_t\|_{L^2}\|\nabla^2\boldsymbol{w}_t\|_{L^2}\,\mathrm{d}s+Cc_0^{14\gamma-6}c_2^2t\\
&\leq Cc_0^2+Cc_2^4c_3^2t+Cc_0^2c_3 \sqrt{t}\Big(\int_0^ts\|\nabla^2\boldsymbol{w}_t\|_{L^2}^2\,\mathrm{d}s\Big)^\frac{1}{2}
\leq Cc_0^2.
\end{aligned}
\end{equation}

Finally, \eqref{nabla3u}, together with \eqref{jizhu1} and \eqref{cv}, gives that, for all $t\in [0,T_2]$,
\begin{equation}\label{nabla3u'}
\begin{split}
\sqrt{t}\|\nabla^3 \boldsymbol{u}\|_{L^2}
&\leq  C\sqrt{t}\|\nabla\boldsymbol{u}_t\|_{L^2}+Cc_2^{7\gamma-3}t^{\frac{4-n}{8}}\big(\sqrt{t}\|\nabla^3 \boldsymbol{w}\|_{L^2}\big)^\frac{n}{4} +Cc_2^{7\gamma-3}\sqrt{t}\\
&\leq Cc_0 +Cc_2^{7\gamma-3}\sqrt{t}+Cc_2^{7\gamma-3}c_3^\frac{n}{4}t^{\frac{4-n}{8}}\leq Cc_0.
\end{split}
\end{equation}

The proof of  Lemma \ref{llm3} is completed.
\end{proof}

\smallskip
Finally, choose $T^*$ 
and constants $c_i$ $(i=1,2,3)$ as:
\begin{equation*}
T^*=T_2,\qquad c_1=C^\frac{1}{2}c_0, \qquad c_2=c_3=C^\frac{2}{4-n}c_1^\frac{28\gamma-12}{4-n}=C^\frac{14\gamma-4}{4-n}c_0^\frac{28\gamma-12}{4-n} .
\end{equation*}
It follows from Lemmas \ref{bos-2-regular} and \ref{llm3} that, for any $t\in [0,T^*]$,
\begin{equation}\label{lgg}
\begin{aligned}
\|\boldsymbol{u}(t)\|_{H^1}^2+\int_0^t\big(\| \nabla \boldsymbol{u}\|_{H^1}^2+\|\boldsymbol{u}_t\|_{L^2}^2\big)\,\mathrm{d} s&\leq c_1^2,\\ 
\|(\nabla^2\boldsymbol{u},\boldsymbol{u}_t)(t)\|_{L^2}^2+\int_0^t \|(\nabla^3\boldsymbol{u},\nabla\boldsymbol{u}_t)\|_{L^2}^2\,\mathrm{d} s&\leq c_2^2,\\ 
t\|(\nabla^3\boldsymbol{u},\nabla\boldsymbol{u}_t)(t)\|_{L^2}^2+\int_0^t s\|(\nabla^2\boldsymbol{u}_t,\boldsymbol{u}_{tt})\|_{L^2}^2\,\mathrm{d}s&\leq  c_3^2,\\
\|\phi^\frac{1}{\gamma-1} (t)\|_{L^1}+\|\phi(t)\|_{L^\infty }+\|(\nabla\phi,\phi_t)(t)\|_{H^1}+\|\nabla\boldsymbol{\psi},\boldsymbol{\psi}_t)(t)\|_{L^2}&\leq Cc_3^{7\gamma-2}.
\end{aligned}    
\end{equation}

\smallskip
\subsubsection{Proof of {\rm Theorem \ref{thh1}}}\label{bani} 
We  give the proof of Theorem \ref{thh1}, which  is based on the classical iteration scheme and conclusions obtained in \S \ref{uape}. Denote as in \S \ref{uape} that 
\begin{equation*}
2+\big\|\phi_0^{\frac{1}{\gamma-1}}\big\|_{L^1}+\|\nabla\phi_0\|_{H^1}+\|\boldsymbol{u}_0\|_{H^2}+\|\boldsymbol{\psi}_0\|_{D^1(\mathbb{R}^3)} \leq c_0.
\end{equation*}
Next, let $\boldsymbol{u}^0$ be the unique spherically symmetric solution of the problem in $[0,\infty)\times \mathbb{R}^n$:
\begin{equation*}
\begin{cases}
\boldsymbol{u}^0_t-\Delta \boldsymbol{u}^0=0,\\
\boldsymbol{u}^0|_{t=0}=\boldsymbol{u}_0 &\qquad\text{for $\boldsymbol{x}\in\mathbb{R}^n$},\\
\boldsymbol{u}^0(t,\boldsymbol{x})\to \boldsymbol{0} \ \ \text{as} \ \ \left|\boldsymbol{x}\right|\to \infty &\qquad\text{for $t\ge 0$}.
\end{cases}
\end{equation*}
We can obtain  the global well-posedness of $\boldsymbol{u}^0$ via the standard theory of linear parabolic equations, and then prove that $\boldsymbol{u}^0$ is spherically symmetric via an argument similar to that in the proof of Lemma \ref{lem1}. Now we can choose $T'\in (0,T^*]$ small enough such that
\begin{equation}\label{jizhu}
\begin{aligned}
\sup_{t\in[0,T']}\|\boldsymbol{u}^0\|_{H^1}^2+\int_0^{T'}\big(\| \nabla \boldsymbol{u}^0\|_{H^1}^2+\|\boldsymbol{u}_t^0\|_{L^2}^2\big)\,\mathrm{d} s\leq c_1^2,\\
\sup_{t\in[0,T']}\|(\nabla^2\boldsymbol{u}^0,\boldsymbol{u}_t^0)\|_{L^2}^2+\int_0^{T'} \|(\nabla^3\boldsymbol{u}^0,\nabla\boldsymbol{u}_t^0)\|_{L^2}^2\,\mathrm{d} s\leq c_2^2,\\
\sup_{t\in[0,T']}t\|(\nabla^3\boldsymbol{u}^0, \nabla\boldsymbol{u}_t^0)\|_{L^2}^2+\int_0^{T'} s\|(\nabla^2\boldsymbol{u}_t^0,\boldsymbol{u}_{tt}^0)\|_{L^2}^2\,\mathrm{d}s\leq c_3^2.
\end{aligned}
\end{equation}

We divide the rest of the proof into four steps.

\smallskip
\textbf{1. Existence.} 
Starting the iteration with $\boldsymbol{w}=\boldsymbol{u}^0$ in \eqref{li4}, 
we first obtain a solution $(\phi^1,\boldsymbol{u}^1,\boldsymbol{\psi}^1)$ 
of problem \eqref{li4}. 
Then we inductively construct approximate sequences $(\phi^{k+1}, \boldsymbol{u}^{k+1},\boldsymbol{\psi}^{k+1})$ as follows: Suppose that $(\phi^{k},\boldsymbol{u}^{k},\boldsymbol{\psi}^k)$ for $k\geq 1$ has been obtained. then we define $(\phi^{k+1}, \boldsymbol{u}^{k+1},\boldsymbol{\psi}^{k+1})$  by solving the following problem in $[0,T']\times\mathbb{R}^n$,
\begin{equation}\label{li6}
\begin{cases}
\displaystyle \phi^{k+1}_t+\boldsymbol{u}^{k}\cdot \nabla \phi^{k+1}+(\gamma-1)\phi^{k+1}\diver  \boldsymbol{u}^{k}=0,\\[4pt]
\displaystyle \boldsymbol{u}^{k+1}_t+\boldsymbol{u}^{k}\cdot\nabla \boldsymbol{u}^{k} +\nabla \phi^{k+1}+L\boldsymbol{u}^{k+1}=\boldsymbol{\psi}^{k+1}\cdot Q(\boldsymbol{u}^{k}),\\
\displaystyle \boldsymbol{\psi}^{k+1}_t+\sum_{l=1}^n A_l(\boldsymbol{u}^k) \partial_l\boldsymbol{\psi}^{k+1}+B(\boldsymbol{u}^k) \boldsymbol{\psi}^{k+1}+\nabla \diver \boldsymbol{u}^k=0,\\[4pt]
\displaystyle (\phi^{k+1}, \boldsymbol{u}^{k+1}, \boldsymbol{\psi}^{k+1} )|_{t=0}=(\phi_0,\boldsymbol{u}_0, \boldsymbol{\psi}_0) \ \  \ \, \quad \text{for $\boldsymbol{x}\in\mathbb{R}^n$},\\[4pt]
\displaystyle
(\phi^{k+1},\boldsymbol{u}^{k+1})\to (0,\boldsymbol{0}) \ \  \text{as} \ \ |\boldsymbol{x}|\to\infty \qquad\quad  \text{for $t\in (0,T']$}.
 \end{cases}
\end{equation}
By replacing $\boldsymbol{w}$  with $\boldsymbol{u}^k$ in \eqref{li4}, we can solve problem \eqref{li6}. Clearly, $(\phi^k, \boldsymbol{u}^k, \boldsymbol{\psi}^k)$ satisfies the uniform estimates \eqref{lgg}: for all $t\in [0,T']$ and $k\in \mathbb{N}$,
\begin{equation}\label{lgg-k}
\begin{aligned}
\|\boldsymbol{u}^k(t)\|_{H^1}^2+\int_0^t\big(\| \nabla \boldsymbol{u}^k\|_{H^1}^2+\|\boldsymbol{u}^k_t\|_{L^2}^2\big)\,\mathrm{d} s&\leq  c_1^2,\\
\|(\nabla^2\boldsymbol{u}^k,\boldsymbol{u}^k_t)(t)\|_{L^2}^2+\int_0^t \|(\nabla^3\boldsymbol{u}^k,\nabla\boldsymbol{u}^k_t)\|_{L^2}^2\,\mathrm{d} s&\leq  c_2^2,\\
t\|(\nabla^3\boldsymbol{u}^k,\nabla\boldsymbol{u}^k_t)(t)\|_{L^2}^2+\int_0^t s\|(\nabla^2\boldsymbol{u}^k_t,\boldsymbol{u}^k_{tt})\|_{L^2}^2\,\mathrm{d}s&\leq  c_3^2,\\
\big\|(\phi^k)^\frac{1}{\gamma-1}(t)\big\|_{L^1}+\|\phi^k(t)\|_{L^\infty}+\|(\nabla\phi^k,\phi_t^k)(t)\|_{H^1}+\|(\nabla\boldsymbol{\psi}^k,\boldsymbol{\psi}^k_t)(t)\|_{L^2}&\leq  Cc_3^{7\gamma-2}.
\end{aligned}    
\end{equation}

Next, we are going to prove that the whole sequence $(\phi^k,\boldsymbol{u}^k,\boldsymbol{\psi}^k)$ converges strongly to a limit $(\phi,\boldsymbol{u},\boldsymbol{\psi})$ in some Sobolev space. Let
\begin{equation*}
\hat \phi^{k+1}=\phi^{k+1}-\phi^k,\qquad  \hat{\boldsymbol{u}}^{k+1}=\boldsymbol{u}^{k+1}-\boldsymbol{u}^k,\qquad  \hat{\boldsymbol{\psi}}^{k+1}=\boldsymbol{\psi}^{k+1}-\boldsymbol{\psi}^k.
\end{equation*}
Then, by \eqref{li6}, one can deduce that
\begin{equation}\label{eq:1.2w}
\begin{cases}
\hat{\phi}^{k+1}_t+\boldsymbol{u}^k\cdot \nabla\hat{\phi}^{k+1} +(\gamma-1)\hat{\phi}^{k+1} \diver\boldsymbol{u}^k=\hat{\mathcal{R}}_1^k,\\[4pt]
\hat{\boldsymbol{u}}^{k+1}_t+ \boldsymbol{u}^k\cdot\nabla \hat{\boldsymbol{u}}^{k}+\nabla\hat\phi^{k+1} +L\hat{\boldsymbol{u}}^{k+1}=\hat{\boldsymbol{\psi}}^{k+1}\cdot Q(\boldsymbol{u}^k)+\hat{\mathcal{R}}_2^k,\\
\displaystyle\hat{\boldsymbol{\psi}}^{k+1}_t+\sum_{l=1}^nA_l(\boldsymbol{u}^k) \partial_l\hat{\boldsymbol{\psi}}^{k+1}+B(\boldsymbol{u}^k)\hat{\boldsymbol{\psi}}^{k+1}+\nabla \diver\hat{\boldsymbol{u}}^k=\hat{\mathcal{R}}^k_3,\\
(\hat{\phi}^{k+1},\hat{\boldsymbol{u}}^{k+1},\hat{\boldsymbol{\psi}}^{k+1})|_{t=0}=(0,\boldsymbol{0}, \boldsymbol{0}) \qquad\quad \ \ \,\text{for $\boldsymbol{x}\in\mathbb{R}^n$},\\[4pt]
\displaystyle (\hat\phi^{k+1},\hat{\boldsymbol{u}}^{k+1})\to (0,\boldsymbol{0}) \ \  \text{as} \ \ |\boldsymbol{x}|\to\infty \qquad\quad\text{for $t\in (0,T']$},
\end{cases}
\end{equation}
where $\hat{\mathcal{R}}_i^k$ ($i=1,2,3$) are defined by
\begin{equation}
\begin{aligned}
\hat{\mathcal{R}}^k_1&=-\hat{\boldsymbol{u}}^k\cdot\nabla\phi ^{k}-(\gamma-1)\phi^{k}\diver\hat{\boldsymbol{u}}^k,\qquad \hat{\mathcal{R}}^k_2=\boldsymbol{\psi}^{k}\cdot Q(\hat{\boldsymbol{u}}^{k})-\hat{\boldsymbol{u}}^{k} \cdot \nabla \boldsymbol{u}^{k-1},\\
\hat{\mathcal{R}}_3^k&=-\sum_{l=1}^n A_l(\hat{\boldsymbol{u}}^k)\partial_l\boldsymbol{\psi}^{k}-B(\hat{\boldsymbol{u}}^k)\boldsymbol{\psi}^{k}.
\end{aligned}
\end{equation}

\smallskip
\textbf{1.1. Estimates of  $\hat{\boldsymbol{\psi}}^{k+1}$.} First, we prove the following axillary lemma, which implies that $\hat{\boldsymbol{\psi}}^{k+1}\in L^\infty([0,T'];L^2(\mathbb{R}^n))$ for each $k\in\mathbb{N}$.

\begin{lem}\label{lpsi}
For each $k\in \mathbb{N}$, $\,\hat{\boldsymbol{\psi}}^{k+1}\in L^\infty([0,T'];L^2(\mathbb{R}^n))$.
\end{lem}
\begin{proof}
Let $\zeta(\boldsymbol{x})\in C_{\rm c}^\infty(\mathbb{R}^n)$ be a test function satisfying 
\begin{equation}
0\le \zeta(\boldsymbol{x})\le 1, \quad \,\,
\zeta(\boldsymbol{x})=\begin{cases}
1&\text{if } 0\leq |\boldsymbol{x}|\leq 1,\\
0&\text{if } |\boldsymbol{x}|\geq 2.
\end{cases}
\end{equation}
Define $\zeta_R(\boldsymbol{x})=\zeta(\frac{\boldsymbol{x}}{R})$ and $  \hat{\boldsymbol{\psi}}^{k+1}_{(R)}=\hat{\boldsymbol{\psi}}^{k+1}\zeta_R$ for $R\geq 1$. Then it follows from
$\eqref{eq:1.2w}_3$ that 
\begin{equation}\label{gga}
\big(\hat{\boldsymbol{\psi}}^{k+1}_{(R)}\big)_t+\sum_{l=1}^nA_l(\boldsymbol{u}^k)\partial_l\hat{\boldsymbol{\psi}}^{k+1}_{(R)}+B(\boldsymbol{u}^k)\hat{\boldsymbol{\psi}}^{k+1}_{(R)}+\zeta_R\nabla\diver\hat{\boldsymbol{u}}^k=\zeta_R\hat{\mathcal{R}}^k_3+\hat{\boldsymbol{\psi}}^{k+1}(\boldsymbol{u}^k\cdot\nabla \zeta_R).
\end{equation}

Multiplying \eqref{gga} by $2\hat{\boldsymbol{\psi}}^{k+1}_{(R)}$ and integrating the resulting equality over $\mathbb{R}^n$, we obtain from \eqref{lgg-k}, Lemmas \ref{ale1}, \ref{lemma-L6}, and \ref{Hk-Ck-vector}, and the H\"older inequality that
\begin{equation}\label{xxx}
\begin{split}
\frac{\mathrm{d}}{\mathrm{d}t}\big\|\hat{\boldsymbol{\psi}}^{k+1}_{(R)}\big\|_{L^2}&\leq C \|\nabla \boldsymbol{u}^k\|_{L^\infty}\big\|\hat{\boldsymbol{\psi}}^{k+1}_{(R)}\big\|_{L^2}+C\|\nabla^2\hat{\boldsymbol{u}}^k\|_{L^2}+C\|\nabla\boldsymbol{\psi}^k\|_{L^2}\|\hat{\boldsymbol{u}}^k\|_{L^\infty}\\
&\quad +C\delta_{2n}\|\boldsymbol{\psi}^k\|_{L^\infty(\mathbb{R}^2)}\|\nabla\hat{\boldsymbol{u}}^k\|_{L^2} +C\delta_{3n}\|\boldsymbol{\psi}^k\|_{L^6(\mathbb{R}^3)}\|\nabla\hat{\boldsymbol{u}}^k\|_{L^3(\mathbb{R}^3)}\\
&\quad +C\delta_{2n}\|\boldsymbol{u}^k\|_{L^2}\|\hat{\boldsymbol{\psi}}^{k+1}\|_{L^\infty(\mathbb{R}^2)}+C\delta_{3n}\|\boldsymbol{u}^k\|_{L^3(\mathbb{R}^3)}\|\hat{\boldsymbol{\psi}}^{k+1}\|_{L^6(\mathbb{R}^3)}\\
&\leq C(c_0)\big(\big\|\hat{\boldsymbol{\psi}}^{k+1}_{(R)}\big\|_{L^2}+1\big),
\end{split}
\end{equation}
where $C(c_0)>0$ is a constant depending only on $(C,c_0)$. Then applying the Gr\"onwall inequality to \eqref{xxx} gives that, for all $t\in [0,T']$ and $R\geq 1$,
\begin{equation*}
\|\hat{\boldsymbol{\psi}}^{k+1}_{(R)}(t)\|_{L^2}\leq C(c_0) T'\exp\big(C(c_0)T'\big),    
\end{equation*}
which, along with Lemma \ref{Fatou}, yields that 
\begin{equation}\label{dah}
\hat{\boldsymbol{\psi}}^{k+1}\in L^\infty([0,T'];L^2(\mathbb{R}^n)).
\end{equation}

The proof of Lemma \ref{lpsi} is completed.
\end{proof}

Now we turn to derive the uniform estimate of $\hat{\boldsymbol{\psi}}^{k+1}$. Multiplying $(\ref{eq:1.2w})_3$ by $2\hat{\boldsymbol{\psi}}^{k+1}$ and integrating the resulting equality over $\mathbb{R}^n$, we see from \eqref{lgg-k}, Lemmas \ref{ale1}, \ref{lemma-L6}, and \ref{Hk-Ck-vector}, and the  H\"older and Young inequalities that
\begin{equation}\label{go64aa}
\begin{aligned}
\frac{\mathrm{d}}{\mathrm{d}t}\|\hat{\boldsymbol{\psi}}^{k+1}\|_{L^2}^2&\leq  C \big(\|\nabla\boldsymbol{u}^k\|_{L^\infty} \|\hat{\boldsymbol{\psi}}^{k+1}\|_{L^2}\!+\!\|\nabla^2 \hat{\boldsymbol{u}}^k\|_{L^2}\!+\!\|\nabla \boldsymbol{\psi}^k\|_{L^2}\|\hat{\boldsymbol{u}}^k\|_{L^\infty}\big)\|\hat{\boldsymbol{\psi}}^{k+1}\|_{L^2}\\
&\quad+C\delta_{2n}\|\boldsymbol{\psi}^k\|_{L^\infty(\mathbb{R}^2)}\|\nabla\hat{\boldsymbol{u}}^k\|_{L^2(\mathbb{R}^2)}\|\hat{\boldsymbol{\psi}}^{k+1}\|_{L^2}\\
&\quad+C\delta_{3n}\|\boldsymbol{\psi}^k\|_{L^6(\mathbb{R}^3)}\|\nabla\hat{\boldsymbol{u}}^k\|_{L^3(\mathbb{R}^3)}\|\hat{\boldsymbol{\psi}}^{k+1}\|_{L^2}\\
&\leq C(c_0)\big(\omega^{-1}+t^{-\frac{1}{2}}\big)\|\hat{\boldsymbol{\psi}}^{k+1}\|_{L^2}^2+\omega\big(\|\hat{\boldsymbol{u}}^k\|_{L^2}^2+\|\nabla\hat{\boldsymbol{u}}^k\|_{H^1}^2\big)
\end{aligned}
\end{equation}
for $t\in[0,T']$, where $C(c_0)$ is a constant depending only on $(C,c_0)$ and $\omega \in (0,1)$ is a constant to be determined later.

\smallskip
\textbf{1.2. Estimates of $\hat\phi^{k+1}$.} First, multiplying $(\ref{eq:1.2w})_1$ by $2\hat\phi^{k+1}$ and then integrating the resulting equality over $\mathbb{R}^n$, 
we obtain from \eqref{lgg-k}, Lemma \ref{ale1}, and the H\"older and Young inequalities that 
\begin{equation}\label{fly1}
\begin{aligned}
\frac{\mathrm{d}}{\mathrm{d}t}\|\hat\phi^{k+1}\|_{L^2}^2&\leq  C\|\nabla \boldsymbol{u}^k\|_{L^\infty}\|\hat\phi^{k+1}\|_{L^2}^2+ C\|\hat{\boldsymbol{u}}^k\|_{L^6}\|\nabla \phi^k\|_{L^3}\|\hat\phi^{k+1}\|_{L^2}\\
&\quad+C\|\nabla\hat{\boldsymbol{u}}^k\|_{L^2}\|\phi^k\|_{L^\infty}\|\hat\phi^{k+1}\|_{L^2}\\
&\leq C(c_0)\big(\omega^{-1}+t^{-\frac{1}{2}}\big)\|\hat\phi^{k+1}\|_{L^2}^2+\omega\big(\|\hat{\boldsymbol{u}}^k\|_{L^2}^2+\|\nabla\hat{\boldsymbol{u}}^k\|_{L^2}^2\big).
\end{aligned}
\end{equation}
Next, applying $2\partial^\varsigma\hat{\phi}^{k+1}\partial^{\varsigma}$ with the multi-index $|\varsigma|=1$ to $(\ref{eq:1.2w})_{1}$ and  integrating the resulting equality over $\mathbb{R}^n$, we similarly obtain 
\begin{equation}\label{fly2}
\begin{aligned}
\frac{\mathrm{d}}{\mathrm{d}t}\|\partial^\varsigma\hat\phi^{k+1}\|_{L^2}^2
&\leq C\big(\|\nabla \boldsymbol{u}^k\|_{L^\infty} \|\nabla \hat{\phi}^{k+1}\|_{L^2}+\|\hat \phi^{k+1}\|_{L^6} \|\nabla^2\boldsymbol{u}^k\|_{L^3}\big)\|\nabla \hat{\phi}^{k+1}\|_{L^2}\\
&\quad+C\big(\|\hat{\boldsymbol{u}}^k\|_{L^\infty}\|\nabla^2 \phi^{k}\|_{L^2}+ \|\nabla\phi^{k}\|_{L^6}\|\nabla \hat{\boldsymbol{u}}^k\|_{L^3}\big)\|\nabla \hat{\phi}^{k+1}\|_{L^2}\\
&\quad +C\|\phi^k\|_{L^\infty}\|\nabla^2\hat{\boldsymbol{u}}^k\|_{L^2}\|\nabla \hat{\phi}^{k+1}\|_{L^2},
\end{aligned}
\end{equation}
which, along with  \eqref{lgg-k}, \eqref{fly1}, Lemma \ref{ale1}, 
and the Young inequality, yields  that, for $ t\in[0,T']$ and $\omega\in (0,1)$,
\begin{equation}\label{fly3}
\frac{\mathrm{d}}{\mathrm{d}t}\| \hat{\phi}^{k+1}\|^2_{H^1}\leq  C(c_0)\big(\omega^{-1}+t^{-\frac{1}{2}}\big)\|\hat{\phi}^{k+1}\|^2_{H^1}+\omega\big(\|\hat{\boldsymbol{u}}^k\|_{L^2}^2+\|\nabla\hat{\boldsymbol{u}}^k\|_{H^1}^2\big).
\end{equation}

\smallskip
\textbf{1.3. Estimates of $\hat{\boldsymbol{u}}^{k+1}$.} Multiplying $(\ref{eq:1.2w})_2$ by $2\hat{\boldsymbol{u}}^{k+1}$ and integrating the resulting equality over $\mathbb{R}^n$, we obtain from Lemmas \ref{ale1}--\ref{GN-ineq}, \ref{lemma-L6}, and \ref{Hk-Ck-vector}, and the H\"older inequality that 
\begin{align*}
&\frac{\mathrm{d}}{\mathrm{d}t}\|\hat{\boldsymbol{u}}^{k+1}\|_{L^2}^2+2\alpha\|\nabla\hat{\boldsymbol{u}}^{k+1}\|_{L^2}^2+2\alpha\|\diver\hat{\boldsymbol{u}}^{k+1}\|_{L^2}^2\notag\\
&\leq C\|\boldsymbol{u}^k\|_{L^\infty} \|\nabla\hat{\boldsymbol{u}}^{k}\|_{L^2}\|\hat{\boldsymbol{u}}^{k+1}\|_{L^2}+C\|\hat{\boldsymbol{u}}^k\|_{L^6} \|\nabla \boldsymbol{u}^{k-1}\|_{L^3}\|\hat{\boldsymbol{u}}^{k+1}\|_{L^2}\notag\\
&\quad +C\|\hat{\boldsymbol{\psi}}^{k+1}\|_{L^2} \|\nabla \boldsymbol{u}^{k}\|_{L^6}\|\hat{\boldsymbol{u}}^{k+1}\|_{L^3}+C\delta_{2n}\|\boldsymbol{\psi}^{k}\|_{L^\infty(\mathbb{R}^2)} \|\nabla\hat{\boldsymbol{u}}^{k}\|_{L^2}\|\hat{\boldsymbol{u}}^{k+1}\|_{L^2}\notag\\
&\quad +C\delta_{3n}\|\boldsymbol{\psi}^{k}\|_{L^6(\mathbb{R}^3)} \|\nabla\hat{\boldsymbol{u}}^{k}\|_{L^2}\|\hat{\boldsymbol{u}}^{k+1}\|_{L^3(\mathbb{R}^3)}+C\|\nabla \hat{\boldsymbol{u}}^{k+1}\|_{L^2} \|\hat\phi^{k+1}\|_{L^2}\notag\\
&\leq C\|\boldsymbol{u}^k\|_{H^2} \|\nabla\hat{\boldsymbol{u}}^{k}\|_{L^2}\|\hat{\boldsymbol{u}}^{k+1}\|_{L^2} +C\|\hat{\boldsymbol{u}}^k\|_{H^1} \|\nabla \boldsymbol{u}^{k-1}\|_{H^1}\|\hat{\boldsymbol{u}}^{k+1}\|_{L^2}\\
&\quad +C\|\hat{\boldsymbol{\psi}}^{k+1}\|_{L^2} \|\nabla \boldsymbol{u}^{k}\|_{H^1}\|\hat{\boldsymbol{u}}^{k+1}\|_{L^2}^\frac{6-n}{6}\|\nabla\hat{\boldsymbol{u}}^{k+1}\|_{L^2}^\frac{n}{6}\notag\\
&\quad +C\delta_{2n}\|\boldsymbol{\psi}^{k}\|_{D^1} \|\nabla\hat{\boldsymbol{u}}^{k}\|_{L^2}\|\hat{\boldsymbol{u}}^{k+1}\|_{L^2}\notag\\
&\quad +C\delta_{3n}\|\boldsymbol{\psi}^{k}\|_{D^1} \|\nabla\hat{\boldsymbol{u}}^{k}\|_{L^2}\|\nabla\hat{\boldsymbol{u}}^{k+1}\|_{L^2}^\frac{1}{2}\|\hat{\boldsymbol{u}}^{k+1}\|_{L^2}^\frac{1}{2}  +C \|\nabla \hat{\boldsymbol{u}}^{k+1}\|_{L^2}\|\hat\phi^{k+1}\|_{L^2},\notag
\end{align*}
which, along with \eqref{lgg-k} and the Young inequality, implies that 
\begin{equation}\label{ghbbb}
\begin{aligned}
&\frac{\mathrm{d}}{\mathrm{d}t}\|\hat{\boldsymbol{u}}^{k+1}\|_{L^2}^2+\alpha\|\nabla\hat{\boldsymbol{u}}^{k+1}\|_{L^2}^2\\
& \leq C(c_0) \omega^{-1} \big(\|\hat{\boldsymbol{u}}^{k+1}\|_{L^2}^2+\|\hat\phi^{k+1}\|_{L^2}^2+\|\hat{\boldsymbol{\psi}}^{k+1}\|_{L^2}^2\big)+\omega\big(\|\hat{\boldsymbol{u}}^k\|_{L^2}^2+\|\nabla\hat{\boldsymbol{u}}^k\|_{L^2}^2\big).    
\end{aligned}
\end{equation}

Similarly, applying $2\partial^\varsigma\hat{\boldsymbol{u}}^{k+1}\partial^{\varsigma}$ with the multi-index $|\varsigma|=1$ to $(\ref{eq:1.2w})_2$ and then integrating the resulting equality over $\mathbb{R}^n$, we have 
\begin{align*}
&\frac{\mathrm{d}}{\mathrm{d}t}\|\partial^\varsigma\hat{\boldsymbol{u}}^{k+1}\|^2_{L^2}+2\alpha\|\partial^\varsigma\nabla \hat{\boldsymbol{u}}^{k+1} \|^2_{L^2}+2\alpha\| \partial^\varsigma\!\diver \hat{\boldsymbol{u}}^{k+1}\|^2_{L^2}\\
&\leq C \|(\nabla \boldsymbol{u}^{k-1},\nabla \boldsymbol{u}^k)\|_{L^6} \|\nabla\hat{\boldsymbol{u}}^{k}\|_{L^3}\|\nabla \hat{\boldsymbol{u}}^{k+1}\|_{L^2} +\|\boldsymbol{u}^k\|_{L^\infty} \|\nabla^2\hat{\boldsymbol{u}}^{k}\|_{L^2}\|\nabla \hat{\boldsymbol{u}}^{k+1}\|_{L^2}\\
&\quad  +C\|\hat{\boldsymbol{u}}^{k}\|_{L^\infty} \|\nabla^2 \boldsymbol{u}^{k-1}\|_{L^2} \|\nabla \hat{\boldsymbol{u}}^{k+1}\|_{L^2}+C\|\hat{\boldsymbol{\psi}}^{k+1}\|_{L^2} \|\nabla\boldsymbol{u}^{k}\|_{L^\infty}\|\partial^\varsigma\nabla \hat{\boldsymbol{u}}^{k+1}\|_{L^2}\\
&\quad +C\delta_{2n} \|\boldsymbol{\psi}^{k}\|_{L^\infty(\mathbb{R}^2)} \|\nabla^2 \hat{\boldsymbol{u}}^{k}\|_{L^2} \|\nabla \hat{\boldsymbol{u}}^{k+1}\|_{L^2} +C\delta_{3n} \|\boldsymbol{\psi}^{k}\|_{L^6(\mathbb{R}^3)} \|\nabla^2 \hat{\boldsymbol{u}}^{k}\|_{L^2} \|\nabla \hat{\boldsymbol{u}}^{k+1}\|_{L^3(\mathbb{R}^3)} \\
&\quad +C\|\nabla\boldsymbol{\psi}^{k}\|_{L^2} \|\nabla \hat{\boldsymbol{u}}^{k}\|_{L^6} \|\nabla \hat{\boldsymbol{u}}^{k+1}\|_{L^3}+C\|\nabla\hat\phi^{k+1}\|_{L^2}\| \partial^\varsigma\nabla\hat{\boldsymbol{u}}^{k+1}\|_{L^2}\\
&\leq C \|(\nabla \boldsymbol{u}^{k-1},\nabla \boldsymbol{u}^k)\|_{H^1}\|\nabla\hat{\boldsymbol{u}}^{k}\|_{H^1} \|\nabla \hat{\boldsymbol{u}}^{k+1}\|_{L^2}  +\|\boldsymbol{u}^k\|_{H^2} \|\nabla^2\hat{\boldsymbol{u}}^{k}\|_{L^2}\|\nabla \hat{\boldsymbol{u}}^{k+1}\|_{L^2}\\
& \quad +C\|\hat{\boldsymbol{u}}^{k}\|_{H^2} \|\nabla^2 \boldsymbol{u}^{k-1}\|_{L^2} \|\nabla \hat{\boldsymbol{u}}^{k+1}\|_{L^2} \!+\!C\|\hat{\boldsymbol{\psi}}^{k+1}\|_{L^2} \|\nabla\boldsymbol{u}^{k}\|_{L^2}^\frac{4-n}{4}\|\nabla^3\boldsymbol{u}^{k}\|_{L^2}^\frac{n}{4}\|\partial^\varsigma\nabla \hat{\boldsymbol{u}}^{k+1}\|_{L^2}\\
&\quad +C\|\boldsymbol{\psi}^{k}\|_{D^1} \|\nabla^2\hat{\boldsymbol{u}}^{k}\|_{L^2}\big(\delta_{2n}\|\nabla \hat{\boldsymbol{u}}^{k+1}\|_{L^2}+\delta_{3n}\|\nabla^2\hat{\boldsymbol{u}}^{k+1}\|_{L^2}^\frac{1}{2} \|\nabla \hat{\boldsymbol{u}}^{k+1}\|_{L^2}^\frac{1}{2}\big)\\
&\quad +C\|\nabla\boldsymbol{\psi}^{k}\|_{L^2} \|\nabla\hat{\boldsymbol{u}}^{k}\|_{H^1}\|\nabla^2 \hat{\boldsymbol{u}}^{k+1}\|_{L^2}^\frac{n}{6} \|\nabla \hat{\boldsymbol{u}}^{k+1}\|_{L^2}^\frac{6-n
}{6} +C\|\nabla\hat\phi^{k+1}\|_{L^2}\|\partial^\varsigma\nabla\hat{\boldsymbol{u}}^{k+1}\|_{L^2},
\end{align*}
which, along with \eqref{lgg-k}, \eqref{ghbbb}, 
and  the Young inequality, yields 
that 
\begin{equation}\label{gogo13}
\begin{aligned}
&\frac{\mathrm{d}}{\mathrm{d}t}\|\hat{\boldsymbol{u}}^{k+1}\|^2_{H^1}+\alpha \|\nabla \hat{\boldsymbol{u}}^{k+1} \|^2_{H^1}\\
&\leq  C(c_0)\big(\omega^{-1}\!+\!t^{-\frac{n}{4}}\big) \big(\|\hat{\boldsymbol{u}}^{k+1}\|^2_{H^1}\!+\!\|\hat{\phi}^{k+1}\|^2_{H^1}+ \|\hat{\boldsymbol{\psi}}^{k+1}\|^2_{L^2}\big)\!+\!\omega\big(\|\hat{\boldsymbol{u}}^k\|_{L^2}^2\!+\!\|\nabla\hat{\boldsymbol{u}}^k\|_{H^1}^2\big),
\end{aligned}    
\end{equation}
for $ t\in[0,T']$ and $\omega\in (0,1)$.

\smallskip
\textbf{1.4.} Finally, define the energy function:
\begin{equation*}\begin{split}
\hat{\mathcal{E}}^{k+1}(t)=&\sup_{\tau\in [0,t]} \|\hat{\phi}^{k+1}(\tau)\|^2_{H^1}+\sup_{\tau\in [0,t]}\|\hat{\boldsymbol{u}}^{k+1}(\tau)\|^2_{H^1}+\sup_{\tau\in [0,t]}\|\hat{\boldsymbol{\psi}}^{k+1}(\tau)\|^2_{L^2}.
\end{split}
\end{equation*}
Collecting \eqref{go64aa}, \eqref{fly3}, and \eqref{gogo13} yields that
\begin{equation*}
\frac{\mathrm{d}}{\mathrm{d}t}\hat{\mathcal{E}}^{k+1}+\alpha \|\nabla\hat{\boldsymbol{u}}^{k+1}\|^2_{H^1}\leq  C(c_0)\big(\omega^{-1}+t^{-\frac{1}{2}}+t^{-\frac{n}{4}}\big)\hat{\mathcal{E}}^{k+1}+\omega\hat{\mathcal{E}}^{k}+\omega \|\nabla\hat{\boldsymbol{u}}^{k}\|^2_{H^1} .
\end{equation*}
which, along with the Gr\"onwall inequality, implies that 
\begin{equation}\label{eee}
\begin{split}
&\hat{\mathcal{E}}^{k+1}(t)+\alpha\int_{0}^{t}\|\nabla\hat{\boldsymbol{u}}^{k+1}\|^2_{H^1}\mathrm{d}s\\
& \leq C(c_0)\exp\Big(C(c_0)\big(\frac{t}{\omega}+t^\frac{1}{2}+t^\frac{4-n}{4}\big)\Big)\Big(\omega t\hat{\mathcal{E}}^{k}(t)+\omega\int_{0}^{t}\|\nabla\hat{\boldsymbol{u}}^{k}\|^2_{H^1}\mathrm{d}s\Big).
\end{split}
\end{equation}

Choose $\omega\in (0,1)$ and $T_* \in (0,\min\{1,T'\})$ small enough such that
\begin{equation*}
\omega=T_*^\frac{1}{2}, \qquad  \exp\Big(C(c_0)\big(T_*^\frac{1}{2}+T_*^\frac{4-n}{4}\big)\Big)T_*^\frac{1}{2}\leq \min\big\{\frac{1}{4},\frac{\alpha}{4}\big\}.    
\end{equation*}
Then we conclude from \eqref{eee} that 
\begin{equation*}\begin{split}
\sum_{k=0}^\infty\Big(\hat{\mathcal{E}}^{k+1}(T_*)+ \alpha\int_{0}^{T_*}\|\nabla\hat{\boldsymbol{u}}^{k+1}\|^2_{H^1}\mathrm{d}s\Big)\leq C(c_0)<\infty,
\end{split}
\end{equation*}
which implies that the whole sequence $(\phi^k,\boldsymbol{u}^k,\boldsymbol{\psi}^k)$ converges to some limit $(\phi,\boldsymbol{u},\boldsymbol{\psi})$ in the following sense as $k\to\infty$: 
\begin{equation}\label{str}
\begin{split}
&\phi^k\rightarrow \phi \ \ \ \text{in} \ \  L^\infty([0,T_*];H^1(\mathbb{R}^n)),\qquad \boldsymbol{u}^k\rightarrow \boldsymbol{u} \ \ \text{in} \ \  L^\infty ([0,T_*];H^1(\mathbb{R}^n)),\\
&\boldsymbol{\psi}^k\rightarrow \boldsymbol{\psi} \ \  \text{in} \ \  L^\infty([0,T_*];L^2(B_R)) \qquad \text{for any $R>0$},
\end{split}
\end{equation}
where $B_R=\{\boldsymbol{x}\in\mathbb R^n: |\boldsymbol{x}|\leq R\}$. 

On the other hand, by virtue of the uniform estimates \eqref{lgg-k}, there exists a subsequence (still denoted by) $(\phi^k, \boldsymbol{u}^k,\boldsymbol{\psi}^k)$ converging to the limit $(\phi, \boldsymbol{u},\boldsymbol{\psi})$ in the weak or weak*  sense. According to the lower semi-continuity of norms, the corresponding estimates in \eqref{lgg} for $(\phi,\boldsymbol{u},\boldsymbol{\psi})$ still hold. Therefore, it is direct to show that $(\phi,\boldsymbol{u},\boldsymbol{\psi})$ is a weak solution of system \eqref{eqn1} in the sense of distributions and 
satisfies the following regularities:
\begin{equation}\label{rjkqq}
\begin{aligned}
&\phi^\frac{1}{\gamma-1} \in L^\infty([0,T_*];L^{1}(\mathbb{R}^n)),\\[1mm]
&(\nabla\phi,\phi_t) \in L^\infty([0,T_*];H^1(\mathbb{R}^n)), \ \ (\nabla \boldsymbol{\psi}, \boldsymbol{\psi}_t) \in L^\infty([0,T_*]; L^2(\mathbb{R}^n)),\\[1mm]
&\partial_t^l\boldsymbol{u}\in L^\infty([0,T_*]; H^{2-2l}(\mathbb{R}^n))\cap L^2([0,T_*] ; D^{3-2l}(\mathbb{R}^n)), \ \ l=0,1,\\[1mm]
&\sqrt{t}\partial_t^l\boldsymbol{u}\in L^{\infty}([0,T_*]; D^{3-2l}(\mathbb{R}^n)),\quad \sqrt{t} \partial_t^{l+1}\boldsymbol{u}\in L^2([0,T_*]; D^{2-2l}(\mathbb{R}^n)), \ \ l=0,1. 
\end{aligned}
\end{equation}

\smallskip
\textbf{2. Uniqueness.}   Let $(\phi_1,\boldsymbol{u}_1,\boldsymbol{\psi}_1)$ and $(\phi_2,\boldsymbol{u}_2,\boldsymbol{\psi}_2)$ be two regular solutions of 
problem \eqref{eqn1} satisfying the uniform estimates in \eqref{lgg}. Set
\begin{equation*}
\hat{\phi}=\phi_1-\phi_2,\qquad\hat{\boldsymbol{u}}=\boldsymbol{u}_1-\boldsymbol{u}_2,\qquad  \hat{\boldsymbol{\psi}}=\boldsymbol{\psi}_1-\boldsymbol{\psi}_2.
\end{equation*}
Then $(\hat{\phi},\hat{\boldsymbol{u}},\hat{\boldsymbol{\psi}})$ satisfies $(\hat{\phi},\hat{\boldsymbol{u}},\hat{\boldsymbol{\psi}})|_{t=0}=(0,\boldsymbol{0},\boldsymbol{0})$ and the system:
\begin{equation}\label{zhuzhu}
\begin{cases}
\hat{\phi}_t+\boldsymbol{u}_1\cdot \nabla\hat{\phi} +(\gamma-1)\hat{\phi}\diver \boldsymbol{u}_1 =\hat{\mathcal{R}}_1,\\[2pt]
\hat{\boldsymbol{u}}_t+ \boldsymbol{u}_1\cdot\nabla \hat{\boldsymbol{u}}+\nabla\hat\phi +L\hat{\boldsymbol{u}}=\boldsymbol{\psi}_1\cdot Q(\hat{\boldsymbol{u}})+\hat{\mathcal{R}}_2,\\[2pt]
\hat{\boldsymbol{\psi}}_t+ \boldsymbol{u}_1\cdot\nabla\hat{\boldsymbol{\psi}}+\nabla\boldsymbol{u}_1\cdot\hat{\boldsymbol{\psi}}+\nabla \diver \hat{\boldsymbol{u}}=\hat{\mathcal{R}}_3,
\end{cases}
\end{equation}
where $\hat{\mathcal{R}}_i$ $(i=1,2,3)$ are defined by
\begin{align*}
\hat{\mathcal{R}}_1&=-\hat{\boldsymbol{u}}\cdot\nabla\phi_{2}-(\gamma-1)\phi_{2}\diver \hat{\boldsymbol{u}},\qquad  \hat{\mathcal{R}}_2 =\hat{\boldsymbol{\psi}}\cdot Q(\boldsymbol{u}_2)- \hat{\boldsymbol{u}}\cdot \nabla \boldsymbol{u}_{2},\\ 
\hat{\mathcal{R}}_3&=-\sum_{l=1}^n A_l(\hat{\boldsymbol{u}})\partial_l\boldsymbol{\psi}_{2}-B(\hat{\boldsymbol{u}})\boldsymbol{\psi}_{2}.
\end{align*}

Next, let
\begin{equation*}
\hat{\mathcal{E}}(t)=\sup_{s\in[0,t]}\|\hat{\phi}(s)\|^2_{H^1}+\sup_{s\in[0,t]}\|\hat{\boldsymbol{u}}(s)\|^2_{H^1}+\sup_{s\in[0,t]}\|\hat{\boldsymbol{\psi}}(s)\|^2_{ L^2}.
\end{equation*}
Similarly to the arguments in Steps 1.1--1.4 above, we can show that 
\begin{equation}\label{gonm}\begin{split}
\frac{\mathrm{d}}{\mathrm{d}t}\hat{\mathcal{E}}(t)+\|\nabla \hat{\boldsymbol{u}}(t)\|^2_{H^1}\leq F(t)\hat{\mathcal{E}}(t) \qquad \text{for some $F(t)\in L^1(0,T_*)$},
\end{split}
\end{equation}
which, along with the Gr\"onwall inequality and $\hat{\mathcal{E}}(0)=0$, leads to $\hat{\mathcal{E}}(t)=0$ for all $t\in [0,T_*]$, that is, $(\hat{\phi},\hat{\boldsymbol{u}},\hat{\boldsymbol{\psi}})=(0,\boldsymbol{0},\boldsymbol{0})$, which leads to the uniqueness. Moreover, by the same argument used in Step 2 of the proof 
for Lemma \ref{lem1}, we see that \eqref{psi=log-phi-0} holds.

\medskip
\textbf{3. Time-continuity and spherical symmetry of the solutions.} The time-continuity follows directly from the same arguments as in Steps 2--3 in \S\ref{solo}, except for replacing $T$ with $T_*$.  To show that $(\phi,\boldsymbol{u},\boldsymbol{\psi})$ is spherically symmetric, taking form \eqref{1.9'}, let $\mathcal{O}\in \mathrm{SO}(n)$ and set 
\begin{equation*}
\hat\phi(t,\boldsymbol{x})=\phi(t,\mathcal{O}\boldsymbol{x}),\qquad \hat{\boldsymbol{u}}(t,\boldsymbol{x})=\mathcal{O}^\top \boldsymbol{u}(t,\mathcal{O}\boldsymbol{x}),\qquad \hat{\boldsymbol{\psi}}(t,\boldsymbol{x})=\mathcal{O}^\top \boldsymbol{\psi}(t,\mathcal{O}\boldsymbol{x}).
\end{equation*}
Then, following the analogous calculations in Steps 3.1--3.3 of the proof of Lemma \ref{lem1} with $\boldsymbol{w}$ replaced by $\hat{\boldsymbol{u}}$, we 
can show that $(\hat\phi,\hat{\boldsymbol{u}},\hat{\boldsymbol{\psi}})$ is also a solution to system \eqref{eqn1} taking the initial data $(\phi_0,\boldsymbol{u}_0,\boldsymbol{\psi}_0)$. Finally, following the same argument as in Step 3.4 of the proof of Lemma \ref{lem1}, 
we obtain the spherical symmetry of $(\phi,\boldsymbol{u},\boldsymbol{\psi})$.

\medskip
\textbf{4. Initial data with $\boldsymbol{\psi}_0\in L^\infty(\mathbb{R}^3)\cap D^1(\mathbb{R}^3)$ in the 3-D case.} So far, we have derived Theorem \ref{thh1} 
for the 2-D case (note that $\boldsymbol{\psi}\in L^\infty([0,T_*]\times \mathbb{R}^2)$ is a direct consequence of Lemma \ref{Hk-Ck-vector} and $\boldsymbol{\psi}\in C([0,T_*];D^1(\mathbb{R}^2))$. Now we focus on Theorem \ref{thh1} for the 3-D case. Assume that the initial data $(\phi_0,\boldsymbol{u}_0,\boldsymbol{\psi}_0)$ with $\phi_0>0$ are spherically symmetric and satisfy \eqref{eq;th2.1-2}. Then, according to the proof of Steps 1--3 above, there exists a unique spherically symmetric solution $(\phi,\boldsymbol{u},\boldsymbol{\psi})$ of the nonlinear problem \eqref{eqn1}--\eqref{e1.3'} in $[0,T_*]\times \mathbb{R}^n$, 
satisfying \eqref{psi=log-phi-0}, and  (i)--(iv) of Lemma \ref{lem1} with $T$ replaced by $T_*$. Thus, it remains to show that 
\begin{equation}\label{claim-psi-linfty}
\boldsymbol{\psi}\in L^\infty([0,T_*]\times \mathbb{R}^3).
\end{equation}
To this end, we introduce the so-called effective velocity:
\begin{equation}\label{def-vv}
\boldsymbol{v}(t,\boldsymbol{x}):=\boldsymbol{u}(t,\boldsymbol{x})+2\alpha\boldsymbol{\psi}(t,\boldsymbol{x})=\boldsymbol{u}(t,\boldsymbol{x})+\frac{2\alpha}{\gamma-1}\nabla\log\phi(t,\boldsymbol{x}).
\end{equation}
Then combining $\eqref{eqn1}_2$ and \eqref{eqn1psi} with \eqref{def-vv} yields that the equation of $\boldsymbol{v}$ takes the form:
\begin{equation}\label{eq-vv}
\boldsymbol{v}_t+\boldsymbol{u}\cdot\nabla\boldsymbol{v}+\frac{\gamma-1}{2\alpha}\phi(\boldsymbol{v}-\boldsymbol{u})=\boldsymbol{0},
\end{equation}
where we have used the fact that $\nabla\boldsymbol{v}=(\nabla\boldsymbol{v})^\top$, since $\boldsymbol{v}$ is spherically symmetric. 

Next, we define the flow mapping $\boldsymbol{\eta}(t,\boldsymbol{x})$ associated with $\boldsymbol{u}$:
\begin{equation}
\boldsymbol{\eta}_t(t,\boldsymbol{x})=\boldsymbol{u}(t,\boldsymbol{\eta}(t,\boldsymbol{x})) \qquad \text{with} \ \ \boldsymbol{\eta}(0,\boldsymbol{x})=\boldsymbol{x}.    
\end{equation}
Then, via the above flow mapping, \eqref{eq-vv} becomes
\begin{equation}\label{3066}
\frac{\mathrm{d}}{\mathrm{d}t}\boldsymbol{v}(t,\boldsymbol{\eta}(t,\boldsymbol{x}))+\frac{\gamma-1}{2\alpha}(\phi \boldsymbol{v})(t,\boldsymbol{\eta}(t,\boldsymbol{x})) =\frac{\gamma-1}{2\alpha}(\phi \boldsymbol{u})(t,\boldsymbol{\eta}(t,\boldsymbol{x})). 
\end{equation}
Solving the above ODE leads to 
\begin{equation}
\begin{aligned}
\boldsymbol{v}(t,\boldsymbol{\eta}(t,\boldsymbol{x}))&=\boldsymbol{v}(0,\boldsymbol{x})\exp\Big(-\int_0^t \frac{\gamma-1}{2\alpha}\phi(\tau,\boldsymbol{\eta}(\tau,\boldsymbol{x}))\,\mathrm{d}\tau\Big)\\
&\quad +\frac{\gamma-1}{2\alpha}\!\int_0^t (\phi\boldsymbol{u})(s,\boldsymbol{\eta}(s,\boldsymbol{x}))\cdot \exp\Big(\!-\!\int_s^t \frac{\gamma-1}{2\alpha}\phi(\tau,\boldsymbol{\eta}(\tau,\boldsymbol{x}))\,\mathrm{d}\tau\Big)\mathrm{d}s,
\end{aligned}
\end{equation}
which, along with $\phi>0$, implies that
\begin{equation*}
\|\boldsymbol{v}(t)\|_{L^\infty(\mathbb{R}^3)}
\leq C\|\boldsymbol{v}(0)\|_{L^\infty(\mathbb{R}^3)}+CT_*\sup_{t\in[0,T_*]}\|\phi\|_{L^\infty(\mathbb{R}^3)}\|\boldsymbol{u}\|_{L^\infty(\mathbb{R}^3)}\leq C(1+T_*),
\end{equation*}
that is, $\boldsymbol{v}\in L^\infty([0,T_*]\times\mathbb{R}^3)$. Therefore, by \eqref{def-vv}, we derive claim \eqref{claim-psi-linfty}.

The proof of Theorem \ref{thh1} is completed.

\subsubsection{Proof of {\rm Theorem \ref{zth1}}} \label{bani2}

First, it follows from Theorem \ref{thh1} that there exist $T_*>0$ and a unique $2$-order regular solution $(\phi,\boldsymbol{u},\boldsymbol{\psi})$ of problem \eqref{eqn1}--\eqref{e1.3'} in $[0,T_*]\times\mathbb{R}^n$ satisfying \eqref{er2} and \eqref{psi=log-phi-0}, which is spherically symmetric taking form \eqref{1.9'}.

Next, based on \eqref{bianhuan}, we define
\begin{equation*} 
\rho(t,\boldsymbol{x}):=\Big(\frac{\gamma-1}{A\gamma}\Big)^{\frac{1}{\gamma-1}}\phi^{\frac{1}{\gamma-1}}(t,\boldsymbol{x}).
\end{equation*}
Clearly, $\rho(t,\boldsymbol{x})>0$ in $[0,T_*]\times \mathbb{R}^n$, due to $\phi(t,\boldsymbol{x})>0$. By a direct calculation, we have 
\begin{equation*}
\frac{\partial\rho}{\partial \phi}(t,\boldsymbol{x})=\frac{1}{\gamma-1}\Big(\frac{\gamma-1}{A\gamma}\Big)^{\frac{1}{\gamma-1}}\phi^{\frac{2-\gamma}{\gamma-1}}(t,\boldsymbol{x}).
\end{equation*}
Then multiplying $\eqref{eqn1}_1$ by $\frac{\partial \rho}{\partial \phi}(t,\boldsymbol{x})$ yields the continuity equation $\eqref{eq:1.1benwen}_1$, and multiplying $\eqref{eqn1}_2$ by $\rho(t,\boldsymbol{x})$ gives the momentum equations $\eqref{eq:1.1benwen}_2$.

Therefore, $(\rho,\boldsymbol{u})$ satisfies the Cauchy problem \eqref{eq:1.1benwen}--\eqref{e1.3} in the sense of distributions and, clearly,  satisfies (ii)--(iii) in Definition \ref{cjk} with $s=2$ and $T$ replaced by $T_*$, and \eqref{er2} with $T$ replaced by $T_*$. In summary, the Cauchy problem \eqref{eq:1.1benwen}--\eqref{e1.3} admits a 2-unique regular solution $(\rho, \boldsymbol{u})$.

Finally, since $(\phi,\boldsymbol{u})$ are spherically symmetric, $(\rho,\boldsymbol{u})$ are also spherically symmetric taking form \eqref{duichenxingshi}. The proof of  Theorem \ref{zth1} is completed.

\subsection{Local well-posedness of the 3-order regular solutions with far-field vacuum}\label{section-local-3-regular}

We now prove Theorem \ref{thh133} in \S \ref{linear233}--\S \ref{bani33}.
At the end of \S \ref{section-local-3-regular}, we also show that 
this theorem indeed implies Theorem \ref{zth2}.

\subsubsection{Linearization}\label{linear233}
We start  with the proof of Theorem \ref{thh133} by considering  the  linearized problem \eqref{li4} with $(\phi_0,\boldsymbol{u}_0,\boldsymbol{\psi}_0)$, $\phi_0>0$, satisfying \eqref{th78qq33}. In this case, the vector function $\boldsymbol{w}=(w_1,\cdots\!,w_n)^{\top}\in \mathbb{R}^n$ given in \eqref{li4} is supposed to be spherically symmetric satisfying $\boldsymbol{w}(0,\boldsymbol{x})=\boldsymbol{u}_0(\boldsymbol{x})$ and, for any $T>0$,
\begin{equation}\label{vg33}
\begin{aligned}
&\boldsymbol{w}(t,\boldsymbol{x})|_{|\boldsymbol{x}|=0}=\boldsymbol{0}, \quad \boldsymbol{w}(t,\boldsymbol{x})=w(t,|\boldsymbol{x}|)\frac{\boldsymbol{x}}{|\boldsymbol{x}|} \ \qquad \text{for $t\in [0,T]$},\\
&\partial_t^l\boldsymbol{w}\in C([0,T]; H^{3-2l}(\mathbb{R}^n))\cap L^2([0,T]; D^{4-2l}(\mathbb{R}^n)), \quad l=0,1,\\[1mm]
&\sqrt{t}\partial_t^l\boldsymbol{w}\in L^{\infty}([0,T]; D^{4-2l}(\mathbb{R}^n)),\quad l=0,1,2,\\[1mm]
&\sqrt{t}\partial_t^{l+1}\boldsymbol{w} \in L^2([0,T]; D^{3-2l}(\mathbb{R}^n)),\quad l=0,1.
\end{aligned}
\end{equation}
 
We then establish the global well-posedness of the linearized problem \eqref{li4} with \eqref{vg33}, which follows from classical arguments as in \cite{CK3,evans} and the similar discussion in the proof of Lemma \ref{lem1}. 

\begin{lem}\label{lem133}
Let $n=2$ or $3$, and  \eqref{cd1local} hold. Assume that the initial data $(\phi_0,\boldsymbol{u}_0,\boldsymbol{\psi}_0)(\boldsymbol{x})$ are spherically symmetric and satisfy \eqref{th78qq33}. Then, for any $T>0$, there exists a unique solution $(\phi,\boldsymbol{u},\boldsymbol{\psi})(t,\boldsymbol{x})$  of the linearized problem \eqref{li4} in $[0,T]\times\mathbb{R}^n$ such that \eqref{psi=log-phi-0} holds and 
\begin{align*}
\mathrm{(i)}& \ (\phi, \boldsymbol{u}) \ \text{with $\phi>0$ satisfies this problem in the sense of distributions};\\
\mathrm{(ii)}& \ \phi^\frac{1}{\gamma-1}\in C([0,T];L^1(\mathbb{R}^n)), \  \ (\nabla\phi,\phi_t) \in C([0,T]; H^{2}(\mathbb{R}^n)), \\
& \boldsymbol{\psi} \in C([0,T]; D^1(\mathbb{R}^n)\cap D^2(\mathbb{R}^n)),\ \  \boldsymbol{\psi}_t\in C([0,T]; H^1(\mathbb{R}^n)), \\
&\boldsymbol{\psi}_{tt}\in L^2([0,T_*];L^2(\mathbb{R}^n)),\ \
\phi_{tt}\in C([0,T_*];L^2(\mathbb{R}^n))\cap L^2([0,T_*];D^1(\mathbb{R}^n));\\
\mathrm{(iii)}& \ 
\boldsymbol{u}(t,\boldsymbol{x})|_{|\boldsymbol{x}|=0}=\boldsymbol{0} \ \ \text{for} \ \ t\in [0,T],\\ 
& \ \partial_t^l\boldsymbol{u}\in C([0,T]; H^{3-2l}(\mathbb{R}^n))\cap L^2([0,T]; D^{4-2l}(\mathbb{R}^n)), \quad l=0,1;\\
\mathrm{(iv)}& \ \sqrt{t}\partial_t^l\boldsymbol{u}\in L^{\infty}([0,T]; D^{4-2l}(\mathbb{R}^n)),\quad l=0,1,2,\\
& \ \sqrt{t}\partial_t^{l+1}\boldsymbol{u} \in L^2([0,T]; D^{3-2l}(\mathbb{R}^n)),\quad l=0,1.
\end{align*}
Moreover, $(\phi,\boldsymbol{u},\boldsymbol{\psi})$ is spherically symmetric which takes form \eqref{1.9'}. 
\end{lem}

\subsubsection{The uniform {\it a priori} estimates}\label{uape33}
Let $(\phi,\boldsymbol{u},\boldsymbol{\psi})(t,\boldsymbol{x})$ be a solution in $[0,T]\times\mathbb{R}^n$ obtained in Lemma \ref{lem133}. We now establish the corresponding  {\it a priori} estimates. For this purpose, we first choose a constant $c_0>0$ such that 
\begin{equation}\label{houmian33}
\begin{split}
2+\big\|\phi_0^{\frac{1}{\gamma-1}}\big\|_{L^1}+\|\nabla\phi_0\|_{H^2}+\|\boldsymbol{u}_0\|_{H^3}+\|\nabla\boldsymbol{\psi}_0\|_{H^1}\leq c_0.
\end{split}
\end{equation}
We assume that there exist $T^*\in (0,T)$ and constants $c_i$ ($i=1,2,3$) such that $1< c_0\leq c_1 \leq c_2 \leq c_3\leq c_4$ and
\begin{equation}\label{jizhu133}
\begin{aligned}
\sup_{t\in [0,T^*]}\|\boldsymbol{w}\|^2_{H^1}+\int_{0}^{T^*}\big(\|\nabla \boldsymbol{w}\|^2_{H^1}+\|\boldsymbol{w}_t\|_{L^2}^2\big)\,\mathrm{d}t \leq c^2_1,\\ 
\sup_{t\in [0,T^*]}\|(\nabla^2\boldsymbol{w},\boldsymbol{w}_t)\|^2_{L^2}+\int_{0}^{T^*} \|(\nabla^3\boldsymbol{w},\nabla\boldsymbol{w}_t)\|^2_{L^2} \,\mathrm{d}t \leq c^2_2,\\
\sup_{t\in [0,T^*]}\|(\nabla^3\boldsymbol{w},\nabla\boldsymbol{w}_t)\|^2_{L^2}+\int_{0}^{T^*} \|(\nabla^4\boldsymbol{w},\nabla^2\boldsymbol{w}_t,\boldsymbol{w}_{tt})\|^2_{L^2} \,\mathrm{d}t \leq c^2_3,\\
\sup_{t\in [0,T^*]} t\|(\nabla^4\boldsymbol{w},\nabla^2\boldsymbol{w}_t,\boldsymbol{w}_{tt})\|^2_{L^2}+\int_{0}^{T^*} t\|(\nabla^3\boldsymbol{w}_t,\nabla\boldsymbol{w}_{tt})\|^2_{L^2}\,\mathrm{d}t \leq c^2_4,
\end{aligned}
\end{equation}
where  $T^*$ and  $c_i$ ($i=1,2,3,4$) will be determined  later, depending only on $c_0$ and the fixed constants $(\alpha,\gamma,A, T)$.

First, we can follow the proof of Lemma \ref{bos-2-regular} to derive the following estimates for $(\phi, \boldsymbol{\psi})$.
\begin{lem}\label{bos} 
For any $t\in [0, T_1]$ with $T_1:=\min \{T^{*}, (1+Cc_3)^{-2}\}$,
\begin{gather*}
\big\|\phi^\frac{1}{\gamma-1}(t)\big\|_{L^1}\leq c_0, \quad \|\phi(t)\|_{L^\infty}+\|\nabla\phi(t)\|_{H^1}\leq Cc_0^{7\gamma-3},\\[2pt]
\|\nabla^{k-1}\phi_t(t)\|_{L^2}\leq Cc_0^{7\gamma-3}c_k \  (k=1,2),\ \ \|\boldsymbol{\psi}(t)\|_{L^\infty\cap D^1}\leq Cc_0, \quad \|\boldsymbol{\psi}_t(t)\|_{L^2}\leq Cc_0c_2.
\end{gather*}
\end{lem}

\begin{proof}
We only show how the $L^\infty(\mathbb{R}^n)$-estimate ($n=2,3$) of $\boldsymbol{\psi}$ for $t\in [0,T_1]$ can be obtained. 

When $n=2$, it follows directly from Lemma \ref{Hk-Ck-vector} that,  
for $t\in [0,T_1]$,
\begin{equation*}
\|\boldsymbol{\psi}(t)\|_{L^\infty(\mathbb{R}^2)}\leq C\|\boldsymbol{\psi}\|_{D^1(\mathbb{R}^2)}\leq Cc_0.
\end{equation*}
When $n=3$, consider the flow mapping $\boldsymbol{\xi}:[0,T]\times \mathbb{R}^3\to \mathbb{R}^3$ satisfying 
\begin{equation}\label{flow-xi}
\boldsymbol{\xi}_t(t,\boldsymbol{x})=\boldsymbol{w}(t,\boldsymbol{\xi}(t,\boldsymbol{x})) \qquad \text{with $\boldsymbol{\xi}(0,\boldsymbol{x})=\boldsymbol{x}$}. 
\end{equation}
Then $\eqref{li4}_3$ can be written as
\begin{equation*}
\frac{\mathrm{d}}{\mathrm{d}t}\boldsymbol{\psi}(t,\boldsymbol{\xi}(t,\boldsymbol{x}))=-(\nabla \boldsymbol{w}\cdot \boldsymbol{\psi})(t,\boldsymbol{\xi}(t,\boldsymbol{x}))-\nabla\diver \boldsymbol{w}(t,\boldsymbol{\xi}(t,\boldsymbol{x})).
\end{equation*}
Integrating the above over $[0,t]$ and taking the $L^\infty(\mathbb{R}^3)$-norm of the resulting equality, along with \eqref{jizhu133}, Lemma \ref{ale1}, 
and the H\"older inequality, yield that
\begin{align*}
\|\boldsymbol{\psi}(t)\|_{L^\infty}&\leq \|\boldsymbol{\psi}_0\|_{L^\infty}+\Big(\sup_{s\in [0,t]}\|\nabla\boldsymbol{w}\|_{L^\infty}\Big)\int_0^t \|\boldsymbol{\psi}\|_{L^\infty}\,\mathrm{d}s+\int_0^t \|\nabla^2\boldsymbol{w}\|_{L^\infty}\,\mathrm{d}s\\
&\leq \|\boldsymbol{\psi}_0\|_{L^\infty}+Cc_3\int_0^t \|\boldsymbol{\psi}\|_{L^\infty}\,\mathrm{d}s+\sqrt{t}\Big(\int_0^t \|\nabla^2\boldsymbol{w}\|_{H^2}^2\,\mathrm{d}s\Big)^\frac{1}{2}\\
&\leq \|\boldsymbol{\psi}_0\|_{L^\infty}+Cc_3\int_0^t \|\boldsymbol{\psi}\|_{L^\infty}\,\mathrm{d}s+c_3\sqrt{t},
\end{align*}
which, along with the Gr\"onwall inequality and Lemmas \ref{GN-ineq} and \ref{lemma-L6}, implies that, for all $t\in [0,T_1]$,
\begin{align*} 
\|\boldsymbol{\psi}(t)\|_{L^\infty(\mathbb{R}^3)}&\leq \|\boldsymbol{\psi}_0\|_{L^\infty(\mathbb{R}^3)}+Cc_0\leq \|\boldsymbol{\psi}_0\|_{L^6(\mathbb{R}^3)}^\frac{1}{2}\|\boldsymbol{\psi}_0\|_{D^2(\mathbb{R}^3)}^\frac{1}{2}+Cc_0\\
&\leq \|\boldsymbol{\psi}_0\|_{D^1(\mathbb{R}^3)}^\frac{1}{2}\|\boldsymbol{\psi}_0\|_{D^2(\mathbb{R}^3)}^\frac{1}{2}+Cc_0\leq Cc_0.
\end{align*}
\end{proof}

Next, we obtain the higher-order estimates for $(\phi,\boldsymbol{\psi})$.
\begin{lem}\label{bos'} 
For any $t\in [0,T_2]$ with $T_2:=\min \{T_1, (1+Cc_4)^{-2}\}$,
\begin{gather*}
\|\nabla^3\phi(t)\|_{L^2}\leq Cc_0^{7\gamma-2},\quad \|\nabla^2\phi_t(t)\|_{L^2}\leq Cc_0^{7\gamma-2}c_3,\quad \|\nabla^2\boldsymbol{\psi}(t)\|_{L^2}\leq Cc_0, \\
\|\nabla\boldsymbol{\psi}_t(t)\|_{L^2}\leq Cc_0c_3,\ \ \|\phi_{tt}(t)\|_{L^2}^2+\int_0^t \big\|(\nabla\phi_{tt},\boldsymbol{\psi}_{tt})\big\|_{L^2}^2 \,\mathrm{d}s\leq Cc_3^{14\gamma}.
\end{gather*}
\end{lem}
\begin{proof}
We divide the proof into three steps.

\smallskip
\textbf{1. Estimates on $(\nabla^3\phi,\nabla^2\boldsymbol{\psi})$.} Applying $\partial^\varsigma$ with multi-index $|\varsigma|=2$ to $\eqref{li4}_3$, then multiplying the resulting equality by $\partial^\varsigma\boldsymbol{\psi}$, 
and integrating over $\mathbb{R}^n$, we obtain from \eqref{jizhu133}, Lemmas \ref{bos} and \ref{ale1}, and  the H\"older inequality that
\begin{align*}
\frac{\mathrm{d}}{\mathrm{d}t} \|\partial^\varsigma\boldsymbol{\psi}\|_{L^2}^2
&\leq C\big(\|\nabla \boldsymbol{w}\|_{L^\infty}\| \nabla^2\boldsymbol{\psi}\|_{L^2} +\|\nabla^2 \boldsymbol{w}\|_{L^4}\|\nabla\boldsymbol{\psi}\|_{L^4} +\|\nabla^4\boldsymbol{w}\|_{L^2}\big)\|\partial^\varsigma\boldsymbol{\psi}\|_{L^2}\\
&\leq C\big(\|\nabla \boldsymbol{w}\|_{H^2}\| \nabla^2\boldsymbol{\psi}\|_{L^2} +\|\nabla^2 \boldsymbol{w}\|_{H^1}\|\nabla\boldsymbol{\psi}\|_{H^1} +\|\nabla^4\boldsymbol{w}\|_{L^2}\big)\|\partial^\varsigma\boldsymbol{\psi}\|_{L^2}\\
&\leq Cc_3\|\nabla^2\boldsymbol{\psi}\|_{L^2}^2 +C (c_3^2+ c_4t^{-\frac{1}{2}})\|\nabla^2\boldsymbol{\psi}\|_{L^2}.
\end{align*}
Summing the above with respect to $\varsigma$ over $|\varsigma|=2$ leads to
\begin{equation*}
\frac{\mathrm{d}}{\mathrm{d}t} \|\nabla^2 \boldsymbol{\psi}\|_{L^2}\leq Cc_3\|\nabla^2\boldsymbol{\psi}\|_{L^2} +C(c_3^2+ c_4t^{-\frac{1}{2}}),
\end{equation*}
which, along with the Gr\"onwall inequality, yields that 
\begin{equation}\label{y4}
\|\nabla^2\boldsymbol{\psi}(t)\|_{L^2}\leq e^{Cc_3t}\big(\|\nabla^2\boldsymbol{\psi}_0\|_{L^2}+Cc_3^2t+Cc_4\sqrt{t}\big)\leq Cc_0.
\end{equation}
for all $t\in [0,T_2]$ with $T_2:=\min\{T_1,(1+Cc_4)^{-2}\}$.

Next, for the $L^2(\mathbb{R}^n)$-estimate of $\nabla^3\phi$, it follows from \eqref{psi=log-phi-0}, \eqref{y4}, Lemmas \ref{bos} and \ref{ale1}, and the H\"older inequality that, for all $t\in [0,T_2]$,
\begin{equation}\label{y55}
\begin{aligned}
\|\nabla^3\phi\|_{L^2}&\leq C\big(\|\nabla^2\phi\|_{L^2}\|\boldsymbol{\psi}\|_{L^\infty}+\|\nabla\phi\|_{L^4}\|\nabla\boldsymbol{\psi}\|_{L^4}+\|\phi\|_{L^\infty}\|\nabla^2\boldsymbol{\psi}\|_{L^2}\big) \\
&\leq C\big(\|\nabla^2\phi\|_{L^2}\|\boldsymbol{\psi}\|_{L^\infty}+\|\nabla\phi\|_{H^1}\|\nabla \boldsymbol{\psi}\|_{H^1}\!+\!\|\phi\|_{L^\infty}\|\nabla^2\boldsymbol{\psi}\|_{L^2}\big)\leq Cc_0^{7\gamma-2}. 
\end{aligned}    
\end{equation}

\smallskip
\textbf{2. Estimates on $(\phi_t,\boldsymbol{\psi}_t)$.}
First,  it follows from $\eqref{li4}_3$, \eqref{y4}, 
and Lemma \ref{ale1} that, for $ t\in [0,T_2]$,
\begin{equation}\label{y555}
\begin{aligned}
\|\nabla\boldsymbol{\psi}_t\|_{L^2} &\leq C\big(\|\nabla^2\boldsymbol{w}\|_{L^2} \|\boldsymbol{\psi}\|_{L^\infty}+\|\nabla\boldsymbol{w}\|_{L^\infty} \|\nabla\boldsymbol{\psi}\|_{L^2}\big)\\
&\quad +C\big(\|\boldsymbol{w}\|_{L^\infty}\|\nabla^2 \boldsymbol{\psi}\|_{L^2}+\|\nabla^3\boldsymbol{w}\|_{L^2}\big)\\
&\leq C\big(\|\boldsymbol{\psi}\|_{L^\infty}+\|\nabla \boldsymbol{\psi}\|_{H^1}+1\big)\|\boldsymbol{w}\|_{H^3} \le Cc_0c_3.
\end{aligned}   
\end{equation}

Then, for the $L^2(\mathbb{R}^n)$-estimate of $\nabla^2\phi_t$, it follows from \eqref{psi=log-phi-0}, \eqref{jizhu133}--\eqref{y4} and \eqref{y555}, Lemmas \ref{bos} and \ref{ale1}--\ref{GN-ineq}, and the H\"older inequality that, for all $t\in [0,T_2]$,
\begin{equation}\label{nabla2phit}
\begin{aligned}
\|\nabla^2\phi_t\|_{L^2}&\leq C\big(\|\nabla \phi_t\|_{L^2}\|\boldsymbol{\psi}\|_{L^\infty}\!+\! \|\phi_t\|_{L^4}\|\nabla \boldsymbol{\psi}\|_{L^4} \!+\!\|\nabla\phi\|_{L^4}\|\boldsymbol{\psi}_t\|_{L^4}\!+\! \|\phi\|_{L^\infty}\|\nabla\boldsymbol{\psi}_t\|_{L^2}\big)\\
&\leq C\big(\|\nabla \phi_t\|_{L^2}\|\boldsymbol{\psi}\|_{L^\infty}+ \|\phi_t\|_{H^1}\|\nabla \boldsymbol{\psi}\|_{H^1}\big)\\
&\quad + C\big(\|\nabla\phi\|_{H^1}\|\boldsymbol{\psi}_t\|_{H^1}+ \|\phi\|_{L^\infty}\|\nabla\boldsymbol{\psi}_t\|_{L^2}\big) \leq Cc_0^{7\gamma-2}c_3.
\end{aligned}
\end{equation}

\smallskip
\textbf{3. Estimates on $(\phi_{tt},\boldsymbol{\psi}_{tt})$.}
We first apply  $\partial_t$ to both sides of $\eqref{li4}_1$ and $\eqref{li4}_3$, respectively, to obtain
\begin{align*}
\phi_{tt}&=-\boldsymbol{w}_t\cdot\nabla\phi-\boldsymbol{w}\cdot\nabla\phi_t-(\gamma-1)(\phi_t\diver\boldsymbol{w}+\phi\diver\boldsymbol{w}_t),\\
\boldsymbol{\psi}_{tt}&=-\sum_{l=1}^n A_l(\boldsymbol{w}_t)\partial_l\boldsymbol{\psi}-\sum_{l=1}^n A_l(\boldsymbol{w})\partial_l\boldsymbol{\psi}_t-(B(\boldsymbol{w}_t)\boldsymbol{\psi}+B(\boldsymbol{w})\boldsymbol{\psi}_t)-\nabla\diver\boldsymbol{w}_t.
\end{align*}
It follows from the above, \eqref{psi=log-phi-0}, \eqref{jizhu133}--\eqref{nabla2phit}, Lemmas \ref{bos} and \ref{ale1}, and the H\"older inequality that 
\begin{equation}\label{phi-psi-xxx1}
\begin{aligned}
\|\boldsymbol{\psi}_{tt}\|_{L^2}&\leq \|\boldsymbol{w}_t\|_{L^4}\|\nabla\boldsymbol{\psi}\|_{L^4}+\|\boldsymbol{w}\|_{L^\infty}\|\nabla\boldsymbol{\psi}_t\|_{L^2}+\|\nabla\boldsymbol{w}_t\|_{L^2}\|\boldsymbol{\psi}\|_{L^\infty}\\
&\quad +\|\nabla\boldsymbol{w}\|_{L^4}\|\boldsymbol{\psi}_t\|_{L^4}+\|\nabla^2\boldsymbol{w}_t\|_{L^2}\\
&\leq C\|\boldsymbol{w}_t\|_{H^1}\|\nabla\boldsymbol{\psi}\|_{H^1}+C\|\boldsymbol{w}\|_{H^2}\|\nabla\boldsymbol{\psi}_t\|_{L^2}+\|\nabla\boldsymbol{w}_t\|_{L^2}\|\boldsymbol{\psi}\|_{L^\infty}\\
&\quad +C\|\nabla\boldsymbol{w}\|_{H^1}\|\boldsymbol{\psi}_t\|_{H^1}+\|\nabla^2\boldsymbol{w}_t\|_{L^2}\leq C\big(c_0c_2c_3 +\|\nabla^2\boldsymbol{w}_t\|_{L^2}\big),\\
\|\phi_{tt}\|_{L^2}&\leq C\big(\|\boldsymbol{w}_t\|_{L^2}\|\nabla \phi\|_{L^\infty} +\|\boldsymbol{w}\|_{L^\infty}\|\nabla\phi_{t}\|_{L^2}\big)\\
&\quad + C\big(\|\phi_t\|_{L^\infty}\|\nabla\boldsymbol{w}\|_{L^2}+ \|\phi\|_{L^\infty}\|\nabla\boldsymbol{w}_{t}\|_{L^2}\big)\\
&\leq C\big(\|\boldsymbol{w}_t\|_{L^2}\|\nabla \phi\|_{H^2} +\|\boldsymbol{w}\|_{H^2}\|\nabla\phi_{t}\|_{L^2}\big)\\
&\quad + C\big(\|\phi_t\|_{H^2}\|\nabla\boldsymbol{w}\|_{L^2}+ \|\phi\|_{L^\infty}\|\nabla\boldsymbol{w}_{t}\|_{L^2}\big)\\
&\leq C\big(c_0^{7\gamma-2}c_2 +c_0^{7\gamma-3}c_2^2+c_0^{7\gamma-2}c_1c_3 +c_0^{7\gamma-3}c_3\big)\leq Cc_0^{7\gamma-2}c_2c_3,\\
\|\nabla\phi_{tt}\|_{L^2}&\leq C\|(\phi\boldsymbol{\psi})_{tt}\|_{L^2}\leq C\|\phi_{tt}\boldsymbol{\psi}\|_{L^2}+C\|\phi_t\boldsymbol{\psi}_{t}\|_{L^2}+C\|\phi\boldsymbol{\psi}_{tt}\|_{L^2}\\
&\leq C\|\phi_{tt}\|_{L^2}\|\boldsymbol{\psi}\|_{L^\infty}+C\|\phi_t\|_{L^\infty}\|\boldsymbol{\psi}_{t}\|_{L^2}+C\|\phi\|_{L^\infty}\|\boldsymbol{\psi}_{tt}\|_{L^2}\\
&\leq C\|\phi_{tt}\|_{L^2}\|\boldsymbol{\psi}\|_{L^\infty}+C\|\phi_t\|_{H^2}\|\boldsymbol{\psi}_{t}\|_{L^2}+C\|\phi\|_{L^\infty}\|\boldsymbol{\psi}_{tt}\|_{L^2}\\
&\leq C\big(c_0^{7\gamma-1}c_2c_3+c_0^{7\gamma-3}\|\nabla^2\boldsymbol{w}_{t}\|_{L^2}\big).
\end{aligned}
\end{equation}
Finally, the above inequalities, together with \eqref{jizhu133}, yields that, for all $t\in [0,T_2]$,
\begin{equation}
\int_0^t \big\|(\nabla\phi_{tt},\boldsymbol{\psi}_{tt})\big\|_{L^2}^2 \,\mathrm{d}s\leq Cc_3^{14\gamma+2}t+c_0^{14\gamma-6}\int_0^t\|\nabla^2\boldsymbol{w}_{t}\|_{L^2}^2\,\mathrm{d}s\leq Cc_3^{14\gamma}.
\end{equation}

The proof of  Lemma \ref{bos} is completed.
\end{proof}

Finally, we derive the uniform energy estimates for $\boldsymbol{u}$.

\begin{lem}\label{llm33}
For any $t\in [0,T_3]$ with $T_3:=\min\big\{T_2, (1+Cc_4)^{-28\gamma}\big\}$,
\begin{align*}
\|\boldsymbol{u}(t)\|_{H^1}^2+\int_0^t\big(\| \nabla \boldsymbol{u}\|_{H^1}^2+\|\boldsymbol{u}_t\|_{L^2}^2\big)\,\mathrm{d} s&\leq Cc_0^2,\\
\|(\nabla^2\boldsymbol{u},\boldsymbol{u}_t)(t)\|_{L^2}^2+\int_0^t \|(\nabla^3\boldsymbol{u},\nabla\boldsymbol{u}_t)\|_{L^2}^2\,\mathrm{d} s&\leq Cc_1^{14\gamma-6} c_2^\frac{n}{2},\\
\|(\nabla^3\boldsymbol{u},\nabla\boldsymbol{u}_t)(t)\|_{L^2}^2+\int_0^t \|(\nabla^4\boldsymbol{u},\nabla^2\boldsymbol{u}_t,\boldsymbol{u}_{tt})\|_{L^2}^2\,\mathrm{d}s&\leq Cc_2^{14\gamma-6}c_3^\frac{n}{2},\\
t\|(\nabla^4\boldsymbol{u},\nabla^2\boldsymbol{u}_t,\boldsymbol{u}_{tt})(t)\|_{L^2}^2+\int_0^t s\|(\nabla^3\boldsymbol{u}_t,\nabla\boldsymbol{u}_{tt})\|_{L^2}^2\,\mathrm{d}s&\leq Cc_0^4.
\end{align*}
\end{lem}

\begin{proof}
The $H^2(\mathbb{R}^n)$-estimates of $\boldsymbol{u}$ can be derived via the same argument as in Steps 1--3 of the proof of Lemma \ref{llm3}. We thus only establish the $D^3(\mathbb{R}^n)$ and  time-weighted estimates of $\boldsymbol{u}$.
We divide the proof into three steps.

\smallskip
\textbf{1. $D^3(\mathbb{R}^n)$-estimate on $\boldsymbol{u}$.} Suppose that $Z(\boldsymbol{w})$ is defined as in \eqref{eqZZ}. We first show the $L^2(\mathbb{R}^n)$-estimate of $\nabla^2 Z(\boldsymbol{w})$. Indeed, it follows from \eqref{eqZZ}, \eqref{jizhu133}, Lemmas \ref{bos}--\ref{bos'}, and \ref{ale1}, and the H\"older inequality that 
\begin{equation}\label{Z-2}
\begin{aligned}
\|\nabla^2 Z(\boldsymbol{w})\|_{L^2}&\leq C\|(\boldsymbol{\psi},\boldsymbol{w})\|_{L^\infty}\|\nabla^3\boldsymbol{w}\|_{L^2}+C\|(\nabla\boldsymbol{\psi},\nabla\boldsymbol{w})\|_{L^4}\|\nabla^2\boldsymbol{w}\|_{L^4}\\
&\quad +C\|(\nabla^2\boldsymbol{\psi},\nabla^2\boldsymbol{w})\|_{L^2}\|\nabla\boldsymbol{w}\|_{\infty}+C\|\nabla^3\phi\|_{L^2},\\
&\leq C\big(\|\boldsymbol{\psi}\|_{L^\infty}+\|\boldsymbol{w}\|_{H^2}+\|\nabla\boldsymbol{\psi}\|_{H^1}\big)\|\nabla \boldsymbol{w}\|_{H^2}+C\|\nabla^3\phi\|_{L^2}\\
&\leq Cc_2c_3+Cc_0^{7\gamma-2}\leq Cc_2^{7\gamma-3}c_3.
\end{aligned}
\end{equation}

Next, it follows from \eqref{lrq-pre}, \eqref{jizhu133}, and the Young inequality that 
\begin{align*}
&\alpha\frac{\mathrm{d}}{\mathrm{d}t}\big(\|\nabla \boldsymbol{u}_t\|_{L^2}^2+\|\diver \boldsymbol{u}_t\|_{L^2}^2\big)+\|\boldsymbol{u}_{tt}\|_{L^2}^2\\ 
&\leq C\Big(c_0c_2\|\nabla \boldsymbol{w}\|_{H^2}+c_0c_2\|\nabla\boldsymbol{w}_t\|_{L^2} + c_0\|\nabla\boldsymbol{w}_t\|_{L^2}^\frac{1}{2}\|\nabla^2\boldsymbol{w}_t\|_{L^2}^\frac{1}{2}+c_0^{7\gamma-3}c_2\Big)\|\boldsymbol{u}_{tt}\|_{L^2}\\
&\leq  C\Big(c_0c_3^\frac{1}{2}\|\nabla^2\boldsymbol{w}_t\|_{L^2}^\frac{1}{2}+c_3^{7\gamma-2}\Big)\|\boldsymbol{u}_{tt}\|_{L^2}\leq C\big(c_0^2c_3 \|\nabla^2\boldsymbol{w}_t\|_{L^2} + c_3^{14\gamma-4}\big)+ \frac{1}{8}\|\boldsymbol{u}_{tt}\|_{L^2}^2.
\end{align*}
Integrating the above over $[\tau,t]$, together with \eqref{jizhu133} and the H\"older inequality, yields that, for all $t\in [0,T_3]$ with  $T_3:=\min\big\{T_2, (1+Cc_4)^{-28\gamma}\big\}$,
\begin{equation}\label{4.8}
\begin{aligned}
&\,\|\nabla \boldsymbol{u}_t(t)\|_{L^2}^2+\int_\tau^t \|\boldsymbol{u}_{tt}\|_{L^2}^2\,\mathrm{d}s\leq C\Big(\|\nabla \boldsymbol{u}_t(\tau)\|_{L^2}^2+ c_0^2c_3\int_0^t \|\nabla^2\boldsymbol{w}_t\|_{L^2}\,\mathrm{d}s+ c_3^{14\gamma-4}t\Big)\\
&\leq C\Big(\|\nabla \boldsymbol{u}_t(\tau)\|_{L^2}^2+ c_0^2c_3\sqrt{t}\Big(\int_0^t \|\nabla^2\boldsymbol{w}_t\|_{L^2}^2\,\mathrm{d}s\Big)^\frac{1}{2} + c_3^{14\gamma-4}t\Big) \leq C(\|\nabla \boldsymbol{u}_t(\tau)\|_{L^2}^2+ c_0^2).
\end{aligned}
\end{equation}
For $\|\nabla \boldsymbol{u}_t(\tau)\|_{L^2}$, it follows from $\eqref{li4}_2$, the time continuity of $(\phi,\boldsymbol{u},\boldsymbol{\psi})$, \eqref{vg33}--\eqref{houmian33}, Lemma \ref{ale1}, and the H\"older inequality that 
\begin{align*}
\limsup_{\tau\to 0}\|\nabla \boldsymbol{u}_t(\tau)\|_{L^2}&\leq C\big(\|\nabla\boldsymbol{u}_0\|_{L^4}^2+\|\boldsymbol{u}_0\|_{L^\infty}\|\nabla^2\boldsymbol{u}_0\|_{L^2}+\|\nabla^2\phi_0\|_{L^2}+\|\nabla^3 \boldsymbol{u}_0\|_{L^2}\big)\\
&\quad +C\big(\|\nabla\boldsymbol{\psi}_0\|_{L^4}\|\nabla\boldsymbol{u}_0\|_{L^4}+\|\boldsymbol{\psi}_0\|_{L^\infty}\|\nabla^2\boldsymbol{u}_0\|_{L^2}\big)\\
&\leq C\big(\|\boldsymbol{u}_0\|_{H^2}^2+\|\nabla^2\phi_0\|_{L^2}+\|\boldsymbol{u}_0\|_{H^3}\big) \\
&\quad +C\big(\|\nabla\boldsymbol{\psi}_0\|_{H^1} +\|\boldsymbol{\psi}_0\|_{L^\infty}\big)\|\boldsymbol{u}_0\|_{H^2} \leq Cc_0^2.
\end{align*}
Thus, based on the above estimates, we let $\tau\to 0$ in \eqref{4.8} to obtain 
\begin{equation}\label{4.9}
\|\nabla \boldsymbol{u}_t(t)\|_{L^2}^2+\int_0^t \|\boldsymbol{u}_{tt}\|_{L^2}^2\,\mathrm{d}s\leq Cc_0^4 \qquad \text{for all $t\in [0,T_3]$}.
\end{equation}
Besides, it follows from \eqref{nabla3u}, \eqref{nabla2x1tu}, \eqref{jizhu133}, \eqref{4.9}, and the H\"older inequality that 
\begin{equation}\label{l2l2-utxx}
\begin{aligned}
&\|\nabla^3\boldsymbol{u}\|_{L^2}^2+\int_0^t\|\nabla^2\boldsymbol{u}_t\|_{L^2}^2\,\mathrm{d}s\\
&\leq  C\big(\|\nabla\boldsymbol{u}_t\|_{L^2}^2+c_2^{14\gamma-6}\|\nabla^3\boldsymbol{w}\|_{L^2}^\frac{n}{2}+c_2^{14\gamma-6}\big)\\
&\quad +C\int_0^t \big(\|\boldsymbol{u}_{tt}\|_{L^2}^2+c_3^6\|\nabla^2\boldsymbol{w}_t\|_{L^2}+c_3^{14\gamma-4}\big)\,\mathrm{d}s\\
&\leq  C\big(c_0^4+c_2^{14\gamma-6}c_3^\frac{n}{2}+c_2^{14\gamma-6}+c_0^4+c_3^7\sqrt{t}+c_3^{14\gamma-8}t\big)\leq Cc_2^{14\gamma-6}c_3^\frac{n}{2}.
\end{aligned}
\end{equation}

Finally, it follows from \eqref{eqzz-tuoyuan}, \eqref{jizhu133}, \eqref{Z-2}, \eqref{l2l2-utxx}, Lemma \ref{lemma-usefulZ}, and the classical regularity theory for elliptic equations in Lemma \ref{df3} that  
\begin{equation}\label{uxxxx}
\|\nabla^4 \boldsymbol{u}\|_{L^2}^2\leq C\|(\nabla^2\boldsymbol{u}_{t},\nabla^2Z(\boldsymbol{w}))\|_{L^2}^2 \leq C\big(\|\nabla^2\boldsymbol{u}_{t}\|_{L^2}^2+c_2^{14\gamma-6}c_3^2\big). 
\end{equation}
Integrating above over $[0,t]$ yields that, for all $t\in [0,T_3]$,
\begin{equation*}
\int_0^t\|\nabla^4 \boldsymbol{u}\|_{L^2}^2\,\mathrm{d}s\leq C\int_0^t\|\nabla^2\boldsymbol{u}_t\|_{L^2}^2\,\mathrm{d}s+ c_2^{14\gamma-6}c_3^2t\leq Cc_2^{14\gamma-6}c_3^\frac{n}{2}.  
\end{equation*}

\smallskip
\textbf{2. Time-weighted estimates on $\boldsymbol{u}$.} We first establish the $L^2(\mathbb{R}^n)$-estimate of $(Z(\boldsymbol{w}))_{tt}$. Indeed, it follows from \eqref{eqZZ}, \eqref{jizhu133}, \eqref{phi-psi-xxx1}, Lemmas \ref{bos}--\ref{bos'} and \ref{ale1}, and the H\"older inequality that, for all $t\in [0,T_3]$,
\begin{equation}\label{ztt-2}
\begin{aligned}
\|(Z(\boldsymbol{w}))_{tt}\|_{L^2}&\leq C\|(\boldsymbol{\psi},\boldsymbol{w})\|_{L^\infty} \|\nabla\boldsymbol{w}_{tt}\|_{L^2}+C\|(\boldsymbol{\psi}_t,\boldsymbol{w}_t)\|_{L^4} \|\nabla\boldsymbol{w}_{t}\|_{L^4}\\
&\quad +C\|(\boldsymbol{\psi}_{tt},\boldsymbol{w}_{tt})\|_{L^2} \|\nabla\boldsymbol{w}\|_{L^\infty}+C\|\nabla\phi_{tt}\|_{L^2}\\
&\leq C\big(\|\boldsymbol{\psi}\|_{L^\infty}+\|\boldsymbol{w}\|_{H^2}\big) \|\nabla\boldsymbol{w}_{tt}\|_{L^2}+C\|(\boldsymbol{\psi}_t,\boldsymbol{w}_t)\|_{H^1} \|\nabla\boldsymbol{w}_{t}\|_{H^1}\\
&\quad +C\|(\boldsymbol{\psi}_{tt},\boldsymbol{w}_{tt})\|_{L^2} \|\nabla\boldsymbol{w}\|_{H^2}+C\|\nabla\phi_{tt}\|_{L^2}\\
&\leq C\big(c_2\|\nabla\boldsymbol{w}_{tt}\|_{L^2}\!+\!c_3^3\!+\!c_3^2\|\nabla^2\boldsymbol{w}_{t}\|_{L^2} \!+\!c_3\|(\boldsymbol{\psi}_{tt},\boldsymbol{w}_{tt})\|_{L^2}\!+\!\|\nabla\phi_{tt}\|_{L^2}\big),\\
&\leq Cc_3^{7\gamma+1}\big(\|(\nabla\boldsymbol{w}_{tt},\nabla^2\boldsymbol{w}_{t},\boldsymbol{w}_{tt})\|_{L^2}+1\big).
\end{aligned}
\end{equation}

Next, formally applying $\partial_{t}^2$ to \eqref{eqZZ} gives
\begin{equation}\label{eqzz-tt}
\boldsymbol{u}_{ttt}+L\boldsymbol{u}_{tt}=(Z(\boldsymbol{w}))_{tt}.    
\end{equation}
Then multiplying above by $\boldsymbol{u}_{tt}$ and integrating over $I$ lead to 
the following energy equality: for $t\in [\tau,T']$ ($\tau>0$),
\begin{equation}\label{energy-uttt}
\frac{1}{2}\frac{\mathrm{d}}{\mathrm{d}t}\|\boldsymbol{u}_{tt}\|_{L^2}^2+\alpha\|\nabla\boldsymbol{u}_{tt}\|_{L^2}^2+\alpha\|\diver\boldsymbol{u}_{tt}\|_{L^2}^2=\int_{\mathbb{R}^n} (Z(\boldsymbol{w}))_{tt}\cdot\boldsymbol{u}_{tt}\,\mathrm{d}\boldsymbol{x}.    
\end{equation}
Here, we temporarily assume that the above energy
equality holds. The rigorous proof of \eqref{energy-uttt} will be given in 
Step 3 below.

Then it follows from \eqref{ztt-2}, \eqref{energy-uttt}, and the Young inequality that 
\begin{equation}\label{cal-uttt}
\begin{aligned}
&\frac{\mathrm{d}}{\mathrm{d}t}\|\boldsymbol{u}_{tt}\|_{L^2}^2+\alpha\|\nabla\boldsymbol{u}_{tt}\|_{L^2}^2+\alpha\|\diver\boldsymbol{u}_{tt}\|_{L^2}^2\\ 
&\leq  Cc_3^{14\gamma+2}t\big(\|(\nabla\boldsymbol{w}_{tt},\nabla^2\boldsymbol{w}_{t},\boldsymbol{w}_{tt})\|_{L^2}^2+1\big)+\frac{C}{t}\|\boldsymbol{u}_{tt}\|_{L^2}^2. 
\end{aligned}
\end{equation}
Multiplying the above by $t$ and integrating the resulting inequality over $[\tau,t]$ ($\tau>0$), we obtain from \eqref{jizhu133} and \eqref{4.9} that, for all $t\in [\tau,T_3]$,
\begin{equation}\label{sss}
\begin{aligned}
&t\|\boldsymbol{u}_{tt}(t)\|_{L^2}^2+\int_\tau^ts\|\nabla\boldsymbol{u}_{tt}\|_{L^2}^2\,\mathrm{d}s\\
&\leq  \tau\|\boldsymbol{u}_{tt}(\tau)\|_{L^2}^2+Cc_3^{14\gamma+2} \Big(\int_0^ts\big(s\|\nabla\boldsymbol{w}_{tt}\|_{L^2}^2+c_4^2\big)\,\mathrm{d}s+t^2\Big) +C\int_0^t \|\boldsymbol{u}_{tt}\|_{L^2}^2\,\mathrm{d}s\\
&\leq  \tau\|\boldsymbol{u}_{tt}(\tau)\|_{L^2}^2+C\big(c_3^{14\gamma+2}c_4^2t +c_3^{14\gamma+2}c_4^2t^2+c_0^4\big)\leq \tau\|\boldsymbol{u}_{tt}(\tau)\|_{L^2}^2+Cc_0^4.
\end{aligned}
\end{equation}
For $\tau\|\boldsymbol{u}_{tt}(\tau)\|_{L^2}$, by \eqref{4.9} and Lemma \ref{bjr}, 
there exists a sequence $\{\tau_k\}_{k=1}^\infty$ such that 
\begin{equation*}
\tau_k\to 0, \quad\, \tau_k\|\boldsymbol{u}_{tt}(\tau_k)\|_{L^2}\to 0 \qquad\,\, \text{as $k\to\infty$}.
\end{equation*}
Thus, setting $\tau=\tau_k$ and letting $\tau_k\to 0$ in \eqref{sss} yield that, for all $t\in [0,T_3]$,
\begin{equation}\label{4.19}
t\|\boldsymbol{u}_{tt}(t)\|_{L^2}^2+\int_0^ts\|\nabla\boldsymbol{u}_{tt}\|_{L^2}^2\,\mathrm{d}s\leq Cc_0^4.
\end{equation}

Next, it follows from \eqref{nabla2x1tu}, \eqref{jizhu133}, \eqref{4.9}, 
\eqref{uxxxx}, and \eqref{4.19} that, for all $t\in [0,T_3]$,
\begin{equation}\label{utxx-uxxxx}
\begin{aligned}
t\|\nabla^2\boldsymbol{u}_t\|_{L^2}^2 &\leq Ct\big(\|\boldsymbol{u}_{tt}\|_{L^2}^2+c_3^6\|\nabla^2\boldsymbol{w}_t\|_{L^2}+c_3^{14\gamma-6}\big) \leq C\big(c_0^4+ c_3^7\sqrt{t}\big)\leq Cc_0^4,\\
t\|\nabla^4 \boldsymbol{u}\|_{L^2}^2&\leq Ct\big(\|\nabla^2\boldsymbol{u}_{t}\|_{L^2}^2+c_2^{14\gamma-6}c_3^2\big)\leq Cc_0^4. 
\end{aligned}    
\end{equation}

Finally, we establish the $L^2(\mathbb{R}^n)$-estimate of $\nabla^3\boldsymbol{u}_t$. First, it follows from \eqref{eqZZ}, \eqref{jizhu133}, Lemmas \ref{bos}--\ref{bos'} and \ref{ale1}, and the H\"older inequality that
\begin{align*}
\|\nabla (Z(\boldsymbol{w}))_t\|_{L^2}&\leq C\big(\|(\nabla\boldsymbol{\psi}_t,\nabla\boldsymbol{w}_t)\|_{L^2}\|\nabla\boldsymbol{w}\|_{L^\infty}+\|(\boldsymbol{\psi}_t,\boldsymbol{w}_t)\|_{L^4}\|\nabla^2\boldsymbol{w}\|_{L^4}\big)\\
&\quad +C\big(\|(\nabla\boldsymbol{\psi},\nabla\boldsymbol{w})\|_{L^4}\|\nabla\boldsymbol{w}_t\|_{L^4}\!+\!\|(\boldsymbol{\psi},\boldsymbol{w})\|_{L^\infty}\|\nabla^2\boldsymbol{w}_t\|_{L^2}\big)\!+\!\|\nabla^2\phi_t\|_{L^2}\\
&\leq C\big(\|(\boldsymbol{\psi}_t,\boldsymbol{w}_t)\|_{H^1}\|\nabla\boldsymbol{w}\|_{H^2} +\|(\nabla\boldsymbol{\psi},\nabla\boldsymbol{w})\|_{H^1}\|\nabla\boldsymbol{w}_t\|_{H^1}\big)\\
&\quad +C\big(\|\boldsymbol{\psi}\|_{L^\infty}+\|\boldsymbol{w}\|_{H^2}\big)\|\nabla^2\boldsymbol{w}_t\|_{L^2}+\|\nabla^2\phi_t\|_{L^2}\\
&\leq Cc_3^2\|\nabla^2\boldsymbol{w}_t\|_{L^2}+Cc_3^{7\gamma-1}.
\end{align*}
Next, we obtain from the above, \eqref{eqzz-t-tuoyuan}, \eqref{jizhu133}, and the classical regularity  theory for elliptic equations in Lemma \ref{df3} that
\begin{align*}
\|\nabla^3\boldsymbol{u}_t\|_{L^2}\leq C\big\|\big(\nabla\boldsymbol{u}_{tt},\nabla(Z(\boldsymbol{w}))_{t}\big)\big\|_{L^2}
\leq C\big(\|\nabla\boldsymbol{u}_{tt}\|_{L^2} + c_3^2\|\nabla^2\boldsymbol{w}_t\|_{L^2}+ c_3^{7\gamma-1}\big),
\end{align*}
which, along with \eqref{4.19}, implies that, for all $t\in [0,T_3]$,
\begin{align*}
\int_0^t s\|\nabla^3\boldsymbol{u}_t\|_{L^2}^2\,\mathrm{d}s
&\leq C\int_0^t s\big(\|\nabla\boldsymbol{u}_{tt}\|_{L^2}^2 + c_3^4\|\nabla^2\boldsymbol{w}_t\|_{L^2}^2+ c_3^{14\gamma-2}\big)\,\mathrm{d}s\\
&\leq C\int_0^t s \|\nabla\boldsymbol{u}_{tt}\|_{L^2}^2\,\mathrm{d}s+Cc_3^6t+Cc_3^{14\gamma-2}t^2\leq Cc_0^4.
\end{align*}

\smallskip
\textbf{3. Justification of the energy equality \eqref{energy-uttt}.} Let $\boldsymbol{\varphi}\in H^1(\mathbb{R}^n)$ be any vector valued test function. We can first multiply \eqref{eqzz-t} by $\boldsymbol{\varphi}$ and then integrate the resulting equality over $\mathbb{R}^n$ to obtain
\begin{equation*}
\int_{\mathbb{R}^n}\boldsymbol{u}_{tt}\cdot\boldsymbol{\varphi}\,\mathrm{d}\boldsymbol{x}=-\alpha\int_{\mathbb{R}^n}(\nabla\boldsymbol{u}_t:\nabla\boldsymbol{\varphi}+\diver\boldsymbol{u}_t \diver\boldsymbol{\varphi})\,\mathrm{d}\boldsymbol{x}+\int_{\mathbb{R}^n}(Z(\boldsymbol{w}))_{t}\cdot\boldsymbol{\varphi}\,\mathrm{d}\boldsymbol{x}.
\end{equation*}

Differentiate the above with respect to $t$. Then it follows from 
the calculation of \eqref{cal-uttt} and {\it a priori} assumptions $(\nabla\boldsymbol{u}_{tt},\nabla\phi_{tt},\boldsymbol{\psi}_{tt})\in L^2([\tau,T'];L^2(\mathbb{R}^n))$ shown in Lemma \ref{lem133} that 
\begin{equation}
\frac{\mathrm{d}}{\mathrm{d}t}\int_{\mathbb{R}^n}\boldsymbol{u}_{tt}\cdot\boldsymbol{\varphi}\,\mathrm{d}\boldsymbol{x}\leq F(t)\|\boldsymbol{u}_{tt}\|_{H^1} 
\end{equation}
holds for $t\in [\tau,T']$ and for some positive function $F(t)\in L^2(\tau,T')$. According to Lemma 1.1 on \cite[page 250]{temam}, the above equality implies that $\boldsymbol{u}_{ttt}\in L^2([\tau,T'];H^{-1}(\mathbb{R}^n))$. \eqref{energy-uttt} thus follows directly from the following identity, due to Lemma \ref{triple}:
\begin{equation*}
\frac{\mathrm{d}}{\mathrm{d}t}\int_{\mathbb{R}^n}|\boldsymbol{u}_{tt}|^2 \,\mathrm{d}\boldsymbol{x}=2\langle \boldsymbol{u}_{ttt},\boldsymbol{u}_{tt}\rangle_{H^{-1}(\mathbb{R}^n)\times H^1(\mathbb{R}^n)} \qquad \text{for {\it a.e.} $t\in (\tau,T')$}.
\end{equation*}
The proof of Lemma \ref{llm33} is completed.
\end{proof} 

\smallskip
Now, choose $T^*$ 
and constants $c_i$ $(i=1,2,3,4)$ as
\begin{align*}
&T^*=T_3,\qquad c_1=C^\frac{1}{2}c_0, \qquad c_2=C^\frac{2}{4-n}c_1^\frac{28\gamma-12}{4-n}=C^\frac{14\gamma-4}{4-n}c_0^\frac{28\gamma-12}{4-n},\\
&c_3=c_4=C^\frac{2}{4-n}c_2^\frac{28\gamma-12}{4-n}= C^\frac{8-2n+(14\gamma-4)(28\gamma-12)}{(4-n)^2}c_0^\frac{(28\gamma-12)^2}{(4-n)^2} .
\end{align*}
Then it follows from Lemmas \ref{bos}--\ref{llm33} that,  for $t\in [0,T^*]$,
\begin{equation}\label{lgg33}
\begin{aligned}
\|\boldsymbol{u}(t)\|_{H^1}^2+\int_0^t\big(\|\nabla \boldsymbol{u}\|_{H^1}^2+\|\boldsymbol{u}_t\|_{L^2}^2\big)\,\mathrm{d} s&\leq c_1^2,\\
\|(\nabla^2\boldsymbol{u},\boldsymbol{u}_t)(t)\|_{L^2}^2+\int_0^t \|(\nabla^3\boldsymbol{u},\nabla\boldsymbol{u}_t)\|_{L^2}^2\,\mathrm{d} s&\leq c_2^2,\\
\|(\nabla^3\boldsymbol{u},\nabla\boldsymbol{u}_t)(t)\|_{L^2}^2+\int_0^t \|(\nabla^4\boldsymbol{u},\nabla^2\boldsymbol{u}_t,\boldsymbol{u}_{tt})\|_{L^2}^2\,\mathrm{d}s&\leq c_3^2,\\
t\|(\nabla^4\boldsymbol{u},\nabla^2\boldsymbol{u}_t,\boldsymbol{u}_{tt})(t)\|_{L^2}^2+\int_0^t s\|(\nabla^3\boldsymbol{u}_t,\nabla\boldsymbol{u}_{tt})\|_{L^2}^2\,\mathrm{d}s&\leq c_4^2,\\
\big\|\phi^\frac{1}{\gamma-1}(t)\big\|_{L^1}+\|(\phi,\boldsymbol{\psi})(t)\|_{L^\infty}+\big\|(\nabla\phi,\phi_t)(t)\big\|_{H^2}+\big\|(\nabla\boldsymbol{\psi},\boldsymbol{\psi}_t)(t)\big\|_{H^1}&\leq Cc_3^{7\gamma-1},\\[2pt]
\|\phi_{tt}(t)\|_{L^2}^2+\int_0^t \big\|(\nabla\phi_{tt},\boldsymbol{\psi}_{tt})\big\|_{L^2}^2 \,\mathrm{d}s&\leq Cc_3^{14\gamma}.
\end{aligned}
\end{equation}

\subsubsection{Proof of {\rm Theorem \ref{thh133}}}\label{bani33} With the help of the uniform estimates \eqref{lgg33}, the existence and uniqueness of the $3$-order regular solutions can be derived by an argument similar to that in Steps 1--2 in \S \ref{bani}. To show that the $3$-order regular solutions are actually classical ones, we can follow the same argument as in Steps 2--3 in \S\ref{soloo2}. We omit the details here for brevity.

\subsubsection{Proof of {\rm Theorem \ref{zth2}}}\label{bani44} 
Theorem \ref{zth2} can be proved by the same argument as in \S\ref{bani2}. We omit the details here for brevity.

\appendix

\section{Some Basic Lemmas}\label{appA}

This appendix lists some useful lemmas that have been used frequently in the previous sections. The first is the classical  Sobolev embedding theorem.  

\begin{lem}[\cite{af}]\label{ale1}
Let $\Omega \subset \mathbb{R}^n$ $(n\in \mathbb{N}^*)$ be a domain with smooth boundary, and $f\in W^{k,p}(\Omega)$ for some $k\in \mathbb{N}^*$ and $p\in [1,\infty]$. 
\begin{enumerate}
\item[$\mathrm{(i)}$] 
Let $kp\leq n$. Then $W^{k,p}(\Omega)\hookrightarrow L^s(\Omega)$ for all $s\in\big[p,\frac{np}{n-kp}\big]$ if $kp<n$ and all $s\in[p,\infty)$ if $kp=n$, and there exists a constant $C_1>0$ depending only on $(k,p,s,n,\Omega)$ such that 
\begin{equation*}
\|f\|_{L^s(\Omega)}\leq C_1\|f\|_{W^{k,p}(\Omega)}.
\end{equation*}
In particular, if $kp<n$, $s=\frac{np}{n-kp}$, 
and $f\in D^{k,p}\cap L^{\frac{np}{n-kp}}$, then 
\begin{equation}\label{class-L6}
\|f\|_{L^{\frac{np}{n-kp}}(\Omega)}\leq C_1\|\nabla^k f\|_{L^p(\Omega)}.
\end{equation}
\item[$\mathrm{(ii)}$] 
Let $(k,p)=(n,1)$. Then $W^{n,1}(\Omega)\hookrightarrow C(\bar\Omega)$ and  there exists a constant $C_2>0$ depending only on $(n,\Omega)$ such that
\begin{equation*}
\|f\|_{L^\infty(\Omega)}\leq C_2\|f\|_{W^{n,1}(\Omega)}. 
\end{equation*}
\item[$\mathrm{(iii)}$] 
Let $kp>n$. Then $W^{k,p}(\Omega)\hookrightarrow C^\ell(\bar\Omega)$ for all $\ell\in \mathbb{N}$ and $0\leq \ell<k-n/p$,  and  there exists a  constant $C_3>0$ depending only on $(k,p,\ell,n,\Omega)$ such that
\begin{equation*}
\max_{0\leq j\leq \ell}\|\nabla ^j f\|_{L^\infty(\Omega)}\leq C_3\|f\|_{W^{k,p}(\Omega)}, 
\end{equation*}
where $C^\ell(\bar\Omega)$ $(\ell\in \mathbb{N},\,C(\bar \Omega)=C^0(\bar\Omega))$ denotes the space of all functions $f$ for which $\nabla^j f$ $(0\leq j\leq \ell)$ are bounded and uniformly continuous in $\Omega$.  In particular, the following inequality holds for $f=f(r)\in H^1(0,R)$ $(R>0)${\rm :}
\begin{equation}\label{impr}
\|f\|_{L^\infty(0,R)}^2 \leq (1+R^{-1})\|f\|_{L^2(0,R)}^2+\|f_r\|_{L^2(0,R)}^2.
\end{equation}
\end{enumerate}
Moreover, if $\Omega=\mathbb{R}^n$, then the above constants $(C_1,C_2,C_3)$ are independent of $\Omega$.

\end{lem} 
\begin{proof}
We only give the proof for inequality \eqref{impr}. It follows from  the fundamental theorem of calculus and the Young inequality that, for all $0\leq \tilde r<r\leq R$,
\begin{align*}
|f(r)|^2&=|f(\tilde r)|^2+2\int_{\tilde r}^r f(z)f_r(z)\,\mathrm{d}z\\
&\leq |f(\tilde r)|^2+\|f\|_{L^2(0,R)}^2+\|f_r\|_{L^2(0,R)}^2.
\end{align*}
Then integrating the above over $[0,R]$ with respect to $\tilde r$ yields
\begin{align*}
R|f(r)|^2&\leq\|f\|_{L^2(0,R)}^2+R\|f\|_{L^2(0,R)}^2+R\|f_r\|_{L^2(0,R)}^2,
\end{align*}
which leads to 
\begin{equation*}
|f(r)|^2\leq  (1+R^{-1})\|f\|_{L^2(0,R)}^2+\|f_r\|_{L^2(0,R)}^2 \qquad 
\text{for all $r\in [0,R]$}.
\end{equation*}
\end{proof}

The second lemma concerns the well-known Gagliardo-Nirenberg inequality.  

\begin{lem}[\cite{ln}]\label{GN-ineq}
Assume that  $f\in L^{p}(\mathbb{R}^n)\cap D^{\ell,q}(\mathbb{R}^n)$ for   $1 \leq p,  q \leq \infty$.  Let real numbers $(\epsilon,b)$ and  natural numbers $(n,\ell,j)$ satisfy
\begin{equation*}
\frac{1}{{b}} = \frac{j}{n} + \Big( \frac{1}{q} - \frac{\ell}{n} \Big) \epsilon + \frac{1 - \epsilon}{p}, \qquad \frac{j}{\ell} \leq \epsilon \leq 1.
\end{equation*}
Then $f\in D^{j,{b}}(\mathbb{R}^n)$, and  there exists $C>0$ depending only on $(\ell,n,j,p,q,\epsilon)$ such that
\begin{equation*} 
\| \nabla^{j} f \|_{L^{{b}}} \leq C \| \nabla^{\ell} f \|_{L^{q}}^{\epsilon} \| f \|_{L^{p}}^{1 - \epsilon}.
\end{equation*}
Moreover, if $j = 0$, $\ell q < n$, and $p = \infty$, then it is necessary to make the additional assumption that either f tends to zero at infinity or that f lies in $L^s(\mathbb{R}^n)$ for some finite $s > 0${\rm;} if $1 < q < \infty$ and $\ell -j -n/q$ is a non-negative integer, then it is necessary to assume also that $\epsilon \neq 1$.
\end{lem}

The third lemma is on the Hardy inequality.
\begin{lem}[\cite{brown,opic}]\label{hardy}  
The following statements hold{\rm :}
\begin{enumerate}
\item[$\mathrm{(i)}$] Let $q\in [2,\infty)$ and $b>0$. Assume that $f=f(r)$ is a function defined on $[0,b]$ such that $r^{K +1+\frac{1}{p}-\frac{1}{q}}(f,f_r)\in L^q(0,b)$ for some $p\in [q,\infty]$ and $K>-\frac{1}{p}$ $(K>0$ if $p=\infty)$. Then $r^K f\in L^p(0,b)$ and the following inequalities hold{\rm :}
\begin{equation*}
\begin{aligned}
\|r^K f\|_{L^p(0,b)}&\leq C(K,p,q,b)\big\|r^{K+1+\frac{1}{p}-\frac{1}{q}}(f,f_r)\big\|_{L^q(0,b)} \qquad\text{if }p\in [q,\infty),\\
\|r^K f\|_{L^\infty(0,b)}&\leq C(K,q,b) \big\|r^{K+1-\frac{1}{q}}(f,f_r)\big\|_{L^q(0,b)}  \qquad\qquad\text{if }p=\infty,
\end{aligned}
\end{equation*}
where $C(K,p,q,b)$ and $C(K,q,b)$ are positive constants depending only on $(K,p,q,b)$ and  $(K,q,b)$, respectively. In particular, if $r^{K +1-\frac{1}{q}}(f,f_r)\in L^q(0,b)$ for some $K>0$, then $r^K f\in C([0,b])$. 
\item[$\mathrm{(ii)}$] Let $f=f(r)\in W^{1,p}(0,\infty)$ for some $p\in (1,\infty)$ with $f|_{r=0}=0$. Then $\frac{f}{r}\in L^p(0,\infty)$ and the following inequality holds{\rm :} 
\begin{equation*}
\Big\|\frac{f}{r}\Big\|_{L^p(0,\infty)}\leq C(p)\|f_r\|_{L^p(0,\infty)}.
\end{equation*}
\end{enumerate}

\end{lem}
\begin{proof}
For brevity, when giving the proof for (i), we only consider $b=1$, 
since the proof for general $b>0$ can be derived analogously. We divide the proof into five steps.

\smallskip
\textbf{1. Proof of (i) when $p=q$.} First, for $K>-\frac{1}{q}$, it follows from integration by parts that
\begin{align}
\int_0^1 r^{qK} |f|^q\,\mathrm{d}r&=\frac{1}{qK+1}\Big(|f|^q(1)-q\int_0^1 r^{qK+1}|f|^{q-2}ff_r \,\mathrm{d}r\Big),\label{AA1}\\
\int_0^1 r^{qK+q} |f|^q\,\mathrm{d}r 
&=\frac{1}{qK+q+1}\Big( |f|^q(1)-q\int_0^1 r^{qK+q+1}|f|^{q-2}ff_r\,\mathrm{d}r\Big).\label{AA2}
\end{align}
Combining  \eqref{AA1}--\eqref{AA2} with the H\"older and Young inequalities 
leads to
\begin{align*}
\int_0^1 r^{qK} |f|^q\,\mathrm{d}r&=\frac{qK+q+1}{qK+1} \int_0^1 r^{qK+q} |f|^q\,\mathrm{d}r +\frac{q}{qK+1} \int_0^1 r^{qK+q+1}|f|^{q-2}ff_r \,\mathrm{d}r\\
&\quad\, -\frac{q}{qK+1}\int_0^1 r^{qK+1}|f|^{q-2}ff_r \,\mathrm{d}r \\
&\leq C(K,q)\Big( \int_0^1 r^{qK+q} |f|^q\,\mathrm{d}r+ \int_0^1 r^{qK+q} |f_r|^q\,\mathrm{d}r\Big) +\frac{1}{2}\int_0^1 r^{qK} |f|^q\,\mathrm{d}r, 
\end{align*}
which implies that
\begin{equation}\label{step1-2}
\int_0^1 r^{qK} |f|^q\,\mathrm{d}r \leq C(K,q)\Big( \int_0^1 r^{qK+q} |f|^q\,\mathrm{d}r+\int_0^1 r^{qK+q} |f_r|^q\,\mathrm{d}r\Big). 
\end{equation}

\smallskip
\textbf{2. Proof of (i) when $p=\infty$.}  
Next, for any $K>0$, since $r^K f|_{r=0}=0$, it follows from \eqref{step1-2}, Lemma \ref{calculus}, and the Young inequality that
\begin{equation}\label{step-infyt}
\begin{aligned}
\|r^K f\|_{L^\infty(0,1)}^q&\leq \int_0^1 |(r^{qK} |f|^q)_r|\,\mathrm{d}r\leq qK \int_0^1  r^{qK-1} |f|^q \,\mathrm{d}r+ q\int_0^1  r^{qK} |f|^{q-1}|f_r|\,\mathrm{d}r \\
&\leq C(K,q)\Big(\int_0^1  r^{qK-1} |f|^q \,\mathrm{d}r+ \int_0^1  r^{qK+q-1} |f_r|^q \,\mathrm{d}r\Big) \\
&\leq C(K,q)\Big( \int_0^1  r^{qK+q-1} |f|^q \,\mathrm{d}r+ \int_0^1  r^{qK+q-1} |f_r|^q \,\mathrm{d}r\Big). 
\end{aligned}
\end{equation}

\smallskip
\textbf{3. Proof of (i) when $p\in (q,\infty)$.} Finally, for $K>-\frac{1}{p}$, via integration by parts,
\begin{align}
\int_0^1 r^{pK} |f|^p\,\mathrm{d}r&=\frac{1}{pK+1}\Big( |f|^p(1)-p\int_0^1 r^{pK+1}|f|^{p-2}ff_r \,\mathrm{d}r\Big),\label{AA3}\\
\int_0^1 r^{pK+p} |f|^p\,\mathrm{d}r 
&=\frac{1}{pK+p+1}\Big( |f|^{p}(1)-p\int_0^1 r^{pK+p+1}|f|^{p-2}ff_r \,\mathrm{d}r\Big).\label{AA4}
\end{align}

By \eqref{step1-2}--\eqref{step-infyt} and the H\"older inequality, we have
\begin{equation*}
\begin{aligned}
\int_0^1 r^{pK+1}|f|^{p-2}ff_r \,\mathrm{d}r&\leq \big\|r^{K+\frac{1}{p}}f\big\|_{L^\infty(0,1)}^{p-q}\big\|r^{K+\frac{1}{p}-\frac{1}{q}}f\big\|_{L^q(0,1)}\big\|r^{K+1+\frac{1}{p}-\frac{1}{q}}f_r\big\|_{L^q(0,1)}\\
&\leq C(K,p)\Big( \big\|r^{K+1+\frac{1}{p}-\frac{1}{q}}f\big\|_{L^q(0,1)}^p+\big\|r^{K+1+\frac{1}{p}-\frac{1}{q}}f_r\big\|_{L^q(0,1)}^p\Big).
\end{aligned}
\end{equation*}
Thus, based on the above inequality and the similar discussions in Step 1, we conclude the inequality 
from \eqref{AA3}--\eqref{AA4}. 

\smallskip
\textbf{4.} Now we show that, if $r^{K +1-\frac{1}{q}}(f,f_r)\in L^q(0,b)$ for some $K>0$ and $q\in [2,\infty)$, then $r^K f\in C([0,b])$. 
For brevity, we give the proof only when $q=2$ and $b=1$. 
For any $f$ satisfying $r^{K +\frac{1}{2}}(f,f_r)\in L^2(0,1)$, 
according to Theorem 7.2 in \cite{kufner}, 
there exists a sequence of $\{f^\epsilon\}_{\epsilon>0}\subset C^\infty([0,1])$ such that 
\begin{equation*}
\|r^{K +\frac{1}{2}}(f^\epsilon-f,f_r^\epsilon-f_r)\|_{L^2(0,1)}\to 0\qquad\,\,
\text{as $\epsilon\to 0$}.
\end{equation*}
Then, following the same calculation \eqref{step-infyt}, we obtain that $\{r^Kf^\epsilon\}_{\epsilon>0}$ is a Cauchy sequence in $L^\infty(0,1)$. Since $\{r^Kf^\epsilon\}_{\epsilon>0}\subset C^\infty([0,1])$, there exists a limit $g\in C([0,1])$ such that
\begin{equation*}
r^Kf^\epsilon\to g \ \text{ uniformly on $[0,1]$} \qquad \text{as $\epsilon\to 0$.}
\end{equation*}
On the other hand, since $f^\epsilon\to f$ in $H^1(\sigma,1)$ 
as $\epsilon\to 0$ for any fixed $\sigma>0$, we obtain from Lemma \ref{ale1} that, for any fixed $\sigma>0$, $f^\epsilon\to f$ uniformly on $[\sigma,1]$ as $\epsilon\to 0$ so that 
\begin{equation*}
r^Kf^\epsilon\to r^K f \ \text{ pointwisely on $(0,1]$} \qquad \text{as $\epsilon\to 0$.}
\end{equation*}
The uniqueness of the limit thus yields $r^Kf=g$ for {\it a.e.} $r\in (0,1)$ and $r^Kf\in C([0,1])$.

\medskip
\textbf{5. Proof of (ii).} We first consider $f\in C_{\rm c}^\infty((0,\infty))$. 
By integration by parts, we have
\begin{align*}
\Big\|\frac{f}{r}\Big\|_{L^p(0,\infty)}^p&=\int_0^\infty \frac{|f|^p}{r^p}\,\mathrm{d}r=\frac{p}{p-1}\int_0^\infty \frac{1}{r^{p-1}}|f|^{p-2}ff_r\,\mathrm{d}r\leq C(p)\Big\|\frac{f}{r}\Big\|_{L^p(0,\infty)}^{p-1}\|f_r\|_{L^p(0,\infty)}.
\end{align*}
Applying the Young inequality yields
\begin{equation}\label{har}
\Big\|\frac{f}{r}\Big\|_{L^p(0,\infty)}\leq C(p)\|f_r\|_{L^p(0,\infty)}.
\end{equation}

Now, for any $f\in W^{1,p}(0,\infty)$ with $f|_{r=0}=0$, there exists a sequence of 
$\{f^\epsilon\}_{\epsilon>0}\subset C_{\rm c}^\infty((0,\infty))$ such that $f^\epsilon\to f$ in $W^{1,p}(0,\infty)$ as $\epsilon\to0$. Then \eqref{har} implies that $\{\frac{f^\epsilon}{r}\}_{\epsilon>0}$ is a Cauchy sequence in $L^p(0,\infty)$, converging to some limit $g\in L^p(0,\infty)$ as $\epsilon\to 0$: 
\begin{equation*}
\frac{f^\epsilon}{r}\to g \,\, \text{ in $L^p(0,\infty)$} \qquad \text{as $\epsilon\to 0$}.
\end{equation*}
On the other hand, thanks to Lemma \ref{ale1}, $f^\epsilon\to f$ pointwisely on $(0,\infty)$ so that 
\begin{equation*}
\frac{f^\epsilon}{r}\to \frac{f}{r} \,\, \text{ for {\it a.e.} $r\in (0,\infty)$} 
\qquad \text{as $\epsilon\to 0$}.
\end{equation*}
Therefore, the uniqueness of the limit yields $\frac{f}{r}=g\in L^p(0,\infty)$.
\end{proof}

The fourth lemma is on the fundamental theorem of calculus.
\begin{lem}[\cite{realrudin}]\label{calculus}
Let $f(r)\in W^{1,1}(a,\infty)$ for some  $a\in[0,\infty)$. Then,
for any $r,r_0\in [a,\infty)$,
\begin{equation}\label{newton}
f(r)=-\int_r^{\infty} f'(z)\,\mathrm{d}z=f(r_0)+\int_{r_0}^r f'(z)\,\mathrm{d}z.
\end{equation}
Moreover, $f(r)\to 0$ as $r\to \infty$.
\end{lem}

\begin{proof}
First, Lemma \ref{ale1} implies that $f\in C([a,\infty))$, 
and thus $f(r)$ is well-defined for each $r\in [a,\infty)$. 
Next, it suffices to check the first formula in \eqref{newton}, 
since the second one can be derived in the same way.

For any $f\in W^{1,1}(a,\infty)$ with $a\in[0,\infty)$, there exists a sequence of 
$\{f^\epsilon\}_{\epsilon>0}\subset C_{\rm c}^\infty([a,\infty))$ satisfying
\begin{equation}\label{AA9}
f^\epsilon \to f \ \ \text{ in $W^{1,1}(a,\infty)$} \qquad\text{as $\epsilon\to 0$}.
\end{equation}
This, together with Lemma \ref{ale1}, implies also that
\begin{equation}\label{AA10}
f^\epsilon \to f \ \ \text{ in $L^\infty(a,\infty)$} \qquad\text{as $\epsilon\to 0$}.
\end{equation}

Thus, based on \eqref{AA9}--\eqref{AA10} and  
\begin{equation*}
f^\epsilon(r)=-\int_{r}^\infty (f^\epsilon)'(z)\,\mathrm{d}z,
\end{equation*}
we let $\epsilon\to 0$ in the above equality to obtain the desired formula in \eqref{newton}.

Finally, let $\chi_{E}$ be the characteristic function defined on the set $E\subset [a,\infty)$, {\it i.e.}, $\chi_E=1$ on $E$ and $\chi_E=0$ on $[a,\infty)\backslash E$. Note that $f'$ is a measurable function defined {\it a.e.} on $[a,\infty)$ due to $f'\in L^1(a,\infty)$. Then it follows from \eqref{newton}, the fact that
\begin{equation*}
\lim_{r\to\infty}f'(z)\chi_{[r,\infty)}(z)=0 \qquad \text{for {\it a.e.} $z\in [a,\infty)$},
\end{equation*}
and the Lebesgue dominated convergence theorem that
\begin{equation*}
\begin{aligned}
\lim_{r\to\infty} f(r)&=-\lim_{r\to \infty}\int_r^\infty f'(z)\,\mathrm{d}z=-\lim_{r\to \infty}\int_a^\infty f'(z)\chi_{[r,\infty)}(z)\,\mathrm{d}z\\
&=-\int_a^\infty \lim_{r\to \infty}f'(z)\chi_{[r,\infty)}(z)\,\mathrm{d}z=0.
\end{aligned}
\end{equation*}
\end{proof}

The fifth lemma is the well-known Fatou lemma which can be found in \cite{realrudin}.
\begin{lem}[\cite{realrudin}]\label{Fatou}
Let  $\{f_n\}$ be a sequence of measurable non-negative functions $f_n: \mathbb{R}^n\rightarrow [0,\infty]$.  Then $f(\boldsymbol{x}):= \liminf_{n\rightarrow \infty} f_n(\boldsymbol{x})$ is measurable and
\begin{equation*}
\int_{\mathbb{R}^n} f(\boldsymbol{x})\,\mathrm{d}\boldsymbol{x} \leq \liminf_{n\rightarrow \infty} \int_{\mathbb{R}^n} f_n(\boldsymbol{x}) \,\mathrm{d}\boldsymbol{x} .
\end{equation*}
\end{lem}

The sixth lemma is used to obtain the time-weighted estimates of the velocity.
\begin{lem}[\cite{bjr}]\label{bjr}
Let $E\subset \mathbb{R}^n$ $(n\in \mathbb{N}^*)$ be any set and $f\in L^2([0,T]; L^2(E))$. Then there exists a sequence $\{t_k\}_{k=1}^\infty$ such that
\begin{equation*}
t_k\rightarrow 0, \quad\ t_k \|f(t_k)\|^2_{L^2(E)}\rightarrow 0 \qquad\,\, \text{as $k\rightarrow\infty$}.
\end{equation*}
\end{lem}

\begin{proof}
Let $F(t)=\|f(t)\|^2_{L^2(\mathcal {O})}$. Clearly, $0\leq F(t)\in L^1(0,T)$. 
Then it suffices to show that, for any $k\ge 1$, there exists $t_k\in\big(0,\frac{T}{1+k}\big)$ such that 
\begin{equation*}
t_kF(t_k)<\frac{1}{k}\to 0 \qquad \text{as $k\rightarrow\infty$}.
\end{equation*}
Assume by contradiction that there exists some $k_0\geq 1$ such that, for any $t\in\big(0,\frac{T}{1+k_0}\big)$, 
$tF(t)\ge \frac{1}{k_0}$.
Then 
\begin{equation*}
\int_0^T F(s)\,\mathrm{d}s\ge \frac{1}{k_0}\int_0^T \frac{1}{s }\,\mathrm{d}s=\infty.
\end{equation*}
This contradicts with the fact $F(t)\in L^1(0,T)$.
Therefore, the claim holds. 
\end{proof}

The seventh lemma shows that any spherical symmetric vector field must vanish at the symmetric center.
\begin{lem}\label{rmk31}
Let $\boldsymbol{f}(\boldsymbol{x})=f(|\boldsymbol{x}|)\frac{\boldsymbol{x}}{|\boldsymbol{x}|}$ be a continuous and spherically symmetric vector field defined on $\mathbb{R}^n$. Then $\boldsymbol{f}(\boldsymbol{0})=\boldsymbol{0}$.
\end{lem}

\begin{proof}
Indeed, since $\boldsymbol{f}(\boldsymbol{x})=f(|\boldsymbol{x}|)\frac{\boldsymbol{x}}{|\boldsymbol{x}|}$ is spherically symmetric, we see that
$\boldsymbol{f}(\mathcal{O}\boldsymbol{x})=(\mathcal{O}\boldsymbol{f})(\boldsymbol{x})$ for all $\boldsymbol{x}\in \mathbb{R}^n$ and  
$\mathcal{O}\in \mathrm{SO}(n)$.
Taking $\boldsymbol{x}=\boldsymbol{0}$ leads to 
\begin{equation}\label{s0}
\boldsymbol{f}(\boldsymbol{0})=(\mathcal{O}\boldsymbol{f})(\boldsymbol{0}) \qquad \text{for all $\mathcal{O}\in \mathrm{SO}(n)$}.
\end{equation}
Next, choosing a rotation $\mathcal{O}=\mathcal{O}_1$ by 180 degrees 
about an axis perpendicular to $\boldsymbol{f}(\boldsymbol{0})$, 
that is, $(\mathcal{O}_1\boldsymbol{f})(\boldsymbol{0})=-\boldsymbol{f}(\boldsymbol{0})$, 
we obtain from \eqref{s0} that $\boldsymbol{f}(\boldsymbol{0})=-\boldsymbol{f}(\boldsymbol{0})$, 
which implies that $\boldsymbol{f}(\boldsymbol{0})=\boldsymbol{0}$. 
\end{proof}

The next lemma is on the well-known Chebyshev inequality.
\begin{lem}[\cite{lg}]\label{cheby}
Let $p\in[1,\infty)$ and $f=f(\boldsymbol{x})\in L^p(\mathbb{R}^n)$. Let
\begin{equation*}
S^\sigma:=\{\boldsymbol{x}\in \mathbb{R}^n:\, |f(\boldsymbol{x})|\geq \sigma\},
\end{equation*}
and let $\chi_{S^\sigma}=\chi_{S^\sigma}(\boldsymbol{x})$ be the characteristic function with respect to $S^\sigma$. Then 
\begin{equation*}
\int_{\mathbb{R}^n} \chi_{S^\sigma}\,\mathrm{d}\boldsymbol{x} \leq \frac{\|f\chi_{S^\sigma}\|_{L^p(\mathbb{R}^n)}^p}{\sigma^p}  \qquad\,\, \text{for any $\sigma>0$}.
\end{equation*}
In particular, if $f(\boldsymbol{x})=f(r)$ is spherically symmetric and $S^\sigma=\{r\in I:\, |f(r)|\geq \sigma\}$, then
\begin{equation*}
\int_0^\infty  \chi_{S^\sigma}r^m\,\mathrm{d}r \leq \frac{\big|r^\frac{m}{p}f\chi_{S^\sigma}\big|_{p}^p}{\sigma^p}  \qquad\,\, \text{for any $\sigma>0$}.
\end{equation*}
\end{lem}

The following lemma is on the evolution triple embedding.
\begin{lem}[\cite{evans}]\label{triple}
Let $T>0$, $n\in\mathbb{N}$, and $n\geq 2$. 
Suppose that $f\in L^2([0,T];H^1(\mathbb{R}^n))$ and  $f_t\in L^2([0,T];H^{-1}(\mathbb{R}^n))$. 
Then $f\in C([0,T];L^2(\mathbb{R}^n))$, and the mapping{\rm :} $t\mapsto |f(t)|_2^2$ is absolutely continuous 
with
\begin{equation*}
\frac{\mathrm{d}}{\mathrm{d}t} \|f(t)\|_{L^2(\mathbb{R}^n)}^2=2\left<f_t, f\right>_{H^{-1}(\mathbb{R}^n)\times H^1(\mathbb{R}^n)} \qquad \text{for {\it a.e.} $t\in (0,T)$}.
\end{equation*}
Moreover, if additionally $f\in L^\infty([0,T];H^1(\mathbb{R}^n))$, then $f\in C([0,T];L^q(\mathbb{R}^n))$ for $q\in [2,\infty)$ if $n=2$ and for $q\in [2,\frac{2n}{n-2})$ if $n\geq 3$.
\end{lem}

The following auxiliary lemma is used to show some equivalent norms for a function $f$ 
satisfying $f\in L^1(\mathbb{R}^n)$ and $\nabla\log f\in L^\infty (\mathbb{R}^n)\cap D^1(\mathbb{R}^n)\cap D^2(\mathbb{R}^n)$. 

\begin{lem}\label{initial3}
Let $n=2,3$, and let $f$ be a spherically symmetric scalar function defined on $\mathbb{R}^n$. If $f\in L^1(\mathbb{R}^n)$, $\nabla\log f\in D^1(\mathbb{R}^n)$ and, additionally, $\nabla\log f \in L^\infty(\mathbb{R}^3)$ if $n=3$, then  
\begin{enumerate}
\item[$\mathrm{(i)}$] $f\in L^p(\mathbb{R}^n)$ for all $p\in (1, \infty]$, and $\nabla\log f\in L^\infty(\mathbb{R}^2)$ if $n=2${\rm ;}

\smallskip
\item[$\mathrm{(ii)}$] $\nabla(f^\nu)\in H^1(\mathbb{R}^n)$ for all $\nu\in [\frac{1}{2},\infty)$.
\end{enumerate}
If $f\in L^1(\mathbb{R}^n)$ and  $\nabla\log f\in D^1(\mathbb{R}^n)\cap D^2(\mathbb{R}^n)$, then {\rm(i)--(ii)} hold and
\begin{enumerate}
\item[$\mathrm{(iii)}$] $\nabla\log f\in  L^\infty(\mathbb{R}^n)$, and $\nabla(f^\nu)\in  H^2(\mathbb{R}^n)$  for all $\nu\in[\frac{1}{2},\infty)$.
\end{enumerate}
\end{lem}
\begin{proof}
Note that, since $f$ is a spherically symmetric scalar function, $\nabla\log f$ is a spherically symmetric vector function.
\smallskip
For (i), we obtain from Lemma \ref{Hk-Ck-vector} in Appendix \ref{improve-sobolev} that
\begin{equation}\label{A13}
\|\nabla\log f\|_{L^\infty(\mathbb{R}^2)}\leq C\|\nabla\log f\|_{D^1(\mathbb{R}^2)}\leq C.
\end{equation}
Then it follows from \eqref{A13}, Lemma \ref{GN-ineq}, and the H\"older inequality that
\begin{align*}
\|f\|_{L^\infty}&\leq C\|f\|_{L^1}^\frac{4-n}{n+4}\|\nabla^2 f\|_{L^2}^\frac{2n}{n+4}\notag\\
&\leq C\|f\|_{L^1}^\frac{4-n}{n+4}\big(\|f|\nabla \log f|^2\|_{L^2}+\|f\nabla^2 \log f\|_{L^2}\big)^\frac{2n}{n+4}\\
&\leq C\|f\|_{L^1}^\frac{4-n}{n+4}\|f\|_{L^2}^\frac{2n}{n+4}\|\nabla \log f\|_{L^\infty}^\frac{4n}{n+4}+C\|f\|_{L^\infty}^\frac{2n}{n+4}\|\nabla^2 \log f\|_{L^2}^\frac{2n}{n+4}\notag\\
&\leq C\|f\|_{L^1}^\frac{4}{n+4}\|f\|_{L^\infty}^\frac{n}{n+4}\|\nabla \log f\|_{L^\infty}^\frac{4n}{n+4}+C\|f\|_{L^\infty}^\frac{2n}{n+4}\|\nabla^2 \log f\|_{L^2}^\frac{2n}{n+4},
\end{align*}
which, along with the Young inequality that
\begin{equation}\label{lemb2-1}
\|f\|_{L^\infty}\leq C\|f\|_{L^1}\|\nabla \log f\|_{L^\infty}^n+C\|\nabla^2 \log f\|_{L^2}^\frac{2n}{4-n}.
\end{equation}
Thus, $f\in L^\infty(\mathbb{R}^n)$, 
and (i) follows easily from this estimate and $f\in L^1(\mathbb{R}^n)$.

\smallskip
For (ii), we need the following identities:
\begin{equation*}
(f^\nu)_{x_i} = \nu f^\nu (\log f)_{x_i},\qquad (f^\nu)_{x_ix_j}=\nu^2 f^\nu (\log f)_{x_i} (\log f)_{x_j}+\nu f^{\nu}(\log f)_{x_ix_j}. 
\end{equation*}
Then these, together with $\nu\in[\frac{1}{2},\infty)$ and the  H\"older inequality, yield that
\begin{align*}
\|\nabla (f^\nu)\|_{L^2}&\leq C\|f^\nu \nabla\log f\|_{L^2} \leq C \|f\|^\nu_{L^{2\nu}} \|\nabla\log f\|_{L^{\infty}}\leq C,\\
\|\nabla^2 (f^\nu)\|_{L^2}&\leq C\|f^\nu |\nabla\log f|^2\|_{L^2}+C\|f^\nu \nabla^2\log f \|_{L^2} \\
&\leq C \|f\|^\nu_{L^{2\nu}}\|\nabla\log f\|_{L^{\infty}}^2  +C \|f\|^\nu_{L^{\infty}} \|\nabla^2\log f\|_{L^{2}} \leq C.
\end{align*}

Finally, for $\mathrm{(iii)}$, we first obtain from Lemmas \ref{GN-ineq} and \ref{lemma-L6} that
\begin{equation}\label{A15}
\begin{aligned}
\|\nabla\log\rho\|_{L^\infty(\mathbb{R}^3)}&\leq C\|\nabla\log\rho\|_{L^6(\mathbb{R}^3)}^\frac{1}{2}\|\nabla\log\rho\|_{D^2(\mathbb{R}^3)}^\frac{1}{2}\\
&\leq C\|\nabla\log\rho\|_{D^1(\mathbb{R}^3)}^\frac{1}{2}\|\nabla\log\rho\|_{D^2(\mathbb{R}^3)}^\frac{1}{2}\leq C.
\end{aligned}
\end{equation}
Since 
\begin{equation*}
\begin{aligned}
(f^\nu)_{x_ix_jx_k}&=\nu^3 f^\nu (\log f)_{x_i} (\log f)_{x_j}(\log f)_{x_k}+\nu^2 f^\nu (\log f)_{x_ix_k} (\log f)_{x_j}\\
&\quad+\!\nu^2 f^\nu (\log f)_{x_i} (\log f)_{x_jx_k}\!+\!\nu^2 f^\nu (\log f)_{x_k} (\log f)_{x_ix_j}\!+\!\nu f^{\nu} (\log f)_{x_ix_jx_k}, 
\end{aligned}
\end{equation*}
it follows from $\nu\in[\frac{1}{2},\infty)$, \eqref{A13}, \eqref{A15}, 
and  the H\"older inequality that 
\begin{align*}
\|\nabla^3 (f^\nu)\|_{L^2}&\leq C\|f^\nu |\nabla\log f|^3\|_{L^2}+C\|f^\nu \nabla^3\log f \|_{L^2} +C\|f^\nu |\nabla\log f||\nabla^2\log f| \|_{L^2}\\
&\leq C \|f\|^\nu_{L^{2\nu}}\|\nabla\log f\|_{L^{\infty}}^3 +C \|f\|^\nu_{L^{\infty}} \|\nabla^3\log f\|_{L^{2}} \\
&\quad +C \|f\|^\nu_{L^{\infty}}\|\nabla\log f\|_{L^{\infty}} \|\nabla^2\log f\|_{L^{2}}\leq C.
\end{align*} 
\end{proof}

Besides, to establish the $L^p(\mathbb{R}^n)$-estimates for transport equations, it also 
requires the following commutator estimates to justify the process of integration by parts. Specifically, consider the following general transport equation:
\begin{equation}\label{diff}
g_t+\diver(\boldsymbol{w} g)=f.
\end{equation}
\begin{lem}[\cite{lions1}]\label{lemma-lions}
Let $n\in\mathbb{N}^*$ and $T>0$. Assume that
\begin{equation*}
f\in L^1([0,T];L^p(\mathbb{R}^n)),\quad \boldsymbol{w}\in L^1([0,T];W^{1,\infty}(\mathbb{R}^n)),\quad g\in L^\infty([0,T];L^p(\mathbb{R}^n)),
\end{equation*}
for some $p\in[1,\infty]$, satisfying \eqref{diff} in the sense of distributions. Denote by $\{\omega_\epsilon\}_{\epsilon>0}$ the standard mollifies on $\mathbb{R}^n$ and define the commutator $\mathcal{C}_\epsilon(\boldsymbol{w},g)$ by 
\begin{equation*}
\mathcal{C}_\epsilon(\boldsymbol{w},g):=\diver(\boldsymbol{w}g)*\omega_\epsilon-\diver(\boldsymbol{w}(g*\omega_\epsilon)).    
\end{equation*} 
Then there exists a constant $C>0$ independent of $(\epsilon,\boldsymbol{w},g)$ such that
\begin{equation}\label{integral-trans}
\big\|\mathcal{C}_\epsilon(\boldsymbol{w},g)\big\|_{L^1([0,T];L^p(\mathbb{R}^n))}\leq C \|\boldsymbol{w}\|_{L^1([0,T];W^{1,\infty}(\mathbb{R}^n))}\|g\|_{L^\infty([0,T];L^p(\mathbb{R}^n))}.
\end{equation}
In addition, $\mathcal{C}_\epsilon(\boldsymbol{w},g)\to 0$ in $L^1([0,T];L^p(\mathbb{R}^n))$ as $\epsilon\to 0$ if $p<\infty$. In particular, for any $p\in[2,\infty)$  and for {\it a.e.} $t\in (0,T)$,
\begin{equation}\label{integral-trans2}
\frac{\mathrm{d}}{\mathrm{d}t}\|g\|^p_{L^p}=-(p-1)\int_{\mathbb{R}^n}|g|^p\diver \boldsymbol{w}\,\mathrm{d}\boldsymbol{x}+p\int_{\mathbb{R}^n}|g|^{p-2}gf\,\mathrm{d}\boldsymbol{x}.
\end{equation}
\end{lem}

\begin{proof}
The proof of \eqref{integral-trans} can be found in \cite[Chapter 2]{lions1}. 
We only give the proof for \eqref{integral-trans2}. First, in view of \eqref{diff}, we see that, for all $\zeta=\zeta(t,\boldsymbol{x})\in C^\infty_{\rm c}((0,T)\times\mathbb{R}^n)$,
\begin{equation}
\int_0^T\int_{\mathbb{R}^n} \zeta_t g^\epsilon \,\mathrm{d}\boldsymbol{x}\mathrm{d}t=\int_0^T\int_{\mathbb{R}^n}\zeta\Big(\underline{\diver(\boldsymbol{w}g^\epsilon)-f^\epsilon+\mathcal{C}_\epsilon(\boldsymbol{w},g)}_{:=\mathcal{R}_\epsilon}\Big)\,\mathrm{d}\boldsymbol{x}\mathrm{d}t.
\end{equation}
where $F^\epsilon:=F*\omega_\epsilon$ for any function $F$. 
Clearly, $g^\epsilon\in L^\infty([0,T];W^{1,p}(\mathbb{R}^n))$, 
and $g^\epsilon$ is smooth in spatial coordinates for each $\epsilon>0$.
Then it follows from the above, \eqref{integral-trans}, and the H\"older inequality that $\mathcal{R}_\epsilon\in L^1([0,T];L^p(\mathbb{R}^n))$, which, along with the definition of weak derivatives, implies that $g^\epsilon$ admits the weak derivative $g^\epsilon_t\in L^1([0,T];L^p(\mathbb{R}^n))$, and 
\begin{equation*}
g^\epsilon_t=-\mathcal{R}_\epsilon=-\diver(\boldsymbol{w}g^\epsilon)+f^\epsilon-\mathcal{C}_\epsilon(\boldsymbol{w},g) \qquad \text{for {\it a.e.} $(t,\boldsymbol{x})\in (0,T)\times \mathbb{R}^n$}.
\end{equation*}

Now multiplying above by $p|g^\epsilon|^{p-2}g^\epsilon$ with $p\in[2,\infty)$ and integrating the resulting equality over $[\tau,t]\times \mathbb{R}^n$ with $0\leq \tau< t\leq T$ yield 
\begin{align*}
&\|g^\epsilon(t)\|^p_{L^p}+(p-1)\int_\tau^t\int_{\mathbb{R}^n}|g^\epsilon|^p\diver \boldsymbol{w}\,\mathrm{d}\boldsymbol{x}\mathrm{d}t'\\
&=\|g^\epsilon(\tau)\|^p_{L^p}+p\int_\tau^t\int_{\mathbb{R}^n}|g^\epsilon|^{p-2}g^\epsilon\big(f^\epsilon-\mathcal{C}_\epsilon(\boldsymbol{w},g) \big)\,\mathrm{d}\boldsymbol{x}\mathrm{d}t'.
\end{align*}
Taking the limit $\epsilon\to 0$, together with $\mathcal{C}_\epsilon(\boldsymbol{w},g)\to 0$ in $L^1([0,T];L^p(\mathbb{R}^n))$, yields that 
\begin{align*}
&\|g(t)\|^p_{L^p}+(p-1)\int_\tau^t\int_{\mathbb{R}^n}|g|^p\diver \boldsymbol{w}\,\mathrm{d}\boldsymbol{x}\mathrm{d}t'=\|g(\tau)\|^p_{L^p}+p\int_\tau^t\int_{\mathbb{R}^n}|g|^{p-2}gf\,\mathrm{d}\boldsymbol{x}\mathrm{d}t'.
\end{align*}
This implies that $\|g(t)\|^p_{L^p}$ is absolutely continuous on $[0,T]$ and thus differentiable for {\it a.e.} $t\in (0,T)$. 
Finally, differentiating the above with respect to $t$ leads to \eqref{integral-trans2}.
\end{proof}

Finally, consider the following equations for the Lam\'e operator $L$:
\begin{equation}\label{aue}
L\boldsymbol{f}=-\alpha\Delta\boldsymbol{f} -\alpha\nabla\diver\boldsymbol{f} =\boldsymbol{g} \qquad\text{in $\mathbb{R}^n$}.
\end{equation}
We focus on the regularity theory of the solution $\boldsymbol{f}$ satisfying 
the asymptotic condition:  
\begin{equation}\label{aue''}
\boldsymbol{f}\to \boldsymbol{0} \qquad \text{as $|\boldsymbol{x}|\to\infty$}.
\end{equation}
The results are stated as follows:
\begin{lem}[\cite{zhangzhifei}]\label{df3}
Let $q\in (1,\infty)$, and let $\boldsymbol{f}\in D^{1,q}(\mathbb{R}^n)$ 
be a weak solution to \eqref{aue} with \eqref{aue''}. 
Then, if $\boldsymbol{g}\in L^q(\mathbb{R}^n)$, there exists a constant $C>0$, depending only on $(n,\alpha,k,q)$ and independent of $(\boldsymbol{f},\boldsymbol{g})$, such that
\begin{equation*}
\|\nabla^{2}\boldsymbol{f}\|_{L^q(\mathbb{R}^n)} \leq C\|\boldsymbol{g}\|_{L^q(\mathbb{R}^n)}.
\end{equation*} 
\end{lem}

\medskip
\section{Conversion of Sobolev Spaces for Spherically Symmetric Functions}\label{appb}

This appendix is devoted to showing the conversion of some Sobolev spaces between the M-D Eulerian coordinate $\boldsymbol{x}$ and the spherical coordinate $r$ for spherically symmetric functions. 
Let $n$ be the spatial dimensions and $m=n-1$.

\begin{lem}\label{lemma-initial}
Let $q\in [1,\infty]$, $0\leq a<b\leq \infty$, $\Omega:=\{\boldsymbol{x}\in \mathbb{R}^n:\, a\leq |\boldsymbol{x}|< b\}$, and $r\in J:=[a,b)$ with $r=|\boldsymbol{x}|$. 
Then
\begin{enumerate}
\item[$\mathrm{(i)}$] for spherically symmetric function $f\in W^{3,q}(\Omega)$ with $f(\boldsymbol{x})=f(r)$,  
\begin{alignat*}{2}
\|f\|_{L^q(\Omega)}&\sim 
\|r^\frac{m}{q}f\|_{L^q(J)},\qquad\qquad \ \,\, \|\nabla f\|_{L^q(\Omega)}&&\sim 
\|r^\frac{m}{q}f_r\|_{L^q(J)},\\
\|\nabla^2 f\|_{L^q(\Omega)}&\sim  \Big\|r^\frac{m}{q}\big(f_{rr},\frac{f_r}{r}\big)\Big\|_{L^q(J)},\quad \|\nabla^3 f\|_{L^q(\Omega)}&&\sim \Big\|r^\frac{m}{q}\Big(f_{rrr},\big(\frac{f_r}{r}\big)_r\Big)\Big\|_{L^q(J)};
\end{alignat*}
\item[$\mathrm{(ii)}$] for spherically symmetric vector function $\boldsymbol{f}\in W^{4,q}(\Omega)$ with $\boldsymbol{f}(\boldsymbol{x})=\frac{\boldsymbol{x}}{r}f(r)$,
\begin{align*}
\|\boldsymbol{f}\|_{L^q(\Omega)}&\sim \|r^\frac{m}{q}f\|_{L^q(J)},\quad \|\nabla \boldsymbol{f}\|_{L^q(\Omega)}\sim \Big\|r^\frac{m}{q}\big(f_r,\frac{f}{r}\big)\Big\|_{L^q(J)},\\
\|\nabla^2 \boldsymbol{f}\|_{L^q(\Omega)}&\sim \Big\|r^\frac{m}{q}\Big(f_{rr},\big(\frac{f}{r}\big)_r\Big)\Big\|_{L^q(J)},\\
\|\nabla^3 \boldsymbol{f}\|_{L^q(\Omega)}&\sim \Big\|r^\frac{m}{q}\Big(f_{rrr},\frac{f_{rr}}{r},\big(\frac{f}{r}\big)_{rr},\frac{1}{r}\big(\frac{f}{r}\big)_r\Big)\Big\|_{L^q(J)},\\
\|\nabla^4 \boldsymbol{f}\|_{L^q(\Omega)}&\sim \Big\|r^\frac{m}{q}\Big(f_{rrrr},\big(\frac{f_{rr}}{r}\big)_r,\big(\frac{f}{r}\big)_{rrr},\Big(\frac{1}{r}(\frac{f}{r}\big)_r\Big)_r\Big)\Big\|_{L^q(J)}.
\end{align*}
\end{enumerate}
Here, $E\sim F$ denotes $C^{-1}E\leq F\leq CE$ for some constant $C\geq 1$ 
depending only on $n$, 
and we have used the following notation for any function space $X$ and functions $(h,g_1,\cdots\!,g_k)$,
\begin{equation*}
\|h(g_1,\cdots\!,g_k)\|_{X}:=\sum_{i=1}^k\|hg_i\|_X.
\end{equation*}
\end{lem}

\begin{proof}
It suffices to prove (ii), since $\nabla f=f_r \frac{\boldsymbol{x}}{r}$ can be regarded as a vector function $\boldsymbol{h}=h\frac{\boldsymbol{x}}{r}$ with $h=f_r$. 
Let $\boldsymbol{f}=(f_1,\cdots\!,f_n)^\top\in\mathbb{R}^n$ with $\boldsymbol{f}(\boldsymbol{x})=f(r)\frac{\boldsymbol{x}}{r}$, and $\boldsymbol{x}=(x_1,\cdots\!,x_n)^\top\in \Omega$. First, it follows from direct calculations that
\begin{align*}
(f_k)_{x_i}&=\frac{x_i x_k}{r^2}f_r+  \frac{\delta_{ik}r^2-x_i x_k}{r^3}f,\\ 
(f_k)_{x_ix_j}&=\frac{x_i x_j x_k}{r^3}f_{rr}+ \Big(\frac{\delta_{ij}x_k+\delta_{ik}x_j+\delta_{jk}x_i}{r}-\frac{3x_i x_j x_k}{r^3}\Big)\big(\frac{f}{r}\big)_r,\\ 
(f_k)_{x_ix_jx_\ell}
&=\frac{x_i x_j x_kx_\ell}{r^4}f_{rrr}\\
&\quad +\Big(\frac{\delta_{i\ell}x_jx_k+\delta_{j\ell}x_ix_k+\delta_{k\ell}x_ix_j}{r^2}\\
&\quad\quad \ \ +\frac{\delta_{ij}x_kx_\ell+\delta_{ik}x_jx_\ell+\delta_{jk}x_ix_\ell}{r^2}-\frac{6x_i x_j x_kx_\ell}{r^4}\Big)\big(\frac{f}{r}\big)_{rr}\\
&\quad + \Big(\delta_{ij}\delta_{k\ell}+\delta_{ik}\delta_{j\ell}+\delta_{jk}\delta_{i\ell}-\frac{\delta_{i\ell}x_jx_k+\delta_{j\ell}x_ix_k+\delta_{k\ell}x_ix_j}{r^2}\\
&\quad\quad \ \ -\frac{\delta_{ij}x_kx_\ell+\delta_{ik}x_jx_\ell+\delta_{jk}x_ix_\ell}{r^2}+\frac{3x_i x_j x_kx_\ell}{r^4}\Big)\Big(\frac{1}{r}\big(\frac{f}{r}\big)_r\Big),\\ 
(f_k)_{x_ix_jx_\ell x_p}
&=\frac{x_i x_j x_kx_\ell x_p}{r^5}f_{rrrr}\\
&\quad +\Big(\frac{\delta_{ip} x_j x_kx_\ell+\delta_{jp} x_i x_kx_\ell+\delta_{kp} x_i x_jx_\ell+\delta_{\ell p} x_i x_jx_k}{r^3}\\ 
&\quad\quad \ \ + \frac{\delta_{i\ell}x_jx_kx_p+\delta_{j\ell}x_ix_kx_p+\delta_{k\ell}x_ix_jx_p}{r^3}\\
&\quad\quad \ \ +\frac{\delta_{ij}x_kx_\ell x_p+\delta_{ik}x_jx_\ell x_p+\delta_{jk}x_ix_\ell x_p}{r^3}-\frac{10x_i x_j x_kx_\ell x_p}{r^5}\Big)\big(\frac{f}{r}\big)_{rrr}\\ 
&\quad +\Big(\frac{\delta_{i\ell}\delta_{jp}x_k+\delta_{i\ell}\delta_{kp}x_j+\delta_{j\ell}\delta_{ip}x_k+\delta_{j\ell}\delta_{kp}x_i+\delta_{k\ell}\delta_{ip}x_j+\delta_{k\ell}\delta_{jp}x_i}{r}\\
&\quad\quad \ \ +\frac{\delta_{ij}\delta_{kp}x_\ell+\delta_{ij}\delta_{\ell p}x_k+\delta_{ik}\delta_{jp}x_\ell+\delta_{ik}\delta_{\ell p} x_j+\delta_{jk}\delta_{ip}x_\ell+\delta_{jk}\delta_{\ell p}x_i}{r}\\
&\quad\quad \ \ +\frac{\delta_{ij}\delta_{k\ell}x_p+\delta_{ik}\delta_{j\ell}x_p+\delta_{jk}\delta_{i\ell}x_p}{r}\\
&\quad\quad \ \ -\frac{3(\delta_{ip} x_j x_kx_\ell+\delta_{jp} x_i x_kx_\ell+\delta_{kp} x_i x_jx_\ell+\delta_{\ell p}x_i x_j x_k)}{r^3}\\ 
&\quad\quad \ \ -\frac{3(\delta_{i\ell}x_jx_kx_p+\delta_{j\ell}x_ix_kx_p+\delta_{k\ell}x_ix_jx_p)}{r^3}\\
&\quad\quad \ \ -\frac{3(\delta_{ij}x_kx_\ell x_p+\delta_{ik}x_jx_\ell x_p+\delta_{jk}x_ix_\ell x_p)}{r^3}+\frac{15x_i x_j x_kx_\ell x_p}{r^5}\Big)\Big(\frac{1}{r}\big(\frac{f}{r}\big)_{r}\Big)_r. 
\end{align*}
Then the above expressions yield
\begin{align*}
|\boldsymbol{f}|^2&=\sum_{i=k}^n |f_k|^2= |f|^2,\quad |\nabla \boldsymbol{f}|^2=\sum_{i,k=1}^n |(f_k)_{x_i}|^2= |f_r|^2+ m\Big|\frac{f}{r}\Big|^2,\notag\\
|\nabla^2 \boldsymbol{f}|^2&=\sum_{i,j,k=1}^n |(f_k)_{x_ix_j}|^2=|f_{rr}|^2+3m\Big|\big(\frac{f}{r}\big)_r\Big|^2,\notag\\
|\nabla^3 \boldsymbol{f}|^2&=\sum_{i,j,k,\ell=1}^n \big|(f_k)_{x_ix_jx_\ell}\big|^2=|f_{rrr}|^2+6m\Big|\big(\frac{f}{r}\big)_{rr}\Big|^2+ (3m^2+6m)  \Big|\frac{1}{r}\big(\frac{f}{r}\big)_r\Big|^2, \\
|\nabla^4 \boldsymbol{f}|^2&=\!\sum_{i,j,k,\ell,p=1}^n\! \big|(f_k)_{x_ix_jx_\ell x_p}\big|^2=|f_{rrrr}|^2\!+10m\Big|\big(\frac{f}{r}\big)_{rrr}\Big|^2\!+ (15m^2\!+\!30m) \Big|\Big(\frac{1}{r}\big(\frac{f}{r}\big)_r\Big)_r\Big|^2,\notag
\end{align*}
which imply that
\begin{equation}\label{BB}
\begin{aligned}
&|\boldsymbol{f}|\sim |f|,\quad |\nabla \boldsymbol{f}|\sim |f_r|+\Big|\frac{f}{r}\Big|,\quad |\nabla^2 \boldsymbol{f}|\sim |f_{rr}|+\Big|\big(\frac{f}{r}\big)_r\Big|,\\
&|\nabla^3 \boldsymbol{f}|\sim |f_{rrr}|+\Big|\big(\frac{f}{r}\big)_{rr}\Big|+\Big|\frac{1}{r}\big(\frac{f}{r}\big)_{r}\Big|,\\
&|\nabla^4 \boldsymbol{f}|\sim |f_{rrrr}|+\Big|\big(\frac{f}{r}\big)_{rrr}\Big|+\Big|\Big(\frac{1}{r}\big(\frac{f}{r}\big)_{r}\Big)_r\Big|.
\end{aligned}
\end{equation}

In addition, since
$\frac{f_{rr}}{r}=\big(\frac{f}{r}\big)_{rr}+\frac{2}{r}\big(\frac{f}{r}\big)_r$, $\eqref{BB}_4$--$\eqref{BB}_5$ can be also written as 
\begin{equation}\label{x-r-2}
\begin{aligned}
&|\nabla^3 \boldsymbol{f}|\sim |f_{rrr}|+\Big|\frac{f_{rr}}{r}\Big|+\Big|\big(\frac{f}{r}\big)_{rr}\Big|+\Big|\frac{1}{r}\big(\frac{f}{r}\big)_{r}\Big|,\\
&|\nabla^4 \boldsymbol{f}|\sim |f_{rrrr}|+\Big|\big(\frac{f_{rr}}{r}\big)_{r}\Big|+\Big|\big(\frac{f}{r}\big)_{rrr}\Big|+\Big|\Big(\frac{1}{r}\big(\frac{f}{r}\big)_{r}\Big)_r\Big|.
\end{aligned}
\end{equation}

Finally, for any spherically symmetric function $g(\boldsymbol{x}) =g(r)$, the following spherical coordinate transformation in $\Omega$ holds:  
\begin{equation*}
\int_\Omega g(\boldsymbol{x})\,\mathrm{d}\boldsymbol{x}= \omega_n\int_J g(r)r^m\,\mathrm{d}r,
\end{equation*}
where $\omega_n$ denotes the surface area of the $n$-sphere. 
Therefore, this, together with \eqref{BB}--\eqref{x-r-2}, yields the desired conclusions. 
\end{proof}

\section{Sobolev Embeddings for Spherically Symmetric Functions}\label{improve-sobolev}

In this appendix, we give the improved Sobolev embeddings for spherically symmetric functions in M-D. 

First, we study the Sobolev embeddings of the type: $D^{1,p}(\mathbb{R}^n)\hookrightarrow L^{q}(\mathbb{R}^n)$ for spherically symmetric vector functions when $p\in [1,n)$. The following auxiliary lemma is used in our analysis, which can be found in Lemma II.6.1 on
\cite[page 81]{galdi}.
\begin{lem}[\cite{galdi}]\label{lemma-galdi}
Let $p\in [1,\infty)$, $n\in \mathbb{N}$, and $n\geq 2$, and let $f\in D^{1,p}(\mathbb{R}^n)$ be a scalar or vector function defined on $\mathbb{R}^n$. Then $f\in W^{1,p}_{\mathrm{loc}}(\mathbb{R}^n)$, that is, $\nabla^j f\in L^p(K)$ $(j=0,1)$ for any bounded domain $K\subset \mathbb{R}^n$.
\end{lem}

Then we have the Sobolev embeddings of the type: $D^{1,p}(\mathbb{R}^n)\hookrightarrow L^{q}(\mathbb{R}^n)$ for spherically symmetric vector functions when $p\in [1,n)$.

\begin{lem}\label{lemma-L6}
Let $\boldsymbol{f}(\boldsymbol{x})=f(r)\frac{\boldsymbol{x}}{r}$ be any spherically symmetric vector function defined on $\mathbb{R}^n$ $(n\geq 2)$. If $\boldsymbol{f}\in D^{1,p}(\mathbb{R}^n)$ for some $p\in [1,n)$, then $\boldsymbol{f}\in L^\frac{np}{n-p}(\mathbb{R}^n)$, and there exists a constant $C(n,p)>0$ depending only on $(n,p)$ such that 
\begin{equation}\label{ineq-L6}
\|\boldsymbol{f}\|_{L^\frac{np}{n-p}(\mathbb{R}^n)}\leq C(n,p)\|\nabla \boldsymbol{f}\|_{L^p(\mathbb{R}^n)}.
\end{equation}
\end{lem}
\begin{proof}
Let $\zeta=\zeta(\boldsymbol{x})\in C^\infty_{\rm c}(\mathbb{R}^n)$ be a spherically symmetric cut-off function such that 
\begin{equation*}
\zeta=\begin{cases}
1& \text{for }|\boldsymbol{x}|\leq 1,\\
\text{smooth}& \text{otherwise},\\
0& \text{for }|\boldsymbol{x}|\geq 2,
\end{cases}\qquad  \zeta\in [0,1],\qquad |\nabla\zeta|\leq C \ \ \text{for some }C>0,
\end{equation*}
and let $\zeta_R=\zeta_R(\boldsymbol{x})=\zeta(\frac{\boldsymbol{x}}{R})$ for $R>0$.

First, since $\boldsymbol{f}\in D^{1,p}(\mathbb{R}^n)$, it follows from Lemma \ref{lemma-galdi} that $\zeta_R \boldsymbol{f}\in W^{1,p}(\mathbb{R}^n)$ for each $R>0$, which, along with Lemma \ref{ale1}, yields $\zeta_R \boldsymbol{f}\in L^{\frac{np}{n-p}}(\mathbb{R}^n)$. Consequently, it follows from \eqref{class-L6} in Lemma \ref{ale1}(i) that
\begin{equation}\label{c1}
\|\zeta_R\boldsymbol{f}\|_{L^\frac{np}{n-p}(\mathbb{R}^n)}\leq C(n,p)\|\nabla(\zeta_R\boldsymbol{f})\|_{L^p(\mathbb{R}^n)}
\end{equation}
for some constant $C(n,p)>0$ depending only on $(n,p)$. Now, based on \eqref{c1} and Lemma \ref{lemma-initial}, we can further derive that
\begin{align*}
\|\zeta_R\boldsymbol{f}\|_{L^\frac{np}{n-p}(\mathbb{R}^n)}&\leq \frac{C(n,p)}{R}\Big(\int_{R\leq |\boldsymbol{x}|\leq 2R} |\boldsymbol{f}|^p\,\mathrm{d}\boldsymbol{x}\Big)^\frac{1}{p}+C(n,p)\|\nabla\boldsymbol{f}\|_{L^p(\mathbb{R}^n)}\\
&\leq \frac{C(n,p)}{R}\Big(\int_{R}^{2R} r^{n-1}|f|^p\,\mathrm{d}r\Big)^\frac{1}{p}+C(n,p)\|\nabla\boldsymbol{f}\|_{L^p(\mathbb{R}^n)}\\
&\leq C(n,p)\Big(\int_{R}^{2R} r^{n-1-p}|f|^p\,\mathrm{d}r\Big)^\frac{1}{p}+C(n,p)\|\nabla\boldsymbol{f}\|_{L^p(\mathbb{R}^n)}\\
&\leq C(n,p)\big(\big|r^{\frac{n-1}{p}-1}f\big|_p+ \|\nabla\boldsymbol{f}\|_{L^p(\mathbb{R}^n)}\big)\leq C(n,p)\|\nabla\boldsymbol{f}\|_{L^p(\mathbb{R}^n)}.
\end{align*}
Note that the above inequality holds uniformly for $R>0$, and $\zeta_{R}|\boldsymbol{f}|^\frac{np}{n-p}\to |\boldsymbol{f}|^\frac{np}{n-p}$ for {\it a.e.} $\boldsymbol{x}\in \mathbb{R}^n$ as $R\to\infty$. Thus, it follow from Lemma \ref{Fatou} that
\begin{align*}
\|\boldsymbol{f}\|_{L^\frac{np}{n-p}(\mathbb{R}^n)}\leq \liminf_{R\to\infty}\|\zeta_{R}\boldsymbol{f}\|_{L^\frac{np}{n-p}(\mathbb{R}^n)} \leq C(n,p)\|\nabla\boldsymbol{f}\|_{L^p(\mathbb{R}^n)}.
\end{align*}
\end{proof}

\begin{rk}\label{remc1}
In contrast to \eqref{class-L6} {\rm(}taking $k=1${\rm)} in {\rm Lemma \ref{ale1}} $\mathrm{(i)}$, 
inequality \eqref{ineq-L6} holds for all spherically symmetric vector functions 
$\boldsymbol{f}$ with $\boldsymbol{f}$ merely in $D^{1,p}(\mathbb{R}^n)$ for some $p\in [1,n)$ 
and without requiring $\boldsymbol{f}\in L^\frac{np}{n-p}(\mathbb{R}^n)$, 
which do not holds for general vector functions $\boldsymbol{f}$ or scalar functions $f$ $($even spherically symmetric ones$)$  belonging to $D^{1,p}(\mathbb{R}^n)$. Such examples inlcude 
$\boldsymbol{f}=(1,\cdots\!,1)^\top$ or $f=1$. 
This distinction lies in the spherical symmetry assumption, since any spherically symmetric constant vector field must vanish. To some extent, the spherical symmetry assumption on $\boldsymbol{f}$ serves as a substitute for $\boldsymbol{f}\in L^\frac{np}{n-p}(\mathbb{R}^n)$ in {\rm Lemma \ref{ale1}} $\mathrm{(i)}$ when $k=1$ and $p<n$.
\end{rk}

\begin{rk}\label{remc2}
When $p=n$, the embeddings of the type{\rm :} 
$D^{1,n}(\mathbb{R}^n)\hookrightarrow L^q(\mathbb{R}^n)$ with $q\in [1,\infty)$ fail to hold for scalar functions or vector functions, even for spherically symmetric ones. 
As a counterexample, consider the vector function $\boldsymbol{f}(\boldsymbol{x})=f(r)\frac{\boldsymbol{x}}{r}$ with $f(r)\in C^\infty([0,\infty))$, $f(r)=0$ if $r\in [0,1]$ and $f(r)=\frac{1}{\log r}$ if $r\in [2,\infty)$. One may verify that both $\boldsymbol{f}(\boldsymbol{x})$ and the scalar function $f(\boldsymbol{x})=f(|\boldsymbol{x}|)$ belong to $D^{1,n}(\mathbb{R}^n)$, yet neither lies in $L^q(\mathbb{R}^n)$ for any $q\in [1,\infty)$.
\end{rk}

Next, we study the Sobolev embeddings of the type: 
$W^{k+1,n}(\mathbb{R}^n)\hookrightarrow W^{k,\infty}(\mathbb{R}^n)$ ($k\in\mathbb{N}$) for spherically symmetric vector functions. For this, we first need to prove the following auxiliary lemma, which indicates that any spherically symmetric vector function in $W^{k,p}(\mathbb{R}^n)$ can be approximated by a sequence of smooth, spherically symmetric vector functions.

\begin{lem}\label{sphere-ruglar}
Suppose that $\boldsymbol{f}\in W^{k,p}(\mathbb{R}^n)$ $(k\in \mathbb{N}^*,n\geq 2, \text{ and }p\in (1,\infty))$ is a spherically symmetric vector function. Then there exists a sequence of spherically symmetric vector functions $\{\boldsymbol{f}^\epsilon\}_{\epsilon>0}\subset C_{\rm c}^\infty(\mathbb{R}^n)$ such that $\|\boldsymbol{f}^\epsilon-\boldsymbol{f}\|_{W^{k,p}(\mathbb{R}^n)}\to 0$ as $\epsilon\to 0$.
\end{lem}

\begin{proof}
Let $\{\omega_\epsilon(\boldsymbol{x})\}_{\varepsilon>0}$ be the standard spherically symmetric mollifier defined on $\mathbb{R}^n$. 
Then we can show that $\boldsymbol{f}^\varepsilon(\boldsymbol{x}):=(\boldsymbol{f}*\omega_\epsilon)(\boldsymbol{x})\varphi(\epsilon\boldsymbol{x})$ satisfies all the requirements given 
in Lemma \ref{sphere-ruglar}, where $\varphi\in C_{\rm c}^\infty(\mathbb{R}^n)$ is a spherically symmetric function such that $\varphi\geq 0$, and  $\varphi=1$ on $|\boldsymbol{x}|\leq 1$ and $\varphi=0$ on $|\boldsymbol{x}|\geq 2$. Here, since the proof of convergence is rather classical, we only check that $\boldsymbol{f}^\varepsilon(\boldsymbol{x})$ is spherically symmetric for brevity, which is equivalent to showing that $\boldsymbol{f}^\varepsilon(\mathcal{O}\boldsymbol{x})=(\mathcal{O}\boldsymbol{f}^\varepsilon)(\boldsymbol{x})$ for any matrix $\mathcal{O}\in \mathrm{SO}(n)$. In fact, we have
\begin{equation*}
\boldsymbol{f}^\varepsilon(\mathcal{O}\boldsymbol{x})=\varphi(\epsilon\mathcal{O}\boldsymbol{x})\int_{\mathbb{R}^n} \boldsymbol{f}(\mathcal{O}\boldsymbol{x}-\boldsymbol{y})\omega_\epsilon(\boldsymbol{y})\,\mathrm{d}\boldsymbol{y}.
\end{equation*}
Changing the coordinates $\boldsymbol{y}=\mathcal{O}\boldsymbol{z}$, along with $|\mathcal{O}\boldsymbol{x}|=|\boldsymbol{x}|$ and $\det \mathcal{O}=1$, gives
\begin{align*}
\boldsymbol{f}^\varepsilon(\mathcal{O}\boldsymbol{x})&=\varphi(\epsilon\mathcal{O}\boldsymbol{x})\int_{\mathbb{R}^n} \boldsymbol{f}(\mathcal{O}\boldsymbol{x}-\mathcal{O}\boldsymbol{z})\omega_\epsilon(\mathcal{O}\boldsymbol{z})(\det \mathcal{O})\,\mathrm{d}\boldsymbol{z}\\
&=\varphi(\epsilon\boldsymbol{x})\int_{\mathbb{R}^n} \boldsymbol{f}(\mathcal{O}(\boldsymbol{x}-\boldsymbol{z}))\omega_\epsilon(\boldsymbol{z}) \,\mathrm{d}\boldsymbol{z}\\
&                                                                                 =\varphi(\epsilon\boldsymbol{x})\int_{\mathbb{R}^n} (\mathcal{O} \boldsymbol{f})(\boldsymbol{x}-\boldsymbol{z})\omega_\epsilon(\boldsymbol{z}) \,\mathrm{d}\boldsymbol{z}\\
&=\Big(\mathcal{O}\big(\varphi(\epsilon\,\cdot)\int_{\mathbb{R}^n} \boldsymbol{f}(\cdot-\boldsymbol{z})\omega_\epsilon(\boldsymbol{z}) \,\mathrm{d}\boldsymbol{z}\big)\Big)(\boldsymbol{x})=(\mathcal{O}\boldsymbol{f}^\varepsilon)(\boldsymbol{x}).
\end{align*}
This completes the proof. 
\end{proof}

We now have the following result, which shows that the Sobolev embedding: $D^{1,n}(\mathbb{R}^n)\hookrightarrow L^\infty(\mathbb{R}^n)$ holds for all spherically symmetric vector functions.
\begin{lem}\label{Hk-Ck-vector}
Let $\boldsymbol{f}(\boldsymbol{x})=f(r)\frac{\boldsymbol{x}}{r}$ be any spherically symmetric vector function defined in $\mathbb{R}^n$ $(n\geq 2)$. If $\boldsymbol{f}\in D^{1,n}(\mathbb{R}^n)$, then $\boldsymbol{f}\in C(\overline{\mathbb{R}^n})$, and there exists a uniform constant $C(n)>0$ depending only on $n$ such that
\begin{equation}\label{vectorC3}
\|\boldsymbol{f}\|_{L^\infty(\mathbb{R}^n)}\leq C(n)\|\nabla\boldsymbol{f}\|_{L^n(\mathbb{R}^n)}.
\end{equation}
\end{lem}
\begin{proof}

We divide the proof into two steps.

\smallskip
\textbf{1. Proof for $\boldsymbol{f}(\boldsymbol{x})\in W^{1,n}(\mathbb{R}^n)$.} We first consider the spherically symmetric 
vector function $\boldsymbol{f}(\boldsymbol{x})\in C_{\rm c}^\infty(\mathbb{R}^n)$ ($n\geq 2$), with $\boldsymbol{f}(\boldsymbol{x})=f(r)\frac{\boldsymbol{x}}{r}$. Certainly, $f\in C_{\rm c}^\infty([0,\infty))$. Then it follows from Lemma \ref{calculus} that, 
for any $r_0\in [0,\infty)$,
\begin{equation}\label{c2-12}
|f(r_0)|^n=-\int_{r_0}^{\infty} \big(|f|^n\big)_r\,\mathrm{d}r \implies |f|_\infty^n\leq \int_0^\infty \big|\big(|f|^n\big)_r\big|\,\mathrm{d}r\leq n\big||f|^{n-1}|f_r|\big|_1,
\end{equation}
which, along with Lemma \ref{lemma-initial} and the H\"older inequality, gives that
\begin{equation}\label{c2-1}
\|\boldsymbol{f}\|_{L^\infty(\mathbb{R}^n)}^n=|f|_\infty^n \leq  n\big||f|^{n-1}|f_r|\big|_1\leq n\big|r^{-\frac{1}{n}}f\big|_n^{n-1}\big|r^\frac{n-1}{n}f_r\big|_n \leq C(n)\|\nabla\boldsymbol{f}\|_{L^n(\mathbb{R}^n)}^n.
\end{equation}

Next, for $\boldsymbol{f}\in W^{1,n}(\mathbb{R}^n)$, thanks to Lemma \ref{sphere-ruglar}, 
there exists a sequence of spherically symmetric vector functions $\{\boldsymbol{f}^\epsilon\}_{\epsilon>0}\subset C^\infty_{\rm c}({\mathbb{R}^n})$ satisfying $\|\boldsymbol{f}^{\epsilon}-\boldsymbol{f}\|_{W^{1,n}(\mathbb{R}^n)}$ as $\epsilon\to 0$. Thus, it follows from \eqref{c2-1} that
\begin{equation*}
\|\boldsymbol{f}^{\epsilon_1}-\boldsymbol{f}^{\epsilon_2}\|_{L^\infty(\mathbb{R}^n)}\leq C(n)\|\nabla\boldsymbol{f}^{\epsilon_1}-\nabla\boldsymbol{f}^{\epsilon_2}\|_{L^n(\mathbb{R}^n)}\to 0 \qquad \text{as $(\epsilon_1,\epsilon_2)\to(0,0)$},   
\end{equation*}
which implies that $\boldsymbol{f}^{\epsilon}$ converges to some limit $\boldsymbol{g}$ as $\epsilon\to 0$. Since $\boldsymbol{f}^\epsilon\in C^\infty_{\rm c}({\mathbb{R}^n})$, $\boldsymbol{g}$ is bounded and uniformly continuous on $\mathbb{R}^n$, and the uniqueness of the limit yields $\boldsymbol{f}=\boldsymbol{g} \in C(\overline{\mathbb{R}^n})$. Therefore, it follows from \eqref{c2-1} that \eqref{vectorC3} holds for $\boldsymbol{f}\in W^{1,n}(\mathbb{R}^n)$, {\it i.e.},
\begin{equation*}
\|\boldsymbol{f}\|_{L^\infty(\mathbb{R}^n)}=\lim_{\epsilon\to0}\|\boldsymbol{f}^\epsilon\|_{L^\infty(\mathbb{R}^n)} \leq C(n)\lim_{\epsilon\to 0}\|\nabla\boldsymbol{f}^\epsilon\|_{L^n(\mathbb{R}^n)}=C(n)\|\nabla\boldsymbol{f}\|_{L^n(\mathbb{R}^n)}.
\end{equation*}

\smallskip
\textbf{2. Proof for $\boldsymbol{f}(\boldsymbol{x})\in D^{1,n}(\mathbb{R}^n)$.} Let $\boldsymbol{f}(\boldsymbol{x})\in D^{1,n}(\mathbb{R}^n)$, and let $(\zeta,\zeta_R)$ be the spherically symmetric cut-off functions defined in Lemma \ref{lemma-L6}. First, using Lemma \ref{lemma-galdi} and $\boldsymbol{f}\in D^{1,n}(\mathbb{R}^n)$ yields that
$\zeta_R \boldsymbol{f}\in W^{1,n}(\mathbb{R}^n)$ for each $R>0$.  

Next, we show that 
\begin{equation}\label{c5}
\lim_{R\to\infty}\big\|\nabla(\zeta_R \boldsymbol{f})-\nabla \boldsymbol{f}\big\|_{L^n(\mathbb{R}^n)}=0.
\end{equation}
It follows from Lemma \ref{lemma-initial} and the absolute continuity of the integral
that, as $R\to\infty$,
\begin{align*}
\big\|\nabla(\zeta_R \boldsymbol{f})-\nabla \boldsymbol{f}\big\|_{L^n(\mathbb{R}^n)}&\leq \big\|\nabla(\zeta_R \boldsymbol{f})-\zeta_R\nabla \boldsymbol{f}\big\|_{L^n(\mathbb{R}^n)}+\big\|\nabla\boldsymbol{f}-\zeta_R\nabla\boldsymbol{f}\big\|_{L^n(\mathbb{R}^n)}\\
&\leq \big\||\nabla\zeta_R| |\boldsymbol{f}|\big\|_{L^n(\mathbb{R}^n)}+\big\|\nabla\boldsymbol{f}-\zeta_R\nabla\boldsymbol{f}\big\|_{L^n(\mathbb{R}^n)}\\
&\leq \frac{C}{R}\Big(\int_{R\leq |\boldsymbol{x}|\leq 2R} |\boldsymbol{f}|^n\,\mathrm{d}\boldsymbol{x}\Big)^\frac{1}{n}+\big\|\nabla\boldsymbol{f}-\zeta_R\nabla\boldsymbol{f}\big\|_{L^n(\mathbb{R}^n)}\\
&\leq \frac{C(n)}{R}\Big(\int_{R}^{2R} r^{n-1}|f|^n\,\mathrm{d}r\Big)^\frac{1}{n}+\big\|\nabla\boldsymbol{f}-\zeta_R\nabla\boldsymbol{f}\big\|_{L^n(\mathbb{R}^n)}\\
&\leq C(n)\Big(\int_{R}^{2R} r^{n-1}\Big|\frac{f}{r}\Big|^n\,\mathrm{d}r\Big)^\frac{1}{n}+\big\|\nabla\boldsymbol{f}-\zeta_R\nabla\boldsymbol{f}\big\|_{L^n(\mathbb{R}^n)}\\
&\leq C(n)\Big(\int_{R\leq |\boldsymbol{x}|\leq 2R} |\nabla\boldsymbol{f}|^n\,\mathrm{d}\boldsymbol{x}\Big)^\frac{1}{n}+ \Big(\int_{|\boldsymbol{x}|\geq R} |\nabla\boldsymbol{f}|^n\,\mathrm{d}\boldsymbol{x}\Big)^\frac{1}{n}\to0,
\end{align*}
which yields claim \eqref{c5}. Clearly, we can also obtain from \eqref{c5} 
that there exists a sufficiently large $R_0>0$ such that 
\begin{equation}\label{c6}
\|\nabla(\zeta_R \boldsymbol{f})\|_{L^n(\mathbb{R}^n)}\leq 2\|\nabla\boldsymbol{f}\|_{L^n(\mathbb{R}^n)}\qquad \text{for all $R\geq R_0$}.   
\end{equation}
Therefore, it follows from $\zeta_R\boldsymbol{f}\in W^{1,n}(\mathbb{R}^n)$, \eqref{c6}, 
and Step 1 above that $\zeta_R\boldsymbol{f}\in C(\overline{\mathbb{R}^n})$ for each $R>0$ and  
\begin{equation}\label{c7}
\|\zeta_R\boldsymbol{f}\|_{L^\infty(\mathbb{R}^n)}\leq C(n)\|\nabla(\zeta_R\boldsymbol{f})\|_{L^n(\mathbb{R}^n)}\leq C(n)\|\nabla \boldsymbol{f}\|_{L^n(\mathbb{R}^n)} \qquad \text{for all $R\geq R_0$}.
\end{equation}
Combining \eqref{c5} and \eqref{c7} yields that, for $R_1,R_2\geq R_0$, when  $(R_1,R_2)\to (\infty,\infty)$,
\begin{equation*}
\|\zeta_{R_1}\boldsymbol{f}-\zeta_{R_2}\boldsymbol{f}\|_{L^\infty(\mathbb{R}^n)}\leq C(n)\|\nabla(\zeta_{R_1}\boldsymbol{f})-\nabla(\zeta_{R_2}\boldsymbol{f})\|_{L^n(\mathbb{R}^n)}\to 0,
\end{equation*}
which implies that  $\{\zeta_R\boldsymbol{f}\}_{R\geq R_0}\subset C(\overline{\mathbb{R}^n})$ is a Cauchy sequence in $L^\infty(\mathbb{R}^n)$, converging to $\boldsymbol{f}\in C(\overline{\mathbb{R}^n})$ as $R\to\infty$. Finally, we obtain from \eqref{c7} that \eqref{vectorC3} holds 
for $\boldsymbol{f}\in D^{1,n}(\mathbb{R}^n)$:
\begin{equation*}
\|\boldsymbol{f}\|_{L^\infty(\mathbb{R}^n)}=\lim_{R\to\infty}\|\zeta_R\boldsymbol{f}\|_{L^\infty(\mathbb{R}^n)}\leq C(n)\|\nabla\boldsymbol{f}\|_{L^2(\mathbb{R}^n)}.
\end{equation*}
The proof of Lemma \ref{Hk-Ck-vector} is completed.
\end{proof}

The following lemma is on the Sobolev embedding of the type: $W^{3,n}(\mathbb{R}^n)\hookrightarrow W^{2,\infty}(\mathbb{R}^n)$.

\begin{lem}\label{Hk-Ck-vector-3}
Let $\boldsymbol{f}(\boldsymbol{x})=f(r)\frac{\boldsymbol{x}}{r}$ be any spherically symmetric vector function defined on $\mathbb{R}^n$ $(n\geq 2)$. If $\boldsymbol{f}\in W^{3,n}(\mathbb{R}^n)$, then $\boldsymbol{f}\in C^2(\overline{\mathbb{R}^n})$, and there exists a uniform constant $C(n)>0$ depending only on $n$ such that
\begin{equation}\label{vectorC3-3}
\|\nabla^2\boldsymbol{f}\|_{L^\infty(\mathbb{R}^n)}\leq C(n)\|\nabla^3\boldsymbol{f}\|_{L^n(\mathbb{R}^n)}.
\end{equation}
\end{lem}

\begin{proof}
We give the proof of \eqref{vectorC3-3} for $\boldsymbol{f}\in C_{\rm c}^\infty(\mathbb{R}^n)$, 
since the rest of Lemma \ref{Hk-Ck-vector-3} can be derived by using an argument similar 
to Step 1 in the proof of Lemma \ref{Hk-Ck-vector}. 

Let $\boldsymbol{f}\in C^\infty_{\rm c}(\mathbb{R}^n)$ be a spherically symmetric vector function with form: $\boldsymbol{f}(\boldsymbol{x})=f(r)\frac{\boldsymbol{x}}{r}$. 
Clearly, it follows from Lemma \ref{lemma-initial} that $\big(f_{rr},(\frac{f}{r})_r\big)\in C_{\rm c}^1([0,\infty))$.  Thus, following the same calculations \eqref{c2-12}--\eqref{c2-1} in Step 1 of 
the proof of Lemma \ref{Hk-Ck-vector}, 
with $f$ replaced by $f_{rr}$ and $(\frac{f}{r})_r$, respectively, we obtain
\begin{align*}
|f_{rr}|_\infty^n&\leq n\big||f_{rr}|^{n-1}|f_{rrr}|\big|_1\leq n\big|r^{-\frac{1}{n}}f_{rr}\big|_n^{n-1}\big|r^\frac{n-1}{n}f_{rrr}\big|_n\leq C(n)\|\nabla^3\boldsymbol{f}\|_{L^n(\mathbb{R}^n)}^n,\\[6pt]
\Big|\big(\frac{f}{r}\big)_r\Big|_\infty^n&\leq n\Big|\big|\big(\frac{f}{r}\big)_r\big|^{n-1}\big|\big(\frac{f}{r}\big)_{rr}\big|\Big|_1\leq n\Big|r^{-\frac{1}{n}}\big(\frac{f}{r}\big)_r\Big|_n^{n-1}\Big|r^\frac{n-1}{n}\big(\frac{f}{r}\big)_{rr}\Big|_n\leq C(n)\|\nabla^3\boldsymbol{f}\|_{L^n(\mathbb{R}^n)}^n,
\end{align*}
which, together with Lemma \ref{lemma-initial}, gives that, for all $\boldsymbol{f}\in C^\infty_{\rm c}(\mathbb{R}^n)$,
\begin{equation*}
\|\nabla^2\boldsymbol{f}\|_{L^{\infty}(\mathbb{R}^n)}\leq C(n)\|\nabla^3\boldsymbol{f}\|_{L^n(\mathbb{R}^n)}. 
\end{equation*}
\end{proof}

\begin{rk}
Since the embeddings{\rm :} $W^{k+1,n}(\mathbb{R}^n) \hookrightarrow W^{k,\infty}(\mathbb{R}^n)$ $(k=0,2)$ 
do not hold in the general case, {\rm Lemmas \ref{Hk-Ck-vector}--\ref{Hk-Ck-vector-3}} amount 
to an extension of the classical Sobolev embedding theorem for the spherically symmetric vector functions. However, {\rm Lemmas \ref{Hk-Ck-vector}--\ref{Hk-Ck-vector-3}} cannot be extended to scalar functions or 
general vector functions, and the embedding $W^{2,n}(\mathbb{R}^n) \hookrightarrow W^{1,\infty}(\mathbb{R}^n)$ cannot hold even for spherically symmetric vector functions. We will provide several counterexamples in {\rm Remarks \ref{k=1rk}--\ref{remarkc4}}.
\end{rk}

Based on Lemmas \ref{Hk-Ck-vector}--\ref{Hk-Ck-vector-3}, we can obtain the Sobolev embeddings for scalar functions.
\begin{lem}\label{Hk-Ck-scalar}
Let $f(\boldsymbol{x})$ be any spherically symmetric function defined on $\mathbb{R}^n$ $(n\geq 2)$. If $f\in W^{k+2,n}(\mathbb{R}^n)$ $(k=0,2)$, then $f\in C^{k+1}(\overline{\mathbb{R}^n})$, and there exists a uniform constant $C(k,n)>0$ depending only on $(k,n)$ such that
\begin{equation}
\|\nabla^{k+1}f\|_{L^\infty(\mathbb{R}^n)}\leq C(k,n)\|\nabla^{k+2}f\|_{L^n(\mathbb{R}^n)}.
\end{equation}
\end{lem}
\begin{proof}
Note that, for any spherically symmetric function $f$ defined on $\mathbb{R}^n$, $\nabla f=f_r \frac{\boldsymbol{x}}{r}$ can be regarded as a new vector function $\boldsymbol{h}=h\frac{\boldsymbol{x}}{r}$ with $h=f_r$. 

Let $f\in W^{k+2,n}(\mathbb{R}^n)$ $(k=0,2)$. Then $\boldsymbol{h}\in W^{k+1,n}(\mathbb{R}^n)$, which, along with Lemmas \ref{Hk-Ck-vector}--\ref{Hk-Ck-vector-3}, leads to $\nabla^{k}\boldsymbol{h}\in C(\overline{\mathbb{R}^n})$, that is, $\nabla^{k+1} f\in C(\overline{\mathbb{R}^n})$. 
\end{proof}

We first give a counterexample to indicate that the embedding:  $W^{2,n}(\mathbb{R}^n)\hookrightarrow W^{1,\infty}(\mathbb{R}^n)$ cannot hold for spherically symmetric vector functions.

\begin{rk}\label{k=1rk}
It is important to note that the embedding $W^{2,n}(\mathbb{R}^n)\hookrightarrow W^{1,\infty}(\mathbb{R}^n)$ does not hold for spherically symmetric vector functions.
For example, let $\boldsymbol{f}(\boldsymbol{x})=f(r)\frac{\boldsymbol{x}}{r}$ be a spherically symmetric vector function with
\begin{equation}\label{form-f}
f(r)\in C^\infty[0,\infty) ,\qquad f(r)=r|\log r|^\nu \text{ on $[0,e^{-1})$}, \qquad f(r)=0 \text{ on $[1,\infty)$}.
\end{equation}
We aim to show that $\boldsymbol{f}\in W^{2,n}(\mathbb{R}^n)$ and $\boldsymbol{f}\notin W^{1,\infty}(\mathbb{R}^n)$ for $\nu\in (0,\frac{n-1}{n})$. Clearly, since $\boldsymbol{f}$ is compactly supported, it suffices to consider the integrability and boundedness of $\boldsymbol{f}$ and its derivatives on $B_*=\{\boldsymbol{x}:\,|\boldsymbol{x}|<e^{-1}\}$.  First, a direct calculation gives that, for $r\in [0,e^{-1})$,
\begin{equation}\label{cal-dd}
\begin{aligned}
f_r&=|\log r|^\nu-\nu|\log r|^{\nu-1},\quad \frac{f}{r}=|\log r|^{\nu},\\
f_{rr}&=-\nu\frac{|\log r|^{\nu-1}}{r}+\nu(\nu-1)\frac{|\log r|^{\nu-2}}{r},\quad \big(\frac{f}{r}\big)_r=-\nu\frac{|\log r|^{\nu-1}}{r}.
\end{aligned}
\end{equation}
Then, since $\nu>0$, we can directly obtain from {\rm Lemma B.1} that
\begin{equation}\label{cl03}
\|\nabla\boldsymbol{f}\|_{L^\infty(B_*)}\geq C^{-1}\Big\|\big(f_r,\frac{f}{r}\big)\Big\|_{L^\infty(0,e^{-1})}=\infty\implies \boldsymbol{f}\notin W^{1,\infty}(\mathbb{R}^n).
\end{equation}
Next, we show that $\boldsymbol{f}\in W^{2,n}(\mathbb{R}^n)$. By \eqref{form-f}{\rm--}\eqref{cal-dd} and {\rm Lemma \ref{lemma-initial}}, it suffices to prove that $\nabla^2\boldsymbol{f}\in L^n(B_*)$. Indeed, since $\nu\in (0,\frac{n-1}{n})$, it follows from {\rm Lemma \ref{lemma-initial}} and the coordinate transformation $z=|\log r|$ that
\begin{align*}
\|\nabla^2\boldsymbol{f}\|_{L^n(B_*)}^n&\leq C\Big\|r^\frac{n-1}{n}\Big(f_{rr},\big(\frac{f}{r}\big)_r\Big)\Big\|_{L^n(0,e^{-1})}^n 
\leq  C\int_0^{e^{-1}}\Big(\frac{|\log r|^{n\nu-n}}{r}+ \frac{|\log r|^{n\nu-2n}}{r}\Big)\,\mathrm{d}r\\
&= C\int_{1}^\infty (z^{n\nu-n} +z^{n\nu-2n} )\,\mathrm{d}z<\infty.
\end{align*}
This implies $\nabla^2\boldsymbol{f}\in L^n(B_*)$. Therefore, we construct a spherically symmetric vector function $\boldsymbol{f}\in W^{2,n}(\mathbb{R}^n)$ such that $\boldsymbol{f}\notin W^{1,\infty}(\mathbb{R}^n)$. 

As indicated by  {\rm Lemma \ref{lemma-initial}}, for a general spherically symmetric function $\boldsymbol{f}$, the $L^n(\mathbb{R}^n)$-boundedness of $\nabla^2\boldsymbol{f}$ does not imply any decay behavior of $\nabla \boldsymbol{f}$ at $\boldsymbol{x}=\boldsymbol{0}$. This suggests that $\nabla \boldsymbol{f}$ may behave singular at $\boldsymbol{x}=\boldsymbol{0}$ $(\text{see \eqref{cl03}})$, even when $\nabla^2\boldsymbol{f}\in L^n(\mathbb{R}^n)$.
\end{rk}

Next, we give a counterexample to Lemma \ref{Hk-Ck-vector} for spherically symmetric scalar functions and general vector functions.
\begin{rk}\label{remarkc4}
It is important to note that the embedding{\rm :} $D^{1,n}(\mathbb{R}^n) \hookrightarrow C(\overline{\mathbb{R}^n})$ fails to hold for spherically symmetric scalar functions or general vector functions. For example, consider the function $f=f(z)$ defined on $[0,\infty)$ satisfying
\begin{equation}\label{log-f}
f(z)=\begin{cases}
0&\,\,\mathrm{on} \  \{z=0\},\\[4pt]
\big|\log z\big|^\frac{1}{3}&\,\,\mathrm{on} \   (0,e^{-1}],\\[4pt]
(ez)^{-2}&\,\,\mathrm{on} \   (e^{-1},\infty).
\end{cases}
\end{equation}
Define $g(\boldsymbol{x})\!=\!f(|\boldsymbol{x}|)$ and $\boldsymbol{h}(\boldsymbol{x})$ with $h_i(\boldsymbol{x})=f(|x_i|)$ $(1\leq i\leq n)$. 
It can be checked that $(g,\boldsymbol{h})$ admits the weak derivatives $(\partial_jg,\partial_j\boldsymbol{h})\in L^n(\mathbb{R}^n)$ $(1\leq j\leq n)$, while $(g,\boldsymbol{h})\notin L^\infty(\mathbb{R}^n)$.

As revealed by {\rm Lemma \ref{lemma-initial}}, these differences arise because the $L^n(\mathbb{R}^n)$-integrability for the gradient of spherically symmetric vector functions provides additional constraints near the symmetry center. Specifically, let $\boldsymbol{\tilde{f}}(\boldsymbol{x}) =\tilde{f}(r) \frac{\boldsymbol{x}}{r}$. Then it follows from {\rm Lemma \ref{lemma-initial}} that 
\begin{equation*}
\begin{aligned}
\nabla \boldsymbol{\tilde{f}} \in L^n(\mathbb{R}^n) &\iff \underline{r^{-\frac{1}{n}}\tilde f \in L^n(I)}_{(\star)} \quad\,\, \text{and} \quad\,\, r^\frac{n-1}{n}\tilde f_r \in L^n(I).
\end{aligned}
\end{equation*}
Remarkably, to some extent, $(\star)$ above implicitly governs the decay behavior of $\boldsymbol{\tilde{f}}$ near the symmetry center since $r^{-\frac{1}{n}}\notin L^n(0,\epsilon)$ for any $\epsilon>0$, thereby preventing the possible concentration of $\boldsymbol{\tilde{f}}$ at the origin, such as the logarithmic growth of $f$ as $z\to 0$ in \eqref{log-f}.
\end{rk}

\bigskip
\noindent{\bf Acknowledgments:}  The authors are grateful to Prof. Didier Bresch  for helpful discussions on the effective velocity and the BD entropy of the degenerate compressible Navier-Stokes equations.
This research is partially supported by National Key R$\&$D Program of China (No. 2022YFA1007300). The research of Gui-Qiang G. Chen was also supported in part by the UK Engineering and Physical Sciences Research
Council Award EP/L015811/1, EP/V008854, and EP/V051121/1.
The research of Shengguo Zhu was also supported in part by 
the National Natural Science Foundation of China under the Grant  12471212, 
and the Royal Society (UK)-Newton International 
Fellowships NF170015.

\bigskip
\noindent{\bf Conflict of Interest:} The authors declare  that they have no conflict of
interest.

\bigskip
\noindent{\bf Data availability:} Data sharing is not applicable to this article as no datasets were generated or analysed during the current study.

\end{document}